%% file: MC-LLT-Monograph9.tex
\documentclass[graybox,envcountchap,sectrefs]{svmono}
\usepackage{mathptmx}
\usepackage{anysize}
\usepackage{enumerate}
\usepackage[normalem]{ulem}
\usepackage[colorlinks,linkcolor=blue,anchorcolor=green,citecolor=blue]{hyperref}
\usepackage{hyperref}

\usepackage{multicol}
\usepackage{makeidx}         
\makeindex             

\usepackage{graphicx,psfrag,epsfig}

\usepackage{amsmath}

\usepackage[usenames,dvipsnames,svgnames,table]{xcolor}
\usepackage{amsfonts}
\usepackage{amssymb}
\usepackage{enumerate}
\newtheorem{lemma}{Lemma}[chapter]
\newtheorem{corollary}[lemma]{Corollary}
\newtheorem{proposition}[lemma]{Proposition}

\newtheorem{definition}[lemma]{Definition}
\newtheorem{remark}[lemma]{Remark}
\newtheorem{example}[lemma]{Example}

\numberwithin{equation}{section}

\newcommand{\ignore}[1]{}

\DeclareFontFamily{OML}{rsfs}{\skewchar\font'177}
\DeclareFontShape{OML}{rsfs}{m}{n}{ <5> <6> rsfs5 <7> <8> <9> rsfs7
  <10> <10.95> <12> <14.4> <17.28> <20.74> <24.88> rsfs10 }{}
\DeclareMathAlphabet{\mathfs}{OML}{rsfs}{m}{n}

\def\un{\underline}
\def\<{\langle}
\def\>{\rangle}
\def\CVar{\mathrm{CVar\,}}

\def\DS{\displaystyle}

\def\fX{\mathfrak X}
\def\fY{\mathfrak Y}
\def\vf{\varphi}
\def\wt{\widetilde}
\def\ov{\overline}
\def\wh{\widehat}
\def\Q{\mathbb Q}

\def\ess{\mathrm{ess\,}}
\def\vec{\overset{\to}}
\def\emptyset{\varnothing}

\input{abbrev.tex}

\begin{document}
\author{Dmitry Dolgopyat and Omri Sarig}
\title{Local limit theorems for inhomogeneous Markov chains}
\date{}

\maketitle

\frontmatter
\mainmatter
\tableofcontents

\addcontentsline{toc}{chapter}{Notation}

\large

\chapter*{Notation}

\vspace{-3cm}
\begin{tabular}{lll}
$\nabla\mathsf a$ &  & the additive functional $\{a^{(N)}_{n+1}(X_{n+1}^{(N)})-a^{(N)}_{n}(X_{n}^{(N)})\}$ (a gradient)\\
$\mathfs B(\fS)$ &  & the Borel $\sigma$-algebra of a separable complete metric  space $\fS$\\
$\mathfrak c^-, \mathfrak c^+$ && large deviations threshold, see \S\ref{Section-Threshold}\\
$C_c(\R)$ && the space of continuous $\vf:\R\to\R$ with compact support\\
$C_{mix}$ && the mixing constant from Proposition \ref{Proposition-Exponential-Mixing}\\
$\mathrm{Cov}$ &  & the covariance\\
$\mathrm{CVar}$ & & the circular variance, see \S \ref{Section-Reduction-Lemmas}\\
$d_n(\xi), d_n^{(N)}(\xi)$ && structure constants, see \S\ref{Section-Structure-Constants}\\
$D_N(\xi)$ && structure constants, see \S\ref{Section-Structure-Constants}\\
$\delta(\pi)$ && the contraction coefficient of a Markov operator $\pi$, see \S\ref{Section-Contraction}\\
$\delta(\mathsf f)$ && the graininess constant of $\mathsf f$, see chapter \ref{Chapter-Irreducibility}\\
$\eps_0$ && (usually) the uniform ellipticity constant, see  \S\ref{section-definition-of-uniform-ellipticity}\\
$\E$, $\E_x$ &  & the expectation operator. $\E_x:=\E(\cdot|X_1=x)$\\
$\ess\sup$ && the essential supremum, see chapter \ref{Chapter-Setup}\\
$\mathsf f, \mathsf g,\mathsf h$ &  &  additive functionals\\
$f_n$, $f_n^{(N)}$ &  & an entry of an additive functional $\mathsf f$ of a Markov chain or array\\
$\mathfs F_N(\xi)$ && the normalized log-moment generating function, see chapter \ref{Chapter-LDP}\\
$G_{alg}(\mathsf X,\mathsf f)$ && the algebraic range, see chapter \ref{Chapter-Irreducibility}\\
$G_{ess}(\mathsf X,\mathsf f)$ && the essential range, see chapter \ref{Chapter-Irreducibility}\\
$\Gamma$ && the balance (of a hexagon), see \S\ref{Section-Hexagons}\\
$H(\mathsf X,\mathsf f)$ && the co-range, see chapter \ref{Chapter-Irreducibility}\\
$\mathrm{Hex}(N,n)$ && the space of level $N$ hexagons at position $n$, see \S\ref{Section-Hexagons}\\
$\mathfs I_N(\eta)$ && the rate function, see chapter \ref{Chapter-LDP}\\
$k_N$ && (usually) the length of the $N$-th row of an array, minus one\\
$\mu(dx)$ && a measure with its integration variable\\
$\log$ && the natural logarithm (same as $\ln$)\\
$\N$ & & $\{1,2,3,\ldots\}$\\
$\Osc$ && the oscillation, see \S\ref{Section-Contraction}\\
$(\Omega,\mathfs F,\mu,T)$ && a measurable map $T:\Omega\to\Omega$ on a measure space $(\Omega,\mathfs F,\mu)$\\
$\Prob(A), \Prob_x(A)$ & &the probability of the event $A$. $\Prob_x(A):=\Prob(A|X_1=x)$\\
$\pi_{n,n+1}(x,dy)$ && the $n$-th transition kernel of a Markov chain\\
$p_n(x,y)$ && (usually) the density of $\pi_{n,n+1}(x,dy)$\\
$\Phi_N(x,\xi)$ && characteristic functions, see \S\ref{Section-Char-Func}\\
$S_N$ &  & $\sum_{i=1}^N f_i(X_i,X_{i+1})$ (chains), or  $\sum_{i=1}^{k_N}f_i^{(N)}(X_i^{(N)},X_{i+1}^{(N)})$ (arrays)\\
$\mathrm{sgn}(x)$ && the sign of $x$: $(+1)$ when $x>0$, $(-1)$ when $x<0$, and $0$ for $x=0$\\
$\fS_n$, $\fS_n^{(N)}$ && the state space of $X_n$ (chains) or of $X_n^{(N)}$ (arrays)\\
$u_n, u_n^{(N)}, U_N$ && structure constants, see \S\ref{Section-Structure-Constants}\\
$\Var$ & & the variance\\
$V_N$ & & the variance of $S_N$\\
$X_n$, $X_n^{(N)}$  &  & an entry of a Markov chain, or the $N$-th row of a Markov array\\
$\mathsf X$ && a Markov chain or a Markov array\\
$z_N$ &  & (usually) a real number not too far from $\E(S_N)$\\
\end{tabular}

\begin{tabular}{lll}
a.e. && almost everywhere\\
a.s. && almost surely\\
TFAE && the following are equivalent\\
s.t. && such that \\
w.l.o.g. && without loss of generality\\
&& \\
$\emptyset$ &  & the empty set\\
$\because$ && because\\
$\therefore$ && therefore\\
$1_E$ && the indicator function of the set $E$, equal to $1$  on $E$ and to\\
&& zero elsewhere\\
$a\pm \eps$ && a quantity inside $[a-\eps,a+\eps]$\\
$e^{\pm \eps}a$ && a quantity in $[e^{-\eps}a,e^{\eps}a]$\\
$\sim$ && $a_n\sim b_n\Leftrightarrow a_n/b_n\xrightarrow[n\to\infty]{}1$\\
$\asymp$ && $a_n\asymp b_n\Leftrightarrow 0<\liminf (a_n/b_n)\leq \limsup (a_n/b_n)<\infty$\\
$\lesssim$ && $a_n\lesssim b_n\Leftrightarrow \limsup (a_n/b_n)<\infty$\\
$\ll$ && for measures: $\mu\ll\nu$ means ``$\nu(E)=0\Rightarrow\mu(E)=0$ for all measurable\\
&& $E$; For numbers: non-rigorous shorthand for  ``much smaller than"\\
$\approx$ && non-rigorous shorthand for ``approximately equal"\\
:= & & is defined to be equal to\\
$\overset{!}{=}$ && an equality that will be justified later\\
$\overset{?}{=}$ && a possibly false equality that requires checking\\
&&\\
$X_n\xrightarrow[n\to \infty]{prob} Y$ && convergence in probability\\
$X_n\xrightarrow[n\to \infty]{dist} Y$ && convergence in distribution\\
$X_n\xrightarrow[n\to \infty]{L^p} Y$ && convergence in $L^p$\\
&&\\
$[S_N>t]$ && conditions in brackets indicate the events that the conditions happen. \\
&& For example, if $\vf:\fS\to\R$, then  $[\vf(\omega)>t]:=\{\omega\in\fS:\vf(\omega)>t\}$\\
$\lfloor x\rfloor$, $\lceil x\rceil$ && $\lfloor x\rfloor:=\max\{n\in\Z:n\leq x\}$, $\lceil x\rceil:=\min\{n\in\Z:n\geq x\}$\\
$\{x\}, \<x\>$ && $\{x\}:=x-\lfloor x\rfloor$; $\<x\>$ is the unique number in $[-\pi,\pi)$ s.t. $x-\<x\>\in 2\pi\Z$\\
$\{x\}_{t\Z}, [x]_{t\Z}$ && $\{x\}_{t\Z}:=t\{x/t\}$, $[x]_{t\Z}:=x-\{x\}_{t\Z}$, so $[x]_{t\Z}\in t\Z$ and $\{x\}_{t\Z}\in [0,t)$
\end{tabular}

\medskip
\noindent
The {\em Fourier transform} of an $L^1$ function $\phi:\R\to\R$ is
$$
\wh{\phi}(\xi):=\int_\R e^{-i\xi u}\phi(u)du.
$$

\noindent
The {\em Legendre-Fenchel transform} of a convex real-valued function $\phi$ on $\R$ is
$$\phi^\ast(\eta):=\sup\limits_{\xi} [\xi\eta-\vf(\xi)].$$

\addcontentsline{toc}{chapter}{Preface}
\chapter*{Preface}

\section*{Setup and aim}\label{Section-LLT-Preface}
Our aim is to provide asymptotic formulas  for the probabilities
$\Prob_x[S_N-z_N\in (a,b)]$, where $X_n$ is a Markov chain, $x$ is some initial state,
$\Prob_x:=\Prob[\ \cdot\ |X_1=x]$, \\
$
\displaystyle S_N=\sum_{n=1}^{N} f_n(X_n,X_{n+1}),
$
 and  $z_N$ are real numbers not too far from $\E(S_N)$.

Such results  are called {\bf local limit theorems (LLT)},\footnote{By contrast, {\em central} limit theorems describe  $\Prob[S_N-z_N\in (a\sqrt{\Var(S_N)},b\sqrt{\Var(S_N)})]$.} and they have a long history, see the end of this chapter.
The novelty of this work is that we allow the Markov chain to be  {\bf inhomogeneous}. This  means that we allow the  set of states, the transition probabilities, and the summands $f_n$
to depend on $n$.

We will always assume that $f_n$ are  uniformly bounded real-valued functions, and that $\{X_n\}$ is uniformly elliptic, a technical condition which will be stated in chapter~\ref{Chapter-Setup}, and which  implies uniform exponential mixing.

These assumptions place us in the Gaussian domain of attraction.  The analogy with classical results for sums of  independent identically distributed (iid) random variables suggests that in the best of all situations, we should expect the following  (in what follows $V_N=\Var(S_N)\text{ and }A_N\sim B_N\Leftrightarrow A_N/B_N\xrightarrow[N\to\infty]{}1$):
\begin{enumerate}[(1)]
\item {\bf Local deviations:} If $\displaystyle\frac{z_N-\E(S_N)}{\sqrt{V_N}}\to z$, then $
\displaystyle\Prob_x[S_N-z_N\in (a,b)]\sim \frac{e^{-z^2/2}}{\sqrt{2\pi V_N}}|a-b|.
$
\item {\bf Moderate deviations:} If $\displaystyle\frac{z_N-\E(S_N)}{V_N}\to 0$, then
\\
\begin{center}$
\displaystyle\Prob_x[S_N-z_N\in (a,b)]\sim \frac{e^{-\frac{1+o(1)}{2}\bigl(\frac{z_N-\E(S_N)}{\sqrt{V_N}}\bigr)^2}}{\sqrt{2\pi V_N}}|a-b|.
$
\end{center}
\item {\bf Large deviations:} If $\displaystyle\left|\frac{z_N-\E(S_N)}{V_N}\right|<c$ with $c>0$ sufficiently small, then for every $x$ in the state space of $X_1$,
$$
\Prob_x[S_N-z_N\in (a,b)]\sim \frac{e^{-V_N \mathfs I_N(\frac{z_N}{V_N})}}{\sqrt{2\pi V_N}}\cdot\int_a^b e^{-t\xi_N\bigl(\frac{z_N-\E(S_N)}{V_N}\bigr)}dt\cdot
\rho_N\left(x, \tfrac{z_N-\E(S_N)}{V_N}\right),\text{ where }
$$

\begin{enumerate}[$\circ$]
\item  $\mathfs I_N(\cdot)$ are the Legendre transforms of  $\mathfs F_N(\xi):=\frac{1}{V_N}\log\E(e^{\xi S_N})$.

\medskip
\item
$\xi_N:(-c,c)\to\R$ satisfy $\xi_N(0)=0$, $\mathrm{sgn}(\xi_N(\eta))=\mathrm{sgn}(\eta)$, and $\exists C>0$ independent of $N$ s.t.  $C^{-1}|\eta|\leq |\xi_N(\eta)|\leq C|\eta|$ for all $N\in\N$, $|\eta|<c$.

\medskip
\item  $\rho_N(x,t)\xrightarrow[t\to 0]{}1$ uniformly in $N,x$, and $\rho_n(\cdot,\cdot)$ are uniformly bounded away from $0,\infty$ on $\mathfrak S_1\times [-c,c]$ where $\mathfrak S_1$ is the state space of $X_1$.

\medskip
\item $c,\xi_N, \rho_N$ depend on the Markov chain, but not on  $z_N$ or on $(a,b)$.
\end{enumerate}
\end{enumerate}
(The asymptotic results in the large deviation regime are more precise than
in the moderate deviation case, but less universal. See Chapter \ref{Chapter-LDP} for more details.)

Although the  asymptotic formulas (1)--(3) above are true in many cases, they do sometime fail --- even when $S_N$ is a sum of iid's.
The aim of this work is  to give general sufficient conditions for (1)--(3),   and to provide the necessary asymptotic corrections when some of  these conditions fail.  To do this we first identify  all the obstructions to (1)--(3), and then we analyze $S_N$ when these obstructions happen.

\section*{The obstructions to the local limit theorems}
The {\em algebraic range} is the smallest closed additive subgroup $G\leq \R$ for which there are $c_n\in\R$ so that  $f_n(X_n,X_{n+1})-c_n\in G$ almost surely for all $n$. We show that the following list is a complete set of obstructions to (1)--(3):
\begin{enumerate}[(I)]
\item {\bf Lattice behavior\/}: The algebraic range is $t\Z$ with $t\in\R$.

\medskip
\item {\bf Center-tightness\/}: $\Var(S_N)$ does not tend to infinity. In chapter \ref{Chapter-Variance} we will see that in this case $\Var(S_N)$ must be bounded.

\medskip
\item {\bf Reducibility\/}:
$
f_n(X_n,X_{n+1})=g_n(X_n,X_{n+1})+c_n(X_n,X_{n+1})
$
where $\{c_n(X_n,X_{n+1})\}$ is center-tight, and the algebraic range of $\{g_n(X_n,X_{n+1})\}$ is strictly smaller than the algebraic range of $\{f_n(X_n,X_{n+1})\}$.
\end{enumerate}
One of our main results is that (1)--(3) hold whenever  (I), (II), (III) fail.

\section*{How to show that the obstructions do not occur}
While it is usually  easy to rule out the lattice obstruction (I), it is often not clear how to rule out (II) and (III). What is needed is a tool that determines from the data of $f_n$ and $X_n$ whether $\{f_n(X_n,X_{n+1})\}$ is center-tight or reducible.

In chapter \ref{Chapter-Setup}, we introduce numerical constants $d_n(\xi)$ $(n\geq 3, \xi\in\R)$ which are defined purely in terms of the  transition probabilities $\pi_{n,n+1}(x,E):=\Prob(X_{n+1}\in E|X_n=x)$ and the functions $f_n(x,y)$, and which can be used to determine which obstructions occur and which vanish:
\begin{enumerate}[$\circ$]
\item If $\sum d_n^2(\xi)=\infty$ for all $\xi\neq 0$, then the obstructions (I),(II),(III) do not occur, and the asymptotic expansions (1)--(3) hold.

\medskip
\item If $\sum d_n^2(\xi)<\infty$ for all $\xi\neq 0$, then $\Var(S_N)$ is bounded  (obstruction II).

\medskip
\item If $\sum d_n^2(\xi)=\infty$ for some but not all $\xi\neq 0$, then $\Var(S_N)\to\infty$ but we are either lattice or reducible: (II) fails, but at least one of (I),(III) occurs.
\end{enumerate}
We call  $d_n(\xi)$ the {\em structure constants} of $\mathsf X=\{X_n\}$ and $\mathsf f=\{f_n\}$.

\section*{What happens when the obstructions do occur}
\subsection*{(I) The lattice case} The lattice obstruction (I) already happens for sums of iid's, and the classical approach how to adjust (1)--(3) to this setup extends without much difficulty  to the inhomogeneous Markov case.

Suppose the algebraic range is $t\Z$ with $t\neq 0$, i.e. there are constants $c_n$ such that  $f_n(X_n,X_{n+1})-c_n\in t\Z$ almost surely for all $n$. Assume further that $t\Z$ is the smallest group with this property.
In this case
$$S_N\in \gamma_N+t\Z\text{ a.s. for all $N$},
$$ where $\gamma_N=\sum_{i=1}^N c_i\mod t\Z$. Instead of analyzing $\Prob_x[S_N-z_N\in (a,b)]$, which might be equal to zero, we study
$
\Prob_x[S_N-z_N=kt],\text{ with }k\in\Z\text{ fixed and  }  \ z_N\in \gamma_N+t\Z.
$

We show that in case (I), if the algebraic range is $t\Z$, and obstructions (II) and (III) do not occur, then (as in the case of iid's):
\begin{enumerate}[(1')]
\item If $\displaystyle\frac{z_N-\E(S_N)}{\sqrt{V_N}}\to z$, $z_N\in \gamma_N+t\Z$ and $k\in\Z$, then
$
\Prob_x[S_N-z_N=kt]\sim \frac{e^{-z^2/2}}{\sqrt{2\pi V_N}}|t|.
$
\item If $\displaystyle\frac{z_N-\E(S_N)}{V_N}\to 0$, $z_N\in \gamma_N+t\Z$ and $k\in\Z$, then\\
$$
\Prob_x[S_N-z_N=kt]\sim \frac{1}{\sqrt{2\pi V_N}}e^{-\frac{1+o(1)}{2}\bigl(\frac{z_N-\E(S_N)}{\sqrt{V_N}}\bigr)^2}|t|.
$$
\item  If $\displaystyle\left|\frac{z_N-\E(S_N)}{V_N}\right|<c$ with $c>0$ sufficiently small, $z_N\in \gamma_N+t\Z$ and $(a,b)\cap t\Z\neq \emptyset$, then for every $x$ in the state space of $X_1$,
$$
\Prob_x[S_N-z_N\in (a,b)]\sim \frac{e^{-V_N \mathfs I_N(\frac{z_N}{V_N})}}{\sqrt{2\pi V_N}}
\cdot \sum_{\tau\in (a,b)\cap t\Z} |t|e^{-\tau\xi_N(\frac{z_N-\E(S_N)}{V_N})} \cdot
\rho_N\left(x,\tfrac{z_N-\E(S_N)}{V_N}\right)
$$
\end{enumerate}
where $\mathfs I_N(\cdot)$, $\rho_N$ and $\xi_N$ have the properties listed in the non-lattice case (3).

\medskip
The previous results hold  for lattice valued, irreducible, non-center tight additive functionals,
that is, when (I) holds and (II),(III) fail. Here is an equivalent condition in terms of the data of $X_n$ and $f_n$:
$$\exists t\neq 0\text{ s.t. }\sum d_n^2(\xi)<\infty\text{  exactly when }\xi\in \frac{2\pi}{t}\Z.$$
Under this condition,  (1')--(3') hold with parameter $|t|$.

\subsection*{(II) The center-tight case}
We show that obstruction (II) happens iff  $f_n(X_n,X_{n+1})$ can be put in the  form
\begin{equation}\tag{$\ast$}
\label{CenterTightDeco}
f_n(X_n,X_{n+1})=a_{n+1}(X_{n+1})-a_n(X_n)+h_n(X_n,X_{n+1})+c_n
\end{equation}
where $a_n(X_n)$ are uniformly bounded, $c_n$ are constants, $h_n(X_n,X_{n+1})$ have mean zero, and $\sum\Var[h_n(X_n,X_{n+1})]<\infty$.

The freedom in choosing $a_n(X_n)$ is too great to allow general statements on the asymptotic behavior of $\Prob_x[S_N-z_N\in (a,b)]$, see  Example \ref{ExTightLimDist}.\footnote{Throughout this work, Example X.Y is example number Y in chapter X. Similarly for Theorems, Propositions etc.} But as we shall we see in chapter \ref{Chapter-Variance}, \eqref{CenterTightDeco} does provide us with some almost sure control:
$$
S_N=a_{N+1}(X_{N+1})-a_1(X_1)+\sum_{n=1}^N h_n(X_n,X_{n+1})+\gamma_N,
$$
 where $\gamma_N=\sum_{i=1}^N c_i$, and $\displaystyle  \sum_{n=1}^\infty h_n(X_n,X_{n+1})$ converges almost surely. This means that in the center-tight scenario, $S_N-\E(S_N)$ can be decomposed into the sum of two terms: A bounded oscillatory term which only depends on $X_1,X_{N+1}$, and a  term which depends on the entire past $X_1,\ldots,X_{N+1}$ and which converges almost surely.

\subsection*{(III) The reducible case}In the reducible case, we can decompose
\begin{equation}\tag{$\ast\ast$}
\label{red-decomp}
f_n(X_n,X_{n+1})=g_n(X_n,X_{n+1})+c_n(X_n,X_{n+1})
\end{equation}
 where $\{c_n(X_n,X_{n+1})\}$ is center-tight, and the algebraic range of $\{g_n(X_n,X_{n+1})\}$ is strictly smaller than the algebraic range of $\{f_n(X_n,X_{n+1})\}$.

 In principle, it is possible that $\{g_n(X_n,X_{n+1})\}$ is  reducible too, but in chapter \ref{Chapter-reducible} we show that one can find an ``optimal" decomposition \eqref{red-decomp} where $\{g_n(X_n,X_{n+1})\}$ is not reducible, and cannot be decomposed further.
 The algebraic range of the ``optimal" $\{g_n(X_n,X_{n+1})\}$ is the ``infimum" of all possible reduced ranges:
 $$
G_{ess}:=\bigcap\left\{G: \begin{array}{l}
 \text{$G$ is the algebraic range of some $\{g_n(X_n,X_{n+1})\}$}\\
 \text{which satisfies \eqref{red-decomp} with $\{c_n(X_n,X_{n+1})\}$ center-tight}
 \end{array}
 \right\}.
 $$
We call  $G_{ess}$  the {\em essential range} of $\{f_n\}$.   It can be calculated explicitly from the data of $f_n$ and $X_n$ in terms of the structure constants, see Theorem \ref{Theorem-essential-range}.

It follows from the definitions that $G_{ess}$ is a proper closed subgroup of $\R$, so $G_{ess}=\{0\}$ or $t\Z$ or $\R$. In the reducible case, $G_{ess}=\{0\}$ or $t\Z$, because if $G_{ess}=\R$, then the algebraic range (which contains $G_{ess}$) is also equal to $\R$.

If $G_{ess}=\{0\}$, then the optimal $\{g_n\}$ has algebraic range $\{0\}$, and $g_n$ are constant functions. In this case $f_n$ is center-tight, and we are
 back in case (II).

If  $G_{ess}=t\Z$ with $t\neq 0$, then $\{g_n(X_n,X_{n+1})\}$ is lattice, non-center-tight, and irreducible. Therefore
\begin{equation}\tag{$\dagger$}
\label{SNSplit}
S_N=\underset{S_N({ g})}{\underbrace{\sum_{n=1}^N g_n(X_n,X_{n+1})}}+\underset{S_n({ c})}{\underbrace{\sum_{n=1}^N c_n(X_n,X_{n+1})}}
\end{equation}
where $S_n({g})$ satisfies the lattice local limit theorems (1')--(3') with parameter $t$, and $\Var[S_N({ c})]=O(1)$. Trading constants between $g$ and $c$, we can also arrange $\E(S_N(c))=O(1)$.

Unfortunately even though $\Var[S_n({f})]\to\infty$ and $\Var[S_N({c})]=O(1)$, examples show that $S_N({c})$ is still powerful enough to disrupt the local limit theorem for $S_N$, lattice or non-lattice (example \ref{Example-SRW-U-First-Step-Chapter-5}).
Heuristically, what happens is that the mass of $S_N({g})$ concentrates on cosets of $t\Z$ according to (1')--(3'), but $S_N({c})$ smudges this mass to a neighborhood  of the lattice in a non-universal manner.

This suggests that (1)--(3) should be approximately true for intervals $(a,b)$ of length $|a-b|\gg|t|$, but false for intervals of length  $|a-b|\ll|t|$.
In chapter \ref{Chapter-reducible} we prove results in this direction.

For  intervals with size   $|a-b|>2|t|$, we show that for all $z_N\in\R$ such that $\frac{z_N-\E(S_N)}{\sqrt{V_N}}\to z$, for all $N$ large enough
$$
\frac{1}{3}\left(\frac{e^{-z^2/2}|a-b|}{\sqrt{2\pi V_N}}\right)\leq \Prob_x[S_N-z_N\in (a,b)]\leq 3\left(\frac{e^{-z^2/2}|a-b|}{\sqrt{2\pi V_N}}\right).
$$
If $|a-b|>L>|t|$, we can replace
 $3$ by a constant $C(L)$ such that $C(L)\xrightarrow[L/|t|\to\infty]{} 1$.

For general intervals, possibly with length less than $|t|$,  we show the following: There are uniformly bounded
 functions $b_N(x_1,x_{N+1})$ and a  random variable $\mathfrak H=\mathfrak H(X_1,X_2,X_3,\ldots)$ so that for every  $z_N\in t\Z$ s.t. $\frac{z_N-\E(S_N)}{\sqrt{V_N}}\to z$, for every  $\phi:\R\to\R$ continuous with compact support,
\begin{equation}\label{LLT-red}\tag{$\ddagger$}
\lim\limits_{N\to\infty}\sqrt{V_N}\E_x[\phi(S_N-z_N-
 b_N(X_1, X_{N+1}))]=\frac{e^{-z^2/2}|t|}{\sqrt{2\pi}}\sum_{m\in\Z}
\E_x[\phi(mt+\mathfrak H)].
\end{equation}

For $\phi\approx 1_{[a,b]}$ with $|a-b|\gg |t|$, the right-hand-side of \eqref{LLT-red} is approximately equal  to  $\frac{e^{-z^2/2}|a-b|}{\sqrt{2\pi}}$, in accordance with (1), see Lemma \ref{Lemma-HHH}.
But for $|a-b|\ll |t|$, the right-hand-side depends on the essential range $t\Z$  and on the detailed structure of  $\{c_n(X_n,X_{n+1})\}$ through $t$, $b_N(X_1,X_{N+1})$ and  $\mathfrak H$.

What are  $b_N(X_N,X_{N+1})$ and $\mathfrak H$? Recall that
the term $c_n(X_n,X_{n+1})$ on the right-hand-side of \eqref{SNSplit} is center-tight.  As such, it can be put in the form
$$ c_n(X_n,X_{n+1})=a_{n+1}(X_{n+1})-a_n(X_n)+h_n(X_n,X_{n+1})+c_n^\ast,$$
where $\sup_n(\ess\sup|a_n|)<\infty$, $c_n^\ast$ are constants, $\E(h_n(X_n,X_{n+1}))=0$, and  $\sum h_n$ converges almost surely. Let $\gamma_N:=\sum_{n=1}^N c_n^\ast=\E(S_N(c))+O(1)=O(1)$. The proof of \eqref{LLT-red} shows that
\begin{enumerate}[$\circ$]
\item $b_N=a_{N+1}(X_{N+1})-a_1(X_1)+\{\gamma_N\}_{t\Z},\text{ where }\{x\}_{t\Z}=|t|\{x/|t|\}=x\mod t\Z$;

\item $\mathfrak H=\sum_{n=1}^\infty h_n(X_n,X_{n+1})$.\footnote{It is possible to replace $\mathfrak H$ by a different random variable $\mathfrak F$ which is bounded, see chapter \ref{Chapter-reducible}.}
\end{enumerate}
This works as follows.
Let $z_N^\ast:=z_N-[\gamma_N]_{t\Z}$, where $[x]_{t\Z}:=x-\{x\}_{t\Z}\in t\Z$. Then $z_N^\ast\in t\Z$,  $\frac{z_N^\ast-\E(S_N)}{V_N}=\frac{z_N-\E(S_N)+O(1)}{V_N}\to z$, and
$$
S_N-b_N-z_N=[S_N(g)-z_N^\ast]+S_N(h).
$$
By subtracting $b_N$ from $S_N$,  we are shifting the distribution of  $S_N$ to the distribution of the sum of two terms: The first, $S_N(g)$, is an {\em irreducible} $t\Z$-valued additive functional; and the second, $S_N(h)$,   converges almost surely to $\mathfrak H$.

Suppose for the sake of discussion that  $S_N(g), S_N(h)$ were independent,  then the lattice LLT for $S_N(g)$ and the definition of $\mathfrak H$
would imply that $$\lim\limits_{N\to\infty}\sqrt{V_N}\E_x[\phi(S_N-b_N-z_N)]=\int_\R\phi(x) m(dx),$$ where  $m:=\frac{e^{-z^2/2}}{\sqrt{2\pi}}m_{t\Z}\ast m_{\mathfrak H}$, and $m_{\mathfrak H}(E):=\Prob[\mathfrak H\in E]$, $m_{t\Z}:=|t|\cdot$counting measure of $t\Z$. Calculating, we find that $\int_\R \phi dm=$right-hand-side of \eqref{LLT-red}.

In general, $S_N(g)$ and $S_N(h)$ are not independent, and  the problem of proving \eqref{LLT-red} reduces to the problem of proving that $S_N(g)$ and $S_N(h)$ are  {\em asymptotically} independent. This is done in chapter \ref{Chapter-reducible}.

For further consequences of \eqref{LLT-red}, including an interpretation in terms of the asymptotic distributional behavior of $S_N$ modulo $t\Z$, see chapter \ref{Chapter-reducible}.

\subsection*{Final words on the setup of this work} Before we end the preface, we would like to comment on a choice we made when we wrote this work, specifically, our focus on additive functionals of the form
$
f_n=f_n(X_n,X_{n+1}).
$

This choice is somewhat unorthodox: The theory of Markov processes is mostly concerned with the case $f_n=f_n(X_n)$ (see e.g. \cite{Do, N, SV}), and the theory of
stochastic processes is mostly concerned with  the case $f_n=f_n(X_n,X_{n+1},\ldots)$, under assumptions of  weak dependence of $X_k,X_n$ when  $|k-n|\gg 1$ (see e.g. \cite{Ibragimov-Linnik, Ruelle-TF}).
We decided to study  $f_n=f_n(X_n,X_{n+1})$ for the following reasons:

\begin{enumerate}[$\circ$]
\item The case $f_n=f_n(X_n,X_{n+1})$ is richer than the case $f_n=f_n(X_n)$ because it contains gradients $a_{n+1}(X_{n+1})-a_n(X_n)$. Two additive functionals which differ by a gradient with uniformly bounded $\ess\sup|a_n|$ will have the same CLT behavior, but they may have different LLT behavior, because their algebraic ranges can be different. This  leads to an interesting reduction theory which we would have missed had we only considered the case  $f_n=f_n(X_n)$.\footnote{We cannot reduce the case $f_n(X_n,X_{n+1})$ to the case $f_n(Y_n)$ by working with the Markov chain $Y_n=(X_n,X_{n+1})$ because $\{Y_n\}$ may no longer satisfy some of our standing assumptions, specifically the uniformly ellipticity condition (see chapter 1).}

\medskip
\item The case $f_n(X_n,\ldots,X_{n+m})$ with $m>1$
 can be deduced from the case $f_n(X_n,X_{n+1})$, and does not require new ideas, see Example \ref{Example-Finite-Memory} and the discussion in \S\ref{Section-Weaker-(E)(c)}. We decided to keep $m=1$ and leave the extension to $m>1$ to the reader.

\medskip
\item The case $f_n=f_n(X_n,X_{n+1},\ldots)$ is of great interest, and we hope to address it in the future, but at the moment our results do not cover it.
\end{enumerate}
We hope to stimulate research into the local limit theorem of additive functionals of general non-stationary stochastic processes with mixing conditions. Such work will  have applications outside the theory of stochastic processes, such as the theory of dynamical systems. Our work here is a step in this direction.

\section*{Notes and references}

\noindent
{\bf Local limit theorems for sums of iid's.} The first LLT is of course the celebrated de Moivre--Laplace Theorem. De Moivre, in his 1738 book \cite{de-Moivre}, gave approximations for $\Prob[a\leq S_n\leq b]$ when $S_n=X_1+\cdots +X_n$, and $X_i$ are iid, equal to zero or one with equal probabilities. Laplace  extended de Moivre's results to the case when  $X_i$ are equal to zero or one with non-equal probabilities \cite{Laplace1,Laplace2}.
 P\'olya, in 1921,  extended these results to  the vector valued iid  which generate the simple random walk on $\Z^d$, and deduced his famous criterion for the recurrence of simple random walks \cite{Polya}.

The next historical landmark  is Gnedenko's 1948 work \cite{Gnedenko48,Gnedenko49} which initiated the study of the  LLT for sums of iid with {\em general} lattice distributions. He asked for the  weakest possible assumptions on the distribution of iid's $X_i$ which lead to a LLT with Gaussian or stable limit. Khinchin  popularized the problem by emphasizing its importance to the foundations of quantum statistical physics \cite{Khinchin}, and it was studied intensively by the Russian school, with important contributions by Linnik, Ibragimov, Prohorov, Richter, Saulis, Petrov  and others. We will comment on some of these contributions in later chapters. For the moment, we refer the reader to the excellent books by Gnendenko \& Kolmogorov \cite{GK}, Ibragimov \& Linnik \cite{Ibragimov-Linnik}, and Petrov \cite{Petrov-Book}
and the many references they contain.

The early works on the local limit theorem all focused on the lattice case. The   Gnedenko--Kolmogorov book \cite{GK} contains the first result we are aware of which could be considered to be a { non-lattice} local limit theorem. The authors assume that each of the iid's $X_i$ have a probability density function $p(x)\in L^r$ with finite variance $\sigma^2$, and show that the density function $p_n(x)$ of $X_1+\cdots+X_n$ satisfies
$$
\sigma\sqrt{n} p_n(\sigma\sqrt{n} x)\xrightarrow[n\to\infty]{}\frac{1}{\sqrt{2\pi}}e^{-x^2/2}.
$$

There could be non-lattice iid's without density functions, for example the iid's
$X_i$ equal to $(-1),$  0, or $\sqrt{2}$ with equal probabilities (the algebraic range is $\R$, because the group generated by $(-1)$ and $\sqrt{2}$ is dense). Shepp \cite{Shepp} was the first to consider non-lattice  LLT  in such cases. His approach was to provide
asymptotic formulas for $\Prob[a\leq S_n-\E(S_N)\leq b]$ for arbitrary intervals $[a,b]$, or for $$\sqrt{2\pi \Var(S_N)}\E[\phi(S_N-\E(S_N))]$$ for all test functions  $\phi:\R\to\R$ which are continuous with compact support.
In this monograph, we use a slight modification of Shepp's formulation of the LLT. Instead of working with $S_N-\E(S_N)$, we work with $S_N-z_N$ subject to the assumptions that $z_N$ is ``not too far" from $\E(S_N)$, and that $S_N-z_N\in$ algebraic range.

Stone proved non-lattice LLT in Shepp's sense for sums of vector valued iid in \cite{S}, extending earlier work of Rva\v{c}eva \cite{Rvaceva} who treated the lattice case. These works are important not only because of the intrinsic interest in the vector valued case, but also because of technical innovations which became tools of the trade, see e.g. \cite{Br}.

\medskip
\noindent
{\bf Local limit theorems for stationary stochastic processes.}
The earliest local limit theorem for non-iid sequences $\{X_i\}$ is due to  Kolmogorov \cite{Kolmogorov-MC-LLT}. He considered stationary homogeneous Markov chains $\{X_i\}$  with a finite set of states $\fS=\{a_1,\ldots,a_n\}$, and proved a local limit theorem for the occupation times
$$
S_N=\sum_{i=1}^N \vec{f}(X_i),\text{ where }\vec{f}(x)=(1_{a_1}(x),\ldots,1_{a_n}(x)).$$
Following further developments for finite state Markov chains by Sirazhdinov \cite{Sirazhdinov}, Nagaev \cite{N} was able to obtain a very general local limit theorems for $S_N=\sum_{i=1}^N f(X_i)$ for a large class of stationary homogeneous countable  Markov chains $\{X_i\}$ and for a variety of unbounded functions $f$, both in the gaussian and stable cases.

Nagaev's paper introduced  the method of characteristic function operators, which is also applicable outside the context of Markov chains. This opened the way for proving LLT for other weakly dependent stationary stochastic processes, and in particular to time series of probability preserving dynamical systems. Guivarc'h \& Hardy \cite{GH} proved gaussian local limit theorems for Birkhoff sums $S_N=\sum_{i=1}^N f(T^i x)$  for Anosov diffeomorphisms $T:X\to X$ with an invariant Gibbs measure, and H\"older continuous functions $f$.
Rosseau-Egele \cite{RE} and Broise \cite{Broise}  proved such theorems for piecewise expanding interval map possessing an absolutely continuous invariant measure, $X=[0,1]$, and $f\in BV$. Aaronson \& Denker \cite{Aaronson-Denker-LLT}  gave general LLT for stationary processes generated by Gibbs-Markov maps both in the gaussian and in the non-gaussian domain of attraction.
These results have found many applications in infinite ergodic theory, dynamical systems and hyperbolic geometry, see for example \cite{Aaronson-Denker-Geodesic}, \cite{Aaronson-Denker-C-minus-Z}, \cite{Aaronson-Denker-Exactness}. The influence of Nagaev's method can also be recognized in other works on other asymptotic problems in dynamics and geometry, see for example \cite{Avila-Dolgopyat-Duryev-Sarig}, {\cite{Babillot-Ledrappier}},
 \cite{Hafouta-Kifer-Book}, {\cite{Katsuda-Sunada}},
\cite{Lalley-Chebotarev}, {\cite{Lalley-Renewal}}, \cite{Ledrappier-Sarig-Cpt}, \cite{Ledrappier-Sarig-Non-cpt},{{\cite{Sharp-Homology}},\cite{Pollicott-Sharp-Chebotarev}},
 \cite{Sharp-Free-Groups}.  For the connection between the LLT and the behavior of local times for stationary stochastic processes,  see \cite{Denker-Zheng, DSV08}.

\medskip
\noindent
{\bf Local limit theorems for non-stationary stochastic processes.}
The interest in  limit theorems for sums of {\em non}-identically distributed, independent, random variables goes back to the works of Chebyshev \cite{Chebyshev-Acta}, Lyapunov \cite{Lyapunov} , and Lindeberg \cite{Lindeberg} who considered the central limit theorem for such sums.

The study of  LLT for sums of non-identically distributed random variables started later, in the works of Prohorov \cite{Prohorov} and Rozanov \cite{Rozanov}.
A common theme in these works and those that followed them is to assume   an asymptotic for $\Prob[a\leq \frac{S_N-A_N}{B_N}\leq b]$ for suitable normalizing constants $A_N, B_N$, and then ask what extra conditions  imply an asymptotic for $\Prob[a\leq {S_N-A_N}\leq b]$.

An important counterexample by Gamerklidze \cite{Gamkrelidze} pointed the way towards the phenomenon that the distribution of  $S_N$ may lie close to a proper sub-group of its algebraic range without actually charging it, and a variety of sufficient conditions which rule this out were developed over the years. We mention especially Rozanov's condition in the lattice case \cite{Rozanov} (see the end of chapter \ref{Chapter-Irreducibility}), the Mineka-Silverman condition in the non-lattice case \cite{Mineka-Silverman}, and Statulevicius's condition  \cite{Statulevicius-Sums-of-Independent}, and conditions motivated by additive number theory such as those appearing in \cite{Moskvin} and \cite{Moskvin-Freiman-Judin}. For a  discussion of these conditions, see \cite{Mukhin-1991}.

Dolgopyat proved a LLT for sums of non-identically distributed, independent random variables which also applies to the reducible case \cite{D-Ind}.

Dobrushin proved a general central limit theorem for inhomogeneous Markov chains in \cite{Do} (see chapter \ref{Chapter-Variance}).
Local limit theorems for inhomogeneous Markov chains are considered in \cite{Statulevicius-LLT-MC}.  Merlev{\`e}de,  M. Peligrad and C. Peligrad proved local limit theorems for sums
$\DS \sum_{i=1}^N f_i(X_i)$ where $\{X_i\}$ is a $\psi$-mixing inhomogeneous Markov chain, under the irreducibility condition of Mineka \& Silverman \cite{Peligrad}. Hafouta obtained local limit theorems for a class of inhomogeneous Markov chains in \cite{Hafouta-Sequential}. In a different direction, central limit theorems for time-series of inhomogeneous sequences of Anosov diffeomorphisms are proved in \cite{Ba} and \cite{Conze-Le-Borgne}.

An important source of examples of inhomogeneous Markov chains is a Markov chain in random environment, when considered for a specific (``quenched") realizations of the environment (see chapter \ref{ChMCRE}). Hafouta \& Kifer proved local limit theorems for non-conventional ergodic  sums in \cite{Hafouta-Kifer-Nonconventional}, and local limit theorems  for random dynamical systems including Markov chains in random environment in \cite{Hafouta-Kifer-Book}.
Demers, P\'ene \& Zhang \cite{DPZ} prove a LLT for an integer valued
observable for a random dynamical system.

Comparing the theory of inhomogeneous Markov chains to theory of Markov chains in random environment studied in  \cite{Hafouta-Kifer-Book}, we note the following differences:
\begin{enumerate}[(a)]
\item
The theory of inhomogeneous Markov chains applies to fixed realizations of noise and not just to almost every realization of noise;
\item  In the random environment setup,  a center--tight additive functional must be a coboundary, while in the general case it can also have a component with summable variances;
\item In the non center-tight random environment setup, the
variance grows linearly for a.e. realization of noise. But for a general inhomogeneous Markov chain it can grow arbitrarily slowly.
\end{enumerate}

\medskip
\noindent
{\bf The contribution of this work.} The novelty of this work is in providing {optimal} sufficient conditions for the classical asymptotic formulas for $\Prob[S_N-z_N\in (a,b)]$, and in the analysis of
$\Prob[S_N-z_N\in (a,b)]$ when these conditions fail.

In particular, we derive a new asymptotic formula for $\Prob[S_N-z_N\in (a,b)]$ in the reducible case, subject to assumption that $V_N:=\Var(S_N)\to\infty$, and we prove a structure theorem for $S_N$ in case $V_N\not\to\infty$.

 Unlike previous works, our analysis does not require any  assumptions on the rate of growth of $V_N$, beyond convergence to infinity.

\bigskip
\noindent
{\bf Acknowledgements:} The work on this monograph was partially supported by the BSF grant 201610.
The authors thank the staff of Weizmann Institute for excellent working conditions.
O.S. was also partially supported by ISF grant 1149/18.
D.D. was  partially supported by NSF grants DMS 1665046 and  DMS 1956049.
 The authors are indebted to Manfred Denker, Yuri Kifer,  and Ofer Zeitouni for useful discussions and suggestions. The authors are in particularly indebted to Yeor Hafouta for many useful and penetrating comments on the first draft of this work.

\chapter{Additive functionals on  Markov arrays}\label{Chapter-Setup}
{\em
This chapter discusses the setup and standing assumptions used in this work.}

\section{The basic setup}
\subsection{Inhomogeneous Markov chains}\label{Section-IMA}
A {\bf Markov chain}\index{Markov chain} is given by the following data:
\begin{enumerate}[$\circ$]
\item  {\bf State spaces\/:}\index{State spaces}\index{Markov chain!state spaces} Borel spaces $
(\fS_n, \mathfs B(\fS_n))$ $(n\geq 1)$, where $\fS_n$ is a complete separable metric space, and  $\mathfs B(\fS_n)$ is the Borel $\sigma$-algebra of $\fS_n$.
$\fS_n$ is the set of  ``the possible states of the Markov chain at time $n$."

\medskip
\noindent
\item {\bf Transition probabilities\/:}\index{Transition probabilities!of Markov chains}\index{Markov chain!transition probabilities} Borel probability measures $\pi_{n,n+1}^{(N)}(x, dy)$ on $\fS_{n+1}$ $(x\in\fS_n, n\geq 1)$, so that  for every Borel $E\subset \fS_{n+1}$, the function $x\mapsto \pi_{n,n+1}^{(N)}(x,E)$ is  measurable. The measure
$\pi_n(x,E)$ is ``the probability of event $E$ at time $n+1$, given  that the state at time $n$ was $x$."

\medskip
\noindent
\item {\bf Initial distribution\/:}\index{Initial distribution!of Markov chains}\index{Markov chain!initial distribution} $\pi(dx)$, a Borel probability measure on $\fS_1$.
$\pi(E)$ is ``the probability that the state $x$ at time $1$ satisfies $x\in E$."
\end{enumerate}
The {\bf Markov chain} associated with this data is the Markov process $\mathsf{X}:=\{X_n\}_{n\geq 1}$
such that $X_n\in\fS_n$ for all $n$, and so that for all Borel $E_i\subset \fS_i$,
$$
\Prob(X_1\in E_1)=\pi(E_1) \ , \ \Prob(X_{n+1}\in E_{n+1}|X_n=x_n)=\pi_{n,n+1}(x_n,E_{n+1}).
$$
$\mathsf X$ is uniquely defined, with joint distribution \index{Markov chain!joint distribution}\index{joint distribution!of a Markov chain}
\begin{align}
&\Prob(X_1\in E_1,\cdots,X_n\in E_n):=\label{Joint-Distribution}\\
&\int_{E_{n-1}}\int_{E_{n-2}}\cdots \int_{E_1}\pi_{n-1,n}(x_{n-1},E_n)\pi(dx_1)\pi_{1,2}(x_1,dx_2)\cdots\pi_{n-2,n-1}(x_{n-2},dx_{n-1}).\notag
\end{align}
$\mathsf X$ satisfies the following important {\bf Markov property}\index{Markov property}:
\begin{align}
&\Prob(X_{n+1}\in E|X_n,X_{n-1},\ldots,X_1)=\Prob(X_{n+1}\in E|X_n)=\pi_{n,n+1}(X_n,E)
\label{Markov property}.
\end{align}
See, for instance, \cite[Ch. 7]{Br}.

In what follows
$
\Prob, \E\text{ and }\Var
$
denote the probability, expectation, and variance calculated using this joint distribution. In the special case when $\pi$ is the point mass at $x$, we write
$
\Prob_x, \E_x\text{ and }\Var_x
$.

If the state spaces and the transition probabilities do not depend on $n$, i.e.,  $\fS_n=\fS_1$
and
$\pi_{n,n+1}(x,dy)=\pi_{1,2}(x,dy)$ for all $n$, then we call $\mathsf X$ a {\bf homogeneous} Markov chain\index{homogeneous Markov chain}\index{Markov chain!homogeneous}. Otherwise, $\mathsf{X}$ is called an {\bf inhomogeneous} Markov chain\index{inhomogeneous Markov chain}\index{Markov chain!inhomogeneous}. In this work, we are mainly interested in the inhomogeneous case.

\begin{example}{\bf (Markov chain with finite state spaces)}.\index{Markov chain! with finite state spaces} These are Markov chains $\mathsf X$ with state spaces
$
\fS_n=\{1,\ldots,d_n\}\ , \ \mathfs B(\fS_n)=\{\text{ subsets of $\fS_n$}\}.
$
\end{example}

In this case the transition probabilities are completely characterized by the rectangular stochastic matrices with entries
$$
\pi^n_{xy}:=\pi_{n,n+1}(x,\{y\})\ \ (x=1,\ldots,d_n\ ; \ y=1,\ldots,d_{n+1}),
$$
and the initial distribution is completely characterized by the probability vector
$$
\pi_x:=\pi(\{x\})\ \ (x=1,\ldots,d_n).
$$
The  joint distribution of $\{X_n\}$ is given by
$$
\Prob(X_1=x_1,\cdots,X_n=x_n)=\pi_{x_1}\pi^1_{x_1 x_2}\pi^2_{x_2 x_3}\cdots \pi^{n-1}_{x_{n-1} x_n},
$$
and this  leads to the following discrete version of \eqref{Joint-Distribution}:
\begin{align*}
&\Prob(X_1\in E_1,\cdots,X_n\in E_n)=
\sum_{x_{n-1}\in E_{n-1}}\sum_{x_{n-2}\in E_{n-2}}\cdots \sum_{x_{1}\in E_{1}}\pi_{x_1}\pi^1_{x_1 x_2}\pi^2_{x_2 x_3}\cdots \pi^{n-1}_{x_{n-1} x_n}.
\end{align*}

\begin{example}{\bf (Markov chains in random environment).}\
\end{example}

 Let $\mathsf X$ denote a homogeneous Markov chain with state space $\fS$, transition probability $\pi(x,dy)$, and initial distribution concentrated at a point $x_1$.
It is possible to view $\mathsf X$ as a model for  the motion of a particle on $\fS$ as follows. At time $1$, the particle is located at $x_1$, and a particle at position $x$ will jump after one time step  to a random location $y$, distributed like  $\pi(x,dy)$:
$
\Prob(y\in E)=\pi(x,E)
$.
 With this interpretation,
$$X_n=\text{ the position of the particle at time $n$}.
$$
The homogeneity of $\mathsf X$ is reflected in the fact that the law of motion which governs the jumps does not change in time.

Let us now refine the model by adding a dependence of  the transition probabilities on an external parameter $\omega$,  which we think of as ``the environment." For example, $\omega$ can represent a external force field which affects the likelihood of various movements, and which can be modified  by
God or some other experimentalist. The transition probabilities become
$
\pi(x,\omega,dy).
$

Suppose the  environment $\omega$ changes in time according to some deterministic rule. This is  modeled by a map $T:\Omega\to\Omega$, where $\Omega$ is the collection of all possible states of the environment, and $T$ is a deterministic law of motion which says that an environment at state $\omega$ will evolve after one unit of time to the state $T(\omega)$.
Iterating we see that if the initial state of the environment at time zero was $\omega$, then its state at time $n$ will be $\omega_n=T^{n-1}(\omega)=(T\circ \cdots\circ T)(\omega).
$

Returning to our particle, we see that  if the initial condition of the environment at time one is $\omega$, then the transition probabilities at time $n$ are
$$
\pi^\omega_{n,n+1}(x,dy)=\pi(x,T^{n-1}(\omega),dy).
$$
Thus each $\omega\in\Omega$ gives rise to  an inhomogeneous Markov chain $\mathsf{X}^\omega$, which describes the Markovian dynamics of a particle, coupled to a changing  environment, and corresponding to the initial condition that at time one, the particle is at position $x_1$ and the environment is at state $\omega$.

If $T(\omega)=\omega$, the environment stays fixed, and the Markov chain is homogeneous, otherwise the Markov chain is inhomogeneous. We will return to Markov chains in random environment in chapter \ref{ChMCRE}.

\begin{example}\label{Example-Finite-Memory}
{\bf (Markov chains with finite memory).}\index{Markov chain!with finite memory}  \ \end{example}

We can weaken the Markov property \eqref{Markov property} by specifying that for some fixed $k_0\geq 1$, for all $E\in\mathfs B(\mathfrak S_{n+1})$,
$$
\Prob(X_{n+1}\in E|X_n,\ldots,X_1)=\begin{cases}
\Prob(X_{n+1}\in E|X_n,\ldots,X_{n-k_0+1}) & n>k_0;\\
\Prob(X_{n+1}\in E|X_n,\ldots,X_{1}) & n\leq k_0.
\end{cases}
$$
Stochastic processes like that are called ``Markov chains with finite memory" (of length $k_0$). Markov chains with memory of length $1$ are ordinary Markov chains. Markov chains with memory of length $k_0>1$ can be recast as ordinary Markov chains by considering the stochastic process $\wt{\mathsf X}=\{(X_n,\ldots,X_{n+k_0-1})\}_{n\geq 1}$ with its natural state spaces, initial distribution, and transition kernels.

\begin{example}
{\bf (A non-example).} Every inhomogeneous Markov chain $\mathsf X$ can be presented as a homogeneous Markov chain $\mathsf Y$, but this is not very useful.
\end{example}

Let $\fS_i$ denote the state spaces of $\mathsf X$. These are complete separable metric spaces, and therefore they are Borel isomorphic to $\R$, or to $\Z$, or to a finite set (see e.g. \cite{Sri}, \S3). So we can construct   Borel bi-measurable injections  $\vf_i:\fS_i\hookrightarrow \R$. Let
$$
Y_n=(\vf_n(X_n),n).
$$

We claim that $\mathsf Y=\{Y_n\}_{n\geq 1}$ is a homogeneous Markov chain.
Let $\delta_\xi$ denote the Dirac measure\index{Dirac measure} at $\xi$, defined by $\delta_\xi(E):=1$ when $E\owns \xi$ and $\delta_\xi(E):=0$  otherwise.
Let $\fS_n, \pi_{n,n+1}$ and $\pi$ denote the states spaces, transition probabilities, and initial distribution of $\mathsf X$. Define a {\em homogeneous} Markov chain
$\mathsf Z$ with
\begin{enumerate}[$\circ$]
\item state space $\fS:=\R\times\N$
\item initial distribution $\wh{\pi}:=(\pi\circ\vf_1^{-1})\times \delta_1$, a measure on  $\fS_1\times\{1\}$
\item transition probabilities
$$
\wh{\pi}\bigl((x,n),A\times B\bigr):=\begin{cases}\pi_{n,n+1}\bigl(\vf_n^{-1}(x),\vf_{n+1}^{-1}(A)\bigr)\delta_{n+1}(B) & x\in\vf_n(\fS_n)\\
\delta_0(A)\delta_1(B) &\text{otherwise.}
\end{cases}
$$
\end{enumerate}
A direct calculation shows that the joint distribution $\mathsf Z$  is equal to the joint distribution of $\mathsf Y=\{(\vf_n(X_n),n)\}_{n\geq 1}$. So $\mathsf Y$ is a homogeneous Markov chain.

Such presentations will not be useful to us, because they destroy useful structures
which are essential for our work on the local limit theorem. For example, they destroy the uniform ellipticity property in section
\ref{Section-Uniform-Ellipticity} below.

\subsection{Inhomogeneous  Markov arrays}\label{Section-Inhomogeneous-Markov-arrays}
For technical reasons that we will explain later, it is useful to consider a generalization of a Markov chain, called a {\bf Markov array}.\index{Markov array}
To define a Markov array, we need the following data:
\begin{enumerate}[$\circ$]
\item {\bf Row lengths:}\index{Markov array!row lengths}\index{row lengths} $k_N+1$ where  $k_N\geq 1$ and $(k_N)_{N\geq 1}$ is strictly increasing.

\medskip
\item {\bf State spaces:}\index{state spaces!of Markov array} $
(\fS_n^{(N)}, \mathfs B(\fS_n^{(N)}))$, $(1\leq n\leq k_N+1)$, where $\fS_n^{(N)}$ is a complete separable metric space with more than one point, and  $\mathfs B(\fS_n^{(N)})$ is its Borel $\sigma$-algebra.

\medskip
\item {\bf Transition probabilities:}\index{transition probabilities!of a Markov array}
$
\{\pi_{n,n+1}^{(N)}(x, dy)\}_{x\in\fS_n^{(N)}}$ $(1\leq n\leq k_N)
$
where  $\pi_{n,n+1}^{(N)}(x, dy)$ are Borel probability measures on
$\fS_{n+1}^{(N)}$, so that for every Borel $E\subset \fS_{n+1}^{(N)}$, the function $x\mapsto \pi_{n,n+1}^{(N)}(x,E)$ is  measurable, and for all $x$,
 and $\pi_{n,n+1}(x,\cdot)$ is not carried by a single atom.

\medskip
\item {\bf Initial distributions:}\index{initial distribution!of Markov array} Borel probability measures $\pi^{(N)}(dx)$  on $\fS_1^{(N)}$.
\end{enumerate}

For each $N\geq 1$, this data determines a finite Markov chain of length $k_N+1$ \\
$\mathsf X^{(N)}=(X^{(N)}_1,X^{(N)}_2,\ldots,X^{(N)}_{k_N+1})$, called the {\bf $N$-th row}\index{Markov array!$N$-th row} of the array. We will continue to denote the joint probability distribution, expectation, and variance of $\mathsf X^{(N)}$ by  $\Prob, \E$, and $\Var$. These objects depend on $N$, but the index $N$ will always be obvious from the context, and can be suppressed. As always, in cases when we wish to condition on the initial state $X_1^{(N)}=x$, we will write $\Prob_x$ and $\E_x$.

The rows $\mathsf X^{(N)}=(X^{(N)}_1,X^{(N)}_2,\ldots,X^{(N)}_{k_N+1})$ can be arranged in an array of random variables
$$
\mathsf X=\left\{
\begin{array}{l}
X^{(1)}_1,\ldots,X^{(1)}_{k_1+1}\\
X^{(2)}_1,\ldots,X^{(2)}_{k_1+1},\ldots,X^{(2)}_{k_2+1}\\
X^{(3)}_1,\ldots,X^{(3)}_{k_1+1},\ldots,X^{(3)}_{k_2+1},\ldots,X^{(3)}_{k_3+1}\\
\ \ \ \ \ \ \cdots\cdots\cdots\cdots\cdots\cdots\cdots\cdots\cdots\cdots
\end{array}
\right.
$$
Each horizontal row $\mathsf X^{(N)}=(X^{(N)}_1,X^{(N)}_2,\ldots,X^{(N)}_{k_N+1})$ comes equipped with a joint distribution, which depends on $N$. But {\em no joint distribution on elements of different rows is specified.}

\begin{example}{\bf (Markov chains as Markov arrays).}\ \end{example}

 Every Markov chain $\{X_n\}$ gives rise to a Markov array with row lengths $k_N=N+1$ and rows $\mathsf X^{(N)}=(X_1,\ldots,X_{N+1})$. In this case $\fS^{(N)}_n=\fS_n$,  $\pi^{(N)}_{n,n+1}=\pi_{n,n+1}$, and $\pi^{(N)}=\pi$.

 Conversely, any Markov array so that $\fS^{(N)}_n=\fS_n$,  $\pi^{(N)}_{n,n+1}=\pi_{n,n+1}$, and $\pi^{(N)}=\pi$ determines a Markov chain with state spaces $\fS_n$, transition probabilities  $\pi^{(N)}_{n,n+1}=\pi_{n,n+1}$, and initial distributions $\pi^{(N)}=\pi$.

\begin{example}
{\bf (Change of measure).}\label{Example-Change-Of-Measure}
Suppose $\{X_n\}_{n\geq 1}$ is a Markov chain with data $\fS_n, \pi_{n,n+1}, \pi$, and let $\vf_n^{(N)}(x,y)$ be a family of positive measurable functions on $\fS_n\times\fS_{n+1}$.
Define new transition probabilities by
$$
\pi^{(N)}_{n,n+1}(x,dy):=
\frac{\vf^{(N)}_{n,n+1}(x,y)}{\int \vf^{(N)}_{n,n+1}(x,y)\pi_{n,n+1}(x,dy)}
\pi_{n,n+1}(x,dy).
$$
Then the data $k_N=N+1$, $\fS^{(N)}_n:=\fS_n$, $\pi^{(N)}:=\pi$ and $\pi^{(N)}_{n,n+1}$ determines a Markov array called the {\bf change of measure}\index{change of measure} of $\{X_n\}$ with {\bf weights} $\vf_n^{(N)}$.
\end{example}

Why study  Markov arrays?\index{Markov array!why study them} There are several reasons, and the one most relevant to this work is the following: The theory of  large deviations for Markov {\em chains}, relies on a change of measure which results in Markov {\em arrays}.
Thus, readers who are only interested in local limit theorems for Markov chains in the {\bf local regime}\index{regime!local} $\frac{z_N-\E(S_N)}{\sqrt{\Var(S_N)}}\to z$, may ignore the theory of arrays and limit their attention to Markov chains. But those who are also interested in the {\bf large deviations regime},\index{regime!of large deviations}  where
$|\frac{z_N-\E(S_N)}{{\Var(S_N)}}|$ is of order 1,
will need the theory for Markov arrays.

\subsection{Additive functionals}
An {\bf additive functional}\index{additive functional!of a Markov chain}  of a Markov chain is  a sequence $\mathsf f=\{f_n\}_{n\geq 1}$ of measurable functions $f_n:\fS_n\times\fS_{n+1}\to\R$, where $\fS_n$ are the states spaces of the Markov chain. The pair $\mathsf X=\{X_n\}, \mathsf f=\{f_n\}$ determines  a stochastic process
$$
S_N=f_1(X_1,X_2)+f_2(X_2,X_3)+\cdots+f_N(X_n,X_{N+1})\ \ \ (N\geq 1).
$$
We will often abuse terminology and call $(\mathsf X, \mathsf f)$ and $\{S_N\}_{N\geq 1}$ ``additive functionals."

An {\bf additive functional}\index{additive functional!of a Markov array}  of a Markov array $\mathsf{X}$ with row lengths $k_N+1$ and state spaces $\fS_n^{(N)}$ is  an  array of measurable functions $f^{(N)}_n:\fS^{(N)}_n\times\fS^{(N)}_{n+1}\to\R$ with row lengths $k_N$:
$$
\mathsf f=\left\{
\begin{array}{l}
f^{(1)}_1,\ldots,f^{(1)}_{k_1}\\
f^{(2)}_1,\ldots,f^{(2)}_{k_1},\ldots,f^{(2)}_{k_2}\\
f^{(3)}_1,\ldots,f^{(3)}_{k_1},\ldots,f^{(3)}_{k_2},\ldots,f^{(3)}_{k_3}\\
\ \ \ \ \ \ \cdots\cdots\cdots\cdots\cdots\cdots\cdots\cdots\cdots\cdots
\end{array}
\right.
$$
Again, this determines a sequence of random variables
$$
S_N=f_1^{(N)}(X_1^{(N)},X_2^{(N)})+f_2^{(N)}(X_2^{(N)},X_3^{(N)})+\cdots+f_{k_N}^{(N)}(X_{k_N}^{(N)},X_{k_N+1}^{(N)})\ \ \ (N\geq 1),
$$
which we also refer to as ``additive functional." But be careful! {\em This is not a stochastic process\/}, because no joint distribution of $S_1,S_2,\ldots$ is specified.

Suppose $\mathsf f, \mathsf g$ are two additive functionals on $\mathsf X$. If $\mathsf X$ is a Markov chain,
$$
\mathsf f+\mathsf g:=\{f_n+g_n\}, \quad c\mathsf f:=\{c f_n\}, \quad
|\mathsf f|:=\sup_n\left(\sup_{x,y} |f_n(x,y)|\right)
$$
and $\ess\sup|\mathsf f|:=\sup\limits_n\left(\ess\sup|f_n(X_n,X_{n+1})|\right)$.

Similarly, if $\mathsf X$  is a Markov array with row lengths $k_N+1$, then
$$
\mathsf f+\mathsf g:=\{f^{(N)}_n+g^{(N)}_n\}, \quad c\mathsf f:=\{c f^{(N)}_n\}, \quad
|\mathsf f|:=\sup_{N}\sup_{1\leq n\leq k_N} \left(\sup_{x,y} |f_n^{(N)}(x,y)|\right),$$ and
$$
\ess\sup|\mathsf f|:=\sup_N \sup_{1\leq n\leq k_N}\left(\ess\sup|f_n^{(N)}(X_n^{(N)},X_{n+1}^{(N)})|\right).
$$
The notation  $|\mathsf f|\leq K\text{ a.s.}$ will mean that $\ess\sup|f|\leq K$ (
``a.s." stands for ``almost surely").
An additive functional is called {\bf uniformly bounded}\index{additive functional!uniformly bounded} if there is a constant $K$ such that $|\mathsf f|\leq K$, and {\bf uniformly bounded a.s.} \index{additive functional!a.s. uniformly bounded} if $\exists K$ such that $|\mathsf f|\leq K$ a.s.

\section{Uniform ellipticity}\label{Section-Uniform-Ellipticity}
\subsection{The definition}\label{section-definition-of-uniform-ellipticity}
A Markov chain $\mathsf X$ with state spaces $\fS_n$ and transition probabilities $\pi_{n,n+1}(x,dy)$ is called {\bf uniformly elliptic}\index{uniform ellipticity!for Markov chains}, if there exists a Borel probability measure $\mu_n$ on $\fS_n$,  Borel measurable functions $p_n:\fS_n\times \fS_{n+1}\to [0,\infty)$, and a constant $0<\epsilon_0<1$ called the {\bf ellipticity constant}\index{ellipticity constant} such that for all $n\geq 1$,
\begin{enumerate}[(a)]
\item $\pi_{n,n+1}(x,dy)=p_{n}(x,y) \mu_{n+1}(dy)$;
\item $0\leq p_{n}\leq 1/\epsilon_0$;
\item $\int_{\fS_{n+1}} p_{n}(x,y) p_{n+1}(y,z)  \mu_{n+1}(dy)>\epsilon_0$.
\end{enumerate}
 We will see in  Proposition \ref{Proposition-nu} below that one can always assume without loss of generality that $\mu_n$ are the measures $\mu_n(E)=\Prob(X_n\in E)$.

The integral in (c) is the two-step transition probability $\Prob(X_{n+2}=z|X_n=x)$, and we will sometime call (c) a {\bf two-step ellipticity condition}.  For more general {\bf $\gamma$-step ellipticity conditions}, see  \S \ref{Section-Weaker-(E)(c)}.\index{uniform ellipticity!$\gamma$-step ellipticity condition} \index{ellipticity condition}

\begin{example}{\bf (Doeblin chains)}\label{Example-Doeblin-Chains} \index{Markov chains!Doeblin}
Suppose $\mathsf X$ has finite state spaces $\fS_n$ s.t $|\fS_n|\leq M<\infty$ for all $n$, and $\pi^n_{xy}:=\pi_{n,n+1}(x,\{y\})$ satisfy
\begin{enumerate}[(1)]
\item $\exists\epsilon_0'>0$ s.t. for all $n\geq 1$ and
$(x,y)\in\fS_n\times\fS_{n+1}$, either $\pi_{xy}^n=0$ or
$\pi_{xy}^n>\epsilon_0'$;
\item for all $n$, for all $(x,z)\in\fS_n\times\fS_{n+2}$, there exists $y\in\fS_{n+1}$ such that $\pi_{xy}^n \pi_{yz}^{n+1}>0$.
\end{enumerate}
\end{example}

Doeblin chains are uniformly elliptic:
Take $\mu_n$ to be the uniform measure on $\fS_n$ and $p_n(x,y):=\pi^n_{xy}/|\fS_{n+1}|$. Then (a) is clear, (b) holds with any $\epsilon_0<1/M$, and (c) holds with $\epsilon_0:=(\epsilon_0'/M)^2$.
 {Doeblin chains} are named after W. Doeblin, who studied homogeneous countable Markov chains satisfying similar conditions.

\medskip
Here is the formulation of the uniform ellipticity conditions for Markov arrays.
A Markov array $\mathsf X$ with state spaces $\fS_n^{(N)}$, transition probabilities $\pi_{n,n+1}^{(N)}(x,dy)$, and row lengths $k_N+1$ is called {\bf uniformly elliptic}\index{uniform ellipticity!for Markov arrays}, if there exist Borel probability measures $\mu_n^{(N)}$ on $\fS_n^{(N)}$,  Borel measurable functions $p_n^{(N)}:\fS_n^{(N)}\times \fS_{n+1}^{(N)}\to [0,\infty)$, and a constant $0<\epsilon_0<1$ as follows: For all $N\geq 1$ and $1\leq n\leq k_N$,
\begin{enumerate}[(a)]
\item $\pi_{n,n+1}^{(N)}(x,dy)=p_{n}^{(N)}(x,y) \mu_{n+1}^{(N)}(dy)$;
\item $0\leq p_{n}^{(N)}\leq 1/\epsilon_0$;
\item $\int_{\fS_{n+1}} p_{n}^{(N)}(x,y) p_{n+1}^{(N)}(y,z) \mu_{n+1}^{(N)}(dy)>\epsilon_0$.
\end{enumerate}

\begin{example} Suppose $\mathsf X$ is a uniformly elliptic Markov chain and suppose $\mathsf Y$ is a Markov array obtained from $\mathsf X$ by the change of measure construction
described in Example \ref{Example-Change-Of-Measure}.
 If the weights $\vf^{(N)}_n(x,y)$ are uniformly bounded away from zero and infinity, then $\mathsf Y$ is uniformly elliptic.
\end{example}

\subsection{Contraction estimates and exponential mixing}\label{Section-Contraction}
Suppose $\fX, \fY$ are  complete and separable metric  spaces. A {\bf transition kernel}\index{transition kernel} from $\fX$ to $\fY$ is a family $\{\pi(x,dy)\}_{x\in\mathfrak X}$ of Borel probability measures on $\fY$  so that $x\mapsto \pi(x,E)$ is measurable for all $E\subset\fX$ Borel. A transition kernel $\{\pi(x,dy)\}_{ x\in\fX}$ determines two {\bf Markov operators}\index{Markov operators}, one acting on measures and the other acting on functions.
The action on measures takes a probability measure $\mu$ on $\fX$ and maps it to a probability measure on $\fY$ via
$$
\pi(\mu)(E):=\int_{\fX}\pi(x,E) \mu(dx).
$$
The action on functions takes a bounded Borel function $u:\fY\to\R$ and maps it to a bounded Borel function on $\fX$ via
$$
\pi(u)(x)=\int_{\fY} u(y) \pi(x,dy).
$$
The two operators are dual:
$
\int u(y)\, \pi(\mu)(dy)=\int \pi(u)(x)\, \mu(dx).
$

These  operators are contractions in the following sense. Define the {\bf oscillation}\index{oscillation} of a function $u:\fY\to\R$ to be
$$
\Osc(u):=\sup_{y_1,y_2\in\fY}|u(y_1)-u(y_2)|.
$$
The {\bf contraction coefficient}\index{contraction coefficient}\index{Markov operator!contraction coefficient}\index{transition kernel!contraction coefficient} of $\{\pi(x,dy)\}_{x\in\fX}$ is
$$
\delta(\pi):=\sup\{|\pi(x_1,E)-
\pi(x_2,E)|: x_1,x_2\in\fX, \;\; E\in\mathfs B(\fY)\}.
$$
The  {\bf total variation distance}\index{total variation distance}
 between two probability measures $\mu_1,\mu_2$ on $\fX$ is
\begin{align*}
\|\mu_1-\mu_2\|_{\Var}&:=\sup\{|\mu_1(A)-\mu_2(A)|:A\subset\fX\text{ is measurable}\}\\
&\equiv \frac{1}{2}\sup\biggl\{\int w(x) (\mu_1-\mu_2)(dx)\big| w:\fX\to [-1,1]\text{ is measurable}\biggr\}.
\end{align*}
Caution! $\|\mu_1-\mu_2\|_{\Var}$ is actually {\em one  half} of the total variation of $\mu_1-\mu_2$, because it is equal to $(\mu_1-\mu_2)^+(\fX)$ and to $(\mu_1-\mu_2)^-(\fX)$, but not to $$|\mu|(\fX)=(\mu_1-\mu_2)^+(\fX)+(\mu_1-\mu_2)^-(\fX).$$

\begin{lemma}[\cite{SV}]\label{Lemma-tempest}
Suppose $\fX, \fY$  are complete and separable metric spaces, and $\{\pi(x,dy)\}_{x\in\fX}$ is a transition kernel from $\fX$ to $\fY$. Then:
\begin{enumerate}[(a)]
\item $0\leq \delta(\pi)\leq 1$.
\item $\delta(\pi)=\sup\{\Osc[\pi(u)]\ |\ \ u:\fY\to\R\text{ measurable, and }\Osc(u)\leq 1\}$.
\item If $\fZ$ is a complete separable metric space,
$\pi_1$ is a transition kernel from $\fX$ to $\fY$, and
$\pi_2$ is a transition kernel from $\fY$ to $\fZ$, then
$\delta(\pi_1\circ\pi_2)\leq \delta(\pi_1)\delta(\pi_2)$.
\item $\Osc[\pi(u)]\leq \delta(\pi)\Osc(u)$  for every $u:\fY\to\R$ bounded and measurable.
\item $\|\pi(\mu_1)-\pi(\mu_2)\|_{\Var}\leq \delta(\pi)\|\mu_1-\mu_2\|_{\Var}$ for all Borel probability measures
$\mu_1,\mu_2$ on
$\fX$.
\item Suppose $\lambda$ is a probability measure on $\fX\times\fY$
 with marginals $\mu_{\fX}$, $\mu_{\fY}$, and transition kernel $\{\pi(x,dy)\}$, i.e. $\lambda(E\times\mathfrak Y )=\mu_{\mathfrak X}(E)$, $\lambda(\mathfrak X\times E)=\mu_{\mathfrak Y}(E)$,  and
 $$
 \lambda(dx,dy)=\int_{\mathfrak X} \pi(x,dy)\mu_{\mathfrak X}(dx).
 $$
Let $f\in L^2(\mu_{\fX}), g\in L^2(\mu_{\fY})$ be two elements with zero integral. Then
$$
\left|\int_{\fX\times\fY} f(x)g(y)\lambda(dx,dy)\right|\leq \sqrt{\delta(\pi)}\|f\|_{L^2(\mu_{\fX})}\|g\|_{L^2(\mu_{\fY})}.
$$
\end{enumerate}
\end{lemma}
\begin{proof}
(a) is trivial.

The inequality $\leq$ in (b) is because for every $E\in\mathfs B(\fY)$,  $u:=1_E$ satisfies $\Osc(u)\leq 1$.
To see $\geq$,\; fix some  $u:\fY\to\R$ measurable such that $\Osc(u)\leq 1$. Suppose first that $u$ is a simple function (a measurable function with finitely many values), then we can write
$
\DS u=c+\sum_{i=1}^m \alpha_i 1_{A_i}
$
where $c\in\R$, $|\alpha_i|\leq \frac{1}{2}\Osc(u)$, and $A_i$ measurable and pairwise disjoint. For every pair of points $x_1,x_2\in \fX$,
\begin{align*}
&|\pi(u)(x_1)-\pi(u)(x_2)|= \left|\sum_{i=1}^{m}\alpha_{i}[\pi(x_1,A_i)-\pi(x_2,A_i)]\right|\\
&\leq \left|\sum_{\pi(x_1,A_i)>\pi(x_2,A_i)}\!\!\!\!\!\alpha_{i}[\pi(x_1,A_i)-\pi(x_2,A_i)]\right|+
\left|\sum_{\pi(x_1,A_i)<\pi(x_2,A_i)}\!\!\!\!\!\alpha_{i}[\pi(x_1,A_i)-\pi(x_2,A_i)]\right|\\
&\leq \frac{1}{2}\Osc(u)\delta(\pi)+\frac{1}{2}\Osc(u)\delta(\pi)=\delta(\pi)\Osc(u)=\delta(\pi).
\end{align*}
So $\Osc[\pi(u)]\leq \delta(\pi)$ for all simple functions $u$ with $\Osc(u)\leq 1$. A standard approximation argument now shows that $\Osc[\pi(u)]\leq \delta(\pi)$ for all measurable $u$ s.t. $\Osc(u)\leq 1$.
This proves (b).
Part  (c) and (d)  immediately follow.

To see (e), we restrict to the non-trivial case  $\mu_1\neq \mu_2$. Let $\mu:=\mu_1-\mu_2$, and decompose  $\mu=\mu^+-\mu^-$ where $\mu^{\pm}$ are singular positive measures (this is the Jordan decomposition). Since $\mu(\fX)=0$, $\mu^+,\mu^-$ has equal total mass, and
$$\mu^{\pm}(\fX)=\frac{1}{2}(\mu^+(\fX)+\mu^-(\fX))=\frac{1}{2}|\mu|(\fX)\equiv \|\mu_1-\mu_2\|_{\mathrm{Var}}.$$
Let
$$\hmu_1:=\mu^+/\|\mu_1-\mu_2\|_{\mathrm{Var}}\ ,\ \hmu_2:=\mu^-/\|\mu_1-\mu_2\|_{\mathrm{Var}}\ , \ \hmu:=\hmu_1-\hmu_2=\frac{\mu_1-\mu_2}{\|\mu_1-\mu_2\|_{\mathrm{Var}}}.$$
Note that $\hmu_1$ and $\hmu_2$ are probability measures.

For every non-constant measurable function $w:\fY\to [-1,1]$,
\begin{align*}
&\frac{\frac{1}{2}\int_{\fY} w(y)\pi(\mu)(dy)}{\|\mu_1-\mu_2\|_{\mathrm{Var}}}=\frac{1}{2}\int_{\fY} w(y_1)\pi(\hmu_1)(dy_1)-\int_{\fY} w(y_2)\pi(\hmu_2)(dy_2)\\
&=\frac{1}{2}\int_{\fX} \pi(w)(x_1)\hmu_1(dx_1)-\int_{\fX} \pi(w)(x_2)\hmu_2(dx_2)\\
&=\frac{1}{2}\int_{\fX}\int_{\fX}[\pi(w)(x_1)-\pi(w)(x_2)]\hmu_1(dx_1)\hmu_2(dx_2) \text{, because $\hmu_i(\fX)=1$,}\\
&\leq \frac{1}{2}\delta(\pi)\Osc(w)\leq \delta(\pi), \text{by (b) and because $\mathrm{Osc}(w)\leq 2\|w\|_\infty\leq 2$.}
\end{align*}
 Passing to the supremum over all $w(y)$ gives part (e).

Part (f) is the content of  Lemma 4.1 in \cite[Lemma 4.1]{SV}, and we reproduce the proof given there.
Consider the $\sigma$-algebra $\mathfs G:=\{\fX\times E:E\subset\fY\text{ is measurable}\}$, which represents the information on the $\fY$--coordinate of $(x,y)\in\fX\times\fY$.

Let $\wt{\pi}_y$ be a measurable family of conditional  probabilities given $\mathfs G$, i.e. $\wt{\pi}_y$ is a probability measure on $\fX\times\{y\}$, $y\mapsto \int f d\wt{\pi}_y$ is Borel for every Borel function $f:\fX\times\fY\to [0,1]$,
$
\lambda=\int_{\fX\times\fY}\wt{\pi}_y d\lambda$, and  for every $\lambda$--absolutely integrable $f(x,y)$,
$$
\E_{\lambda}(f(x,y)|\mathfs G)(y)=\int_{\fX} f d\wt{\pi}_y\ \text{$\lambda$-a.e.}
$$
We may identify $\wt{\pi}_y$ with a probability measure $\wh{\pi}(y, dx)$ on $\fX$ defined by
$$
\wh{\pi}(y,E)=\wt{\pi}_y(E\times\{y\})\ \ \ (E\subset\fX\text{ Borel}).
$$
It is useful to  think of $\wh{\pi}(y,dx)$ as the transition kernel ``which goes the opposite way" to $\pi(x,dy)$. Indeed, if $\pi(x,dy)$ is the transition probability of a Markov chain $\{X_n\}$ from $n$ to $n+1$, and $\lambda$ is the joint distribution of $(X_n,X_{n+1})$,  then $\wh{\pi}(y,dx)$ is the transition probability from $n+1$ to $n$, i.e. $\wh{\pi}(y,E)=\Prob(X_n\in E|X_{n+1}=y)$.

The operators $\pi: L^2(\mu_{\fY})\to L^2(\mu_{\fX})$ and $\wh{\pi}:L^2(\mu_{\fX})\to L^2(\mu_{\fY})$ are dual to one another, because $\int_{\fX} f(x)\pi(g)(x)d\mu_{\fX}(x)$ and $\int_{\fY} \wh{\pi}(f)(y)g(y)d\mu_{\fY}(y)$ are both equal to $\int f(x) g(y) \lambda(dx,dy)$.

\medskip
\noindent
{\sc Claim:\/} {\em   $Q:=\pi\circ\wh{\pi}: L^2(\mu_{\fX})\to L^2(\mu_{\fX})$ is self-adjoint, $Q$ preserves the linear subspace $L^2_0(\mu_{\fX}):=\{f\in L^2(\mu_{\fX}):\int f d\mu_{\fX}=0\}$, and the spectral radius of $Q:L^2_0\to L^2_0$ is at most $\delta(Q)$.}

\medskip
\noindent
{\em Proof of the claim:\/} $Q$ is self adjoint, because $Q^\ast=(\pi\wh{\pi})^\ast=\wh{\pi}^\ast\pi^\ast=\pi\wh{\pi}$.

It is useful to notice that $Q$ is given by
$
(Qf)(x)=\int_{\mathfrak X} f(x') Q(x,dx')
$
where $Q(x,E)$ is the probability measure on $\mathfrak X$ given by
$
Q(x,E)=\int \wh{\pi}(y,E)\pi(x,dy)
$.
$Q(x,dx')$ is a transition probability from $\mathfrak X$ to $\mathfrak X$. Notice that
 $Q(\mu_{\fX})=\mu_{\fX}$:
\begin{align*}
&(Q\mu_{\fX})(E)=\int_{\fX} Q(x,E)\mu_{\fX}(dx)=\int_{\fX}\int_{\fY}\mu_{\fX}(dx)\pi(x,dy)\wt{\pi}_y(E\times\{y\})\\
&=\int_{\fX\times\fY}\wt{\pi}_y(E\times\{y\})\lambda(dx,dy)=\int_{\fX\times\fY}\wt{\pi}_y(E\times\fY)d\lambda=\lambda (E\times\fY)=\mu_{\fX}(E).
\end{align*}

Thus, for all $f\in L^2(\mu_{\mathfrak X})$, $\int Qf d\mu_{\mathfrak X}=\int f d(Q\mu_{\mathfrak X})=\int f d\mu_{\mathfrak X}$.
It follows that
$Q:L^2(\mu_{\mathfrak X})\to L^2(\mu_{\mathfrak X})$ preserves the linear space
$L^2_0.$

For every $\vf\in L^2_0\cap L^\infty$, $\|\vf\|_\infty\leq \Osc(\vf)$. Since $Q$ preserves $L^2_0\cap L^\infty$, for every $f$ in this space, we have by parts (c) and (d) that
\begin{equation}\label{Q-n}
\|Q^n f\|_2\leq \|Q^n f\|_\infty\leq \Osc(Q^n f)\leq \delta(Q)^n\Osc(f).
\end{equation}

This implies that the spectral radius of $Q:L^2_0\to L^2_0$ is less than or equal to $\delta(Q)$. Otherwise there is an $L^2_0$-function, part of whose spectral decomposition corresponds to  the part of the spectrum outside $\{\lambda\in\R: |\lambda|\leq \delta(Q)+\epsilon\}$ (self-adjoint operators have real spectrum). Any sufficiently close $L^2_0\cap L^\infty$--function would have components with similar properties; but the existence of such components is inconsistent with \eqref{Q-n}. The proof of the claim is complete.

\medskip
We are ready for the proof of (f). Since $Q:L_0^2\to L_0^2$ is a self-adjoint operator on a Hilbert space with spectral radius  $\leq\delta(Q)$,  $\<Q(f),f\>_{L^2_0}\leq \delta(Q)\|f\|^2_{L^2_0}$ for all $f\in L^2_0(\mu_{\fX})$. It follows that
\begin{align*}
&\|\wh{\pi}(f)\|_{L^2_0(\mu_{\fY})}^2=\<\wh{\pi}(f),\wh{\pi}(f)\>_{L^2_0(\mu_{\fY})}=\<Q(f),f\>_{L^2_0(\mu_{\fX})}
\leq \delta(Q)\|f\|^2_{L^2_0(\mu_{\fX})}.
\end{align*}
So every $f\in L^2_0(\mu_{\fX}), g\in L^2_0(\mu_{\fY})$
\begin{align*}
&\left|\int_{\fX\times\fY} f(x) g(y) \lambda(dx,dy)\right|=\left|\int_{\fY}\mu_{\fY}(dy)\int_{\fX}\wh{\pi}(y,dx) f(x) g(y)\right|=\<\wh{\pi}(f),g\>_{L^2(\mu_{\fY})}\\
&\leq \|\wh{\pi}(f)\|_2\|g\|_2\leq \sqrt{\delta(Q)}\|f\|_2\|g\|_2,\text{ as required.\hspace{4.35cm}$\Box$}
\end{align*}
\end{proof}

We now return to the setup of Markov arrays $\mathsf X=\{X^{(N)}_n: 1\leq n\leq k_N+1, N\geq 1\}$ and consider the following
{\em two-step transition probabilities}
$$
\pi_{n,n+2}^{(N)}(x,E):=\int \pi_{n+1,n+2}^{(N)}(y,E)\, \pi_{n,n+1}^{(N)}(x,dy)
$$
defined for $1\leq n<N<\infty$, $x\in\fS_n^{(N)}$, and $E\in\mathfs B(\fS_{n+2}^{(N)})$.
The uniform ellipticity condition gives the following uniform bound for $\delta(\pi_{n,n+2}^{(N)})$:
\begin{lemma}\label{Lemma-Contraction}
Let $\mathsf X$ be a uniformly elliptic Markov array with ellipticity coefficient $\epsilon_0$. Then  $\sup\limits_{N}\sup\limits_{1\leq n<k_N}  \delta(\pi_{n,n+2}^{(N)})\leq 1-\epsilon_0$. Similarly for Markov chains.
\end{lemma}
\begin{proof}
We fix $N$ and drop the superscripts $^{(N)}$.

Uniform ellipticity implies that
$\pi_{n, n+2}(x, E)\ll\mu_{n+2}$ and that the Radon-Nikodym  density is bounded from
below by $\eps_0.$ This allows us to write
\begin{equation}
\label{DoeblinTD}
\pi_{n,n+2}(x, dy)=\eps_0 \mu_{n+2}(dy)+(1-\eps_0) \hat\pi_{n, n+2}(x,dy).
\end{equation}
Note that the first term does not depend on $x.$

Let $u:\fS_{n+2}\to\R$ be a measurable function with $\Osc(u)\leq 1$, then we can write $u(\cdot)=c+w(\cdot)$ where $c$ is a constant and  $\|w\|_\infty\leq\frac{1}{2}.$ A direct calculation shows that
\begin{align*}
&\left|\int_{\fS_n} u(z) \pi_{n, n+2} (x_1, dz)-
\int_{\fS_n} u(z) \pi_{n, n+2} (x_2, dz)\right|\\
&=
\left|\int_{\fS_n} w(z) \pi_{n, n+2} (x_1, dz)-
\int_{\fS_n} w(z) \pi_{n, n+2} (x_2, dz)\right|\\
&=(1-\eps_0)
\left|\int_{\fS_n} w(z) \hat\pi_{n, n+2} (x_1, dz)-
\int_{\fS_n} w(z) \hat\pi_{n, n+2} (x_2, dz)\right|\\
&\leq (1-\eps_0) \|w\|_\infty \left[\pi_{n, n+2}(x_1, \fS_{n+2}) +\pi_{n, n+2}(x_2, \fS_{n+2})
\right]
\leq 1-\eps_0,
\end{align*}
where the last inequality holds since $\|w\|_\infty\leq \frac{1}{2}.$ \qed
\end{proof}

\begin{proposition}\label{Proposition-Exponential-Mixing}
If $\mathsf X$ is uniformly elliptic,  then there exist  $\theta\in (0,1)$ and $C_{mix}>0$, which only depend on the ellipticity constant $\epsilon_0$ as follows.  Suppose $h_n^{(N)}(x,y)$ are measurable functions on $\fS_n^{(N)}\times\fS_{n+1}^{(N)}$, and let $h_n^{(N)}:=h_n^{(N)}(X_n^{(N)},X_{n+1}^{(N)})$, then
\begin{enumerate}[(1)]
\item If $h_n^{(N)}$ is bounded and $\E(h_n^{(N)})=0$, then for all $1\leq m<n\leq k_N$
\begin{equation}\label{Exp-Mixing-L-infinity}
\|\E\bigl(h_n^{(N)}|X_m^{(N)}\bigr)\|_\infty\leq C_{mix}\theta^{n-m}\|h_n^{(N)}\|_\infty.
\end{equation}
\item If $\Var(h_n^{(N)}), \Var(h_m^{(N)})<\infty$ and $\E(h_n^{(N)}),\E(h_m^{(N)})=0$, then for all $1\leq m<n\leq k_N$
\begin{align}
&\|\E(h_n^{(N)}|X_m^{(N)})\|_2\leq C_{mix}\theta^{n-m}\|h_n^{(N)}\|_2.\label{Exp-Mixing-L-two} \\
&|\E(h_m^{(N)} h_n^{(N)})|\leq C_{mix}\theta^{n-m}\|h_m^{(N)}\|_2\|h_n^{(N)}\|_2. \label{Exp-Mixing-L-three}
\end{align}
\end{enumerate}
The analogous  statements  hold for Markov chains.\index{mixing}\index{decay of correlations}\index{uniform ellipticity!and decay of correlations}
\end{proposition}
\begin{proof}
We fix $N$ and let $\pi_{n,n+1}:=\pi_{n,n+1}^{(N)}$, $X_n=X_n^{(N)}$, $h_n:=h^{(N)}_n$.
Define for $k\leq n$
$$
w_{n,k}(X_k):=\E(h_n|X_k),
$$
then
$
w_{n,n}(X_n):=\E(h_n|X_n)=\int h_n(X_n,y) \pi_{n,n+1}(X_n,dy)
=\pi_{n,n+1}[h_n(X_n,\cdot)].
$
By the Markov property, $w_{n,n}(X_n)=\E(h_n|X_n, X_{n-1},\ldots,X_1)$, and this allows us to write $\pi_{n-1,n}(w_{n,n})(X_{n-1})\equiv\E(w_{n,n}(X_n)|X_{n-1})=\E(\E(h_n|X_n,\ldots,X_1)|X_{n-1}))=\E(h_n|X_{n-1})$. So
$
\pi_{n-1,n}(w_{n,n})(X_{n-1})=w_{n,n-1}(X_{n-1}).
$

Applying the Markov operator $\pi_{n-2,n-1}$ on both sides gives in a similar way
$
(\pi_{n-2,n-1}\circ\pi_{n-1,n})(w_{n,n})(X_{n-2})=w_{n,n-2}(X_{n-2}).
$

Continuing in this way we arrive eventually to the identity
$$
{w_{n,m}(X_m)}:=\E(h_n|X_m)=(\pi_{m,m+1}\circ\cdots\circ \pi_{n-1,n})(w_{n,n})(X_m).
$$
By the previous lemmas
$
\Osc[{ w_{n,m}}]\leq (1-\epsilon_0)^{\lfloor \frac{n-m}{2}\rfloor}\Osc[w_{n,n}].
$

Notice that for every {bounded measurable} function $v$,
$\|v\|_\infty\leq |\E(v)|+\Osc(v).$
Since by assumption $\E({ w_{n,m}(X_m)})=\E(h_n)=0$,
$$
\| {w_{n,m} (X_m)}\|_\infty\leq  (1-\epsilon_0)^{\lfloor \frac{n-m}{2}\rfloor}\Osc[w_{n,n}].
$$
 $\Osc[w_{n,n}]\leq 2\|w_{n,n}\|_\infty\leq 2\|h_n\|_\infty$, and part 1 follows.

Part 2 is proved in a similar way, using Lemma \ref{Lemma-tempest}(f).\qed
\end{proof}

\subsection{Hitting probabilities and bridge probabilities}\label{Section-Bridge}
Throughout this section, let $\mathsf X$ be an inhomogeneous Markov array with row lengths $k_N$, and data $\fS^{(N)}_n$, $\pi^{(N)}_{n,n+1}$, $\pi^{(N)}$.
Suppose $\mathsf X$ is uniformly elliptic: $$\pi^{(N)}_{n,n+1}(x,dy)=p_n^{(N)}(x,y)\mu_{n+1}(dy)$$ where $0\leq p_n^{(N)}\leq 1/\epsilon_0$ and
$\int_{\fS_{n+1}} p_{n}^{(N)}(x,y) p_{n+1}^{(N)}(y,z) \mu_{n+1}(dy)>\epsilon_0$.

The following proposition estimates $\Prob(X^{(N)}_n\in E)$ in terms of  $\mu_n^{(N)}$:

\begin{proposition}\label{Proposition-nu}
Under the above assumptions,  for every $3\leq n\leq k_N+1<\infty$ and every Borel set $E\subset\fS_N^{(N)}$,
$
\epsilon_0\leq \frac{\Prob(X_n^{(N)}\in E)}{\mu_n^{(N)}(E)}\leq \epsilon_0^{-1}.
$
Similarly for Markov chains.
\end{proposition}
\begin{proof}
We fix a row $N$, and drop the superscripts $^{(N)}$. Define a probability measure on $\fS_n$ by $P_n(E)=\Prob(X_n\in E)$, then for every $1\leq n<k_N$, for every bounded measurable $\vf:\fS_{n+2}\to\R$,
\begin{align*}
&\int\vf dP_{n+2}=\E(\vf(X_{n+2}))=\E\biggl(\E\biggl(\E\bigl(\vf(X_{n+2})\big|X_{n+1},X_n\bigr)\bigg|X_n\biggr)\biggr)\\
&=\E\bigl(\E\bigl(\E\bigl(\vf(X_{n+2})\big|X_{n+1}\bigr)\big|X_n\bigr)\bigr)\ \ (\because \text{Markov property})\\
&=\int\!\!\!\int\!\!\! \int \vf(z)\, \pi_{n+1,n+2}(y,dz)\, \pi_{n,n+1}(x,dy) P_n(dx)\\
&=\int\!\!\!\int\!\!\! \int \vf(z)\, p_{n+1}(y,z)p_{n}(x,y) \mu_{n+2}(dz) \mu_{n+1}(dy) P_n(dx)
\end{align*}
\begin{align*}
&=\int \vf(z) \left[\int\!\!\! \left(\int p_{n+1}(y,z)p_{n}(x,y)  \mu_{n+1}(dy) \right)P_n(dx)\right]\mu_{n+2}(dz)
\end{align*}
The quantity in the square brackets is bounded below by $\epsilon_0$ and bounded above by $\epsilon_0^{-1}$.
So the measures $P_{n+2},\mu_{n+2}$ are equivalent, and $\epsilon_0\leq \frac{dP_{n+2}}{d\mu_{n+2}}\leq\epsilon_0^{-1}$.\qed
\end{proof}

Notice that in checking the uniform ellipticity condition, we are free to  modify $\mu_n^{(N)}$ by a density bounded away form zero and infinity. Thus, proposition \ref{Proposition-nu} allows us to assume without loss of generality that $\mu_n^{(N)}(E)=\Prob(X^{(N)}_n\in E)$ for $3\leq n\leq k_N$.

\medskip

The ellipticity property implies that  for all $x\in\fS_n^{(N)}, z\in\fS_{n+2}^{(N)}$,
$$
Z_n^{(N)}(x,z):=\int_{\fS_{n+1}}p_n^{(N)}(x,y)p_{n+1}^{(N)}(y,z)\mu_{n+1}^{(N)}(dy)\neq 0.
$$
This allows us to make the following definition: The {\bf bridge distribution}\index{bridge distribution} of $X^{(N)}_{n+1}$ given that $X_n^{(N)}=x$ and $X_{n+2}^{(N)}=z$ is  the measure on $\fS_{n+1}^{(N)}$  which assigns to a Borel set $E\subset\fS^{(N)}_{n+1}$ the probability
\begin{equation}\label{bridge-distribution}
\Prob\left(\ E\ \ \bigg|\begin{array}{l}
X_n^{(N)}=x\\
X_{n+2}^{(N)}=z
\end{array}
\right):=\frac{1}{Z_n^{(N)}(x,z)}\int_E p_n^{(N)}(x,y) p_{n+1}^{(N)}(y,z)\mu_{n+1}^{(N)}(dy).
\end{equation}
The definition makes sense because $Z_n^{(N)}(x,z)\neq 0$.
The following lemma explains why the formula \eqref{bridge-distribution} is reasonable:
\begin{lemma}
Let $\psi_E(x,z):=$right hand side of \eqref{bridge-distribution},  then
$$
\psi_E(X_n^{(N)},X_{n+2}^{(N)})=\Prob\left(X_{n+1}^{(N)}\in E\bigg|
X_n^{(N)},
X_{n+2}^{(N)}
\right) \ \ \text{$\Prob$-almost everywhere.}
$$
\end{lemma}
\noindent
We omit the proof, which is routine.
The lemma does {not} ``prove"  \eqref{bridge-distribution}:  Conditional probabilities are only defined almost everywhere, and are by their very nature non--canonical. But \eqref{bridge-distribution} makes sense everywhere. It is a definition, not a theorem.

\section{Structure constants}\label{Section-Structure-Constants}
Throughout this section we assume that $\mathsf f$ is an additive functional on a uniformly elliptic Markov array  $\mathsf X$ with row lengths $k_N+1$, state spaces $\fS_n^{(N)}$, and transition probabilities as in the ellipticity condition: $
\pi_{n,n+1}^{(N)}(x,dy)=p_n(x,y)\mu_n(dy)$,  where $\mu_n^{(N)}(E)=\Prob(X_n^{(N)}\in E).
$
See \S\ref{Section-Bridge}
why we may assume this on $\mu_n^{(N)}$.

\subsection{Hexagons, balance, and structure constants}\label{Section-Hexagons}

A {\bf Level $N$ hexagon at position $3\leq n\leq k_N$}\index{hexagon} is a configuration
$$
P_n^{(N)}:=\left(x_{n-2}; \begin{array}{c}x_{n-1}\\ y_{n-1}\end{array}; \begin{array}{c}x_{n}\\ y_{n}\end{array}; y_{n+1}\right)
$$
where $x_i,y_i\in\mathfrak S_i^{(N)}$.
A hexagon is called {\bf admissible}\index{hexagon!admissible} if
\begin{align*}
&p_{n-2}^{(N)}(x_{n-2},x_{n-1})p_{n-1}^{(N)}(x_{n-1},x_{n})p_{n}^{(N)}(x_{n},y_{n+1})\neq 0\\
&p_{n-2}^{(N)}(x_{n-2},y_{n-1})p_{n-1}^{(N)}(y_{n-1},y_{n})p_{n}^{(N)}(y_{n},y_{n+1})\neq 0
\end{align*}
Admissible hexagons exist because of uniform ellipticity.

The  space of level $N$ admissible hexagons at position $n$  will be denoted by
$
\mathrm{Hex}(N,n).
$

One can put a natural probability measure on $\mathrm{Hex}(N,n)$
by taking $\{Y^{(N)}_n\}$ to be an independent copy of $\{X^{(N)}_n\}$, and looking at the distribution of
$$
\left(X_{n-2}^{(N)}; \begin{array}{c}X_{n-1}^{(N)}\\ Y_{n-1}^{(N)}\end{array}; \begin{array}{c}X_{n}^{(N)}\\ Y_{n}^{(N)}\end{array}; Y_{n+1}^{(N)}\right) \text{ conditioned on }\begin{array}{l}
X^{(N)}_{n-2}=Y^{(N)}_{n-2}\\
X^{(N)}_{n+1}=Y^{(N)}_{n+1}.
\end{array}
$$
Writing the measure explicitly is possible, but cumbersome.\index{hexagon!distribution} It is better to think of it  as the result of  the following sampling procedure for $\left({ x_{n-2}; \begin{array}{c}x_{n-1}\\ y_{n-1}\end{array}; \begin{array}{c}x_{n}\\ y_{n}\end{array}; y_{n+1}}\right)$:
\begin{enumerate}[$\circ$]
\item $(x_{n-2},x_{n-1})$ is sampled from  the distribution of $(X^{(N)}_{n-2}, X^{(N)}_{n-1})$;
\item $(y_{n},y_{n+1})$ is sampled from the distribution of $(Y^{(N)}_n, Y^{(N)}_{n+1})$ (so it is independent of $(x_n,x_{n+1})$);
\item $x_n$ and $y_{n-1}$ are conditionally independent given the previous choices, and are sampled using the
bridge distributions
$$\Prob(x_{n}\in E|x_{n-1},y_{n+1})=\Prob\left(X^{(N)}_{n}\in E\bigg|
 \begin{array}{l}
X^{(N)}_{n-1}=x_{n-1}\\
X^{(N)}_{n+1}=y_{n+1}
\end{array}\right)
$$
$$\Prob(y_{n-1}\in E|x_{n-2}, y_{n})=\Prob\left(Y^{(N)}_{n-1}\in E\bigg|
\begin{array}{l}
Y^{(N)}_{n-2}=x_{n-2}\\
Y^{(N)}_{n}=y_{n}
\end{array}\right).
$$
\end{enumerate}
We call the resulting measure the {\bf hexagon measure}\index{hexagon!measure} on $\mathrm{Hex}(N,n)$.

The {\bf balance}\index{balance}\index{hexagon!balance} of a hexagon
$
P_n^{(N)}:=\left(x_{n-2}; \begin{array}{c}x_{n-1}\\ y_{n-1}\end{array}; \begin{array}{c}x_{n}\\ y_{n}\end{array}; y_{n+1}\right)
$
 is
\begin{equation}\label{balance}
\begin{aligned}
&\Gamma(P_n^{(N)}):= f_{n-2}^{(N)}(x_{n-2},x_{n-1})+f_{n-1}^{(N)}(x_{n-1},x_{n})+f_{n}^{(N)}(x_{n},y_{n+1})\\
&\hspace{1.5cm} -f_{n-2}^{(N)}(x_{n-2},y_{n-1})-f_{n-1}^{(N)}(y_{n-1},y_{n})-f_n^{(N)}(y_{n},y_{n+1}).
\end{aligned}
\end{equation}

\begin{definition}The {\em structure constants} of $\mathsf f=\{f^{(N)}_n\}$ are
\begin{equation}\label{Structure-Constants}\index{structure constants!definition}
\begin{aligned}
&u_n^{(N)}:=u_n^{(N)}(\mathsf f):=\E\bigl((\Gamma(P_n^{(N)})^2\bigr)^{1/2} \text{(expectation on $\mathrm{Hex}(N,n)$)}\\
&d_n^{(N)}(\xi):=d_n^{(N)}(\xi,\mathsf f):=\E(|e^{i\xi\Gamma(P_n^{(N)})}-1|^2)^{1/2}  \text{(expectation on $\mathrm{Hex}(N,n)$)}\\
&U_N:=U_N(\mathsf f):=\sum_{n=3}^{k_N} (u_n^{(N)})^2\ , \ D_N(\xi):=\sum_{n=3}^{k_N} d_n^{(N)}(\xi)^2.
\end{aligned}
\end{equation}
\end{definition}
\noindent
If $\mathsf X$ is a Markov chain, we write $u_n=u_n^{(N)}$, $d_n(\xi)=d_n^{(N)}(\xi)$.

The significance of the structure constants will become clear in later chapters. At this point we can only hint and say that the behavior of $U_N$ determines if $\Var(S_N)\to\infty$, and the behavior of $D_N(\xi)$ determines ``how close"  $\textsf f$ is to an additive functional whose values all belong to the lattice $(2\pi/\xi)\Z$.

\begin{lemma}\label{Lemma-Sum}
Suppose $\mathsf f,\mathsf g$ are two additive functionals of on a uniformly elliptic Markov array $\mathsf{X}$, then
\begin{enumerate}[(a)]
\item $d_n^{(N)}(\xi+\eta,\mathsf f)^2\leq 8(d_n^{(N)}(\xi,\mathsf f)^2+d_n^{(N)}(\eta,\mathsf f)^2)$;
\item $d_n^{(N)}(\xi,\mathsf f+\mathsf g)^2\leq 8(d_n^{(N)}(\xi,\mathsf f)^2+d_n^{(N)}(\xi,\mathsf g)^2)$;
\item $d_n^{(N)}(\xi,\mathsf f)\leq |\xi|u_n^{(N)}(\mathsf f)$;
\item $u_n^{(N)}(\mathsf f+\mathsf g)^2\leq 2[u_n^{(N)}(\mathsf f)^2+u_n^{(N)}(\mathsf g)^2]$.
\end{enumerate}
\end{lemma}
\begin{proof} For any $z,w\in \mathbb C$ such that $|z|,|w|\leq 2$, we have
\footnote{$(zw+z+w)^2=z^2 w^2+z^2+w^2+2(z^2 w+zw^2+zw)$, and $|z^2 w^2|\leq 4|zw|\leq 2|z|^2+2|w|^2$, $|z^2 w|\leq 2|z|^2$, $2|z w|\leq |z|^2+|w|^2$,  $|zw^2|\leq 2|w|^2$.}
$$|zw+z+w|^2\leq  8(|z|^2+|w|^2).$$
So if $P$ is a level $N$ hexagon $P$ at position $n$, and  $\xi_P:=\xi\Gamma(P)$, $\eta_P:=\eta\Gamma(P)$, then
\begin{align}
&|e^{i(\xi_P+\eta_P)}-1|^2=|(e^{i\xi_P}-1)(e^{i\eta_P}-1)+(e^{i\xi_P}-1)+(e^{i\eta_P}-1)|^2 \notag\\
&\leq 8\bigl(|e^{i\xi_P}-1|^2+|e^{i\eta_P}-1|^2\bigr).
\end{align}
Part (a) follows by integrating over all $P\in\mathrm{Hex}(n,N)$. Part (b) has a similar proof which we omit.
Part (c) is follows from the  inequality $|e^{i\theta}-1|^2=4\sin^2 \frac{\theta}{2}\leq |\theta|^2$.
Part (d) follows from Minkowski's inequality and  $|ab|\leq \frac{1}{2}(a^2+b^2)$.\qed
\end{proof}

\begin{example}[Gradients] Suppose $f_n(x,y)=a_{n+1}(y)-a_n(x)+c_n$ for all $n$, then the balance of each hexagon is zero and $u_n, d_n(\xi)$ are all zero. For a converse statement, see \S\ref{Section-Gradient-Lemma}.

Suppose $f_n(x,y)=a_{n+1}(y)-a_n(x)+c_n\mod \frac{2\pi}{\xi}\Z$ for all $n$. Then $e^{i\xi\Gamma(P)}=1$ for all hexagons $P$, and $d_n(\xi)$ are all zero. For a converse statement, see \S\ref{Section-Reduction-Lemmas}.
\end{example}

\begin{example}[Sums of independent random variables]
Let $
S_N=X_1+\cdots+X_N.
$
where $X_i$ are independent real valued random variables with non-zero variance. Let us see what $u_n$ and $d_n(\xi)$ measure in this case.

\end{example}
\begin{proposition}
$\DS u_n^2=2\bigl(\Var(X_{n-1})+\Var(X_n)\bigr)$
and
$\DS \sum_{n=3}^N u_n^2\asymp \Var(S_N)$ (i.e   $\exists N_0$ such that the ratio of the two sides is uniformly bounded for $N\geq N_0$).
\end{proposition}

\begin{proof}
Let $\{Y_n\}$ be an independent copy of $\{X_n\}$, and let $X_i^\ast:=X_i-Y_i$ (the symmetrization of $X_i$). A simple calculation shows that  the balance of a position $n$ hexagon is equal in distribution to
$
X_{n-1}^\ast+X_n^\ast
$
.
Clearly $\E[X_i^\ast]=0$ and $\E[( X_i^\ast)^2]=2\Var(X_i)$. Consequently,
\begin{align*}
u_n^2(\xi)&=\E[(X^\ast_{n-1})^2+(X^\ast_{n})^2]=2\Var(X_{n-1})+2\Var(X_{n}).
\end{align*}
Summing over $n$ we obtain $\sum_{n=3}^N u_n^2\asymp \Var(S_N)$.\qed\end{proof}

We remark that the proposition also holds for Markov arrays satisfying the {\em one-step} ellipticity condition (see \S \ref{Section-Weaker-(E)(c)}).

\medskip

Next we relate $d_n^2(\xi)$ to  the distance of $X_i$ from a coset of $\frac{2\pi}{\xi}\Z$.
The distance of a random variable $X$ from  a coset $\frac{2\pi}{\xi}\Z$ is measured by the following quantity:
$$
\mathfrak D(X,\xi):=\min_{\theta\in\R}\E\left[\dist^2\left(X,\theta+\frac{2\pi}{\xi}\Z\right)\right]^{1/2}.
$$
The minimum exists because the quantity we are minimizing is a periodic and continuous function of $\theta$.

\begin{proposition}\label{Prop-Structure-Const-IND}
For every $\xi\neq 0$ $d_n(\xi)=0$ iff $X_i\in$ coset of $\frac{2\pi}{\xi}\Z$ a.s. $(i=n-1,n)$. In addition,
 there exists $C(\xi)>1$ such that if $d_n(\xi)\neq 0$ then
$$
C(\xi)^{-1}\leq \frac{d_n^2(\xi)}{\mathfrak D(X_{n-1},\xi)^2+\mathfrak D(X_{n},\xi)^2}\leq
C(\xi).
$$
\end{proposition}

\begin{proof}
Choose $\theta_i\in [0,\frac{2\pi}{\xi}]$ s.t. $\mathfrak D(X_i,\xi)=\E[\dist^2(X_i,\theta_i+\frac{2\pi}{\xi}\Z)]$. There is no loss of generality in assuming that $\theta_i=0$, because the structure constants of $f_i(x)=x$ and $g_i(x)=x-\theta_i$ are the same. Henceforth we assume that
\begin{equation}\label{theta=zero}
\mathfrak D(X_i,\xi)=\E[\dist^2(X_i,\frac{2\pi}{\xi}\Z)].
\end{equation}

As in the proof of the previous proposition,
 the balance of a position $n$ hexagon is equal in distribution to
$
X_{n-1}^\ast+X_n^\ast,
$
where  $X_i^\ast:=X_i-Y_i$ and $\{Y_i\}$ is an independent copy of $\{X_i\}$.
So  $d_n^2(\xi)=\E(|e^{i(X_{n-1}^\ast+X_n^\ast)}-1|^2)$.

We need the following elementary facts:
\begin{align}
&|e^{i(x+y)}-1|^2=4\sin^2\tfrac{x+y}{2}=4(\sin\tfrac{x}{2}\cos\tfrac{y}{2}+\sin\tfrac{y}{2}\cos\tfrac{x}{2})^2\ \ \ (x,y\in\R)\label{id1}\\
&\tfrac{4}{\pi^2}\dist^2(t,\pi\Z)\leq \sin^2 t\leq \dist^2(t,\pi\Z)\ \ \   (t\in\R)\label{ineq1}\\
&\Prob[X_i^\ast\in [0,\tfrac{\pi}{2\xi}]+\tfrac{2\pi}{\xi}\Z]\geq \frac{1}{4}\ \ \ (i\geq 1)\label{1/4}
\end{align}
\eqref{id1} is trivial; \eqref{ineq1} is because of the inequality $2t/\pi\leq \sin t \leq t$ on $[0,\frac{\pi}{2}]$, which the reader may verify by drawing the graphs.
To see \eqref{1/4} note that
$
\R=\left([0,\tfrac{\pi}{2\xi}]+\tfrac{\pi}{\xi}\Z\right)\uplus \left([0,\tfrac{\pi}{2\xi}]+\tfrac{\pi}{2\xi}+\tfrac{\pi}{\xi}\Z\right),
$
and therefore there exists $k=0,1$ such that $\Prob[X_i\in [0,\tfrac{\pi}{2\xi}]+\tfrac{k\pi}{2\xi}+\tfrac{\pi}{\xi}\Z]\geq \frac{1}{2}$. Since $Y_i$ is an independent copy of $X_i$, $\Prob[X_i,Y_i\in [0,\tfrac{\pi}{2\xi}]+\tfrac{k\pi}{2\xi}+\tfrac{\pi}{\xi}\Z]\geq \frac{1}{4}$. This event is a subset of
$\bigl[X_i^\ast\in [0,\frac{\pi}{2\xi}]+\frac{2\pi}{\xi}\Z\bigr]$.

Returning to the identity $d_n^2(\xi)=\E(|e^{i(X_{n-1}^\ast+X_n^\ast)}-1|^2)$, we see that by \eqref{id1}
\begin{align}
&d_n^2(\xi)=\E(|e^{i\xi (X^\ast_{n-1}+X^\ast_{n})}-1|^2)\notag\\
&=4\E\left(\sin^2\tfrac{\xi X^\ast_{n-1}}{2}\cos^2\tfrac{\xi X^\ast_{n}}{2}+\sin^2\tfrac{\xi X^\ast_{n}}{2}\cos^2\tfrac{\xi X^\ast_{n-1}}{2}+\tfrac{1}{2}\sin(\xi X^\ast_{n-1})\sin(\xi X^\ast_n)\right)\notag\\
&=4\E\left(\sin^2\tfrac{\xi X^\ast_{n-1}}{2}\right)\E\left(\cos^2\tfrac{\xi X^\ast_{n}}{2}\right)+4\E\left(\sin^2\tfrac{\xi X^\ast_{n}}{2}\right)\E\left(\cos^2\tfrac{\xi X^\ast_{n-1}}{2}\right)
\end{align}
where we used the symmetry of the distribution of $X_i^\ast$ to see that $\E[\sin(\xi X_i^\ast)]=0$.
By \eqref{1/4}, $\E\left(\cos^2\tfrac{\xi X^\ast_{i}}{2}\right)\geq \cos^2(\frac{\pi}{4})\Prob[X_i^\ast\in [0,\frac{\pi}{2\xi}]+\frac{2\pi}{\xi}\Z]\geq \frac{1}{8}$, and therefore there exists $C_n\in [\frac{1}{8},4]$ such that
\begin{equation}\label{boris-johnson}
d_n^2(\xi)=C_n\left[\E\left(\sin^2\tfrac{\xi X^\ast_{n-1}}{2}\right)+\E\left(\sin^2\tfrac{\xi X^\ast_{n}}{2}\right)\right].
\end{equation}
It remains to bound $\E\left(\sin^2\tfrac{\xi X^\ast_{n-1}}{2}\right)$ in terms of $\mathfrak D(X_i,\xi)$.

Recall that $X_i^\ast=X_i-Y_i$ where $Y_i$ is an  independent copy of $X_i$, and use \eqref{id1} and independence to find that
\begin{align*}
&\E\left(\sin^2\tfrac{\xi X^\ast_{i}}{2}\right)=\E\left[\left(
\sin\frac{\xi X_i}{2}\cos\frac{\xi Y_i}{2}-\sin\frac{\xi Y_i}{2}\cos\frac{\xi X_i}{2}
\right)^2\right]\\
&=2\E(\sin^2\tfrac{\xi X_i}{2})\E(\cos^2\tfrac{\xi X_i}{2})-\frac{1}{2}\E(\sin(\xi X_i))^2\leq 2\E(\sin^2\tfrac{\xi X_i}{2})\\
&\leq 2 \E(\dist^2(\tfrac{\xi X_i}{2},\pi\Z))\equiv \tfrac{\xi^2}{2}\E(\dist^2( X_i,\tfrac{2\pi}{\xi}\Z))=\frac{\xi^2}{2}\mathfrak D(X_i,\xi),\text{ by \eqref{theta=zero},\eqref{ineq1}.}
\end{align*}
Next  by \eqref{ineq1} and the definition of $\mathfrak D(X_i,\xi)$,
\begin{align*}
&\E\left(\sin^2\tfrac{\xi X^\ast_{i}}{2}\right)\geq \tfrac{4}{\pi^2}\E\left(\dist^2(\tfrac{\xi X^\ast_i}{2},\pi\Z)\right)=\frac{\xi^2}{\pi^2}\E\left(\dist^2(X_i-Y_i,\frac{2\pi}{\xi}\Z)\right)\\
&=\frac{\xi^2}{\pi^2}\E_{Y_i}\left[\E_{X_i}\left(\dist^2(X_i,Y_i+\frac{2\pi}{\xi}\Z)\right)\right]\geq
\frac{\xi^2}{\pi^2}\E_{Y_i}\left[\mathfrak D(X_i,\xi)\right]=\frac{\xi^2}{\pi^2}\mathfrak D(X_i,\xi).
\end{align*}
The proposition follows from \eqref{boris-johnson}.
\qed\end{proof}

\subsection{The ladder process}\label{Section-Ladder}
{\em The material of this section is  needed
 for the proofs of the gradient lemma and the reduction lemma in chapters \ref{Chapter-Variance} and \ref{Chapter-Irreducibility}, but will not be used elsewhere.}

\medskip
Suppose $\mathsf X=\{X^{(N)}_i\}$ is a Markov array with row lengths $k_N+1$, state spaces $\fS^{(N)}_n$, and transition probabilities $\pi_{n,n+1}^{(N)}(x,dy)$. Let  $\mu^{(N)}_n(E):=\Prob(X^{(N)}_n\in E)$.  Suppose $\mathsf X$ is uniformly elliptic. In particular,
$$
\pi_{n,n+1}^{(N)}(x,dy)=p_n^{(N)}(x,y)\mu_{n+1}^{(N)}(dy),
$$
with $p_n^{(N)}(x,y)$ as in the uniform ellipticity condition.

  We would like to define a new Markov array $\mathsf L$, called the {\bf Ladder process}\index{ladder process}, with the following structure (figure \ref{Figure-Ladder}):
\begin{enumerate}[(a)]
\item Each row has  entries
$
\un{L}_n^{(N)}=(Z_{n-2}^{(N)},Y_{n-1}^{(N)},X_n^{(N)})\ \ \ \ (3\leq n\leq k_N+1),
$
\item  $\{Z^{(N)}_i\}$ is an independent copy of $\mathsf X$,
\item  $Y_{n-1}^{(N)}\in \fS_{n-1}^{(N)}$ are independent given $\{X_i^{(N)}\}, \{Z_i^{(N)}\}$, and
\item
$
\DS\Prob\left(Y^{(N)}_{n-1}\in E\bigg|\{X_i^{(N)}\}=\{x_i\}, \{Z_i^{(N)}\}=\{z_i\}\right)=\Prob\left(X^{(N)}_{n-1}\in E\bigg|\begin{array}{l}
X^{(N)}_{n-2}=z_{n-2}\\
X^{(N)}_{n}=x_{n}
\end{array}
\right),
$
\end{enumerate}
see the discussion of bridge probabilities above.

\begin{figure}[htbp]
\label{Figure-Ladder}
\begin{center}
\includegraphics[width=5cm, angle=0]{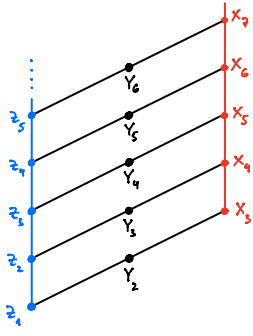}
\caption{The ladder process.
$\{Z_i^{(N)}\}$, $\{X_i^{(N)}\}$ are independent copies.  $Y_n^{(N)}$
are conditionally independent given $\{X_i^{(N)}\}, \{Z_i^{(N)}\}$.}
\end{center}
\end{figure}

 Let $\un{L}_n^{(N)}=(z_{n-2},y_{n-1},x_n)$. Define the probability measures
\begin{align*}
&m_{n}^{(N)}(d\un{L}_{n}^{(N)})
:=
\frac{p^{(N)}_{n-2}(z_{n-2},y_{n-1})p_{n-1}^{(N)}(y_{n-1},x_{n})}{\int_{\fS_{n-1}^{(N)}} p_{n-2}^{(N)}(z_{n-2},\eta)p_{n-1}^{(N)}(\eta,x_{n})\mu_{n-1}^{(N)}(d\eta)}
\mu_{n-2}^{(N)}(dz_{n-2})\mu_{n-1}^{(N)}(dy_{n-1})\mu_{n}^{(N)}(dx_{n}).
\end{align*}

\begin{lemma}\label{Lemma-Ladder}
$\mathsf L$ exists, is Markov, and is uniformly  elliptic
 with ellipticity constant
 $\epsilon_0^2$  (with respect to the background measure $m_n$),
 where $\epsilon_0$ is the ellipticity constant of $\mathsf X$.
For every $N$,
 \begin{enumerate}[(1)]
 \item  $\{X_n^{(N)}\}_{n=3}^{k_N+1}$ , $\{Z_n^{(N)}\}_{n=1}^{k_N-1}$ are independent, and distributed like the corresponding pieces of the $N$-th rows of $\mathsf X$.
 \item $Y_n^{(N)}$ are conditionally independent given $\{X_i^{(N)}\}$, $\{Z_i^{(N)}\}$.
 \item  $P_n^{(N)}:=\left(Z_{n-2}^{(N)},{\begin{array}{l}
 Z_{n-1}^{(N)}\\
 Y_{n-1}^{(N)}
 \end{array}}
 { \begin{array}{l}
 Y_n^{(N)}\\
 X_{n}^{(N)}
 \end{array}},X_{n+1}^{(N)}\right)$ is distributed like the  level $N$, position $n$,  random hexagon.
 \end{enumerate}
\end{lemma}
\begin{proof}
Let
$\Prob\left(
dy_n\bigg|
\begin{array}{l}
X^{(N)}_{n-1}=z_{n-1}\\
X^{(N)}_{n+1}=x_{n+1}
\end{array}
\right)$ denote the bridge measure on $\fS^{(N)}_n$ with boundary conditions $X^{(N)}_{n-1}=z_{n-1}, X^{(N)}_{n+1}=x_{n+1}$. Define the Markov array $\mathsf L$ with
\begin{enumerate}[$\circ$]
\item Rows $\un{L}^{(N)}_n=(z_{n-2},y_{n-1},x_n)$ ($3\leq n\leq k_N+1$, $N\geq 1$)
\item State spaces:
$
\un{\fS}_n^{(N)}:=\fS_{n-2}^{(N)}\times\fS_{n-1}^{(N)}\times\fS_n^{(N)}
$ $(3\leq n\leq k_N+1)$.
\item Initial distribution:  $\pi^{(N)}(dz_1,dy_2,dx_3)
=\!\!\!\!\!\!\int\limits_{\fS_{1}^{(N)}\times \fS_3^{(N)}}\!\!\!\!\!\!
\mu_{1}^{(N)}(dz)\mu_3^{(N)}(dx)
\Prob\left(
dy\bigg|
\begin{array}{l}
X^{(N)}_{1}=z\\
X^{(N)}_3=x
\end{array}
\right)
$
\item Transition probabilities $\pi_n^{(N)}((z_{n-2},y_{n-1},x_n),E_{n-1}\times E_n\times E_{n+1})=$
\begin{align*}
&=\int_{E_{n-1}\times E_n\times E_{n+1}} p_{n-2}^{(N)}(z_{n-2}, z_{n-1})p_n^{(N)}(x_n,x_{n+1})
\Prob\left(
dy_n\bigg|
\begin{array}{l}
X^{(N)}_{n-1}=z_{n-1}\\
X^{(N)}_{n+1}=x_{n+1}
\end{array}
\right).
\end{align*}
\end{enumerate}
(We evolve $z_{n-2}\to z_{n-1}$ and $x_n\to x_{n+1}$ independently according to $\pi_{n-2}^{(N)}(z_{n-2},dz)$, $\pi_{n}^{(N)}(x_n,dx)$, and then sample $y_n$ using the relevant bridge distribution.)

It is routine to check that $\mathsf L$ has the structure described at the beginning of the section, and that it satisfies the properties listed in the lemma.

Here for example is the proof of uniform ellipticity. In what follows we fix $N$, suppose $x_i,y_i,z_i\in \fS_i$, and write $p_n^{(N)}=p$  whenever the subscript is clear from the variables.

Then $\pi_n^{(N)}(\un{L}_n,d\un{L}_{n+1})
=P(\un{L}_n,\un{L}_{n+1})
m_{n+1}(d\un{L}_{n+1})$, where
\begin{align*}
&P(\un{L}_n,\un{L}_{n+1}):=p(z_{n-2},z_{n-1})p(x_n,x_{n+1}).
\end{align*}
If $\mathsf X$ has ellipticity constant $\eps_0$, then
$P(\un{L}_n,\un{L}_{n+1})\leq \eps_0^{-2}$, and
\begin{align*}
&\int P(\un{L}_n, \un{L}_{n+1})P(\un{L}_{n+1}, \un{L}_{n+2})m_{n+1}(d\un{L}_{n+1})\\
&\geq
\iiint p(z_{n-2},z_{n-1})p(x_n,x_{n+1})
p(z_{n-1},z_{n})p(x_{n+1},x_{n+2})\times\\
&\hspace{2cm}\times\frac{p(z_{n-1},y_{n})p(y_{n},x_{n+1})}{\int p(z_{n-1},\eta)p(\eta,x_{n+1})\mu_{n}(d\eta)}\mu_{n-1}(dz_{n-1})\mu_{n}(dy_{n})\mu_{n+1}(dx_{n+1})\\
&=\iint p(z_{n-2},z_{n-1})p(x_n,x_{n+1})
p(z_{n-1},z_{n})p(x_{n+1},x_{n+2})\mu_{n-1}(dz_{n-1})\mu_{n+1}(dx_{n+1})\\
&=\int p(z_{n-2},z_{n-1})
p(z_{n-1},z_{n})\mu_{n-1}(dz_{n-1})\int p(x_n,x_{n+1})p(x_{n+1},x_{n+2})\mu_{n+1}(dx_{n+1})\\
&\geq \eps_0^{-2}.
\end{align*}
So the ladder process is uniformly elliptic with ellipticity constant $\eps_0^2$.
\qed
\end{proof}

\subsection{$\gamma$-step ellipticity conditions}
\label{Section-Weaker-(E)(c)}

We  mention a few possible variants of the uniform ellipticity condition discussed in this chapter.
Suppose $\mathsf X$ is a Markov array with row lengths $k_N+1$ and transition probabilities taking the form $\pi_{n,n+1}^{(N)}(x,dy)=p_n^{(N)}(x,y)\mu_{n+1}^{(N)}(dy)$.

The {\bf one-step ellipticity condition} is that for some $\eps_0>0$, for all $N\geq 1$, $1\leq n\leq k_N$, and for every $x\in\fS_n^{(N)}, y\in \fS_{n+1}^{(N+1)}$,  $$
\eps_0<p_n^{(N)}(x,y)\leq \eps_0^{-1}.
$$
Notice that this  implies that all transitions $x\to y$ have positive probability.

The {\bf $\gamma$-step ellipticity condition} ($\gamma=2,3,\ldots$)
\index{uniform ellipticity!$\gamma$-step} is that for some $\eps_0>0$, for all $N\geq 1, n\leq k_N$,
$$0\leq p_n^{(N)}\leq 1/\eps_0$$
and for all $n\leq k_N-\gamma+1$, and every $x\in\fS_n^{(N)}, z\in\fS_{n+\gamma}^{(N)}$, the iterated integral
$$
\int\limits_{\mathfrak S_{n+1}^{(N)}}\!\!\!\!\cdots\!\!\!\! \int\limits_{\mathfrak S_{n+\gamma}^{(N)}} p_n^{(N)}(x,y_1)\prod_{i=1}^{\gamma-2} p_{n+i}^{(N)}(y_i,y_{i+1}) p_{n+\gamma-1}^{(N)}(y_{\gamma-1},z)\,
\mu_{n+1}(dy_1)\cdots\mu_{n\gamma}(dy_{\gamma-1})
$$
is bigger than $\epsilon_0$ (with the convention that $\DS \prod_{i=1}^{0}:=1$).

The ellipticity condition we use in this work corresponds to $\gamma=2$. This is weaker than the one-step condition, but stronger than the $\gamma$-step condition for $\gamma\geq 3$.

The results of this work could in principle be reproduced assuming only a $\gamma$-step condition with $\gamma\geq 2$. To do this, one needs to replace the space of hexagons by the space of $2(\gamma+1)$-gons
$
\left(x_{n-\gamma}; \begin{array}{c}x_{n-\gamma+1}\\ y_{n-\gamma+1}\end{array}
\cdots
\begin{array}{c}x_{n}\\ y_{n}\end{array}; y_{n+1}\right)
$
with its associated structure constants, and its associated {\em $\gamma$-ladder process} $\un{L}^{(N)}_n=(Z_{n-\gamma-1}^{(N)},Y_{n-\gamma}^{(N)},\ldots,Y_{n-1}^{(N)},X_n^{(N)})$.
Since no new ideas are needed, and since our notation is already heavy enough as it is, we will only treat the case $\gamma=2$ in this work.

\subsection{Uniform ellipticity and  strong mixing conditions}

\noindent
{\em The contents of this section are not used elsewhere in this work.}

\medskip
 Suppose $(\Omega,\mathfs F,\Prob)$ is a probability space, and let  $\mathcal A,\mathcal B$ be two sub $\sigma$-algebras of $\mathcal F$. There
 are several standard   measures for the dependence between $\mathcal A$ and $\mathcal B$:
 \begin{align*}
& \alpha(\mathcal A,\mathcal B):=\sup\{|\Prob(A\cap B)-\Prob(A)\Prob(B)|:A\in\mathcal A, B\in\mathcal B\};\\[1mm]
& \rho(\mathcal A,\mathcal B):=\sup\left\{|\E(fg)-\E(f)\E(g)|:\begin{array}{l}
f\in L^2(\mathcal A), g\in L^2(\mathcal B);\\
 \|f-\E(f)\|_2=1, \|g-\E(g)\|_2=1
 \end{array}\right\};\\[2mm]
& \phi(\mathcal A,\mathcal B):=\sup\left\{\bigl|\Prob(B|A)-\Prob(B)\bigr|:A\in\mathcal A, B\in\mathcal B, \Prob(A)\neq 0\right\};\\[1mm]
& \psi(\mathcal A,\mathcal B):=\sup\left\{\left|\frac{\Prob(A\cap B)}{\Prob(A)\Prob(B)}-1\right|:A\in\mathcal A, B\in\mathcal B\text{ with non-zero probabilities}\right\}.
\end{align*}

If one of these quantities vanishes then they all vanish,  and this happens iff $\Prob(A\cap B)=\Prob(A)\Prob(B)$ for all $A\in\mathcal A$, $B\in\mathcal B$. In this case we say that $\mathcal A,\mathcal B$ are {\bf independent}.\index{independent $\sigma$-algebras} In the dependent case, $\alpha,\rho,\phi,\psi$ can be used to bound the covariance between (certain) $\mathcal A$-measurable and  $\mathcal B$-measurable random variables:

\begin{theorem}
 Suppose $X$ is $\mathcal A$-measurable, $Y$ is $\mathcal B$-measurable, then
\begin{enumerate}[(1)]
\item $|\Cov(X,Y)|\leq 8\alpha(\mathcal A,\mathcal B)^{1-\frac{1}{p}-\frac{1}{q}}\|X\|_p\|Y\|_q$ whenever $p\in (1,\infty]$, $q\in(1,\infty]$, \\[1mm]
$\frac{1}{p}+\frac{1}{q}<1$, $X\in L^p$, $Y\in L^q$.
\item $|\Cov(X,Y)|\leq \rho(\mathcal A,\mathcal B)\|X-\E Y\|_2\|Y-\E Y\|_2$ whenever $X,Y\in L^2$.
\item $|\Cov(X,Y)|\leq 2\phi(\mathcal A,\mathcal B)\|X\|_1\|Y\|_\infty$ whenever $X\in L^1, Y\in L^\infty$.
\item $|\Cov(X,Y)|\leq \psi(\mathcal A,\mathcal B)\|X\|_1\|Y\|_1$ whenever $X\in L^1, Y\in L^\infty$.
\end{enumerate}
\end{theorem}
\noindent
For proof and references, see \cite[vol 1, ch. 3]{Bradley}.

\begin{definition}
Let $\mathsf X:=\{X_n\}_{n\geq 1}$ be a general stochastic process, not necessarily stationary or Markov. Let $\mathcal F_1^n$ denote the $\sigma$-algebra generated by $X_1,\ldots,X_n$, and let $\mathcal F_{m}^\infty$ denote the $\sigma$-algebra generated by $X_k$ for $k\geq m$.
\begin{enumerate}[(1)]
\item $\mathsf X$ is called  {\bf $\alpha$-mixing}\index{mixing conditions!$\alpha$-mixing}, if
$
\alpha(n):=\sup_{k\geq 1} \alpha(\mathcal F_1^k, \mathcal F_{k+n}^\infty)\xrightarrow[n\to\infty]{}0.
$
\item $\mathsf X$ is called  {\bf $\rho$-mixing}\index{mixing conditions!$\rho$-mixing}, if
$
\rho(n):=\sup_{k\geq 1} \rho(\mathcal F_1^k, \mathcal F_{k+n}^\infty)\xrightarrow[n\to\infty]{}0.
$
\item $\mathsf X$ is called  {\bf $\phi$-mixing}\index{mixing conditions!$\phi$-mixing}, if
$
\phi(n):=\sup_{k\geq 1} \phi(\mathcal F_1^k, \mathcal F_{k+n}^\infty)\xrightarrow[n\to\infty]{}0.
$

\item $\mathsf X$ is called  {\bf $\psi$-mixing}\index{mixing conditions!$\psi$-mixing}, if
$
\psi(n):=\sup_{k\geq 1} \psi(\mathcal F_1^k, \mathcal F_{k+n}^\infty)\xrightarrow[n\to\infty]{}0.
$
\end{enumerate}
\end{definition}

\begin{theorem} If $(\Omega,\mathfs F,\Prob)$ is a probability space, and  $\mathcal A,\mathcal B$ are sub-$\sigma$-algebras of $\mathcal F$, then  $\alpha:=\alpha(\mathcal A,\mathcal B)$,
$\rho:=\rho(\mathcal A,\mathcal B)$,
$\phi:=\phi(\mathcal A,\mathcal B)$,
$\psi:=\psi(\mathcal A,\mathcal B)$ satisfy the  inequalities
\begin{equation}\label{measures-of-dependence-ineq}
2\alpha\leq \phi\leq \frac{1}{2}\psi\ ,\ \rho\leq 2\sqrt{\phi}.
\end{equation}
\end{theorem}
For the proof, see \cite[vol 1, Prop. 3.11]{Bradley}). It follows that
$$
\text{$\psi$-mixing}\Rightarrow \text{$\phi$-mixing}\Rightarrow
\text{$\rho$-mixing}\Rightarrow \text{$\alpha$-mixing}.
$$
These implications are strict, see \cite[vol 1 \S5.23]{Bradley}.

\medskip
Let us see what is the connection of $\phi$-mixing to uniform ellipticity.
First we'll show that uniform ellipticity implies exponential $\psi$-mixing, and then we'll give a weak converse of this statement  for finite state Markov chains.

\begin{proposition}\label{Proposition-u-elliptic-is-psi}
Let $\mathsf X$ be a uniformly elliptic Markov chain,  then for every $x\in\fS_1$,   $\mathsf X$ conditioned on $X_1=x$ is $\psi$-mixing. Moreover, $\alpha(n),\rho(n),\phi(n),\psi(n)\xrightarrow[n\to\infty]{}0$ exponentially fast, uniformly in $x$.
\end{proposition}

\begin{proof} We will need the following fact:

\medskip
\noindent
{\sc Claim.\/} {\em
There exists a constant $K$ which only depends on the ellipticity constant of $\mathsf X$ as follows.
For every $x\in\fS_1$, $k\geq 2$, and  for every bounded measurable function $h_k:\fS_k\to\R$, we have the inequality $\|\E_x(h_k(X_k)|X_{k-2})\|_\infty\leq K\E_x(|h_k(X_k)|)$.}

\medskip
\noindent
{\em Proof of the claim.\/}
By the uniform ellipticity of $\mathsf X$, the transition kernels of $\mathsf X$ can be put in the form $\pi_{n,n+1}(x,dy)=p_n(x,y)\mu_{n+1}(dy)$, where  $0\leq p_n\leq \eps_0^{-1}$ and $\int p_n(x,y)p_{n+1}(y,z)\mu_{n+1}(dy)>\eps_0$. In addition,   Prop. \ref{Proposition-nu} tells us that  the Radon-Nikodym derivative of $\mu_{n+1}$ with respect to the measure $\Prob_x(X_{n+1}\in E)$ is almost everywhere in $[\eps_0,\eps_0^{-1}]$.   It follows that for all $\xi$,
\begin{align*}
&|\E_x(h_{k+2}(X_{k+2})|X_k=\xi)|\leq \iint p_k(\xi,y)p_{k+1}(y,z)|h_{k+2}(z)|\mu_{k+1}(dy)\mu_{k+2}(dz)\\
&\leq \eps_0^{-2}\int|h_{k+2}(z)|\mu_{k+2}(dz)\leq \eps_0^{-3}\E_x(|h_{k+2}(X_{k+2})|).
\end{align*}

\medskip
We  now prove the proposition.
Fix $x\in \fS_1$, and let $\psi_x$ denote the $\psi$ measure of dependence for $\mathsf X$ conditioned on $X_1=x$.
Let $\mathcal F_k$ denote the $\sigma$-algebra generated by $X_k$.
 Using the Markov property, it is not difficult to see that
$$
\psi_x(n)=\sup_{k\geq 1}\psi_x(\mathcal F_k,\mathcal F_{k+n}),
$$
see \cite[vol 1, pp. 206--7]{Bradley}.

Suppose now that $n>2$, and fix some $A\in\mathcal F_k, B\in \mathcal F_{k+n}$ with positive $\Prob_x$-measure. Let
$h_k:=1_A$  and $h_{k+n}:=1_B-\Prob_x(B)$. Then
\begin{align*}
&|\Prob_x(A\cap B)-\Prob_x(A)\Prob_x(B)|=|\E_x(h_k h_{k+n})|=|\E_x(\E_x(h_k h_{k+n}|\mathcal F_k))|\\
&=|\E_x(h_k \E_x(h_{k+n}|X_k))|\leq \E_x(|h_k|)\|\E_x(h_{k+n}|X_k)\|_\infty\\
&=  \Prob_x(A)\|\E_x(\E_x(h_{k+n}|X_{k+n-2})|X_k)\|_\infty\\
&\leq \Prob_x(A)\cdot C_{mix}\theta^{n-2}\|\E_x(h_{k+n}|X_{k+n-2})\|_\infty,\text{ by uniform ellipticity and  \eqref{Exp-Mixing-L-infinity}}\\
&\leq \Prob_x(A)\cdot C_{mix}\theta^{n-2}\cdot K\E_x(|h_{k+n}|),\text{ by the claim}\\
&\leq 2KC_{mix}\theta^{n-2}\Prob_x(A)\Prob_x(B).
\end{align*}
Dividing by $\Prob_x(A)\Prob_x(B)$ and passing to the supremum over $A\in\mathcal F_k, B\in\mathcal F_{k+n}$, gives
$
\psi_x(n)\leq 2KC_{mix}\theta^{n-2}.
$

Recall from Proposition \ref{Proposition-Exponential-Mixing} that $C_{mix},\theta$ depend on the ellipticity constant of $\mathsf X$, but not on $x$. So $\psi_x(n)\to 0$ exponentially fast, uniformly in $x$.
By \eqref{measures-of-dependence-ineq}, $\alpha_x(n),\rho_x(n),\phi_x(n)\to 0$ exponentially fast, uniformly in $n$.
\qed
\end{proof}

\begin{proposition}
\label{PrPhiEll}
Let $\mathsf X$ be a Markov chain such that
\begin{enumerate}[(1)]
\item $\exists\kappa>0$ s.t. $\Prob(X_n=x)>\kappa$ for every $n\geq 1$, $x\in \fS_n$ (in particular, $|\fS_n|<1/\kappa$).
\item $\phi(n)\xrightarrow[n\to\infty]{}0$.
\end{enumerate}
Then $\mathsf X$ satisfies the $\gamma$-step ellipticity condition for all $\gamma$ large enough.
\end{proposition}
\begin{proof}
By (1), all state spaces are finite sets.
Define a measure on $\fS_n$ by $\mu_n(E)=\Prob(X_n\in E)$, and let
$
\DS p_{n}(x,y):=\frac{\Prob(X_{n+1}=y|X_n=x)}{\Prob(X_{n+1}=y)}.
$
This is well-defined  by (1), and:
\begin{enumerate}[(a)]
\item By construction, $\pi_{n,n+1}(x,dy)=p_{n}(x,y)\mu_{n+1}(dy)$.
\item By (1),
$p_{n}(x,y)\leq 1/{\Prob(X_{n+1}=y)}\leq \kappa^{-1}$.
\item By (2), for all $\gamma$ large enough, $\phi(\gamma)<\frac{1}{2}\kappa$. For such $\gamma$,
\begin{align*}
&\int\limits_{\mathfrak S_{n+1}}\!\!\!\!\cdots\!\!\!\! \int\limits_{\mathfrak S_{n+\gamma}} p_n(x,y_1)\prod_{i=1}^{\gamma-2} p_{n+i}(y_i,y_{i+1}) p_{n+\gamma-1}(y_{\gamma-1},z)\,
\mu_{n+1}(dy_1)\cdots\mu_{n+\gamma}(dy_{\gamma-1})\\
&=\Prob(X_{n+\gamma}=z|X_n=x)\geq \Prob(X_{n+\gamma}=z)-\phi(\mathcal F_n,\mathcal F_{n+\gamma})\geq \kappa-\phi(\gamma)>\frac{1}{2}\kappa.
\end{align*}
\end{enumerate}
We obtain the $\gamma$-ellipticity condition with ellipticity constant $\frac{1}{2}\kappa$. \qed
\end{proof}

\section{Notes and references}
For a comprehensive treatment of inhomogeneous Markov chains on general state spaces, see Doob's book \cite{Doob}.
The uniform ellipticity condition is one of a plethora of contraction conditions for Markov operators, which were developed over the years as sufficient conditions for results such as Propositions  \ref{Proposition-Exponential-Mixing} and \ref{Proposition-nu}. We mention in particular the works of  Markov \cite{Markov}, Doeblin \cite{Doeblin-Czech,Doeblin-Roum}, Hajnal \cite{Hajnal}, Doob \cite{Doob}, and  Dobrushin \cite{Do} (see also Seneta \cite{Seneta-History-Doeblin} and Sethuraman \& Varadhan \cite{SV}).

The contraction coefficient mentioned in section \ref{Section-Contraction} is also called an ``{ergodicity coefficient}," and it plays a major role in Dobrushin's proof of the CLT for inhomogeneous Markov chains \cite{Do}. Our treatment of contraction coefficients follows closely \cite{SV}. In particular, Lemma \ref{Lemma-tempest} and the proof of part (f) of that lemma is taken from there.

Proposition
\ref{Proposition-nu} is similar in spirit to Doeblin's estimates for the stationary probability vector of a Markov chain satisfying Doeblin's condition in terms of the stochastic matrix of the chain \cite{Doeblin-Czech,Doeblin-Roum}.

For a discussion of the ``change of measure" construction see chapter \ref{Chapter-LDP}.
The quantities $\mathfrak D(X,\xi)$ were introduced by Mukhin for the purpose of studying local limit theorem for sums of independent random variables. See \cite{Mukhin-1991} and references therein.

For a comprehensive account of measures of dependence and mixing conditions, see \cite{Bradley}.

\chapter{Variance growth, center-tightness, and the central limit theorem}\label{Chapter-Variance}

\noindent
{\em In this chapter we analyze the variance of $S_N=f_1(X_1,X_2)+\cdots+f_N(X_N,X_{N+1})$ as $N\to\infty$,  characterize the additive functionals for which $\Var(S_N)\not\to\infty$, and prove Dobrushin's Theorem: If $V_N\to\infty$ then the central limit theorem holds.}

\section{Main results}\label{Section-Tight-Results}
Let $\mathsf X$ be a Markov array with row lengths $k_N+1$, let $\mathsf f$ an additive functional on $\mathsf X$, and define $\DS S_N=\sum_{i=1}^{k_N} f_i^{(N)}(X^{(N)}_i,X^{(N)}_{i+1})$.
\begin{definition}
$\mathsf f$  is called {\bf center-tight}\index{center tightness!definition}\index{additive functional!center tight} if there are constants $m_N$ s.t. for every $\epsilon>0$, there exists $M$ s.t. $\Prob[|S_N-m_N|>M]<\epsilon \text{ for all $N$.}$
\end{definition}
\noindent
Center-tightness is an obstruction\index{obstructions to the LLT!center tightness}\index{center tightness!obstruction to the LLT} to the local limit theorem. We shall see below (Theorem \ref{Theorem-center-tight}) that  $\mathsf f$  is center-tight iff $\Var(S_N)\not\to\infty$. Obviously, in such a situation the right hand side in
$
\Prob[S_N-z_N\in (a,b)]\overset{?}{\sim} \frac{e^{-z^2/2}|a-b|}{\sqrt{2\pi V_N}}
$
can be made bigger than one by choosing $|a-b|$ sufficiently big, and the asymptotic relation fails.
One could hope for a different universal asymptotic behavior, but as the  following class of examples shows, this is hopeless:

\begin{example}{\bf (Non-universality in the LLT for  center-tight functionals):}
\label{ExTightLimDist} \
\end{example}

Let $\mathsf X=\{X_n\}_{n\geq 1}$ be a sequence of identically distributed independent random variables with uniform distribution on $[0,1]$. Choose an {\em arbitrary} sequence of random variables $\{Z_n\}_{n\geq 1}$ taking values in $[0,1]$. By the isomorphism theorem for Lebesgue spaces, there are measurable functions $g_n:[0,1]\to [0,1]$ such that
$$
g_0\equiv 0\ , \ g_n(X_n)=Z_n\text{ in distribution}.
$$
Let $\mathsf f=\{f_n\}_{n\geq 1}$ with
$
f_n(X_n,X_{n+1}):=g_{n+1}(X_{n+1})-g_n(X_n).
$
Then
$
S_N=Z_{N+1}$  in distribution, whence $\Prob(S_N\in (a,b))=\Prob(Z_{N+1}\in (a,b))$ is completely arbitrary.

\medskip
Every Markov array admits center-tight additive functionals.
Here are three constructions which lead to such  examples (in the uniformly bounded, uniformly elliptic case,  {\em all} center-tight additive functional arise this way, see Theorem \ref{Theorem-center-tight} below):

\begin{example}
{\bf (Gradients):} Gradients on  Markov chains are  additive functionals of the form\index{gradient}\index{additive functional!gradient}\index{obstructions to the LLT!gradients}
$$
f_n(x,y):=(\nabla{\mathsf a})_n(x,y):=a_{n+1}(y)-a_n(x).
$$ where $a_n:\fS_n\to\R$ is measurable, and  $\mathsf a=\{a_n\}$ is a.s. uniformly bounded.

Gradients on Markov arrays are defined similarly by the formula
$
f_n^{(N)}(x,y):=a_{n+1}^{(N)}(y)-a_n^{(N)}(x).
$ where $a^{(N)}_n:\fS_n^{(N)}\to\R$ is measurable, and $\mathsf a=\{a_n^{(N)}\}$ is a.s. uniformly bounded. We write $\mathsf f=\nabla\mathsf a$, and say that $\mathsf f$ is the {\bf gradient} of $\mathsf a$ and $\mathsf a$ is the {\bf potential}\index{potential} of $\mathsf f$.\footnote{In the ergodic theoretic literature, $\mathsf f$ is called a {\bf coboundary}\index{coboundary} and $\mathsf a$ is called a {\bf transfer function}\index{transfer function}.}
\end{example}

The gradient of an a.s. uniformly bounded potential is center-tight because if $|\mathsf a|\leq K$, then
$
|S_N|=|a_{k_N+1}^{(N)}(X_{N+1})-a_1^{(N)}(X_1)|\leq 2K
$.

\begin{example}
{\bf (Summable variance):} We say that an additive functional $\mathsf f$ on a Markov chain $\mathsf X$ has {\bf summable variance}\index{summable variance}\index{additive functional!with summable variance}\index{obstructions to the LLT!gradients} if it is a.s. uniformly bounded, and
$$
V_\infty:=\sum_{n=1}^\infty \Var[f_n(X_n,X_{n+1})]<\infty.
$$
The definition of summable variance for additive functionals on arrays is similar, except that now $V_\infty$ is defined by
$
\displaystyle V_\infty:=\sup_N \sum_{n=1}^{k_N} \Var[f_n^{(N)}(X^{(N)}_n,X^{(N)}_{n+1})]<\infty.
$
\end{example}

If $\mathsf X$ is uniformly elliptic and $|\mathsf f|\leq K$ a.s., then summable variance  implies center-tightness.
This follows from Chebyshev's inequality and the following lemma:

\begin{lemma}
\label{LmVarSum}
Let $\mathsf f$ be a uniformly bounded functional of the uniformly elliptic Markov array.
Then
$ V_N\leq \brV_N \left(1+\frac{2 C_{mix}}{1-\theta}\right)$ where
$\DS \brV_N:=\sum_{n=1}^{k_N} \Var(f_n^{(N)}(X_n^{(N)},X_{n+1}^{(N)}))
$, and  $C_{mix}$ and $0<\theta<1$ are as in Prop. \ref{Proposition-Exponential-Mixing}.
\end{lemma}
\begin{proof}
We give the proof for Markov chains (the proof for arrays is identical):
\begin{align*}
&\Var\left(S_N\right)=\sum_{n=1}^N \Var(f_n)+2\sum_{n=1}^{N-1}\sum_{m=n+1}^{N} \Cov(f_n,f_m)\\
&\leq \brV_N+2C_{mix}\sum_{n=1}^{N-1}\sum_{m=n+1}^{N} \theta^{m-n}\sqrt{\Var(f_n)\Var(f_m)},
\text{ with $C_{mix},\theta$ as in \eqref{Exp-Mixing-L-three}}\\
&\leq \brV_N+2C_{mix} \sum_{j=1}^{N-1}\theta^j \sum_{n=1}^{N-j}
\sqrt{\Var(f_n)\Var(f_{n+j})}< \brV_N+\frac{2C_{mix}\brV_N}{1-\theta}
\end{align*}
 by the Cauchy-Schwarz inequality.\qed
\end{proof}

\begin{example}
Suppose $\mathsf X$ is uniformly elliptic. Then every additive functional of the form  $\mathsf f=\mathsf g+\mathsf h$ where $\mathsf g$ is a gradient and $\mathsf h$ has summable variance is center-tight.
\end{example}

We will now state the main results of this chapter. We assume throughout that
\begin{enumerate}
\item[(E)] $\mathsf X=\{X^{(N)}_n\}$ is a {\bf uniformly elliptic} inhomogeneous Markov array with row lengths $k_N+1$, state spaces $\fS^{(N)}_n$, transition probabilities $\pi^{(N)}_{n,n+1}$,  initial distributions $\pi^{(N)}$, and ellipticity constant $\epsilon_0$.

\medskip
\item[(B)] $\mathsf f=\{f^{(N)}_n\}$ is an {\bf a.s. uniformly bounded} additive functional on $\mathsf X$, satisfying the bound  $|\mathsf f|\leq K$ almost surely.
\end{enumerate}
Let $V_N:=\Var(S_N)$, and  $
\displaystyle U_N:=\sum_{n=3}^{k_N} (u^{(N)}_n)^2$
where $u^{(N)}_n$ are as in \eqref{Structure-Constants}.

\begin{theorem}
\label{LmVarCycles}
There are constants  $C_1, C_2>0$ which only depend on $\epsilon_0, K$  s.t. for every uniformly elliptic array with ellipticity constant $\epsilon_0$ and every additive functional
 $\mathsf f$ on $\mathsf X$ s.t. $|\mathsf f|\leq K$ a.s.,
$$
C_1^{-1}U_N-C_2\leq \Var(S_N) \leq C_1 U_N+C_2\ \ \text{ for all }N.
$$
\end{theorem}

\begin{corollary}\label{Theorem-MC-Variance}
\index{Variance estimate}\index{structure constants!and growth of variance}
Suppose $\mathsf X$ is a \underline{Markov chain}. Either $\Var(S_N)\to\infty$ or $\Var(S_N)=O(1)$. Moreover, $\Var(S_N)\asymp \sum\limits_{n=3}^N u_n^2$ where $u_n$ are the structure constants from \eqref{Structure-Constants}.
\end{corollary}
\noindent
(The corollary  is clearly false for arrays.) Returning to arrays, we'll show:

\begin{theorem}\label{Theorem-center-tight}\index{center tightness!characterization}
$\Var(S_N)$ is bounded iff
$\mathsf f$ is center-tight iff
$
\mathsf f=\nabla a+\mathsf h
$
where $\mathsf a$ is a uniformly bounded potential, and $\mathsf h$ has summable variance.
\end{theorem}

\begin{corollary}\label{Corollary-Tight}
\index{structure constants!and center tightness}
$\mathsf f$ is center-tight iff $\sup\limits_N U_N<\infty$.
\end{corollary}

Theorem \ref{LmVarCycles} is a statement on the localization of cancellations.
In general, if the variance of an additive functional of a stochastic process does not tend to infinity, then there must be some strong cancellations in $S_N$. A priori, these cancellations may involve summands located far apart from one another. Theorem \ref{LmVarCycles} says that strong cancellations  must  already occur among three consecutive terms $f_{n-2}^{(N)}+f_{n-1}^{(N)}+f_n^{(N)}$: This is what  $U_N$ measures.

If $f$ depends only on one variable $f_n(x,y)=f_n(x)$, and we have the one-step
ellipticity condition $p_N(x,y)\geq \eps_0$ one can define the ladder process using
quadrilaterals
$$ Q_n^N=\left(\begin{array}{ccc} X_{n-1}^N &
\begin{array}{c} X_n^N \\ Y_n^N \end{array} & Y_{n+1}^N \end{array}\right)
$$
instead of hexagons. As a result $u_n$ is replaced by
\begin{equation}
\label{U-1Var}
 (\bru_n^{(N)})^2\asymp\iint |f_n^{(N)}(y_1)-f_n^{(N)}(y_2)|^2 d\mu_n(y_1)d\mu_n(y_2)=2\Var{(f_n)}.
\end{equation}
Repeating the arguments from the proof of Theorem \ref{LmVarCycles} we obtain that there are constants
$\hC_1,  \hC_2$ such that
$$ {\hC_1}^{-1}{\sum_n \Var(f_n(X_n))}-\hC_2\leq V_N\leq
\hC_1\left(\sum_n \Var(f_n(X_n))\right)+\hC_2 . $$
This estimate has been previously obtained in \cite{Do, SV} under weaker ellipticity assumptions.
A similar estimate does {\em not} hold in case $f_n^{(N)}$ depends on two variables. Indeed if
$f_n^{(N)}$ is a gradient, then $V_N$ is bounded while
$\DS \sum_{n=1}^N \Var(f_n(X_n, X_{n+1}))$ can be arbitrarily large.

\medskip
We end the chapter with the reproduction of the proofs of the following two known well-known results.

\begin{theorem}[Dobrushin]\label{Theorem-Dobrushin}
\index{CLT}\index{Dobrushin's Theorem}
Let $\mathsf f$ be an a.s. uniformly bounded additive functional on a uniformly elliptic Markov array $\mathsf X$.
If $\Var(S_N)\to\infty$, then for every interval,
$$
\Prob\left[\frac{S_N-\E(S_N)}{\sqrt{\Var(S_N)}}\in (a,b)\right]\xrightarrow[N\to\infty]{}\frac{1}{\sqrt{2\pi}}\int_a^b e^{-t^2/2}dt.
$$
\end{theorem}
\medskip
\noindent
The proof we give, which is due to Sethuraman \& Varadhan, is based on McLeish's martingale central limit theorem.  For the convenience of the reader we prove the martingale CLT in section \ref{Section-McLeish}.

The next result reduces in the case of identically distributed independent random variables to Khintchin-Kolmogorov's Two-Series Theorem. The result is stated for  Markov chains, and not Markov arrays, because it relates to the properties of $S_N$ as a stochastic process.\index{Kolmogorov two-series theorem}\index{summable variance!a.s. convergence of $\{S_N\}_{N\geq 1}$}

\begin{theorem}\label{Proposition-Kolmogorov-Three-Series}
Let $\mathsf f=\{f_n\}$ be an a.e. uniformly bounded additive functional of a uniformly elliptic inhomogeneous Markov chain $\mathsf X=\{X_n\}$. If $\sum\limits_{n=1}^\infty \Var[f_n(X_n,X_{n+1})]$ is finite,  then
$$ \sum\limits_{n=1}^\infty \left[f_n(X_n,X_{n+1})-\E(f_n(X_n,X_{n+1}))\right]\text{ converges almost surely.}
$$
\end{theorem}

\section{Proofs}

\subsection{The Gradient Lemma}\label{Section-Gradient-Lemma}

\begin{lemma}[Gradient Lemma]
\label{LmVarAbove}
Suppose $\mathsf f$ is an additive functional on a uniformly elliptic Markov array $\mathsf X$, and assume $|\mathsf f|\leq K$ almost surely. Then we can write
$$
\mathsf f=\wt{\mathsf f}+\nabla\mathsf a+\mathsf c,
$$
where $\wt{\mathsf f}, \mathsf a, \mathsf c$ are additive functionals on $\mathsf X$ with the following properties:
\begin{enumerate}[(a)]
\item $|\mathsf a|\leq 2K$ and $a^{(N)}_n(x)$ are measurable functions on $\fS^{(N)}_n$.
\item $|\mathsf c|\leq K$ and $c^{(N)}_n$ are constant functions.
\item $|\wt{\mathsf f}|\leq 6K$ and $\wt{f}^{(N)}_n(x,y)$  satisfy $\|\wt{f}^{(N)}_n\|_2\leq u_n^{(N)}$ for all $3\leq n\leq k_N+1$.
\end{enumerate}
If $\mathsf X$ is a Markov chain, we can choose $f^{(N)}_n=f_n$, $a^{(N)}_n=a_n$, $c^{(N)}_n=c_n$.
\end{lemma}

\medskip
\noindent
{\bf Proof for Doeblin chains:} Before proving the lemma in full generality, we  consider the important special case of Doeblin chains (Example \ref{Example-Doeblin-Chains}), for which the proof is particularly simple.

Recall that a Doeblin chain is a  Markov {chain} $\mathsf X$ with {\em finite} state spaces $\fS_n$ of {\em uniformly bounded cardinality}, and  whose associated transition matrices $\pi^n_{xy}:=\pi_{n,n+1}(x,\{y\})$
satisfy the following properties:
\begin{enumerate}[(E1)]
\item $\exists\epsilon_0'>0$ s.t. for all $n\geq 1$ and
$(x,y)\in\fS_n\times\fS_{n+1}$, either $\pi_{xy}^n=0$ or
$\pi_{xy}^n>\epsilon_0'$;
\item for all $n$, for all $(x,z)\in\fS_n\times\fS_{n+2}$,  $\exists y\in\fS_{n+1}$ such that $\pi_{xy}^n \pi_{yz}^{n+1}>0$.
\end{enumerate}
We saw in example \ref{Example-Doeblin-Chains} that  $\mathsf X$ is uniformly elliptic.

We re-label the states in $\mathfrak S_n$ so that $\mathfrak S_n=\{1,\ldots,d_n\}$ where $d_n\leq d$, and in such a way that
$
\pi_{11}^n>0\text{ for all }n$.
Assumption (E2) guarantees that  for every $n\geq 3$ and every $x\in \mathfrak S_n$ there exists a state $\xi_{n-1}(x)\in\mathfrak S_{n-1}$ s.t.
$
\pi^{n-2}_{1,\xi_{n-1}(x)}\pi^{n-1}_{\xi_{n-1}(x),x}>0.
$
Let
\begin{align*}
& a_0\equiv 0, \ \ a_1\equiv 0,\ \ \text{ and }  a_n(x):=f_{n-2}(1,\xi_{n-1}(x))+f_{n-1}(\xi_{n-1}(x),x)\text{ for }n\geq 3\\
& c_0:=0,\ \  c_1:=0,\ \ \text{ and } c_n:=f_{n-2}(1,1)\text{ for }n\geq 3\\
&  \wt{\mathsf f}:=\mathsf f-\nabla\mathsf a-\mathsf c.
\end{align*}

We claim that $\wt{\mathsf f}, \mathsf a, \mathsf c$ satisfy our requirements.
To explain why and to motivate the construction, consider  the special case  $u_n=0$. In this $\|\wt{\mathsf f}\|_2=0$ and  the lemma reduces to constructing functions $b_n:\mathfrak S_n\to\R$ s.t. $\mathsf f=\nabla\mathsf b+\mathsf c$. We first try to solve $\mathsf f=\nabla\mathsf b$ with $\mathsf c=0$. Any solution must satisfy
\begin{equation}\label{fgg}
f_{n}(x,y)=b_{n+1}(y)-b_{n}(x).
\end{equation}
Necessarily,
$
b_n(y)=b_2(x_2)+f_2(x_2,x_3)+\cdots +f_{n-2}(x_{n-2},x_{n-1})+f_{n-1}(x_{n-1},y)
$
for all paths $(x_2,\ldots,x_{n-1},y)$ with positive probability.  The path $x_2=\cdots=x_{n-2}=1$, $x_{n-1}=\xi_{n-1}(y)$ suggests to define
\begin{align*}
b_2\equiv 0\ , \ b_n(y)&:=\sum_{k=2}^{n-3} f_k(1,1)+f_{n-2}(1,\xi_{n-1}(y))+f_{n-1}(\xi_{n-1}(y),y)
\end{align*}
This works: for every $n\geq 3$, if $\pi_{xy}^n>0$ then
\begin{align}
&b_{n+1}(y)-b_n(x)=[f_{n-2}(1,1)+f_{n-1}(1,\xi_n(y))+f_n(\xi_n(y),y)\notag\\
&\hspace{3cm} -f_{n-2}(1,\xi_{n-1}(x))-f_{n-1}(\xi_{n-1}(x),x)-f_n(x,y)]+f_n(x,y)\notag\\
&\therefore b_{n+1}(y)-b_n(x) =\Gamma_n\left(1\  \begin{array}{c} 1\\
\xi_{n-1}(x)
\end{array} \begin{array}{c}{\xi_n(y)}\\
x
\end{array}\  y\right)+f_n(x,y)\overset{!}{=}f_n(x,y).\label{Gamma-defect}
\end{align}
Here is the justification of $\overset{!}{=}$. In the setup we consider, the natural measure on the level n hexagons is atomic, and every admissible hexagon has positive mass. So $u_n=0$ implies that $\Gamma_n(P)=0$ for every admissible hexagon, and $\overset{!}{=}$ follows.

We proved \eqref{fgg}, but we are not yet done because  $\mathsf b$  is not necessarily uniformly bounded. To fix this decompose $b_n(y)=a_n(y)+\sum_{k=2}^{n-3} f_k(1,1)$. Then  $|\mathsf a|\leq 2K$, and a direct calculation shows that
$
f_n(x,y)=a_{n+1}(y)-a_n(x)+f_{n-2}(1,1),
$
 whence $\mathsf f=\nabla\mathsf a+\mathsf c$ as we claimed.

This proves the lemma  in case $u_n=0$.
The general case $u_n\geq 0$ is done in exactly the same way, except that now
the identity \eqref{Gamma-defect} gives for $\wt{\mathsf f}:=\mathsf f-\nabla a-\mathsf c$
\begin{align*}
\wt{f}_n(x,y)&=f_n(x,y)-(a_{n+1}(y)-a_n(x))-c_n=-\Gamma_n\left(1\  \begin{array}{c} 1\\
\xi_{n-1}(x)
\end{array} \begin{array}{c}{\xi_n(y)}\\
x
\end{array}\  y\right).
\end{align*}
If $|\mathsf f|\leq K$, then $|\Gamma_n|\leq 6K$, whence $|\wt{\mathsf f}|\leq 6K$. Next,
\begin{align*}
\|\wt{f}_n\|_2^2&\leq \E\left[\Gamma_n\left(1\  \begin{array}{c} 1\\
\xi_{n-1}(X_n)
\end{array} \begin{array}{c}{\xi_n(X_{n+1})}\\
X_n
\end{array}\  X_{n+1}\right)^2\right].
\end{align*}
In the scenario we consider the space of admissible hexagons has a finite number of elements, and each has probability uniformly bounded below. So there is a global constant $C$ which only depends on $\sup |\fS_n|$ and on $\epsilon_0'$ in (E2) such that
$$
\E\left[\Gamma_n\left(1\  \begin{array}{c} 1\\
\xi_{n-1}(X_n)
\end{array} \begin{array}{c}{\xi_n(X_{n+1})}\\
X_n
\end{array}\  X_{n+1}\right)^2\right]\leq C\E[\Gamma(P)^2],
$$
where the last expectation is over all position $n$ hexagons.
So $\|\wt{\mathsf f}\|_2\leq \sqrt{C}\cdot u_n^2$.

(The gradient lemma says that we can choose $\mathsf a$ and $\mathsf c$ so that $C=1$. The argument we gave does not quite give this, but the value of the constant is not important for the applications we have in mind.)

\medskip
\noindent
{\bf The proof of the gradient lemma in the general case:}
Recall the ladder process $\mathsf L=\{\un{L}^{(N)}_n\}$,  $\un{L}^{(N)}_n=(Z_{n-2}^{(N)},Y_{n-1}^{(N)},X_n^{(N)})$ from
\S \ref{Section-Ladder}.
In what follows we omit the  superscripts $^{(N)}$  on the right hand side of identities. Define
\begin{align*}
&F_n^{(N)}(\un{L}_n^{(N)}):=F_n(\un{L}_n)=f_{n-2}(Z_{n-2},Y_{n-1})+f_{n-1}(Y_{n-1},X_n)\\
&\Gamma_n^{(N)}(\un{L}_n^{(N)},\un{L}_{n+1}^{(N)}):=\Gamma_n(\un{L}_n,\un{L}_{n+1})=\Gamma\left(
Z_{n-2} \begin{array}{l} Z_{n-1}\\ Y_{n-1}
\end{array}
\begin{array}{l} Y_{n}\\ X_{n}
\end{array}
X_{n+1}
\right),\ \ \text{see \eqref{balance}}.
\end{align*}
Then we have the following  identity:
\begin{equation}\label{Gamma-identity}
f_n^{(N)}(X_n,X_{n+1})=F_{n+1}(\un{L}_{n+1})-F_n(\un{L}_n)+f_{n-2}(Z_{n-2},Z_{n-1})-\Gamma_n(\un{L}_n,\un{L}_{n+1}).
\end{equation}
Next define $a_n^{(N)}:\fS_n^{(N)}\to\R$ and $c_n^{(N)}\in\R$ by
\begin{align}
\label{DefGrad}
&a_n^{(N)}(\xi):=\E\biggl(\E(F_n(\un{L}_n)|X_{n}=\xi\bigr)\biggr) \ \ (3\leq n\leq k_N)\\
&c_n^{(N)}:=\E[f_{n-2}(Z_{n-2},Z_{n-1})]. \label{DefConst}
\end{align}
We will show that the lemma holds with $\mathsf a, \mathsf c$ and $\wt{\mathsf f}:=\mathsf f-\nabla\mathsf a-\mathsf c$.

Since $|\mathsf f|\leq K$ by assumption, it is clear that $|\mathsf a|\leq 2K$ and $|\mathsf c|\leq K$. It remains to bound $\wt{\mathsf f}$ in $L^\infty$ and $L^2$.

\medskip
\noindent
{\sc Claim:} For every $(\xi,\eta)\in\fS_n\times\fS_{n+1}$,
\begin{align*}
&c_n^{(N)}=\E\left[\E\biggl(f_{n-2}(Z_{n-2},Z_{n-1})\bigg|\begin{array}{l}X_{n+1}=\eta\\
X_n=\xi\\
\end{array}\biggr)\right],\\
&a_{n}^{(N)}(\xi)= \E\biggl(F_{n}(\un{L}_{n})\bigg|\begin{array}{l}X_{n+1}=\eta\\
X_n=\xi\\
\end{array}\biggr)\\
&a_{n+1}^{(N)}(\eta)= \E\biggl(F_{n+1}(\un{L}_{n+1})\bigg|\begin{array}{l}X_{n+1}=\eta\\
X_n=\xi
\end{array}\biggr)
\end{align*}

\medskip
\noindent
{\em Proof of the claim.\/} The proof is based on Lemma \ref{Lemma-Ladder}.
The first identity is because $\{Z_n\}$ is independent from $\{X_n\}$.
The second identity is because conditioned on $X_n$, $\un{L}_n$ is independent of $X_{n+1}$. The third identity is because  conditioned on $X_{n+1}$, $\un{L}_{n+1}$ is independent of $X_n$.

\medskip
With the claim proved, we can proceed to bound $\wt{\mathsf f}$.
Taking the conditional expectation $\E(\ \cdot\ |X_{n+1}^{(N)}=\eta\ ,\
X_n^{(N)}=\xi)$ on both sides of
\eqref{Gamma-identity}, we find that
$$
f_n^{(N)}(\xi,\eta)=a_{n+1}(\eta)-a_n(\xi)+c_n-\E\left(\Gamma_n(\un{L}_n,\un{L}_{n+1})\bigg|\begin{array}{l}X_{n+1}=\eta\\
X_n=\xi
\end{array}\right),
$$
whence
$
\tf_n(\xi,\eta):=-\E\left(\Gamma_n(\un{L}_n,\un{L}_{n+1})\bigg|\begin{array}{l}X_{n+1}=\eta\\
X_n=\xi
\end{array}\right)
$.

Clearly $|\wt{\mathsf f}|\leq 6K$.  To bound the $L^2$ norm we recall that the marginal distribution of $\{X_n\}$ with respect to the distribution of the ladder process is precisely the distribution of our original array. Therefore
\begin{align*}
&\|\wt{f}_n^{(N)}\|_2^2\equiv \E\left[\wt{f}_n^{(N)}(X_n,X_{n+1})^2\right]=\E\left[\E\left(\Gamma_n(\un{L}_n,\un{L}_{n+1})\bigg|\begin{array}{l}X_{n+1}
X_n
\end{array}\right)^2\right]\\
&\leq \E\left[\E\left(\Gamma_n(\un{L}_n,\un{L}_{n+1})^2\right)\right]
\end{align*}
because conditional expectations contract $L^2$-norms.

Next we use Lemma \ref{Lemma-Ladder}(3) to see that $\Gamma_n^{(N)}(\un{L}_n,\un{L}_{n+1})$ is equal in distribution to the balance of a random level $N$ hexagon at position $n$, whence $\E(\Gamma_n^2)=(u^{(N)}_n)^2$.\qed

\medskip
 The gradient lemma splits an additive functional into a gradient term, and a term
with controlled variance. The next lemma estimates the covariances between the two terms.
\begin{lemma}
\label{LmCoVarSumOne}
Suppose $\mathsf f$ is a uniformly bounded functional of a uniformly elliptic Markov array. There is a constant $C$ s.t. if $h_{\ell_{N}}^{(N)}$ are
uniformly bounded measurable functions on $\fS_{\ell_N}^{(N)}\times\fS_{\ell_{N+1}}^{(N)}$,
 and
$ \ess\sup |f_n^{(N)}|\leq K$, $\ess\sup |h_{\ell_N}^{(N)}|\leq L$,
then
$$ \mathrm{Cov}\left(S_N, h_{\ell_N}^{(N)}(X_{\ell_N}^{(N)},X_{\ell_{N+1}}^{(N)})\right)\leq C KL. $$
\end{lemma}

\begin{proof}
This follows from the decomposition $\DS \Cov(S_N, h_{\ell_N}^{(N)})=\sum_{n=1}^{k_N} \Cov(f_n^{(N)}, h_{\ell_N}^{(N)})$ and the exponential mixing of $\mathsf X$ (Proposition \ref{Proposition-Exponential-Mixing}). \qed
\end{proof}

\subsection{The estimate for $\Var(S_N)$}
We prove Theorem \ref{LmVarCycles}.
Let $\mathsf f=\{f^{(N)}_n\}$ be an a.s. uniformly bounded additive functional on a uniformly elliptic Markov array $\mathsf X=\{X^{(N)}_n\}$ with row lengths $k_N+1$. Our aim is to bound $\Var(S_N)$ above and below by affine functions  of the structure constants $U_N=\sum_{n=3}^{k_N} (u_n^{(N)})^2$. Assume $|\mathsf f|\leq K$ almost surely.

Throughout the proof, we fix $N$ and drop the superscripts $^{(N)}$. So $X^{(N)}_n=X_n$, $f^{(N)}_n=f_n$, $u^{(N)}_n=u_n$ etc.

\medskip
\noindent
{\bf Lower bound for the variance.}
 Let's  split $U_N=\sum_{n=3}^{k_N}u_n^2$ into three sums:
$$
U_N=\sum_{\gamma=0,1,2}U_N(\gamma),\text{ where }U_N(\gamma):=\left(\sum_{n=3}^{k_N} u_n^2 1_{[n=\gamma\mod 3]}(n)\right).
$$
For every $N$ there is at least one $\gamma_N\in \{0,1,2\}$ such that
$
U_N(\gamma_N)\geq \frac{1}{3}U_N.
$
Let $$
\alpha_N:=\gamma_N+1.$$
and  define  $\beta_N$ by $$
k_N-\beta_N+1=\max\{n\leq k_N: n=\alpha_N\mod 3\}.
$$
With these choices, $\alpha_N,\beta_N\in \{1,2,3\}$, and $k_N-\beta_N+1=\alpha_N\mod 3$.

\medskip
We begin by bounding from below the variance of
$
\DS S_N':=\sum_{k=\alpha_N}^{k_N-\beta_N}f_j(X_j,X_{j+1}).
$
Write $k_N-\beta_N+1=3M_N+\alpha_N$, with $M_N\in\N$, then
$$
S_N'=F_0+\cdots+F_{M_N-1},\text{ where }
{ F_k:=f_{3k+\alpha_N}+f_{3k+\alpha_N+1}+f_{3k+\alpha_N+2}}.
$$
Observe that $S_N'$ is a function of the following variables:
$$
\fbox{$X_{\alpha_N}$}, X_{\alpha_N+1},X_{\alpha_N+2},
\fbox{$X_{\alpha_N+3}$}, X_{\alpha_N+4},X_{\alpha_N+5},\cdots,\fbox{$X_{k_N-\beta_N+1}$},
$$
where we have boxed the terms with indices congruent to $\alpha_N$ mod $3$.
Let $\mathfs F_N$ denote the $\sigma$-algebra generated by the boxed random variables. Conditioned on $\mathcal F_N$, $F_k$ are independent.
Therefore,
\begin{align*}
&\Var(S_{N}'|\mathcal F_N)=\sum_{k=0}^{M_N-1} \Var(F_k|\mathcal F_N)=
\sum_{k=0}^{M_N-1} \Var(F_k|{ X_{3k+\alpha_N},X_{3(k+1)+\alpha_N}})
\end{align*}
Taking the expectation on both sides, and using  the general inequality $\Var(S_{N}')\geq \E(\Var(S_{N}'|\mathcal F_N))$, we obtain
$$
\Var(S_N')\geq \sum_{k=0}^{M_N-1} \E\biggl(\Var(F_k|{ X_{3k+\alpha_N},X_{3(k+1)+\alpha_N}})\biggr).
$$
To  estimate the summands, we recall that for every random variable $W$, $\Var(W)=\frac{1}{2}\E[(W'-W'')^2]$ where $W',W''$ are two independent copies of $W$.
Thus
\begin{align*}
&\Var(F_k|X_{3k+\alpha_N}=a,X_{3(k+1)+\alpha_N}=b)\\
&=
\frac{1}{2}\E\biggl[\Gamma\left({ X_{3k+\alpha_N}\ \begin{array}{l}
X_{3k+\alpha_N+1}\\ Y_{3k+\alpha_N+1}
\end{array} \ \begin{array}{l}
X_{3k+\alpha_N+2}\\ Y_{3k+\alpha_N+2}
\end{array}\  Y_{3(k+1)+\alpha_N}}
\right)^2\bigg|{ \begin{array}{l}
X_{3k+\alpha_N}=Y_{3k+\alpha_N}=a\\
X_{3(k+1)+\alpha_N}=Y_{3(k+1)+\alpha_N}=b
\end{array}}
\biggr],
\end{align*}
whence $\E\bigl(\Var(F_k|X_{3k+\alpha_N},X_{3(k+1)+\alpha_N})\bigr)\equiv \E(\Gamma(P)^2)\equiv (u^{(N)}_{{3k+\alpha_N+2}})^2$ where $\Gamma(P)$ is the balance of a random hexagon $P\in\mathrm{Hex}(N,{3k+\alpha_N+2})$. So
\begin{align*}
&\Var(S_N')\geq \frac{1}{2}\sum_{k=0}^{M_N-1} (u^{(N)}_{{3k+\alpha_N+2}})^2
=\frac{1}{2}\sum_{k=0}^{M_N-1} (u^{(N)}_{3(k+1)+\gamma_N})^2 \ \ \ \ (\because\alpha_N=\gamma_N+1)\\
&\geq
\frac{1}{2}\sum_{n=3}^{k_N}u_n^2 1_{[n=\gamma_n\mod 3]}(n)-2\sup\{u_j^2\}\geq \frac{1}{2}U_N(\gamma_N)-2\cdot (6K)^2\\
&> \frac{1}{6}U_N-100K^2,\text{ by choice of $\gamma_N$.}
\end{align*}

Now we claim that $|\Var(S_N)-\Var(S_N')|$ is uniformly bounded from below. To see this,
let $f_j^\ast:=f_j-\E(f_j)$, and let $A_N:=\{j\in\N: 1\leq j\leq \alpha\text{ or }k_N-\beta\leq j\leq k_N\}$. Then $S_N=S_N'+\sum_{j\in A_N}f_j$, whence
$$
\Var(S_N)=\Var(S_N')+\Var(\sum_{j\in A_N} f_j)+2\sum_{j\in A_N}\Cov(S_N',f_j).
$$
Since $|\mathsf f|\leq K$ and $|A_N|\leq 6$, the second term bounded by $4K^2|A_N|\leq 24K^2$.  Next by   uniform ellipticity and \eqref{Exp-Mixing-L-three}, there are mixing constants $\theta\in (0,1)$ and $C_{mix}>0$ which only depend on $\eps_0$,  the ellipticity constant of $\mathsf X$, so that
\begin{align*}
&\Cov(S_N',f_j)\leq C_{mix} \sum_{n=1}^{k_N}\|f_{n}^\ast\|_2\|f_j^\ast\|_2\theta^{|n-j|}\leq \frac{2C_{mix}K^2}{1-\theta}.
\end{align*}
It follows that $\Var(S_N)\geq \Var(S_N')-const\geq const U_N-const$, where the constants depends only on $K$ and the ellipticity constant $\eps_0$.

\medskip
\noindent
{\bf Upper bound for the variance.} Write $\mathsf f=\wt{\mathsf f}+\nabla\mathsf a+\mathsf c$ as in the gradient lemma. In particular,
$\Var(\tf_n(X_{n-1}, X_n))\leq u_n^2$.
 Then
$$ \Var\left(\sum_{n=1}^{k_N} f_n\right)=
\Var\left(\sum_{n=1}^{k_N} \tf_n\right)+\Var\left(a_{N+1}-a_1\right)+
2\Cov\left(\sum_{n=1}^{k_N} \tf_n, a_{N+1}-a_1 \right) .$$
The first term is smaller than $C_1 U_N+C_2'$ due to the gradient Lemma and Lemma \ref{LmVarSum}
the second term is smaller than $C_2''$ due to Lemma \ref{LmCoVarSumOne}.
\hfill$\Box$

\subsection{Characterization of center-tight additive functionals}\label{Section-Non-Universal}
We prove Theorem \ref{Theorem-center-tight}. Suppose $\mathsf f$ is an a.s. uniformly bounded functional on a uniformly elliptic array $\mathsf X$. We will show that the following conditions are equivalent:
\begin{enumerate}[(a)]
\item $\Var(S_N)=O(1)$;
\item $\mathsf f$ is the sum of a gradient and an additive functional with summable variance;
\item $\mathsf f$ is center tight.
\end{enumerate}

\medskip
\noindent
{\bf (a)$\Rightarrow$(b):} By the gradient lemma
$
\mathsf f=\nabla\mathsf a+(\wt{\mathsf f}+\mathsf c),
$
where $a_n^{(N)}(x)$ are measurable functions on $\fS^{(N)}_n$ with uniformly bounded $L^\infty$ norm, $c_n^{(N)}$ are uniformly bounded constants, and $\|\wt{\mathsf f}_n\|_2\leq u_n^{(N)}$.  By Theorem \ref{LmVarCycles}, $\displaystyle\sup_N \sum_{n=3}^{k_N} (u_n^{(N)})^2<\infty$, so  $\wt{\mathsf f}+\mathsf c$ has summable variance, proving (b).

\medskip
\noindent
{\bf (b)$\Rightarrow$(c):} We already saw that gradients and functionals with summable variance are center-tight. Since the sum of center-tight functionals is center-tight, (c) is proved.

\medskip
\noindent
{\bf (c)$\Rightarrow$(a):} Assume by way of contradiction that $\exists N_i\uparrow\infty$ such that  $V_{N_i}=\Var(S_{N_i})\to\infty$. By Dobrushin's CLT (see \cite{Do}, \cite{SV} and
\S \ref{SSProofDobr}),   $\displaystyle\frac{S_{N_i}-\E(S_{N_i})}{\sqrt{V_{N_i}}}$ converges in distribution to a standard Gaussian distribution. But
center-tightness implies that there are constants $\mu_{N}'$ s.t. $\frac{S_{N}-\mu_{N}'}{\sqrt{V_{N}}}$ converges in distribution to the deterministic random variable $W\equiv 0$, and both statements cannot be true simultaneously.  \hfill$\Box$

\subsection{McLeish's martingale central limit theorem}\label{Section-McLeish}
A {\bf martingale difference array}\index{martingale difference array} with row lengths $k_N$ is a (possibly non-Markov) array $\Delta$ of random variables
$$\Delta=\{\Delta_j^{(N)}: N\geq 1, 1\leq j\leq k_N\}$$ together with an array of $\sigma$-algebras $\{\mathfs F^{(N)}_j: N\geq 1, 1\leq j\leq k_N\}$, so that:
\begin{enumerate}[(1)]
\item For each $N$, $\Delta^{(N)}_1,\ldots,\Delta^{(N)}_{k_N}$ are
 random variables on the same probability space $(\mathfrak S_N, \mathfs F_N,\mu_N)$.
\item $\mathfs F_1^{(N)}\subset \mathfs F_2^{(N)}\subset \mathfs F_3^{(N)}\subset \cdots\subset \mathfs F_{k_N}^{(N)}$ are sub $\sigma$-algebras of $\mathfs F_N$.
\item $\Delta^{(N)}_j$ is $\mathfs F^{(N)}_j$--measurable, $\E(|\Delta^{(N)}_j|)<\infty$, and
$
\E(\Delta^{(N)}_{j+1}|\mathfs F^{(N)}_j)=0.
$
\end{enumerate}
We say that  $\Delta$ has {\bf finite variance}, if every $\Delta^{(N)}_j$ has finite variance. Notice that $\E(\Delta^{(N)}_j)=0$ for all $j=2,\ldots,k_{N+1}$. If in addition $\E(\Delta^{(N)}_1)=0$ for all $N$, then we say that $\Delta$ has {\bf zero mean}.

\begin{example} Suppose $\{S_n\}$ is a  martingale relative to $\{\mathfs F_n\}$, then
$$\Delta_1^{(N)}:=S_1\ , \
\Delta_{j}^{(N)}:=S_{j}-S_{j-1}\ , \ \mathfs F_{j}^{(N)}:=\mathfs F_j\ ,\ j=1,\ldots,N
$$
is a martingale difference array.
\end{example}

\medskip
The following basic observation on martingale difference arrays is a key to many of their properties:

\begin{lemma}\label{Lemma-Not-Correlated}
Suppose $\Delta$ is a martingale difference array with finite variance, then for each $N$ $\Delta_1^{(N)},\ldots,\Delta_{k_N}^{(N)}$ are uncorrelated, and if $\Delta$ has zero mean, then
$$\Var(\sum_{n=1}^{k_N}\Delta_n^{(N)})=\sum_{n=1}^{k_N}\E[(\Delta_n^{(N)})^2].$$
\end{lemma}
\begin{proof}
Fix $N$ and write $\Delta_j^{(N)}=\Delta_j$, $\mathfs F_j^{(N)}=\mathfs F_j$.

If $i<j$, then
$
\E(\Delta_j\Delta_i)=\E[\E(\Delta_j\Delta_i|\mathfs F_{j-1})]=
\E[\E(\Delta_i\E(\Delta_j|\mathfs F_{j-1}))]=\E(\Delta_i\cdot 0)=0
$. The identity for the variance immediately follows.\qed
\end{proof}

\begin{theorem}[McLeish's Martingale Central Limit Theorem]\label{Theorem-Martingale-CLT}\index{Martingale CLT}\index{Central Limit Theorem}
Let $\Delta=\{\Delta^{(N)}_j\}$ be a martingale difference array with row lengths $k_N$, zero mean, and finite variance, and  let $V_N:=\sum_{j=1}^{k_N}\E[(\Delta_j^{(N)})^2]$. Suppose:
\begin{enumerate}[$(1)$]
\item $\max\limits_{1\leq j\leq k_N}\frac{|\Delta_j^{(N)}|}{\sqrt{V_N}}$ has uniformly bounded $L^2$ norm;
\item $\max\limits_{1\leq j\leq k_N}\frac{|\Delta_j^{(N)}|}{\sqrt{V_N}}\xrightarrow[N\to\infty]{}0$
in probability; and
\item $\frac{1}{V_N}\sum_{n=1}^{k_N} (\Delta^{(N)}_n)^2\xrightarrow[N\to\infty]{}1$ in probability.
\end{enumerate}
Then for all intervals $(a,b)$, $
\Prob\left[\frac{1}{\sqrt{V_N}}\sum_{j=1}^{k_N}\Delta_j^{(N)}\in (a,b)\right]\xrightarrow[N\to\infty]{}\frac{1}{\sqrt{2\pi}}\int_a^b e^{-t^2/2}dt.
$
\end{theorem}
\noindent
We prepare the ground for the proof.

A sequence of random variables $\{Y_n\}$ on $(\Omega,\mathfs F,\mu)$ is called {\bf uniformly integrable}\index{uniform integrability} if  for every  $\epsilon$, $\exists K$ s.t. $\E(|Y_n| 1_{[|Y_n|>K]})<\epsilon$ for all $n$. This is strictly stronger than tightness (there are tight non-integrable random variables).

\begin{example} If $M_p:=\sup\|Y_n\|_p<\infty$ for some $p> 1$, then $\{Y_n\}$ is uniformly integrable.
\end{example}
Indeed, by Chebyshev's inequality, $\mu[|Y_n|>K]\leq \frac{1}{K^p}M_p^p$, and by H\"older's inequality\\  $\E(|Y_n|1_{[|Y_n|>K]})\leq M_p \mu[|Y_n|>K]^{1/q}=O(K^{-p/q})$ for the $q$ s.t. $\frac{1}{p}+\frac{1}{q}=1$.

\begin{lemma}\label{Lemma-Uniform-Integrability}Suppose $Y_n, Y\in L^1(\Omega,\mathfs F,\mu)$, then $Y_n\xrightarrow[n\to\infty]{L^1}Y$ iff $\{Y_n\}$ are uniformly integrable and $Y_n\xrightarrow[n\to\infty]{}Y$ in probability. In this case $\E(Y_n)\xrightarrow[n\to\infty]{}\E(Y)$.
\end{lemma}
\begin{proof} We include the well-known, standard proof for completeness.

\smallskip
Proof of $(\Rightarrow)$:
Since $Y\in L^1$,  it follows (for example from the Dominated Convergence
Theorem) that $\DS \lim_{K\to\infty} \E(|Y| 1_{|Y|\geq K})=0.$
Given $\eps$ take $K$ so that $\E(|Y| 1_{|Y|\geq K})<\eps.$
Let $\delta=\Prob(|Y||\leq K)$ then it is easy to see that
\begin{equation}\label{abs-cont}
\E(|Y| 1_F)<\eps\text{ for all measurable sets $F$ s.t. $\mu(F)<\delta$}.
\end{equation}
Fix $\eps>0$, and choose $\delta$ as in \eqref{abs-cont}.

Suppose $\|Y_n-Y\|_1\to 0$. By
 Markov's inequality $\Prob[|Y_n-Y|>\eps]\leq \|Y_n-Y\|_1/\eps\to 0$, and $Y_n\to Y$ in probability.

Markov's inequality also implies that  $\Prob[|Y_n|>K]\leq K^{-1}\sup\|Y_n\|_1=O(K^{-1})$, so
there exists $K$ s.t. $\Prob[|Y_n|>K]<\delta$ for all $n$. By the choice of $\delta$,
\begin{align*}
&\int_{[|Y_n|>K]}|Y_n|d\mu\leq \int_{[|Y_n|>K]}|Y|d\mu+\int_{[|Y_n|>K]}|Y_n-Y|d\mu\leq \eps+\|Y_n-Y\|_1\xrightarrow[n\to\infty]{}\eps.
\end{align*}
Uniform integrability follows.

\medskip
Proof of $(\Leftarrow)$: Given a random variable $Z$, let $Z^K:=Z 1_{[|Z|\leq K]}$. Since $\{Y_n\}$ is uniformly integrable, for every $\epsilon$ there is a $K>1$ s.t.  $\|Y_n^K-Y_n\|_1<\epsilon$ for all $n$.
Similarly, $\|Y^K-Y\|_1<\epsilon$ for all $K$ large enough. Thus for all $n$,
\begin{align*}
&\|Y_n-Y\|_1\leq \|Y_n^K-Y^K\|_1+2\epsilon
\leq \epsilon\mu[|Y_n^K-Y^K|\leq \epsilon]+2K\mu[|Y_n^K-Y^K|> \epsilon]+2\epsilon\\
&\leq 3\epsilon+2K\biggl(\mu[|Y_n-Y|>\epsilon]+\mu[|Y_n|>K]+\mu[|Y|>K]\biggr)\\
&\leq 3\epsilon +2K\mu[|Y_n-Y|>\epsilon]+2\E(|Y_n|1_{[|Y_n|>K]})+2\E(|Y|1_{|Y|>K})\\
&\therefore \limsup_{n\to\infty}\|Y_n-Y\|_1\leq 3\epsilon+2\sup_{n}\E(|Y_n|1_{[|Y_n|>K]})+2\E(|Y|1_{|Y|>K}),
\end{align*}
where we have used the assumption that $Y_n\to Y$ in probability.
The last expression can be made arbitrarily small,  by choosing $\epsilon$ sufficiently small, $K$ sufficiently large, and appealing to the uniform integrability of $Y_n$.\qed
\end{proof}

\begin{lemma}[McLeish]
Let $\{W_{j}^{(N)}:1\leq j\leq k_N\}$ be a triangular array of random variables\footnote{Not necessarily a martingale difference array or a Markov array.}, where $W_1^{(N)},\ldots,W_{k_N}^{(N)}$ are defined on the same probability space. Fix $t\in\R$  and let $\DS T_N(t):=\prod_{j=1}^{k_N}(1+itW_{j}^{(N)})$. Suppose
\begin{enumerate}[$(1)$]
\item $\{T_N(t)\}$ is uniformly integrable and $\E(T_N)\xrightarrow[N\to\infty]{}1$,
\item $\sum_{j=1}^{k_N} (W_{j}^{(N)})^2\xrightarrow[N\to\infty]{}1$ in probability,
\item $\max\limits_{1\leq j\leq k_N}{|W_{j}^{(N)}|}\xrightarrow[N\to\infty]{}0$ in probability.
\end{enumerate}
Then $\E(e^{it(W_{1}^{(N)}+\cdots+W_{k_N}^{(N)})})\xrightarrow[N\to\infty]{}e^{-\frac{1}{2}t^2}$.
\end{lemma}
\begin{proof}
Define a function $r(x)$ on $[-1,1]$ by the identity
$e^{ix}=(1+ix)e^{-\frac{1}{2}x^2+r(x)}$,  then
$r(x)=-\log(1+ix)+ix+\frac{1}{2}x^2=O(|x|^3)$. Fix $C$ s.t. $|r(x)|\leq C|x|^3$ for $|x|<1$.

Substituting  $S_N:=W_{1}^{(N)}+\cdots+W_{k_N}^{(N)}$ in $e^{ix}=(1+ix)e^{-\frac{1}{2}x^2+r(x)}$ gives (in what follows we drop the superscripts $^{(N)}$ and abbreviate $T_n:=T_n(t)$):
\begin{align*}
\E(e^{it S_N})&=\E(\prod_{j=1}^{k_N} e^{it W_{j}})=\E(T_N  e^{-\frac{1}{2}\sum_{j=1}^{k_N} t^2W_{j}^2+r(tW_{j})})\\
&=\E(T_N U_N),\text{ where }U_N:=\exp\left[-\frac{1}{2}\sum_{j=1}^{k_N} t^2(W_{j}^{(N)})^2+r(tW_{j}^{(N)})\right].
\end{align*}
$T_N$ and $U_N$ have the following properties:
\begin{enumerate}[(a)]
\item $\E(T_N)\xrightarrow[N\to\infty]{}1$, by assumption.
\item $\{T_N\}$ is uniformly integrable by assumption, and $|T_N U_N|=|e^{itS_N}|=1$.
\item $\DS U_N\xrightarrow[N\to\infty]{\text{prob}}\exp\left(-\frac{1}{2}t^2\right)$, because
\begin{enumerate}[$\circ$]
\item $\DS \sum_{j=1}^{k_N} \left(W_{j}^{(N)}\right)^2\xrightarrow[N\to\infty]{\text{prob}}1$, by assumption,
\item $\max\limits_{1\leq j\leq k_N}\left|W_{j}^{(N)}\right|\xrightarrow[n\to\infty]{\text{prob}}0$ by assumption, so with asymptotic probability one,
$$
\left|\sum_{i=1}^{k_N} r(tW^{(N)}_{j})\right|\leq C|t|^3\max\limits_{1\leq j\leq k_N}
\left|W^{(N)}_{j}\right|\sum_{j=1}^{k_N} \left(W_j^{(N)}\right)^2\xrightarrow[N\to\infty]{\text{prob}}0.
$$
\end{enumerate}
\end{enumerate}

We claim that this implies that  $\E(e^{itS_N})=\E(T_N U_N)\xrightarrow[N\to\infty]{}e^{-\frac{1}{2}t^2}$. Let $L:=e^{-\frac{1}{2}t^2}$. Since $|\E(T_N U_N)-L|\leq |\E(T_N(U_N-L))|+L|\E(T_N)-1|$, (a) tells us that
\begin{align}
|\E(T_N U_N)-L|&\leq |\E(T_N(U_N-L))|+o(1).\label{McLeish-est-1}
\end{align}
Next, for every $K,\epsilon$,  $\mu[|T_N(U_N-L)|>\epsilon]\leq \mu[|T_N|>K]+\mu[|U_N-L|>\epsilon/K]$.  Therefore by (b) and (c),
\begin{align}
T_N(U_N-L)\xrightarrow[N\to\infty]{}0 \text{ in probability}.
\label{McLeish-est-2}
\end{align}
Finally, $|T_N(U_N-L)|\leq 1+L|T_N|$, so $T_N(U_N-L)$ is uniformly integrable by (b).
By Lemma \ref{Lemma-Uniform-Integrability}, $\E(T_N (U_N-L))\to 0$, and by \eqref{McLeish-est-1},  $\E(e^{itS_N})=\E(T_N U_N)\to e^{-\frac{1}{2}t^2}$.\qed
\end{proof}

\medskip
\noindent
{\bf Proof of the Martingale CLT \cite{McLeish}:}
 Let $\Delta=\{\Delta_j^{(N)}\}$ be a martingale difference array with row lengths $k_N$, which satisfies the assumptions of Theorem \ref{Theorem-Martingale-CLT}, and let
 $$
 S_N:=\sum_{j=1}^{k_N}\Delta_j^{(N)}\text{ and }V_N:=\Var(S_N)\equiv\sum_{j=1}^{k_N}\E[(\Delta_j^{(N)})^2]\ \text{(see Lemma \ref{Lemma-Not-Correlated})}.
 $$

It is tempting to  apply McLeish's Lemma to the normalized array
$
\Delta_j^{(N)}/\sqrt{V_N}
$, but to do this we need to check the uniform integrability of $\prod_{j=1}^n(1+it \Delta^{(N)}_j/\sqrt{V_N})$ and this  is difficult.
It is easier to work with the following array of truncations:
$$
W_1^{(N)}:=\tfrac{1}{\sqrt{V_N}}\Delta_{1}^{(N)}\ , \ W_{n}^{(N)}:=\tfrac{1}{\sqrt{V_N}}\Delta_{n}^{(N)} 1_{[\sum_{k=1}^{n-1}(\Delta^{(N)}_k)^2\leq 2V_N]}.
$$
It is easy to check that $\{W_{n}^{(N)}\}$ is a martingale difference array relative to $\mathfs F_{n}^{(N)}$, and that  $\{W^{(N)}_n\}$ has zero mean, and finite variance.

In addition,
 $ S_N^\ast:=\sum_{n=1}^{k_N}W_{n}^{(N)}$ are close to $S_N/\sqrt{V_N}$ in probability:
$$\mu[S_N^\ast\neq \tfrac{S_N}{\sqrt{V_N}}]\leq \mu\biggl[\exists 1\leq j\leq k_N\text{ s.t. }\sum_{k=1}^{j-1}(\Delta^{(N)}_k)^2>2V_N \biggr]
\leq \mu\biggl[\sum_{j=1}^{k_N}(\Delta^{(N)}_k)^2>2V_N\biggr]
\xrightarrow[ N\to\infty]{}0 $$
because $\DS \tfrac{1}{V_N}\sum_{j=1}^{k_N}\left(\Delta^{(N)}_j\right)^2
\xrightarrow[N\to\infty]{\text{prob}}1 $ by assumption.

Thus to prove the theorem, it is enough to show that $S_N^\ast$ converges in distribution to the standard Gaussian distribution. To do this, we check that $\{W^{(N)}_n\}$ satisfies the conditions of McLeish's Lemma.

Fix $t\in\R$, and let $T_N=T_N(t):=\prod_{j=1}^{k_N} (1+itW_{j}^{(N)}).$
Let $J_N:=\max\{2\leq j\leq k_N: \sum_{k=1}^{j-1} (\Delta^{(N)}_n)^2\leq 2 V_N\}$ (or $J_N=1$ if the maximum is over the empty set). Writing $W_j=W_j^{(N)}$ and $\Delta_j=\Delta_j^{(N)}$, we obtain
\begin{align*}
&|T_N|=\prod_{j=1}^{k_N} (1+t^2W_{j}^2)^{1/2}=\prod_{j=1}^{J_N} \biggl(1+\frac{t^2\Delta_j^2}{V_N} \biggr)^{1/2}\\
&=\left(\prod_{j=1}^{J_N-1} \biggl(1+\frac{t^2 \Delta_j^2}{V_N}\biggr)\right)^{1/2}\cdot \biggl(1+\frac{t^2 \Delta_{J_N}^2}{V_N} \biggr)^{1/2},\text{ where }\prod_{j=1}^0(\cdots):=1\\
&\leq \exp\biggl(\frac{t^2}{2V_N}\sum_{j=0}^{J_N-1}\Delta_j^2\biggr)\biggl(1+\frac{t^2}{V_N} \Delta_{J_N}^2\biggr)^{1/2}\leq e^{t^2}\biggl(1+|t|\max\limits_{1\leq j\leq k_N}\biggl|\frac{\Delta_{j}^{(N)}}{\sqrt{V_N}}\biggr|\biggr).
\end{align*}
Thus
$$\|T_N(t)\|_2^2\leq e^{2t^2}\left(1+|t|
\E\left(\max\limits_{1\leq j\leq k_N}\left|\frac{\Delta^{(N)}_j}{\sqrt{V_N}}\right|\right)\right)^2.$$
By the first assumption of the theorem, the last quantity is uniformly bounded for each $t$. It follows that
$
\{T_N(t)\}_{N\geq 1}$  is uniformly integrable for each $t$. Next, successive conditioning shows that
$\DS \E(T_N)=1+it\E\left(\Delta_1^{(N)}\right)=1.$
The first condition of McLeish's Lemma is verified.

 The second condition of McLeish's Lemma follows from the assumption
$\DS \frac{1}{V_N}\sum_{n=1}^{k_N}\left(\Delta_n^{(N)}\right)^2\to 0$
in probability, and the  estimate
\begin{align*}
&\mu\left[\sum_{n=1}^{k_N}(W^{(N)}_n)^2\neq \sum_{n=1}^{k_N}(\tfrac{\Delta_n^{(N)}}{\sqrt{V_N}})^2 \right]\leq
\mu\left[\exists 1\leq n\leq k_N\text{ s.t. }\sum_{j=1}^{n}(\Delta_j^{(N)})^2>2V_N\right]\leq
\end{align*}
\begin{align*}
&\leq  \mu\left[\sum_{n=1}^{k_N}(\Delta_n^{(N)})^2>2V_N\right]\xrightarrow[N\to\infty]{}0, \text{ because $\frac{1}{V_n}\sum_{j=1}^{k_N}(\Delta_n^{(N)})^2\to 1$ in probability.}
\end{align*}
The third condition of McLeish's Lemma follows from the assumption that \\
$\DS \max_{1\leq j\leq k_N}|W^{(N)}_j|\to 0$ in probability, for similar reasons.

So McLeish's lemma applies to $\{W^{(N)}_n\}$, and  $\E(e^{it S_N^\ast})\to e^{-\frac{1}{2}t^2}$ for all $t\in\R$. By L\'evy's continuity theorem, this implies that $S_N^\ast\xrightarrow[N\to\infty]{\text{dist}}N(0,1)$.

 As explained above, this implies that
$\frac{S_N}{\sqrt{V_N}}\xrightarrow[N\to\infty]{\text{dist}}N(0,1)$.\hfill$\Box$

\subsection{Proof of Dobrushin's central limit theorem}
\label{SSProofDobr}

Let $\mathsf X=\{X^{(N)}_n\}$ be a uniformly elliptic Markov array with row lengths $k_N+1$, and let $\mathsf f=\{f^{(N)}_n\}$ be an a.s. uniformly bounded additive functional on $\mathsf X$. Define as before
$
\displaystyle S_N=\sum_{n=1}^{k_N} f_n^{(N)}(X^{(N)}_n,X^{(N)}_{n+1})\ , \ V_N:=\Var(S_N).
$
Without loss of generality,
$$
\E[f_n^{(N)}(X^{(N)}_n,X^{(N)}_{n+1})]=0\text{ and }|f_n^{(N)}|\leq K\text{ for all }n,N.
$$

Define $\mathfs F^{(N)}_n:=\sigma(X^{(N)}_1,\ldots,X^{(N)}_{n+1})$ for $n\geq 1$, and
$\mathfs F^{(N)}_0:=$trivial $\sigma$-algebra. Fix $N$ and write $f_k=f_k^{(N)}(X^{(N)}_k,X^{(N)}_{k+1})$ and $\mathfs F_k=\mathfs F_k^{(N)}$, then $\E(f_k|\mathfs F_k)=f_k$, $\E(f_k|\mathfs F_0)=\E(f_k)=0$, and therefore
\begin{align*}
S_N&=\sum_{k=1}^{k_N} f_k=\sum_{k=1}^{k_N}\bigl(\E(f_k|\mathfs F_k)-\E(f_k|\mathfs F_0)\bigr)=\sum_{k=1}^{k_N} \sum_{n=1}^{k}\bigl(\E(f_k|\mathfs F_n)-\E(f_k|\mathfs F_{n-1})\bigr)\\
&=\sum_{n=1}^{k_N}\sum_{k=n}^{k_N}\bigl(\E(f_k|\mathfs F_n)-\E(f_k|\mathfs F_{n-1})\bigr)\\
&=\sum_{n=1}^{k_N} \Delta^{(N)}_n,\text{ where }\Delta^{(N)}_n:=\sum_{k=n}^{k_N}\bigl(\E(f_k^{(N)}|\mathfs F_n^{(N)})-\E(f_k^{(N)}|\mathfs F_{n-1}^{(N)})\bigr).
\end{align*}
The array $\{\Delta^{(N)}_n: 1\leq n\leq k_N; N\geq 1\}$ is a martingale difference array relative to the filtrations $\mathfs F^{(N)}_n$, with zero mean and finite variances. To prove the theorem, it suffices to check that $\{\Delta^{(N)}_n\}$ satisfies the  conditions of the martingale CLT.

\medskip
\noindent
{\sc Step 1:} {\em  $\displaystyle\max\limits_{1\leq j\leq k_N}\frac{|\Delta_j^{(N)}|}{\sqrt{V_N}}$ has uniformly bounded $L^2$ norm, and $\max\limits_{1\leq j\leq k_N}\frac{|\Delta_j^{(N)}|}{\sqrt{V_N}}\xrightarrow[N\to\infty]{prob}0$}.

\medskip
\noindent
{\em Proof.\/} The proof is based on the exponential mixing of uniformly elliptic Markov arrays (Proposition \ref{Proposition-Exponential-Mixing}):  Let $K:=\ess\sup|\mathsf f|$, then there are constants $C_{mix}>1$ and $0<\theta<1$ such that for all $k\geq n$,
$$
\|\E(f_k^{(N)}|\mathfs F^{(N)}_n)\|_\infty\leq C_{mix} K \theta^{k-n-1}.
$$
It follows that $|\Delta_j^{(N)}|<2C_{mix} K \sum_{\ell=-1}^\infty \theta^\ell=\frac{2C_{mix} K\theta^{-2}}{1-\theta}$. The step follows from the assumption that $V_N\to\infty$.

\medskip
\noindent
{\sc Step 2:} $\displaystyle\frac{1}{V_N}\sum_{n=1}^{k_N} (\Delta_n^{(N)})^2\xrightarrow[N\to\infty]{}1$ in probability.

\medskip
\noindent
{\em Proof.\/} We follow \cite{SV} closely.

Let $Y^{(N)}_i:=(\Delta^{(N)}_i)^2/V_N$. We will show that
$
\left\|\sum_{i=1}^{k_N}Y^{(N)}_i-1\right\|_2^2\xrightarrow[N\to\infty]{}0
$,
and use the general fact that $L^2$-convergence implies convergence in probability (by Chebyshev's inequality).

Notice that $\DS \E\Big(\sum_{i=1}^{k_N} Y_i^{(N)}\Big)=1$, because by
Lemma \ref{Lemma-Not-Correlated}, this expectation equals
$$\frac{1}{V_N}\times \Var\left(\sum_{n=1}^{k_N}\Delta_n^{(N)}\right)=\frac{1}{V_N}\Var(S_N)=1.$$
So
\begin{align}
&\bigl\|\sum_{i=1}^{k_N} Y^{(N)}_i-1\bigr\|_2^2=\E\biggl[\bigl(\sum_{i=1}^{k_N} Y^{(N)}_i\bigr)^2\biggr]-2\E\biggl[\sum_{i=1}^{k_N} Y^{(N)}_i\biggr]+1\notag\\
&\hspace{0.5cm}=\E\biggl[\sum_{i=1}^{k_N} \bigl(Y^{(N)}_i\bigr)^2\biggr]+2\E\biggl[\sum_{i<j}^{k_N} Y^{(N)}_i Y^{(N)}_j\biggr]-2+1\notag\\
&\hspace{0.5cm}=O(\max\limits_{1\leq \ell\leq {k_N}}\|Y^{(N)}_\ell\|_\infty)\cdot \E\biggl[\sum_{\ell=1}^{k_N} Y^{(N)}_i\biggr]+2\E\biggl[\sum_{i<j} Y^{(N)}_i Y^{(N)}_j\biggr]-1.\notag
\end{align}
We saw in the proof of step 1 that $\|\Delta_j^{(N)}\|_\infty$ are uniformly bounded. Thus $\max\limits_{1\leq \ell\leq {k_N}}\|Y^{(N)}_\ell\|_\infty=O(1/V_N)$,  so
$
\bigl\|\sum_{i=1}^{k_N} Y^{(N)}_i-1\bigr\|_2^2=2\E\bigl[\sum_{i<j} Y^{(N)}_i Y^{(N)}_j\bigr]-1+o(1).
$
It remains to show that
\begin{equation}
 \label{master-equation}
2\E\biggl[\sum_{i<j} Y^{(N)}_i Y^{(N)}_j\biggr]\xrightarrow[N\to\infty]{}1.
\end{equation}

The proof of \eqref{master-equation} is based on the following fact:
\begin{equation}\label{Dobrushin-Fact}
\mathrm{Osc}(N):=\max_{1\leq i\leq k_N}\mathrm{Osc}\left(\E\biggl(\sum_{j=i+1}^{k_N} Y^{(N)}_j\bigg|\mathfs F^{(N)}_i\biggr)\right)\xrightarrow[N\to\infty]{}0.
\end{equation}
Here $\mathrm{Osc}$ is the oscillation, which was defined in \S \ref{section-definition-of-uniform-ellipticity}.
Before proving this, we explain why \eqref{Dobrushin-Fact} implies \eqref{master-equation}. Write $x=y\pm \epsilon$ whenever $y-\epsilon\leq x\leq y+\epsilon$.
Every bounded
 function $\vf$ satisfies $\vf=\E(\vf)\pm\mathrm{Osc}(\vf)$. So
\begin{align*}
&2\E\biggl[\sum_{i<j} Y^{(N)}_i Y^{(N)}_j\biggr]=
2\E\biggl[\sum_{i=1}^{k_N}Y^{(N)}_i \sum_{j=i+1}^{k_N} Y^{(N)}_j\biggr]=2\E\biggl[\sum_{i=1}^{k_N}Y^{(N)}_i \E\bigl(\sum_{j=i+1}^{k_N} Y^{(N)}_j\big|\mathfs F^{(N)}_i\bigr)\biggr]\\
&=
2\E\biggl[\sum_{i=1}^{k_N}Y^{(N)}_i \E\bigl(\sum_{j=i+1}^{k_N} Y^{(N)}_j\bigr)\biggr]\pm 2\E\biggl[\sum_{i=1}^{k_N}Y^{(N)}_i \biggr]\mathrm{Osc}(N)
\end{align*}
\begin{align*}
&=2\sum_{i=1}^{k_N}\E(Y^{(N)}_i) \sum_{j=i+1}^{k_N} \E(Y^{(N)}_j)\pm 2\mathrm{Osc}(N)\ \ (\because \sum_{i=1}^{k_N}\E(Y^{(N)}_i)=1)\\
&=\left(\sum_{i=1}^{k_N}\E(Y^{(N)}_i)\right)^2-\sum_{i=1}^{k_N}\E(Y^{(N)}_i)^2\pm 2\mathrm{Osc}(N)\\
&=1+O\biggl(\max_{1\leq i\leq k_N} \|Y_i^{(N)}\|_\infty\biggr)\pm 2\mathrm{Osc}(N),\because
\sum\E(Y^{(N)}_i)^2\leq \underset{=1}{\underbrace{\sum \E(Y^{(N)}_i)}}\max \|Y_i^{(N)}\|_\infty\\
&=1+O(V_N^{-1})+O(\Osc(N)).
\end{align*}
So \eqref{Dobrushin-Fact} implies \eqref{master-equation}, and with it the step.

 We turn to the proof of \eqref{Dobrushin-Fact}. Henceforth we fix $N$ and drop all the $^{(N)}$ superscripts.  First we note that a routine modification of the proof of Lemma \ref{Lemma-Not-Correlated} shows that for all $j,k>i$, \
 $
 \E(\Delta_j \Delta_k|\mathfs F_i)=0.
 $
 It follows that
 \begin{align}
 &\E\left(\sum_{j=i+1}^{k_N} Y_j\bigg|\mathfs F_i\biggr)\equiv\frac{1}{V_N}\E\biggl(\sum_{j=i+1}^{k_N} \Delta^2_j\bigg|\mathfs F_i\right)=\frac{1}{V_N}\E\biggl(\bigl(\sum_{n=i+1}^{k_N} \Delta_n\bigr)^2\bigg|\mathfs F_i\biggr)\notag\\
 &=\frac{1}{V_N}\E\biggl(\biggl(\; \sum_{n=i+1}^{k_N}\sum_{k=n}^{k_N}
 \left[\E(f_k|\mathfs F_n)-\E(f_k|\mathfs F_{n-1}) \right] \biggr)^2\bigg|\mathfs F_i\biggr)\notag\\
 &=\frac{1}{V_N}\E\biggl(\biggl(\sum_{k=i+1}^{k_N}\sum_{n=i+1}^{k}\E(f_k|\mathfs F_n)-\E(f_k|\mathfs F_{n-1}) \biggr)^2\bigg|\mathfs F_i\biggr)\notag\\
  &=\frac{1}{V_N}\E\bigg(\biggl(\sum_{k=i+1}^{k_N}
  \left[ f_k-\E(f_k|\mathfs F_i)\right] \biggr)^2\bigg|\mathfs F_i\biggr)\notag\\
  &=\frac{1}{V_N}\sum_{k,\ell=i+1}^{k_N}
  \E\biggl[\bigl(\left[f_k-\E(f_k|\mathfs F_i)\right]\bigr)\bigl(f_\ell-\E(f_\ell|\mathfs F_i)\bigr)\bigg|\mathfs F_i\biggr]\notag\\
   &=\frac{1}{V_N}\sum_{k,\ell=i+1}^{k_N}\E\biggl[f_k f_\ell +\E(f_k|\mathfs F_i)\E(f_\ell|\mathfs F_i)-f_k\E(f_\ell|\mathfs F_i)-f_\ell\E(f_k|\mathfs F_i)\bigg|\mathfs F_i\biggr]\notag\\
   &=\frac{1}{V_N}\sum_{k,\ell=i+1}^{k_N}
   \left[\E\bigl[f_k f_\ell|\mathfs F_i\bigr] -\E(f_\ell|\mathfs F_i)\E(f_k|\mathfs F_i)\right]
   \label{couperin}
 \end{align}
The oscillation of the summands can be estimated as follows.
By Lemma \ref{Lemma-tempest}(d)
 $$
 \mathrm{Osc}\biggl(\E\bigl(u(X_{k}^{(N)},X_{k+1}^{(N)})\big|X_j^{(N)}\bigr)\biggr)\leq
 \delta\left(\pi^{(N)}_{j,k}\right)\mathrm{Osc}(u),
 $$
 where $\delta\left(\pi^{(N)}_{j,k}\right)$ is the contraction coefficient of the
 $(k-j)$-step Markov operator $\pi^{(N)}_{j,k}$.
 In the uniformly elliptic case, by Lemma \ref{Lemma-Contraction}, \;
  $\delta(\pi^{(N)}_{j,j+2})\leq 1-\epsilon_0$, where $\epsilon_0>0$ is the ellipticity constant of $\mathsf X$.
  Iterating Lemma \ref{Lemma-tempest}(c) we conclude that
  there exists $C_0>0$ and $0<\theta<1$ such that for all $k>i+1$, and for every bounded function $u:\fS_k^{(N)}\times\fS_{k+1}^{(N)}\to\R$,
 $$
 \mathrm{Osc}\biggl(\E\bigl(u(X_{k}^{(N)},X_{k+1}^{(N)})\big|\mathfs F_i^{(N)}\bigr)\biggr)\leq C_0 \theta^{k-i}\mathrm{Osc}(u).
 $$

 This, \eqref{Exp-Mixing-L-infinity}, and the inequalities  $|f_j|\leq K$, $\mathrm{Osc}(u)\leq 2\|u\|_\infty$ and $\mathrm{Osc}(uv)\leq \|u\|_\infty\mathrm{Osc}(v)+\|v\|_\infty\mathrm{Osc}(u)$ imply  the existence of constants $C_1>0$ and $0<\theta<1$ such that   for every $N\geq 1$ and $i+2\leq k\leq \ell\leq k_N$,
 \begin{align*}
 &\mathrm{Osc}\bigl(\E(f_\ell|\mathfs F_i)\E(f_k|\mathfs F_i)\bigr)\\
 &\leq
\mathrm{Osc}(\E(f_\ell|\mathfs F_i))\| \E(f_k|\mathfs F_i)\|_\infty+\|\E(f_\ell|\mathfs F_i)\|_\infty\mathrm{Osc}(\E(f_k|\mathfs F_i))
\leq C_1\theta^{k-i}\theta^{\ell-i}.\\
 &\mathrm{Osc}\biggl(\E\bigl[f_k f_\ell|\mathfs F_i\bigr]\biggr)=\mathrm{Osc}\biggl(\E\bigl[f_k \E(f_\ell|\mathfs F_k)|\mathfs F_i\bigr]\biggr)\\
 &\leq
 C_0\theta^{k-i}\mathrm{Osc}(f_k \E(f_\ell|\mathfs F_k))\leq C_0 \theta^{k-i}[K\cdot \mathrm{Osc}(\E(f_\ell|\mathfs F_k))+\mathrm{Osc}(f_k)\|\E(f_\ell|\mathfs F_k)\|_\infty]\\
 &\leq C_1\theta^{k-i}\theta^{\ell-k}.
 \end{align*}
We have stated these bounds for $k,\ell\geq i+2$, but in fact they remain valid for $k=i+2$ or $\ell=i+2$, if we increase $C_1$ to guarantee that $C_1\theta^2>2K^2$.

Substituting these bounds in \eqref{couperin}, we find that
 \begin{align*}
& \mathrm{Osc}(N)\leq \frac{2C_1}{V_N}\sum_{k,\ell=i+1}^\infty \theta^{k-i}\theta^{\ell-k}\leq \frac{2C_1}{V_N}\left(\frac{\theta}{1-\theta}\right)^2\xrightarrow[N\to\infty]{}0.
 \end{align*}
 This proves \eqref{Dobrushin-Fact}, and completes the proof of step 2.

 \medskip
 Steps 1 and 2 verify the conditions of the martingale CLT. So
 $\frac{1}{\sqrt{V_N}}\sum_{n=1}^{k_N}\Delta^{(N)}_n$  converges in distribution to the standard Gaussian distribution. By construction, $\frac{1}{\sqrt{V_N}}S_N\equiv \frac{1}{\sqrt{V_N}}\sum_{n=1}^{k_N}\Delta^{(N)}_n$, and the theorem is proved. \hfill$\Box$

\subsection{Almost sure convergence for sums of functionals with summable variance}\label{Section-Kolmogorov}
We prove Proposition \ref{Proposition-Kolmogorov-Three-Series}.
Let $f_0^\ast:=0$,
 $f_n^\ast:=f_n(X_n,X_{n+1})-\E f_n(X_n,X_{n+1})$, let
 $\cF_0$ denote the trivial $\sigma$-algebra, and let
$
\cF_n$ denote the $\sigma$-algebra generated by $X_1,\ldots,X_n$.
Then $f_k^\ast$ is $\cF_{k+1}$-measurable, so
$$
f_k^\ast=\E(f_k^\ast|\cF_{k+1})-\E(f_k^\ast|\cF_0)=\sum_{n=0}^k \E(f_k^\ast|\cF_{n+1})-\E(f_k^\ast|\cF_{n}).
$$
Therefore (numbered equalities are justified below):
\begin{align*}
&\sum_{k=1}^Nf_k^\ast=\sum_{k=1}^N \sum_{n=0}^k
\left[\E(f_k^\ast|\cF_{n+1})-\E(f_k^\ast|\cF_{n})\right]=
\sum_{n=0}^N \sum_{k=n}^N \left[\E(f_k^\ast|\cF_{n+1})-\E(f_k^\ast|\cF_{n})\right]\\
&\overset{(1)}{=}\sum_{n=0}^N \sum_{k=n}^\infty \left(\E(f_k^\ast|\cF_{n+1})-\E(f_k^\ast|\cF_{n})\right)-\sum_{n=0}^N \sum_{k=N+1}^\infty \left(\E(f_k^\ast|\cF_{n+1})-\E(f_k^\ast|\cF_{n})\right)\\
&\overset{(2)}{=} \sum_{n=0}^N \sum_{k=n}^\infty \left(\E(f_k^\ast|\cF_{n+1})-\E(f_k^\ast|\cF_{n})\right)-\sum_{k=N+1}^\infty\sum_{n=0}^N  \left(\E(f_k^\ast|\cF_{n+1})-\E(f_k^\ast|\cF_{n})\right)\\
& \overset{(3)}{=}\sum_{n=0}^N \sum_{k=n}^\infty \left(\E(f_k^\ast|\cF_{n+1})-\E(f_k^\ast|\cF_{n})\right)-\sum_{k=N+1}^\infty \E(f_k^\ast|\cF_{N+1}).
\end{align*}

To justify the numbered inequalities almost surely, we need to establish the convergence of the series which they involve.

By \eqref{Exp-Mixing-L-two}, $\|\E(f_k^\ast|\cF_{n+1})\|_2+\|\E(f_k^\ast|\cF_{n})\|_2\leq 2C_{mix}\sqrt{\Var(f_k)}\theta^{k-n+1}$, so by the Cauchy-Schwarz inequality and the assumption $\sum \Var(f_n)<\infty$,
$$
\sum_{n=0}^N \sum_{k=n}^\infty \left\|\E(f_k^\ast|\cF_{n+1})-\E(f_k^\ast|\cF_{n})\right\|_2<\infty.
$$
This justifies $\overset{(1)}{=}$ and $\overset{(2)}{=}$.

Next by assumption, $|\mathsf f|\leq K$ a.s. for some constant $K$. By \eqref{Exp-Mixing-L-infinity}, $\|\E(f_k^\ast|\mathcal F_0)\|_\infty+\|\E(f_k^\ast|\mathcal F_n)\|_\infty\leq 4KC_{mix}\theta^{n-k}$ so
$
 \sum_{k=N+1}^\infty |\E(f_k^\ast|\mathcal F_{N+1})|<\infty.
$
This justifies $\overset{(3)}{=}$.

In summary,  ${\displaystyle \sum_{k=1}^Nf_k^\ast=\sum_{n=0}^N\Delta_n-Z_N}$, where
$$
\Delta_n:= \sum_{k=n}^\infty \left(\E(f_k^\ast|\cF_{n+1})-\E(f_k^\ast|\cF_{n})\right)\ ,\
Z_N:=\sum_{k=N+1}^\infty \E(f_k^\ast|\cF_{N+1}).
$$
To finish the proof, we show that $\DS  \sum_{n=0}^{ \infty} \Delta_n$ and
$\DS \lim_{N\to\infty} Z_N$ exist a.s.

\medskip
\noindent
{\sc Claim 1.} {\em $M_N:=\sum_{n=0}^{N-1}\Delta_n$ is a martingale relative to $\{\cF_N\}$, and $\sup\|M_N\|_2<\infty$. Consequently, $\lim M_N$ exists almost surely.}

\medskip
\noindent
{\em Proof.\/}
$\DS \E(M_{N+1}-M_{N}|\cF_N)=\E(\Delta_{N}|\cF_N)
\overset{!}{=}\sum_{k=N}^\infty \E(\E(f_k^\ast|\cF_{N+1})|\cF_N)-\E(\E(f_k^\ast|\cF_{N})|\cF_N)=0.$
To justify $\overset{!}{=}$ we note that the series
$$
\DS \Delta_N=\sum_{k=N}^\infty \left[\E(f_k^\ast|\cF_{n+1})-\E(f_k^\ast|\cF_{n})\right]
$$
converges in $L^2$, because $
\|\E(f_k^\ast|\cF_{n+1})-\E(f_k|\cF_n)\|_\infty=O(\theta^{k-n}),
$ so  its conditional expectation can be calculated term-by-term.

Next we show that $\|M_N\|_2$ is uniformly bounded:
\begin{align*}
&\|M_{N+1}\|_2\leq \biggl\|\sum_{n=0}^N \sum_{k=n}^\infty\E(f_k^\ast|\cF_{n+1})-
\E(f_k^\ast|\cF_{n})
\biggr\|_2\\
&\leq \biggl\|\sum_{k=0}^\infty \sum_{n=0}^{k\wedge N} \E(f_k^\ast|\cF_{n+1})-
\E(f_k^\ast|\cF_{n})
\biggr\|_2=\biggl\|\sum_{k=0}^\infty \E(f_k^\ast|\cF_{(k\wedge N)+1})
\biggr\|_2\\
&\leq \bigl\|\sum_{k=0}^N f_k^\ast\bigr\|_2+\bigl\|\sum_{k=N+1}^\infty
\E(f_k^\ast|\cF_{N+1})
\bigr\|_2\\
&\leq \sqrt{\sum_{k=0}^N\|f_k^\ast\|_2^2+2\sum_{0\leq k<\ell\leq N}\Cov(f_k^\ast,f_\ell^\ast)}+\sum_{k=N+1}^\infty \|\E(f_k^\ast|\cF_{N+1})\|_\infty
\end{align*}
\begin{align*}
&\leq \sqrt{\sum_{k=0}^\infty\|f_k^\ast\|_2^2+2C_{mix}\sum_{0\leq k<\ell\leq \infty}\theta^{\ell-k}\|f_k^\ast\|_2 \|f_\ell^\ast\|_2}+C_{mix}\sum_{k=N+1}^\infty\|f_k^\ast\|_\infty\theta^{k-N}.
\end{align*}
The last expression is uniformly bounded, because $\sum \Var(f_k)<\infty$ and
\begin{align*}
&\sum_{0\leq k<\ell<\infty}\theta^{\ell-k}\|f_k^\ast\|_2 \|f_\ell^\ast\|_2\leq \sum_{r=1}^\infty \theta^r \sum_{k=0}^\infty \|f_k^\ast\|_2\|f_{k+r}^\ast\|_2\leq \frac{1}{1-\theta}\sum_{k=0}^\infty \|f_k^\ast\|_2^2\\
&\sum_{k=N+1}^\infty \|f_k^\ast\|_\infty\theta^{k-N}=\frac{1}{1-\theta}\sup_k\|f_k^\ast\|_\infty.
\end{align*}

\medskip
\noindent
{\sc Claim 2.} {\em $Z_N\xrightarrow[N\to\infty]{}0$ almost surely.
}

\medskip
\noindent
{\em Proof.\/} It is enough to prove that  $\sum \|Z_N\|_2^2<\infty$, because this implies using Chebyshev's inequality that
$
\sum \Prob[|Z_N|>\epsilon]\leq \frac{1}{\epsilon^2}\sum\|Z_N\|_2^2<\infty\text{ for all }\epsilon>0,
$
whence, by the Borel-Cantelli Lemma, $\limsup |Z_N|\leq \epsilon$  a.s. for all $\epsilon$. Equivalently, $\lim Z_N=0$ a.s.

Here is the proof that $\sum \|Z_N\|_2^2<\infty$:
\begin{align*}
&\frac{1}{2}\sum_{N=1}^\infty \|Z_N\|_2^2=\sum_{N=1}^\infty \sum_{k_2\geq k_1>N}
\E\biggl[\E(f_{k_1}^\ast|\mathcal F_{N+1}) \E(f_{k_2}^\ast|\mathcal F_{N+1})\biggr]\\
&=\sum_{N=1}^\infty \sum_{k_2\geq k_1>N}
\E\biggl[f_{k_2}^\ast\E(f_{k_1}^\ast|\mathcal F_{N+1}) \biggr]\\
&\leq C_{mix}\sum_{N=1}^\infty \sum_{k_2\geq k_1>N} \theta^{k_2-N+1}\|f_{k_2}^\ast\|_2\|\E(f_{k_1}^\ast|\mathcal F_{N+1})\|_2\quad \text{by \eqref{Exp-Mixing-L-three}}
\end{align*}
\begin{align*}
&\leq  C_{mix}^2\sum_{N=1}^\infty \sum_{k_2\geq k_1>N} \theta^{k_2-N+1}\|f_{k_2}^\ast\|_2\cdot \theta^{k_1-N+1}
\|f_{k_1}^\ast\|_2\quad \text{by \eqref{Exp-Mixing-L-two}}\\
&= C_{mix}^2  \sum_{j\geq 0} \theta^{j}\sum_{k>0}\theta^{2k} \sum_{N=1}^\infty\|f_{k+N+j}^\ast\|_2
\|f_{k+N}^\ast\|_2\\
&\ \ \ \ \ \  \text{ (after changing indices  $j=k_2-k_1$, $k=k_1-N+1$)}\\
&\leq C_{mix}^2  \sum_{j\geq 0} \theta^{j}\sum_{k>0}\theta^{2k} \sqrt{\sum_{N=1}^\infty\|f_{k+N+j}^\ast\|_2^2
\sum_{N=1}^\infty\|f_{k+N}^\ast\|_2^2}\\
&\leq \frac{C_{mix}^2}{1-\theta}  \sum_{k>0}\theta^{2k}
\sum_{N=k}^\infty\|f_{N}^\ast\|_2^2=\frac{C_{mix}^2}{1-\theta}\sum_{N=1}^\infty \|f_N^\ast\|_2^2 \sum_{k=1}^N \theta^{2k}<\infty,
\end{align*}
because $0<\theta<1$ and $\sum \|f_k^\ast\|_2^2<\infty.$
\qed

\subsection{Convergence of moments.}
Dobrushin's CLT (Theorem \ref{Theorem-Dobrushin}) shows that if $V_N\to\infty$ then for any
bounded continuous
function $\phi:\R\to\R$ we have
\begin{equation}
\label{WeakConvNorm}
\lim_{N\to\infty} \EXP\left[\phi\left(\frac{S_N-\EXP(S_N)}{\sqrt{V_N}}\right)\right]=
\frac{1}{\sqrt{2\pi}}
\int_{-\infty}^\infty \phi(z)
e^{-z^2/2} dz.
\end{equation}
In applications, one often need to have convergence of expectations for unbounded functions,
such as polynomials. This problem is addressed in the present section.

\begin{lemma}
\label{LmMomBounds}
Let $\mathsf{f}$ be a centered bounded additive functional of a uniformly elliptic Markov chain such that
$V_N\to\infty. $ Then for each $r\in \N$ there is a constant $C_r$ such that for all $N$,
$$ \left|\EXP\left[S_N^r\right]\right|\leq C_r V_N^{\lfloor r/2 \rfloor} . $$
\end{lemma}

\begin{corollary}
Under the assumptions of Lemma \ref{LmMomBounds}
$$ \lim_{N\to\infty} \frac{\EXP[S_N^r]}{V_N^{r/2}}=
\begin{cases} 0 & r\text{ is odd,} \\ (r-1)!!=\prod_{k=0}^{(r/2)-1}(r-2k-1) & r \text{ is even}. \end{cases} $$
\end{corollary}
The corollary follows from Dobrushin's CLT (Theorem \ref{Theorem-Dobrushin}),  using the fact that by Lemma \ref{LmMomBounds} and the de la Vall\'ee-Poussin Lemma,  $(S_N/\sqrt{V_N})^r$  is uniformly integrable for all $r>1$ even, and therefore  $\lim\E[(S_N/\sqrt{V_N})^r]=\E[N^r]$, where $N$ is a Gaussian random variable with mean zero and variance one.

The proof of Lemma \ref{LmMomBounds} proceeds by expanding $S_N^r$ into a sum of $r$-tuples $f_{n_1}\cdots f_{n_r}$ $(n_1\leq \cdots\leq n_r)$, and by estimating the expectation  of each tuple. (Here and throughout, $f_n=f_n(X_n,X_{n+1})$.)
In view of the gradient lemma it is sufficient to prove Lemma \ref{LmMomBounds}
under the assumption that there is some constant $C>0$ such that
$\tu_n:=\|f_n\|_{L^2}$ satisfy $\DS\sum_n \tu_n^2 \leq C V_N.$

 Consider an $r$ tuple $f_{n_1}\cdots f_{n_r}$ where
$ n_1\leq n_2\leq\dots \leq n_r. $
Segments of the form $[n_j, n_{j+1}]$ will be called {\em edges}. The vertices belonging to an edge are called {\em bound}, the other vertices are
called {\em free}.

A {\em marking} is a non-empty  collection of edges satisfying the following two conditions.
Firstly, each vertex $n_j$ belongs to at most
one edge.   Secondly, for every free vertex $n_l$, either
\begin{enumerate}[(i)]
\item there exists a minimal  $f(l)>l$ such that $n_{f(l)}$ is bound, and for all
$l\leq i<f(l)$,  $n_{i+1}-n_i\leq n_{f(l)+1}-n_{f(l)}$; or
\item there exists a maximal $p(l)<l$ such that $n_{p(l)}$ is bound, and  for all $p(l) < i\leq l$, $n_{i}-n_{i-1}\leq n_{p(l)}-n_{p(l)-1}$.
\end{enumerate}
If (i) holds we will say that $n_l$ is associated to the edge
$[n_{f(l)}, n_{f(l)+1}]$ otherwise it is associated to
$[n_{p(l)-1}, n_{p(l)}].$
\begin{lemma} \label{LmMark}
There are  constants $L=L(r)>0$ and $0<\theta<1$ such that
$$ \left|\EXP\left[\prod_{i=1}^r f_{n_i}\right]\right| \leq
L \sum_{markings}\; \prod_{[n_j,n_{j+1}]\text{ is an edge}}
\left(\theta^{(n_{j+1}-n_j)} \; \;\tu_{n_j} \tu_{n_{j+1}}\right). $$
\end{lemma}

\begin{proof}
If $r=1$ then the result holds since $\EXP[f_n]=0$ (in this case there are no
markings,
and we let
the empty
sum  be equal to zero).

If $r=2$ then the lemma says that
$ \left|\EXP\left[ f_{n_1} f_{n_2}\right]\right|\leq K \theta^{n_2-n_1} \|f_{n_1}\|_{L^2} \|f_{n_2}\|_{L^2}
$
which is true due to Proposition \ref{Proposition-Exponential-Mixing}(2).

For $r\geq 3$ we use induction. Take $j$ such that $n_{j+1}-n_j$ is the largest. Then
$$ \EXP\left[\prod_{i=1}^r f_{n_i} \right]=
\EXP\left[\prod_{i=1}^j f_{n_i} \right]\EXP\left[\prod_{i=j+1}^r f_{n_i} \right]+
O\left(\theta^{(n_{j+1}-n_j)}
\left\|\prod_{i=1}^j f_{n_i} \right\|_{L^2}
\left\|\prod_{i=j+1}^r f_{n_i} \right\|_{L^2}\right). $$
Let $K:=\ess\sup|\mathsf f|$, then  the second term is smaller than
$\DS \theta^{(n_{j+1}-n_j)} \tu_{n_j} \tu_{n_{j+1}} K^{r-2}$. Thus this term is controlled by the marking with only one marked edge
$[n_j, n_{j+1}].$
Applying the inductive assumption to each factor in the first term we obtain the result.
\qed \end{proof}

\begin{lemma}
\label{LmSumAbs}
There exists $\brC_r>0$ s.t. for every
set  $\cC$  of $r$ tuples $1\leq n_1\leq\dots\leq n_r\leq N$,
$$ \Gamma_{\cC}:=\sum_{(n_1,\dots,n_r)\in \cC}
\left|\EXP\left[\prod_{i=1}^r f_{n_i}\right]\right|\leq \brC_r V_N^{\lfloor r/2\rfloor}. $$
\end{lemma}

Lemma \ref{LmSumAbs} implies Lemma \ref{LmMomBounds} since
$$ \EXP\left[S_N^r\right]=
\sum_{s=1}^r\;\; \sum_{k_1+\dots +k_s=r} \frac{r!} {k_1!\cdots k_s!}
\sum_{1\leq n_1<\dots <n_s\leq N} \EXP\left[\prod_{j=1}^s f_{n_j}^{k_j}\right]. $$
Therefore it suffices to prove Lemma \ref{LmSumAbs}.

\begin{proof}
By Lemma \ref{LmMark}
$$ \Gamma_\cC\leq
L \sum_{(n_1,\dots, n_r)\in \cC} \;\; \underset{\text{ of $(n_1,\ldots,n_r)$}}{\sum_{\text{ markings
 $(e_1,\ldots,e_s)$}}} \;\;\prod_{j=1}^s
\left(\tu_{e_j^-}\tu_{e_j^+} \theta^{(e_j^+-e_j^-)} \right)$$
where the marked edges are $e_j=[e_j^-, e_j^+]$, $j=1,\dots, s.$
Collecting all terms with a fixed set of marked edges
$(e_1,\dots, e_s)$ we obtain
\begin{equation}
\label{GammaSumMom}
\Gamma_\cC\leq  C(r)
\sum_s \sum_{(e_1,\dots, e_s)}
\prod_{j=1}^s \left(\tu_{e_j^-}\tu_{e_j^+} \theta^{(e_j^+-e_j^-)} (e_j^+-e_j^-)^{r-2} \right)
\end{equation}
where  $\DS C(r) \prod_j (e_j^+-e_j^-)^{r}$ accounts for all tuples which admit a marking $(e_1, \dots e_s)$.
Indeed, for every edge $e=[e^-, e^+]$ there are at most $0\leq j\leq r-2$ vertices which may be associated
to $e$ and the positions of those vertices are located inside
$$ \left[e^--(r-2) (e^+-e^-), e^-\right)
\cup \left(e^+, e^++(r-2) (e^+-e^-)\right] .$$
It follows that there are at most $2(r-2)(e^+-e^-)$ choices to place each vertex associated to a given
edge.
This gives
$$\prod_e \left(\sum_{j=0}^{r-2} \left[2 (r-2)(e^+-e^-)\right]^j\right)\leq
C(r) \prod_e (e^+-e^-)^{r-2}$$ possibilities for tuples with marking $(e_1,\ldots,e_s)$
proving \eqref{GammaSumMom}.

The sum over $(e_1,\dots e_s)$ in \eqref{GammaSumMom}
can be estimated by
$$ \left( \sum_{n=1}^{N-1} \sum_{m=1}^{N-n} \tu_n \tu_{n+m} \theta^m m^{r-2} \right)^s. $$

For each $m$,
$\DS \sum_n \tu_n \tu_{n+m}=O(V_N)$ due to the Cauchy-Schwartz inequality and because $\sum_{n=1}^N \tu_n^2\leq CV_N$ by assumption. Summing over $m$
gives
$\DS  \Gamma_\cC\leq const \sum_{2s\leq r} V_N^s $
where the condition $2s\leq r$ appears because each edge involves two distinct
vertices, and no vertex belongs to more than one edge. The result follows.
\smartqed\qed \end{proof}

\section{Notes and references}

The connection between the non-growth of variance and representation in terms of  gradients is well-known for stationary  stochastic processes.  The first result in this direction we are aware of is Leonov's Theorem  \cite{Leonov}.
He showed that the asymptotic variance of a homogeneous additive functional of a stationary homogeneous Markov chain is zero iff the additive functional is the sum of a gradient and a constant. Rousseau-Egele \cite{RE} and Guivarc'h \& Hardy \cite{GH} extended this to the context of dynamical systems preserving an invariant Gibbs measure. Kifer \cite{Kifer-CLT}, Conze \& Raugi \cite{Conze-Raugi-Sequential}, Dragi\v{c}evi\'c,Froyland \& Gonz\'alez-Tokman \cite{Dragicevic-Froyland-Gonzalez-Tokman}  have proved versions of Leonov's theorem for random and/or sequential dynamical systems.

The connection between center-tightness and gradients is a central feature of the theory of cocycles over ergodic transformations.
Suppose  $T:X\to X$ is an ergodic probability preserving transformation on a non-atomic probability space. For every measurable $f:X\to\R$, $\{f\circ T^n\}$ is a stationary stochastic process, and
$$
S_N=f+f\circ T+\cdots + f\circ T^{N-1}
$$
are called the ``ergodic sums of the cocycle $f$." A ``coboundary" is a function of the form  $f=g-g\circ T$ with $g$ measurable.
Schmidt  characterized  cocycles with center-tight $S_N$  as those arising from  coboundaries \cite[page 181]{Schmidt-Cocycles}. These results extend to cocycles taking values in locally compact groups, see Moore \& Schmidt \cite{Moore-Schmidt} and Aaronson \& Weiss \cite{Aaroson-Weiss-tightness}. For more on this, see Aaronson \cite[chapter 8]{Aaronson-Book}, and
Bradley \cite[chapters 8,19]{Bradley}.
 We also refer to \cite{GH55} for an analogous result in the continuous setting.

Notice that  inhomogeneous theory is different from the stationary theory in that there is another cause for center-tightness: Having summable variance. This cannot happen in the stationary homogeneous world (unless all $f_i$ are constant).

Theorem \ref{Theorem-Dobrushin} is a special case of a more general result due to Dobrushin, which can be found in \cite{Do}. The conditions for Dobrushin's full result are more general than uniform boundedness or uniform ellipticity. Our proof follows the paper of Sethuraman \& Varadhan \cite{SV}, except for some changes we needed to make to deal with additive functionals of the form $f_k(X_k,X_{k+1})$, and not just $f_k(X_k)$ as in \cite{SV}. McLeish's Lemma,  the  martingale CLT, and their proofs are due to McLeish \cite{McLeish}. We refer the reader to Hall \& Heyde \cite{Hall-Heyde} for the history of this result, further extensions, and references.

Theorem \ref{Proposition-Kolmogorov-Three-Series} is extends
the Kolmogorov-Khintchin ``Two-Series Theorem"  \cite{Kolmogorov-Two-Series}. There are other extensions to sums of dependent random variables. We mention for example a version for martingales  (Hall \& Heyde \cite[chapter 2]{Hall-Heyde}),  for  sums of negatively dependent random variables (Matu\l a, \cite{Matula}) and
 for expanding maps (\cite{Conze-Raugi-Sequential}).

The proofs of theorems \ref{Theorem-Dobrushin} and \ref{Proposition-Kolmogorov-Three-Series} use
Gordin's ``martingale-coboundary decomposition" \cite{Gordin}, see also \cite{Hall-Heyde},\cite{Korepanov-Kosloff-Melbourne}.

\chapter{The essential range and irreducibility}\label{Section-Reducibility}
\label{Chapter-Irreducibility}

{\em In this chapter we discuss the following question: How small can we make the range of an additive functional,  by  subtracting from it a center-tight  functional?}

\section{Definitions and motivation}\label{Section-definition-essential-range}
Let $\mathsf f=\{f_n\}$ be an additive functional of a Markov chain $\mathsf X:=\{X_n\}$. The {\bf algebraic range}\index{algebraic range!Markov chains}\index{range!algebraic} of $(\mathsf X,\mathsf f)$ is the intersection $G_{alg}(\mathsf X,\mathsf f)$ of all closed groups $G$ s.t. ,
\begin{equation}\label{alg-range}
\exists c_n\in\R\text{ s.t. }\Prob[f_n(X_n,X_{n+1})-c_n\in G]=1\text{ for all } n\geq 1.
\end{equation}

We will see later (Lemma \ref{l.alg-range}) that $G_{alg}(\mathsf X,\mathsf f)$ itself satisfies \eqref{alg-range},  therefore $G_{alg}(\mathsf X,\mathsf f)$ is the smallest closed group satisfying \eqref{alg-range}.

\begin{example}{\bf (The simple random walk).}\index{simple random walk} Suppose $\{X_n\}$ are independent random variables  such that $\Prob(X_n=\pm 1)=\frac{1}{2}$, and let $f_n(x,y)=x$. Then
$
S_n=X_1+\cdots+X_n
$
is the simple random walk on $\Z$. The algebraic range in this case is $2\Z$.
\end{example}
\noindent
{\em Proof\/}:  $G_{alg}\subset 2\Z$, because we can take $c_n:=-1$. Assume by contradiction that $G_{alg}\subsetneq 2\Z$, then  $G_{alg}=t\Z$ for $t\geq 4$, and the supports of $S_n$ are cosets of $t\Z$.

But this is false, because $\exists a_1,a_2$ s.t. $|a_1-a_2|<t$ and $\Prob(S_n=a_i)\neq 0$: For  $n$ even take $a_i=(-1)^i$, and for $n$ odd take $a_i=1+(-1)^i$.
\qed

\medskip
The {\bf lattice case}\index{lattice case} is the case when $G_{alg}(\mathsf X,\mathsf f)=t\Z$ for some $t\geq 0$. The {\bf non-lattice case}\index{non-lattice case} is the case when $G_{alg}(\mathsf X,\mathsf f)=\R$.
The distinction is important for the following reason. If $G_{alg}(\mathsf X,\mathsf f)=t\Z$ and $\gamma_N:=c_1+\cdots+c_N$,  then
$$
\Prob(S_N\in \gamma_N+t\Z)=1\text{ for all }N.
$$
In this case it is not true that $\Prob(S_N-z_N\in (a,b))\sim\frac{e^{-z^2/2}|a-b|}{\sqrt{2\pi V_N}}$ whenever $\frac{z_N-\E(S_N)}{\sqrt{V_N}}\to z$, because  $\Prob(S_N-z_N\in (a,b))=0$ whenever $|a-b|<t$ and  $z_N+(a,b)$ falls inside the gaps of $\gamma_N+t\Z$. This is the {\bf lattice obstruction} to the local limit theorem.\index{obstructions to the LLT!lattice case}

There is a related, but more subtle, obstruction.
An additive functional $\mathsf f$ is called {\bf reducible}\index{reducible}\index{irreducible}\index{additive functional!reducible and irreducible} on $\mathsf X$, if there is
another additive functional $\mathsf g$ on $\mathsf X$ such that  $\mathsf f-\mathsf g$ is center-tight, and
$$
G_{alg}(\mathsf X,\mathsf g)\subsetneq G_{alg}(\mathsf X,\mathsf f).
$$
In this case we say that  $\mathsf g$ is a {\bf reduction}\index{reduction} of $\mathsf f$, and call the algebraic range of $\mathsf g$ a {\bf reduced range}\index{reduced range}\index{range!reduced} of $\mathsf f$.

\begin{example}
{\bf (Simple random walk with continuous first step):}\label{Example-SRW-U-First-Step} Suppose $\{X_n\}_{n\geq 1}$ are independent real valued random variables such that $X_1$ has continuous non-uniform distribution $\mathfrak F$ with compact support, and $X_2, X_3,\ldots$ are equal to $\pm 1$ with equal probabilities. Let $f_n(x,y)=x$, then
$
S_n=X_1+X_2+\cdots+X_n.
$
\end{example}
Because of the continuously distributed first step, $G_{alg}(\mathsf f)=\R$. But if we subtract from $\mathsf f$ the center-tight functional $\mathsf c$ with components
$$
c_n(x,y)=x \text{ when $n=1$ and }c_n(x,y)\equiv 0\text{ when $n>1$},
$$
then the result $\mathsf g:=\mathsf f-\mathsf c$ has algebraic range $2\Z$. So $\mathsf f$ is reducible.

The reduction $\mathsf g$ satisfies the lattice local limit theorem (see the preface), because it generates the (delayed) simple random walk. But by the assumptions on $\mathfrak F$, the original functional $\mathsf f=\mathsf g+\mathsf c$ does not satisfy the LLT, lattice or non-lattice. This can be seen by direct calculation from the observation that the distribution of $S_n$ is  the convolution of $\mathfrak F$ and the centered binomial distribution. See  chapter \ref{Chapter-reducible} for details.

\medskip
Here we see  an instance of the  {\bf reducibility obstruction}\index{obstructions to the LLT!reducibility} to the local limit theorem: A situation when the LLT fails because the additive functional is a sum of a lattice term which satisfies the lattice LLT and a non-lattice center-tight term which spoils it. The reducibility obstruction to the LLT raises the following questions:
\begin{enumerate}
\item Given an additive functional $\mathsf f$, how small can we make its algebraic range by subtracting from it a center-tight term?

\medskip
\item Is there an ``optimal" center-tight functional $\mathsf c$ such that the algebraic range of  $\mathsf f-\mathsf c$ cannot be reduced further?
\end{enumerate}

Motivated by these questions, we introduce the following definitions.
The {\bf essential range} of  $\mathsf f$ is
$$
G_{ess}(\mathsf X,\mathsf f):=\bigcap \left\{G_{alg}(\mathsf X,\mathsf g): \mathsf f-\mathsf g\text{ is center tight}\right\}.
$$
This is a closed sub-group of $G_{alg}(\mathsf X,\mathsf f)$.

An additive functional without reductions is called {\bf irreducible}. Equivalently, $\mathsf f$ is irreducible iff
$
G_{ess}(\mathsf X,\mathsf f)=G_{alg}(\mathsf X,\mathsf f).
$

In this terminology questions 1 and 2 call for the calculation of $G_{ess}(\mathsf X,\mathsf f)$ and ask for an irreducible reduction of $\mathsf f$.

\section{Main results}
\subsection{Results for Markov chains}
The questions raised at the end of the last section can be answered using the structure constants $d_n(\xi)$ introduced in \eqref{Structure-Constants}.
Define the {\bf co-range}\index{co-range} of $\mathsf f$ to be the set
$$
H(\mathsf X,\mathsf f):=\{\xi\in\R: \sum_{n=3}^\infty d_n(\xi)^2<\infty\}.
$$
\begin{theorem}\label{Theorem-co-range}
Let $\mathsf f$ be an a.s. uniformly bounded  additive functional on a uniformly elliptic Markov chain $\mathsf X$. If $\mathsf f$ is center-tight then $H(\mathsf X,\mathsf f)=\R$, and
if not then either $H(\mathsf X,\mathsf f)=\{0\}$, or  $H(\mathsf X,\mathsf f)=t\Z$ for  some
$t\geq \pi/(6\ess\sup|f|)$.
\index{center-tightness!and the co-range}
\end{theorem}

\begin{theorem}\label{Theorem-essential-range}
Let $\mathsf f$ be an a.s. uniformly bounded  additive functional on a uniformly elliptic Markov chain $\mathsf X$, then
\begin{enumerate}[(a)]
\item If $H(\mathsf X,\mathsf f)=0$, then $G_{ess}(\mathsf X, \mathsf  f)=\R$.
\item If $H(\mathsf X,\mathsf f)=t\Z$ with $t\neq 0$, then $G_{ess}(\mathsf X, \mathsf f)=\frac{2\pi}{t}\Z$.
\item If $H(\mathsf X,\mathsf f)=\R$, then $G_{ess}(\mathsf X, \mathsf f)=\{0\}$.
\end{enumerate}
\end{theorem}

\begin{theorem}\label{Theorem-minimal-reduction}
Let $\mathsf f$ be an a.s. uniformly bounded  additive functional on a uniformly elliptic Markov chain $\mathsf X$. Then there exists an  irreducible uniformly bounded additive functional $\mathsf g$  such that $\mathsf f-\mathsf g$ is center-tight, and
$$G_{alg}(\mathsf X, \mathsf g)=G_{ess}(\mathsf X, \mathsf g)=G_{ess}(\mathsf X, \mathsf f).$$
\end{theorem}

\begin{corollary}\label{Cor-delta-f}
Let $\mathsf f$ be an a.s. uniformly bounded  additive functional on a uniformly elliptic Markov chain $\mathsf X$.
If $G_{ess}(\mathsf X,\mathsf f)=t\Z$ with $t\neq 0$, then $|t|\leq 12\ess\sup|\mathsf f|$.
\end{corollary}
\noindent
The corollary follows directly from  Theorems \ref{Theorem-co-range} and \ref{Theorem-essential-range}(b).

\subsection{Results for Markov arrays}\label{Section-Reducibility-Results-Arrays}

The previous discussion applies to Markov arrays.
Let  $\mathsf f$ be an additive functional on a Markov array $\mathsf X$  with row lengths $k_N+1$:
\begin{enumerate}[(1)]
\item The {\bf algebraic range}\index{algebraic range!arrays}\index{range!algebraic} $G_{alg}(\mathsf X,\mathsf f)$ is the intersection of all closed subgroups $G$ of $\R$ such that for all $1\leq k\leq k_N, N\geq 1$
$$\exists c^{(N)}_k\in\R\text{ s.t. }\Prob[f^{(N)}_k(X^{(N)}_k,X^{(N)}_{k+1})-c^{(N)}_k\in G]=1.
$$

\item The {\bf essential range}\index{essential range} $G_{ess}(\mathsf X,\mathsf f)$ is the intersection of the algebraic ranges of all additive functionals of the form $\mathsf f-\mathsf h$ where $\mathsf h$ is center-tight.

\item The {\bf co-range}\index{co-range}  is
$
H(\mathsf X,\mathsf f):=\{\xi\in\R: \sup\limits_N\sum\limits_{k=3}^{k_N} d_k^{(N)}(\xi)^2<\infty\}.
$

\medskip
\item An additive functional $\mathsf f$ is called {\bf irreducible}\index{irreducible} if $G_{ess}(\mathsf X,\mathsf f)=G_{alg}(\mathsf X,\mathsf f)$.
\end{enumerate}
This is consistent with the definitions for Markov chains, see Corollary \ref{Corollary-Essential-Range-for-Arrays} below.

\begin{theorem}\label{Theorem-Results-for-arrays}
The results of Theorems \ref{Theorem-co-range}, \ref{Theorem-essential-range} \ref{Theorem-minimal-reduction} and of Corollary \ref{Cor-delta-f}
hold  for all  a.s. uniformly bounded additive functionals on uniformly elliptic Markov arrays.
\end{theorem}

\begin{corollary}\label{Corollary-Essential-Range-for-Arrays}
Suppose $\mathsf f=\{f_n\}$ is an a.s. uniformly bounded  additive functional on a uniformly elliptic Markov chain $\mathsf X=\{X_n\}$. Let $\wt{\mathsf f}=\{f^{(N)}_n\}$ be an additive functional on a Markov array $\wt{\mathsf X}=\{X^{(N)}_n\}$ s.t. $f^{(N)}_n=f_n$ and  $X^{(N)}_n=X_n$.
Then
$$
G_{alg}(\wt{\mathsf X}, \wt{\mathsf f})=G_{alg}({\mathsf X}, {\mathsf f})\ , \
G_{ess}(\wt{\mathsf X}, \wt{\mathsf f})=G_{ess}({\mathsf X}, {\mathsf f})\ , \
H(\wt{\mathsf X}, \wt{\mathsf f})=H({\mathsf X}, {\mathsf f}).
$$
\end{corollary}
\begin{proof}
The equality of the algebraic ranges and co-ranges is trivial, but the equality of the essential ranges requires justification, because some  center-tight functionals of $\{X^{(N)}_n\}$ are not of the form $h^{(N)}_n=h_n$.

However, since the co-ranges agree, the essential ranges must also agree,  by the version of
Theorem \ref{Theorem-essential-range}  for  arrays.\qed
\end{proof}

\subsection{Hereditary arrays}\label{Section-Hereditary}

Some results for Markov chains do not extend to general Markov arrays. Of particular importance is the following fact, which we need for the proof of the LLT (see the proof of Theorem \ref{ThLLT-classic}, claim 2). Recall the definition of $D_N(\xi)$ from \eqref{Structure-Constants}.
\begin{theorem}\label{Theorem-MC-array-difference}
Suppose $\mathsf f$ is an a.s. uniformly bounded additive functional on a uniformly elliptic Markov {\em chain} $\mathsf X$, then
\begin{equation}\label{D_N-to-infinity}
D_N(\xi)\xrightarrow[N\to\infty]{}\infty \text{ uniformly on compact subsets of  }\R\setminus H(\mathsf X,\mathsf f).
\end{equation}
\end{theorem}
\begin{proof}
Suppose $\xi\in\R\setminus H(\mathsf X,\mathsf f)$, then $\displaystyle\sup_{N} D_N(\xi)=\infty$, whence
$$
D_N(\xi)=\sum\limits_{k=3}^N d_k^{(N)}(\xi)^2\xrightarrow[N\to\infty]{}\sum\limits_{k=3}^\infty d_k^{(N)}(\xi)^2\equiv\sup_N D_N(\xi)=\infty.
$$
Since $D_N(\xi)$ is non-decreasing and $\xi\mapsto D_N(\xi)$ are continuous, the convergence is uniform on compact subsets of $\R\setminus H(\mathsf X,\mathsf f)$.\qed \end{proof}

The following two examples show that Theorem \ref{Theorem-MC-array-difference} fails for some arrays:

\begin{example}
\label{ExIIDHer}
Let $X_n$ be a sequence of independent  uniform
random variables with zero mean and variance equal to one. Form an array  by setting
$$
X^{(N)}_k=\begin{cases}
X_k & 1\leq k\leq N+1, N\text{ odd }\\
0 & 1\leq k\leq N+1, N\text{ even }
\end{cases}\ \ \ \ \ (k=1,\ldots,N)
$$
and let $f_k^{(N)}(x,y):=x.$
 Then for every $0\neq \xi\in\R\setminus H(\mathsf X,\mathsf f)$, $D_N(\xi)\not\to\infty$.
\end{example}
\begin{proof}
We  claim that  $\sup\limits_N D_{2N+1}(\xi)=\infty$  for every $\xi\neq 0$.

To see this,
suppose $P={\left(X_{n-2}\begin{array}{l} X_{n-1}\\ Y_{n-1}\end{array} \begin{array}{l} X_{n}\\ Y_{n}\end{array} ,Y_{n+1}\right)}$
    is a random level $2N+1$ hexagon at position $n$, then $\Gamma(P)=X_{n-1}+X_n-Y_{n-1}-Y_n$ where $X_i,Y_j$ are independent  random variables each having uniform distribution  with mean zero and unit variance. So $\Gamma(P)$ is a {\it non lattice random variable}
  and
  for every $\xi\neq 0$,  $d_n^{(2N+1)}(\xi)^2=\E(|e^{i\xi\Gamma(P)}-1|^2)=c(\xi)$, where $c(\xi)$  is a positive constant independent of $n$. So
  $$
  D_{2N+1}(\xi)=(2N-1)c(\xi)\xrightarrow[N\to\infty]{}\infty.
  $$
Thus $H(\mathsf X,\mathsf f)=\{0\}$. But $D_N(\xi)\not\to\infty$ for $\xi\neq 0$, because $D_{2N}(\xi)=0$.\qed\end{proof}

\begin{example}\label{Example-Non-Stably-Hereditary}
Suppose $X_n$ are a sequence of independent identically distributed random variables, equal to $\pm 1$ with probability $\frac{1}{2}$. Form an array with row lengths $N+1$ by setting $X^{(N)}_n=X_n$, and let
$$
f^{(N)}_n(X_n,X_{n+1}):=\frac{1}{2}\left(1+\frac{1}{\sqrt[3]{N}}\right)X_n\ \ (1\leq n\leq N+1).
$$
Then $D_N(\xi)\to\infty$ for all $\xi\not\in H(\mathsf X,\mathsf f)$, but the convergence is not uniform on compact subsets of $\R\setminus H(f)$.
\end{example}
\begin{proof}
$\Gamma\left( +1
\begin{array}{l} +1\\ +1\end{array} \begin{array}{l} +1\\ -1\end{array} +1\right)=1+N^{-1/3}$. Since
$\mathrm{Hex}(N,n)$ consists of $2^6$ hexagons, the hexagon  $\left( +1
\begin{array}{l} +1\\ +1\end{array} \begin{array}{l} +1\\ -1\end{array} +1\right)$ has probability $2^{-6}$. It follows that
\begin{align*}
d_n^{(N)}(\xi)&\geq 2^{-6}|e^{i\xi(1+N^{-1/3})}-1|^2=\frac{1}{16}\sin^2\frac{\xi(1+N^{-1/3})}{2}\\
D_N(\xi)&\geq \frac{N-2}{16}\sin^2\frac{\xi(1+N^{-1/3})}{2}\sim \begin{cases}
16^{-1} N \sin^2\frac{\xi}{2} & \xi\not\in 2\pi\Z\\
16^{-1} \sqrt[3]{N} & \xi\in 2\pi\Z.
\end{cases}
\end{align*}
We see that $D_N(\xi)\to\infty$ for all $\xi\neq 0$, whence  $H(\mathsf X,\mathsf f)=\{0\}$,  and $D_N(\xi)\to\infty$ for all $\xi\not\in H(\mathsf X,f)$. But the convergence is not uniform on any compact neighborhood of $2\pi k$, $k\neq 0$, because
 $D_N(\xi_N)\equiv 0$ for $\xi_N=2\pi k(1+N^{-1/3})^{-1}\to 2\pi k$.\qed
\end{proof}

Because of the importance of property \eqref{D_N-to-infinity} to the proof of the LLT, we would like to  characterize the additive functionals on Markov arrays which satisfy it. Examples 1 and 2 point the way.

Let $\mathsf X$ be a Markov array with row lengths $k_N$.
 A {\bf  sub-array}\index{sub-array}\index{Markov array!sub-arrays} of $\mathsf{X}$ is an array $\mathsf X'$ of the form
$\{X^{(N_\ell)}_k: 1\leq k\leq k_{N_\ell}+1, \ell\geq 1\}$ where $N_\ell\uparrow\infty$.
The {\bf restriction}\index{restriction} of $\mathsf f$  to $\mathsf{X}'$ is
$$
\mathsf f|_{\mathsf{X}'}=\{f^{(N_\ell)}_k: 1\leq k\leq k_{N_\ell},\ell\geq 1\}.
$$
$(\mathsf X,\mathsf f)$ is called {\bf hereditary}\index{hereditary}, if $G_{ess}(\mathsf X', \mathsf f|_{\mathsf{X}'})=G_{ess}(\mathsf X, \mathsf f)$ for all sub-arrays
$\mathsf{X}'$, and  {\bf stably hereditary}\index{stably hereditary}\index{hereditary!stably hereditary} if $(\mathsf X,\mathsf g)$ is hereditary whenever  $\mathsf g=\{(1+\epsilon_N)f^{(N)}_k\}$ with $\epsilon_N\to 0$.

\begin{theorem}\label{Theorem-Hereditary-Property}
Let $\mathsf f$ be an a.s. uniformly bounded additive functional on a uniformly elliptic Markov array $\mathsf{X}$, then the following conditions are equivalent:
\begin{enumerate}[(1)]
\item $\mathsf f$ is hereditary;
\item for all $\xi$, $\liminf\limits_{N\to\infty}\sum\limits_{k=3}^{k_N} d_k^{(N)}(\xi)^2<\infty\Rightarrow \limsup\limits_{N\to\infty}\sum\limits_{k=3}^{k_N} d_k^{(N)}(\xi)^2<\infty$;
\item for all $\xi\not\in H(\mathsf X,\mathsf f)$, $D_N(\xi)\xrightarrow[N\to\infty]{}\infty$;
\item $H(\mathsf X',f|_{\mathsf{X}'})=H(\mathsf X,f)$ for every sub-array $\mathsf{X}'$ of $\mathsf{X}$.
\end{enumerate}
In addition, $\mathsf f$ is stably hereditary iff the convergence in (3) is uniform on compact subsets of $\R\setminus H(\mathsf X, \mathsf f)$.
\end{theorem}

\begin{example}{\bf (Markov chains):}\label{Example-MC-hereditary} Suppose $\mathsf f$ is an a.s. uniformly bounded additive functional on a uniformly elliptic Markov array $\mathsf X$. If $f^{(N)}_n=f_n$ and $X^{(N)}_n=X_n$, then $\mathsf f$ is stably hereditary.
\end{example}
{\em Proof.\/} This follows from Theorems  \ref{Theorem-MC-array-difference} and
\ref{Theorem-Hereditary-Property}.\qed

\begin{example}{\bf (``Change of measure"):} \label{Example-COM-is-Hereditary}
\index{change of measure}
Let $\mathsf Y$ be an array obtained from a Markov chain $\mathsf X$ using the change of measure construction
(example \ref{Example-Change-Of-Measure}). Let $\vf^{(N)}_n$ denote the weights of the change of measure.
If  $\exists C>0$ s.t.
$$C^{-1}<\vf_n^{(N)}<C\text{ for all $n,N$},$$   then for every a.s. uniformly bounded additive functional  $\mathsf f$ on $\mathsf X$, the additive functional $f^{(N)}_n:=f_n$ is  stably hereditary on $\mathsf{Y}$
.
\end{example}
\begin{proof}
If $d_n(\xi,\mathsf{X})$ are the structure constants of $\mathsf f$ on $\mathsf{X}$, and $d_n^{(N)}(\xi,\mathsf{Y})$ are the structure constants of $\mathsf f$ on $\mathsf{Y}$, then
$C^{-6}d_n(\xi,\mathsf{X})\leq d_n^{(N)}(\xi,\mathsf{Y})\leq C^6 d_n(\xi,\mathsf{X}).$ So $H(\mathsf{Y},\mathsf f)=H(\mathsf{X}, \mathsf f)$.

 Theorem \ref{Theorem-MC-array-difference} says that
$D_N(\xi,\mathsf{X})\to\infty$ uniformly on compact subsets of $\R\setminus H(\mathsf{X},\mathsf f)$. Since
$D_N(\xi,\mathsf{Y})\geq C^{-6} D_N(\xi,\mathsf{X})$,  $D_N(\xi,\mathsf{X})\to\infty$ uniformly on compact subsets of $\R\setminus H(\mathsf{Y},\mathsf f)$.\qed
\end{proof}

Sometimes (but not always, Example \ref{Example-Non-Stably-Hereditary}), every hereditary functional is stably hereditary:
\begin{theorem}\label{Theorem-Sufficient-Conditions-for-Stable-Heredity}
 Let $\mathsf f$ be an  a.s. uniformly bounded  additive functional on a uniformly elliptic  Markov array $\mathsf{X}$.
\begin{enumerate}[(a)]
\item  Suppose $G_{ess}(\mathsf X,\mathsf f)=t\Z$ or $\{0\}$. If $\mathsf f$ is  hereditary then
$\mathsf f$ is stably hereditary.
\item  Suppose  $\mathsf f$ is integer valued and not  center-tight, and  $|\mathsf f|\leq K$,  then $G_{ess}(\mathsf X, \mathsf f)=\frac{k}{2\pi}\Z$ for some $0<k\leq 12K$, and if $\mathsf f$ is hereditary then $\mathsf f$ is stably hereditary.
 \end{enumerate}
\end{theorem}

\section{Proofs}
\subsection{Reduction lemmas}\label{Section-Reduction-Lemmas}

\begin{lemma}\label{l.alg-range}
Let $\mathsf f$ be an additive functional on a Markov array $\mathsf X$ with row lengths $k_N+1$. For every $N\geq 1$ and $1\leq n\leq k_N$,
there exists $c_k^{(N)}\text{ s.t. }
f^{(N)}_n(X_n^{(N)},X_{n+1}^{(N)})-c^{(N)}_n\in G_{alg}(\mathsf X,\mathsf f)\text{ almost surely}.$
\end{lemma}
\begin{proof}
$G_{alg}(\mathsf X,\mathsf f)$ is the intersection of all closed subgroups $G$ such that
\begin{equation}\label{alg-range2}
\exists c_k^{(N)}\text{ s.t. }
f^{(N)}_n(X_n^{(N)},X_{n+1}^{(N)})-c^{(N)}_n\in G\text{ almost surely}.
\end{equation}
This is a closed subgroup of $\R$.
The lemma is trivial  when $G_{alg}(\mathsf X,\mathsf f)=\R$ (take $c^{(N)}_n\equiv 0$), so we focus on the  case $G_{alg}(\mathsf X,\mathsf f)\neq \R$.

In this case \eqref{alg-range2} holds with some $G=t\Z$ with $t\geq 0$, and  $f^{(N)}_n(X_n,X_{n+1})$ must be a discrete random variable. Let $A^{(N)}_n$ denote the set of values attained by $f^{(N)}_n(X_n,X_{n+1})$ with positive probability. Since $G=t\Z$ satisfies \eqref{alg-range2}, $A^{(N)}_n\subset$coset of $t\Z$, and
$
D^{(N)}_n:=A^{(N)}_n-A^{(N)}_n\subset t\Z.$
Let $G_0$ denote the group generated by $\bigcup_{N\geq 1}\bigcup_{1\leq n\leq k_N}D^{(N)}_n$. Then $G_0$ is a subgroup of $t\Z$. In particular, $G_0$ is closed.

By the previous paragraph, $G_0\subset t\Z$ for any group $t\Z$ which satisfies \eqref{alg-range2}. So $G_0\subseteq G_{alg}(\mathsf X,\mathsf f)$.
Next, we fix $n, N$ and observe  that  all the values of $f^{(N)}_n(X_n,X_{n+1})$ belong to the same translate of $A_n^{(N)}-A_n^{(N)}$, and therefore to the same coset of $G_0$. So $G_0$ satisfies \eqref{alg-range2}, and $G_0\supset G_{alg}(\mathsf X,\mathsf f)$.
So $G_{alg}(\mathsf X,\mathsf f)=G_0$.
Since $G_0$ satisfies \eqref{alg-range2},  $G_{alg}(\mathsf X,\mathsf f)$ satisfies \eqref{alg-range2}.
\qed\end{proof}

\begin{lemma}[Reduction Lemma]\label{Lemma-Reduction}
\index{Reduction lemma}
Let $\mathsf f$ be an a.s. uniformly bounded additive functional on a uniformly elliptic Markov array $\mathsf{X}$. If $\xi\neq 0$ and $\displaystyle\sup_N\sum_{k=3}^{k_N} d_k^{(N)}(\xi)^2<\infty$, then there exists a uniformly bounded additive functional $\mathsf g$ on $\mathsf X$ s.t.
$$
\mathsf f-\mathsf g\text{ is center-tight, and }G_{alg}(\mathsf g)\subset \frac{2\pi}{\xi}\Z.$$
If $X^{(N)}_n=X_n$ { and $f_n^{(N)}=f_n$} (as in the case additive functionals of Markov chains), then we can take $\mathsf g$ such that $g^{(N)}_n=g_n$.
\end{lemma}

\medskip
\noindent
{\bf Proof for Doeblin chains:} As in the case of the gradient lemma, the reduction lemma has a particularly simple proof in the important special case of Doeblin Markov chains (Example \ref{Example-Doeblin-Chains}).
Recall that Doeblin chains have finite state spaces $\fS_n$. Let $\pi_{xy}^n:=\pi_{n,n+1}(x,\{y\})$, and   relabel the states $\fS_n=\{1,\ldots,d_n\}$ in such a way that $\pi_{11}^n=\pi_{n,n+1}(1,\{1\})\neq 0$ for all $n$. The Doeblin condition guarantees that  for every $x\in\fS_n$, there exists  a state $\xi_n(x)\in\fS_{n+1}$ such that $\pi_{1,\xi_n(x)}^{n-1} \pi_{\xi_n(x),1}^{n}>0$.

Define as in the proof of the gradient lemma,
\begin{align*}
& a_0\equiv 0, \ \ a_1\equiv 0,\ \ \text{ and }  a_n(x):=f_{n-2}(1,\xi_{n-1}(x))+f_{n-1}(\xi_{n-1}(x),x)\text{ for }n\geq 3\\
& c_0:=0,\ \  c_1:=0,\ \ \text{ and } c_n:=f_{n-2}(1,1)\text{ for }n\geq 3\\
&  \wt{\mathsf f}:=\mathsf f-\nabla\mathsf a-\mathsf c.
\end{align*}
Then $\displaystyle \wt{f}_n(x,y)=f_n(x,y)-(a_{n+1}(y)-a_n(x))-c_n=-\Gamma_n\left(1\  \begin{array}{c} 1\\
\xi_{n-1}(x)
\end{array} \begin{array}{c}{\xi_n(y)}\\
x
\end{array}\  y\right)
$, where $\Gamma_n$ denotes the  balance of a hexagon, see \eqref{balance}.

For Doeblin chains, there are finitely many admissible hexagons at position $n$, and  the hexagon measure assigns each of them a mass which is uniformly bounded from below. Let $C^{-1}$ be a uniform lower bound for this mass, then
\begin{align*}
&|e^{i\xi \wt{f}_n(x,y)}-1|^2\leq C\E(|e^{i\xi\Gamma_n}-1|^2)=Cd_n^2(\xi).
\end{align*}

Decompose $\wt{f}_n(x,y)=g_n(x,y)+h_n(x,y)$ where
$
g_n(x,y)\in\frac{2\pi}{\xi}\Z\text{ and }h_n(x,y)\in [-\frac{\pi}{\xi},\frac{\pi}{\xi}).
$
Clearly  $|\mathsf g|\leq |\mathsf f|+|\nabla a|+|\mathsf c|+|\mathsf h|\leq 6|\mathsf f|+\pi/\xi$, and $G_{alg}(\mathsf X,\mathsf h)\subset \frac{2\pi}{\xi}\Z$.

We show that $\mathsf f-\mathsf g$ is center tight.
We need the following  inequality:\footnote{Proof of \eqref{eix-1-inequality}: Since $y=\sin x$ is concave on $[0,\frac{\pi}{2}]$, its graph lies above the chord $y=2x/\pi$ and below the tangent $y=x$. So $2x/\pi\leq\sin x\leq x$ on $[0,\frac{\pi}{2}]$. Now use the identity $|e^{ix}-1|^2=2(1-\cos x)=4\sin^2\frac{x}{2}.$
}
\begin{equation}\label{eix-1-inequality}
\frac{4x^2}{\pi^2}\leq |e^{ix}-1|^2\leq x^2\text{ for all $|x|\leq \pi$}.
\end{equation}
By \eqref{eix-1-inequality}, $|h_n(x,y)|^2\leq \frac{\pi^2}{4\xi^2}|e^{i\xi h_n(x,y)}-1|^2=\frac{\pi^2}{4\xi^2}|e^{i\xi \wt{f}_n(x,y)}-1|^2\leq C\frac{\pi^2}{4\xi^2} d_n^2(\xi)$, whence
 $$
 \displaystyle\sum_{n=3}^\infty \Var(h_n(X_n,X_{n+1})+c_n)=\sum_{n=3}^\infty \Var(h_n(X_n,X_{n+1}))\leq \frac{C\pi^2}{4\xi^2}\sum_{n=3}^\infty d_n^2(\xi)<\infty
 .$$ So  $\mathsf h+\mathsf c$ has summable variance. Therefore $\mathsf f-\mathsf g=\nabla a + (\mathsf h+\mathsf c)$ is center tight. \qed

\medskip
\noindent
{\bf Preparations for the proof in the general case.}

\begin{lemma}\label{Lemma-Number-Of-Events}
Suppose $E_1,\ldots,E_N$ are measurable events, and let $W$ denote the random variable which counts how many of $E_i$ occur simultaneously, then
$$
\Prob(W\geq t)\leq \frac{1}{t}\sum_{k=1}^N\Prob(E_k).
$$
\end{lemma}
\begin{proof}
Apply Markov's inequality to $W=\sum 1_{E_k}$.\qed
\end{proof}

Suppose $W$ is a real-valued random variable. A {\bf circular mean}\index{circular mean} of $W$ is a real number $\theta\in [-\pi,\pi)$ which minimizes the quantity $\E(|e^{i(W-\theta)}-1|^2)$. Such numbers always exist,  because $\theta\mapsto \E(|e^{i(W-\theta)}-1|^2)$ is continuous and $2\pi$-periodic. But circular means are not unique:  If, for example,  $W$ is uniformly distributed on $[-\pi,\pi]$,  then every $\theta\in [-\pi,\pi)$ is a circular mean.

The {\bf circular variance}\index{circular variance}\index{variance!circular} of a real random variable $W$ is defined to be
$$
\CVar(W):=\min_{\theta\in [-\pi,\pi)}\E(|e^{i(W-\theta)}-1|^2)\equiv \min_{\theta\in [-\pi,\pi)}4\E\bigl(\sin^2\tfrac{W-\theta}{2}\bigr).
$$
For every $x\in\R$, let
\begin{equation}\label{fractional-part}
\<x\>:=\text{unique element of $[-\pi,\pi)$ s.t. $x-\<x\>\in 2\pi\Z$.}
\end{equation}
It is not difficult to see, using \eqref{eix-1-inequality}, that for every circular mean $\theta$
\begin{equation}\label{CVar-inequality}
 \frac{4}{\pi^2}\Var\<W-\theta\>\leq\CVar(W)\leq \Var(W).
\end{equation}

\begin{lemma}
For every real-valued random variable $W$, we can write  $W=W_1+W_2$ where
$W_1\in 2\pi\Z$ almost surely, and $\Var(W_2)\leq \frac{\pi^2}{4}\CVar(W)$.
\end{lemma}
\begin{proof}
$W_1:=(W-\theta)-\<W-\theta\>$, $W_2:=\<W-\theta\>+\theta$,
 $\theta:=$ a circular mean. \qed
\end{proof}

\medskip
\noindent
{\bf Proof of the Reduction Lemma in the general case:} Suppose $\mathsf f$ is an a.s. uniformly bounded additive functional on a uniformly elliptic Markov array $\mathsf X$, with row lengths $k_N$, and fix $\xi\neq 0$ such that
$$
\sup_N \sum_{n=3}^{k_N} d_n^{(N)}(\xi)^2<\infty.
$$

Let $\mathsf L$ denote the ladder process associated to $\mathsf X$ (see
section \ref{Section-Ladder}). We remind the reader that  this is a Markov array with entries $\un{L}^{(N)}_n=(Z^{(N)}_{n-2},Y^{(N)}_{n-1},X^{(N)}_n)$ $(3\leq n\leq k_N)$,  and  for every $N$: (a) $\{X^{(N)}_n\}$, $\{Z^{(N)}_n\}$ are two independent copies of $\mathsf X^{(N)}$; (b) $Y^{(N)}_n$ are conditionally independent given $\{X^{(N)}_i\}$ and $\{Z^{(N)}_i\}$; and (c) the conditional distribution of $Y^{(N)}_n$ given $\{Z^{(N)}_i\}$ and $\{X^{(N)}_i\}$ is given by
$$
\Prob\left(Y^{N}_{n-1}\in E\bigg|\begin{array}{l}
\{Z_{i}^{(N)}\}=\{\zeta^{(N)}_i\}\\
\{X_{i}^{(N)}\}=\{\xi^{(N)}_i\}
\end{array}\right)=\begin{array}{l}
\text{bridge probability for $\mathsf X$ that $X^{(N)}_{n-1}\in E$}\\
\text{given that $X_{n-2}^{(N)}=\zeta^{(N)}_{n-2}$ and $X^{(N)}_{n}=\xi^{(N)}_n$.}\\
\text{(see \S\ref{Section-Bridge}).}
\end{array}
$$
Let $\mathsf F,\mathsf H$ be the additive functionals on $\mathsf L$ with entries
$$
\begin{aligned}
&F^{(N)}(\un{L}_n):=f_{n-2}^{(N)}(Z_{n-2}^{(N)},Y_{n-1}^{(N)})+f_{n-1}^{(N)}(Y_{n-1}^{(N)},X_n^{(N)})\\
&H_n^{(N)}(\un{L}_n^{(N)},\un{L}_{n+1}^{(N)}):=\left\<\xi\Gamma\left(Z_{n-2}^{(N)}\ \ {\begin{array}{l}
 Z_{n-1}^{(N)}\\
 Y_{n-1}^{(N)}
 \end{array}}\ \
 {\begin{array}{l}
 Y_n^{(N)}\\
 X_{n}^{(N)}
 \end{array}}\ \ X_{n+1}^{(N)}\right)
\right\>
\end{aligned}
\ \ \ \  (3\leq n\leq k_N,\ N\geq 1)
$$
(see \eqref{balance} and \eqref{fractional-part}). Clearly $\ess\sup|\mathsf F|\leq 2\ess\sup|\mathsf f|$ and $|\mathsf H|\leq \pi$.

\medskip
\noindent
{\sc Step 1:}
$|\E(H_n^{(N)})|\leq \frac{\pi}{4}d_n^{(N)}(\xi)^2$,
$\E[(H_n^{(N)})^2]\leq \frac{\pi^2}{4}d_n^{(N)}(\xi)^2$, {\em  and  }
$$\sup_N \E[(H_3^{(N)}+\cdots+H_{k_N}^{(N)})^2]<\infty.$$

\medskip
\noindent
{\sc Proof of step 1.\/}
We fix $N$ and drop the superscripts ${(N)}$.

The map $
\imath:\left(Z_{n-2}\ {\begin{array}{l}
 Z_{n-1}\\
 Y_{n-1}
 \end{array}}\
 {\begin{array}{l}
 Y_n\\
 X_{n}
 \end{array}},X_{n+1}\right)
 \mapsto \left(Z_{n-2}\ {\begin{array}{l}
 Y_{n-1}\\
 Z_{n-1}
 \end{array}}\
 {\begin{array}{l}
 X_n\\
 Y_{n}
 \end{array}},X_{n+1}\right)
$
preserves the natural measure on the space of hexagons, and is an involution: $\imath^2=id$. Clearly  $$\Gamma\circ\imath=-\Gamma.$$
Using the partial symmetry $\<-x\>=-\<x\>\text{ for all }x\not\in -\pi+2\pi\Z,$
we find that $H_n\circ\imath=-H_n$ on $[H_n\neq -\pi]$.
So $\E(H_n 1_{[H_n\neq -\pi]})=0$, and therefore
$$
|\E(H_n)|=\pi \Prob(H_n=-\pi)\leq \frac{\pi}{4}\E(|e^{i H_n}-1|^2)=\frac{\pi}{4}\E(|e^{i\xi\Gamma}-1|^2)=\frac{\pi}{4}d_n(\xi)^2,
$$
which is the first statement we needed to show.

Next we observe from \eqref{eix-1-inequality} that
$
\E(H_n^2)\leq \frac{\pi^2}{4}\E(|e^{i\xi\Gamma}-1|^2)=\frac{\pi^2}{4} d_n^{(N)}(\xi)^2
$, which is the second statement we had to prove.

The two statements already proven and the boundedness of $d_n$
show that there is a constant $C$ s. t.  $\DS \Var(H_n)\leq C d_n^2(\xi)^2.$
Now the third statement follows from Lemma~\ref{LmVarSum}.
The proof of step 1 is complete.

\medskip
From now on, fix a constant $D$ such that
$$
\sup_N \sum_{n=3}^{k_N} d_n^{(N)}(\xi)^2+\sup_N \E\left[\left(\sum_{n=3}^{k_N} H_n^{(N)}\right)^2\right]<D.
$$

\medskip
\noindent
{\sc Step 2:} {\em For every $N\geq 1$ there exists
$\DS \un{\zeta}^{(N)}=(\zeta^{(N)}_1,\ldots,\zeta^{(N)}_{k_N+1})\in \prod_{i=1}^{k_N+1}\fS_i^{(N)}$ s.t.  }
\begin{align*}
&\sum_{n=3}^{k_N} \E\biggl(H_n^{(N)}(\un{L}_n^{(N)},\un{L}_{n+1}^{(N)})^2\bigg|\{Z_i^{(N)}\}=\un{\zeta}^{(N)}\biggr)<\pi^2 D,\\
&\E\left[\left(\sum_{n=3}^{k_N} H_n^{(N)}(\un{L}_n^{(N)},\un{L}_{n+1}^{(N)})\right)^2\bigg|\{Z_i^{(N)}\}=\un{\zeta}^{(N)}\right]<\pi^2 D,
\end{align*}
\begin{align*}
&\E_X\left[\sum_{n=3}^{k_N}\CVar\left(\xi F_n^{(N)}(\un{L}_n^{(N)})\bigg|\{Z_i^{(N)}\}=\un{\zeta}^{(N)}, X_n^{(N)}\right)\right]<\pi^2 D\text{ and }\\
&|f^{(N)}_n(\zeta^{(N)}_n,\zeta^{(N)}_{n+1})|\leq \ess\sup|f|\text{ for all }3\leq n\leq k_N.
\end{align*}
{\em Here and throughout  $\un{L}_n=(Z_{n-2}^{(N)},Y_{n-1}^{(N)},X_n^{(N)})$, and $\E_X$ indicates averaging on $\{X_i^{(N)}\}$.}

\medskip
\noindent
{\sc Proof of Step 2.\/} We fix $N$ and drop the $^{(N)}$ superscripts.

Let
$
\Omega_1:=\left\{\un{\zeta}: \sum_{n=3}^{k_N}\E(H_n^2|\{Z_i\}=\un{\zeta})
\leq\pi^2 D\right\}.
$
By step 1,
$$
\E_{Z}\left[\E\left(\sum_{n=3}^{k_N} H_n^2\bigg|\{Z_i\}=\un{\zeta}\right)\right]= \sum_{n=3}^{k_N}\E(H_n^2)\leq\frac{\pi^2}{4}\sum_{n=3}^{k_N} d_n^{(N)}(\xi)^2\leq \frac{\pi^2}{4}D,
$$
where $\E_Z=$ integration over $\un{\zeta}$ with respect to the distribution of $\{Z^{(N)}_i\}$ (recall that $\{Z^{(N)}_i\}\overset{dist}{=}\{X^{(N)}_i\}$).
By Markov's inequality, $\Prob[\{Z_i^{(N)}\}\in\Omega_1]>\frac{3}{4}$.

Let $\Omega_2:=\{\un{\zeta}:\E\bigl[\bigl(\sum_{n=3}^{k_N} H_n(\un{L}_n,\un{L}_{n+1})\bigr)^2\big|\{Z_i\}=\un{\zeta}\bigr]<\pi^2 D\}$. As before, by Markov's inequality, $\Prob[\{Z_i^{(N)}\}\in \Omega_2]\geq 1-\frac{1}{\pi^2}$.

Let
$
\Omega_3:=\left\{\un{\zeta}: \E_X\biggl[\sum\limits_{n=3}^{k_N} \CVar\bigl(\xi F(\un{L}_{n+1})\big|\{Z_i\}=\un{\zeta}, X_{n+1}\bigr)\biggr]<\pi^2 D
\right\}
$,
$$
\theta^\ast(\un{L}_{n},X_{n+1},Z_{n-1}):=-\xi f_{n-2}(Z_{n-2},Z_{n-1})+\xi F(\un{L}_{n})+\xi f_{n}(X_{n},X_{n+1}).
$$
Then
$\exp[{iH_n(\un{L}_n,\un{L}_{n+1})}]=\exp[{i\xi F(\un{L}_{n+1})-i\theta^\ast (\un{L}_{n+1},X_{n+1},Z_{n-2})}].$

Given $X_{n+1}$ and $\{Z_i\}$,
$\un{L}_{n+1}$ is conditionally independent from $\un{L}_{n}$, $\{X_i\}_{i\neq n+1}$. So
\begin{align*}
&\E_{Z,X}\biggl(\CVar\biggl(\xi F(\un{L}_{n+1})\big|\{Z_i\},X_{n+1}\biggr)\biggr)
=\E\biggl(\CVar\biggl(\xi F(\un{L}_{n+1})\big|\un{L}_n,\{Z_i\},\{X_i\}\biggr)\biggr)\\
&\overset{!}{\leq} \E\biggl(\E\bigl(|e^{i\xi F(\un{L}_{n+1})-i\theta^\ast(\un{L}_n,X_{n+1},Z_{n-1})}-1|^2\big|
\un{L}_n,\{X_i\}, \{Z_i\}\bigr)
\biggr)\\
&=\E(|e^{i(\xi F(\un{L}_{n+1})-\theta^\ast)}-1|^2)\equiv\E(|e^{i H_n}-1|^2)=\E(|e^{i\xi\Gamma}-1|^2)=d_n(\xi)^2,
\end{align*}
where $\overset{!}{\leq}$ is because $\theta^\ast$ is conditionally constant.
So $$\E_Z\left[\E_X\left(\sum_{n=3}^{k_N}
\CVar\bigl(\xi F(\un{L}_{n+1})\big|\{Z_i\},X_{n+1})\right)\right]<D.$$
 By Markov's inequality,
 $\Prob(\{Z_i^{(N)}\}\in\Omega_3)\geq 1-\frac{1}{\pi^2}$.

Finally, let $\Omega_4:=\{\un{\zeta}: |f_n(\zeta_n,\zeta_{n+1})|\leq \ess\sup|\mathsf f|\}$, then $\Prob(\{Z_i^{(N)}\}\in\Omega_4)=1$.

In summary
 $\DS \Prob\left[\bigcup_{1\leq i\leq 4}\Omega_i^c\right]\leq \frac{2}{\pi^2}+\frac{1}{4}<1.$
Necessarily $\Omega_1\cap\Omega_2\cap\Omega_3\cap \Omega_4\neq \varnothing$. Any $\un{\zeta}=\un{\zeta}^{(N)}$ in the intersection satisfies the  requirements of step 2.

\medskip
\noindent
{\sc Step 3:\/} {\em There exist measurable functions $\theta_n^{(N)}:\fS_n^{(N)}\to [-\pi,\pi)$ s.t.
$$
\sum_{n=3}^{k_N}\E\biggl(|e^{i\xi F_n^{(N)}(\un{L}_n^{(N)})-i\theta_n^{(N)}(X_n)}-1|^2\bigg|\{Z_i^{(N)}\}=\un{\zeta}^{(N)}\biggr)
<2\pi^2D.$$
}

\medskip
\noindent
{\em Proof.} We fix $N$ and drop the $^{(N)}$ superscripts.

Clearly, $\theta\mapsto \E(|e^{i(W-\theta)}-1|^2)$ is continuous for every random variable $W$. So $\CVar(W)=\inf_{q\in\mathbb Q} \E(|e^{i(W-q)}-1|^2)$, an infimum over a countable set, whence
$
\CVar(\xi F_n|\{Z_i\}=\un{\zeta}, X_n)=\inf_{q\in\mathbb Q}\E(|e^{i\xi F(\un{L}_n)-iq}-1|^2|\{Z_i\}=\un{\zeta}, X_n=\xi_n).
$

The expectation can be expressed  explicitly using integrals with respect to the  bridge distributions, and this expression shows that
 $$
 \eta\mapsto\CVar(\xi F_n|\{Z_i\}=\un{\zeta}, X_n=\eta)
 $$ is measurable  on $\fS_n^{(N)}$.

Fix $N$ and  $\un{\zeta}=\un{\zeta}^{(N)}$. We say that $(\eta,q)\in  S_n^{(N)}\times \mathbb R$ have  ``property $P_n(\eta,q)$", if the following condition holds:
 \begin{equation}\tag{$P_n(\xi,q)$}
\begin{aligned}
& \E(|e^{i\xi F_n(\un{L}_n)-iq}-1|^2|\{Z_n\}=\un{\zeta}, X_n=\eta)\\ &\leq \CVar(\xi F_n(\un{L}_n)|\{Z_n\}=\un{\zeta}, X_n=\eta)+\frac{D}{n^2}
\end{aligned}
\end{equation}
By the previous paragraph, $\{\eta: P_n(\eta,q)\text{ holds}\}$ is measurable, and for every $\eta$ there exists $q\in\mathbb Q\cap(-\pi,\pi)$ such that $P_n(\eta,q)$ holds. Let
$$
 \theta_n(\eta)=\theta_n^{(N)}(\eta):=\inf\left\{ q: {q\in \mathbb Q\cap(-\pi,\pi)}\text{ s.t. }P_n(\eta,q)\text{ holds}\right\}.
$$
Again,  this is a measurable function, and  since for fixed $\eta$, $P_n(\eta,q)$ is a closed property of $q$,
  $\theta_n^{(N)}(\eta)$ itself  satisfies property $P_n(\eta,\theta_n^{(N)}(\xi))$.  So
\begin{align*}
 &\E_X\left[\sum_{n=3}^{k_N} \E(|e^{i\xi F_n(\un{L}_n)-i\theta_n^{(N)}(X_n)}-1|^2|\{Z_n\}=\un{\zeta}, X_n\right]\\
&\leq \E_X\left[\sum_{n=3}^{k_N}\CVar\biggl(\xi F_n(\un{L}_n)\bigg|\{Z_n\}=\un{\zeta}, X_n\biggr)\right]+\frac{\pi^2}{6}D\\
&<2\pi^2D, \text{ by choice of $\un{\zeta}$.}
\end{align*}

\medskip
\noindent
{\sc Step 4 (the reduction).\/} {\em  Let $\un{\zeta}=\un{\zeta}^{(N)}$, $\theta_n=\theta_n^{(N)}$, $f_n=f_n^{(N)}$, $F_n=F_n^{(N)}$, $X_n=X_n^{(N)}$, $Z_n=Z_n^{(N)}$. Define
\begin{align*}
&c_n^{(N)}:=f_n(\zeta_{n-2},\zeta_{n-1})\\
&a_n^{(N)}(x):=\frac{1}{\xi}\biggl[\theta_n(X_n)+\E\bigl(\<\xi F_n(\un{L}_{n})-\theta_n(X_n)\>\big|\{Z_i\}=\un{\zeta}, X_n=x\bigr)\biggr]\ \ (x\in\fS^{(N)}_n)\\
&\wt{\mathsf f}:=\frac{1}{\xi}\biggl<\xi \bigl(\mathsf f -\nabla \mathsf a-\mathsf c\bigr)\biggr>\\
&\mathsf g:=\mathsf f-\nabla\mathsf a-\mathsf c
-\wt{\mathsf f}.
\end{align*}
Then $\mathsf a, \mathsf c, \wt{\mathsf f},  \mathsf g$ are uniformly bounded, and $G_{alg}(\mathsf g)\subset \frac{2\pi}{\xi}\Z$. }

\medskip
\noindent
{\em Proof.\/} By choice of $\un{\zeta}^{(N)}$,  $|\mathsf c|\leq \ess\sup|\mathsf f|$, and by the definition of $\theta^{(N)}$ and $\<\cdot\>$, $|\mathsf a|\leq 2\pi/|\xi|$ and  $|\wt{\mathsf f}|\leq \pi/|\xi|$. It follows that  $
|\mathsf g|\leq 2\ess\sup|\mathsf f|+3\pi/|\xi|.
$
Next, $$\mathsf g\equiv \frac{1}{\xi}\biggl(\xi(\mathsf f-\nabla\mathsf a-\mathsf c)-\<\xi(\mathsf f-\nabla\mathsf a-\mathsf c)\>\biggr).$$
The term in the brackets belongs to $2\pi\Z$ by the definition of $\<\cdot\>$, so $G_{alg}(\mathsf g)\subset \frac{2\pi}{\xi}\Z$, and the proof of step 4 is complete.

\medskip
Notice that $\mathsf f-\mathsf g=\nabla \mathsf a+\mathsf c+\wt{\mathsf f}$. Gradients and constant functionals are center tight. So to complete the proof of the reduction lemma, it suffices to show:

\medskip
\noindent
{\sc Step 5:} {\em $\wt{\mathsf f}$ is center-tight.}

\medskip
\noindent
{\em Proof.\/} We fix $N$ and drop  the $^{(N)}$ superscripts.

We begin with a few  identities.
Suppose $\{Z^{(N)}_i\}=\{\zeta^{(N)}_i\}$, and consider the hexagon
 $P_n:=\left(Z_{n-2}\ { \begin{array}{l}
 Z_{n-1}\\
 Y_{n-1}
 \end{array}}\
 { \begin{array}{l}
 Y_n\\
 X_{n}
 \end{array}}\ X_{n+1}\right)=\left(\zeta_{n-2}\ { \begin{array}{l}
 \zeta_{n-1}\\
 Y_{n-1}
 \end{array}}\
 { \begin{array}{l}
 Y_n\\
 X_{n}
 \end{array}}\ X_{n+1}\right),$
 then $$-\Gamma(P_n)=-f_{n-2}(Z_{n-2},Z_{n-1})-F_{n+1}(\un{L}_{n-1})+F_n(\un{L}_n)+f_n(X_n,X_{n+1}),$$ whence
 \begin{align*}
 &\xi\tf_n(X_n,X_{n+1})=\biggl\<\xi\bigl(-\Gamma(P_n)+a_n(X_n)-F(\un{L}_n)+F(\un{L}_{n+1})-a_{n+1}(X_{n+1})\bigr)
\biggr\>\\
&=\biggl\<-H_n(\un{L}_n,\un{L}_{n+1})+\xi\bigl(a_n(X_n)-F(\un{L}_n)\bigr)+\xi\bigl(F(\un{L}_{n+1})-a_{n+1}(X_{n+1})\bigr)
\biggr\>.
 \end{align*}
Define a new functional $\mathsf W$ of the ladder process $\{\un{L}_n\}$ with entries
$$
W(\un{L}_n):=\<\xi F(\un{L}_n)-\theta_n(X_n)\>-\E\left(\<\xi F(\un{L}_n)-\theta_n(X_n)\>\big|\{Z_i\}=\un{\zeta}^{(N)},X_n\right).
$$
Notice that $W(\un{L}_n)=\xi(F(\un{L}_n)-a_n(X_n))\mod 2\pi\Z$. Therefore
\begin{equation}\label{tf-n-identity}
\xi\tf_n(X_n,X_{n+1})=\biggl\<W(\un{L}_{n+1})-W(\un{L}_{n})-H_n(\un{L}_n,\un{L}_{n+1})
\biggr\>.
\end{equation}

\medskip
\noindent
{\em {\sc Claim.\/} Given $\delta>0$, let $T_\delta= {11\pi^2 D}/{\delta}.$ Then
there exists
 a measurable set $\Omega_X$ of $\{X_i\}$ such that $\Prob(\Omega_X)>1-\delta$ and such that for all $\un{\xi}\in\Omega_X$,
\begin{enumerate}[(1)]
\item $\sum_{n=3}^{k_N} \Prob\biggl(|W(\un{L}_n)|>\frac{\pi}{3}\bigg|\{Z_i\}=\un{\zeta}, \{X_i\}=\un{\xi}\biggr)<T_\delta,$
\item $\sum_{n=3}^{k_N} \Prob\biggl(|H_n(\un{L}_n,\un{L}_{n+1})|>\frac{\pi}{3}\bigg|\{Z_i\}=\un{\zeta}, \{X_i\}=\un{\xi}\biggr)<T_\delta,$
\item $\E\biggl(\bigl|\sum_{n=3}^{k_N} H_n(\un{L}_n,\un{L}_{n+1})\bigr|\bigg|\{Z_i\}=\un{\zeta}, \{X_i\}=\un{\xi}\biggr)<T_\delta.$
\end{enumerate}
}

\medskip
\noindent
{\em Proof of the claim.\/}
 $\un{L}_n$ is conditionally independent of $\{X_i\}_{i\neq n}$  given $\{Z_i\}, X_n$. So
\begin{align*}
&\sum_{n=3}^{k_N} \Prob\biggl(|W(\un{L}_n)|\geq \frac{\pi}{4}\biggl|\{Z_i\}=\un{\zeta}, \{X_i\}=\un{\xi}\biggr)\\
&=\sum_{n=3}^{k_N} \Prob\biggl(|W(\un{L}_n)|\geq \frac{\pi}{4}\biggl|\{Z_i\}=\un{\zeta}, X_n=\xi_n\biggr).
\end{align*}
Since $\E(W(\un{L}_n)|\{Z_i\}=\un{\zeta}, X_n)=0$, we can use the Chebyshev inequality to bound the sum of probabilities from above by
\begin{align*}
&\leq \frac{16}{\pi^2}\sum_{n=3}^{k_N}\Var\bigl(\<\xi F(\un{L}_n)-\theta_n(X_n)\> | \{Z_i\}=\un{\zeta}, X_n\bigr)\\
&\leq 4\sum_{n=3}^{k_N} \E\bigl(|e^{i\xi F(\un{L}_n)-i\theta_n(X_n)}-1|^2 \big| \{Z_i\}=\un{\zeta}, X_n\bigr), \text{ see \eqref{eix-1-inequality}}.
\end{align*}
Integrating over $\{X_i\}$ we have by the choice of $\theta_n^{(N)}(X_n)$ (step 3) that
$$
\E_X\left[\sum_{n=3}^{k_N} \Prob\biggl(|W(\un{L}_n)|\geq \frac{\pi}{4}\biggl|\{Z_i\}=\un{\zeta}, \{X_i\}=\un{\xi}\biggr)\right]\leq 8\pi^2 D.
$$
By Markov's  inequality, the set
$$
\Omega_X^1(T):=\left\{\un{\xi}: \sum_{n=3}^{k_N} \Prob\biggl(|W(\un{L}_n)|>\frac{\pi}{3}\bigg|\{Z_i\}=\un{\zeta}, \{X_i\}=\un{\xi}\biggr)\leq T
\right\}
$$
has probability  $\Prob[\Omega_X^1(T)]\geq 1-8\pi^2 D/ T$.

Similarly,  by Markov's inequality
\begin{align*}
&\Prob\biggl(|H_n|\geq \frac{\pi}{4}\biggl|\{Z_i\}=\un{\zeta}, \{X_i\}=\un{\xi}\biggr)\leq \frac{16}{\pi^2}\E\biggl(H_n^2\biggl|\{Z_i\}=\un{\zeta}, \{X_i\}=\un{\xi}\biggr).
\end{align*}
By the choice of $\un{\zeta}$,
$
\E_X\left[\sum_{n=3}^{k_N} \Prob\biggl(|H_n|\geq \frac{\pi}{4}\biggl|\{Z_i\}=\un{\zeta}, \{X_i\}\biggr)\right]\leq 16D
$.
So
the set
$$
\Omega_X^2(T):=\left\{\un{\xi}: \sum_{n=3}^{k_N} \Prob\biggl(|H_n(\un{L}_n,\un{L}_{n+1})|>\frac{\pi}{3}\bigg|\{Z_i\}=\un{\zeta}, \{X_i\}=\un{\xi}\biggr)\leq T
\right\}
$$
has probability $\Prob[\Omega^2_X(T)]\geq 1-16 D /T>1-2\pi^2 D/T$.

Finally, since conditional expectations contract $L^2$-norms,
\begin{align*}
&\E_X\biggl[\E\biggl(\bigl|\sum_{n=3}^{k_N} H_n(\un{L}_n,\un{L}_{n+1})\bigr|\bigg|\{Z_i\}=\un{\zeta}, \{X_i\}=\un{\xi}\biggr)^2\biggr]\\
&\leq \E\biggl[\biggl(\sum_{n=3}^{k_N} H_n(\un{L}_n,\un{L}_{n+1})\biggr)^2\bigg|\{Z_i\}=\un{\zeta}\biggr]\leq \pi^2 D.
\end{align*}
So
$
\Omega_X^3(T):=\left\{\un{\xi}: \E\biggl(\bigl|\sum_{n=3}^{k_N} H_n(\un{L}_n,\un{L}_{n+1})\bigr|\bigg|\{Z_i\}=\un{\zeta}, \{X_i\}=\un{\xi}\biggr)\leq T
\right\}
$
has probability $$\Prob[\Omega_X^3(T)]>1-\pi^2D/T^2.$$
We see that if $T>1$, then $\Prob[\Omega_X^1(T)\cap\Omega_X^2(T)\cap\Omega_X^3(T)]>1-\frac{11\pi^2 D}{T}$. The claim follows.\qed

\medskip
We can now complete the proof of the step 5 (and the reduction lemma) and show that $\wt{\mathsf f}$ is center-tight.

\medskip
Fix $\delta>0$ and $\Omega_X$, $T_\delta$ as in the claim.
Fix $N$ and define the random set
$$
A_N(\{\un{L}_n^{(N)}\}):=\{3\leq n\leq k_N: |W(\un{L}_n)|\geq \frac{\pi}{3}\text{ or }
|H_n(\un{L}_n,\un{L}_{n+1})|\geq \frac{\pi}{3}\}.
$$
For all $\un{\xi}\in \Omega_X$, we have the following bound (Lemma \ref{Lemma-Number-Of-Events}):
$$
\Prob\biggl(|A_N|>4T_\delta\bigg|\{Z_i\}=\un{\zeta}, \{X_i\}=\un{\xi}\biggr)<\frac{1}{2}.
$$
Similarly, for all $\un{\xi}\in\Omega_X$,
$
\DS\Prob\biggl(\bigl|\sum_{n=3}^{k_N} H_n\bigr|>4T_\delta\bigg|\{Z_i\}=\un{\zeta}, \{X_i\}=\un{\xi}\biggr)\leq \frac{1}{4}.
$

Since the probabilities of these events add up to less than one, the intersection of their complements is non-empty. So for every $\un{\xi}\in\Omega_X$ we can find $\{Y_i^{(N)}(\un{\xi})\}_{i=2}^{k_N-1}$ such that $\un{L}_n^\ast:=\un{L}_n^\ast(\un{\xi})=(\zeta_{n-2}^{(N)},Y_{n-1}^{(N)}(\un{\xi}),\xi_n)$ has the following two properties:
$$
\left|\sum_{n=3}^{k_N} H_n(\un{L}_n^\ast,\un{L}_{n+1}^\ast)\right|\leq 4T_\delta,\text{ and }
$$
$$
M:=\#\left\{3\leq n\leq k_N:  |W(\un{L}_n^\ast)|\geq \frac{\pi}{3}\text{ or }
|H_n(\un{L}_n^\ast,\un{L}_{n+1}^\ast)|\geq \frac{\pi}{3}\right\}\leq 4T_\delta.
$$
Let $n_1<\cdots<n_M$ be an enumeration of the  indices $n$ where $|W(\un{L}_n^\ast)|\geq \frac{\pi}{3}$ or
$|H_n(\un{L}_n^\ast,\un{L}_{n+1}^\ast)|\geq \frac{\pi}{3}$.
By \eqref{tf-n-identity},   if $n_i<n<n_{i+1}-1$,
$$
\xi\tf_n(\xi_n,\xi_{n+1})=W(\un{L}_{n+1}^\ast)-W(\un{L}_n^\ast)-H_n(\un{L}_n^\ast,\un{L}_{n+1}^\ast),
$$
because $\<x+y+z\>=x+y+z$ whenever $|x|,|y|,|z|<\frac{\pi}{3}$. So
$$-\sum_{n=n_i}^{n_{i+1}-1}\xi\tf_n(\xi_n,\xi_{n+1})= \sum_{n=n_i+1}^{n_{i+1}-1}H_n(\un{L}_n^\ast,\un{L}_{n+1}^\ast)\pm 6\pi$$
where we have used the bounds $|\mathsf W|\leq 2\pi$ and $|H_{n_i}|\leq \pi$.
Summing over $i$ we find that for every $\un{\xi}\in\Omega_X$,
$$
\left|\xi\sum_{n=3}^{k_N}\tf_n(\xi_n,\xi_{n+1})\right|\leq \left|\sum_{n=3}^{k_N}H_n(\un{L}_n^\ast,\un{L}_{n+1}^\ast)\right|+10M\pi\leq 4T_\delta+40T_\delta\pi<42\pi T_\delta.
$$
Setting $ C_\delta:=42\pi T_\delta/\xi$, we find that  $\Prob(\bigl|\sum_{n=3}^{k_N} \tf_n^{(N)}\bigr|\geq C_\delta)<\delta$ for all $N$, whence the (center-)tightness of $\wt{\mathsf f}$.\qed

\medskip
In chapter \ref{Chapter-reducible} we will need the  following variant of the reduction lemma for integer valued $\mathsf f$.
\begin{lemma}[Integer Reduction Lemma]
\label{LmIntRed}
Let $\mathsf X$ be a uniformly elliptic Markov chain, and $\mathsf f$ an integer valued  additive functional on $\mathsf X$ s.t. $|f|\leq K$ a.s. For every $N$,
$
f_n(x,y)=g_n^{(N)}(x,y)+a_n^{(N)}(x)-a_{n+1}^{(N)}(y)+c_n^{(N)}\ \ \ (n=1,\ldots,N)
$
where
\begin{enumerate}[(1)]
\item $c^{(N)}_n$ are integers such that $|c^{(N)}_n|\leq K,$
\item $a_n^{(N)}$ are measurable integer valued functions on $\fS_n$ s.t.  $|a^{(N)}_n|\leq 2K,$
\item $g_n^{(N)}$ are measurable, integer valued, and $\displaystyle \sum_{n=3}^N \E[g_n^{(N)}(X_n,X_{N+1})^2]\leq 10^3K^4\sum_{n=3}^N u_n^2$, with $u_n$  the structure constants of $\mathsf f$.
\end{enumerate}
\end{lemma}
\begin{proof}
Let $\left(Z_{n-2}\ \begin{array}{ll}
 Z_{n-1} & Y_n \\ Y_{n-1} & X_n
\end{array}\ X_{n+1}\right)$ be a random hexagon.
By the definition of the structure constants,
$$\E\left[\sum_{n=3}^N\E\left(\Gamma\left(Z_{n-2}\ \begin{array}{ll}
 Z_{n-1} & Y_n \\ Y_{n-1} & X_n
\end{array}\ X_{n+1}\right)^2\bigg|Z_{n-2}, Z_{n-1}\right)\right]=\sum_{n=3}^N u_n^2.$$
Therefore, for every $N$ there exists $z_n=z_n(N)\in \fS_n$ $(n=1,\ldots,N-2)$ such that
$$\sum_{n=3}^N\E\left[\E\left(\Gamma\left(Z_{n-2}\ \begin{array}{ll}
 Z_{n-1} & Y_n \\ Y_{n-1} & X_n
\end{array}\ X_{n+1}\right)^2\bigg|Z_{n-2}=z_{n-2}, Z_{n-1}=z_{n-1}\right)\right]\leq \sum_{n=3}^N u_n^2.$$
We emphasize that $z_n$ depends on $N$.

Let $c^{(N)}_n:=f_{n-2}(z_{n-2},z_{n-1})$, and let $a_n^{(N)}(x_{n})$  be the (smallest) most likely value of $$f_{n-2}(z_{n-2},Y)+f_{n-1}(Y,x_n),$$ where $Y$ has the bridge distribution of $X_{n-1}$ conditioned on $X_{n-2}=z_{n-2}$ and $X_n=x_n$.
The most likely value exists, and has probability bigger than
$
\delta_K:=\frac{1}{5K},
$
because $f_{n-2}(z_{n-2},Y)+f_{n-1}(Y,x_n)\in [-2K,2K]\cap\Z$.

Set $
g_n^{(N)}(x_n,x_{n+1}):=f_n(x_n,x_{n+1})+a_n^{(N)}(x_n)-a_{n+1}^{(N)}(x_{n+1})-c^{(N)}_n.
$
Equivalently,  $g_n^{(N)}(x_n,x_{n+1})=-\Gamma\left(z_{n-2}\
\begin{array}{ll}
z_{n-1} & y_n \\
y_{n-1} & x_n
\end{array}\ x_{n+1}
\right)$ for the $y_k$ which maximize the likelihood of the value $f_{k-1}(z_{k-1},Y)+f_k(Y,x_{k+1})$ when $Y$ has the bridge distribution of $X_k$ given $X_{k-1}=z_{k-1}, X_{k+1}=x_{k+1}$.

Our task is to estimate $\sum_{n=3}^N \E[g_n^{(N)}(X_n,X_{n+1})^2]$. Define for this purpose
the functions $h_n^{(N)}:\fS_n\times\fS_{n+1}\to \R$,
$$
h_n^{(N)}(x_n,x_{n+1}):=\E\left(\Gamma\left(Z_{n-2}\ \begin{array}{ll}
 Z_{n-1} & Y_n \\ Y_{n-1} & X_n
\end{array}\ X_{n+1}\right)^2\bigg|\begin{array}{ll}
Z_{n-2}=z_{n-2} & Z_{n-1}=z_{n-1}\\
X_n=x_n & X_{n+1}=x_{n+1}
\end{array}
\right)^{1/2},
$$
Our plan is to show the following:
\begin{enumerate}[(a)]
\item $\displaystyle \sum_{n=3}^N \E(h_n^{(N)}(X_n,X_{n+1})^2)\leq \sum_{n=3}^N u_n^2$
\item If $h_n^{(N)}(x_n,x_{n+1})<\delta_K$, then $g^{(N)}_n(x_n,x_{n+1})=0$.
\item $\E(g^{(N)}_n(X_n,X_{n+1})^2)\leq (6K)^2\Prob[h_n^{(K)}\geq \delta_k]\leq 36 K^2 \delta_K^{-2} \E[h_n^{(N)}(X_n,X_{n+1})^2]$.
\end{enumerate}

Part (a) is because of the choice of $z_n$. To see part (b), note that since $\mathsf f$ is integer valued, either the balance of a hexagon is zero, or it has absolute value $\geq 1$. Therefore, if $h_n^{(N)}(X_n,X_{n+1})< \delta_K$, then necessarily
\begin{align*}
&\Prob\left[\Gamma\left(Z_{n-2}\ \begin{array}{ll}
 Z_{n-1} & Y_n \\ Y_{n-1} & X_n
\end{array}\ X_{n+1}\right)\neq 0\bigg|\begin{array}{ll}
Z_{n-2}=z_{n-2} & Z_{n-1}=z_{n-1}\\
X_n=x_n & X_{n+1}=x_{n+1}
\end{array}\right]\\
&\leq
\E\left[\Gamma\left(Z_{n-2}\ \begin{array}{ll}
 Z_{n-1} & Y_n \\ Y_{n-1} & X_n
\end{array}\ X_{n+1}\right)^2\bigg|\begin{array}{ll}
Z_{n-2}=z_{n-2} & Z_{n-1}=z_{n-1}\\
X_n=x_n & X_{n+1}=x_{n+1}
\end{array}\right]\\
&=h_n^{(N)}(X_n,X_{n+1})^2<\delta_K^2,
\end{align*}
whence
$
\Prob\left[\Gamma\left(Z_{n-2}\ \begin{array}{ll}
 Z_{n-1} & Y_n \\ Y_{n-1} & X_n
\end{array}\ X_{n+1}\right)=0\bigg|\begin{array}{ll}
Z_{n-2}=z_{n-2} & Z_{n-1}=z_{n-1}\\
X_n=x_n & X_{n+1}=x_{n+1}
\end{array}\right]>1-\delta_K^2.
$

At the same time, by the structure of the distribution of random hexagons,
$$
\Omega_n:=\left\{\left(Z_{n-2}\ \begin{array}{ll}
 Z_{n-1} & Y_n \\ Y_{n-1} & X_n
\end{array}\ X_{n+1}\right):
\begin{array}{l}
f_{n-1}(Z_{n-1},Y_n)+f_n(Y_n,X_{n+1})=a_{n+1}^{(N)}(X_{n+1})\\
f_{n-2}(Z_{n-2},Y_{n-1})+f_{n-1}(Y_{n-1},X_{n})=a_{n}^{(N)}(X_{n})
\end{array}
\right\}
$$
satisfies
$
\Prob\left[\Omega_n\bigg|\begin{array}{ll}
Z_{n-2}=z_{n-2}& Z_{n-1}=z_{n-1}\\
X_{n}=x_{n} & X_{n+1}=x_{n+1}
\end{array}\right]> \delta_K^2,
$

If the sum of the probabilities of two events is bigger than one, then they must intersect. It follows that
there exist $y_{n-1},y_n$ such that
\begin{enumerate}[$\circ$]
\item $a_{n}^{(N)}(X_{n})=f_{n-2}(z_{n-2},y_{n-1})+f_{n-1}(y_{n-1},X_{n})$;
\item $a_{n+1}^{(N)}(X_{n+1})=f_{n-1}(z_{n-1},y_{n})+f_{n}(y_{n},X_{n+1})$;
\item $\Gamma\left(z_{n-2}\ \begin{array}{ll}
 z_{n-1} & y_n \\ y_{n-1} & X_n
\end{array}\ X_{n+1}\right)=0$.
\end{enumerate}
By the definition of $g^{(N)}_n$, this implies  that $g^{(N)}_n(X_n,X_{n+1})=0$, which proves part (b).

Part (c) follows from part (b),  Chebyshev's inequality, and  the estimate $\|g^{(N)}_n\|_\infty\leq 6K$ (as is true for the balance of every hexagon).
\qed
\end{proof}

 Combining Lemmas \ref{Lemma-Reduction} and \ref{LmIntRed}
we obtain the following result

\begin{corollary}
\label{CrJoint}
{\bf (Joint Reduction)}
There is a constant $L=L(\eps, K)$ such that
under the conditions of the Reduction Lemma we can arrange, in addition to the other conclusions
of Lemma \ref{Lemma-Reduction}, that $\DS \sum_{n=3}^{k_N} \|g_n^{(N)}\|^2_2\leq LU_N.$
\end{corollary}

\begin{proof}
Apply Lemma \ref{Lemma-Reduction} and then apply Lemma \ref{LmIntRed} to the resulting integer
valued additive functional
$ \frac{\xi \mathsf{g}}{2\pi}$.
Notice that the reduction in this corollary depends on $N$ even if $\mathsf f$ is an additive functional of a Markov chain.
\qed \end{proof}

Corollary \ref{CrJoint} says the following. Suppose we have an additive functional
$\mathsf{f}$
such that both
$U_N$ is small and $D_N(\xi)$ is small for some $\xi$ (but $D_N(\xi)$ can be much smaller
than $U_N$). Then we can adjust $\mathsf{f}$ such that at time $N$, the resulting functional will have a small
norm as prescribed by $U_N$ and small distance to $\frac{2\pi}{\xi} \Z$ as prescribed by
$D_N$ at the same time.

\subsection{The possible values of the  co-range}
We prove Theorem \ref{Theorem-co-range} in its version for Markov arrays: {\em The co-range of an a.s. uniformly bounded additive functional on a uniformly elliptic Markov array $\mathsf X$ is equal to $\R$ when $\mathsf f$ is center tight, and to $\{0\}$ or $t\Z$ $(t>0)$ otherwise.}

\medskip
Recall that the co-range is defined by
$$
H:=H(\mathsf X,\mathsf f)=\{\xi\in\R: \sup_N D_N(\xi)<\infty\},\text{ where }D_N(\xi)=\sum\limits_{n=3}^{k_N} d_n^{(N)}(\xi)^2.
$$

\medskip
\noindent
{\sc Step 1.} {\em $H$ is a subgroup of $\R$.}

\medskip
\noindent
{\em Proof.\/} $H=-H$, because  $d_n^{(N)}(-\xi)=d_n^{(N)}(\xi)$.
$H\owns 0$, because $d_n^{(N)}(0)=0$.  $H$ is closed under addition, because if
 $\xi,\eta\in H$, then   by Lemma \ref{Lemma-Sum},
$$
\sup_N\sum_{n=3}^{k_N} d_n^{(N)}(\xi+\eta)^2\leq 8\left[\sup_N\sum_{n=3}^{k_N} d_n^{(N)}(\xi)^2+\sup_N\sum_{n=3}^{k_N} d_n^{(N)}(\eta)^2\right]<\infty.
$$

\medskip
\noindent
{\sc Step 2.} {\em If $\mathsf f$ is center-tight, then $H=\R$.}

\medskip
\noindent
{\em Proof.\/} Suppose $\mathsf f$ is center-tight.
By  Corollary \ref{Corollary-Tight} and the center-tightness of $\mathsf f$, $\sup\limits_N \sum\limits_{k=3}^{k_{N}} (u_k^{(N)})^2<\infty$.  By  Lemma \ref{Lemma-Sum}(c), $\sup\limits_N \sum\limits_{k=3}^{k_{N}} d_n^{(N)}(\xi)^2<\infty$ for all $\xi\in\R$.

\medskip
\noindent
{\sc Step 3.} {\em If $f$ is not center-tight, then $\exists t_0$ s.t. }
\begin{equation}
\label{CoRangeDiscr}
H\cap (-t_0,t_0)=\{0\}.
\end{equation}

\medskip
\noindent
{\em Proof.} Let $K:=\ess\sup|\mathsf f|$, then  $|\Gamma(P)|\leq 6K$ for a.e. hexagon $P$.

Fix $\tau_0>0$ such that
$
|e^{it}-1|^2\geq \frac{1}{2}t^2\text{ for all }|t|<\tau_0,$
and let $t_0:=\tau_0(6K)^{-1}$.
Then
then for all $|\xi|<t_0$,
$
|e^{i\xi\Gamma(P)}-1|^2\geq \frac{1}{2}\xi^2\Gamma(P)^2\text{ for all hexagons }P.
$

Taking the expectation over $P\in\mathrm{Hex}(N,n)$, we obtain that
\begin{equation}
\label{DN-UN-SmallXi}
d_n^{(N)}(\xi)^2\geq \frac{1}{2}\xi^2 (u_n^{(N)})^2 \text{ for all }|\xi|<t_0, 1\leq n\leq k_N, N\geq 1.
\end{equation}

Now assume by way of contradiction that  there is
$0\neq \xi \in H\cap (-t_0,t_0)$, then
$
\sup\limits_{N}\sum\limits_{n=3}^{k_N} (u_n^{(N)})^2\leq \frac{2}{\xi^2}\sup\limits_{N}\sum\limits_{n=3}^{k_N} d_n^{(N)}(\xi)^2<\infty
$.
By Corollary~\ref{Corollary-Tight}, $f$ is center-tight, in contradiction to our assumption.

\medskip
\noindent
{\sc Step 4.} {\em If $f$ is not center-tight, then $H=\{0\}$, or $H=t\Z$ with $t\geq \frac{\pi}{6 \ess\sup|f|}$}.

\medskip
\noindent
{\em Proof.\/} By steps 2 and 3, $H$ is a proper closed  subgroup of $\R$. So it must be equal to $\{0\}$ or $t\Z$ where $t>0$. To see that $t\geq \frac{\pi}{6\ess\sup|f|}$, assume by contradiction that $t=(\frac{\pi}{6\ess\sup|f|})\rho$ with $0<\rho<1$, and let $\kappa:=\min\{|e^{iu}-1|^2/|u|^2:|u|\leq \pi\rho\}>0$. Then $|t\Gamma(P)|\leq 6t\ess\sup|f|=\pi\rho$ for every position $n$ hexagon $P$, whence
\begin{align*}
d_n^2(t)&=\E(|e^{it\Gamma}-1|^2)\geq \kappa\E(\Gamma^2)=\kappa u_n^2.
\end{align*}
This is impossible: $t\in H$ so $\sum d_n^2(t)<\infty$, whereas $\mathsf f$ is not center-tight so $\sum u_n^2=\infty$.
\hfill$\Box$

\subsection{Calculation of the essential range}
We prove Theorem \ref{Theorem-essential-range} in its version for Markov arrays: {\em For every  a.s. uniformly bounded additive functional $\mathsf f$ on a uniformly elliptic Markov array $\mathsf X$, }
\begin{equation}\label{essential-range-formula}
G_{ess}(\mathsf X,\mathsf f)=\begin{cases}
\{0\} & H(\mathsf X,\mathsf f)=\R\\
\frac{2\pi}{\xi}\Z & H(\mathsf X,\mathsf f)=\xi\Z\\
\R & H(\mathsf X,\mathsf f)=\{0\}.
\end{cases}
\end{equation}

\begin{lemma}\label{Lemma-co-range}
Suppose $\mathsf f,\mathsf g$  are two a.s. uniformly bounded  additive functionals on the same uniformly elliptic  Markov array.
If $\mathsf f-\mathsf g$ is center-tight, then $\mathsf f$ and $\mathsf g$ have the same co-range.
\end{lemma}

\begin{proof}
By Corollary \ref{Corollary-Tight}, if $\mathsf h=\mathsf g-\mathsf f$ is center-tight, then
$
\displaystyle\sup_N \sum_{n=3}^{k_N} u_n^{(N)}(\mathsf h)^2<\infty.
$
By Lemma \ref{Lemma-Sum} (b),(c), $$ \sup_N\sum_{n=3}^{k_N} d_n^{(N)}(\xi,\mathsf g)^2\leq 8\sup_N\sum_{n=3}^{k_N} d_n^{(N)}(\xi,\mathsf f)^2 +8\xi^2\sup_N\sum_{n=3}^{k_N}  u_n^{(N)}(\mathsf h)^2.$$
So the co-range of $\mathsf f$ is a subset of the co-range of $\mathsf g$.
By symmetry they are equal.\qed
\end{proof}

\medskip
\noindent
{\bf Proof of Theorem \ref{Theorem-essential-range}}:
As we saw in the previous section,  the possibilities for the co-range  are
$\R$, $t\Z$ with $t\neq 0$, and $\{0\}$.

\medskip
\noindent
{\sc Case 1:} {\em The co-range equals $\R$.} As we saw above,  this can only happen if $f$ is center-tight, in which case the essential range is  $\{0\}$ because we may subtract $\mathsf f$ from itself.

\medskip
\noindent
{\sc Case 2.:}  {\em The co-range equals $\xi\Z$ with $\xi\neq 0$.} We show that $G_{ess}(\mathsf X,\mathsf f)=\frac{2\pi}{\xi}\Z$.

By assumption, $\xi$ is in the co-range: $\sup_N \sum_{n=3}^{k_N} d_n^{(N)}(\xi)^2<\infty$. By the Reduction Lemma, $\mathsf f$ differs by a center-tight functional from a functional with algebraic range $\subseteq \frac{2\pi}{\xi}\Z$. So $G_{ess}(\mathsf X,\mathsf f)\subseteq\frac{2\pi}{\xi}\Z$.

Assume by way of contradiction that $G_{ess}(\mathsf X,\mathsf f)\subsetneq \frac{2\pi}{\xi}\Z$, then there exists a center-tight $\mathsf h$ such that the algebraic range of $\mathsf g:=\mathsf f-\mathsf h$ is a subset of $\frac{2\pi \ell}{\xi}\Z$ for some integer $\ell>1$. The structure constants of $\mathsf g$ must satisfy $d_n^{(N)}(\frac{\xi}{\ell},\mathsf g)\equiv 0$, whence $\frac{\xi}{\ell}\in$co-range of $\mathsf g$. By Lemma \ref{Lemma-co-range}, $\frac{\xi}{\ell}\in$ co-range of $\mathsf f$, whence $\frac{\xi}{\ell}\in \xi\Z$. But this contradicts $\ell>1$.

\medskip
\noindent
{\sc Case 3.:}  {\em The co-range equals $\{0\}$.} We claim that the essential range is $\R$. Otherwise, there exists a center-tight $\mathsf h$ such that the algebraic range of $\mathsf g:=\mathsf f-\mathsf h$ equals  $t\Z$ with $t\neq 0$  or $\{0\}$. But this is impossible:
\begin{enumerate}[(a)]
\item  If the algebraic range of $\mathsf g$ is $t\Z$, then $d_n^{(N)}(\frac{2\pi}{t},\mathsf g)=0$ for all $3\leq n\leq k_N$, $N\geq 1$, so the co-range of $\mathsf g$ contains $2\pi/t$. By Lemma \ref{Lemma-co-range}, the co-range of $\mathsf f$ contains $2\pi/t$, in contradiction to the assumption that it is $\{0\}$.

 \medskip
\item If the algebraic range of $\mathsf g$ is $\{0\}$, then $\mathsf f\equiv \mathsf h$, and $\mathsf f$ is center-tight. But by Theorem \ref{Theorem-co-range}, the co-range of a center-tight functional is $\R$, whereas the co-range of  our functional is $\{0\}$. \hfill$\Box$
\end{enumerate}

\subsection{Existence of irreducible reductions}
We prove Theorem \ref{Theorem-minimal-reduction}, in its version for Markov arrays:
{\em For every a.s. uniformly bounded additive functional on a uniformly elliptic Markov array $\mathsf X$, there exists an irreducible functional $\mathsf g$ such that $\mathsf f-\mathsf g$ is center-tight and
$
G_{alg}(\mathsf X,\mathsf g)=G_{ess}(\mathsf X,\mathsf g)=G_{ess}(\mathsf X,\mathsf f).
$
}

\medskip
\noindent
{\bf Proof.} The essential range is a closed subgroup of $\R$, so $G_{ess}(\mathsf X,f)=\{0\}, t\Z$ or $\R$.
\begin{enumerate}[(a)]
\medskip
\item If $G_{ess}(\mathsf X,\mathsf f)=\{0\}$, then $H(\mathsf X,\mathsf f)=\R$, and  $\mathsf f$ is center-tight. So take $\mathsf g\equiv 0$.

\medskip
\item If $G_{ess}(\mathsf X,f)=t\Z$ with $t\neq 0$, then by Theorem \ref{Theorem-essential-range} the co-range of $f$ is $\xi\Z$ with $\xi:=2\pi/t$. So $\sup\limits_N \sum\limits_{n=3}^{k_N} d_n^{(N)}({\xi},f)^2<\infty$.
By the reduction lemma, there exists an additive functional $\mathsf g$ such that $\mathsf f-\mathsf g$ is center-tight, and $G_{alg}(\mathsf  X,\mathsf g)\subseteq t\Z$. By Lemma \ref{Lemma-co-range}  $G_{ess}(\mathsf X,\mathsf f)=G_{ess}(\mathsf X,\mathsf g)$, whence
    $
    G_{ess}(\mathsf  X, \mathsf f)=G_{ess}(\mathsf  X, \mathsf g)\subseteq G_{alg}(\mathsf X, \mathsf g)\subseteq t\Z=G_{ess}(\mathsf X, \mathsf f)
    $, and   $G_{ess}(\mathsf  X, \mathsf g)=G_{alg}(\mathsf  X, \mathsf g)=G_{ess}(\mathsf X, \mathsf f)$.

\medskip
\item
If $G_{ess}(\mathsf X,f)=\R$, take $\mathsf g:=\mathsf f$.    \hfill$\Box$
\end{enumerate}

\subsection{Proofs of results on hereditary arrays}

\noindent
{\bf Proof of  Theorem \ref{Theorem-Hereditary-Property}:} Suppose $\mathsf f$ is an a.s. uniformly bounded additive functional on a uniformly elliptic Markov array $\mathsf X$.

The first part of the theorem asks for the equivalence of the following conditions:
\begin{enumerate}[(1)]
\item $\mathsf f$ is hereditary
\item for all $\xi$, $\liminf\limits_{N\to\infty}\sum\limits_{k=3}^{k_N} d_k^{(N)}(\xi)^2<\infty\Rightarrow \limsup\limits_{N\to\infty}\sum\limits_{k=3}^{k_N} d_k^{(N)}(\xi)^2<\infty$
\item for all $\xi\not\in H(\mathsf X,\mathsf f)$, $D_N(\xi)\xrightarrow[N\to\infty]{}\infty$
\item $H(\mathsf X',\mathsf f|_{\mathsf{X}'})=H(\mathsf X,\mathsf f)$ for every sub-array $\mathsf{X}'$ of $\mathsf{X}$.
\end{enumerate}

\medskip
\noindent
{\bf (1)$\Rightarrow$(2):} Assume that  $\mathsf f$ is hereditary and  $L_{\inf}(\xi):=\liminf D_N(\xi)<\infty$. We'll show that $L_{\sup}(\xi):=\limsup D_N(\xi)<\infty$.
 This is obvious for $\xi=0$,  so suppose $\xi\neq 0$.

Choose  $N_\ell, M_\ell\uparrow\infty$ such that
$D_{N_\ell}(\xi)\xrightarrow[\ell\to\infty]{}L_{\inf}(\xi)$, $D_{M_{\ell}}(\xi)\xrightarrow[\ell\to\infty]{}L_{\sup}(\xi)$. Let
$$
\mathsf{X}':=\{X^{(N_\ell)}_k\}\text{ and }
\mathsf{X}'':=\{X^{(M_\ell)}_k\}.
$$
Since $L_{\inf}(\xi)<\infty$, $H(\mathsf X', \mathsf f|_{\mathsf X'})$ contains $\xi$, whence by \eqref{essential-range-formula}, $G_{ess}(\mathsf X', f|_{\mathsf{X}'})\subseteq \frac{2\pi}{\xi}\Z$. By the hereditary property, $G_{ess}(\mathsf X'', \mathsf f|_{\mathsf{X}''})=G_{ess}(\mathsf X, \mathsf f)=G_{ess}(\mathsf X', \mathsf f|_{\mathsf{X}'})\subseteq \frac{2\pi}{\xi}\Z$. This implies by \eqref{essential-range-formula} that $H(\mathsf X'', \mathsf f|_{\mathsf{X}''})\owns\xi$, whence
$L_{\sup}(\xi)<\infty$.

\medskip
\noindent
{\bf (2)$\Rightarrow$(3):} We assume that $L_{\inf}(\xi)<\infty\Rightarrow L_{\sup}(\xi)<\infty$ and show that $D_N(\xi)\to\infty$ for all $\xi\not\in H(\mathsf X,\mathsf f)$.
If $\xi\not\in H(\mathsf X, \mathsf f)$, then $\sup\limits_N D_N(\xi)=\infty$, so $L_{\sup}(\xi)=\infty$. By assumption, this forces $L_{\inf}(\xi)=\infty$, whence $D_N(\xi)\to\infty$.

\medskip
\noindent
{\bf (3)$\Rightarrow$(4):}
We assume that   $D_N(\xi)\to\infty$ for all $\xi\not\in H(\mathsf X, \mathsf f)$, and show that $H(\mathsf X, \mathsf f)=H(\mathsf X', \mathsf f|_{\mathsf{X}'})$ for all sub-arrays $\mathsf{X}'=\{X^{(N_\ell)}_n\}$.
If $\xi\in H(\mathsf X, \mathsf f)$, then $\sup\limits_N D_N(\xi)<\infty$, whence $\sup\limits_\ell D_{N_\ell}(\xi)<\infty$ and $\xi\in H(\mathsf X', \mathsf f|_{\mathsf{X}'})$. If $\xi\not\in H(\mathsf X, \mathsf f)$, then $D_N(\xi)\to\infty$, whence $D_{N_\ell}(\xi)\to\infty$ and $\xi\not\in H(\mathsf X', \mathsf f|_{\mathsf{X}'})$.

\medskip
\noindent
{\bf (4)$\Rightarrow$(1):} We assume that $H(\mathsf X', f|_{\mathsf{X}'})=H(\mathsf X, \mathsf f)$ for all sub-arrays $\mathsf{X}'$, and show that $G_{ess}(\mathsf X', f|_{\mathsf{X}'})=G_{ess}(\mathsf X, \mathsf f)$ for all sub-arrays.  The inclusion $G_{ess}(\mathsf X', f|_{\mathsf{X}'})\subseteq G_{ess}(\mathsf X, \mathsf f)$ is obvious, so we focus on $G_{ess}(\mathsf X', f|_{\mathsf{X}'})\supseteq G_{ess}(\mathsf X, \mathsf f)$.

If $G_{ess}(\mathsf X', \mathsf f|_{\mathsf{X}'})=\R$ then there is nothing to prove.

Suppose $G_{ess}(\mathsf X', \mathsf f|_{\mathsf{X}'})\neq \R$, then $G_{ess}(\mathsf X', \mathsf f|_{\mathsf{X}'})=t\Z$ for some $t\in\R$. Let $\xi:=2\pi/t$ when $t\neq 0$ or any real number otherwise. By \eqref{essential-range-formula},
$$
H(\mathsf X', \mathsf f|_{\mathsf X'})\owns\xi.
$$
By assumption (4), this implies that $H(\mathsf X, \mathsf f)\owns\xi$, whence by \eqref{essential-range-formula}, $G_{ess}(\mathsf X,\mathsf f)\subseteq \frac{2\pi}{\xi}\Z=G_{ess}(\mathsf X', f|_{\mathsf{X}'})$, and the proof of (1) is complete.

\medskip
\noindent
This finishes the proof that properties (1)--(4) are equivalent.

\medskip
The second part of the theorem asks to show that $\mathsf f$ is {\em stably} hereditary iff $D_N(\xi)\to\infty$ {\em uniformly} on compact subsets of $\R\setminus H(\mathsf X, \mathsf  f)$.

Suppose $\mathsf f$ is stably hereditary, then $\mathsf f$ is hereditary, whence  $D_N(\xi)\to\infty$ for all $\xi\not\in H(\mathsf X, \mathsf f)$. To show that the convergence is uniform on compacts, we check that
\begin{equation}\label{72-birthday}
\forall\xi\not\in H(\mathsf X, \mathsf f),  \forall M>0, \exists N_\xi, \delta_\xi>0\left(\begin{array}{l}
N>N_\xi\\
|\xi'-\xi|<\delta_\xi
\end{array}
\!\!\!\!\Rightarrow D_N(\xi')>M\right).
\end{equation}
Suppose this were false for some $\xi$ and $M$, then
$
\exists \xi_N\to\xi\text{ such that } D_N(\xi_N){\leq } M.
$
But this implies that $\{(1+\epsilon_N)f^{(N)}_k\}$ is not hereditary for $\epsilon_N:=\frac{\xi_N}{\xi}-1$, in contradiction to our assumptions.

Conversely, if $D_N(\xi)\to\infty$ uniformly on compact subsets of $\R\setminus H(\mathsf X, \mathsf f)$, and $\epsilon_N\to 0$, then $\{g^{(N)}_k\}=\{(1+\epsilon_N)f^{(N)}_k\})$ is hereditary, because  for all $\xi\not\in H(\mathsf X,\mathsf f)$,
$D_N(\xi,\mathsf g)\equiv D_N((1+\epsilon_N)\xi,\mathsf f)\to\infty$, and as we saw above (2)$\Rightarrow$(1).
\qed

\medskip
\noindent
{\bf Proof of Theorem \ref{Theorem-Sufficient-Conditions-for-Stable-Heredity}:}
The first part of the theorem assumes that $G_{ess}(\mathsf X, \mathsf f)=t\Z$ or $\{0\}$ and that $\mathsf f$ is hereditary, and asks to show that $\mathsf f$ is stably hereditary.

  We begin with several reductions.
    It is sufficient to consider the case $G_{ess}(\mathsf X, \mathsf f)=\Z$: If $G_{ess}(\mathsf X, \mathsf f)=t\Z$ with $t\neq 0$ we work with $t^{-1} \mathsf f$, and if $G_{ess}(\mathsf X, \mathsf f)=\{0\}$ then  $H(\mathsf X, \mathsf f)=\R$ and $D_N(\xi)\to\infty$ uniformly on compact subsets of $\R\setminus H(\mathsf X, \mathsf f)$ (vacuously), so $\mathsf f$ is stably hereditary by Theorem \ref{Theorem-Hereditary-Property}.

Next we claim that it is enough to treat the special case $G_{alg}(\mathsf X, \mathsf f)=G_{ess}(\mathsf X, \mathsf f)=\Z$. Otherwise we use Theorem \ref{Theorem-minimal-reduction} to write $\mathsf f=\mathsf g-\mathsf h$ where $G_{alg}(\mathsf X, \mathsf g)=G_{ess}(\mathsf X, \mathsf g)=G_{ess}(\mathsf X, \mathsf f)$ and $\mathsf  h$ is center-tight. By  Lemma \ref{Lemma-co-range}, $H(\mathsf X, \mathsf g)=H(\mathsf X, \mathsf f)$, and by Lemma \ref{Lemma-Sum} and  Corollary \ref{Corollary-Tight},
$$
D_N(\xi,\mathsf f)\geq \frac{1}{8} D_N(\xi,\mathsf g)-\frac{1}{8}\xi^2\sup_n \sum_{k=3}^{k_n+1}u_k^{(n)}(\mathsf h)^2=\frac{1}{8} D_N(\xi,\mathsf g)-O(1).
$$
 Thus, if $D_N(\xi,\mathsf g)\to\infty$ uniformly on compact subsets of $\R\setminus H(\mathsf X, \mathsf g)$, then $D_N(\xi,\mathsf f)\to\infty$ uniformly on compact subsets of $\R\setminus H(\mathsf X, \mathsf f)$.

By assumption, $\ess\sup|f|\leq K$ for some integer $K$.   Then for every hexagon $P\in \mathrm{Hex}(N,n)$,
$
\Gamma(P)\in \Z\cap [-6K,6K]$.

Let $m^{(N)}_n$ denote the probability measure on the space of hexagons $\mathrm{Hex}(N,n)$ and define for every $\gamma\in\Z\cap [-6K,6K]$,
$$\mu_N(\{\gamma\}):=\sum\limits_{n=3}^{k_N} m^{(N)}_n\{P\in \mathrm{Hex}(N,n): \Gamma(P)=\gamma\}.
$$
Using the identity $|e^{i\xi\gamma}-1|^2=4\sin^2\frac{\xi\gamma}{2}$, we see that
$$
d_N^2(\xi)
=4\sum_{\gamma=-6K}^{6K} \mu_N(\gamma)\sin^2\frac{\xi\gamma}{2}.
$$
Since $\mathsf f$ is hereditary, $D_N\to\infty$ on $\R\setminus H(\mathsf X,\mathsf f)$,
and the  expression for $d_N^2(\xi)$  shows that  if $D_N\to\infty$ at $\xi$, then $D_N\to\infty$ uniformly on an open neighborhood of $\xi$.

It follows that  $D_N\to\infty$ uniformly on compact subsets of $\R\setminus H(\mathsf X, \mathsf f)$.
By Theorem~\ref{Theorem-Hereditary-Property}, $f$ must be stably hereditary. This is the first part of the theorem.

\medskip
The second part of the theorem says that
if $\mathsf f$ is integer valued and not center-tight, and if $\ess\sup|\mathsf f|\leq K$,  then $G_{ess}(\mathsf X,\mathsf f)=k\Z$
for some integer $0<k\leq 12 K$.

To see this recall that $G_{ess}(\mathsf X, \mathsf f)\subset G_{alg}(\mathsf X, \mathsf f)\subset\Z$, whence
$G_{ess}(\mathsf X, \mathsf f)=k\Z$ for some $k\in\Z$. Since $\mathsf f$ is not center-tight, $k\neq 0$.  By \eqref{essential-range-formula}, $H(\mathsf X, \mathsf f)=\frac{2\pi}{k}\Z$.

The inequality $|f|\leq K$ implies that every hexagon $P$ has balance $|\Gamma(P)|\leq 6K$. This implies that $k\leq 12K$: Otherwise  $|\frac{2\pi \Gamma(P)}{k}|<0.95\pi$ and  \eqref{eix-1-inequality} gives
$$
|e^{(2\pi i/k)\Gamma(P)}-1|^2\geq \mathrm{const}\; \Gamma(P)^2.
$$
But this implies that  $d_n^{(N)}(\frac{2\pi}{k})\geq \mathrm{const}\; u_n^{(N)}$, whence
$$\sup_N\sum_{n=3}^{k_N} d_n^{(N)}(\frac{2\pi}{k})^2\geq \sup_N\sum_{n=3}^{k_N} (u_n^{(N)})^2=\infty\text{ by non-center-tightness}.$$
This contradicts $\frac{2\pi}{k}\in H(\mathsf X, \mathsf f)$. Thus $0<k\leq 12K$.

\medskip
It follows from the first part of the theorem and from Theorem \ref{Theorem-Hereditary-Property}, that if $\mathsf f$ is integer valued and not center-tight, then  the properties of being hereditary and of being stably hereditary are equivalent.\qed

\section{Notes and references}
In the stationary world, a center-tight cocycle is a coboundary (Schmidt \cite{Schmidt-Cocycles}) and the problems discussed in this chapter reduce to the question  how small can one make the range of a cocycle by subtracting from it a coboundary. The question  appears naturally in the ergodic theory of group actions, because of its relation to the ergodic decomposition of skew-products \cite[chapter 8]{Aaronson-Book}, \cite{Schmidt-Cocycles}, \cite{Conze-Raugi-Ergodic-Decomp}, and to the structure of locally finite ergodic invariant measures for skew-products \cite{ANSS}, \cite{Sa-horocycle}, \cite{Raugi}. In the general setup of ergodic theory, minimal reductions such as in Theorem \ref{Theorem-minimal-reduction} are not always possible \cite{Lem}, although they do sometime exist \cite{Sa-horocycle},\cite{Raugi}.

The relevance of (ir)reducibility to the local limit theorem appears in different form  in the papers of
Guivarc'h \& Hardy \cite{GH}, Aaronson \& Denker \cite{Aaronson-Denker-LLT}, and Dolgopyat \cite{D-Ind}. There ``irreducibility" is expressed in terms of a condition which rules out non-trivial solutions for certain cohomological equations.

It is more difficult to uncover the irreducibility condition in the probabilistic literature on the LLT for sums of independent random variables. Rozanov's paper \cite{Rozanov}, for example,  proves a LLT for independent $\Z$-valued random variables
 $X_k$ assuming  Lindeberg's condition (which is automatic for bounded random variables),
 $\sum \Var(X_k)=\infty$, and subject to the assumption that\index{Rozanov's condition}

\begin{equation}\label{Rozanov}
\prod_{k=1}^\infty \left(\max_{0\leq m<t} \Prob(X_k=m\ \mathrm{ mod }\ t)\right)=0\text{ for all integers }t\geq 2.
\end{equation}
Let $\mathsf X=\{X_k\}$ and $\mathsf f=\{f_k\}$ where $f_k(x)=x$. Clearly, \eqref{Rozanov} implies that $G_{alg}(\mathsf X,\mathsf f)=\Z$.
We claim that \eqref{Rozanov} is equivalent to the irreducibility: $G_{ess}(\mathsf X,\mathsf f)=\Z$.

To see why,  it is useful first to note that
\eqref{Rozanov} is equivalent to
\begin{equation}\label{Rozanov2}
\sum_k \Prob[X_k\neq m_k \;\mathrm{mod}\; t]=\infty
\end{equation}
where $m_k$ is the (smallest) most likely residue mod $t$ for $X_k.$

\medskip
\noindent
{\bf Irreducibility$\Rightarrow$Rozanov's condition:}
Define for $x\in\Z$ and $2\leq t\in\Z$,  $\{x\}_{t\Z}:=t\{x/t\}$, $[x]_{t\Z}:=x-\{x\}_{t\Z}$, and set
\begin{enumerate}[$\circ$]
\item $y_k(x):=\text{the (smallest) integer in $m_k+t \Z$ closest to $x$}$
\item $z_k(x):=x-y_k(x)$
\item $g_k(x):=(y_k(x)-m_k)+[x-y_k(x)]_{t\Z}$ ($g_k$ takes values in $t\Z$)
\item $h_k(x):=\{x-y_k(x)\}_{t\Z}$ ($h_k$ takes values in $\Z$). Then
\end{enumerate}
$$X_k=g_k(X_k)+h_k(X_k)+m_k.$$
The algebraic range of $g_k$ is inside $t\Z$,  and by the Borel-Cantelli Lemma,
$$
\eqref{Rozanov2}\text{ fails}\Leftrightarrow X_k\neq m_k\mod t\Z\text{ finitely often a.s.}\Leftrightarrow h_k(X_k)\neq 0\text{ finitely often a.s. }
$$
If \eqref{Rozanov2} fails, then $\DS \sum_{k=0}^\infty h_k(X_k)$ converges a.s.
(since a.s. there are only finitely non-zero terms).
Hence  $\mathsf h$ is center-tight. Since $G_{alg}(\mathsf g)\subset t\Z$,  we have a contradiction to  irreducibility.

\medskip
\noindent
{\bf Rozanov's condition $\Rightarrow$ irreducibility:}
 Fix $\theta\in [0,t)$ and let $m$ be the closest integer in $[0,t)\cap\Z$ to $\theta$. Then  $|m'-\theta|\geq \frac{1}{2}$ for $m'\neq m$, whence
$$
\E[\dist^2(X_n,\theta+t\Z)]\geq \frac{1}{4}\Prob(X_n\neq m\,\mathrm{mod}\,t)\geq \frac{1}{4}[1-\max_{0\leq m<t} \Prob(X_n=m\ \mathrm{ mod }\ t)].
$$
Passing to the infimum over $\theta$, we obtain that
$$
\mathfrak D^2(X_{n},\tfrac{2\pi}{t})\geq \frac{1}{4}[1-\max_{0\leq m<t} \Prob(X_n=m\ \mathrm{ mod }\ t)].
$$
(See \S \ref{Section-Structure-Constants}.) We now obtain from Proposition \ref{Prop-Structure-Const-IND} that
\begin{align*}
&\sum_{n=3}^\infty d_n^2(\tfrac{2\pi}{t})\geq const \sum_{n=3}^\infty\bigl(\mathfrak D^2(X_{n-1},\tfrac{2\pi}{t})+\mathfrak D^2(X_{n},\tfrac{2\pi}{t})\bigr)\\
&\geq const \sum_{n=2}^\infty \left(1-\max_{0\leq m<t} \Prob(X_k=m\ \mathrm{ mod }\ t)\right)=\infty, \text{ by \eqref{Rozanov2}}.
\end{align*}
We find that the co-range does {\em not} contain $2\pi/t$ for $t=2,3,4,\ldots$. We already know that the co-range does contain $2\pi$ (because $X_k$ are integer valued). The only closed sub-group of $\R$ with these properties is $2\pi\Z$. So the co-range is $2\pi\Z$, and the essential range is $\Z=$the algebraic range

\smallskip
Other sufficient conditions for the LLT for sums of independent random variables such as those appearing in \cite{Mineka-Silverman},\cite{Statulevicius-Sums-of-Independent} and \cite{Mukhin-1991} can be analyzed in a similar way.
The reduction lemma was proved for sums of independent random variables in \cite{D-Ind}. A version of Theorem \ref{Theorem-MC-array-difference} for sums of independent random variables appears in \cite{Mukhin-1991}.

\label{PartLLT}

\chapter{The local limit theorem in the irreducible case}\label{Section-LLT-irreducible}

\noindent
{\em In this chapter we prove the local limit theorem for $\Prob(S_N-z_N\in (a,b))$
when $\frac{z_N-\E(S_N)}{\sqrt{\Var(S_N)}}$ converges to a finite limit and $\mathsf f$ is  irreducible. In this regime, the asymptotic behavior of $\Prob(S_N-z_N\in (a,b))$ does not to depend on the details of $\mathsf X$ and $\mathsf f$ (``universality").
}

\section{Main results}

\subsection{Local limit theorems for Markov chains}
In the next two theorems, we assume that $\mathsf f$ is an a.s. uniformly bounded additive functional on a uniformly elliptic Markov {\bf chain} $\mathsf X$, and we let  $\mathsf X=\{X_n\}$, $\mathsf f=\{f_n\}$, $S_N=f_1(X_1,X_2)+\cdots+f_N(X_N,X_{N+1})$, and $V_N:=\mathrm{Var}(S_N)$.
We make no assumptions on the initial distribution and allow $\Prob=\Prob_x=\Prob(\ \cdot\ |X_1=x)$.
\begin{theorem}
\label{ThLLT-classic}
Suppose $\mathsf f$ is irreducible, with  algebraic range $\R$. Then $V_N\to\infty$, and  for every  interval $(a,b)$ and $z_N\in\R$ s.t. $\frac{z_N-\E(S_N)}{\sqrt{V_N}}$ converges to a finite limit $z$,
\begin{equation}\label{LLT-non-arithmetic-limit}
\Prob[S_N-z_N\in (a,b)]=[1+o(1)]\frac{e^{-z^2/2}}{\sqrt{2\pi V_N}}(b-a), \text{ as }N\to\infty.
\end{equation}
\end{theorem}

\begin{theorem}\label{Thm-LLT-Lattice}
Suppose $t>0$ and  $\mathsf f$ is irreducible with algebraic range $t\Z$. Then $V_N\to\infty$ and    there are constants $0\leq \gamma_N<t$ such that for all $k\in\Z$, and  for all   $z_N\in \gamma_N+t\Z$ s.t. $\frac{z_N-\E(S_N)}{\sqrt{V_N}}$ converges to a finite limit $z$,
\begin{equation}\label{LLT-arithmetic-limit}
\Prob[S_N-z_N=kt]=[1+o(1)]\; \frac{e^{-z^2/2}t}{\sqrt{2\pi V_N}}, \quad \text{as}\quad N\to\infty.
\end{equation}
The constants $\gamma_N$ are determined by the condition $\Prob[S_N\in \gamma_N+t\Z]=1$ for all $N$.
\end{theorem}

The conditions of the theorems can be checked from the data of $\mathsf X$ and $\mathsf f$ using the structure constants $d_n(\xi)$ from \S\ref{Section-Structure-Constants}:

\begin{lemma}\label{Lemma-Non-Lattice-Irreducible}
Let $\mathsf f$ be an a.s. uniformly bounded additive functional on a uniformly elliptic Markov chain $\mathsf X$. Then
\begin{enumerate}[(1)]
\item $\mathsf f$ is non-lattice and irreducible iff $\sum d_n^2(\xi)=\infty$ for all $\xi\neq 0$.
\item  $\mathsf f$ is lattice and irreducible with algebraic range $t\Z$, $t>0$, iff $\sum d_n^2(\xi)<\infty$ for   $\xi\in(2\pi/t)\Z$ and $\sum d_n^2(\xi)=\infty$ for $\xi\not\in(2\pi/t)\Z$.
\item   $\mathsf f$ is lattice and irreducible with algebraic range $\{0\}$ iff $f_n(X_n,X_{n+1})$ are a.s. constant for all $n$.
\end{enumerate}
\end{lemma}
\begin{proof}
$\mathsf f$ is non-lattice and irreducible iff $G_{ess}(\mathsf X,\mathsf f)=G_{alg}(\mathsf X, \mathsf f)=\R$. By Theorem~\ref{Theorem-co-range}, this happens iff $\mathsf f$ has co-range $\{0\}$, which proves part (1). Part (2) is proved in a similar way, and part (3) is a triviality.\qed
\end{proof}

\subsection{Local limit theorems for Markov arrays}
In this section, we assume that $\mathsf f$ is an a.s. uniformly bounded additive functional on a uniformly elliptic Markov array $\mathsf X$ with row lengths $k_N+1$, and we let  $\mathsf X=\{X_n^{(N)}\}$, $\mathsf f=\{f_n^{(N)}\}$, $S_N=\sum_{i=1}^{k_N}f_i^{(N)}(X_i^{(N)},X_{i+1}^{(N)})$, and $V_N:=\mathrm{Var}(S_N)$. We make no assumptions on the initial distribution, and allow $\Prob=\Prob_{x_1^{(N)}}=\Prob(\ \cdot\ |X_1^{(N)}=x_1^{(N)})$.

The LLT for $S_N$ may fail due to the possibility that $\mathsf f|_{\mathsf X'}$ may have different essential range for different sub-arrays $\mathsf X'$. To deal with this we need to assume hereditary behavior, see  \S\ref{Section-Hereditary}.

\medskip
\noindent
{\bf Theorem \ref{ThLLT-classic}'.} {\em
Suppose $\mathsf f$ is stably hereditary,  non-lattice and irreducible. Then $V_N\to\infty$, and  for every interval $(a,b)$ and $z_N\in\R$ s.t. $\frac{z_N-\E(S_N)}{\sqrt{V_N}}\xrightarrow[N\to\infty]{}z\in\R$,}
\begin{equation}
\Prob[S_N-z_N\in (a,b)]=[1+o(1)]\frac{e^{-z^2/2}}{\sqrt{2\pi V_N}}(b-a), \text{ as }N\to\infty.
\end{equation}

\medskip
\noindent
{\bf Theorem \ref{Thm-LLT-Lattice}'.}
{\em Suppose $t>0$ and  $\mathsf f$ is hereditary, irreducible,   and  with algebraic range $t\Z$. Then $V_N\to\infty$, and there are
$0\leq \gamma_N<t$ such that for all $k\in\Z$ and $z_N\in \gamma_N+t\Z$ s.t. $\frac{z_N-\E(S_N)}{\sqrt{V_N}}\xrightarrow[N\to\infty]{}z\in\R$,
\begin{equation}\label{LLT-lattice-with-k}
\Prob[S_N-z_N=kt]=[1+o(1)]\; \frac{e^{-z^2/2}t}{\sqrt{2\pi V_N}}, \quad \text{as}\quad N\to\infty.
\end{equation}
The constants $\gamma_N$ are determined by the condition $\Prob[S_N\in \gamma_N+t\Z]=1$ for all $N$.
}

\medskip
\noindent
Notice that whereas in the non-lattice case we had to assume that $\mathsf f$ is stably hereditary, in the lattice case it is sufficient to assume that $\mathsf f$ is hereditary. This is because in the lattice case the two assumptions are equivalent, see Theorem \ref{Theorem-Sufficient-Conditions-for-Stable-Heredity}.

Again, it is possible to check the assumptions of the theorems from the data of $\mathsf X$ and $\mathsf f$ using the structure constants:

\medskip
\noindent
{\bf Lemma \ref{Lemma-Non-Lattice-Irreducible}'.}
{\em Let $\mathsf f$ be an a.s. uniformly bounded additive functional on a uniformly elliptic Markov array $\mathsf X$ with row lengths $k_N+1$. Let $d_n^{(N)}(\xi)$ be as \S\ref{Section-Structure-Constants}, then
\begin{enumerate}[(1)]
\item
 $\mathsf f$ is stably hereditary, irreducible, and with algebraic range $\R$ iff
$$
\sum_{n=3}^{k_N}d_n^{(N)}(\xi)^2\xrightarrow[N\to\infty]{}\infty\text{ uniformly on compacts in  $\R\setminus\{0\}$}.
$$
\item Suppose $t\neq 0$, then
$\mathsf f$ is hereditary and irreducible with algebraic range $t\Z$  if and only if
$
\sum_{n=3}^{k_N} d_n^{(N)}(\xi)^2\xrightarrow[N\to\infty]{}\infty\text{ for all }\xi\not\in\frac{2\pi}{t}\Z.
$ In this case $\mathsf f$ is also stably hereditary.
\end{enumerate}
 }

\begin{proof}
As in the case of Markov chains, $\mathsf f$ is non-lattice and irreducible iff its co-range equals $\{0\}$. By Theorem \ref{Theorem-Hereditary-Property}, $\mathsf f$ is stably hereditary iff $\sum_{n=3}^{k_N}d_n^{(N)}(\xi)^2\xrightarrow[N\to\infty]{}\infty$ uniformly on compacts in $\R\setminus\{0\}$, which proves part (1).

Part (2) is proved in a similar way, with the additional observation that thanks to Theorem \ref{Theorem-Sufficient-Conditions-for-Stable-Heredity}, in the irreducible lattice case, every hereditary additive functional is automatically stably hereditary.
\qed
\end{proof}

\subsection{Mixing local limit theorems}\label{Section-Mixing-LLT}
\index{Mixing LLT}\index{Local limit theorem!mixing LLT}
Let $\mathsf f$ be an additive functional on a Markov $\mathsf X$  with row lengths $k_N+1$, and state spaces $(\fS^{(N)}_n, \mathfs B(\fS^{(N)}_n))$. Let $S_N$ and $V_N$ be as in the previous section.

\begin{theorem}[Mixing LLT]\label{Theorem-Mixing-LLT}
Suppose $\mathsf X$ is a uniformly elliptic Markov array, and $\mathsf f$ is an additive functional on $\mathsf X$ which is stably hereditary,  a.s. uniformly bounded,   and irreducible. Let $\mathfrak A_{N}\subset\fS_{k_N+1}^{(N)}$ be measurable events such that $\Prob[X^{(N)}_{k_N+1}\in\fA_N]$ is bounded away from zero, and let $x_N\in \fS_{1}^{(N)}$. Then for every $\phi:\R\to\R$ continuous with compact support,
\begin{enumerate}[(1)]
\item {\bf Non-lattice case:} Suppose $\mathsf f$ has algebraic range $\R$. For every $z_N\in\R$ s.t. $\frac{z_N-\E(S_N)}{\sqrt{V_N}}\to z\in\R$,
$$
\lim\limits_{N\to\infty}\sqrt{V_N}\E[\phi(S_N-z_N)|X_{k_N+1}^{(N)}\in \fA_{N}, X^{(N)}_1=x_N]=\frac{e^{-z^2/2}}{\sqrt{2\pi}}
\int_{-\infty}^\infty \phi(u)du.
$$
\item {\bf Lattice case:} Suppose $\mathsf f$ has algebraic range $t\Z$ ($t>0$) and $\Prob[S_N\in \gamma_N+t\Z]=1$ for all $N$. For every $z_N\in \gamma_N+t\Z$ s.t. $\frac{z_N-\E(S_N)}{\sqrt{V_N}}\to z\in\R$,
$$
\lim\limits_{N\to\infty}\sqrt{V_N}\E[\phi(S_N-z_N)|X_{k_N+1}^{(N)}\in\mathfrak A_{N}, X^{(N)}_1=x_N]=\frac{e^{-z^2/2}|t|}{\sqrt{2\pi}}
\sum_{u\in\Z}\phi(tu).
$$
\end{enumerate}
\end{theorem}
\noindent
To understand what this means, think of $\phi\approx 1_{(a,b)}$.

\medskip
In the next chapter, we will use
mixing LLT for {\em irreducible} additive functionals to study the LLT for some {\em reducible} additive functionals, as follows. Suppose
$\mathsf f=\wt{\mathsf f}+\nabla h$, where $\mathsf f$ is irreducible and $\mathsf h$ is uniformly bounded. Then
$$S_N(\mathsf f)={S}_N(\wt{\mathsf f})+h_1^{(N)}(X_1^{(N)})-h_{k_N+1}^{(N)}(X_{k_N+1}^{(N)}).$$ To pass from the LLT for $S_N(\wt{\mathsf f})$ (which we know since $\wt{\mathsf f}$ is irreducible) to the LLT for $S_N({\mathsf f})$ (which we do not know because of the reducibility of $\mathsf f$), we need to understand
the {\em joint} distribution of ${S}_N(\wt{\mathsf f})$, $h_1^{(N)}(X_1^{(N)})$ and $h_{k_N+1}^{(N)}(X_{k_N+1})$.
This is the task achieved by the mixing LLT.

 \section{Proofs}
We will provide the proofs in the general context of Markov arrays.

\medskip
\noindent
{\bf Standing assumptions and notation for the remainder of the chapter:} \\
$\mathsf X=\{X^{(N)}_n\}$ is a Markov array with row lengths $k_N+1$, state spaces $\fS^{(N)}_n$, and transition probabilities $\pi^{(N)}_{n,n+1}(x,dy)$, and $\mathsf f=\{f^{(N)}_n\}$ is an additive functional on $\mathsf X$. As always,  $d_n^{(N)}(\xi)$ are the structure constants of $\mathsf f$.

We assume that $\ess\sup|f|<K<\infty$, and that $\mathsf X$ is uniformly elliptic with ellipticity constant $\epsilon_0$.
By the uniform ellipticity assumption,
$$
\pi_{n,n+1}^{(N)}(x,dy)=p_n^{(N)}(x,y)\mu_{n+1}^{(N)}(dy)
$$
with $0\leq p_n^{(N)}(x,y)<\epsilon_0^{-1}$ such that $\int p_n^{(N)}(x,y)p_{n+1}^{(N)}(y,z)\mu_{n+1}^{(N)}(dy)>\epsilon_0$.
There is no loss of generality in  assuming that $\mu_{k}^{(N)}(E)=\Prob(X_{k}^{(N)}\in E)$, see Proposition \ref{Proposition-nu} and the discussion which follows it.

\subsection{Characteristic functions}\label{Section-Char-Func}
The classical approach to limit theorems in probability theory, due to P. L\'evy,  is to apply the Fourier transform, and analyze the  {\bf characteristic functions} of the random variables in the problem. In our case the relevant  characteristic functions are:\index{characteristic functions}
\begin{align*}
\Phi_N(x, \xi)&:=\EXP_x\left(e^{i\xi S_N}\right)\equiv \EXP\left(e^{i\xi S_N}|X^{(N)}_1=x\right).\\
\Phi_N(x, \xi|\fA)& :=\EXP_x\left(e^{i\xi S_N}|X_{k_N+1}\in \fA\right)\equiv \EXP\left(e^{i\xi S_N}|X_{k_N+1}^{(N)}\in \fA,X^{(N)}_1=x\right).
\end{align*}
Here $x\in\mathfrak S_1^{(N)}$, $\fA\subset \mathfrak S_{k_N+1}^{(N)}$,   $\xi\in\R$, and $\E_x(\cdot)=\E(\ \cdot\ |X^{(N)}_1=x)$.

We write these functions in terms of {\bf perturbation operators}\index{perturbation operators} as in \cite{N}.
For every $N\in\mathbb N$ and $1\leq n\leq k_N+1$,  define $\mathcal L_{n,\xi}^{(N)}: L^\infty(\mathfrak S_{n+1}^{(N)})\to L^\infty (\mathfrak S_{n}^{(N)})$ by

\begin{align*}
\left(\cL_{n, \xi}^{(N)} v\right)(x)&:=\int_{\fS_{n+1}^{(N)}}  p_n^{(N)}(x,y)  e^{i\xi f_n^{(N)}(x,y)}  v(y) d\mu_{n+1}^{(N)}(y) \\
&\equiv\E\bigl(e^{i\xi f_n^{(N)}(X_n^{(N)},X_{n+1}^{(N)})}v(X_{n+1}^{(N)})|X_n^{(N)}=x\bigr).
\end{align*}

\begin{lemma}[Nagaev]\index{Nagaev's identities}
Let $1(\cdot)\equiv 1$, then the  following identities hold:
\begin{align}
&\EXP\left(e^{i\xi S_{N}} v(X_{k_N+1}^{(N)})\bigg| X^{(N)}_1=x \right)=\left(\cL_{1,\xi}^{(N)} \cL_{2,\xi}^{(N)}\dots \cL_{k_N,\xi}^{(N)} v \right)(x) \label{Nagaev2},\\
&\Phi_N(x, \xi)=\left(\cL_{1,\xi}^{(N)} \cL_{2,\xi}^{(N)}\dots \cL_{k_N,\xi}^{(N)} 1 \right)(x)
\label{Nagaev1},\\
&\Phi_N(x, \xi|\fA)=\frac{\left(\cL_{1,\xi}^{(N)} \cL_{2,\xi}^{(N)}\dots \cL_{N,\xi}^{(N)} 1_{\fA} \right)(x)}{\Prob_x[X_{k_N+1}^{(N)}\in \fA]}. \label{Nagaev3}
\end{align}
\end{lemma}
\begin{proof}
$ \E(e^{i\xi S_N}v(X_{k_N+1}^{(N)})\big|X^{(N)}_1=x)=$
$$\int p_1^{(N)}(x, y) e^{i\xi f_1^{(N)}(x, y)}
\EXP\bigl(e^{i\xi \sum_{n=2}^{N}f_n^{(N)}} v|X^{(N)}_2=y\bigr)  d\mu_2^{(N)} (y).$$ Proceeding by induction, we obtain \eqref{Nagaev2}, and
\eqref{Nagaev2} implies \eqref{Nagaev1},\eqref{Nagaev3}.
\qed
\end{proof}

Let $\|\cdot\|$ denote the operator norm on $\mathrm{Hom}(L^\infty,L^\infty)$.
\begin{lemma}
\label{LmThreeTerm}
$\cL_{n,\xi}^{(N)}$ are bounded linear operators, and there is a positive constant $\teps$ which only depends on $\eps_0$ such that  for all $N\geq 1$ and $5\leq n\leq k_N$, $\|\cL_{n,\xi}^{(N)}\|\leq 1$, and
$$ \left\Vert
\cL_{n-4, \xi}^{(N)}\cL_{n-3, \xi}^{(N)}
 \cL_{n-2, \xi}^{(N)}\cL_{n-1, \xi}^{(N)}
 \cL_{n, \xi}^{(N)}
\right\Vert \leq e^{-\teps d_n^{(N)}(\xi)^2}. $$
\end{lemma}

\begin{proof} Throughout this proof we fix $N$ and drop the superscripts $^{(N)}$, and we use the notation $x_i,z_i$ etc. to denote points in $\fS_i=\fS_i^{(N)}$.

It is clear that $\|\cL_{n,\xi}^{(N)}\|\leq 1$. To estimate the norm of
$$
\cL:=\cL_{n-4, \xi}\cL_{n-3, \xi}\cL_{n-2, \xi}^{(N)}\cL_{n-1, \xi}^{(N)}\cL_{n, \xi}:L^\infty(\fS_{n+1})\to L^\infty(\fS_{n-4}),
$$
we represent this operator as an integral operator, and analyze the kernel.
Let
\begin{enumerate}[$\circ$]
\item
$\DS p(x_k,\ldots,x_m):=\prod_{i=k}^{m-1}p_i(x_i,x_{i+1}),$
\item $\DS f(x_k,\ldots,x_m):=\sum_{i=k}^{m-1}f_i(x_i,x_{i+1}),$
\item
$
L(x_{n-4},z_{n+1})=$\\
\ \ \ \ \ \ \ \ \ \ \ $\displaystyle=\hspace{-0.6CM}\int\limits_{\fS_{n-3}\times\cdots\times\fS_{n}}\hspace{-0.6CM}
p(x_{n-4},z_{n-3},\ldots,z_{n+1})e^{i\xi f(x_{n-4},z_{n-3},\ldots,z_{n+1})}\mu_{n-3}(dz_{n-3})\cdots\mu_n(dz_n).
$
\end{enumerate}
Then
$
\displaystyle (\cL v)(x_{n-4})=\int_{\fS_{n+1}} \biggl[L(x_{n-4},z_{n+1}) v(z_{n+1})\biggr] \mu_{n+1}(dz_{n+1})
$, whence
$$
\|\cL v\|_\infty \leq \|v\|_\infty \sup_{x_{n-4}\in\fS_{n-4}}\int_{\fS_{n+1}}|L(x_{n-4},z_{n+1})|\mu_{n+1}(dz_{n+1}).
$$
 To estimate this integral we change the order of integration:
\begin{align}
&\int_{\fS_{n+1}}|L(x_{n-4},z_{n+1})|\mu_{n+1}(dz_{n+1})\leq \iint\limits_{\fS_{n-2}\times\fS_{n+1}}\Bigg[
|K_{n}(z_{n-2},z_{n+1})| \notag \\
&\hspace{0.2cm}\left.
\int_{\fS_{n-3}}p(x_{n-4},z_{n-3},z_{n-2}) \mu_{n-3}(dz_{n-3})\right]
\mu_{n-2}(dz_{n-2})\mu_{n+1}(dz_{n+1}), \label{L-estimate}
\end{align}
where $K_n(z_{n-2},z_{n+1}):=$
\begin{align*}
&\hspace{-0.3cm}\iint\limits_{\fS_{n-1}\times\fS_n}\hspace{-0.2cm} p(z_{n-2},z_{n-1},z_n,z_{n+1})
 e^{i\xi f(z_{n-2},z_{n-1},z_n,z_{n+1})}\mu_{n-1}(dz_{n-1})\mu_n(dz_n).
\end{align*}

\medskip
\noindent
{\sc Claim:} {\em Let $p(z_{n-2}\to z_{n+1}):=\Prob(X_{n+1}=z_{n+1}|X_{n-2}=z_{n-2})$, then}
\begin{equation}\label{K-estimate}
\begin{aligned}
&|K_n(z_{n-2},z_{n-1})|\leq
p(z_{n-2}\to z_{n+1})-\\
&\hspace{1cm}-\frac{1}{4}p(z_{n-2}\to z_{n+1})
\E\biggl(|e^{i\xi\Gamma(P)}-1|^2\bigg|{ \begin{array}{l}
X_{n-2}=Y_{n-2}=z_{n-2}\\
X_{n+1}=Z_{n+1}=z_{n+1}
\end{array}}\biggr).
\end{aligned}
\end{equation}

\medskip
\noindent
{\em Proof of the claim.\/}
Set
$\wt{K}_n(z_{n-2},z_{n+1}):=\frac{K_n(z_{n-2},z_{n+1})}{p(z_{n-2}\to z_{n+1})},
$
then
$$\wt{K}_n(z_{n-2},z_{n+1})=\E\biggl(e^{i\xi \sum_{k=n-2}^n f_k(X_k,X_{k+1})}\bigg|{\begin{array}{l}
X_{n-2}=z_{n-2}\\
X_{n+1}=z_{n+1}
\end{array}}\biggr).$$
Writing  $|\wt{K}_n(z_{n-2},z_{n+1})|^2=\wt{K}_n(z_{n-2},z_{n+1})\ov{\wt{K}_n(z_{n-2},z_{n+1})}$, we find that
\begin{align*}
&|\wt{K}_n(z_{n-2},z_{n+1})|^2=
\E\left(e^{i\xi\Gamma\bigl({\tiny X_{n-2} \begin{array}{l} X_{n-1}\\ Y_{n-1}\end{array} \begin{array}{l} X_{n}\\ Y_{n}\end{array} X_{n+1}}\bigr)}\bigg|{ \begin{array}{l}
X_{n-2}=Y_{n-2}=z_{n-2}\\
X_{n+1}=Z_{n+1}=z_{n+1}
\end{array}}\right),
\end{align*}
where $\{Y_n\}$ is an independent copy of $\{X_n\}$, and $\Gamma$ is as in \eqref{balance}.

The imaginary part  is necessarily zero, so writing $P=\bigl({ X_{n-2} \begin{array}{l} X_{n-1}\\ Y_{n-1}\end{array} \begin{array}{l} X_{n}\\ Y_{n}\end{array} X_{n+1}}\bigr)$ we have by the identity $1-\cos\alpha=\frac{1}{2}|e^{i\alpha}-1|^2$ that
\begin{align*}
&|\wt{K}_n(z_{n-2},z_{n-1})|^2=1-\E\bigl(1-\cos(\xi\Gamma(P))|{\tiny \begin{array}{l}
X_{n-2}=Y_{n-2}=z_{n-2}\\
X_{n+1}=Z_{n+1}=z_{n+1}
\end{array}}\bigr)\\
&\equiv 1-\frac{1}{2}\E\bigl(|e^{i\xi\Gamma(P)}-1|^2|{\tiny \begin{array}{l}
X_{n-2}=Y_{n-2}=z_{n-2}\\
X_{n+1}=Z_{n+1}=z_{n+1}
\end{array}}\bigr).
\end{align*}
The claim follows, since $\sqrt{1-t}\leq 1-\frac{t}{2}$ for all $0\leq t\leq 1$.

\medskip
We now   substitute \eqref{K-estimate} in \eqref{L-estimate}. The result is a difference of two terms:
\begin{enumerate}[(a)]
\item The first term is obtained by replacing $K_n(z_{n-2},z_{n+1})$ in \eqref{L-estimate} by $p(z_{n-2}\to z_{n+1})$. It has the following {\em upper} bound:
$$
\iint_{\fS_{n-2}\times\fS_{n+1}}\int_{\fS_{n-3}}  p(x_{n-4},z_{n-3},z_{n-2})p(z_{n-2}\to z_{n+1}) =1.
$$
\item The second term is obtained by replacing $K_n(z_{n-2},z_{n+1})$ in \eqref{L-estimate} by
$$\frac{1}{4}p(z_{n-2}\to z_{n+1})\E\bigl(|e^{i\xi\Gamma(P)}-1|^2|{\tiny \begin{array}{l}
X_{n-2}=Y_{n-2}=z_{n-2}\\
X_{n+1}=Z_{n+1}=z_{n+1}
\end{array}}\bigr).$$ The inner-most  integral satisfies
$
\DS\int_{\fS_{n-3}} p(x_{n-4}, z_{n-3},z_{n-2})\mu_{n-3}(dz_{n-3})\geq \epsilon_0
$
because of uniform ellipticity. This leads to the following {\em lower} bound for the second term:
$$
 \frac{1}{4}\epsilon_0^2 \E\bigl(|e^{i\xi\Gamma(P)}-1|^2\bigr)=\frac{1}{4}\epsilon_0^2 d_n(\xi)^2.
$$
\end{enumerate}
In total we get: $\int|L(x_{n-4},z_{n+1})|\mu_{n+1}(dz_{n+1})\leq 1-\wt{\epsilon} d_n(\xi)^2$, where
$
\wt\epsilon:=\frac{1}{4}\epsilon_0^2.
$
 Since $1-t\leq e^{-t}$, we are done.\qed \end{proof}

Recall that
$D_N(\xi)=\sum_{n=3}^{k_N} d_n^{(N)}(\xi)^2$.
Write
$\DS D_N=\sum_{j=0}^4 D_{j,N}$ where
$$ D_{j, N}(\xi)=\sum_{\substack{3\leq n\leq k_N\\n\equiv j \text{ mod }5}} d_n^{(N)}(\xi)^2. $$
Applying Lemma \ref{LmThreeTerm} iteratively  we conclude that there is a constant $C$ independent of $N$ s.t. for all $N$,
\begin{equation}
\label{BoundChar}
 \left|\Phi_N(x, \xi)\right|\leq
C e^{-\teps \max\left(D_{0,N},\ldots, D_{4,N}
\right)}
\leq C e^{-\frac{1}{5}\teps D_N(\xi)}.
\end{equation}

If $\Prob(X_{k_N+1}^{(N)}\in \fA)\geq\brdelta$ then by \eqref{Nagaev3},
 $|\Phi_N(x,\xi|\fA)|\leq \brdelta^{-1}\|\cL_{1,\xi}^{(N)} \cL_{2,\xi}^{(N)}\dots \cL_{k_N,\xi}^{(N)} 1_{\fA}\|$
 whence
\begin{equation}
\label{CondBoundChar}
 \left|\Phi_N(x, \xi|\fA)\right|\leq C e^{-\frac{1}{5}\teps D_N(\xi)}.
\end{equation}

The next result shows that  if $u_n^{(N)}$ is big,
then $d_n^N(\cdot)$ cannot be small at two nearby points. Recall the standing assumption $\ess\sup\|f^{(N)}_n\|_\infty\leq K$,
 and the definition  of the structure constants $u_n^{(N)}$ in \eqref{Structure-Constants}.

\begin{lemma}
\label{Lm2P}
$\exists\tdelta=\tdelta(K)>0$ s.t. if $|\delta|\leq \tdelta$ then for all $3\leq n\leq k_N$,
\begin{align}
d_n^{(N)}(\xi+\delta)^2&\geq \frac{2}{3}\delta^2 \left(u_n^{(N)}\right)^2-2|\delta| u_n^{(N)} d_n^{(N)}(\xi).
\end{align}
\end{lemma}
\begin{proof}
Fix a hexagon $P={ \left(x_{n-2} \begin{array}{l} x_{n-1}\\ y_{n-1} \end{array} \begin{array}{l} x_{n}\\ y_{n} \end{array} y_{n+1}\right)}\in\mathrm{Hex}(N,n)$, and let
$$
\fu_n:=\Gamma(P)\ , \ \fd_n(\xi):=|e^{i\xi\fu_n}-1|,
$$
then
the identity $|e^{i\theta}-1|^2=2(1-\cos\theta)$ implies
\begin{align}
&\fd_n^2(\xi+\delta) =|e^{i(\xi+\delta)\fu_n}-1|^2=2[1-\cos((\xi+\delta)\fu_n)]\notag\\
&=2[1-\cos(\xi\fu_n)\cos(\delta\fu_n)+\sin(\xi\fu_n)\sin(\delta\fu_n)]\notag\\
&=2[(1-\cos(\xi\fu_n))
\cos(\delta\fu_n)+(1-\cos(\delta\fu_n))+\sin(\xi\fu_n)\sin(\delta\fu_n)]\label{sin-con}\\
&\geq 2\bigl[(1-\cos(\delta\fu_n))-|\sin(\xi\fu_n)\sin (\delta\fu_n)|\bigr]\notag \text{ provided }|\tdelta|<\frac{\pi}{12K},
\end{align}
because in this case $|\delta\fu_n|<\frac{\pi}{2}$, so $\cos(\delta\fu_n)\geq 0$. Make $\tdelta$ even smaller to guarantee
$
0\leq |t|\leq 6K\tdelta\Rightarrow \tfrac{1}{3}t^2\leq 1-\cos t\leq t^2$, then
\begin{align*}
&\fd_n^2(\xi+\delta) \geq 2\bigl(
\tfrac{1}{3}\delta^2\fu_n^2-
|\delta\fu_n|\sqrt{1-\cos^2(\xi\fu_n)}
\bigr)\\
&=2\biggl(
\tfrac{1}{3}\delta^2 \fu_n^2-|\delta\fu_n|\sqrt{(1-\cos(\xi\fu_n))(1+\cos(\xi\fu_n))}\biggr)\\
&\geq 2\biggl(
\tfrac{1}{3}\delta^2 \fu_n^2-|\delta \fu_n|\sqrt{2(1-\cos(\xi\fu_n))}\biggr)
= \tfrac{2}{3}\delta^2 \fu_n^2-2|\delta\fu_n| |e^{i\xi\fu_n}-1|\\
&=\tfrac{2}{3}\delta^2 \fu_n^2-2|\delta\fu_n|\fd_n(\xi).
\end{align*}
Integrating on $P\in \mathrm{Hex}(N,n)$, and using Cauchy-Schwarz to estimate the second term
we obtain the lower bound for $d_n(\xi+\delta)^2$.\qed
\end{proof}

Lemma \ref{Lm2P} and the Cauchy-Schwarz inequality together give
\begin{align}
D_N(\xi+\delta)&\geq \frac{2}{3}\delta^2 U_N-2|\delta| \sqrt{U_N D_N(\xi)}\label{D-Lower}
\end{align}
where $\DS U_n:=\sum_{k=3}^{k_N} (u_k^{(N)})^2$.
If $V_N:=\Var(S_N)\to\infty$, then   as soon as $V_N>2C_2$
where $C_2$ is the constant from Theorem \ref{LmVarCycles}, we have\begin{equation}
\label{V-U-Rat}
\frac{U_N}{2C_1}\leq V_N\leq 2 C_1 U_N.
\end{equation}
So there are   $\heps_1, \hc_1>0$ s.t.
$
D_N(\xi+\delta)\geq \heps_1 \delta^2 V_N-\hc_1|\delta| \sqrt{V_N D_N(\xi)}.
$
By \eqref{BoundChar}, there are  $\heps, \hc>0$ s.t. for  all $N$ so large that $V_N>2C_2$, for all $\xi$ and $|\delta|<\tdelta$
\begin{equation}
\label{PhiPivot}
\left| \Phi_N(x, \xi+\delta)\right|
 \leq C \exp\left(-\heps V_N \delta^2+\hc |\delta| \sqrt{V_N D_N(\xi)}\right).
\end{equation}

We rephrase \eqref{PhiPivot} as follows. Given a compact interval $I\subset\R$, let
\begin{equation}\label{A_N}
A_N(I):=-\log \sup_{(x,\xi)\in\mathfrak S_1^{(N)}\times I}  \bigl|\Phi_N(x,\xi)
\bigr|
\end{equation}
and choose some pair $(\tx_N,\txi_N)\in \fS_1^{(N)}\times I$ such that
$$
A_N(I)\leq -\log|\Phi_N(\tx_N,\txi_N)|\leq A_N(I)+\ln 2.
$$
So  $|\Phi(\tx_N,\txi_N)| \geq \frac{1}{2} e^{-A_N(I)}=
 \frac{1}{2}\sup |\Phi_N(\cdot,\cdot)|$ on $\mathfrak S_1^{(N)}\times I$.

\begin{corollary}
\label{CrChar-DMax}
For each $\brdelta$ there are  $\tC,\heps,\brc>0$ s.t. for every compact interval $I$ s.t.  $|I|\leq \tdelta$,
for all $N$
for every $(x,\xi)\in \mathfrak S_1^{(N)}\times I$,
 for
every $\fA\subset\fS_{k_N+1}^{(N)}$ s.t. $\mu_{k_N+1}^{(N)}(\fA)\geq \brdelta$,
\begin{align*}
&\left|\Phi_N(x, \xi)\right|\leq \tC \exp\left(-\heps V_N (\xi-\txi_N)^2+\brc |\xi-\txi_N| \sqrt{V_N A_N(I)}\right); \\
& \left|\Phi_N(x, \xi|\fA)\right|
\leq \tC \exp\left(-\heps V_N (\xi-\txi_N)^2+\brc |\xi-\txi_N| \sqrt{V_N A_N(I)}\right).
\end{align*}
\end{corollary}

\begin{proof}
We only give the proof in the case $V_N$ is large, so that \eqref{PhiPivot} holds.
This is the case we need.  We remark that  the result also holds generally, because
the estimate we seek  is trivial when $V_N$ is small.

Applying \eqref{PhiPivot} with $\txi_N$
 instead of $\xi$ and $\delta=\xi-\txi_N$
gives
$$
 |\Phi_N(x,\xi)|\leq C \exp\left(-\heps V_N (\xi-\txi_N)^2+\hc |\txi_N-\xi| \sqrt{V_N D_N(\txi_N)}\right).$$
 By \eqref{BoundChar},
 $e^{-A_N(\txi_N)}\leq 2 |\Phi_N(\tx_N,\txi_N)|\leq 2Ce^{-\tfrac{1}{5}\teps D_N(\txi_N)}$. We conclude that
$$D_N(\txi)\leq C_1 A_N(I)+C_2$$ for some global constants $C_1,C_2$.
The estimate of $|\Phi_N(x,\xi)|$ follows.
The second estimate is proved in the same way.\qed
\end{proof}

\subsection{The LLT in the irreducible non-lattice  case}\label{section-proof-of-llt-irred-non-lattice}

We give the proof for arrays (Theorem \ref{ThLLT-classic}').
 Theorem \ref{ThLLT-classic} on chains follows, because every additive functional on a Markov chain is stably hereditary (Example \ref{Example-MC-hereditary}).

We begin by proving that $V_N\xrightarrow[N\to\infty]{}\infty.$
Otherwise $\liminf V_N<\infty$, and  one can find $N_\ell\uparrow\infty$ such that $\Var(S_{N_\ell})=O(1)$. Let $\mathsf X'$ denote the sub-array with rows ${\mathsf X'}^{(\ell)}=\mathsf X^{(N_\ell)}$. By Theorem \ref{Theorem-center-tight}, $\mathsf f|_{\mathsf X'}$ is center-tight, whence
$
G_{ess}(\mathsf X', \mathsf f|_{\mathsf X'})=\{0\}.
$
At the same time, $G_{ess}(\mathsf X, \mathsf f)=G_{alg}(\mathsf X, \mathsf f)=\R$, because $\mathsf f$ is irreducible and non-lattice. So $G_{ess}(\mathsf X', \mathsf f|_{\mathsf X'})\neq G_{ess}(\mathsf X, \mathsf f)$, in contradiction to the assumption that $\mathsf f$ is stably hereditary.

Next we fix $z_N\in\R$ such that $\frac{z_N-\E(S_N)}{\sqrt{V_N}}\to z$, and show that for every non-empty interval $(a,b)$, for every choice of $x_1^{(N)}\in\mathfrak S^{(N)}_1$ $(N\geq 1)$,
\begin{equation}\label{LLT-non-arithmetic-limit-proof}
 \Prob_{x_1^{(N)}}[S_N-z_N\in (a,b)]\sim\frac{e^{-z^2/2}}{\sqrt{2\pi V_N}}(b-a), \text{ as $N\to\infty$.}
\end{equation}

A well-known approximation argument \cite{S}, \cite[chapter 10]{Br}
reduces \eqref{LLT-non-arithmetic-limit-proof} to showing that for all $\phi\in L^1(\R)$ whose Fourier transform
$\wh{\phi}(\xi):=\int_{\R}e^{-i\xi u}\phi(u)du$ has compact support,
\begin{equation}\label{ExpLLT3}
\lim_{N\to \infty} \sqrt{V_N} \EXP_{x_1^{(N)}}\bigl[\phi\bigl(S_N-z_N\bigr)\bigr]=\frac{ e^{-z^2/2}}{\sqrt{2\pi}}
 \int_{-\infty}^\infty \phi(u) du.
\end{equation}

\medskip
Fix $\phi\in L^1$ such that  $\supp(\hphi) \subseteq [-L, L].$
By the Fourier inversion formula,
$
\displaystyle\EXP_{x_1^{(N)}}(\phi(S_N-z_N))=\frac{1}{2\pi} \int_{-L}^L \hphi(\xi) \Phi_N(x_1^{(N)},\xi) e^{-i\xi z_N} d\xi.
$
So \eqref{ExpLLT3} is equivalent to
\begin{equation}\label{ExpLLT4}
\lim_{N\to\infty}\sqrt{V_N}\cdot \frac{1}{2\pi} \int_{-L}^L \hphi(\xi)
\Phi_N(x_1^{(N)},\xi) e^{-i\xi z_N} d\xi=\frac{ e^{-z^2/2}}{\sqrt{2\pi}}
 \wh{\phi}(0).
\end{equation}
Below, we give a proof of \eqref{ExpLLT4}.

We note for future reference that the proof
of \eqref{ExpLLT4} below works under the milder assumption that $\wh{\phi}$ is
bounded, continuous at zero and has compact support, e.g.  $\wh{\phi}=\frac{1}{2\pi}1_{[-\pi,\pi]}$
(which is the Fourier transform of
$\phi(u)=\frac{\sin(\pi u)}{\pi u}\not\in L^1$).

Divide $[-L, L]$ into segments $I_j$ of length $\leq \tdelta$ where $\tdelta$ is given by Lemma \ref{Lm2P},   so that $I_0$ is centered at $0$.
Let
$$
J_{j,N}:=\frac{1}{2\pi}\int_{I_j} \hphi( \xi)\Phi_N(x_1^{(N)},\xi)e^{-i\xi z_N} d\xi.
$$

\medskip
\noindent
{\sc Claim 1 (contribution of $J_{0,N}$):}
\begin{equation}
\label{MainInt}
 \sqrt{V_N} J_{0,N}\xrightarrow[N\to\infty]{} \frac{1}{\sqrt{2\pi}} e^{-z^2/2} \wh{\phi}(0).
\end{equation}

\medskip
\noindent
{\em Proof of the claim.\/}
Fix $R>0$. Since $J_{0,N}\owns 0$, $A_N(J_{0,N})=0$.
By Corollary \ref{CrChar-DMax}, given $\eps>0$ there is $R>0$ such that
$$\left| \sqrt{V_N} \int_{\{\xi\in I_0: |\xi|>R /\sqrt{V_N}\}}
\hphi(\xi)\Phi_N(x_1^{(N)},\xi)e^{-i\xi z_N} d\xi\right| {\leq \eps}. $$
Next, a change of variables $\xi=s/\sqrt{V_N}$ gives
\begin{align*}
 &\sqrt{V_N} \int_{[|\xi|\leq R /\sqrt{V_N}]} \hphi(\xi)\Phi_N(x_1^{(N)},\xi)e^{-i\xi z_N} d\xi
 =\int_{[|s|\leq R]} \hphi\left(\frac{s}{\sqrt{V_N}}\right)\;
\E_{x^{(N)}_1}(e^{i s\frac{S_N-z_N}{\sqrt{V_N}}}) \; ds.
\end{align*}

By Dobrushin's CLT for inhomogeneous Markov arrays  (Theorem \ref{Theorem-Dobrushin})
$\frac{S_N-z_N}{\sqrt{V_N}}$ converges in distribution w.r.t. $\Prob_{x_1^{(N)}}$ to the  normal distribution
 with mean $-z$ and variance $1$.
 By L\'evy's continuity theorem, this implies that
 $$\E_{x_1^{(N)}}(e^{i s\frac{S_N-z_N}{\sqrt{V_N}}})\xrightarrow[N\to\infty]{}e^{-isz-s^2/2}$$
 uniformly on compacts, and so
\begin{align*}
 &\sqrt{V_N} \int_{|\xi|\leq R /\sqrt{V_N}} \hphi(\xi)\Phi_N(x_1^{(N)},\xi)e^{-i\xi z_N} d\xi=
\widehat{\phi}(0)\int_{-R}^R e^{-isz}e^{-s^2/2}ds+o_{N\to\infty}(1).
\end{align*}
Since this is true for all $R$,  we can let $R\to\infty$ sufficiently slow to obtain \eqref{MainInt}.

\medskip
\noindent
{\sc Claim 2 (contribution of the other  $J_{j,N}$):}
$
\sqrt{V_N} J_{j,N}\xrightarrow[N\to\infty]{} 0\text{ for }j\neq 0.
$

\noindent
{\em Proof of the claim.} Since $\mathsf f$ is irreducible with algebraic range $\R$, the co-range of $\mathsf f$ is $\{0\}$ (Theorems \ref{Theorem-co-range}, \ref{Theorem-Results-for-arrays}). Since $\mathsf f$ is stably hereditary,
$$
D_N(\xi)\xrightarrow[N\to\infty]{}\infty\text{ uniformly on compacts in  }\R\setminus\{0\}.
$$
By  \eqref{BoundChar}, $\Phi_N(x_1^{(N)},\xi)\to 0$ uniformly on compacts in $\R\setminus\{0\}$.

We will use this to show that for any interval $I\subset \R\setminus\{0\}$
\begin{equation}
\label{SmallL1Norm}
 \sqrt{V_N} \int_I |\Phi(x_1^{(N)}, \xi)|d\xi\to 0.
\end{equation}

By subdividing $I$ into finitely many subintervals we see that it suffices to prove the claim for
$I=I_j$ for some $j.$
Recall that $A_N(I_j)=-\log\sup |\Phi_N(\cdot,\cdot)|$ on $\mathfrak S_1^{(N)}\times I_j$, and  $(\tx_{j,N}, \txi_{j,N})$ are  points where this supremum is achieved up to factor $2$.
Set $A_{j,N}:=A_N(I_j)$, then
 $A_{j,N}\to \infty$ as $N\to\infty$ for each $j\neq 0.$

Take large $R$ and split $I_j$ into two regions
$$ I_{j,N}':=\left\{\xi\in I_{j}: |\xi-\txi_{j,N}|\leq R \sqrt{\frac{A_{j,N}}{V_N}}\right\}, \quad
I_{j, N}'':=I_{j}\setminus I_{j,N}' .$$
Split the integral $\int_{I_j} |\Phi(x_1^{(N)}, \xi)| d\xi $ into two integrals
$ J_{j,N}'$, $J_{j,N}''$ accordingly.
\begin{enumerate}[$\circ$]
\item On $I_{j,N}'$,
$|\Phi_N(x_1^{(N)},\xi)|\leq  e^{-A_{j,N}}$
and $|I_{j,N}'|\leq 2R \sqrt{\frac{A_{j,N}}{V_N}}$, so
$$ \sqrt{V_N}|J_{j,N}'|\leq 2R  \sqrt{A_{j,N}} e^{-A_{j,N}}. $$

\item On $I_{j,N}''$, by  Corollary \ref{CrChar-DMax},
\begin{align*}
&|\Phi_N(x_1^{(N)},\xi)|\leq
\tC \exp\biggl(-\heps V_N |\xi-\txi_{j,N}|R\sqrt{\frac{A_{j,N}}{V_N}}+\brc|\xi-\txi_{j,N}|\sqrt{V_N A_{j,N}}\biggr)\\
&\leq \tC\exp\biggl(-\frac{\heps}{2}{|\xi-\txi_{j,N}|\sqrt{A_{j,N}V_N}}
\biggr),{ \text{ provided }R\heps>\brc+\frac{\heps}{2}}.
\end{align*}
Hence
$\displaystyle\sqrt{V_N}J''_{j,N}\leq \sqrt{V_N} \tC\int_{-\infty}^{\infty} e^{-\frac{\heps}{2} |s|\sqrt{A_{j,N} V_N}}ds
=O(A_{j,N}^{-\frac{1}{2}}).
$
\end{enumerate}
Combining these estimates, we obtain
\begin{equation}
\label{Jj-Aj}
 \sqrt{V_N} \|\Phi_N(x_1^{(N)}, \cdot)\|_{L^1(I_j)} \leq 2R \sqrt{A_{j,N}} \; e^{-A_{j,N}}+\frac{C}{\sqrt{A_{j,N}}}.
 \end{equation}
 Since $A_{j,N}\to \infty$ as $N\to\infty$ \eqref{SmallL1Norm} follows.

Since
$\DS |J_{j,N}|\leq \frac{\|\hphi\|_\infty \; \|\Phi_N(x_1^{(N)}, \cdot)\|_{L_1(I_j)}}{2\pi} $,
claim 2 follows from \eqref{SmallL1Norm}.
\smallskip

\begin{remark}
Note that in the proof of
\eqref{SmallL1Norm}  the irreducibility assumption is only used at the last sentence, namely, to conclude that
$A_{N,j}\to 0$ as $N\to\infty.$ In particular, \eqref{Jj-Aj} holds for arbitrary arrays, irreducible or not.
\end{remark}

\smallskip
Claims 1 and 2 imply \eqref{ExpLLT4}, and \eqref{ExpLLT4} implies \eqref{LLT-non-arithmetic-limit-proof} by
\cite[chapter 10]{Br}.
This proves the LLT theorem for initial distributions concentrated at single points (i.e. $\Prob=\Prob_{x^{(N)}_1}$). To deduce the theorem for arbitrary initial distribution $\mu^{(N)}_1(dx_1^{(N)})$, it is sufficient to prove the following claim and then integrate:

\smallskip
\noindent
{\sc Claim 3:}  \eqref{LLT-non-arithmetic-limit-proof} holds uniformly with respect to the choice of $\{x^{(N)}_n\}$.

\medskip
\noindent
{\em Proof of the claim.\/} Assume by contradiction that this is false, then there exists $\eps>0$ and  $N_k\to\infty$ with  $y^{(N_k)}_1$ such that
$
\Prob_{y_1^{(N_k)}}[S_{N_k}-z_{N_k}\in (a,b)]\big/\frac{e^{-z^2/2}(b-a)}{\sqrt{2\pi V_{N_k}}}\not\in [e^{-\eps},e^{\eps}].
$ But this contradicts \eqref{LLT-non-arithmetic-limit-proof}
for any sequence $\{x_1^{(N)}\}$ such that $x^{(N_k)}_1=y^{(N_k)}_1$.
\qed

\subsection{The LLT for the irreducible lattice case}
We give the proof in the context of arrays (Theorem \ref{Thm-LLT-Lattice}'): $\mathsf X$ is a uniformly elliptic array, and $\mathsf f$ is an additive functional on $\mathsf X$ which is a.s. uniformly bounded, hereditary, irreducible, and with algebraic range $t\Z$ with $t>0$. Without loss of generality, $t=1$, otherwise work with $t^{-1}\mathsf f$.

By Lemma \ref{l.alg-range} and the assumption that $G_{alg}(\R)=\Z$, there are constants $c^{(N)}_n$ such that
$
f^{(N)}_n(X^{(N)}_n, X^{(N)}_{n+1})-c^{(N)}_n\in \Z\text{ a.s. }
$
We may assume without loss of generality that $c^{(N)}_n=0$, otherwise we work with $\mathsf f-\mathsf c$. So
$$
S_N\in\Z\text{ a.s. for every $N\geq 1$.}
$$
We will show that for every sequence of numbers $z_N\in \Z$ such that $\frac{z_N-\E(S_N)}{\sqrt{V_N}}\to z$, and for every $x_1^{(N)}\in\fS^{(N)}_n$,
\begin{equation}\label{LLT-lattice-without-k}
\Prob_{x_1^{(N)}}(S_N=z_N)=[1+o(1)]\frac{e^{-z^2/2}}{\sqrt{2\pi V_N}},\text{ as $N\to\infty$}.
\end{equation}
As in the irreducible case, once we prove \eqref{LLT-lattice-without-k} for all choices of $\{x_1^{(N)}\}$, it automatically follows that \eqref{LLT-lattice-without-k} holds uniformly in  $\{x_1^{(N)}\}$.  Integrating over $(\fS^{(N)}_1,\mathcal B(\fS^{(N)}_1),\mu^{(N)}_1)$ gives \eqref{LLT-lattice-with-k} with $k=0$. For general $k$, take $z_N':=z_N+k$.

\medskip
The assumptions on $\mathsf f$ imply that
$\Var(S_N)\xrightarrow[N\to\infty]{}\infty$. The proof is a routine modification of the argument we used in the non-lattice case, so we omit it.

Observe that $\frac{1}{2\pi}\int_{-\pi}^\pi e^{im\xi}d\xi$ is equal to zero when $m\in\Z\setminus\{0\}$, and equal to one when $m=0$. In particular, since $S_N-z_N\in\Z$ almost surely, for every $x_1^{(N)}\in\fS^{(N)}_1$
$$
\Prob_{x_1^{(N)}}(S_N-z_N=0)=\E_{x_1^{(N)}}\left(\frac{1}{2\pi}\int_{-\pi}^\pi e^{i\xi(S_N-z_N)}d\xi\right)=\frac{1}{2\pi}\int_{-\pi}^{\pi} \Phi(x_1^{(N)},\xi) e^{-i\xi z_N}d\xi.
$$
Thus to prove \eqref{LLT-lattice-without-k} it is sufficient to show that
\begin{equation}\label{ExpLLT5}
\lim_{N\to\infty}\sqrt{V_N}\cdot \frac{1}{2\pi}\int_{-\pi}^\pi \Phi_N(x_1^{(N)},\xi)e^{-i\xi z_N}d\xi=\frac{1}{\sqrt{2\pi}} e^{-z^2/2}.
\end{equation}

Notice that \eqref{ExpLLT5} is  \eqref{ExpLLT4} in the  case $\phi(u)=\frac{\sin (\pi u)}{\pi u}$, $\wh{\phi}(\xi)=\frac{1}{2\pi}1_{[-\pi,\pi]}(\xi)$, and can be proved in almost exactly the same way.

Here is a sketch of the proof.  One divides $[-\pi,\pi]$ into segments $I_j$ of length less than the $\wt{\delta}$ of Lemma \ref{Lm2P}.

The contribution of the interval which contains zero is asymptotic to $\frac{1}{\sqrt{2\pi V_N}} e^{-z^2/2}$. This is shown  as in claim 1 of the preceding proof.

 The remaining intervals are bounded away from $2\pi\Z$. Their contribution is $o(1/\sqrt{V_N})$. This can be seen as in claim 2 of the preceding proof, using the facts that since $\mathsf f$ is irreducible with algebraic range $\Z$, $H(\mathsf X,\mathsf f)=2\pi\Z$ (Theorems \ref{Theorem-co-range}, \ref{Theorem-Results-for-arrays}), and since $f$ is hereditary and $G_{alg}(f)=\Z$, $f$ is stably hereditary,  whence
    $D_N(\xi)\xrightarrow[N\to\infty]{} 0$
    uniformly on compacts in $\R\setminus 2\pi\Z.$\qed

\subsection{The mixing LLT}

The proof is very similar to the proof of the local limit theorem,  except that we  use $\Phi(x,\xi|\mathfrak A)$ instead of  $\Phi(x,\xi)$.

\medskip
We outline the proof in the non-lattice case, and leave the lattice case to the reader.
Suppose $\mathsf X$ is a uniformly elliptic Markov array, and that $\mathsf f$ is a.s. uniformly bounded, stably hereditary, irreducible and with algebraic range $\R$.

Let $\mathfrak A_{N}\in\fS_{k_N+1}^{(N)}$ be measurable sets s.t. $\Prob(X^{(N)}_{k_N+1}\in\mathfrak A_{N})>\delta>0$, and let $x_N\in\fS^{(N)}_1$ be points. Suppose $\frac{z_N-\E(S_N)}{\sqrt{V_N}}\to z$.
As before, $V_N\to\infty$, and a standard approximation argument (\cite{Br}, chapter 10) says that it is enough to show that for every $\phi\in L^1(\R)$ s.t. $\supp(\hat\phi)\subset [-L,L]$,
$$
\lim_{N\to\infty}\sqrt{V_N}\cdot \frac{1}{2\pi} \int_{-L}^L \hphi(\xi)
\Phi_N(x_N,\xi|\mathfrak A_{N}) e^{-i\xi z_N} d\xi=\frac{ e^{-z^2/2}}{\sqrt{2\pi}}
 \wh{\phi}(0).
$$

Divide $[-L,L]$ as before into intervals $I_j$ of length $\leq \wt{\delta}$ where $\wt{\delta}$ is given by Lemma \ref{Lm2P} and $I_0$ is centered at zero,  and let
$$
J_{j,N}:=\frac{1}{2\pi}\int_{I_j}\hat{\phi}(\xi)\Phi_N(x_N,\xi|\mathfrak A_{N})e^{-i\xi z_N}d\xi.
$$

\medskip
\noindent
{\sc Claim 1:} $\sqrt{V_N} J_{0,N}\xrightarrow[N\to\infty]{}(2\pi)^{-\frac{1}{2}}e^{-z^2/2}\hat\phi(0)$.

\medskip
\noindent
{\em Proof of the claim\/}: Fix $R>0$.
As before,
 applying Corollary \ref{CrChar-DMax} with $A_N=0$ we conclude that for each $\eps>0$
there is $R>0$ such that
$$
\left|\sqrt{V_N}\int_{\{\xi\in I_0: |\xi|>R/\sqrt{V_N}\}}\hat\phi(\xi)\Phi_N(x_N, \xi|\mathfrak A_{N})e^{-i\xi z_N}d\xi\right|{ \leq \eps.}
$$
Next the change of variables $\xi=s/\sqrt{V_N}$ gives
\begin{align}
&\sqrt{V_N}\int_{\{\xi\in I_0: |\xi|\leq R/\sqrt{V_N}\}}\hat\phi(\xi)\Phi(x_N, \xi|\mathfrak A_{N})e^{-i\xi z_N}d\xi\notag\\
&=\int_{-R}^R \hat\phi\left(\frac{s}{\sqrt{V_N}}\right)\E_{x_N}\left(e^{is(\frac{S_N-z_N}{\sqrt{V_N}})}\bigg| X^{(N)}_{k_N+1}\in\mathfrak A_{N}\right)d\xi\notag\\
&=\frac{1}{\Prob(X^{(N)}_{k_N+1}\in\mathfrak A_{N})}\int_{-R}^R \hat\phi\left(\frac{s}{\sqrt{V_N}}\right)\E_{x_N}\left(e^{is(\frac{S_N-z_N}{\sqrt{V_N}})}1_{\mathfrak A_{N}}(X^{(N)}_{k_N+1})\right)d\xi.\label{Shab'a}
\end{align}

We analyze the expectation in the integrand.
Take $1\leq r_N\leq k_N$ such that  $r_N\to\infty$ and   $r_N/\sqrt{V_N}\to 0$, and let
$$
S_N^\ast:=\sum_{j=1}^{k_N-r_N}f_j^{(N)}(X^{(N)}_j,X^{(N)}_{j+1})\equiv S_N-\sum_{j=k_N-r_N+1}^{k_N}f_j^{(N)}(X_j^{(N)},X^{(N)}_{j+1}).
$$
Since $\ess\sup|\mathsf f|<\infty$, $|S_N-S_N^\ast|=o(\sqrt{V_N})$, and so
\begin{align*}
&\E_{x_N}\left(e^{is(\frac{S_N-z_N}{\sqrt{V_N}})}1_{\mathfrak A_{N}}(X^{(N)}_{k_N+1})\right)=
\E_{x_N}\left(e^{is(\frac{S_{N}^\ast-z_N}{\sqrt{V_N}})}1_{\mathfrak A_{N}}(X^{(N)}_{k_N+1})\right)+o(1)\\
&=\E_{x_N}\left(e^{is(\frac{S_{N}^\ast-z_N}{\sqrt{V_N}})}\E\biggl(1_{\mathfrak A_{N}}(X^{(N)}_{k_N+1})|X^{(N)}_1,\ldots,X^{(N)}_{k_n-r_N}\biggr)\right)+o(1)\\
&=\E_{x_N}\left(e^{is(\frac{S_{N}^\ast-z_N}{\sqrt{V_N}})}\E\biggl(1_{\mathfrak A_{N}}(X^{(N)}_{k_N+1})|X^{(N)}_{k_n-r_N}\biggr)\right)+o(1)\text{ by the Markov property}\\
&\overset{!}{=}\E_{x_N}\left(e^{is(\frac{S_{N}^\ast-z_N}{\sqrt{V_N}})}
\left[\Prob(X^{(N)}_{k_N+1}\in\mathfrak A_{N})+O(\theta^{r_N})\right]\right)+o(1),\text{ where $0<\theta<1$}
\end{align*}
and $\overset{!}{=}$ uses the exponential mixing estimate \eqref{Exp-Mixing-L-infinity}.
Since $\Prob(X^{(N)}_{k_N+1}\in\mathfrak A_{N})$ is bounded below, and $\frac{S_{N}^\ast-z_N}{\sqrt{V_N}}$ converges in distribution to the standard normal distribution by Dobrushin's theorem, we may conclude that
$$
\E_{x_N}\left(e^{is(\frac{S_N-z_N}{\sqrt{V_N}})}1_{\mathfrak A_{N}}(X^{(N)}_{k_N+1})\right)= \frac{1+o(1)}{\sqrt{2\pi}}e^{-z^2/2-izs}\Prob\left(X^{(N)}_{k_N+1}\in\mathfrak A_{N}\right).
$$
Substituting this in \eqref{Shab'a} gives the claim.

\medskip
\noindent
{\sc Claim 2:} $\sqrt{V_N}J_{j,N}\xrightarrow[N\to\infty]{}0$ for $j\neq 0$.

\medskip
\noindent
The claim is proved as in the previous proof, but with \eqref{CondBoundChar} replacing \eqref{BoundChar}.
Together, claims 1 and 2 imply the theorem.\qed

\section{Notes and references}
For a brief account of the history of the local limit theorem, see the end of the preface.

Many of the techniques we used in this chapter have a long history. The reduction of the LLT to the asymptotic analysis of the integrals  \eqref{ExpLLT4} and \eqref{ExpLLT5} for $\phi\in L^1$ with Fourier transforms with compact support was already used by Stone \cite{S} for proving local limit theorems for sums of iid random variables.  As mentioned at the end of the synopsis, the method of characteristic function operators is due to Nagaev \cite{N}, who used it to prove central and local limit theorems for homogeneous Markov chains, and this method was used extensively in dynamical systems.  Hafouta \& Kifer \cite{Hafouta-Kifer-Book}, Hafouta \cite{Hafouta-Skew-Products,Hafouta-Sequential}, and  Dragi\v{c}evi\'c, Froyland, \&  Gonz\'alez-Tokman \cite{Dragicevic-Froyland-Gonzalez-Tokman},  used this technique to prove the local limit theorem in a non-homogeneous setup.

The terminology ``mixing LLT" is due to R\'enyi \cite{Renyi-Mixing}, who initiated the study of the stability of limit theorems under conditioning and changes of measure.
The relevance of Mixing LLT to the study of reducible case is noted by
Guivarc'h \& Hardy \cite{GH}. Mixing LLT have numerous other applications including mixing
of special flows \cite{GH, DN19}, homogenization \cite{DN-Mech} and skew products
(see in particular, Theorem \ref{Theorem-Characterization-Irreducibility}
in Chapter \ref{Chapter-reducible}).
Mixing LLT for additive functionals of (stationary) Gibbs-Markov  processes were  proved
by Aaronson \& Denker \cite{Aaronson-Denker-LLT}.

\chapter{The local limit theorem in the reducible case}\label{Chapter-reducible}

\noindent
{\em In this chapter we prove the local limit theorem for $\Prob(S_N-z_N\in (a,b))$
when $\frac{z_N-\E(S_N)}{\sqrt{\Var(S_N)}}$ converges to a finite limit and $\mathsf f$ is  reducible. In the reducible case, the asymptotic behavior of $\Prob(S_N-z_N\in (a,b))$  depends on the details of $f_n(X_n,X_{n+1})$. The dependence is strong for small intervals, and weak for large intervals.}

\section{Main results}

\subsection{Heuristics and warm up examples}
An additive functional is called  {\bf reducible}\index{reducible!local limit theorems} if $$\mathsf f=\mathsf g+\mathsf c$$ where $\mathsf c$ is center-tight, and the algebraic range of $\mathsf g$ is strictly smaller than the algebraic range of $\mathsf f$. By the  results of Chapter \ref{Chapter-Irreducibility},
if $\Var(S_N(\mathsf f))\to\infty$, $\mathsf X$ is uniformly elliptic, and $\mathsf f$ is a.s. bounded, then
we can choose $\mathsf g$ to be irreducible.
In this case
$$
S_N(\mathsf f)=S_N(\mathsf g)+S_N(\mathsf c).
$$
where  $\Var(S_N(\mathsf g))\sim \Var(S_N(\mathsf f))\to\infty$,  $\Var(S_N(\mathsf c))=O(1)$, and $S_N(\mathsf g)$ satisfies the lattice local limit theorem.
{\bf The contribution of $S_n(\mathsf c)$ cannot be neglected.}
In this chapter we give the corrections to the LLT needed to take $S_n(\mathsf c)$  into account.

Before stating our results in general, we discuss two simple examples which demonstrate some of the possible effects of $S_N(\mathsf c)$.

\begin{example}
{\bf (Simple random walk with continuous first step and drift):}\label{Example-SRW-U-First-Step-Chapter-5}
\end{example}

Suppose $\{X_n\}_{n\geq 1}$ are independent real-valued random variables, where $X_1$ is distributed like a random variable $\mathfrak F$, and  $X_i$ $(i\geq 2)$ are equal to $0,1$ with equal probabilities.

$\mathfrak F$ could be arbitrary, but we assume for simplicity that  $0\leq \mathfrak F<1$ a.s., $\E[\mathfrak F]=\frac{1}{2}$, the distribution of $\mathfrak F$ has a density, and $\mathfrak F$
is {\em not}  uniformly distributed on $[0,1]$. Let $\mu_{\mathfrak F}$ denote the probability measure associated with the distribution of $\mathfrak F$.

$S_N:=X_1+\cdots +X_N$ is exactly $S_N(\mathsf f)$, where $f_n(x,y):=x$. Since the distribution of $\mathfrak F$ has a density, $\mathsf f$ has algebraic range $\R$.

The following decomposition shows that $\mathsf f$ is reducible, with essential range $\Z$:
 Let $\delta_{ij}$ be Kronecker's delta, then $\mathsf f=\mathsf g+\mathsf c$ where
$$
g_n(x,y):=(1-\delta_{1,n})x, \ c_n(x,y):=\delta_{1,n} x,
$$
 $\mathsf g$ is irreducible with essential range $\Z$, and $\mathsf c$ is center tight.

We have
$
S_N=(\underset{S_N(\mathsf g)}{\underbrace{X_2+\cdots + X_N}})+\underset{S_N(\mathsf c)}{\underbrace{X_1}}.
$
Clearly, $S_N(\mathsf g)$, $S_N(\mathsf c)$ are independent; $S_N(\mathsf c)\sim\mathfrak F$; and  $S_N(\mathsf g)$ has the binomial distribution $B(\frac{1}{2},N-1)$.
So $S_N$ has distribution
$
\mu_{\mathfrak F}\ast B(\frac{1}{2},N-1).
$
This distribution has a density, which we denote by  $p_N(x)dx$. The following holds as  $N\to\infty$:

\medskip
\noindent
{\bf (A) Non-uniform scaling limit for $p_N(x)dx$:} \; $m_N:=p_N(x)dx$ is a positive functional on $C_c(\R)=\{\text{continuous functions with compact support}\}$. Fix $z_N:=\E(S_N)=N/2$ and let    $V_N:=\Var(S_N)\sim N/4$. Then for every $\phi\in C_c(\R)$ and $N$ even,
\begin{align*}
&\int \phi(x-z_N)p_N(x)dx=\E[\phi(S_N-z_N)]=\E[\phi(S_N(\mathsf g)+S_N(\mathsf c)-z_N)]\\
&=\sum_{m\in\Z} \E[\phi(\mathfrak F+m-z_N)]\Prob[S_n(\mathsf g)=m]=
\sum_{m=0}^{N-1} {N-1\choose m}\frac{1}{2^{N-1}}\E[\phi(\mathfrak F+m-z_N)]\\
&=\frac{1}{2^{N-1}}\sum_{m=0}^{N-1} {N-1\choose m}\psi(m-\tfrac{N}{2}),\text{ where }\psi(m):=\E[\phi(\mathfrak F+m)]\\
&=\frac{1}{2^{N-1}}\sum_{m=-N/2}^{N/2-1}{N-1\choose m+N/2}\psi(m)\sim \frac{1}{\sqrt{2\pi V_N}}\sum_{m\in\Z}\psi(m)\text{ by Stirling's formula}\\
&\sim \frac{1}{\sqrt{2\pi V_N}}\sum_{m\in\Z}\E[\phi(\mathfrak F+m)],\text{ as $N\to\infty$. This also holds for $N$ odd.}
\end{align*}
Thus the distribution of $S_N-z_N$ tends to zero in the vague topology of Radon measure on $\R$ ``at a rate of $1/\sqrt{2\pi N}$," and if we inflate it by $\sqrt{2\pi V_N}$ then it converges in the vague topology to $\mu_{\mathfrak F}\ast$(counting measure on $\Z$).

By the assumptions on $\mathfrak F$, the scaling limit
$\mu_{\mathfrak F}\ast$(counting measure on $\Z$) is not a Haar measure on a  closed subgroup of $\R$. This is different from the irreducible case, when  the scaling limit is the  Haar measure on $G_{ess}(\mathsf X,\mathsf f)$.

\medskip
\noindent
 {\bf (B) Non-standard limit for $\sqrt{2\pi V_N}\Prob[S_N-z_N\in (a,b)]$:}
 Fix  $a,b\in\R\setminus\Z$ s.t. $|a-b|>1$.
 Repeating the previous calculation with $\phi_i\in C_c(\R)$
such that $\phi_1\leq  1_{(a,b)}\leq \phi_2$ and
$\DS \sum_{m\in\Z}\E[\phi_i(m+\mathfrak F)]\approx \sum_{m\in\Z}\E[1_{(a,b)}(m+\mathfrak F)]$
gives
 for $z_N=\E(S_N)$ that
    \begin{equation}\label{P1}
    \sqrt{2\pi V_N}\Prob[S_N-z_N\in (a,b)]\xrightarrow[N\to\infty]{}\sum_{m\in\Z}\E[1_{(a,b)}(m+\mathfrak F)].
    \end{equation}
This is different  than the limit in the { irreducible} non-lattice LLT (Theorem \ref{ThLLT-classic}):
 \begin{equation}\label{P3}
   {\sqrt{2\pi V_N}} \Prob[S_N-z_N\in (a,b)]\xrightarrow[N\to\infty]{} |a-b|;
    \end{equation}
or the limit in the { irreducible} lattice LLT with range $\Z$ (Theorem \ref{Thm-LLT-Lattice}):
    \begin{equation}\label{P2}
    {\sqrt{2\pi V_N}}\Prob[S_N-z_N\in (a,b)]\xrightarrow[N\to\infty]{} \sum_{m\in\Z} 1_{(a,b)}(m).
    \end{equation}

\medskip
\noindent
 {\bf (C) Robustness for large intervals:} Although different, the limits in \eqref{P1},\eqref{P2} and \eqref{P3} are nearly the same as $|a-b|\to\infty$.

The ratio between the limits in \eqref{P2},\eqref{P3} tends to one as $|a-b|\to\infty$.  The ratio between the limits in \eqref{P1},\eqref{P2} tends to one too, because $\supp(\mathfrak F)\subset[0,1]$, so
$
\DS |a-b|-2\leq \sum_{m\in\Z}1_{(a,b)}(m+\mathfrak F)\leq |a-b|+2
$ a.s., whence
$$
\left|\frac{\sum\limits_{m\in\Z}\E[1_{(a,b)}(m+\mathfrak F)]}{\sum\limits_{m\in\Z}1_{(a,b)}(m)}-1\right|\leq \frac{2}{|a-b|}\xrightarrow[|a-b|\to\infty]{}0.
$$

\medskip
Example \ref{Example-SRW-U-First-Step-Chapter-5} is very special in that $S_n(\mathsf g),S_N(\mathsf c)$ are independent. Nevertheless, we will see below that (A), (B), (C) are general phenomena, which also happen when $S_N(\mathsf g)$, $S_N(\mathsf h)$ are strongly correlated.
The following simple example demonstrates another   pathology that is quite general:

\begin{example}[Gradient perturbation of the lazy random walk]:\label{Example-Gradient-Lazy-RW}
\end{example}

 Suppose $X_n,Y_n$  are independent random variables such that $X_n=-1,0,+1$ with equal probabilities, and $Y_n$ are uniformly distributed in $[0,1]$.
Let $\mathsf X=\{(X_n,Y_n)\}_{n\geq 1}$.
\begin{enumerate}[$\circ$]
\item The additive functional
$
g_n((x_n,y_n);(x_{n+1},y_{n+1}))=x_n
$ generates the lazy random walk on $\Z$,
$
S_N(\mathsf g)=X_1+\cdots +X_N.
$
It is irreducible, and satisfies the lattice LLT with range $\Z$.
\item The additive functional $c_n((x_n,y_n),(x_{n+1}, y_{n+1}))=y_n-y_{n+1}$  is center-tight, and
$
S_N(\mathsf c)=Y_{N+1}-Y_1.
$
\item The sum $\mathsf f=\mathsf g+\mathsf c$  is reducible, with algebraic range $\R$ (because of $\mathsf c$) and essential range $\Z$ (because of $\mathsf g$). It generates the process
$$
S_N(\mathsf f)=S_N(\mathsf g)+Y_{N+1}-Y_1.
$$
\end{enumerate}

$S_N(\mathsf f)$ lies in a {\em random coset}  $b_N+\Z$, where $b_N=Y_{N+1}-Y_1$. Since the distribution of $b_N$ is continuous, $\Prob[S_N-z_N=k]=0$ for all $z_N,k\in\Z$, and the standard lattice LLT fails. To deal with this, we must ``shift" $S_N-z_N$ back to $\Z$. This leads to the following (correct) statement: For all $z_N\in\Z$ s.t. $\frac{z_N}{\sqrt{V_N}}\to z$, for all $k\in\Z$,
$$
\Prob[S_N-z_N-b_N=k]\sim \frac{e^{-z^2/2}}{\sqrt{2\pi V_N}}
$$
Notice the shift by  a {\em random} bounded quantity $b_N$.

\subsection{The LLT in the reducible case}

\begin{theorem}\label{Theorem-Reducible-LLT}
Let $\mathsf X=\{X_n\}$ be a uniformly elliptic Markov chain, and let $\mathsf f$ be a reducible  a.s. uniformly bounded additive functional with essential range $\delta(\mathsf f)\Z$, where $\delta(\mathsf f)\neq 0$. Then there are random variables $b_N=b_N(X_1,X_{N+1})$ and  $\mathfrak F=\mathfrak F(X_1,X_2,\ldots)$ with the following properties:
\begin{enumerate}[(1)]
\item For every $z_N\in \delta(\mathsf f)\Z$ such that $\frac{z_N-\E(S_N)}{\sqrt{V_N}}\to z$, for every $\phi\in C_c(\R)$  and $x\in \fS_1$,
$$
\lim\limits_{N\to\infty}\sqrt{V_N} \E_x\left[\phi(S_N-z_N-b_N)\right]=\frac{\delta(\mathsf f)e^{-z^2/2}}{\sqrt{2\pi}}\sum_{m\in\Z} E_x[\phi(m\delta(\mathsf f)+\mathfrak F)].
$$
\item For every $\mathfrak A_{N+1}\subset\fS_{N+1}$ measurable such that $\Prob[X_{N+1}\in\mathfrak A_{N+1}]$ is bounded below, and for every $x\in\fS_1$,
$$
\lim\limits_{N\to\infty}\sqrt{V_N}\E_x\left[\phi(S_N-z_N-b_N)\big|X_{N+1}\in\mathfrak A_{N+1}\right]=\frac{\delta(\mathsf f)e^{-z^2/2}}{\sqrt{2\pi}}\sum_{m\in\Z}
\E_x[\phi(m\delta(\mathsf f)+\mathfrak F)].
$$
\item  $\|b_N\|_\infty\leq 9\delta(\mathsf f)$, and $\mathfrak F\in [0,\delta(\mathsf f))$.
\end{enumerate}
\end{theorem}

The statement may seem at first sight different from the previous LLT we discussed, so we'd like to spend some time on clarifying what it is saying.
\begin{enumerate}[$\circ$]
\item $
\E_x\left[\phi(S_N-z_N-b_N)\right],
$
when viewed as a positive functional on $C_c(\R)$, represents the measure on $\R$,
$
m_{x,N}(E)=\Prob_x[S_N-z_N-b_N(X_1,X_{N+1})\in E].
$
This is the distribution of $S_N$, conditioned on $X_1=x$, after a shift by   $z_N+b_N(X_1,X_{N+1})$.  The deterministic shift by $z_N$ cancels the drift of $S_N$ (notice that $z_N\approx \E(S_N)\approx \E_x(S_N)$).
The random shift $b_N$ is needed to force $S_N$ to stay inside $\delta(\mathsf f)\Z$, see Example \ref{Example-Gradient-Lazy-RW}.

\smallskip
\item The linear functional
\begin{equation}
\label{LocLimRed}
\cA_x(\phi):=\delta(\mathsf f)\sum_{m\in\Z}\E_x[\phi(m\delta(\mathsf f)+\mathfrak F)]
\end{equation}
defines  the element of $C_c(\R)^\ast$ which represents the measure
    $\mu_{x,\mathfrak F}\ast m_{\delta(\mathsf f)}$, where $\mu_{x,\mathfrak F}(E)=\Prob_x(\mathfrak F\in E)$ and
    $m_{\delta(\mathsf f)}:=\delta(\mathsf f)\times \text{counting measure on $\delta(\mathsf f)\Z.$}$
    So part (1) of Theorem \ref{Theorem-Reducible-LLT} says that $m_{x,N}\to 0$ in $C_c(\R)^\ast$  at rate $1/\sqrt{V_N}$, and gives the scaling limit
$
\sqrt{2\pi V_N}m_N\xrightarrow[N\to\infty]{w^\ast}\mu_{x,\mathfrak F}\ast m_{\delta(\mathsf f)}
$ when $z=0$.
See Example \ref{Example-SRW-U-First-Step-Chapter-5}.

\smallskip
\item As in Example \ref{Example-SRW-U-First-Step-Chapter-5},
part (1) implies the following:  For all $a<b$ s.t. $\mathfrak F$ has no atoms in
$\{a,b\}+\delta(\mathsf f)\Z$, and for all
$z_N\in\delta(\mathsf f)\Z$ s.t.  $\frac{z_N-\E(S_N)}{\sqrt{V_N}}\to z$,
$$
\Prob_x[S_N-z_N-b_N\in (a,b)]=[1+o(1)]\frac{e^{-z^2/2}}{\sqrt{2\pi V_N}}\cdot \cA_x(1_{(a,b)}),\text{ and }
$$
$$
\cA_x(1_{(a,b)})\sim \begin{cases}
|a-b| & \text{ as }|a-b|\to\infty\\
\Prob_x[\mathfrak F\in (a,b)] & \text{ for }(a,b)\subset [0,\delta(\mathsf f)].
\end{cases}
$$
Viewed from this perspective, $A_x(1_{(a,b)})$ is a ``correction" to the term $|a-b|$ in classical LLT \eqref{LLT-non-arithmetic-limit}, which is needed for intervals with length of order $\delta(\mathsf f)$.
\end{enumerate}

These observations should be sufficient to understand the content of part (1). Part (2) is a ``mixing" version of part (1), in the sense of \S\ref{Section-Mixing-LLT}. Such results are particularly useful in the reducible setup for the following reason. The random shift $b_N(X_1,X_{N+1})$ is sometimes a nuisance, and it is tempting to  turn it into  a deterministic quantity by conditioning on $X_1,X_{N+1}$. We  would have liked to say that part (1) survives such conditioning, but we cannot. The best we can say in general is that part (1)  remains valid
under conditioning of the form $X_1=x_1, X_{N+1}\in\fA_{N+1}$ {\em provided $\Prob(X_{N+1}\in\fA_{N+1})$ is bounded below.} This the content of part (2). For an example how to use such a statement,
see \S \ref{Section-Irreducibility-As-Necessary-Condition}.

\medskip
In the following sections, we explore some of the consequences of Theorem \ref{Theorem-Reducible-LLT}.

\subsection{Irreducibility as a necessary condition for the mixing LLT }\label{Section-Limit-Modulo-t}

Theorem \ref{Theorem-Reducible-LLT} exposes the pathologies that could happen in the reducible case. But is irreducibility a necessary condition for the non-lattice LLT? No!

\begin{example}\label{Example-Accident}
Take example \ref{Example-SRW-U-First-Step-Chapter-5} with fixed $x$ and  $\mathfrak F$ uniformly  distributed on $[0,1]$, given $X_1=x$.
In this case, $\delta(\mathsf f)=1$,
$\mu_{x,\mathfrak F}\ast m_{\delta(\mathsf f)}=$Lebesgue's measure, $A_x(1_{(a,b)})=|a-b|$, and
$
\DS\frac{z_N-\E(S_N)}{\sqrt{V_N}}\to z\Rightarrow \Prob_x[S_N-z_N\in (a,b)]\sim \frac{e^{-z^2/2}}{\sqrt{2\pi V_N}}|a-b|,
$
even though $\mathsf f$ is reducible, with essential range $\Z$. Of course, such behavior is immediately destroyed if we modify $X_1$.
\end{example}

In this section we show that irreducibility {\em is } a necessary condition for the mixing LLT, provided we impose the mixing LLT not just for $(\mathsf X,\mathsf f)$, but also for all $(\mathsf X',\mathsf f')$ obtained from $(\mathsf X,\mathsf f)$ by changing finitely many terms.

Let $\mathsf f$ be an additive functional on a Markov chain $\mathsf X$. Denote the state spaces of $\mathsf X$ by  $\fS_n$, and write $\mathsf X= \{X_n\}_{n\geq 1}$, $\mathsf f=\{f_n\}_{n\geq 1}$. A sequence of events $\mathfrak A_k\subset\fS_k$ is called {\bf regular} if $\mathfrak A_k$ are measurable, and $\Prob(X_n\in\mathfrak A_n)$ is bounded away from zero.
\begin{enumerate}[$\circ$]
\item We say that $(\mathsf X,\mathsf f)$ satisfies the {\bf mixing non-lattice local limit theorem}
\index{Local limit theorem!mixing non-lattice} if
$V_N:=\Var(S_N)\to\infty$, and for every regular sequence of events $\mathfrak A_n\in\fS_n$, $x\in\fS_1$,  for all $z_N\in\R$ such that $\frac{z_N-\E(S_N)}{\sqrt{V_N}}\to z$,  and for each non-empty interval $(a,b)$,
$$
\Prob_x\biggl(S_N-z_N\in (a,b)\big|X_{N+1}\in\mathfrak A_{N+1}\biggr)=[1+o(1)]\frac{e^{-z^2/2}}{\sqrt{2\pi V_N}}|a-b|\text{ as $N\to\infty$}.
$$
\item Fix $t>0$. We say that $(\mathsf X,\mathsf f)$ satisfies the {\bf mixing uniform distribution mod $t$ property}, \index{mixing uniform distribution mod $t$}
if for every regular sequence of events $\mathfrak A_n\subset\fS_n$,  $x\in\fS_1$,
and a non-empty interval $(a,b)$ with length less than one,
    $$
    \Prob_x\bigl(S_N\in (a,b)+t\Z|X_{N+1}\in\mathfrak A_{N+1}\bigr)\xrightarrow[N\to\infty]{} \frac{|a-b|}{t}.
    $$
\end{enumerate}

\begin{theorem}\label{Theorem-Characterization-Irreducibility}
Let $\mathsf f$ be an a.s. uniformly bounded additive functional on a uniformly elliptic Markov chain. Given $m$, let
$(\mathsf X_m, \mathsf f_m):=(\{X_n\}_{n\geq m}, \{f_n\}_{n\geq m})$.
 The following are equivalent:
\begin{enumerate}[(1)]
\item $\mathsf f$ is irreducible with algebraic range $\R$;
\item $(\mathsf X_m,\mathsf f_m)$ satisfy the mixing non-lattice local limit theorem for all $m$;
\item $(\mathsf X_m,\mathsf f_m)$ satisfy the mixing uniform distribution mod $t$ for all $m$ and $t>0$.
\end{enumerate}
\end{theorem}

\subsection{Universal bounds for $\Prob_x[S_N-z_N\in (a,b)]$}

So far we have considered the problem of finding $\Prob_x[S_N-z_N\in (a,b)]$ up to asymptotic equivalence. We now consider the  problem of finding $\Prob_x[S_N-z_N\in (a,b)]$ up to bounded multiplicative error, assuming only that $V_N\to\infty$.

We already saw that the predictions of the LLT for {\em large} intervals $(a,b)$ are nearly the same both in the reducible and irreducible, lattice and non-lattice cases. Therefore we expect  {\em universal} lower and upper bounds, for all  sufficiently large intervals without further assumptions on irreducibility or on the arithmetic structure of the range. The question is how large is ``sufficiently large."

We certainly cannot expect universal lower and upper bounds for intervals smaller than the  {\bf graininess constant}\index{graininess constant} of $(\mathsf X,\mathsf f)$:
\begin{equation}
\label{DefGrain}
\delta(\mathsf f):=\begin{cases}
t & G_{ess}(\mathsf X,\mathsf f)=t\Z,\;\; t>0\\
0 & G_{ess}(\mathsf X,\mathsf f)=\R\\
\infty & G_{ess}(\mathsf X,\mathsf f)=\{0\},
\end{cases}
\end{equation}
because intervals with length less than $\delta(\mathsf f)$ may fall in the gaps of the support of $S_N-z_N$.
Theorem \ref{Theorem-Reducible-LLT} can be used to see that universal bounds do apply as soon as  $|a-b|>\delta(\mathsf f)$:

\begin{theorem}\label{Theorem-Reducible-Universal-Bounds}
Suppose $\mathsf f$ is an a.s. uniformly bounded additive functional on a uniformly elliptic Markov chain $\mathsf X$. Then for every interval $(a,b)$ of length $L>\delta(\mathsf f)$,  for all $\epsilon>0$,
$x\in\fS_1$ and $z_N\in\R$ such that $\frac{z_N-\E(S_N)}{\sqrt{V_N}}\to z$, for all  for all $N$ large enough,
\begin{align}
\label{TRUB1}
\Prob_x(S_N-z_N \in (a,b))&\leq \frac{e^{-z^2/2}|a-b|}{\sqrt{2\pi V_N}}\left(1+\frac{21\delta(\mathsf f)}{L}+\epsilon\right),\\
\label{TRUB2}
\Prob_x(S_N-z_N\in (a,b))&\geq \frac{e^{-z^2/2}|a-b|}{\sqrt{2\pi V_N}}\left(1-\frac{\delta(\mathsf f)}{L}-\epsilon\right).
\end{align}
In addition, if $0<\delta(\mathsf f)<\infty$ and $k\delta(f) \lvertneqq L\lvertneqq(k+1)\delta(f) $, $k\in\N$,  then
$$
 \left(\frac{e^{-z^2/2}}{\sqrt{2\pi V_N}}\right)k\delta(f) \lesssim
\Prob_x(S_N-z_N\in (a,b))
\lesssim
\left(\frac{e^{-z^2/2}}{\sqrt{2\pi V_N}}\right) (k+1)\delta(f).
$$
Here $A_N \lesssim B_N$ means that $\DS \limsup_{N\to\infty} (A_N/B_N)\leq 1$.
\end{theorem}

We note that both upper and lower bound become asymptotic to the Gaussian density as $L\to\infty.$
Notice also that the theorem makes no assumptions on the irreducibility of $\mathsf f$.

\smallskip
Theorem \ref{Theorem-Reducible-Universal-Bounds} is an easy corollary of Theorem \ref{Theorem-Reducible-LLT}, see \S \ref{SSLLTMult},
but this is an overkill. At the end of the chapter we will  supply a proof of universal bounds for intervals of length $L>2\delta(\mathsf f)$,
which does not require the full force of Theorem \ref{Theorem-Reducible-LLT}, and which also applies
to arbitrary initial distributions and to arrays.

\section{Proofs}
\subsection{Characteristic functions}
\noindent
{\bf Setup:} Throughout this section we assume that $\mathsf X=\{X_n\}$ is a uniformly elliptic Markov {\em chain} with state spaces $\fS_n$, marginals $\mu_n(E)=\Prob(X_n\in E)$, and transition probabilities $\pi_{n,n+1}(x,dy)=p_n(x,y)\mu_{n+1}(dy)$ which satisfy the uniform ellipticity condition with ellipticity constant $\epsilon_0$.

For every bounded measurable function $\vf:\fS_n\times\fS_{n+1}\to\R$, we let
$$
\E(\vf):=\E[\vf(X_n,X_{n+1})]\ , \ \sigma(\vf):=\sqrt{\Var(\vf(X_n,X_{n+1}))}.
$$

Next we assume that $K>0$, $\epsilon\in (0,1)$ and  $\mathsf f=\{f^{(N)}_n:1\leq n\leq N<\infty\}$ is an {\em array} of measurable functions $f_n^{(N)}:\fS_n\times\fS_{n+1}\to\R$  which satisfy the following  assumptions for all $N$:
\begin{enumerate}[(I)]
\item $\E(f^{(N)}_n)=0$ and $\ess\sup|\mathsf f|<K$.

\item Let $\displaystyle S_N:=\sum_{n=1}^N f_n^{(N)}(X_n,X_{n+1})$ and $V_N:=\Var(S_N)$, then there exists $\hC>0$ s.t.
\begin{equation}
\label{U-V-Uniform}
V_N\to\infty\ \ \text{ and }\ \  \frac{1}{V_N}\sum_{n=1}^N \sigma^2(f^{(N)}_n)\leq \hC.
\end{equation}
\item
$
\mathsf f=\mathbb F+\mathsf h+\mathsf c$, where
\begin{enumerate}[(a)]
\item $\mathbb F=\{\bbf^{(N)}_n\}$ are measurable functions such that
$
\ess\sup|\bbf|\leq K\ , \ G_{alg}(\mathsf X, \bbf)\subset\Z.
$
\item $\mathsf h=\{h^{(N)}_n\}$ are measurable functions such that
$$
\E(h^{(N)}_n)=0,\quad  \ess\sup|\mathsf h|<K, \quad \sum_{n=1}^N \sigma^2(h^{(N)}_n)\leq \epsilon.
$$
\item  $\mathsf c=\{c^{(N)}_n\}$ are constants.
    Necessarily $|c^{(N)}_n|\leq 3K$ and  $c^{(N)}_n=-\E(\bbf^{(N)}_n)$. Let $c^{(N)}:=\sum_{n=1}^N c^{(N)}_n$.
    \end{enumerate}
\end{enumerate}
We are {\em not} assuming that $\E(\bbf^{(N)}_n)=0$: $\bbf^{(N)}_n$ are integer valued, and we do not wish to destroy this by subtracting the mean.

\begin{lemma}
\label{LmCharAlm-AP}
Under the above assumptions,\;  for every  $\brK>0$, $m\in\Z$, there are  $\brC, \brN>0$ s.t. for every $N>\brN$,  $|s|\leq \brK$, $x\in\fS_1$, and  $v_{N+1}:\mathfrak S_{N+1}\to\R$ with $\|v_{N+1}\|_\infty\leq 1$,
\begin{align*}
&\EXP_x\left(e^{i\bigl(2\pi m+\frac{s}{\sqrt{V_N}}\bigr) S_N} v_{N+1}(X_{N+1}) \right)
=
e^{2\pi i m  c^{(N)}} e^{-s^2/2}\;\EXP(v_{N+1}(X_{N+1}))+\eta_N(x)
\end{align*}
where   $\E(|\eta|)\leq \brC\left[\sum_{n=1}^N \sigma^2(h^{(N)}_n)\right]^{1/2}\leq \brC\sqrt{\epsilon}$.
\end{lemma}

\begin{proof}
In this proof we fix the value of $N$,  and drop the superscripts $N$ for the ease of notation (for example $c^{(N)}=c$).

We develop a perturbation theory of transfer operators similar to \cite{Ba}.
Recall the operators $\mathcal L_{n,\xi}: L^\infty(\mathfrak S_{n+1})\to L^\infty(\mathfrak S_{n})$
given by
$$
(\mathcal L_{n,\xi}u)(x)=\int_{\fS_{n+1}}p_n(x,y) e^{i\xi f_n(x,y)} u(y) \mu_{n+1}(dy).
$$
Let  $\xi=\xi(m,s):=2\pi m+\dfrac{s}{\sqrt{V_N}}$.
Since $\bbf_n$ is integer valued,
$$ e^{i\xi f_n}=\exp[{2\pi i m \bbf_n+\frac{is}{\sqrt{V_N}}\bbf_n+i\xi c_n+i\xi h_n)}] =
{ e^{2\pi i m c_n }} e^{i\left(\frac{s}{\sqrt{V_N}} (\bbf_n{+c_n})+\xi h_n\right)}. $$
We now split
$ e^{-{2\pi i m c_n}} \cL_{n, \xi}=\brcL_{n,\xi}+\hcL_{n, \xi}+\tcL_{n, \xi}$
where
\begin{align*}
&\left(\brcL_{n,\xi} u\right)(x)
=\int_{\fS_{n+1}}p_n(x,y)  e^{\frac{is}{\sqrt{V_N}} (\bbf_n(x,y){+c_n})} u(y)\mu_{n+1}(dy),\\
&\left(\hcL_{n,\xi} u\right)(x)
=i\xi \int_{\fS_{n+1}}p_n(x,y)
 h_n(x,y) u(y)\mu_{n+1}(dy),\text { and }
 \end{align*}
\begin{align*}
&\left(\tcL_{n,\xi} u\right)(x)
=\!\!\!\int_{\fS_{n+1}}p_n(x,y) \left[ e^{i\xi h_n+\frac{i s}{\sqrt{V_N}} (\bbf_n(x,y){+c_n})}\!-e^{\frac{i s}{\sqrt{V_N}} (\bbf_n(x,y){+c_n})}
\!-i\xi h_n(x,y)\right]
 u(y) \mu_{n+1}(dy).
\end{align*}

We claim that there exists $C_1(\brK, m)>1$ such that  for $|s|\leq \brK,$ $n\geq 1$

\begin{align}
&\left\Vert \cL_{n, \xi}\right\Vert:=\left\Vert \cL_{n, \xi}\right\Vert_{L^\infty\to L^\infty} \leq 1,\\
&\left\Vert \cL_{n, \xi}\right\Vert_{L^1\to L^\infty} \leq C_1(\brK,m), \label{Gould}\\
& \left\Vert \brcL_{n,\xi}\right\Vert :=\left\Vert \brcL_{n,\xi}\right\Vert_{L^\infty\to L^\infty}\leq 1,\label{BLBound}\\
&\left\Vert \hcL_{n,\xi}\right\Vert_{ L^\infty\to L^1} \leq C_1(\brK,m) \sigma(h_n), \label{HLBoundInf1}\\
& \left\Vert \tcL_{n,\xi}\right\Vert_{ L^\infty\to L^1}\leq C_1(\brK,m) \left[\sigma^2(h_n)+ \frac{\sigma(h_n)\sigma(f_n)}{\sqrt{V_N}}\right] \label{TLBoundInf1}.
\end{align}

To see this, we represent these operators as integral operators, and estimate their kernels.
For example,  $\hcL_{n,\xi}$ is an integral operator whose kernel has absolute value
$|i\xi p_n(x,y)h_n(x,y)|\leq \epsilon_0^{-1}|\xi||h_n(x,y)|$. So
$$
\|\hcL_{n,\xi}\|_{L^\infty\to L^1}\leq \epsilon_0^{-1}\sqrt{4\pi^2 m^2+\brK^2}\|h_n\|_{L^1}\leq
\epsilon_0^{-1}\sqrt{4\pi^2 m^2+\brK^2}\|h_n\|_{L^2},
$$
and \eqref{HLBoundInf1} follows from the identity $\|h_n\|_{L^2}\equiv \sigma(h_n)$.
Similarly,  $\tcL_{n,\xi}$ has kernel with absolute value
$$
\quad p_n(x,y) \bigl|e^{i\xi h_n+i s \frac{\bbf_n(x,y)+c_n}{\sqrt{V_N}}}-e^{i s \frac{\bbf_n(x,y)+c_n}{\sqrt{V_N}}}
-i\xi h_n(x,y)\bigr|
\leq \epsilon_0^{-1}\bigl|e^{is \frac{\bbf_n(x,y)+c_n}{\sqrt{V_N}}} \bigl(e^{i\xi h_n}-1\bigr)-i\xi h_n\bigr|$$
$$=\epsilon_0^{-1}\bigl|e^{is \frac{\bbf_n(x,y)+c_n}{\sqrt{V_N}}} \bigl(i\xi h_n+O(\xi^2 h_n^2)\bigr)-i\xi h_n\bigr|
=\epsilon_0^{-1}
\bigl|e^{is \frac{\bbf_n(x,y)+c_n}{\sqrt{V_N}}}-1\bigr||\xi h_n|+O\left(h_n^2\right) $$
$$=O\left( \frac{1}{\sqrt{V_N}}\left| h_n (\bbf_n+c_n)\right|\right)+O\left(h_n^2\right)
$$
where the implicit constants in $O(\cdot)$ are uniform on compact sets of $\xi.$
It follows that uniformly on compact sets of $\xi$,
\begin{align*}
&\|\tcL_{n,\xi}\|_{L^\infty\to L^1}=O(V_N^{-1/2})\E(|h_n(\bbf_n+c_n)|)+O(\|h_n\|_2^2)\\
&=O(V_N^{-1/2}) \|h_n\|_2\|\bbf_n+c_n\|_2+O(\|h_n\|_2^2)\\
&= O(V_N^{-1/2}) \|h_n\|_2(\|f_n-h_n\|_2)+O(\|h_n\|_2^2)\\
&= O(V_N^{-1/2}) \|h_n\|_2(\|f_n\|_2+\|h_n\|_2)+O(\|h_n\|_2^2)\\
&= O\left(\frac{\|h_n\|_2\|f_n\|_2}{\sqrt{V_N}}+\|h_n\|_2^2\right)= O\left(\frac{\sigma(h_n)\sigma(f_n)}{\sqrt{V_N}}+\sigma^2(h_n)\right),
\end{align*}
as claimed in \eqref{TLBoundInf1}.

\medskip
Recall Nagaev's identity \eqref{Nagaev2}:
\index{Nagaev's identities}
 $\E_x[e^{i\xi S_N}v_{N+1}(X_{N+1})]=(\mathcal L_{1,\xi}\mathcal L_{2,\xi}\cdots\mathcal L_{N,\xi} v_{N+1})(x)$.
 The decomposition $ e^{- 2\pi i m c_n} \cL_{n, \xi}=\brcL_{n,\xi}+\hcL_{n, \xi}+\tcL_{n, \xi}$ implies that
 \begin{equation}\label{Phi-decomposition}
\EXP_x\left(e^{i\xi S_N} v_{N+1}(X_{N+1}) \right)
=e^{2\pi i m c} \left(\brPhi_N(x,\xi)+\hPhi_N(x,\xi)+\tPhi_N(x,\xi)\right)
\end{equation}
where $c=c^{(N)}=c_1+\dots+c_N$, and
\begin{align*}
\brPhi_N(x,\xi)&:=\left(\brcL_{1,\xi} \dots \brcL_{N,\xi} v_{N+1} \right)(x),\\
\tPhi_N(x,\xi)&:=\sum_{k=1}^{N-1} e^{-2\pi i m (c_1+\cdots+c_{k-1})} \left(\cL_{1,\xi} \cdots \cL_{k-1, \xi} \tcL_{k, \xi} \brcL_{k+1, \xi} \dots \brcL_{N,\xi} v_{N+1} \right)(x),\\
\hPhi_N(x,\xi)& :=\sum_{k=1}^{N-1} e^{-2\pi i m (c_1+\cdots+c_{k-1})}
\left(\cL_{1,\xi} \dots \cL_{k-1, \xi} \hcL_{k, \xi} \brcL_{k+1, \xi} \dots \brcL_{N,\xi}v_{N+1} \right)(x).
\end{align*}
We will analyze each of these summands.

\medskip
\noindent
{\sc Claim 1:}{\em   For every $m\in\Z$,  $\left|\brPhi_N(x, \xi)-e^{-s^2/2}\E_x(v_{N+1}(X_{N+1}))\right|\xrightarrow[N\to\infty]{}0$ uniformly in $s$ on $\{s\in\R:|s|\leq \brK\}$,  $x\in\fS_1$,
$v_{N+1}\in \{v\in L^\infty(\mathfrak S_{N+1}):\|v\|\leq 1\}$.}

\medskip
\noindent
{\sc Proof:}
$\brPhi_N(x, \xi)=\EXP_x\left(\exp\left(is\frac{\sum_{k=1}^N \bbf_k{+c}}{\sqrt{V_N}} \right)v_{N+1}(X_{N+1}) \right)$, where $\E(\sum_{k=1}^N\bbf_n)=-c$.
Fix $1\leq r\leq N$. Using the decomposition $\mathsf f=\bbf+\mathsf h+\mathsf c$, we find that
$$ \frac{1}{\sqrt{V_N}}\biggl(\sum_{k=1}^N \bbf_k+c\biggr)=\frac{1}{\sqrt{V_N}}\sum_{k=1}^{N-r} f_k+
\frac{1}{\sqrt{V_N}}\biggl(O(r)-\sum_{k=1}^N h_k\biggr). $$
By assumption III(b), the $L^2$ norm  of the second summand is $O(1/\sqrt{V_N})$. Therefore the
second term converges to $0$ in probability as $N\to\infty$, and
\begin{equation}
\label{CLTrBack}
 \brPhi_N(x, \xi)=\EXP_x\left(e^{\frac{is}{\sqrt{V_N}}S_{N-r}} v_{N+1}(X_{N+1}) \right)
+o(1),
\end{equation}
where we have abused notation and wrote $S_{N-r}=f_1^{(N)}+\cdots+f_{N-r}^{(N)}$.

The  rate of convergence to $0$ depends on $r$ and $m$, but is uniform when $|s|\leq \brK$ and $\|v_{N+1}\|_\infty\leq 1$.
At the same time, by exponential mixing (see \eqref{Exp-Mixing-L-infinity}), there is $0<\theta<1$ such that
\begin{align}
&\EXP_x\left(e^{\frac{is}{\sqrt{V_N}}S_{N-r}} v_{N+1}(X_{N+1}) \right)= \EXP_x\left[e^{\frac{is}{\sqrt{V_N}}S_{N-r}} \EXP_x \bigl(v_{N+1}(X_{N+1}) \big|
X_1,\ldots,X_{N-r}\bigr)\right] \notag\\
&= \EXP_x\left[e^{\frac{is}{\sqrt{V_N}}S_{N-r}} \EXP_x \bigl(v_{N+1}(X_{N+1}) \big| X_{N-r}\bigr)\right]  \text{(Markov property)} \notag\\
&= \EXP_x\left(e^{\frac{is}{\sqrt{V_N}}S_{N-r}}[\E_x(v_{N+1}(X_{N+1})
+O(\theta^r)] \right)
\quad \text{(exponential mixing)} \notag
\end{align}
\begin{align}
&=\E_x(e^{is S_{N-r}/\sqrt{V_N}})\E_x(v_{N+1}(X_{N+1}))+O(\theta^r)
\label{CLTMix}
\end{align}
where the $O(\theta^r)$ is uniform in $\|v_{N+1}\|_\infty$.

A similar mixing argument shows that $$\E_x(S_{N-r})=\E(S_{N-r}|X_1=x)=\E(S_N)+O(1)=O(1)$$
uniformly in $x\in\fS_1$.  By Dobrushin's CLT,
$$
\E_x(e^{is S_{N-r}/\sqrt{V_N}})=[1+o(1)]\E_x(e^{is \frac{S_{N}-\E_x(S_N)}{\sqrt{V_N}}})=[1+o(1)]e^{-s^2/2}\text{ as $N\to\infty$}.
$$
The claim follows from this,
\eqref{CLTrBack},  and \eqref{CLTMix}.

\medskip
\noindent
{\sc Claim 2.} {\em There exists $C_2(\brK,m)$ s.t. for all $|s|\leq \brK$ and  $\|v_{N+1}\|_\infty\leq 1$,
 $$\bigl\|\wt\Phi_N(x,\xi)\bigr\|_{L^1}\leq C_2(\brK,m)\sqrt{\eps}.$$}
\noindent
{\sc Proof:} $\|\tPhi_N(x,\xi)\|_1 \leq
\|\tcL_{1, \xi}\|_{L^\infty\to L^1}
\left\Vert \brcL_{2, \xi} \right\Vert \cdots \left\Vert
\brcL_{N,\xi} \right\Vert$
\begin{align*}
&+
\sum_{k=2}^N \left(\left\Vert \cL_{1,\xi} \right\Vert
\cdots \left\Vert \cL_{k-1, \xi}  \tcL_{k, \xi} \right\Vert
\left\Vert \brcL_{k+1, \xi} \right\Vert \cdots \left\Vert
\brcL_{N,\xi} \right\Vert
\right) \end{align*}
Suppose $|s|\leq \brK$, then \eqref{Gould}, \eqref{BLBound}
and  \eqref{TLBoundInf1}  tell us that
\begin{align*}
&\|\tcL_{1, \xi}\|_{L^\infty\to L^1}
\left\Vert \brcL_{2, \xi} \right\Vert \cdots \left\Vert
\brcL_{N,\xi} \right\Vert \leq
C_1(\brK,m)\left[ \sigma(h_1)^2+\frac{\sigma(h_1)\sigma(f_1)}{\sqrt{V_N}} \right] ,\\
&\left\Vert \cL_{k-1, \xi}  \tcL_{k, \xi} \right\Vert\leq
\left\Vert \cL_{k-1, \xi} \right\Vert_{L^1\to L^\infty} \left\Vert \tcL_{k, \xi} \right\Vert_{L^\infty\to L^1}
\leq C_1(\brK,m)^2\left[ \sigma(h_k)^2+\frac{\sigma(h_k)\sigma(f_k)}{\sqrt{V_N}} \right]. \end{align*}
Therefore $\|\tPhi_N(x,\xi)\|_1\leq  C_1(\brK,m)^2 \sum_{k=1}^{N-1} \left[ \sigma(h_k)^2+\frac{\sigma(h_k)\sigma(f_k)}{\sqrt{V_N}}\right]$. By Cauchy-Schwarz,
\begin{align*}
&\sum_{k=1}^{N-1}\sigma(h_k)^2+\frac{\sigma(h_k)\sigma(f_k)}{\sqrt{V_N}}\leq
\sum_{k=1}^{N-1}\sigma(h_k)^2+\sqrt{\sum_{k=1}^{N-1}\sigma^2(h_k)\cdot\frac{1}{V_N}\sum_{k=1}^{N-1}\sigma^2(f_k)}\\
&\leq \epsilon+\sqrt{\hC\epsilon},\text{ by assumptions II and III(b). The claim follows.}
\end{align*}

\noindent
{\sc Claim 3.} {\em There exists $C_3(\brK,m)$ s.t. for all $|s|\leq \brK$, and $\|v_{N+1}\|_\infty\leq 1$,\\
 $\|\hPhi_N(x,\xi)\|_1\leq C_3(\brK,m)\sqrt{\eps}$.}

\medskip
\noindent
{\sc Proof.}
Fix $N$, $v_{N+1}\in L^\infty(\fS_{N+1})$ such that $\|v_{N+1}\|_\infty\leq 1$, and define $\zeta_k\in L^\infty(\mathfrak S_{k})$, $\eta_{k}\in\R$ s.t.
$$ \phi_k(\cdot):=(\brcL_{k, \xi}\dots \brcL_{N,\xi}) v_{N+1}=\zeta_k(\cdot)+\eta_k  $$
where $\eta_k:=\EXP\bigl[(\brcL_{k, \xi}\cdots \brcL_{N,\xi}) v_{N+1}(X_k)\bigr]$, and $\EXP[\zeta_k(X_k)]=0$. Then
\begin{equation}\label{Phi-Zeta-Nu}
\bigl\|\hPhi_N(x,\xi)\bigr\|_1\leq \sum_{k=1}^N \|\cL_{1,\xi}\cdots \cL_{k-1,\xi}\hcL_{k,\xi}
(\zeta_{k+1}+\eta_{k+1} 1)\|_1.
\end{equation}
By \eqref{BLBound},  $|\eta_k|\leq 1$. We will now work towards a control of  $\zeta_k$:

\medskip
\noindent
{\sc Sub-claim.} {\em We can decompose $\zeta_k=\zeta_k'+\zeta_k''$ so that for all $|s|\leq \brK$,
there exist $ \hC_0, \hK_0>0$ and  $0<\htheta_0<1$ s.t. for all $k=1,\ldots,N-2$}
\begin{equation}
\label{Claim4-1}
\|\zeta_k'\|_\infty\leq \htheta^2_0\|\zeta_{k+2}'\|_\infty+\hK_0 \|\zeta_{k+2}''\|_1,
\end{equation}
\begin{equation}
\label{Claim4-2}
\|\zeta_k''\|_\infty\leq \hC_0\left(\frac{\sigma(f_k)+\sigma(f_{k+1}) +\sigma(h_k)+\sigma(h_{k+1})}{\sqrt{V_N}}\right).
\end{equation}

\medskip
\noindent
{\em Proof.\/}
In what follows, $\cL_k=\cL_{k,0}$.
Write $$ \eta_k+\zeta_k=\phi_k=\left(\brcL_{k, \xi} \brcL_{k+1, \xi}\right) \phi_{k+2}=\left(\brcL_{k, \xi} \brcL_{k+1, \xi}\right) \left(\eta_{k+2}+\zeta_{k+2}\right)
$$
$$=\left(\cL_{k} \cL_{k+1}\right)\eta_{k+2}+
\left(\cL_{k} \cL_{k+1}\right)\zeta_{k+2}+
\left(\brcL_{k, \xi} \brcL_{k+1, \xi}-\cL_{k} \cL_{k+1}\right)\phi_{k+2}
.$$
 Observe that $\cL_k 1=1$, so $\left(\cL_{k} \cL_{k+1}\right)\eta_{k+2}=\eta_{k+2}$.
This leads to the decomposition
$$
 \zeta_k=\underset{\zeta_k'}{\underbrace{
\left(\cL_{k} \cL_{k+1}\right)\zeta_{k+2}}}+\underset{\zeta_k''}{\underbrace{
\left(\brcL_{k, \xi} \brcL_{k+1, \xi}-\cL_{k} \cL_{k+1}\right)\phi_{k+2}+\eta_{k+2}-\eta_k}}
$$
We use this decomposition to define  $\zeta_k', \zeta_k''$. This gives the following recursion:
\begin{equation}
\label{ZetaPP}
\begin{aligned}
\zeta_k'&=\left(\cL_{k} \cL_{k+1}\right)\zeta'_{k+2}+
\left(\cL_{k} \cL_{k+1}\right)\zeta_{k+2}'',\\
\zeta_k''&=
\left(\brcL_{k, \xi} \brcL_{k+1, \xi}-\cL_{k} \cL_{k+1}\right)\phi_{k+2}+ \eta_{k+2}-\eta_k.
\end{aligned}
\end{equation}

Notice that $\zeta_k',\zeta_k''$ both have zero means.
Indeed in our setup, $\mu_j(E)=\Prob(X_j\in E)$ and $(\cL_k u)(x)=\E(u(X_{k+1})|X_k=x)$,
 whence
$$
\int \zeta_k' d\mu_k=\E(\zeta_k'(X_k))=\E[\E(\E(\zeta_{k+2}(X_{k+2})|X_{k+1})|X_k)]=\E(\zeta_{k+2}(X_{k+2}))=0,
$$
and $\E(\zeta_k'')=\E(\zeta_k)-\E(\zeta_k')=0-0=0$.

To  prove the estimates on $\|\zeta'_k\|_\infty$, we first make the following general observations.
If $\psi_{k+2}\in L^\infty(\fS_{k+2})$, then $\left(\cL_{k} \cL_{k+1}\psi_{k+2}\right)(x)=\int \tp(x,z)\psi_{k+2}(z)\mu_{k+2}(dz)$, where
$$\tp(x,z)=\int_{\fS_{k+1}} p_k(x,y)p_{k+1}(y,z)\mu_{k+1}(dy).$$
By uniform ellipticity,
 $\tp\geq \eps_0$ so we can decompose $\tp_k=\eps_0+(1-\eps_0) \tq_k$
where $\tq_k$ is a probability density.
Hence if $\psi_{k+2}$ has zero mean then
\begin{align*}
&\left(\cL_{k} \cL_{k+1}\psi_{k+2}\right)(x)=\eps_0\int \psi_{k+2} d\mu_{k+2}
+(1-\eps_0) \int \tq_k(x,y) \psi_{k+2}(y) \mu_{k+2}(dy)\\
&=(1-\eps_0) \int \tq_k(x,y) \psi_{k+2}(y) \mu_{k+2}(dy).
\end{align*}
Thus
$ \left\|\cL_{k} \cL_{k+1}\psi_{k+2}\right\|_\infty\leq (1-\eps_0)\|\psi_{k+2}\|_\infty. $

We apply this to $\zeta_{k+2}'=\left(\cL_{k} \cL_{k+1}\right)\zeta'_{k+2}+
\left(\cL_{k} \cL_{k+1}\right)\zeta_{k+2}''$:
\begin{align*}
\|\zeta_k'\|_\infty&\leq (1-\epsilon_0)\|\zeta_k'\|_\infty+\|\cL_k \cL_{k+1}\zeta_{k+2}''\|_\infty\\
&\leq (1-\epsilon_0)\|\zeta_k'\|_\infty+\|\cL_k\|_{L^1\to L^\infty}\|\cL_{k+1}\|_{L^1\to L^1}\|\zeta_{k+2}''\|_1\leq (1-\epsilon_0)\|\zeta_k'\|_\infty+\epsilon_0^{-2}\|\zeta_{k+2}''\|_1.
\end{align*}
The last step is because $0\leq p_n(x,y)\leq \epsilon_0^{-1}$.
This proves \eqref{Claim4-1}.

Next we analyze $\|\zeta_k''\|_\infty$. Since $\zeta_k''$ has zero mean and $\eta_{k+2}-\eta_k$ is constant, we can write  $\zeta_k''=\wh{\zeta}_k''-\E(\wh{\zeta}_k'')$ with $\wh{\zeta}_k'':=\left(\brcL_{k, \xi} \brcL_{k+1, \xi}-\cL_{k} \cL_{k+1}\right)\phi_{k+2}$.
Observe that the kernel of
$\left(\brcL_{k, \xi} \brcL_{k+1, \xi}-\cL_{k} \cL_{k+1}\right)$ is bounded by
$$\textrm{const}\frac{|s|}{\sqrt{V_N}}\int \bigl(\bbf_k(x, z)+\bbf_{k+1}(z,y)+c_k+c_{k+1} \bigr)\mu_{k+1}(dz).$$
By assumptions II and III,
 the $L^1$-norm of the kernel is bounded by
\begin{align*}
&O\left(\frac{|s|}{\sqrt{V_N}}\right)\biggl(\|f_k-h_k\|_1+\|f_{k+1}-h_{k+1}\|_1
\biggr)
= O\left(\frac{|s|}{\sqrt{V_N}}\right)\biggl(\|f_k\|_1+\|h_k\|_1+\|f_{k+1}\|_1+\|h_{k+1}\|_1
\biggr)
\end{align*}
\begin{align*}
&\leq O\left(\frac{|s|}{\sqrt{V_N}}\right)\biggl(\|f_k\|_2+\|h_k\|_2+\|f_{k+1}\|_2+\|h_{k+1}\|_2
\biggr).
\end{align*}
This implies  that $\|\wh{\zeta}_k''\|_\infty=O\left(\frac{|s|}{\sqrt{V_N}}\right)\bigl(\sigma(f_k)+\sigma(f_{k+1})+\sigma(h_k)+\sigma(h_{k+1})
\bigr),$ whence  $\|{\zeta}_k''\|_\infty\leq 2\|\wh{\zeta}_k''\|_\infty=O\left(\frac{|s|}{\sqrt{V_N}}\right)\bigl(\sigma(f_k)+\sigma(f_{k+1})+\sigma(h_k)+\sigma(h_{k+1})
\bigr).$
\eqref{Claim4-2} and the sub-claim are proved.

\medskip
We return to the proof of Claim 3.
Iterating the estimate in the sub-claim, we conclude that for some constant $\brC$
$$\|\zeta_k'\|_\infty\leq
\brC \left[\htheta^{ 2\lfloor\frac{N-k}{2}\rfloor}_0
+\sum_{r=1}^{\lfloor\frac{N-k}{2}\rfloor-1}
\frac{\htheta^{2r}_0
\left(\sigma(f_{k+2r})+\sigma(f_{k+2r+1})+\sigma(h_{k+2r})+\sigma(h_{k+2r+1})\right)}{\sqrt{V_N}}\right]
$$
$$
\leq \brC \hat{\theta}_0^{-1}\left[\htheta^{N-k}_0+\sum_{r=1}^{N-k}
\frac{ \hat\theta^{r}_0}{\sqrt{V_N}}
\biggl(\sigma(f_{k+r})+\sigma(h_{k+r})\biggr)\right].
$$
Since $\cL_{j, \xi}$ are contractions and
$\|\hcL_{k,\xi}\|_{L^\infty\to L^1}\leq C_1(\brK,m)\sigma(h_k)$, this implies that
\begin{align*}
&\sum_k \bigl\|\cL_{1, \xi} \dots \cL_{k-1, \xi} \hcL_{k, \xi} (\zeta_{k+1}') \bigr\|_{L^1}\\
&\leq
\brC C_1(\brK,m) \left[\sum_{r} \htheta^r_0 \sum_k \sigma(h_k) \frac{\sigma(f_{k+r})+\sigma(h_{k+r})}{\sqrt{V_N}}+\sum_k
\sigma(h_k) \htheta^{N-k}_{ 0}\right].
\end{align*}
As in the proof of Claim 2, it follows from the  Cauchy Schwartz inequality, \eqref{U-V-Uniform}, and assumption III(b) that
the sum over $k$ is $O(\sqrt{\eps})$.
Hence
\begin{equation}
\label{ZetaPEps}
 \sum_k  \left\Vert\cL_{1, \xi} \dots \cL_{k-1, \xi} \hcL_{k, \xi} (\zeta_{k+1}') \right\Vert_{L^1}=O(\sqrt{\eps}).
\end{equation}

Next we claim that \begin{equation}
\label{ZetaPPEps}
  \sum_k \left\Vert\cL_{1, \xi} \dots \cL_{k-1, \xi} \hcL_{k, \xi} (\zeta_{k+1}'') \right\Vert_{L^1}=O(\sqrt{\eps}).
\end{equation}
The proof is similar to the proof of \eqref{ZetaPEps}, except that now we  use \eqref{ZetaPP} to see that as in the  proofs of \eqref{HLBoundInf1},\eqref{TLBoundInf1} and \eqref{Claim4-2},
$$
\|\cL_{1, \xi} \dots \cL_{k-1, \xi} \hcL_{k, \xi} (\zeta_{k+1}'')\|_{L^1}
\leq C_4(\brK,m) \sigma(h_k)\frac{\sigma(f_{k+1})+\sigma(f_{k+2})+\sigma(h_{k+1})+\sigma(h_{k+2})}{\sqrt{V_N}} $$
for some constant $C_4(\brK,m)$.

\eqref{ZetaPEps} and \eqref{ZetaPPEps} give us an $O(\sqrt{\epsilon})$ bound for contribution of $\zeta_{k+1}$ to \eqref{Phi-Zeta-Nu}.
It remains to estimate the contribution of $\eta_{k+1}$ to \eqref{Phi-Zeta-Nu}.

\medskip
Split
$ \cL_{n,\xi}=e^{2\pi i mc_n} \cL_{n}+\cL_{n,\xi}'$. As before,
\begin{equation}
\begin{aligned}
&\cL_{1, \xi} \cdots \cL_{k-1, \xi} \hcL_{k, \xi}(1)=
e^{2\pi i m(c_1+\cdots+c_{k-1})}\cL_{1} \cdots \cL_{k-1} \hcL_{k, \xi}(1)\\
&\hspace{0.3cm}+
\sum_j e^{2\pi i m(c_{j+1}+\cdots+c_{k-1})}\cL_{1, \xi} \cdots \cL_{j-1, \xi} \cL_{j, \xi}' \cL_{ j+1} \cdots \cL_{k-1} \hcL_{k, \xi}(1).
\end{aligned}
\label{ChangeRow}
\end{equation}
Since $\EXP(h_k)=0$,  $\EXP[(\hcL_{k, \xi}1)(X_k)]=0.$
By exponential mixing \eqref{Exp-Mixing-L-infinity},
 the first term on the RHS of \eqref{ChangeRow} has
$L^\infty$ (whence $L^1$) norm no larger than
$$ C_{mix}\theta^{k-1}\|\hcL_{k,\xi} 1\|_\infty\leq \hC_3 \sigma(h_k)\theta^k$$
for some constant $\hC_3=\hC_3(\brK,m)$ and $0<\theta<1$.
Similarly each summand in the second term on the RHS of \eqref{ChangeRow} has $L^1$ norm less than
\begin{align*}
&\|\cL_{j,\xi}'\| \hC_3 \sigma(h_k) \theta^{k-j}
\leq \hC_4 \sigma(h_k) \theta^{k-j}\left\|\frac{s}{\sqrt{V_N}}(\bbf_j+c_j)+\xi h_j\right\|_2\\
&\leq
\hC_4 \sigma(h_k) \theta^{k-j}\left(\frac{\sigma(f_j)}{\sqrt{V_N}}+\sigma(h_j)\right),
\end{align*}
for $\hC_4=\hC_4(\brK,m). $
So the second term on the RHS of \eqref{ChangeRow} has norm less than
\begin{equation}\label{second-term}
\hC_5 \sigma(h_k)\sum_{j=1}^k \theta^{k-j}\biggl(\frac{\sigma(f_j)}{\sqrt{V_N}}+\sigma(h_j)\biggr)
\end{equation}
for some constant $\hC_5$.

It follows that
$\DS \sum_k \left\|\cL_{1, \xi} \dots \cL_{k-1, \xi} \hcL_{k, \xi}(1)  \right\|_1$ is bounded by
\begin{align*}
&\sum_{k=1}^N\left(
\hC_3 \sigma(h_k)\theta^k+\hC_5
\sigma(h_k)\sum_{j=1}^k \theta^{k-j}\biggl(\frac{\sigma(f_j)}{\sqrt{V_N}}+\sigma(h_j)\biggr)\right)\\
&\leq  \hC_3\sqrt{\sum_{k=1}^N \sigma^2(h_k)}
\sqrt{\sum_{k=1}^N\theta^{2k}}+\hC_5 \sum_{r=0}^{N-1} \theta^r \sum_{j=1}^N  \biggl(\frac{\sigma(f_j)}{\sqrt{V_N}}+\sigma(h_j)\biggr) \sigma(h_{j+r})\\
&\leq  \frac{\hC_3\sqrt{\epsilon}}{\sqrt{1-\theta^2}}+\hC_5\sum_{r=0}^{N-1} \theta^r
\left[\left(
\sqrt{\sum_{j=1}^N \frac{\sigma^2(f_j)}{V_N}}+\sqrt{\sum_{j=1}^N \sigma^2(h_j)}\;\;\right)\sqrt{\sum_{j=1}^N \sigma^2(h_{j+r})
}\;\;\right].
\end{align*}
By  assumptions II and III,  there is a constant $\hC_6=\hC_6(\brK,m)$ such that
\begin{equation}\label{thirty-five}
\sum_k \left\|\cL_{1, \xi} \dots \cL_{k-1, \xi} \hcL_{k, \xi}(1)  \right\|_1\leq \hC_6\sqrt{\eps}.
\end{equation}
Claim 3 now follows from \eqref{Phi-Zeta-Nu}, \eqref{ZetaPEps}, \eqref{ZetaPPEps},
and \eqref{thirty-five}.

\medskip
Lemma \ref{LmCharAlm-AP} now follows from Claims 1--3 and \eqref{Phi-decomposition}.\qed
\end{proof}

\subsection{Proof of the LLT in the reducible case}
{\bf Setup and reductions.} Let  $\mathsf f=\{f_n\}$ be an  a.s. uniformly bounded additive functional on a Markov chain $\mathsf X=\{X_n\}$ with state spaces $\fS_n$ and marginals $\mu_n(E)=\Prob(X_n\in E)$.
We assume that $\mathsf f$ is not center-tight, and that $\mathsf f$ is reducible.
In this case $G_{ess}(\mathsf X,\mathsf f)=\delta(\mathsf f)\Z$ with some $\delta(\mathsf f)>0$. Without  loss of generality,
$$
\delta(\mathsf f)=1\ , \ G_{ess}(\mathsf f)=\Z\ , \
\E(f_n):=\E[f_n(X_n,X_{n+1})]=0\text{ for all $n$,}
$$
otherwise we center and rescale $\mathsf f$.

By the reduction lemma (Lemma \ref{Lemma-Reduction}), $\mathsf f=\bbf+\nabla \mathsf a+\mathsf h+\mathsf c$, where
$$G_{alg}(\mathsf X,\bbf)=G_{ess}(\mathsf X,\bbf)=\Z,$$ $\mathsf h$ has summable variances and $\E(h_n):=\E(h_n(X_n,X_{n+1}))=0$,  $\mathsf c=\{c_n\}$ are constants, and  $\bbf, \mathsf a, \mathsf h, \mathsf c$ are a.s. uniformly bounded.
There is no loss of generality in assuming that $\mathsf a\equiv 0$, because Theorem \ref{Theorem-Reducible-LLT} holds for $\mathsf f$ with $b_N$ iff Theorem \ref{Theorem-Reducible-LLT} holds for $\mathsf f-\nabla \mathsf a$ with $b_N'(X_1,X_{N+1}):= b_N(X_1,X_{N+1})+a_{N+1}(X_{N+1})-a_1(X_1)$.

Henceforth we assume
$
\mathsf f=\bbf+\mathsf h+\mathsf c,
$
and $\E(f_n)=\E(h_n)=0$. So $c_n=-\E(\bbf_n)$.
Let
\begin{equation}\label{c(N)}
c(N):=-\sum_{k=1}^N \E[\bbf_k(X_k,X_{k+1})].
\end{equation}
By Theorem \ref{Proposition-Kolmogorov-Three-Series}, the following sum converges a.s.:
$$
\mathfrak H(X_1,X_2,\ldots):=\sum_{n=1}^\infty h_n(X_n,X_{n+1}).
$$

\begin{lemma}
\label{CrLimChar-AP}
Under the previous assumptions,
for every
sequence of non-negative functions $v_{N+1}\in L^\infty(\mathfrak S_{N+1})$  s.t. $\|v_{N+1}\|_\infty\neq 0$
 and for some $\brdelta>0$
 \begin{equation}
 \label{Inf-LInf}
 \int_{\fS_{N+1}} v_{N+1} d\mu_{N+1}\geq \brdelta ||v_{N+1}||_\infty,
 \end{equation}
 for all $m\in\Z$, $s\in\R$ and  $x\in\fS_1$,
\begin{equation}
\label{LimChar-AP}
 \frac{\EXP_x\left(e^{i(2\pi m +\frac{s}{\sqrt{V_N}}) S_N} { v_{N+1}(X_{N+1})}\right)}{{ \EXP(v_{N+1}(X_{N+1})})}=e^{2\pi i m c(N)-s^2/2} \EXP_x\left(e^{2\pi m i \mathfrak H}\right)+o_{N\to\infty}(1).
\end{equation}
where $o(\cdot)$ term converges to 0 uniformly  when $|m+is|$ are bounded, $v_{N+1}$ are bounded, and \eqref{Inf-LInf} holds.
\end{lemma}

\begin{proof}
 Since the LHS of \eqref{LimChar-AP} remains unchanged upon multiplying
$v_{N+1}$ by a constant, we may assume that $\|v_{N+1}\|_\infty=1.$

Fix $\epsilon>0$ small and $r$ so large that $\DS \sum_{k=r}^\infty\Var(h_k)<\epsilon$.
Fix $N$. Applying the Integer Reduction Lemma (Lemma \ref{LmIntRed}) to $\{\bbf_n\}_{n=r}^N$,
we obtain a decomposition
$$ \bbf_n(x_n,x_{n+1})=\fa_{n+1}^{(N)}(x_{n+1})-\fa_{n}^{(N)}(x_{n})
+\fc_n^{(N)}+\tf_n^{(N)}(x_n, x_{n+1}) $$
where $\fc_n^{(N)}$ are bounded  integers,  and $\fa_n^{(N)}(\cdot)$, $\tf_n^{(N)}(\cdot,\cdot)$ are uniformly bounded measurable integer valued functions   such that
$$\sum_{n=r}^N \|\tf_n^{(N)}\|_2^2=O\left(\sum_{n=r}^N u_n^2(\bbf)\right).$$
There is no loss of generality in assuming that  $\fa_{N+1}^{(N)}=\fa_{r}^{(N)}=0$, otherwise
replace $\tf_r^{(N)}(x,y)$  by   $\tf_r^{(N)}(x,y)-\fa_r^{(N)}(x)$, and $\tf_N^{(N)}(x,y)$ by $\tf_N^{(N)}(x,y)+\fa_{N+1}^{(N)}(y)$.
Then $\DS \sum_{n=r}^N \bbf_n=\sum_{n=r}^N (\fc_n^{(N)}+\tf_n^{(N)})$, whence \begin{equation}\label{Sf=Sg}
S_N-S_{r-1}=\sum_{n=r}^N f_n=\sum_{n=r}^N \fc_n^{(N)}+\tf_n^{(N)}+h_n+c_n=\sum_{n=r}^N \tf_n^{(N)}+h_n-\E(\tf_n^{(N)}+h_n).
\end{equation}
(The last equality is because $\E(S_N-S_{r-1})=0$.)

 Let $\mathsf g$ denote the array with rows $g^{(N)}_n:=\tf_n^{(N)}+h_n-\E(\tf_n^{(N)}+h_n)$ $(n=r,\ldots,N)$, $N>r$. We claim that $\mathsf g$ satisfies assumptions (I)--(III) of Lemma  \ref{LmCharAlm-AP}.
 (I) is clear, and (III) holds by choice of $r$ and because $\tf_n^{(N)}$ is integer valued.  To see (II), note that
\begin{align*}
& \sum_{n=1}^N \sigma^2(g_n^{(N)})=\sum_{n=1}^N \sigma^2(\tf_n^{(N)}+h_n)=\sum_{n=1}^N \sigma^2(\tf_n^{(N)})+\sigma^2(h_n)+2\Cov(\tf_n^{(N)},h_n)\\
&\leq \sum_{n=1}^N \sigma^2(\tf_n^{(N)})+\sigma^2(h_n)+2\sigma(\tf_n^{(N)})\sigma(h_n)\leq 2\sum_{n=1}^N\sigma^2(\tf_n^{(N)})+\sigma^2(h_n)\ \ (\because 2ab\leq a^2+b^2)\\
&=O\biggl(\sum_{n=r}^N u_n^2(\bbf)\biggr)+O(1),\text{ by choice of $\wt{\mathsf f}$ and $\mathsf h$}.
\end{align*}
Since $\mathsf f=\bbf+\mathsf h+\mathsf c$,\; $u_n^2(\bbf)=u_n^2(\mathsf f+\mathsf h)\leq 2[u_n^2(\mathsf f)+u_n^2(\mathsf h)]$, see Lemma \ref{Lemma-Sum}(4). Thus  by Theorem \ref{Theorem-MC-Variance} and the assumption that $\mathsf h$ has summable variances,
$$
\sum_{n=r}^N u_n^2(\bbf)\leq 2\sum_{n=r}^N u_n^2(f)+u_n^2(h)=O\bigl(V_N\bigr)+O(1)=O(V_N).
$$
Assumption (II) is checked.

We now apply Lemma \ref{LmCharAlm-AP}  to $\mathsf g$, and deduce that for every $\brK>0$ and $m\in\Z$ there are $\brC,\brN>0$ such that for all $N>\brN+r$, $|s|\leq \brK$, and $v_{N+1}$ in the unit ball of $L^\infty$
\begin{align*}
&\EXP\left(e^{i(2\pi m+\frac{s}{\sqrt{V_N}}) (S_N-S_{r-1})} v_{N+1}(X_{N+1})\bigg|X_r\right)\\
&=e^{2\pi i m c^{(N)}}\cdot e^{-s^2/2} {\EXP(v_{N+1}(X_{N+1}))} +\eta_{N-r}(X_r),
\end{align*}
 where $c^{(N)}:=-\sum_{n=r}^N\E(\tf^{(N)}_n)$ and $\|\eta_{N-r}\|_1\leq \brC\sqrt{\epsilon}$. Since $\|v_{N+1}\|_\infty= 1$, we also have the trivial bound $\|\eta_{N-r}\|_\infty\leq 2$.

\medskip
We are ready to prove the lemma. The left-hand-side of \eqref{LimChar-AP} equals
\begin{align*}
&\frac{\EXP_x\left(e^{i(2\pi m+\frac{s}{\sqrt{V_N}}) S_N} v_{N+1}(X_{N+1})\right)}
{\EXP(v_{N+1}(X_{N+1}))}=\\
&=\EXP_x\left(
e^{i(2\pi m+\frac{s}{\sqrt{V_N}}) S_{r-1}}
\frac{\EXP\left(e^{i(2\pi m+\frac{s}{\sqrt{V_n}}) (S_N-S_{r-1})} v_{N+1}(X_{N+1})\big|X_r\right)}
{\EXP(v_{N+1}(X_{N+1}))}
\right)
\end{align*}
\begin{align*}
&=\EXP_x\left[
e^{i(2\pi m+\frac{s}{\sqrt{V_N}}) S_{r-1}}\left(
e^{2\pi i m c^{(N)}-s^2/2}+\frac{\eta_{N-r}(X_r)}{\E(v_{N+1}(X_{N+1}))}\right)
\right]\\
&=\underset{A}{\underbrace{e^{2\pi i m c^{(N)}-s^2/2}\E_x(e^{2\pi i m S_{r-1}+o(1)})}}+O(\brdelta^{-1})\underset{B}{\underbrace{\E_x(\eta_{N-r}(X_r))}},\text{ as $N\to\infty$}.
\end{align*}

We examine $A$,$B$. Let $\DS \hc^{(r-1)}:=\sum_{k=1}^{r-1}c_k=
-\E(\sum_{k=1}^{r-1}\bbf_k(X_k,X_{k+1}))$. Since $\DS c(N)=-\sum_{k=1}^{N}\E(\bbf_k)$,
\begin{align*}
c(N)&=-\sum_{k=1}^{r-1}\E(\bbf_k)
-\sum_{k=r}^N(\E(\tf_n^{(N)})+\fc_n^{(N)})\text{ because $\sum_{n=r}^N \bbf_n=\sum_{n=r}^N (\fc_n^{(N)}+\tf_n^{(N)})$}\\
&\equiv \hc^{(r-1)}+c^{(N)}\mod\Z , \text{ because $\fc_n^{(N)}\in\Z$}.
\end{align*}

By assumption,  $\mathsf f=\bbf+\mathsf h+\mathsf c$ with $\bbf$ integer valued. Necessarily,
\begin{equation}
\label{CharFr}
 \exp(2\pi i m S_{r-1})=\exp(2\pi i m \mathfrak{H}_r
  +2\pi i m \hc^{(r-1)})
  \end{equation}
where
 $\DS  \mathfrak{H}_r:=\sum_{k=1}^{r-1} h_k(X_k,X_{k+1}).$
By choice of $r$ and Lemma \ref{LmVarSum},
$$\bigl| \EXP_x(e^{i\xi\mathfrak{H}})-\EXP_x(e^{i\xi\mathfrak{H}_r})\bigr|
\leq |\xi| \EXP_x\left(\left|\mathfrak{H}-\mathfrak{H}_r\right|\right)\leq |\xi|\Var\biggl(\sum_{k=r}^\infty h_k(X_k,X_{k+1})\biggr)^{1/2}=O\left(\sqrt{\eps}\right) $$
uniformly when $\xi$ varies in a compact domain. Substituting \eqref{CharFr} in $A$, we obtain
$$
A=[1+o(1)] e^{2\pi i m c(N)-\frac{s^2}{2}}
\E_x\left(e^{2\pi i m \mathfrak H}\right)+ O\left(\sqrt{\eps}\right).
$$
Next, the  exponential mixing of $\mathsf X$ implies that  for all $N$ large enough, $$B:=\E_x(\eta_{N-r}(X_r))=\E(\eta_{N-r}(X_r))+o(1)=O(\sqrt{\epsilon}).$$

Thus the left-hand-side of  \eqref{LimChar-AP} equals
$
e^{2\pi i m c(N)-s^2/2}\E_x(e^{2\pi i m \mathfrak H+o(1)})+O(\sqrt{\epsilon})
$. The lemma follows, because $\epsilon$ was arbitrary.
\qed
\end{proof}

\medskip
\noindent
{\bf Proof of Theorem \ref{Theorem-Reducible-LLT}.}
Suppose $\mathsf f$ is an a.s. uniformly bounded additive functional on a uniformly elliptic Markov chain $\mathsf X$, and assume $G_{ess}(\mathsf f)=\delta(\mathsf f)\Z$ with $\delta(\mathsf f)\neq 0$.

We begin with some reductions.
By Theorem \ref{Theorem-minimal-reduction}, $\mathsf f$ has an optimal reduction, and we can write
$\mathsf f=\bbf+{\mathsf F}$ where $\bbf$ has algebraic range $\delta(\mathsf f)\Z$ and ${\mathsf F}$
is a.s. uniformly bounded and center-tight. There is no loss of generality in assuming that
$\ess\sup|{\mathsf F}|\leq \delta(\mathsf f)$, since this can always be arranged
by replacing ${F}_n$ by ${F}_n\mod\delta(\mathsf f)$.
Next by the gradient lemma (Lemma \ref{LmVarAbove}), we decompose
$$
{\mathsf F}=\nabla \mathsf a+\wt{\mathsf f}+\wt{\mathsf c}
$$
where $\ess\sup|\mathsf a|\leq 2\ess\sup|\mathsf F|$,
$\wt{\mathsf f}$ has summable variances, and $\wt{c}_n$ are  constants.

It is convenient to introduce
 $\DS f_n^\ast:=\frac{1}{\delta(\mathsf f)}[f_n-\nabla a_n-\E(f_n-\nabla a_n)].$
 $G_{ess}(\mathsf X,\mathsf f^\ast)=\Z$, and
\begin{equation}\label{cello}
\mathsf f^\ast=\frac{1}{\delta(\mathsf f)}\bbf+\mathsf h+\mathsf c,
\end{equation}
where $h_n:=\frac{1}{\delta(\mathsf f)}[\wt{f}_n-\E(\wt{f}_n)]$ is a centered additive functional with summable variances,
and $c_n:=\frac{1}{\delta(\mathsf f)}[\wt{c}_n+\E(\wt{f}_n)-\E(f_n-\nabla a_n)]$.

We first prove the theorem in the special case when
\begin{equation}\label{special-case}
\text{ $\delta(\mathsf f)=1$, $\E(f_n)=0$ for all $n$,  and
$\mathsf a\equiv 0$. }
\end{equation}
In this case $\mathsf f=\mathsf f^\ast$ and  \eqref{cello}
 places us in the
setup of Lemma \ref{CrLimChar-AP}.
Given this lemma, the proof is very similar to the proof of the
local limit theorem in the irreducible non-lattice case, but we reproduce it for completeness. We focus on parts (2) and (3)
of the theorem, because part (1) follows from them.

Define as in \eqref{c(N)},
 $\DS c(N):=-\frac{1}{\delta(\mathsf f)}\sum_{k=1}^N \E[\bbf_k(X_k,X_{k+1})]$, and
let $$\mathfrak H:=\sum_{n=1}^\infty h_n(X_n,X_{n+1}), \quad b_N:=\{c(N)\}.$$

Fix $\phi\in L^1(\R)$ such that
$\supp(\hphi)\subset [-L, L]$,
and
let $v_{N+1}$ denote the indicator function of
$\fA_{N+1}$. By the Fourier inversion formula
\begin{align}
&\EXP_x(\phi(S_N-b_N-z_N)|X_{N+1}{ \in \fA_{N+1}})\notag\\
&=\frac{1}{2\pi}
\int_{-L}^L \hphi(\xi)\frac{\EXP_x\left(e^{i\xi(S_N-b_N-z_N)}  v_{N+1}(X_{N+1})\right)}
{\EXP( v_{N+1}(X_{N+1}))}\label{the-integral}
 d\xi
 \end{align}
 and the task is to find the asymptotic behavior of \eqref{the-integral} in case $z_N\in \Z$,
 $\frac{z_N}{\sqrt{V_N}}\to z$.

Let $K:=\ess\sup|\mathsf f|$ and recall the constant $\tdelta=\tdelta(K)$ from Lemma \ref{Lm2P}.
Split $[-L, L]$ into a finite collection of subintervals $I_j$  of length less than $\min\{\tdelta, \pi\}$,
 in such a way that every $I_j$ is either bounded away from $2\pi\Z$, or intersects it an unique point $2\pi m $ exactly at its center.

If $I_j\cap 2\pi\Z=\varnothing$, then $\sum d_n^2(\xi)=\infty$ uniformly on $I_j$ (Theorem \ref{Theorem-MC-array-difference}). Thus by \eqref{CondBoundChar},
  $\Phi_N(x,\xi)\to 0$ uniformly on $I_j$. In this case
we   can argue as in the proof of \eqref{SmallL1Norm}
    and show that the contribution of $I_j$ to the integral \eqref{the-integral} is $o\bigl(V_N^{-1/2}\bigr).$

If $I_j\cap 2\pi\Z\neq\varnothing$, then the center of $I_j$ equals $2\pi m $ for some $m\in\Z$. Fix some large $R$.
Let
$ J_{j,N}'$ be the contribution to the integral from the set  $\{\xi\in I_j: |\xi-2\pi m|\leq RV_N^{-1/2}\}$,
and let
$ J_{j,N}''$ be the integral over $\{\xi\in I_j: |\xi-2\pi m|> RV_N^{-1/2}\}$.

The main contribution  comes from $J_{j,N}'$, because
one can show as in Claim 2 in \S \ref{section-proof-of-llt-irred-non-lattice} that
$\DS |J_{j,N}''|\leq C \int_{|u|>R V_N^{-1/2}} e^{-c V_N u^2} du\leq C \frac{e^{-cR^2}}{R\sqrt{V_N}}$, which is negligible for $R\gg 1$.

To estimate $J_{j,N}'$,
we make the change of variables $\xi=2\pi m+\frac{s}{\sqrt{V_N}}$. Since $z_N\in\Z$ and
$b_N=\{c(N)\}$, we have
$ \xi(S_N-b_N-z_N)=\xi S_N-2\pi m c(N)-\frac{s}{\sqrt{V_N}} (z_N+\{c(N)\}) \quad \text{mod} \quad 2\pi.$
So
{
$$ J_{j,N}'=
\frac{1}{2\pi\sqrt{ V_N}} \left[
\int_{|s|<R}{ \hphi(2\pi m +\frac{s}{\sqrt{V_N}})}
\frac{e^{-2\pi i m c(N) }
\EXP_x\left(e^{i\xi S_N} { v_{N+1}(X_{N+1})}\right)}{\EXP({ v_{N+1}(X_{N+1})})} \; e^{-is \frac{z_N+O(1)}{\sqrt{V_N}}}
 ds\right]. $$}
Fixing $R$ and letting $N\to\infty$, we see by  Lemma  \ref{CrLimChar-AP} that
$$
\sqrt{V_N}J_{j,N}'= \frac{\hphi(2\pi m)}{2\pi}\EXP_x\left(e^{2\pi i m   {\mathfrak H}}\right)
\int_{|s|<R}  e^{-isz-s^2/2} ds+o_{N\to\infty}(1)$$
$$= \frac{\hphi(2\pi m)}{\sqrt{2\pi}}
\EXP_x\left(e^{2\pi i m   \mathfrak H}\right) e^{-z^2/2}
+o_{R\to\infty} (1)+o_{N\to\infty} (1). $$
Combining the estimates for $J_{j,N}$ we obtain that
$$ \lim_{N\to\infty} \sqrt{V_N}  J_{j,N}=
 \frac{e^{-z^2/2}}{\sqrt{2\pi}}
\EXP_x\left(e^{2\pi i m   \mathfrak H}\right) \hphi(2\pi m ), $$
if $I_j$ intersects $2\pi \Z$, and this limit is zero otherwise.
Hence
\begin{align*}
&\lim_{n\to\infty} \sqrt{V_N} \EXP_x(\phi(S_N-b_N-z_N)|X_{N+1}\in \fA_{N+1})
\notag \\
&=\frac{ e^{-z^2/2}}{\sqrt{2\pi}}
\sum_{m\in \Z\cap [-L ,L ]}  \EXP_x\left(e^{ 2\pi i m   \mathfrak H}\right)
\hphi(2\pi m ) \notag
=\frac{ e^{-z^2/2}}{\sqrt{2\pi}}
\sum_{m\in \Z}  \EXP_x \left(e^{ 2\pi i m   \mathfrak H}\right) \hphi(2\pi m )\\
&\equiv\frac{ e^{-z^2/2}}{\sqrt{2\pi}}
\sum_{m\in \Z}  \EXP_x \left(e^{ 2\pi i m   \mathfrak F}\right) \hphi(2\pi m ),
 {
\text{ where $\mathfrak F\in [0,1)$,
$\mathfrak F:=\mathfrak H\mod\Z$}}
\end{align*}
\begin{align*}
&=\frac{ e^{-z^2/2}}{\sqrt{2\pi}}\sum_{m\in \Z}
\widehat{\left(\cC_{x} \phi\right)}(2\pi m ), \text{ where }(\cC_x \phi)(t):=\EXP_x[\phi(t+\mathfrak F)]
\notag \\
&
=\frac{e^{-z^2/2}}{\sqrt{2\pi}}\sum_{m\in\Z} \left(\cC_{x}\phi \right)(m)
\equiv \frac{e^{-z^2/2}}{\sqrt{2\pi}}\sum_{m\in\Z} \E_x[\phi(m+\mathfrak F)],
\notag
\end{align*}
by the Poisson summation formula.\label{Poisson-Summation-Trick}

This proves part (2) of the theorem in the special case \eqref{special-case}, and in particular for the
additive functional $\mathsf f^\ast$ defined above.
Now consider the general case:
$$
S_N(\mathsf f)-\E[S_N(\mathsf f)]\equiv \delta(\mathsf f) S_N(\mathsf f^\ast)+a_{N+1}(X_{N+1})-a_1(X_1)+\E[a_1(X_1)-a_{N+1}(X_{N+1})].
$$
Since part (2) of the theorem holds for $\mathsf f^\ast$ with
$\mathfrak F=\{\sum h_n\}\in [0,1)$ and $b_N=\{c(N)\}$, it must hold for $\mathsf f$ with $\delta(\mathsf f)\mathfrak F$ and
$$
b_N(X_1,X_{N+1}):=\delta(\mathsf f)\{c(N)\}+a_{N+1}(X_{N+1})-a_1(X_1)+\E[a_1(X_1)-a_{N+1}(X_{N+1})].
$$
Clearly $|b_N|\leq \delta(\mathsf f)+4\ess\sup|\mathsf a|$. Recalling that $\ess\sup|\mathsf a|
\leq 2\ess\sup|\mathsf F|\leq 2\delta(\mathsf f)$, we find that $\ess\sup|b_N|\leq 9\delta(\mathsf f)$, proving part (3) as well.
\hfill$\Box$

\subsection{Necessity of the irreducibility assumption}\label{Section-Irreducibility-As-Necessary-Condition}
Suppose $\mathsf f$ is an a.s. uniformly bounded additive functional on a uniformly elliptic Markov chain $\mathsf X$. Recall that $\mathsf f_r=\{f_n\}_{n\geq r}$ and $\mathsf X_r=\{X_n\}_{n\geq r}$. In this section we prove Theorem \ref{Theorem-Characterization-Irreducibility}, which asserts the equivalence of the following three conditions:
\begin{enumerate}[(a)]
\item $\mathsf f$ is irreducible with algebraic range $\R$.
\item $(\mathsf X_r,\mathsf f_r)$ satisfies the mixing non-lattice local limit theorem, for all $r$.
\item $(\mathsf X_r,\mathsf f_r)$ satisfies the mixing uniform distribution mod $t$ for all $r$ and $t>0$.
\end{enumerate}

\medskip
\noindent
{\bf (a)$\Rightarrow$(b):} To see this recall that additive functionals on uniformly elliptic Markov chains are special cases of stably hereditary additive functionals on uniformly elliptic Markov arrays, and apply Theorem \ref{Theorem-Mixing-LLT}(1) to $\phi$ continuous with compact support which approximate indicators of intervals in $L^1(\R)$.

\medskip
\noindent
{\bf (b)$\Rightarrow$(a):} Assume $\mathsf f$ satisfies the ``mixing non-lattice LLT" property.
By definition, $V_N\to\infty$, and therefore $\mathsf f$ is not center-tight.

Also,  $G_{alg}(\mathsf X,\mathsf f)=\R$, otherwise $\Prob_x(S_N-z_N\in (a,b)|X_{N+1}\in\mathfrak A_{N+1})=0$ for $z_N$ and  $(a,b)$ such that $z_N+(a,b)\subset \R\setminus G_{alg}(\mathsf X,\mathsf f)$.

If $G_{ess}(\mathsf X,\mathsf f)=\R$ then $\mathsf f$ is irreducible and we are done. Assume by way of contradiction that $G_{ess}(\mathsf X,\mathsf f)\neq \R$, then $G_{ess}(\mathsf X,\mathsf f)=t\Z$ for some $t>0$ ($t=0$ is impossible because $\mathsf f$ is not center-tight). There is no loss of generality in assuming that
$$
G_{ess}(\mathsf X,\mathsf f)=\Z\quad\text{and}\quad
\E(f_n(X_n,X_{n+1}))=0\text{ for all $n$.}
$$

Let $S_N^{(r)}:=f_r(X_r,X_{r+1})+\cdots +f_N(X_N,X_{N+1})$ and $V_N^{(r)}:=\Var(S_N^{(r)})$.  By the exponential mixing of $\mathsf X$ (Proposition \ref{Proposition-Exponential-Mixing}),
$$
|V_N-V_N^{(r)}|=|V_{r-1}+2\Cov(S_N^{(r)},S_{r-1})|\leq V_r+2\sum_{j=1}^{r-1}\sum_{k=r}^\infty \Cov(f_j,f_k)=O(1).
$$
Therefore, for fixed $r$, $V_N/V_N^{(r)}\xrightarrow[N\to\infty]{}1$.

Since $G_{alg}(\mathsf X,\mathsf R)=\R$ and $G_{ess}(\mathsf X,\mathsf f)=\Z$, $\mathsf f$ is reducible, and we can write
$$
\mathsf f=\bbf+\nabla\mathsf a+\mathsf h+\mathsf c,
$$
where $\bbf$ is irreducible with algebraic range $\Z$, $a_n(x)$ are uniformly bounded (say by $K$),
$\mathsf h$ has summable variances, $\E(h_n)=0$, and $\mathsf c$ are constants.

Let
$$
 b_N^{(r)}(X_r,X_{N+1}):=a_{N+1}(X_{N+1})-a_r(X_r)+\left\{-\sum_{k=r}^N \E(\bbf_k(X_k,X_{k+1}))\right\},$$
$$ \mathfrak F:=\sum_{n=1}^\infty h_n(X_n,X_{n+1})\text{ , } \quad
 \mathfrak F_r:=\sum_{n=r}^\infty h_n(X_n,X_{n+1}). $$
By Theorem \ref{Proposition-Kolmogorov-Three-Series}, these sums converge  almost surely and in $L^2$.

As
we saw in  the proof of Theorem \ref{Theorem-Reducible-LLT}, if $\frac{z_N-\E(S_N^{(r)})}{\sqrt{V_N^{(r)}}}\to 0$ and $\Prob(X_n\in\mathfrak A_n)$ is bounded below, then for all $\phi\in C_c(\R)$ and $x_r\in\fS_r$,
\begin{equation}\label{hasata}
\lim_{N\to\infty}\sqrt{2\pi V_N}\E_{x_r}[\phi(S_N^{(r)}-b_N^{(r)}-z_N)]=\sum_{m\in\Z}\E[\phi(m+\mathfrak F_r)].
\end{equation}
We are going to choose $r, x_r, z_N, \mathfrak A_{N}$ and $\phi$ in such a way that  \eqref{hasata} is inconsistent with (b). Here are the choices:
\begin{enumerate}[$\circ$]
\item {\em Choice of $r$\/}: Since $\mathfrak F_r$ is the tail of a convergent series, $
\mathfrak F_r\xrightarrow[r\to\infty]{}0$  a.s., whence in probability.
Choose $r$ s.t. $\Prob(|\mathfrak F_r|\geq 0.2)<10^{-3}$.
\item {\em Choice of $x_r$\/}: $\Prob(|\mathfrak F_r|\geq 0.2)
=\int\Prob_x(|\mathfrak F_r|\geq 0.2)]\mu_r(dx)$. So there exist $x_r\in\fS_r$ s.t.
$$
\Prob_{x_r}(|\mathfrak F_r|>0.2)<10^{-3}.
$$
\item
{\em Choice of $\mathfrak A_N$\/}: By construction,  $\ess\sup|b_N^{(r)}|\leq 2K+1$. Divide $[-2K-1,2K+1]$ into equal intervals of length less than $10^{-2}$. At least one such interval, call it $J_N$, satisfies
$
\Prob(b_N^{(r)}\in J_N)\geq 10^{-2}(4K+2)^{-1}\text{ and }|J_N|\leq 10^{-2}.
$
Let $$\mathfrak A_{N+1}^{(r)}:=[b_N^{(r)}\in J_N].$$
\item $z_N:=-$center of $J_N$, then $z_N=O(1)$ and $\frac{z_N-\E(S_N^{(r)})}{\sqrt{V_N^{(r)}}}\to 0$.
\item  Choose a sequence $N_k\to\infty$ such that $z_{N_k}\to a$. Let  $I:=-a+[0.4,0.6]$.
\item Choose $\phi\in C_c(\R)$ s.t. $0\leq \phi\leq 1$,  $\phi|_{[0.3,0.7]}\equiv 1$ and $\phi|_{\R\setminus[0.2,0.8]}\equiv 0$.
\end{enumerate}
With these choices,
\begin{align*}
&\liminf_{N\to\infty}\sqrt{2\pi V_N^{(r)}}\Prob_{x_r}\bigl(S_N^{(r)}-z_N\in I\big|X_{N+1}\in \mathfrak A_{N+1}^{(r)}\bigr)\\
&\leq \lim_{k\to\infty}\sqrt{2\pi V_{N_k}}\Prob_{x_r}\bigl(S_{N_k}^{(r)}-z_{N_k}\in I\big|b_{N_k}^{(r)}(X_r,X_{N_k+1})\in J_{N_k}\bigr)\ \left(\text{because }\frac{V_N^{(r)}}{V_N}\to 1\right)\\
&\leq \lim_{k\to\infty}\sqrt{2\pi V_{N_k}}\Prob_{x_r}\bigl(S_{N_k}^{(r)}-b_{N_k}^{(r)}-z_{N_k}\in [0.3, 0.7]
\,\big|\,b_{N_k}^{(r)}\in J_{N_k}\bigr), \text{(because for $k\gg 1$}\\
&\hspace{0.5cm} I-b_{N_k}\subset I-J_{N_k}\subset I+
\left(z_{N_k}-
\tfrac{|J_{N_k}|}{2}, z_{N_k}+\tfrac{|J_{N_k}|}{2}\right)\subset
I+(a-0.1, a+0.1)\subset [0.3,0.7])\\
&\leq \lim_{k\to\infty}\sqrt{2\pi V_{N_k}}\E_{x_r}\bigl(\phi(S_{N_k}^{(r)}-b_{N_k}^{(r)}-z_{N_k})\,\big|\,b_{N_k}^{(r)}\in J_{N_k+1}\bigr)\\
&=\sum_{m\in\Z}\E_{x_r}[\phi(m+\mathfrak F_r)],\text{ by \eqref{hasata}}
\end{align*}
\begin{align*}
&\leq \sum_{m\in\Z}\Prob_{x_r}\bigl(m+\mathfrak F_r\in [0.2,0.8]\bigr)\leq \Prob_{x_r}\bigl(|\mathfrak F_r|\geq 0.2\bigr)<10^{ -3}<|I|.
\end{align*}
But this contradicts (b).

\medskip
\noindent
{\bf (a)$\Rightarrow$(c):\/} Suppose $(\mathsf X,\mathsf f)$ is non-lattice and irreducible, then
$(\mathsf X_r,\mathsf f_r)$ is non-lattice and irreducible for all $r$.
Fix $t>0$, $x_1\in\fS_1$, and  some sequence of measurable events $\mathfrak A_n\subset\fS_n$ such that $\Prob(X_n\in\mathfrak A_n)$ is bounded below. Let $\DS S_N^{(r)}:=\sum_{k=r}^N f_k(X_k,X_{k+1})$.

We show that for every continuous and periodic $\phi(x)$ with period $t$,
\begin{equation}\label{dist-mod-t}
\E_x(\phi(S_N^{(r)})|X_{N+1}\in\mathfrak A_{N+1})\xrightarrow[N\to\infty]{}\frac{1}{t}\int_{0}^t \phi(x) dx.
\end{equation}
It is enough to show \eqref{dist-mod-t} for trigonometric polynomials
$\DS \phi(u)=\sum_{|n|<L} c_n e^{2\pi i n u/t}$, as these are dense in $C[0,t]$. For such functions,
\begin{align*}
& \E_x(\phi(S_N^{(r)})|X_{N+1}\in\mathfrak A_{N+1})=\sum_{|n|<L} c_n\E_x(e^{2\pi i n S_N^{(r)}/t}|X_{N+1}\in \mathfrak A_{N+1})\\
&=c_0+\sum_{0<|n|<L}\Phi_N\left(x,\tfrac{2\pi n}{t}|\mathfrak A_{N+1}\right),
\text{where $\Phi_N$ are the characteristic functions of $(\mathsf X_r,\mathsf f_r)$}\\
&=c_0+o(1),\text{ by irreducibility and \eqref{CondBoundChar}.}
\end{align*}
Since $c_0=\frac{1}{t}\int_0^t \phi(u)du$, \eqref{dist-mod-t} follows. Standard approximation arguments show that \eqref{dist-mod-t} implies that
$$
\Prob_x(S_N^{(r)}\in (a,b)|X_{N+1}\in\fA_{N+1})\xrightarrow[N\to\infty]{}\frac{|a-b|}{t}\text{ for all intervals }(a,b).
$$

\medskip
\noindent
{\bf (c)$\Rightarrow$(a):} We need the following lemma.
\begin{lemma}
\label{LmUniformInt}
Fix a regular sequence of sets $\fA_N$, $x$, and $t>0$, and suppose that
$$
\Prob_x(S_N^{(r)}\in (a,b)+t\Z|X_{N+1}\in\fA_{N+1})\xrightarrow[N\to\infty]{}\frac{|a-b|}{t}$$
for all intervals $(a,b)$ s.t. $0<|a-b|<t$.
Then the convergence  is uniform in $(a,b)$.
\end{lemma}
\begin{proof}
Without loss of generality, $(a,b)\subset [0,t)$. We are asked to find for each $\epsilon>0$ an $N_0$ such that
$$
|\Prob_x(S_N^{(r)}\in (a,b)+t\Z|X_{N+1}\in\fA_{N+1})-\tfrac{|a-b|}{t}|<\epsilon\text{ for all $N>N_0$ and $a<b$.
}
$$
Choose $0<\delta<\min\{\frac{\epsilon}{5},1\}$, and divide $[0,t]$ into finitely many equal disjoint intervals $\{I_j\}$ with length $|I_j|<\delta$. Choose $N_0$ so that for all $N>N_0$, for all $I_j$,
\begin{equation}\label{I-j-Estimates}
\bigl|\Prob_x(S_N^{(r)}\in I_j+t\Z|X_{N+1}\in\fA_{N+1})-\frac{|I_j|}{t}\bigr|<\frac{\delta|I_j|}{t}.
\end{equation}
$I:=(a,b)$ can be approximated from within and from outside by finite (perhaps empty) unions of intervals $I_j$ whose total length differs from $|a-b|$ by no more than $2\delta$. Summing \eqref{I-j-Estimates} over these unions we see that for all $N>N_0$,
\begin{align*}
&\Prob_x(S_N^{(r)}\in I+t\Z|X_{N+1}\in\fA_{N+1})\leq \frac{|a-b|+2\delta}{t}+\frac{\delta(|a-b|+2\delta)}{t}\\
&\Prob_x(S_N^{(r)}\in I+t\Z|X_{N+1}\in\fA_{N+1})\geq \frac{|a-b|-2\delta}{t}-\frac{\delta|a-b|}{t}.
\end{align*}
By choice of $\delta$, $|\Prob_x(S_N^{(r)}\in I+t\Z|X_{N+1}\in\fA_{N+1})-\frac{|a-b|}{t}|<\epsilon$.\qed
\end{proof}

We can now prove that $(c)\Rightarrow (a)$.
Suppose $(\mathsf X_r,\mathsf f_r)$ has the  ``mixing uniform distribution mod $t$" property for all $r$ and $t$.  This property is invariant under centering, because of Lemma \ref{LmUniformInt}.
So we may assume without loss of generality that $\E[f_n(X_n,X_{n+1})]=0$ for all $n$.

 First we claim that $(\mathsf X,\mathsf f)$ is not center-tight. Otherwise there are constants $c_N$ and $M$ such that $\Prob(|S_N-c_N|>M)<0.1$ for all $N$. Take $t:=5M$ and $N_k\to\infty$ such that $c_{N_k}\xrightarrow[k\to\infty]{}c\mod t\Z$, then by the bounded convergence theorem and (c),
\begin{align*}
0.9&\leq \lim_{k\to\infty}\Prob\bigl(S_{N_k}\in [c-2M,c+2M]\bigr)
\leq \lim_{N\to\infty}\Prob\bigl(S_{N}\in [c-2M,c+2M]+t\Z\bigr)\\
&=\int_{\fS_1}\lim_{N\to\infty}\Prob_x\bigl(S_{N}\in [c-2M,c+2M]+t\Z| X_{N+1}\in\fS_{N+1}\bigr)\mu_1(dx)={ \frac{4M}{t}}=0.8,
\end{align*}
a contradiction. Thus $(\mathsf X,\mathsf f)$ is not center-tight and $V_N\to\infty$.

Assume by way of contradiction that $G_{ess}(\mathsf X,\mathsf f)\neq \R$, then $G_{ess}(\mathsf X,\mathsf f)=t\Z$ for some $t$, and $t\neq 0$ because $V_N\to\infty$. Without loss of generality $t=1$, otherwise we can rescale $\mathsf f$.  By the integer reduction lemma,  we can write
$$
f_n(x,y)+a_n(x)-a_{n+1}(y)=\bbf_n(x,y)+h_n(x,y)+c_n
$$
where $a_k,\bbf_k,h_k,c_k$ are uniformly bounded, $\bbf_n$ are integer valued, $h_n$ have summable variances, and $\E(h_n)=0$. Then $\mathfrak F:=\sum_{n\geq 1}h_n(X_n,X_{n+1})$ converges a.s., and
$\mathfrak F_r:=\sum_{n\geq r}h_n(X_n,X_{n+1})\xrightarrow[r\to\infty]{}0$ almost surely.

Working as in the proof of $(b)\Rightarrow(a)$, we construct  $x\in\fS_1$ and $r>1$ such that
$$
|\E_x(e^{2\pi i \mathfrak F_r})|>0.999.
$$
Next we construct a regular sequence of measurable sets $\fA_{N+1}$, and intervals $J_N$ with lengths $<0.0001$ and centers $z_N=O(1)$ such that
$
a_{N+1}(X_{N+1})-a_1(X_r)\in J_N,\text{ whenever }X_{N+1}\in\fA_{N+1}, X_r=x.
$

By Lemma \ref{CrLimChar-AP} with $s=0$, $m=1$, and $v_{N+1}\equiv 1$, there are $c(r,N)\in\R$ s.t.
$$
\E_x\bigl(e^{ 2\pi i(S_N^{(r)}+a(X_1)-a(X_{N+1})-z_N)}|X_{N+1}\in\fA_{N+1}\bigr)=
e^{2\pi i (c(r,N)-z_N)}\E_x(e^{2\pi i \mathfrak F_r})+o(1),
$$
as $N\to\infty$.
Since
$$\left\Vert\left(e^{ 2\pi i(S_N^{(r)}+a(X_r)-a(X_{N+1})-z_N)}-e^{ 2\pi i(S_N^{(r)})}\right)1_{[X_{N+1}\in\fA_{N+1}, X_r=x]}\right\Vert_\infty<0.1,$$
we find that  for all $N$ large enough,
$|\E_x\bigl(e^{ 2\pi i(S_N^{(r)})}|X_{N+1}\in\fA_{N+1}\bigr)|>\frac{1}{2}$.

But this is a contradiction, since  (c) implies that
$$
\E_x\bigl(e^{2\pi iS_N^{(r)}}|X_{N+1}\in\fA_{N+1}\bigr)\xrightarrow[N\to\infty]{}\frac{1}{2\pi}\int_0^{2\pi} e^{iu}du=0
.$$
So $G_{ess}(\mathsf X,\mathsf f)=\R$ and (a) is proved.
\qed

\subsection{Universal bounds for Markov chains}
\label{SSLLTMult}

\begin{lemma}\label{Lemma-HHH}
Suppose $\mathfrak F$ is a real random variable such that $0\leq \mathfrak F<\delta$ almost surely.
Then for every interval $(a,b)$ of length $L>\delta$,
$$
\left(1-\frac{\delta}{L}\right)
|a-b|<
\delta\sum_{m\in\Z}\E[1_{(a,b)}(m\delta+\mathfrak F)]< \left(1+\frac{\delta}{L}\right)|a-b|.
$$

\end{lemma}
\begin{proof}
Fix $k$ large, and divide $[0,\delta)$  into $k$ intervals $I_j:=\frac{j\delta}{k}+[0,\frac{\delta}{k})$. For each $j$,
\begin{align*}
&\delta\sum_{m\in\Z}\E[1_{(a,b)}(m\delta+\mathfrak F)|\mathfrak F\in I_j]\leq \delta\sum_{m\in\Z}\E[1_{(a+\frac{(j-1)\delta}{k},b+\frac{(j+1)\delta}{k})}(m\delta)|\mathfrak F\in I_j]\\
&=\delta\sum_{m\in\Z} 1_{(a+\frac{(j-1)\delta}{k},b+\frac{(j+1)\delta}{k})}(m\delta)\leq
|a-b|+1+\frac{2\delta}{k}\xrightarrow[k\to\infty]{}|a-b|+1.
\end{align*}
Multiplying by $\Prob[\mathfrak F\in I_j]$ and summing over $j=0,\ldots,k-1$ gives the bound
$
\delta\sum_{m\in\Z}\E[1_{(a,b)}(m\delta+\mathfrak F)]\leq |a-b|+\delta.
$
Similarly,  $
\delta\sum_{m\in\Z}\E[1_{(a,b)}(m\delta+\mathfrak F)]\geq |a-b|-\delta$. The lemma follows.\qed
\end{proof}

\medskip
\noindent
{\bf Proof of Theorem \ref{Theorem-Reducible-Universal-Bounds}:}
If $\delta(\mathsf f)=\infty$ then there is nothing to prove, and if $\delta(\mathsf f)=0$ then $(\mathsf X,\mathsf f)$ is non-lattice and irreducible, and the universal bounds follow from Theorem \ref{ThLLT-classic}. So assume
$\delta(\mathsf f)$ is finite and positive.

Suppose $\frac{z_N-\E(S_N)}{\sqrt{V_N}}\to z$.
Let $\mathfrak F$ and $b_N(X_1,X_N)$ be as in Theorem \ref{Theorem-Reducible-LLT}.

\medskip
\noindent
{\bf Upper bound \eqref{TRUB1}:} Fix $x\in\fS_1$, let $\delta:=\delta(\mathsf f)$ and suppose $(a,b)$ is an interval of length $L>\delta$.
We may assume without loss of generality that $a-10\delta, b+10\delta$
are not atoms of the distribution of  $\mathfrak F$ given $X_1=x$ (otherwise change $a,b$ a little).

Suppose $\frac{z_N-\E(S_N)}{\sqrt{V_N}}\to z$, and write
$
z_N=\ov{z}_N+\zeta_N,\ \ov{z}_N\in\delta\Z\ , |\zeta_N|\leq \delta.
$
Recall that by Theorem \ref{Theorem-Reducible-LLT}, $|b_N|\leq 9\delta$.
Therefore
$$
S_N-z_N\in (a,b)\Rightarrow S_N-\ov{z}_N-b_N\in (a-10\delta,b+10\delta)
$$
So
\begin{align}
&\quad \limsup_{N\to\infty}\sqrt{2\pi V_N}\Prob_x[S_N-z_N\in (a,b)] \notag \\
&\leq   \limsup_{N\to\infty}\sqrt{2\pi V_N}\Prob_x[S_N-\ov{z}_N-b_N\in (a-10\delta,b+10\delta)] \notag\\
&=e^{-z^2/2}\delta\sum_{m\in\Z}\E_x[1_{(a-10\delta,b+10\delta)}(m\delta+\mathfrak F)] \ \ \ \text{by Theorem \ref{Theorem-Reducible-LLT}} \notag\\
&\leq
 \left(1+\frac{\delta} {|a-b|+20\delta}\right)\; e^{-z^2/2}(|a-b|+20\delta)\text{ by Lemma \ref{Lemma-HHH}} \label{sando}\\
&\leq  \left(|a-b|+21\delta\right)e^{-z^2/2}\leq \left(1+\frac{21\delta}{L}\right) e^{-z^2/2}|a-b|.\notag
\end{align}

\medskip
\noindent
{\bf Lower bound \eqref{TRUB2}:}
 Fix $x\in\fS_1$ and  an interval $(a,b)$ with length bigger than some $L>\delta(\mathsf f)$.
Recall that $|b_N|$ are uniformly bounded.
Choose some $\brK$ so that
$
\Prob[|b_N|\leq \brK]=1
$
and fix $x\in \fS_1$ s.t. $\Prob_x[\sup|b_N|\leq \brK]=1$.

Next, divide $[-\brK,\brK]$ into $k$ disjoint intervals $I_{j,N}$ of equal length $\frac{2\brK}{k}$,
with $k$  large. For each $N$,
$
\DS\sum_{\Prob_x[b_N\in I_{j,N}]\geq k^{-2}}\Prob_x[b_N\in I_{j,N}]\geq 1-\frac{1}{k},
$
because to complete the left-hand-side to one we need to add the probabilities of $[b_N\in I_{j,N}]$ for the $j$ s.t.  $\Prob_x[b_N\in I_{j,N}]<k^{-2}$, and there are at most $k$ such events.

Therefore, we can divide $\{I_{j,N}\}$ into two groups of size at most $k$: The first contains the $I_{j,N}$ with $\Prob_x[b_N\in I_{j,N}]\geq k^{-2}$, and the second corresponds to events with  total probability less than $\frac{1}{k}$ (conditioned on $X_1=x$).

Re-index the intervals in the first group (perhaps with repetitions) in such a way that it takes the form $I_{j,N}$ $(j=1,\ldots,k)$ for all $N$. Then
for each $j$,  $\fA_{j,N}:=[b_N\in I_{j,N}, X_1=x]$ is a regular sequence of events.

Let $\beta_{j,N}:=$ center of $I_{j,N}$ and set $
z_{j,N}:=z_N-\beta_{j,N}
$.
Every sequence has a subsequence s.t. $z_{j,N}$
converges mod $\delta(\mathsf f)$. We will henceforth assume that $z_{j,N}=\ov{z}_{j,N}+\zeta_0+\zeta_{j,N}$ where $\ov{z}_{j,N}\in\delta(\mathsf f)\Z$ and $|\zeta_{j,N}|<\frac{K}{k}$, and $|\zeta_0|<\delta(\mathsf f)$ is fixed.

Recall that $|I_{j,N}|=\frac{2\brK}{k}$. Conditioned on $\fA_{j,N}$, $b_N=\beta_{j,N}\pm\frac{2\brK}{k}$, therefore $\ov{z}_{j,N}+\zeta_0+b_N=z_N\pm\frac{3K}{k}$, whence
$$
S_N-\ov{z}_{j,N}-b_N\in \left(a-\zeta_0+\frac{3\brK}{k}, b-\zeta_0-\frac{3\brK}{k}\right)
\Rightarrow S_N-z_N\in (a,b).
$$
There is no loss of generality in assuming that the endpoints of this interval are not atoms of the distribution of  $\mathfrak F$ given $X_1=x$, otherwise perturb $K$ a little. Since $\fA_{j,N}$ is a regular sequence, we have by Theorem \ref{Theorem-Reducible-LLT} part (2) and the lemma that
\begin{align}
&\liminf\limits_{N\to\infty}\sqrt{2\pi V_N}\Prob_x(S_N-z_N\in (a,b)|\fA_{j,N})\notag\\
&\geq \liminf\limits_{N\to\infty}\sqrt{2\pi V_N}\Prob_x(S_N-\ov{z}_{j,N}-b_N
\in (a-\zeta_0+\tfrac{3\brK}{k},b-\zeta_0-\tfrac{3\brK}{k})|\fA_{j,N})\notag\\
&=\delta(\mathsf f)e^{-z^2/2}\sum_{m\in\Z}\E_x[1_{(a-\zeta_0+\frac{3\brK}{k},b-\zeta_0-\frac{3\brK}{k})}(m\delta(\mathsf f)+\mathfrak F)]\notag\\&
\geq \left(1-\frac{\delta}{L}\right)\bigl(|a-b|-\tfrac{6\brK}{k}\bigr)e^{-z^2/2}. \label{statin}
\end{align}
We now multiply these bounds by $\Prob_x[\fA_{j,N}]$ and sum over $j$. This gives
\begin{align*}
&\liminf_{N\to\infty}\sqrt{2\pi V_N}\Prob_x\left(\left[S_N-z_N\in (a,b)\right]\bigcap
\bigcup_{j=1}^k \fA_{j,N}\right)\\
&\geq
\left(1-\frac{\delta}{L}\right)\left(|a-b|-\tfrac{6\brK}{k}\right)e^{-z^2/2}\left(1-\frac{1}{k}\right).
\end{align*}
Passing to the limit $k\to\infty$, we obtain
$$
\liminf\limits_{N\to\infty}\sqrt{2\pi V_N}\Prob_x\left(\left[S_N-z_N\in (a,b)\right]\right)
\geq \left(1-\frac{\delta}{L}\right)e^{-z^2/2}|a-b|
,$$
and the lower bound is proved.

 To prove the last statement of the theorem let
 $\cA_x$
 be the positive functional on $C_c(\R)$ defined by \eqref{LocLimRed}, and let
  $\mu_{\cA_x}$ be the Radon measure on $\R$ s.t. $\mu_{\cA_x}(\phi)=\cA_x[\phi]$ for $\phi\in C_c(\R)$.

The inequalities  \eqref{sando}, \eqref{statin} can be used to see that
$$
(1-\delta L^{-1})(|a-b|-O(1))\leq \mu_{\cA_x}(a,b)\leq (1+21\delta L^{-1})(|a-b|+O(1)),
$$
whence $\DS \lim_{L\to \infty} \frac{\mu_{\cA}[0, L]}{L}=1. $ Since $\mu_{\cA_x}$ is clearly invariant under translation by $\delta(\mathsf f)$, it must be the case that for each $a$, $\mu_{\cA}[a, a+\delta(\mathsf{f}))=\delta(\mathsf{f})$, whence
\begin{equation}
\forall k\in \N \;\;
\mu_{\cA}([a, a+\delta(k\mathsf{f})))=k\delta(\mathsf{f}).
\label{LLTExactStep}
\end{equation}
Given an interval $(a,b)$ of length $L$ with $k\delta(\mathsf{f})<L<(k+1)\delta(\mathsf{f})$
take two intervals $I^-, I^+$  such that
$$ I^-\subset (a, b)\subset I^+, \quad \mu_{\cA}(\partial I^-)=\mu_{\cA}(\partial I^+)=0, \quad
|I^-|=k\delta(\mathsf{f}), \quad
|I^+|=(k+1)\delta(\mathsf{f}).
$$
Next let $\phi^-, \phi^+$ be continuous functions with compact support such that
$$ 1_{I^-}<\phi^-<1_{[a,b]}<\phi^+<1_{I^+}. $$
Then for large $N$,
$ \sqrt{V_N} \Prob(S_N-z_N\in (a,b))$ is sandwiched between
$\cA(\phi^-)$ and $\cA(\phi^+)$ which in turn is sandwiched between
$$\mu_{\cA}(I^-)={k\delta(\mathsf{f})}\quad\text{and}\quad
\mu_{\cA}(I^+)=(k+1)\delta(\mathsf{f})$$
where the equalities rely on
\eqref{LLTExactStep}. The proof of the theorem is complete.
\hfill$\Box$

\subsection{Universal bounds for Markov arrays}
\label{SSUniversalMA}
Next, we give a different proof of universal lower and upper bounds, which does not rely on Theorem \ref{Theorem-Reducible-LLT}, and which also applies to arrays
and to arbitrary initial distributions.

\begin{theorem}\label{Theorem-Other-Universal-Bounds}
Let $\mathsf X$ be a uniformly elliptic Markov array, and $\mathsf f$ an a.s. uniformly bounded additive functional which is stably hereditary and not center tight. For every $\epsilon>0$ there is $N_\epsilon>0$ as follows. Suppose $\frac{z_N-\E(S_N)}{\sqrt{V_N}}\xrightarrow[N\to\infty]{}z\in\R$, and $|a-b|>2\delta(\mathsf f)+\epsilon$, then for all $N>N_\epsilon$,
$$
\frac{1}{3}\left(\frac{e^{-z^2/2}|a-b|}{\sqrt{2\pi V_N}}\right)\leq \Prob(S_N-z_N\in (a,b))\leq 3\left(\frac{e^{-z^2/2}|a-b|}{\sqrt{2\pi V_N}}\right).
$$
\end{theorem}

Recall that by our conventions, the {\bf Fourier transform}\index{Fourier transform} of an $L^1$ function $\gamma:\R\to\R$ is
$
\wh{\gamma}(x)=\int_{-\infty}^\infty e^{-itx}\gamma(t)dt.
$
Fix some $b>0$, and define the Fourier pair
$$
\psi_b(t):=\frac{\pi}{4b}1_{[-b,b]}(t)\ , \ \wh{\psi}_b(x)=\frac{\pi}{2b}\left(\frac{\sin(bx)}{x}\right).
$$

\begin{lemma}\label{Lemma-Psi-Graph}
 $1\leq\wh{\psi_b}(x)\leq \frac{\pi}{2}$ for $|x|\leq \frac{\pi}{2b}$;
  and $|\wh{\psi_b}(x)|<1$ for $|x|>\frac{\pi}{2b}$.
\end{lemma}
\begin{proof}
The function $\wh{\psi_b}(x)$ is even, with  zeroes at $z_n=\pi n/b$, $n\in\mathbb Z\setminus\{0\}$.
The critical points are $c_0=0$ and $\pm c_n$ where $n\geq 1$ and
$$
c_n:=\text{the unique solution of $\tan(bc_n)=bc_n$ in $\left(z_n,z_n+\frac{\pi}{2b}\right)$}.
$$
It is easy to see that $c_n=z_n+\frac{\pi}{2b}-o(1)$ as $n\to\infty$, and that
$$
\mathrm{sgn}[\wh{\psi_b}(c_n)]=(-1)^n\ , |\wh{\psi_b}(c_n)|\leq \frac{1}{2n}\ , \ \wh{\psi_b}(c_n)\sim \frac{(-1)^n}{2n}\text{ as }n\to\infty.
$$
So $\wh{\psi_b}$ attains global maximum $\wh{\psi_b}(0)=\frac{\pi}{2}$ at $c_0$, and  $|\wh{\psi_b}(t)|\leq \frac{1}{2n}$ everywhere on $[\pi n/b, \pi(n+1)/b]$.

In particular, $|\wh{\psi_b}(t)|<1/2$ for $|t|\geq \pi/b$.
On $(0,\pi/b)$ the  function is decreasing from its global maximum  $\wh{\psi_b}(0)=\frac{\pi}{2}$ to $\wh{\psi_b}(\frac{\pi}{b})=0$, passing through $\wh{\psi_b}(\frac{\pi}{2b})=1$. It follows that $1\leq \wh{\psi_b}(t)\leq \frac{\pi}{2}$ on $(0,\frac{\pi}{2b})$ and $|\wh{\psi_b}(t)|<1$ for $t>\frac{\pi}{2b}$. The lemma follows, because $\wh{\psi_b}(-t)=\wh{\psi_b}(t)$.\qed
\end{proof}

\begin{lemma}\label{Lemma-Gamma-Bounds}
There exist two continuous functions $\gamma_1(x),\gamma_2(x)$ s.t.
$\supp(\gamma_i)\subset [-2,2]$;
 ${\gamma}_1(0)>\frac{1}{3}$;  ${\gamma}_2(0)<3$; and
 $
 \wh{\gamma}_1(x)\leq 1_{[-\pi,\pi]}(x)\leq \wh{\gamma}_2(x)\ \ \ (x\in\R).
 $
\end{lemma}
\begin{proof}
Throughout this proof, $\psi^{\ast n}:=\psi\ast\cdots\ast\psi$ ($n$ times),  where $\ast$ denotes the convolution.
Let $\gamma_1(t):=\frac{1}{4}[\psi_{\frac{1}{2}}^{\ast 4}(t)-\psi_{\frac{1}{2}}^{\ast 2}(t)]$.
Then
$
\wh{\gamma}_1(x)=\frac{1}{4}[\wh{\psi}_{\frac{1}{2}}(x)^4-\wh{\psi}_{\frac{1}{2}}(x)^2].
$
By Lemma \ref{Lemma-Psi-Graph}, $1\leq \wh{\psi}_{\frac{1}{2}}\leq \frac{\pi}{2}$ on $[-\pi,\pi]$ and $|\wh{\psi}_{\frac{1}{2}}|<1$ outside $[-\pi,\pi]$. So
\begin{align*}
&\max_{|x|\leq\pi}\wh{\gamma}_1(x)\leq \max_{1\leq y\leq \frac{\pi}{2}}\frac{1}{4}(y^4-y^2)=\frac{1}{4}\left[\left(\frac{\pi}{2}\right)^4-\left(\frac{\pi}{2}\right)^2\right]<1,\\
&\max_{|x|\geq \pi}\wh{\gamma}_1(x)\leq \max_{|y|\leq 1}\frac{1}{4}(y^4-y^2)=0.
\end{align*}
So $\wh{\gamma}_1(x)\leq 1_{[-\pi,\pi]}(x)$ for all $x\in\R$.

It is obvious from the definition of the convolution that
$$\supp(\gamma_1)=\{x+y+z+w: x,y,z,w\in [-\frac{1}{2},\frac{1}{2}]\}=[-2,2].$$
Here is the calculation showing that $\gamma_1(0)>\frac{1}{3}$:
\begin{align*}
(\psi_b^{\ast 2})(t)&=\frac{\pi^2}{16 b^2} (1_{[-b,b]}\ast 1_{[-b,b]})(t)=\frac{\pi^2}{16 b^2} 1_{[-2b,2b]}(t)(2b-|t|)\\
(\psi_b^{\ast 4})(0)&=(\psi_b^{\ast 2}\ast \psi_b^{\ast 2})(0)\\
&=\frac{\pi^4}{256 b^4}\int_{-\infty}^\infty
1_{[-2b,2b]}(t)(2b-|t|)1_{[-2b,2b]}(-t)(2b-|-t|)dt\\
&=\frac{\pi^4}{256 b^4}\int_{-2b}^{2b}
(2b-|t|)^2 dt=\frac{\pi^4}{128 b^4}\int_{0}^{2b}
(2b-t)^2 dt=
\frac{\pi^4}{128 b^4}\cdot \frac{(2b)^3}{3}
=\frac{\pi^4}{48b}.
\end{align*}
So $\psi_{\frac{1}{2}}^{\ast 4}(0)=\frac{\pi^4}{24}$, $\psi_{\frac{1}{2}}^{\ast 2}(0)=\frac{\pi^2}{4}$, and
$\gamma_1(0)=\frac{1}{4}(\frac{\pi^4}{24}-\frac{\pi^2}{4})>\frac{1}{3}
$.

Next we set
$
\gamma_2(t):=(\psi_{\frac{1}{2}}\ast\psi_{\frac{1}{2}})(t)\equiv \frac{\pi^2}{4}1_{[-1,1]}(t)(1-|t|).
$
Then $\supp(\gamma_2)=[-1,1]$ and  $\gamma_2(0)=\frac{\pi^2}{4}<3$. Finally,  $\wh{\gamma}_2\geq 1_{[-\pi,\pi]}(x)$, because by Lemma \ref{Lemma-Psi-Graph},
\begin{enumerate}[$\circ$]
\item $\wh{\gamma}_2(t)=(\wh{\psi}_{\frac{1}{2}})^2(x)\geq 1$ for all $|x|\leq \frac{\pi}{2\cdot\frac{1}{2}}=\pi$, and
\item $
\wh{\gamma}_2(t)=(\wh{\psi}_{\frac{1}{2}})^2(x)\geq 0$  for all $|x|\geq \pi$.\qed
\end{enumerate}
\end{proof}

\noindent
{\bf Proof of Theorem \ref{Theorem-Other-Universal-Bounds}.}
If $G_{ess}(\mathsf X,\mathsf f)=\R$ then the theorem follows from the LLT in the irreducible case. Otherwise (since $\mathsf f$ is not center-tight), $G_{ess}(\mathsf X,\mathsf f)=t\Z$ for some $t>0$, and there is no loss of generality in assuming that  $G_{ess}(\mathsf X,\mathsf f)=\Z$.

Henceforth we assume that $G_{ess}(\mathsf X,\mathsf f)=\Z$.  In this case our interval $I:=[a,b]$ has length bigger than $2$. Notice that we can always center $I$ by modifying $z_N$ by a constant. So we may take our interval to be of the form
$$
I=[-a,a],\text{ with }a>1.
$$

Let $\gamma_i(t)$ be the functions constructed in Lemma \ref{Lemma-Gamma-Bounds}, then
$$
\wh{\gamma}_1\left(\frac{\pi t}{a}\right)\leq 1_{I}(t)\leq \wh{\gamma}_2\left(\frac{\pi t}{a}\right).
$$
Therefore,
for every choice of $x^{(N)}_1\in\fS^{(N)}_1$ $(N\geq 1)$,
\begin{align*}
&
\Prob_{x^{(N)}_1}(S_N-z_N\in I)=\E_{x^{(N)}_1}[1_I(S_N-z_N)]\geq
\E_{x^{(N)}_1}\left[\wh{\gamma}_1\left(\frac{\pi(S_N-z_N)}{a}\right)\right]\\
&=\E_{x^{(N)}_1}\left[\int_{-\infty}^\infty e^{-i\frac{\pi t}{a}(S_N-z_N)}\gamma_1(t)dt\right]=\int_{-\infty}^\infty \E_{x^{(N)}_1}(e^{-i\frac{\pi t}{a}(S_N-z_N)})\gamma_1(t)dt.
\end{align*}
Recalling that $\supp(\gamma_1)\subset [-2,2]$, and substituting $t=a\xi/\pi$, we obtain
\begin{equation}\label{cardi}
\Prob_{x^{(N)}_1}(S_N-z_N\in I)\geq \frac{|I|}{2\pi}\int_{-2\pi/a}^{2\pi/a}
\E_{x^{(N)}_1}(e^{-i\xi(S_N-z_N)})\gamma_1(\tfrac{a\xi}{\pi})d\xi.
\end{equation}
Similarly, we have
\begin{equation}\label{lloc}
\Prob_{x^{(N)}_1}(S_N-z_N\in I)\leq \frac{|I|}{2\pi}\int_{-2\pi/a}^{2\pi/a}
\E_{x^{(N)}_1}(e^{-i\xi(S_N-z_N)})\gamma_2(\tfrac{a\xi}{\pi})d\xi.
\end{equation}
Next we claim that under the assumptions of
Theorem \ref{Theorem-Other-Universal-Bounds}:

\begin{lemma}
\label{LmHGammaI}
If  $G_{ess}(\mathsf X,\mathsf f)=\Z$ and
 $\frac{z_N-\E(S_N)}{\sqrt{V_N}}\xrightarrow[N\to\infty]{}z\in\R$, then for every $a>1$
$$
\sqrt{V_N}\int\limits_{-2\pi/a}^{2\pi/a} \E_{x^{(N)}_1}(e^{-i\xi(S_N-z_N)})
\gamma_i\left(\tfrac{a\xi}{\pi}\right)d\xi\xrightarrow[N\to\infty]{}
\sqrt{2\pi}e^{-\frac{1}{2}z^2}\gamma_i(0)
$$
and the convergence is  uniform in $a$ on compact subsets of $\R\setminus [-1,1]$.
\end{lemma}

\medskip
\noindent
{\em Proof.\/}
In what follows we fix $i\in\{1,2\}$ and let
$\DS \gamma(\xi):=\gamma_i\left(\tfrac{a\xi}{\pi}\right)$.
Divide $[-\frac{2\pi}{a},\frac{2\pi}{a}]$ into segments $I_j$ of length at most $\wt{\delta}$, where $\wt{\delta}$ is given by Lemma \ref{CrLimChar-AP}, making sure that $I_0$ is centered at zero. Let
$$
J_{j,N}:=\int_{I_j}\E_{x^{(N)}_1}(e^{-i\xi(S_N-z_N)})\gamma(\xi)d\xi.
$$

\medskip
\noindent
{\sc Claim 1.\/} $\sqrt{V_N} J_{0,N}\xrightarrow[N\to\infty]{}\sqrt{2\pi} e^{-z^2/2}\gamma(0)$.

\medskip
\noindent
{\em Proof.\/} The proof  is similar to the proof of {\eqref{MainInt}.}

Applying Corollary \ref{CrChar-DMax} to the interval $I_{0}$,
and noting that $A_N(I_0)=0$ and $\wt{\xi}_N=0$ we find that
$$
|\E_{x^{(N)}_1}(e^{-i\xi(S_N-z_N)})|\leq \wt{C}\exp(-\wh{\eps}\xi^2 V_N).
$$
So for every $R>1$,
$$
\sqrt{V_N}\int_{\xi\in I_0:|\xi|>\frac{R}{\sqrt{V_N}}}\; \E_{x^{(N)}_1}(e^{-i\xi(S_N-z_N)})\gamma(\xi)d\xi=O(e^{-\wh{\eps}R^2})
.$$
Similarly,  for all $N$ large enough
\begin{align*}
&\sqrt{V_N}\int_{\xi\in I_0:|\xi|\leq \frac{R}{\sqrt{V_N}}}\E_{x^{(N)}_1}(e^{-i\xi(S_N-z_N)})\gamma(\xi)d\xi=\int_{-R}^R
\E_{x^{(N)}_1}(e^{-i\eta(\frac{S_N-z_N}{\sqrt{V_N}})})\gamma(\tfrac{\eta}{\sqrt{V_N}})d\eta\\
&=\int_{-R}^R
\E_{x^{(N)}_1}(e^{-i\eta(\frac{S_N-\E(S_N)}{\sqrt{V_N}})})e^{i\eta(\frac{z_N-\E(S_N)}{\sqrt{V_N}})}\gamma_i(\tfrac{a\eta}{\pi\sqrt{V_N}})d\eta\\
&\overset{!}{=}\int_{-R}^R e^{-\frac{1}{2}\eta^2+i\eta z}\gamma(0)d\eta+o_{N\to\infty}(1)\text{ uniformly on compact sets of $a$ }\\
&{=}\sqrt{2\pi} e^{-\frac{1}{2}z^2}\gamma(0)+o_{R\to\infty}(1)+o_{N\to\infty}(1),
\end{align*}
where $\overset{!}{=}$ is a consequence of Dobrushin's CLT and the bounded convergence theorem. (When applying Dobrushin's Theorem it is useful to recall that by the exponential mixing of uniformly elliptic arrays, $|\E(S_N)-\E_{x^{(N)}_1}(S_N)|=O(1)$, therefore the condition $\frac{z-\E(S_N)}{\sqrt{V_N}}\to z$ is equivalent to the condition $(z-\E_{x^{(N)}_1}(S_N))/\sqrt{V_N}\to z$.)
In summary,
$$
\sqrt{V_N}J_{0,N}=\sqrt{2\pi} e^{-\frac{1}{2}z^2}\gamma(0)+o_{R\to\infty}(1)+o_{N\to\infty}(1).
$$
Fixing $R$, we see that $\limsup \sqrt{V_N}J_{0,N}$ and
$\liminf\sqrt{V_N}J_{0,N}$ are both equal to
$$\sqrt{2\pi} e^{-\frac{1}{2}z^2}\gamma(0)+o_{R\to\infty}(1).$$
Passing to the limit $R\to\infty$ gives us that the limit exists and is equal to $\sqrt{2\pi} e^{-\frac{1}{2}z^2}\gamma(0)$.

It is not difficult to see that the convergence is uniform on compact subsets of $a$.

\medskip
\noindent
{\sc Claim 2.\/} $\sqrt{V_N} J_{j,N}\xrightarrow[N\to\infty]{}0$ for every $j\neq 0$.

\medskip
\noindent
{\em Proof.\/} Since $G_{ess}(\mathsf f)=\Z$, the co-range is $H(\mathsf f)=2\pi\Z$. So
$$I_j\subset [-\tfrac{2\pi}{a},\tfrac{2\pi}{a}]\setminus \mathrm{int}(I_0)\subset\text{ a compact subset of }\R\setminus H(\mathsf f).
$$
This implies by the stable hereditary property of $\mathsf f$ that
$$
D_N(\xi)\xrightarrow[N\to\infty]{}\infty\text{ uniformly on }I_j,
$$
whence by \eqref{BoundChar},
$|\E_{x^{(N)}_1}(e^{-i\eta(S_N-z_N)})|\xrightarrow[N\to\infty]{} 0$ uniformly on $I_j$.

Let $A_{j,N}:= -\log \{\sup |\E_{x^{(N)}_1}(e^{-i\xi(S_N-z_N)})|:(x,\xi)\in \mathfrak S^{(N)}_1\times I_j\}$, then
$A_{j,N}\xrightarrow[N\to\infty]{}\infty$, and this divergence is uniform for $a$ ranging over compact subsets of $\R\setminus [-1,1]$.

From this point onward,  the proof of the claim is identical to the  proof of \eqref{SmallL1Norm}. We omit the details.

\smallskip
The Lemma follows by summing over all subintervals $I_j$ in $[-\frac{2\pi}{a},\frac{2\pi}{a}]$, and noting that the number of these intervals is uniformly bounded
$\left(\text{by }1+\frac{4\pi}{\wt{\delta}}\right).$\hfill$\Box$

\medskip
We now return to the proof of theorem.
Lemma \ref{LmHGammaI}, \eqref{cardi}, \eqref{lloc}, and the inequalities $\gamma_1(0)>\frac{1}{3}$ and $\gamma_2(0)<3$ imply that for every choice of $\{x^{(N)}_1\}_{N\geq 1}$, for all $N$ sufficiently large
\begin{equation}\label{brilinta}
\frac{1}{3}\cdot \frac{|I|}{\sqrt{2\pi V_N}}e^{-z^2/2}\leq \Prob_{x^{(N)}_1}(S_N-z_N\in I)\leq 3\cdot \frac{|I|}{\sqrt{2\pi V_N}}e^{-z^2/2}.
\end{equation}
This estimate is uniform in $\{x^{(N)}_1\}_{N\geq 1}$: There is an $N_0$ such that \eqref{brilinta} holds for all $N\geq N_0$ and {\em for all} choices of $\{x^{(N)}_1\}_{N\geq 1}$.  Otherwise, there exist $N_k\to\infty$ and $x^{(N_k)}_1\in\fS^{(N_k)}_1$ which violate \eqref{brilinta}. But then \eqref{brilinta} fails for any choice of  $x^{(N)}_1$ which contains  $x^{(N_k)}_1$ as a subsequence, whereas \eqref{brilinta} holds for all possible choices.

Since \eqref{brilinta} holds uniformly in $\{x^{(N_k)}_1\}_{N\geq 1}$, we can integrate and deduce that
for all $N$ sufficiently large
\begin{equation}\label{brilinta}
\frac{1}{3}\cdot \frac{|I|}{\sqrt{2\pi V_N}}e^{-z^2/2}\leq \Prob(S_N-z_N\in I)\leq 3\cdot \frac{|I|}{\sqrt{2\pi V_N}}e^{-z^2/2}
\end{equation}
for {\em any} initial distributions $\mu^{(N)}_1(dx_1^{(N)})$ on $\fS^{(N)}_1$. \hfill$\Box$

 We end this section by recording a useful consequence of the previous proof: The upper
bound in Theorem \ref{Theorem-Other-Universal-Bounds} does not require any information about the arithmetic properties of $\mathsf f$.

\begin{lemma}
\label{LmAntiCon}
For each $K, \eps_0$ and $\ell$ there is a constant $C^*=C^*(K, \eps_0, \ell)$ s.t. if $\mathsf f$
is an additive functional of a uniformly elliptic Markov chain with ellipticity constant $\eps_0$, and if
 $|\mathsf f|\leq K$, then for every $x\in \fS_1$, $N\geq 1$, and  for each interval $J$ of length $\ell$,
$$\Prob_x\left(S_N\in J\right)\leq \frac{C^*}{\sqrt{V_N}}. $$
\end{lemma}

\begin{proof}
It suffices to prove the result for $\ell=4$ since longer intervals could be covered by a finite number
of intervals of length 4. Thus $J=z_N+I$ with $I=[-2, 2].$ Applying \eqref{lloc} with $a=2$ we get
$$
\Prob_{x}(S_N\in J)\leq \hC
\int_{-\pi}^{\pi}
\left|\Phi_N(x, -\xi)\right|d\xi.
$$
where $\DS \hC= \frac{2}{\pi} \|\gamma_2\|_\infty. $
Dividing $[-\pi, \pi]$ into finitely many subintervals of  length  $\tdelta/2$ where
$\tdelta$ comes from Lemma \ref{Lm2P},
and applying \eqref{Jj-Aj} on each subinterval we obtain the result.
$\hfill\Box$
\end{proof}

\section{Notes and references}
Dolgopyat proved a version of Theorem \ref{Theorem-Reducible-LLT} for sums of independent random variables.
The connection between the LLT and uniform distribution modulo $t$ was considered for sums of independent random variables by Prohorov \cite{Prohorov},
Rozanov \cite{Rozanov},
 and Gamkrelidze \cite{Gamkrelidze}.

 The question of estimating $\Prob[S_N-z_N\in (a,b)]$ is related to the study of the rate of
convergence in the CLT. In particular,
a Berry-Esseen type result on the rate of convergence in the CLT would certainly imply that $\exists M$ s.t. for all $|a-b|>M$, if $\frac{z_N-\E(S_N)}{\sqrt{V_N}}\to z$, then  for all $N$ large enough,
$
\Prob[S_N-z_N\in (a,b)]$ equals $ \frac{e^{-z^2/2}|a-b|}{\sqrt{2\pi V_N}}
$ up to bounded multiplicative error. Such results were shown to us by Y. Hafouta.
The Berry-Esseen approach has the advantage of gives information on the time $N$ when the universal estimates kick in, but has the disadvantage that it only applies to very large intervals (how large depends on the growth of the third moment of $S_N$). By contrast, the results of this chapter apply to intervals of length $>\delta(\mathsf f)$, which is optimal, but do not say on how large $N$ should be for the estimates to work.

 Lemma \ref{LmAntiCon} for the sums of independent random variables appears in \cite[Section III.1]{Petrov-Book}. The proof in the Markov case is essentially the same.

\chapter{Local limit theorems for large and moderate deviations}\label{Chapter-LDP}

\medskip
{\em In this chapter we  prove the local limit theorem in the regimes of moderate and large deviations. In these cases the asymptotic behavior of $\Prob(S_N-z_N\in (a,b))$ is determined by the ``rate functions," the Legendre transforms of the log-moment generating functions of $S_N$.  }

\section{The moderate deviations and large deviations regimes}\label{Section-six-point-one}
Suppose $\mathsf f$ is an  irreducible, a.s. uniformly bounded, additive functional on a uniformly elliptic Markov {\em chain} $\mathsf X$, with algebraic range $\R$ or $t\Z$ with $t>0$. Let
$$
S_N=f_1(X_1,X_2)+\cdots+f_N(X_N,X_{N+1})\ , \ V_N:=\Var(S_N).
$$
In the previous chapters, we analyzed  $\Prob(S_N-z_N\in(a,b))$ as $N\to\infty$, in the  {\bf regime of local deviations}\index{regime!local deviations}, $\frac{z_N-\E(S_N)}{\sqrt{\Var(S_N)}}\to const$.
In this chapter we ask what happens \index{regime!of large deviations} when
$\frac{z_N-\E(S_N)}{\sqrt{\Var(S_N)}}\to\infty$.

Usually in the literature the large deviations regime is defined by the condition
$|z_N-\E(S_N)|\geq \breps\Var(S_N)$for some fixed $\breps>0$. However, to get meaningful results we need
to assume some upper bounds $|z_N-\E(S_N)|$ as well.
\index{regime! of large deviations}\index{large deviations}
We will study the following  regimes:
\begin{enumerate}[(1)]
\item {\bf Moderate deviations:} $\frac{z_N-\E(S_N)}{\sqrt{\Var(S_N)}}\to\infty$ and $z_N-\E(S_N)=o(\Var(S_N)),$

\smallskip
\item {\bf Large deviations:}
$\frac{z_N-\E(S_N)}{\sqrt{\Var(S_N)}}\to\infty$ and $|z_N-\E(S_N)|\leq \epsilon\Var(S_N)$ for some $\epsilon>0$ ``small
 enough."
\end{enumerate}
In some cases we can take $\eps=\infty$, see e.g. \S \ref{SSSmall}, but in others $\eps$ must really be finite, see Example~\ref{ExCoreDrop}.
To see why it is forced on us, let us consider a few examples of  what might go wrong when $|z_N-\E(S_N)|/\Var(S_N)$ is ``too big."

If $\frac{z_N-\E(S_N)}{V_N}$ grows too fast, e.g. if $\frac{z_N-\E(S_N)}{V_N}> \frac{2\ess\sup |S_N|}{V_N}$, then the probabilities $
\Prob[S_N-z_N\in (0,\infty)]
$ are all equal to zero, and our problem is vacuous.
A more subtle but related issue arises when $\frac{z_N-\E(S_N)}{V_N}$ falls  at the boundary of the domain of the Legendre transforms of $t\mapsto\frac{1}{V_N}\log \E(e^{t (S_N-\E(S_N))})$. Why this matters will  be clear once we explain the strategy of our  proofs (see the end of \S\ref{section-strategy-LD} and \S\ref{Section-Threshold}). At this point we can only present an example:

\begin{example}\label{Example-Edge}
If $\frac{z_N-\E(S_N)}{V_N}$ falls  near the boundary of the domain of the Legendre transforms of $t\mapsto\frac{1}{V_N}\log \E(e^{t (S_N-\E(S_N))})$, then the behavior of $\Prob[S_N-z_N\in (a,b)]$ may depend not just on  $\lim\limits_{N\to\infty}\frac{z_N-\E(S_N)}{V_N}$ but also on  $z_N$ itself.
\end{example}
\begin{proof}
Let $S_N:=X_1+\cdots +X_N$ where $X_i$ are identically distributed independent random variables equal to $-1,0,1$ with equal probabilities.
Here $\E(S_N)=0$, $V_N=2N/3$, the Legendre transforms of the log-moment generating functions have domains $(-\frac{3}{2},\frac{3}{2})$, and the classical theory of large deviations says  that if $z\in (-\frac{3}{2},\frac{3}{2})$, then $\DS \lim_{z_N/V_N\to z}\frac{1}{V_N}\log \Prob[S_N-z_N>0]$ exists and is finite.
But no such conclusion holds when $z= \frac{3}{2}$:
\begin{enumerate}[$\circ$]
\item If $z_N=N$, then $[S_N-z_N>0]=\emptyset$ and
$
\frac{1}{V_N}\log \Prob[S_N-z_N>0]=-\infty$;
\item If $z_N=N-1$, then $[S_N-z_N>0]=[S_N=N]$, and  $\frac{1}{V_N}\log\Prob[S_N-z_N>0]=-\frac{3}{2}\log 3
$.
\end{enumerate}
So the limit depends on how $z_N/V_N$ approaches
$\frac{3}{2}$, and it could be infinite.\qed
\end{proof}

\medskip
For general additive functionals on  Markov chains  (homogeneous or not),
we do not know how to determine the asymptotic behavior of $\Prob[S_N-z_N\in (a,b)]$ when $\frac{z_N}{V_N}$ is close to $\partial C_N$, where
$$C_N:=\text{domain of the Legendre transform of $\frac{1}{V_N}\log \E(e^{t (S_N-\E(S_N))})$.}
$$
We can only analyze the case where
$\frac{z_N-\E(S_N)}{\Var(S_N)}$\;
 is well inside the interior of $C_N$ for all $N$. This is why we must assume that $|z_N-\E(S_N)|\leq \eps\Var(S_N)$ for $\eps$ small enough.

\medskip
It is instructive  to compare the regime of large deviations to the regime of the LLT from the point of view of universality.

\medskip
The asymptotic behavior\index{regime!universality}\index{universality} of  $\Prob[S_N-z_N\in (a,b)]$ in the regime of local deviations does not depend on  the details of the distributions of $f_n(X_n,X_{n+1})$. It depends  only on rough features such as  $\Var(S_N)$, the algebraic range, and (in case the algebraic range is $t\Z$) on the constants $c_N$ s.t. $S_N\in c_N+t\Z$ almost surely.

By contrast, in the  regime of large deviations  the asymptotic behavior of $\Prob[S_N-z_N\in (a,b)]$ depends on the entire distribution of $S_N$. The dependence is through the Legendre transform of $\log\E(e^{tS_N})$, a function  which encodes the entire distribution of $S_N$,  not just its rough features.

\medskip
We will consider two partial remedies to the lack of universality:
\begin{enumerate}[(a)]
\item {\em Conditioning:\/} The  conditional distributions of $S_N-z_n$ given
that $S_N-z_N>a$ has a universal scaling limit, see Corollary \ref{CrConditioning-R}.

\item {\em Moderate deviations:\/} \index{regime!moderate} If  $|z_N-\E(S_N)|=o(\Var(S_N))$, then $\Prob[S_N-z_N\in (a,b)]$ have  universal lower and upper bounds (Theorems \ref{Thm-LLT-LD-intermediate-R}, \ref{Thm-LLT-LD-intermediate-Z}).
\end{enumerate}

\section{Local limit theorems for large deviations}
\subsection{The log moment generating functions}

Suppose $|{\mathsf f}|<K$ almost surely. For every $N$ such that $V_N\neq 0$, we define  the {\bf normalized log moment generating function} of $S_N$ to be  \index{log-moment generating function}\index{normalized log-moment generating function}
$$
\mathcal F_N(\xi):=\frac{1}{V_N}\log \E(e^{\xi S_N})\ \ \ \ (\xi\in\R).
$$
The a.s. uniform boundedness of $\mathsf f$ guarantees the finiteness of the expectation,  and the real analyticity of $\mathcal F_N(\xi)$ on $\R$.

\begin{example}[Sums of iid's]
\label{ExLDInd1}
\end{example}
Suppose that $\DS S_N=\sum_{n=1}^N X_n$ where $X_n$ where $X_N$ are i.i.d. bounded random variables
with non-zero variance.
Let $X$ denote the common law of
$X_n.$ Then
$$\DS  \cF_N(\xi)=\cF_X(\xi):=\frac{1}{\Var(X)}\log \EXP(e^{\xi X})$$
is independent of $n.$ In addition,
\begin{itemize}
\item[(i)] $\mathfs F_X(\xi)$ is strictly convex, by  H\"older's inequality and because $X\neq const$ a.s. Since $\mathfs F_X(\xi)$ is smooth, its second derivative must be
 bounded away from
zero on compacts. So $\mathfs F_N(\xi)$ are uniformly strictly convex on compacts.

\medskip
\item[(ii)]
$\DS  \lim_{\xi\to-\infty} \cF_N'(\xi)=\ess\inf (X)/\Var(X),$
$\DS \lim_{\xi\to +\infty} \cF_N'(\xi)=\ess\sup (X) /\Var(X). $
To see this, use convexity to see that $\lim \cF_N'(\xi)$ are the slopes of the asymptotes of $\mathfs F_X(\xi)$, or equivalently
$\lim \frac{1}{\xi} \cF_N(\xi)$. The last limits can be easily found to be equal to $\ess\sup(X)/\Var(X)$ as $\xi\to\infty$, and $\ess\inf(X)/\Var(X)$ as $\xi\to -\infty$.
\end{itemize}

\medskip
Properties (i) and (ii) play a key role in the study of large deviations for sums of i.i.d. random variables.
A significant part of the effort in this chapter is to understand to which extent similar results holds
in the setting of bounded additive functionals of uniformly elliptic Markov chains.
We start with the following facts.

\begin{theorem}\label{Theorem-F_N}
Let $\mathsf f$ be an a.s. uniformly bounded additive functional
 of a uniformly elliptic Markov chain $\mathsf X$, and assume  $V_N\neq 0$ for all $N\geq N_0$,  then
\begin{enumerate}[(1)]
\item For all $N\geq N_0$, $\mathcal F_N(0)=0\ , \ \mathcal F_N'(0)=\frac{\E(S_N)}{V_N}\ , \ \mathcal F_N''(0)=1$.

\medskip
\item   For every $N\geq N_0$, $\mathcal F_N(\xi)$ is strictly convex on $\R$.

\medskip
\item The convexity is uniform on compacts: For every $R>0$ there is $C=C(R)$ positive s.t.  for all $N\geq N_0$, $C^{-1}\leq \mathcal F_N''(\xi)\leq C$ on $[-R,R]$.

\medskip
\item Suppose $V_N\to\infty$. For every $\epsilon>0$ there are $\delta, N_\epsilon>0$ s.t. for all $|\xi|\leq \delta$, $N>N_\epsilon$, we have  $e^{-\epsilon}\leq \mathcal F_N''(\xi)\leq e^{\epsilon}$, and
$$e^{-\epsilon}\frac{1}{2}\left(\xi-\frac{\E(S_N)}{V_N}\right)^2\leq \mathcal F_N(\xi)-\frac{\E(S_N)}{V_N}\xi\leq e^{\epsilon}\frac{1}{2}\left(\xi-\frac{\E(S_N)}{V_N}\right)^2.$$
\end{enumerate}
\end{theorem}
This is very similar to what happens for iid's, but there is one important difference: In our setting $V_N$ may be much smaller than~$N.$

For the proof of this theorem see \S \ref{Section-Proof-Log-Moment-Generating}. Here is an immediate corollary:
\begin{corollary}
Suppose $\mathsf f$ is an a.s. uniformly bounded additive functional on a uniformly elliptic Markov chain $\mathsf X$. If $V_N:=\Var(S_N)\to\infty$, then for all $0<\alpha<\frac{1}{2}$ and  $\kappa>0$, if $\frac{z_N-\E(S_N)}{V_N}\sim \kappa V_N^{-\alpha}$ as $N\to\infty$, then
$$
\lim\limits_{N\to\infty}\frac{1}{V_N^{2\alpha-1}}\log\Prob[S_N-z_N\geq 0]=-\frac{1}{2}\kappa^2.
$$
\end{corollary}
\begin{proof}
There is no loss of generality in assuming that $\E(S_N)=0$ for all $N$.
Let $a_n:=V_n^{1-2\alpha}$, $b_n:=V_n^{\alpha}$, $W_n:=S_n/b_n$. Then $a_n\to\infty$, whence by
Theorem \ref{Theorem-F_N}(4),
$$
\mathfs F(\xi):=\lim\limits_{n\to\infty}\frac{1}{a_n}\log \E(e^{\xi W_n})=\lim\limits_{n\to\infty} V_n^{2\alpha}\mathfs F_N(\frac{\xi}{V_n^\alpha})=\frac{1}{2}\xi^2.
$$
We may now use the G\"artner-Ellis Theorem see e.g. \cite[Thm II.6.1]{El}) and $\frac{z_n}{a_n b_n}\to\kappa$ to  deduce that
$
\lim\limits_{n\to\infty}\frac{1}{a_n}\log\Prob[{S_n-z_n}\geq 0]=
\lim\limits_{n\to\infty}\frac{1}{a_n}\log\Prob[\frac{W_n}{a_n}\geq \frac{z_n}{a_n b_n}]=-\frac{1}{2}\kappa^2
$.\qed
\end{proof}

\subsection{The rate functions}

Suppose $V_N\neq 0$.
The {\bf rate functions}\index{rate functions} $\mathcal I_N(\eta)$ are the {\bf Legendre transforms}\index{Legendre transform} of $\mathcal F_N(\xi)$. Specifically,
 let $a_N:=\inf \mathcal F_N'$ and $b_N:=\sup\mathcal F_N'$; then $\mathcal I_N:(a_N,b_N)\to \R$ is
$$
\mathcal I_N(\eta):=\xi\eta-\mathcal F_N(\xi)\text{ for the unique $\xi$ s.t. $\mathcal F_N'(\xi)=\eta$}.
$$
The existence and uniqueness of $\xi$ is because of the smoothness and strict convexity of $\mathcal F_N$ on $\R$.
We call $(a_N,b_N)$ the {\bf  domain} of $\mathcal I_N$\index{Legendre transform!domain}, and denote it by
$$
\mathrm{dom}(\mathcal I_N):=(a_N,b_N).
$$
Equivalently, $\mathrm{dom}(\mathcal I_N)=(\mathfs F'(-\infty),\mathfs F'(+\infty))$, where $\mathfs F'(\pm\infty):=\lim\limits_{t\to\pm\infty}\mathfs F'(t)$.
Later we will also need the sets $(a_N^R,b_N^R)\subset \mathrm{dom}(\mathfs I_N)$, where $R>0$ and
\begin{equation}\label{a_n^R}
a_N^R:=\mathfs F_N'(-R), \quad
b_N^R:=\mathfs F_N'(R).
\end{equation}

The functions $\mathfs I_N$ and their domains depend on $N$.
The following theorem identifies certain uniformity and universality in their behavior.

\begin{theorem}\label{Theorem-I_N}
Let $\mathsf f$ be an a.s. uniformly bounded additive functional on a uniformly elliptic Markov chain $\mathsf X$, and assume  $V_N\neq 0$ for all $N$ large enough,   then
\begin{enumerate}[(1)]
 \item $\exists c,N_1,R>0$ s.t. for all $N>N_1$,
$\mathrm{dom}(\mathcal I_N)\supset[a_N^R,b_N^R]\supseteq\left[\frac{\E(S_N)}{V_N}-c, \frac{\E(S_N)}{V_N}+c\right].$
\item For each $R$ there exists $\rho=\rho(R)$  s.t.
$\rho^{-1}\leq \mathcal I_N''\leq \rho$ on $[a_N^R, b_N^R]$ for all $N>N_1$.
\item  Suppose $V_N\to\infty$. For every $\epsilon>0$  there exists $\delta>0$ and $N_\epsilon$ such that   for all $\eta\in [\frac{\E(S_N)}{V_N}-\delta, \frac{\E(S_N)}{V_N}+\delta]$ and $N>N_\delta$,
$$e^{-\epsilon}\frac{1}{2}\left(\eta-\frac{\E(S_N)}{V_N}\right)^2\leq \mathcal I_N(\eta)\leq e^{\epsilon}\frac{1}{2}\left(\eta-\frac{\E(S_N)}{V_N}\right)^2.$$  \item Suppose $V_N\to\infty$ and   $\frac{z_N-\E(S_N)}{V_N}\to 0$,  then
$$V_N \mathcal I_N\left(\frac{z_N}{V_N}\right)=\frac{1+o(1)}{2}\left(\frac{z_N-\E(S_N)}{\sqrt{V_N}}\right)^2\quad \text{as}\quad N\to\infty.
$$
\end{enumerate}
\end{theorem}
\noindent
The proof of the theorem will be given in \S \ref{Section-Legendre}.

The significance of part (4) will become apparent in  \S \ref{SSLDIntermediate}.

\subsection{The LLT  for moderate deviations.}
\label{SSLDIntermediate}
Recall that the state spaces of $\mathsf X$ are denoted by $\fS_i$ $(i\geq 1)$, and that $\Prob_x$ denotes the conditional probability given $X_1=x$.

\begin{theorem}\label{Thm-LLT-LD-intermediate-R}
Let $\mathsf f$ be an a.s. uniformly bounded additive functional on a uniformly elliptic Markov chain $\mathsf X$. Suppose $\mathsf f$ is irreducible with algebraic range $\R$. If $z_N\in\R$ satisfy $\frac{z_N-\E(S_N)}{V_N}\to 0$, then for every non-empty  $(a,b)$ and $x\in\mathfrak S_1$,
\begin{align*}
&\Prob_x[S_N-z_N\in (a,b)]=[1+o(1)]\frac{|a-b|}{\sqrt{2\pi V_N}}
\exp\left(-V_N \mathcal I_N\left(\frac{z_N}{V_N}\right)\right)\text{as $N\to\infty$,}\\
&\Prob_x[S_N-z_N\in (a,b)]=[1+o(1)]\frac{|a-b|}{\sqrt{2\pi V_N}}\exp\left[-\frac{1+o(1)}{2}\left(\frac{z_N-\E(S_N)}{\sqrt{V_N}}\right)^2\right] \text{ as $N\to\infty$.}
\end{align*}

\end{theorem}
\begin{theorem}\label{Thm-LLT-LD-intermediate-Z}
Let $\mathsf f$ be an a.s. uniformly bounded additive functional on a uniformly elliptic Markov chain $\mathsf X$.
Assume $\mathsf f$ is irreducible with algebraic range $\Z$, and $S_N\in c_N+\Z$ almost surely. If $z_N\in c_N+\Z$ satisfy $\frac{z_N-\E(S_N)}{V_N}\to 0$, then for every $x\in\mathfrak S_1$,
\begin{align*}
&\Prob_x[S_N=z_N]=\frac{[1+o(1)]}{\sqrt{2\pi V_N}}\exp
\left(-V_N \mathcal I_N\left(\frac{z_N}{V_N}\right)\right) \text{ as $N\to\infty$,}\\
&\Prob_x[S_N=z_N]=\frac{[1+o(1)]}{\sqrt{2\pi V_N}}\exp\left[-\frac{1+o(1)}{2}\left(\frac{z_N-\E(S_N)}{\sqrt{V_N}}\right)^2\right] \text{ as $N\to\infty$.}
\end{align*}
\end{theorem}
\noindent
We will obtain these results as special cases of a more complicated and general asymptotic relation  which we will state in the next section.

The two asymptotic relations in Theorems \ref{Thm-LLT-LD-intermediate-R} and \ref{Thm-LLT-LD-intermediate-Z} complement each other. The first is a precise asymptotic, but it is not universal, because it is expressed in terms of the  rate functions, which depend on the fine details of the distributions of $S_N$.
The second  is universal, but it is not an asymptotic equivalence  because the right-hand-side is only determined up to a multiplicative error of size $\exp[o(\frac{z_N-\E(S_N)}{\sqrt{V_N}})^2]$.

\subsection{The LLT for large deviations.}\label{Section-LLT-LDP}
Recall the definition of the subsets $(a_N^R,b_N^R):=(\mathfs F_N'(-R),\mathfs F_N'(R))\subset\mathrm{dom}(\mathfs I_N)$ from \eqref{a_n^R}. It is convenient to define
$$
[\ha_N^R,\hb_N^R]:=\left[a_N^R-\frac{\E(S_N)}{V_N},b_N^R-\frac{\E(S_N)}{V_N}\right].
$$

\begin{theorem}\label{Thm-LLT-LDP}
Let $\mathsf f$ be an a.s. uniformly bounded, irreducible,  additive functional on a uniformly elliptic Markov chain $\mathsf X$.
 For every $R$ large enough there are functions
$\DS \rho_N:\mathfrak S_1\times
\left[\ha_N^R ,\hb_N^R\right]\to\R^+$,
$\xi_N: [\ha_N^R,\hb_N^R]\to \R $
as follows:
\begin{enumerate}[(1)]
\item $\exists c>0$ such that $[\ha_N^R,\hb_N^R]\supset [-c,c]$ for all $N$ large enough.
\item {\bf Non Lattice case:} Suppose $G_{alg}(\mathsf X,\mathsf f)=\R$, then  for every sequence of  $z_N\in\R$ s.t.
$ \frac{z_N-\E(S_N)}{V_N}\in  [\ha_N^R, \hb_N^R] $,
for all finite non-empty intervals $(a,b)$, and for every $x\in\mathfrak S_1$,
we have the following asymptotic as $N\to\infty$:
$$
\Prob_x[S_N-z_N\in (a,b)]=[1+o(1)]\cdot\frac{e^{-V_N \mathcal I_N(\frac{z_N}{V_N})}}{\sqrt{2\pi V_N}}
\rho_N\left(x,\tfrac{z_N-\E(S_N)}{V_N}\right)
\int_a^b e^{-t\xi_N\left(\tfrac{z_N-\E(S_N)}{V_N}\right) }dt.
$$

\item {\bf Lattice case:}  Suppose $G_{alg}(\mathsf X,\mathsf f)=\Z$ and $S_N\in c_N+\Z$ a.s.,  then
for every sequence of $z_N\in c_N+\Z$ s.t.
$\frac{z_N-\E(S_N)}{V_N}\in [\ha_N^R, \hb_N^R] $,
 for all finite non-empty intervals $(a,b)$  and  $x\in\mathfrak S_1$, the following asymptotic holds when $N\to\infty$:
$$
\Prob_x[S_N-z_N\in (a,b)]=[1+o(1)]\cdot \frac{e^{-V_N \mathcal I_N(\frac{z_N}{V_N})}}{\sqrt{2\pi V_N}}
\rho_N\left(x,\tfrac{z_N-\E(S_N)}{V_N}\right)\cdot
\sum_{t\in (a,b)\cap\Z}e^{-t\xi_N\left(\tfrac{z_N-\E(S_N)}{V_N}\right)}. $$

\item {\bf Properties of the error terms:}
\begin{enumerate}[(a)]
\item  $\rho_N(x,\eta)$ are bounded away from $0,\infty$ on $\mathfrak S_1\times [\ha_N^R,\hb_N^R]$ uniformly in $N$,  and $\rho_N(x,\eta)\xrightarrow[\eta\to 0]{}1$ uniformly in $N$ and $x$.
\item For each $R>0$ there exists $C=C_R>0$ such that for all $\eta\in [\ha_N^R,\hb_N^R]$ and $N$,  $C^{-1}|\eta|\leq |\xi_N(\eta)|\leq C|\eta|$ and $\mathrm{sgn}(\xi(\eta))=\mathrm{sgn}(\eta)$.
\end{enumerate}
\end{enumerate}
\end{theorem}
\noindent
The proof of this result will occupy us in  \S\S \ref{section-strategy-LD}--\ref{SSLLT-LD-Proof}.

Theorem \ref{Thm-LLT-LDP} above assumes irreducibility. Without this assumption we have a following weaker bound.
\begin{theorem}
\label{ThLDOneSided}
Let $K:=\ess\sup|\mathsf  f|$, and
suppose $V_N\to\infty$.
For each $\eps, R$ there is $D(\eps, R,K)$ and $N_0$ such that for all $z_N\in [\cF_N'(\eps), b_N^R] $ and $N>N_0$,
$$ D^{-1}\leq \frac{\sqrt{V_N} \Prob(S_N\geq z_N)}
{e^{-V_N \mathcal I_N\left(\frac{z_N}{V_N}\right)}}\leq D. $$
\end{theorem}

To assist the reader in digesting the statement of Theorem \ref{Thm-LLT-LDP}, we now explain how to use it to obtain  Theorems \ref{Thm-LLT-LD-intermediate-R},
\ref{Thm-LLT-LD-intermediate-Z} on moderate deviations, as well as other consequences.

\medskip
\noindent
{\bf Proof of Theorems \ref{Thm-LLT-LD-intermediate-R} and \ref{Thm-LLT-LD-intermediate-Z}:}
By Theorem \ref{Thm-LLT-LDP}(1), $\exists R>0$ s.t.
if $\frac{z_N-\E(S_N)}{V_N}\to 0$, then
$\frac{z_N-\E(S_N)}{V_N}\in [\ha_N^R,\hb_N^R]$ for all $N$ large enough, and
$$
 \rho_N(x,\tfrac{z_N-\E(S_N)}{V_N})\xrightarrow[N\to\infty]{}1,\quad
\xi_N(\tfrac{z_N-\E(S_N)}{V_N})\to 0,\quad
\frac{1}{b-a}\int_a^b e^{-t\xi_N(\tfrac{z_N-\E(S_N)}{V_N}) }dt\to 1.
$$
Suppose $G_{alg}(\mathsf X,\mathsf f)=\R$, then theorem \ref{Thm-LLT-LDP}(2)  implies that
$$\Prob[S_N-z_N\in (a,b)]\sim \frac{|a-b|}{\sqrt{2\pi V_N}}\exp(-V_N\mathfs I_N(z_N/V_N)).$$
Next, by Theorem \ref{Theorem-I_N}(2),  if $\frac{z_n-\E(S_N)}{V_N}\to 0$, then
$$V_N\mathfs I_N\left(\frac{z_N}{V_N}\right)\sim \frac{1}{2}\left(\frac{z_n-\E(S_N)}{\sqrt{V_N}}\right)^2,$$
whence $\Prob[S_N-z_N\in (a,b)]\sim \frac{|a-b|}{\sqrt{2\pi V_N}}\exp(-\frac{1+o(1)}{2}(\frac{z_n-\E(S_N)}{\sqrt{V_N}})^2)$. This proves Theorem \ref{Thm-LLT-LD-intermediate-R}.
The proof of Theorem \ref{Thm-LLT-LD-intermediate-Z} is similar, and we leave it to the reader.\qed

\medskip
Here are some other consequences of Theorem \ref{Thm-LLT-LDP}.
\begin{corollary}
\label{CrConditioning-R}
Let  $ \mathsf f$ be an a.s. uniformly bounded  additive functional on a uniformly elliptic Markov chain.
Suppose $\mathsf f$ is irreducible, with algebraic range $\R$.
\begin{enumerate}[(1)]
\item  If $\frac{z_N-\EXP(S_N)}{V_N}\to 0 $ then for any finite non empty interval $(a, b)$ the distribution of
$S_N-z_N$ conditioned on $S_N-z_N\in (a, b)$ is asymptotically uniform on $(a, b).$

\medskip
\item  If $\lim\inf\frac{z_N-\EXP(S_N)}{V_N}>0$ and
 there exists $R$ s.t.
$\frac{z_N-\E(S_N)}{V_N}\in [\ha_N^R,  \hb_N^R]$ for all sufficiently large $N,$
then  the distribution of
$$\xi_N\left(\frac{z_N-\EXP(S_N)}{V_N}\right)\cdot (S_N-z_N)\text{ conditioned on $S_N\geq z_N$}
$$ is asymptotically exponential with parameter 1.
\end{enumerate}
\end{corollary}

\noindent
{\em Remark.\/}
The condition in (2) is satisfied whenever $\liminf \frac{z_N-\E(S_N)}{V_N}>0$, and $\limsup \frac{z_N-\E(S_N)}{V_N}>0$ is small enough, see  Theorem \ref{Thm-LLT-LDP}(1).

\begin{proof}
To see  part (1), note first that if $\frac{z_N-\E(S_N)}{V_N}\to 0$, then
$\xi_N=\xi_N(\frac{z_N-\E(S_N)}{V_N})\to 0$, whence $\frac{1}{\beta-\alpha}\int_{\alpha}^{\beta} e^{-t\xi_N}dt\xrightarrow[N\to\infty]{}1$ for every non-empty interval $(\alpha,\beta)$.
Thus by Theorem \ref{Thm-LLT-LDP},
for every interval $[c,d]\subset [a,b]$,
$$
\lim_{N\to\infty}\frac{\Prob_x[S_N-z_N\in (c,d)]}{\Prob_x[S_N-z_N\in (a,b)]}=\frac{|c-d|}{|a-b|}.
$$
(the prefactors  $\rho_N$ are identical, and they cancel out).

To see part (2), note first that our assumptions on $ z_N$ guarantee that
$\xi_N=\xi_N\left(\frac{z_N-\EXP(S_N)}{V_N}\right)$ is bounded
 from  away from zero and infinity, and that all its limit points are strictly positive.

Suppose $\xi_{N_k}\to \xi$. Then arguing as in part (1) it is not difficult to see that for all $(a,b)\subset(0,\infty)$ and $r>0$,
\begin{align*}
&\lim_{k\to\infty}\frac{\Prob_x[\xi_{N_k}(S_{N_k}-z_{N_k})\in (a+r,b+r)|S_{N_k}>z_{N_k}]}{\Prob_x[\xi_{N_k}(S_{N_k}-z_{N_k})\in (a,b)|S_{N_k}>z_{N_k}]}=e^{-r}.
\end{align*}
Since this is true for all convergent $\{\xi_{N_k}\}$, and since any subsequence of $\{\xi_N\}$ has a convergent subsequence,
\begin{align*}
&\liminf_{N\to\infty}\frac{\Prob_x[\xi_{N}(S_N-z_N)\in (a+r,b+r)|S_N>z_N]}{\Prob_x[\xi_{N}(S_N-z_N)\in (a,b)|S_N>z_N]}=e^{-r},\\
&\limsup_{N\to\infty}\frac{\Prob_x[\xi_{N}(S_N-z_N)\in (a+r,b+r)|S_N>z_N]}{\Prob_x[\xi_{N}(S_N-z_N)\in (a,b)|S_N>z_N]}=e^{-r},
\end{align*}
and so $\displaystyle{\lim_{N\to\infty}\frac{\Prob_x[\xi_{N}(S_N-z_N)\in (a+r,b+r)|S_N>z_N]}{\Prob_x[\xi_{N}(S_N-z_N)\in (a,b)|S_N>z_N]}=e^{-r}.
}$
So conditioned on $S_N>z_N$,  $\xi_N(S_N-z_N)$ is asymptotically exponential with parameter $1$.\qed
\end{proof}

\begin{corollary}\label{CrConditioning-Z}
Let  $ \mathsf f$ be an a.s. uniformly bounded  additive functional on a uniformly elliptic Markov chain.
Suppose $\mathsf f$ is irreducible, with algebraic range $\Z$.
Let $z_N$ be a sequence of integers.
\begin{enumerate}[(1)]
\item  If $\frac{z_N-\EXP(S_N)}{V_N}\to 0 $ then for any $a<b$ in $\Z$
the distribution of
$S_N-z_N$ conditioned on $S_N-z_N\in [a, b]$ is asymptotically uniform on $[a, b].$

\item If $\lim\inf\frac{z_N-\EXP(S_N)}{V_N}>0$
 and there exists $R$ s.t.
$\frac{z_N-\E(S_N)}{V_N}\in [\ha_N^R,  \hb_N^R]$ for all sufficiently large $N,$
$\xi_N\left(\frac{z_N-\E(S_N)}{V_N}\right)\to \xi$,
then
$$(S_N-z_N)\text{ conditioned on $S_N\geq z_N$}
$$ is asymptotically geometric with parameter
$e^{-\xi}. $
\end{enumerate}
\end{corollary}
\noindent
The proof is similar to the proof in the non-lattice case, so we omit it.

It worthwhile to note the following consequence of this result.  In the following statement, ``local distribution" means a functional on $C_c(\R)$ and ``vague convergence" means convergence on all continuous functions with compact support.
\begin{corollary}
\label{CrIM}
Let  $ \mathsf f$ be an a.s. uniformly bounded  additive functional on a uniformly elliptic Markov chain.
Let $z_N$ be a sequence s.t.
for some $R$, $\frac{z_N-\E(S_N)}{V_N}\in [\ha_N^R, \hb_N^R]$
for large $N.$
Let $\zeta_N$ be the local distribution of $S_N$ around $z_N,$ that is
$\zeta_N(\phi)=\EXP_x(\phi(S_N-z_N)).$ Let $\zeta$ be a vague limit of $\{q_N \zeta_N\}$
for some
sequence $q_N>0.$
If $\mathsf f$ is irreducible then $\zeta$ has density
$c_1 e^{c_2 t} $ with respect to the Haar measure on the algebraic
range of $\mathsf f$ for some $c_1\in \R_+, c_2\in \R.$
\end{corollary}
\noindent
If the restriction  $\frac{z_N-\E(S_N)}{V_N}\in [\ha_N^R, \hb_N^R]$
is dropped, then
it is likely that  $\zeta$ is either
as above,  or an atomic measure with one atom, but our methods
 are insufficient for proving this.

\section{Proofs}
We prove  Theorems \ref{Theorem-F_N}, \ref{Theorem-I_N},
\ref{Thm-LLT-LDP}  and \ref{ThLDOneSided}. (Theorems \ref{Thm-LLT-LD-intermediate-R} and \ref{Thm-LLT-LD-intermediate-Z} are direct consequences, and were proved in \S \ref{Section-LLT-LDP}.)

We assume throughout that $\mathsf \{X_n\}$ is a uniformly elliptic Markov chain with state spaces $\fS_n$, transition probabilities $\pi_{n,n+1}(x,dy)$, and stationary distributions $\mu_k(E):=\Prob(X_k\in E)$. Let $\mathsf f=\{f_n\}$ be an a.s. uniformly bounded additive functional on $\mathsf X$. Let $\epsilon_0$ denote the ellipticity constant of $\mathsf X$, and $K=\ess\sup|\mathsf f|$.

\subsection{Strategy of proof}\label{section-strategy-LD}
The proof  can be briefly described as an implementation of ``change of measure" technique  (aka ``Cram\'er's transform").\index{change of measure!and large deviations}

We explain the idea.
Suppose $\mathsf f$ is an a.s. uniformly bounded additive functional on a uniformly elliptic Markov chain $\mathsf X$, and let $z_N$ be as in Theorem \ref{Thm-LLT-LDP}.
We will modify the transition probabilities of $\mathsf X=\{X_n\}$ to generate a Markov array $\wt{\mathsf X}=\{\wt{X}^{(N)}_n\}$ whose row sums
$
\wt{S}_N={f}_1(\wt{X}_1^{(N)},\wt{X}_2^{(N)})+\cdots+{f}_N(\wt{X}_N^{(N)},\wt{X}_{N+1}^{(N)})
$
satisfy
\begin{equation}\label{bartok}
{z_N-\E(\wt{S}_N)}={o\left(\sqrt{\mathrm{Var}(\wt{S}_N)}\right)}.
\end{equation}
\eqref{bartok} places us in the regime of local deviations which we have analyzed in Chapter \ref{Section-LLT-irreducible}.  The results of that chapter provide asymptotics for $\Prob(\wt{S}_N-z_N\in (a,b))$, and these can be translated into  asymptotics for $\Prob(S_N-z_N\in (a,b))$.

The array $\wt{\mathsf X}$ is constructed from $(\mathsf X,\mathsf f)$ as follows: Let $\fS_n$ and $\pi_{n,n+1}(x,dy)$ denote the state spaces and transition probabilities of the original Markov chain $\mathsf X$, then we take $f^{(N)}_n=f_n$, $\fS^{(N)}_n=\fS_n$, and we let $\wt{\mathsf X}$ be the Markov array with state spaces $\fS^{(N)}_n$  and transition probabilities
$$
\tpi_{n,n+1}^{(N)}(x,dy):=e^{\xi_N f_n(x,y)} \frac{h_{n+1}(y,\xi_N)}{e^{p_n(\xi_N)}h_n(x,\xi_N)}\cdot \pi_{n,n+1}(x,dy).
$$
Here $\xi_N$ is a parameter that is calibrated to get \eqref{bartok},  and $p_n, h_n, h_{n+1}$ are chosen to guarantee that $\tpi_{n,n+1}^{(N)}(x,dy)$ has total mass equal to one. This technique is called a ``change of measure."

The value of $\xi_N$ depends on $\frac{z_N-\E(S_N)}{V_N}$. To construct $\xi_N$ and to control it, we must know that $\frac{z_N}{V_N}$ belong to a sets where  $\mathfs F_N$ are strictly convex, uniformly in $N$.
This is the reason  why we need to assume that
$\exists R$ s.t.
$\frac{z_N-\E(S_N)}{V_N}\in [\ha_N^R,\hb_N^R]$ for all $N$,  a condition we can check as soon as
$|\frac{z_N-\E(S_N)}{V_N}|<c$ with $c$ small enough.\footnote{Other situations where the condition
$\frac{z_N-\E(S_N)}{V_N}\in [\ha_N^R,\hb_N^R]$ can be checked are discussed
in  \S\ref{SS-LDTreshold}.}

We remark that the dependence of $\xi_N$ on $N$ means that $\{\wt{X}^{(N)}_n\}$ is an array, not a chain. The fact that the change of measure produces arrays from chains is the reason we insisted on working with arrays in the first part of this work.

\subsection{A parameterized family of changes of measure}\label{Section-h}
In this section
we construct, for an arbitrary given  sequence of  constants $\xi_N\in\R$, transition probabilities of the form
\begin{equation}\label{pi-tilde}
\tpi_{n,n+1}^{(N)}(x,dy):=e^{\xi_N f_n(x,y)} \frac{h_{n+1}(y,\xi_N)}{e^{p_n(\xi_N)}h_n(x,\xi_N)}\cdot \pi_{n,n+1}(x,dy),
\end{equation}
where $p_n(\xi_N)$ are real numbers and $h_k^{\xi_N}(\cdot)=h_k(\cdot, \xi_N)$ are positive functions on $\fS_k$  which are  chosen to guarantee that $\tpi_{n,n+1}^{(N)}(x,dy)$ has total mass equal to one.

\medskip
We treat the sequence of parameters $\xi_N$ as arbitrary.
In the next section we will explain how to choose a particular $\{\xi_N\}$ to guarantee \eqref{bartok}.

\begin{lemma}\label{Lemma-h-exist}
Given $\xi\in\R$ and
a sequence of real numbers $\{a_n\}_{n\in \naturals}$,
there are unique numbers ${p}_n(\xi)\in\R$,  and unique non-negative $h_n(\cdot,\xi)\in L^\infty(\mathfrak S_n,\mathfs B(\mathfrak S_n),\mu_n)$  s.t.
 $\int_{\mathfrak S_n} h_n(x,\xi)\mu_n(dx)=\exp(a_n\xi)$ for all $n$, and {for a.e. $x$}
\begin{equation}\label{h-condition}
\int_{\mathfrak S_{n+1}} e^{\xi f_n(x,y)} \frac{h_{n+1}(y,\xi)}{e^{p_n(\xi)}h_n(x,\xi)}\, \pi_{n,n+1}(x,dy)=1.
\end{equation}
The unique solution is positive almost everywhere.
\end{lemma}

\noindent
\medskip
{\em Remark.:} Notice that if $\{\ov{h}_n(\cdot,\xi)\}$, $\{\ov{p}_n(\xi)\}$ satisfy the Lemma with $a_n=0$, then the unique solution with general $\{a_n\}$ is given by
\begin{equation}\label{general-h}
h_n(\cdot,\xi):=e^{a_n\xi} \ov{h}_n(\cdot,\xi)\ , \ p_n(\xi):=\ov{p}_n(\xi)-a_{n}\xi+a_{n+1}\xi.
\end{equation}
 Evidently, $h_n,p_n$ give rise to the same probability kernel \eqref{pi-tilde} as do $\ov{h}_n,\ov{p}_n$. We call $\{\ov{h}_n\}$ and $\{\ov{p}_n\}$   the {\em fundamental solution}.

\begin{proof}  It is enough to  prove the existence and uniqueness of the fundamental solution, so henceforth we assume $a_n=0$. We may also assume without loss of generality that $|\xi|\leq 1$, else scale $\mathsf f$.

Set $V_n:=L^\infty(\mathfrak S_n,\mathfs B(\mathfrak S_n),\mu_n)$,  and define operators $L_n^\xi :V_{n+1}\to V_n$ by
\begin{equation}\label{Operators-L_n-xi}
(L_n^\xi h)(x)=\int\limits_{\mathfrak S_{n+1}} e^{\xi f_n(x,y)} h(y) \pi_{n,n+1}(x,dy).
\end{equation}
The operators $L_n^\xi$ are linear, bounded, and positive.

For \eqref{h-condition} to hold,  it is necessary and sufficient that $h_n^\xi(\cdot):=h_n(\cdot,\xi)$ be positive a.e.,  and  $L_n^\xi h_{n+1}^\xi=e^{p_n(\xi)} h_n^\xi$ for some $p_n(\xi)\in\R$.

Positivity everywhere may be replaced by the weaker property that $h_n^\xi\in L^\infty\setminus\{0\}$ are all non-negative a.e., because for such functions, since $|\mathsf f|\leq K$ a.s. and $\mathsf X$ is uniformly elliptic with ellipticity constant $\epsilon_0$,
$$
h_n^\xi(x)=e^{-p_n(\xi)-p_{n+1}(\xi)}(L_n^\xi L_{n+1}^\xi h_{n+2}^\xi)(x)\geq e^{-p_n(\xi)-p_{n+1}(\xi)-2K}\epsilon_0\|h_{n+2}^\xi\|_1.
$$
Thus to prove the lemma it is enough  to find a sequence numbers $p_n(\xi)\in\R$ and  non-negative $h_n^\xi\in L^\infty\setminus\{0\}$ such that  $L_n^\xi h_{n+1}^\xi=e^{p_n(\xi)} h_{n}^\xi$ for some $p_n(\xi)\in\R$.

The existence and uniqueness of such ``generalized eigenvectors" can be proved as in
\cite{FS},\cite{BG},\cite{K} using {\bf Hilbert's projective metrics}\index{Hilbert's projective metric!definition}. We recall what these are.
Let
$C_n:=\{h\in V_n: h\geq 0\text{ a.e. }\}$. These are closed cones and
$L_n^\xi(C_{n+1})\subset C_{n}.$  Define
$$
d_n(h,g):=\log\biggl(\frac{M(h|g)}{m(h|g)}\biggr)\in [0,\infty],\ \ (h,g\in C_n),
$$
where $M=M(f|g), m=m(f|g)$ are the best constants in the estimate $m h\leq f\leq Mh$. This is a pseudo-metric on the interior of $C_n$, and  $d(h, g)=0\Leftrightarrow h,g$ are proportional.
Also, for all $h,g\in C_n\setminus\{0\}$,
\begin{equation}\label{d-n-ineq}
\left\|\frac{h}{\int h}-\frac{g}{\int g}\right\|_1\leq e^{d_n(h,g)}-1.
\end{equation}
Birkhoff's theorem \cite{Bi}\index{Hilbert's projective metric!contraction properties} says that any linear map $T:C_{n+2}\to C_n$ such that the $d_n$--diameter of $T(C_{n+2})$ in $C_n$ is less than some $\Delta>0$, contracts the Hilbert's projective metric at least
by a factor $\theta:=\tanh(\Delta/4)\in (0,1)$.

We will apply Birkhoff's theorem to the linear transformations
$$
T_{n}^\xi:=L_n^\xi L_{n+1}^\xi : C_{n+2}\to C_n.
$$
One checks using  the standing assumptions and $|\xi|\leq 1$ that
\begin{equation}\label{Delta-Inequality}
e^{-2K} \epsilon_0\|h\|_1\leq  (T_n^\xi h)(x)\leq e^{2K}\epsilon_0^{-2}\|h\|_1\ \ \ (h\in C_{n+2}),
\end{equation}
whence $d_n(T_n^\xi h,1)\leq 4K+3\log(1/\epsilon_0)$. So the diameter of $T^\xi_n(C_{n+2})$ in $C_n$ is less than $\Delta:=8K+6\log(1/\epsilon_0)$. Hence by Birkhoff's Theorem mentioned above,
\begin{equation}\label{birkhoff-contraction}
d_n(T_{n+1}^\xi h, T_{n+1}^\xi g)\leq \theta d_{n+2}(h,g)\ \ (h,g\in C_{n+2}).
\end{equation}
where
$
\theta:=\tanh(2K+\frac{3}{2}\log(1/\epsilon_0))\in (0,1).
$

It follows that for every $n$,  $\{L_n^\xi L_{n+1}^\xi\cdots L_{n+k-1}^\xi 1_{\mathfrak S_{n+k}}\}_{k\geq 1}\subset C_n$ is a Cauchy sequence with respect to $d_n$. By
\eqref{d-n-ineq},
$$
\frac{L_n^\xi L_{n+1}^\xi\cdots L_{n+k-1}^\xi 1_{\mathfrak S_{n+k}}}
{\|L_n^\xi L_{n+1}^\xi\cdots L_{n+k-1}^\xi 1_{\mathfrak S_{n+k}}\|_1}
$$
is a Cauchy sequence in $L^1$.

The limiting function $h_n^\xi$ has integral one, and is positive and bounded, because of  \eqref{Delta-Inequality}.
  Clearly, $L_{n}^\xi h_{n+1}^\xi=e^{p_n} h_n^\xi$ for some $p_n\in\R$.   So $\{h_n^\xi\}, \{p_n\}$ exist.

Moreover,
the proof shows that $\diam\left(\bigcap_{k\geq 1}L_n^\xi\cdots L_{n+k-1}^\xi (C_{n+k})\right)=0$. It follows that $h_n^\xi$ is unique up to multiplicative constant, whence by the normalization condition, unique.
The lemma is proved. \qed
\end{proof}

The proof has the  following consequence, which
we mention for future reference:
 For every $R>0$, there exists $C_0>0$ and  $\theta\in (0,1)$ (depending on $R$)
such that for every $|\xi|\leq R$
\begin{equation}\label{exp-contraction-birkhoff}
d_1\left(L_1^\xi\cdots L_N^\xi h_{N+1}^\xi, \;
L_1^\xi\cdots L_N^\xi 1\right)\leq C_0\theta^{N/2} d_{N+1}\left(h_{N+1}^\xi,1\right).
\end{equation}
The case when $N$ is even follows directly from \eqref{birkhoff-contraction} and does not require the constant $C_0$. The case of odd $N$ is obtained from the even case by  using the exponential contraction of  $L_2^\xi\cdots L_N^\xi$ and the fact that one  additional application of $L_1^\xi$ (or any other positive linear operator) does not increase the Hilbert norm. This implies  \eqref{exp-contraction-birkhoff} with $C_0:=\theta^{-1/2}$.

\begin{lemma}\label{Lemma-h-bounded}
Let $h_n^\xi(\cdot)=h(\cdot,\xi)$ be as in Lemma \ref{Lemma-h-exist}. If $a_n$ is bounded, then for every $R>0$ there is  $C=C(R)$ s.t. for all $n\geq 1$, a.e. $x\in\mathfrak S_n$ and $|\xi|\leq R$,
$$ C^{-1}\leq h_n(x,\xi)\leq C\quad\text{and}\quad C^{-1}<e^{p_n(\xi)}<C . $$
\end{lemma}
\begin{proof}
It is enough to consider the fundamental solution ($a_n=0$, $\int h_n=1$); the general case follows from \eqref{general-h}. It is also sufficient to consider the case $|\xi|\leq 1$; the general case follows by scaling $\mathsf f$.

Let $\{h_n^\xi\}$ be the fundamental solution, then in the notation of the previous proof, $T^\xi_{n}h_{n+2}^\xi=e^{p_n(\xi)+p_{n+1}(\xi)}h_n^\xi$, whence by \eqref{Delta-Inequality},
$$
 e^{-2K}\epsilon_0\leq e^{p_n(\xi)+p_{n+1}(\xi)}h_{n+2}^\xi\leq e^{2K}\epsilon_0^{-2}.
$$
Integrating, and recalling that $\int h_{n+2}^\xi \, d\mu_{n+2}=\exp(a_{n+2}\xi)=1$, we obtain
$$
e^{-2K}\epsilon_0\leq e^{p_n(\xi)+p_{n+1}(\xi)}\leq e^{2K}\epsilon_0^{-2}.
$$ So  $e^{-4K}\epsilon_0^2\leq h_n^\xi(\cdot)\leq e^{4K}\epsilon_0^{-4}$.

Observe that $e^{p_n}=\int L^\xi_n h_{n+1}^\xi d\mu_{n+1}=e^{\pm K}\int h_{n+1}^\xi d\mu_{n+1}$. So $e^{p_n}$ is also uniformly bounded away from zero and infinity.\qed
\end{proof}

In the next section we will choose $\xi_N$ to guarantee \eqref{bartok}, and as it turns out, the choice involves
a condition on $\frac{\partial p_n}{\partial\xi}$. Later, we will also require information on $\frac{\partial^2 p_n}{\partial\xi^2}$. In preparation for this, we will now study the differentiability of
$$\xi\mapsto h_n^\xi\text{ and }\xi\mapsto p_n(\xi).$$
The map $\xi\mapsto h_n^\xi$ takes values in the Banach space $L^\infty$. To analyze it, we will use the theory of {\bf real-analytic maps into Banach spaces}\index{real-analyticity for functions on Banach spaces}  \cite{Die}.

Let us briefly review this theory.
Suppose $\mathfrak X,\mathfrak Y$ are Banach spaces. Let  $a_n:\mathfrak X^n\to \mathfrak Y$ be a multilinear map.
The {\bf norm} of $a_n$ is
$$\|a_n\|:=\sup\{\|a_n({x}_1,\ldots,{x}_n)\|:x_i\in\mathfrak X,\ \|{x}_i\|\leq 1\text{ for all }i\}.$$
A multilinear map is called {\bf symmetric}\index{symmetric multilinear function} if it is invariant under the permutation of its coordinates.
Given ${x}\in \mathfrak X$, we denote
$$
a_n {x}^n:=a_n({x},\ldots,{x}).
$$
A {\bf power series}\index{power series on Banach spaces} is a formal expression $\sum_{n\geq 1} a_n {x}^n$ where $a_n: \mathfrak X^n\to \mathfrak Y$ are multilinear  and symmetric.

A function $\phi:\mathfrak X\to \mathfrak Y$ is called {\bf real analytic}\index{real-analyticity for functions on Banach spaces} at ${x}_0$ if there is some $r>0$ and a power series $\sum a_n {x}^n$ (called the {\bf Taylor series at $x_0$})\index{Taylor series on Banach spaces} such that $\sum \|a_n\|r^n<\infty$ and  $$\phi(x)=\phi(x_0)+\sum_{n\geq 1} a_n ({x}-{x}_0)^n$$ whenever $\|{x}-{x}_0\|<r$. One can check that if this happens, then
\begin{equation}\label{a_n-identity}
a_n(x_1,\ldots,x_n)=\frac{1}{n!}\left.\frac{d}{dt_1}\right|_{t_1=0}\cdots
\left.\frac{d}{dt_n}\right|_{t_n=0}\phi(x_0+\sum_{i=1}^n t_i x_i).
\end{equation}
Conversely, if $\sum a_n (x-x_0)^n$ has positive radius of convergence with $a_n$ as in \eqref{a_n-identity}, then $\phi$ is real-analytic, and equal to its Taylor series $\phi(x_0)+\sum a_n (x-x_0)^n$ on a neighborhood of $x_0$.

\begin{example}\label{Example-RA}
Let $\phi: \mathfrak{X}\times \mathfrak{X}\times \mathbb{R}\to \mathfrak{X}$ be the map
$\phi(x,y,z):=x-y/z.$ Then $\phi$ is real-analytic at every $(x_0,y_0,z_0)$ such that $z_0\neq 0$, with  Taylor series
$$\phi(x,y,z)=\phi(x_0,y_0,z_0)+\sum_{n=1}^\infty a_n(x-x_0,y-y_0,z-z_0)^n,$$
where
$\|a_n\|=O(\|y_0\|/|z_0|^{n+1})+O(n/|z_0|^{n+1})$.
\end{example}
\begin{proof}
If $|z-z_0|<|z_0|$, then  $x-y/z=x-\frac{y}{z_0}\sum_{k\geq 0} (-1)^k \frac{1}{z_0^k}(z-z_0)^k$.  For each $n\geq 1$,  $\un{x}_0:=(x_0,y_0,z_0)$, $\un{x}_i:=(x_i,y_i,z_i)\ (1\leq i\leq n)$, and  $(t_1,\ldots,t_n)\in \R^n$,
\begin{align*}
&\phi(\un{x}_0+\sum_{i=1}^n t_i \un{x}_i)=
x_0+\sum_{i=1}^n t_i x_i+\sum_{k=0}^\infty\frac{(-1)^{k+1}}{z_0^{k+1}}
\left(y_0+\sum_{i=1}^n t_i y_i\right)\left(\sum_{i=1}^n t_i z_i\right)^k
\end{align*}
converges in norm whenever $(t_1,\ldots,t_n)\in A_n:=\bigl[|\sum_{i=1}^n t_i z_i|<|z_0|\bigr]$. In particular,  on $A_n$, this series is real-analytic in each $t_i$, and  can be differentiated term-by-term infinitely many times.

To find $a_n(\un{x}_1,\ldots,\un{x}_n)$ we observe that the differential \eqref{a_n-identity} is equal to the coefficient of $t_1\cdots t_n$ in the previous series. So for $n>2$,
$$
a_n(\un{x}_1,\ldots,\un{x}_n)=\frac{(-1)^{n+1}y_0}{z_0^{n+1}}\cdot z_1\cdots z_n+\frac{(-1)^{n}}{z_0^{n}}\sum_{i=1}^n y_i {z_1\cdots \hat{z_i}\cdots z_n}
$$
where the hat above $z_i$ indicates that the $i$-th term should be omitted.
It follows that $\|a_n\|=O(\|y_0\|/|z_0|^{n+1})+O(n/|z_0|^n)$.\qed
\end{proof}

\begin{lemma}\label{Lemma-Analyticity}
The functions
$\xi\mapsto h_n^\xi, p_n(\xi)$ are real-analytic. If $a_n$ is bounded, then for every $R>0$  there is $C(R)>0$ s.t. for every $|\xi|\leq R$ and $n\geq 1$,
$$\left\|\frac{\partial}{\partial\xi} h_n(\cdot,\xi)\right\|_\infty \leq C(R), \quad
\left\|\frac{\partial^2}{\partial\xi^2} h_n(\cdot,\xi)\right\|_\infty\leq C(R).$$
\end{lemma}

\begin{proof} The proof is based on \S 3.3 in \cite{Du}, although it is somewhat simpler because our setup is more elementary.

It is enough to consider the special case $R=1$ and $a_n=0$. In particular, $\int h_n^\xi=1$.

Fix $|\xi|\leq 1$ and let $T_n:=T_n^\xi$, $h_n(\cdot)=h_n(\cdot,\xi)$ be as in the proof of Lemma \ref{Lemma-h-exist}.
Define two Banach spaces:
\begin{align*}
X&:=\left\{(S_n)_{n\in\N}:
\begin{array}{l}
S_n: L^\infty(\mathfrak S_{n+2})\to  L^\infty(\mathfrak S_{n})\text{ are bounded linear }\\
\text{operators, and }\|S\|:=\sup_n{\|S_n\|}<\infty
\end{array}
\right\}\\
Y&:=\{(\vf_{n})_{n\in \N}: \vf_{n}\in L^\infty(\mathfrak S_{n+2}) \text{ , }  \|\vf\|:=\sup\|\vf_{n}\|_\infty<\infty\}
\end{align*}
Using  \eqref{Delta-Inequality}, it is not difficult to see that $T:=(T_{n})$ belongs to $X$. By Lemma \ref{Lemma-h-bounded}, $h:=(h_{n})_{n\in\N}$ belongs to $Y$.

\medskip
\noindent
{\sc Step 1.\/} {\em There exists $0<\delta<1$ s.t.  for every $(S,\vf)\in X\times Y$, for all $|\xi|\leq 1$,    if $\|S-T\|<\delta$ and  $\|\vf-h\|<\delta$, then $\inf|\int(S_n\vf_{n+2})|>\delta$.}

\medskip
\noindent
{\em Proof.\/} By \eqref{Delta-Inequality}, $\|T_n\|\leq M$ where $M:=e^{2K}\epsilon_0^{-2}$, and by Lemma \ref{Lemma-h-bounded}, there is a constant $\eps_1>0$ so that for all $n$ and $|\xi|\leq 1$
$$
\eps_1\leq (T_n h_{n+2})(x)\leq \eps_1^{-1}.
$$
So if $\|S-T\|<\delta$ and  $\|\vf-h\|<\delta$, then for a.e. $x$,
\begin{align*}
S_n\vf_{n+2}(x)&= (T_n h_{n+2})(x)-(T_n-S_n)h_{n+2}(x)
-S_n(h_{n+2}-\vf_{n+2})(x)\\
&\geq \eps_1-\|T-S\|\|h\|-(\|S-T\|+\|T\|)\|h-\vf\|\\
&\geq \eps_1-\delta\|h\|-(\delta+M)\delta.
\end{align*}

Let $C$ be a uniform upper bound for  $\|h\|$ which holds for all $|\xi|\leq 1$. If  $0<\delta<(\frac{\eps_1}{C+M+2})\wedge 1$, then  $S_n\vf_{n+2}>\delta$ a.e., and the step follows.

\medskip
Henceforth we fix $\delta$ as in step 1. Let $B_\delta(T):=\{S\in X:\|S-T\|<\delta\}$ and $B_\delta(h):=\{\vf\in Y: \|\vf-h\|<\delta\}$, and define
$$
\Upsilon:B_\delta(T)\times B_\delta(h)\to Y,\ \quad \Upsilon(S,\vf):=\left(\vf_n-\frac{S_n \vf_{n+2}}{\int(S_n\vf_{n+2})d\mu_{n+2}}\right)_{n\in\N}.
$$
This is well-defined by the choice of $\delta$, and $\Upsilon(T,h)=0$.

\medskip
\noindent
{\sc Step 2.\/} {\em $\Upsilon$ is real-analytic on $B_\delta(T)\times B_\delta(h)$.}

\medskip
\noindent
{\em Proof.\/} First we write $\Upsilon=\Phi(\Upsilon^{(1)},\Upsilon^{(2)}, \Upsilon^{(3)})$ with
\begin{enumerate}[$\circ$]
\item $\Upsilon^{(1)}:X\times Y\to Y$, $\Upsilon^{(1)}(S,\vf)=\vf$
\item $\Upsilon^{(2)}:X\times Y\to Y$, $\Upsilon^{(2)}(S,\vf)=(S_n\vf_{n+2})_{n\in\N}$.
\item $\Upsilon^{(3)}:X\times Y\to \ell^\infty$, $\Upsilon^{(3)}(S,\vf)=(\int(S_n\vf_{n+2})d\mu_{n+2})_{n\in\N}$.
\item $\Phi:\{(\vf,\psi,\xi)\in Y\times Y\times\ell^\infty: \inf|\xi_i|>0\}\to Y$, $$\Phi((\vf,\psi,\xi)_{i\geq 1})=(\vf_i-\xi_i^{-1}\psi_i)_{i\geq 1}.$$
\end{enumerate}
By step 1, $\vec{\Upsilon}:=(\Upsilon^{(1)},\Upsilon^{(2)}, \Upsilon^{(3)})$ maps $B_\delta(T)\times B_\delta(h)$ into
 $$U:=\{(\vf,\psi,\xi)\in Y\times Y\times\ell^\infty: \|\vf\|<C+\delta, \|\psi\|<M+\delta,\ \inf|\xi_i|>\delta/2\},$$ whence into
 the domain of $\Phi$.

We claim that for each of the functions $\Upsilon^{(i)}$, some high enough derivative of $\Upsilon^{(i)}$ is identically  zero. Let  $D$ be the derivative, and let  $D_i$ be the partial derivative with respect to the $i$-th variable, then
\begin{enumerate}[(1)]
\item $\Upsilon^{(1)}$ is linear, so $(D\Upsilon^{(1)})(S,\vf)$ is constant, and $D^2\Upsilon^{(1)}=0$.
\medskip
\item $\Upsilon^{(2)}:X\times Y\to Y$, $\Upsilon^{(2)}(S,\vf)=(S_n\vf_{n+2})_{n\in\N}$. Here
$$
\begin{array}{l}
(D_1\Upsilon^{(2)})(S,\vf)(S')=(S'_n\vf_{n+2})_{n\in\Z} \hspace{1cm}, \hspace{1cm} (D_1^{2}\Upsilon^{(2)})(S,\vf)=0\\
(D_2\Upsilon^{(2)})(S,\vf)(\vf')=(S_n\vf'_{n+2})_{n\in\Z}\hspace{1cm}, \hspace{1cm} (D_2^2\Upsilon^{(2)})(S,\vf)=0\\
(D_1 D_2\Upsilon^{(2)})(S,\vf)(S',\vf')=(S_n'\vf'_{n+2})_{n\in\Z}
\end{array}
$$
We see that $D^2\Upsilon^{(2)}$ does not depend on $(S,\vf)$, so $D^3\Upsilon=0$
\medskip
\item $\Upsilon^{(3)}:X\times Y\to \ell^\infty$, $\Upsilon^{(3)}(S,\vf)=(\int(S_n\vf_{n+2})d\mu_{n+2})_{n\in\N}$. As before, the third derivative is zero.
\end{enumerate}

Consequently, $\Upsilon^{(i)}$ are real-analytic on its domain (with finite Taylor series at every point).
Next we show that $\Phi$ is real-analytic on $U$.
To do this we recall that by Example \ref{Example-RA},
$
\DS x-\frac{y}{z}=\sum_{n=0}^\infty a_n(x_0,y_0,z_0)(x-x_0,y-y_0,z-z_0)^n
$
where $a_n(x_0,y_0,z_0):(\R^3)^n\to\R$ are  symmetric multilinear functions depending on $(x_0,y_0,z_0)$, s.t.
$
\|a_n(x_0,y_0,z_0)\|=O(|y_0|/|z_0|^{n+1})+O(n/|z_0|^n)
$.
So
\begin{equation}\label{A-T}
\Phi(\vf,\psi,\xi)=\Phi(\vf^{(0)},\psi^{(0)},\xi^{(0)})+\sum_{n=1}^{\infty}A_n(\vf-\vf^{(0)},\psi-\psi^{(0)},\xi-\xi^{(0)})^n,
\end{equation}
where  $A_n:(Y\times Y\times \ell^\infty)^n\to Y$, has entries
\begin{align*}
&A_n((\vf^{(1)},\psi^{(1)},\xi^{(1)}),\ldots,(\vf^{(n)},\psi^{(n)},\xi^{(n)}))_i(x):=\\
&a_n\bigl(\vf^{(0)}_i(x),\psi^{(0)}_i(x),\xi^{(0)}_i\bigr)
((\vf_i^{(1)}(x),\psi_i^{(1)}(x),\xi^{(1)}_i),\ldots,(\vf_i^{(n)}(x),\psi_i^{(n)}(x),\xi^{(n)}_i))
\end{align*}
 $A_n$ inherits multilinearity and symmetry from $a_n$, and
by construction,
\begin{align*}
\|A_n(\vf^{(0)},\psi^{(0)},\xi^{(0)}))\|\leq &\sup\left\{\|a_n(x_0,y_0,z_0)\|:|x_0|, |y_0|\leq C+M+\delta,  |z_0|>\frac{\delta}{2}\right\}
=O(2^n n/\delta^{n}).
\end{align*}
So the right-hand-side of \eqref{A-T}  has positive radius of convergence, proving the analyticity of $\Phi:U\to Y$.

The step follows from the well-known result that the composition of real-analytic functions is real-analytic, see \cite{Die}.

\medskip
\noindent
{\sc Claim 4.\/} {\em $(D_2\Upsilon)(T,h):Y\to Y$, the partial derivative of $\Upsilon$ at $(T,h)$ with respect to the second variable, has bounded inverse.}

\medskip
\noindent
{\em Proof.\/} A direct calculation shows that $(D_2\Upsilon)(T,h)(\vf)=\vf-\Lambda\vf$, where
$$
(\Lambda\vf)_n=\frac{T_n\vf_{n+2}}{\int(T_n h_{n+2})d\mu_n}-\left(\frac{\int(T_n\vf_{n+2})d\mu_n}{\int(T_n h_{n+2})d\mu_n}
\right)h_n.
$$
To prove the claim,  we show that $\Lambda$ has spectral radius $<1$.

Let $T^{(k)}_n:=T_{n} T_{n+2}\cdots T_{n+2(k-1)}$, then  we claim that
\begin{equation}\label{Lambda-k}
(\Lambda^k\vf)_n=\frac{T^{(k)}_n\vf_{n+2k}}{\int(T^{(k)}_n h_{n+2k})d\mu_n}-\left(\frac{\int(T_n^{(k)}\vf_{n+2k})d\mu_n}{\int(T_n^{(k)} h_{n+2k})d\mu_n}
\right)h_n.
\end{equation}
To see this we first note, using $T_m h_{m+2}\propto h_m$ and $\int h_md\mu_m=1$, that
$$\int(T^{(k+1)}_n h_{n+2(k+1)})d\mu_n=\int(T_n h_{n+2})d\mu_n \int(T^{(k)}_{n+2} h_{n+2(k+1)})d\mu_{n+2}.$$
With this identity in mind, the formula for $\Lambda^k$ follows by induction.

We now explain why \eqref{Lambda-k} implies that the spectral radius of $\Lambda$ is less than one.
Fix $\vf\in Y$. Recall that $C^{-1}\leq h_n\leq C$ for all $n$, and let
$$
\psi:=\vf+{2C\|\vf\|}h.
$$
Then $\psi\in Y$, $\Lambda^k\psi=\Lambda^k\vf$ for all $k$ (because $\Lambda h=0$), and for all $n$
\begin{equation}\label{psi-h}
C\|\vf\|h_n\leq \psi_n\leq 3C\|\vf\| h_n
\end{equation}
In particular, if $C_n$ is the cone from the proof of Lemma \ref{Lemma-h-exist}, and $d_n$ is its projective Hilbert metric, then  $\psi_n\in C_n$ and $d_n(\psi_n,h_n)\leq \log 3$.
Since $T_n$ contracts the Hilbert projective norm by a factor $\theta\in (0,1)$,
$$
d_n(T^{(k)}_n \psi_{n+2k},T^{(k)}_n h_{n+2k})\leq \theta^k\log 3.
$$
This implies by the definition of $d_n$ that for a.e. $x\in\mathfrak S_n$,
$$
\left|\frac{(T^{(k)}_n \psi_{n+2k})(x)/\int(T^{(k)}_n \psi_{n+2k})}{(T^{(k)}_n h_{n+2k})(x)/\int(T^{(k)}_n h_{n+2k})}-1\right|\leq \max\{3^{\theta^k}-1, 1-3^{-\theta^k}\}=3^{\theta^k}-1=:\eps_k.
$$
The denominator simplifies to $h_n$. So
\begin{equation}\label{tnk}
\left\|
\frac{(T^{(k)}_n \psi_{n+2k})}{\int(T^{(k)}_n \psi_{n+2k})}-h_n
\right\|_\infty\leq \eps_k\|h\|.
\end{equation}
Next we use the positivity of $T_n^{(k)}$ and \eqref{psi-h} to note that
$$
C\|\vf\|T_n^{(k)} h_{n+2k}\leq T_n^{(k)} \psi_{n+2k}\leq 3C\|\vf\|T_n^{(k)} h_{n+2k}.
$$
We deduce that
\begin{equation}\label{psi-dividide-by-h}
C\|\vf\|\leq\frac{\int(T_n^{(k)} \psi_{n+2k})}{\int(T_n^{(k)} h_{n+2k})}\leq 3C\|\vf\|.
\end{equation}
By \eqref{Lambda-k}, \eqref{tnk} and \eqref{psi-dividide-by-h},
\begin{align*}
&\|\Lambda^k\vf\|_\infty\equiv \|\Lambda^k\psi\|=\sup_n
\left\|\frac{T_n^{(k)}\psi_{n+2k}}{\int T_n^{(k)}h_{n+2k}}-\frac{\int T_n^{(k)}\psi_{n+2k}}{\int T_n^{(k)}h_{n+2k}}\cdot h_n\right\|_\infty
\end{align*}
\begin{align*}
&\leq \sup_n\left\|\frac{T_n^{(k)}\psi_{n+2k}}{\int T_n^{(k)}\psi_{n+2k}}-h_n\right\|_\infty\cdot \sup_n\left\|\frac{\int T_n^{(k)}\psi_{n+2k}}{\int T_n^{(k)}h_{n+2k}} \right\|_\infty\leq 3C\epsilon_k\|h\|\cdot\|\vf\|,
\end{align*}
whence $\rho(\Lambda)\leq \lim\sqrt[k]{\epsilon_k}=\theta<1$.

\medskip
\noindent
{\sc Completion of the proof of the Lemma.\/} We  constructed a real-analytic function $\Upsilon:X\times Y\to Y$ such that $\Upsilon(T,h)=0$ and $(D_2\Upsilon)(T,h):Y\to Y$ has a bounded inverse.
By the implicit function theorem for real-analytic functions on Banach spaces  \cite{Wh}, $T$ has a neighborhood $W\subset X$ where one can define a real-analytic function $h:W\to Y$ so that $\Upsilon(S,h(S))=0$.

Recall that $T=T^\xi:=\{T_n^\xi\}_{n\in\N}$ and $h=\{h_n(\cdot,\xi)\}_{n\geq 1}$.
By the uniqueness part of Lemma \ref{Lemma-h-exist}, $h(T)=h(\cdot,\xi)$. It is easy to see using $\ess\sup|\mathsf f|<\infty$ that  $\xi\mapsto T^\xi$ is real-analytic (even holomorphic). So
 $\xi\mapsto h(T^\xi)$ is real-analytic, whence continuously differentiable infinitely many times.
Thus $\xi\mapsto h_n(\cdot,\xi)$ is real-analytic for all $n$, and
$\left\{\frac{\partial ^k}{\partial \xi^k}h_n(\cdot,\xi)\right\}_{n\geq 1}=\frac{\partial}{\partial\xi^k}h(T^\xi)\in Y$
for all $k$.
By the definition of $Y$, $\sup\limits_{|\xi|\leq 1}\sup\limits_{n\geq 1}\|\frac{\partial}{\partial \xi}h_n(\cdot,\xi)\|_\infty=\|\frac{\partial}{\partial \xi}h(T)\|<\infty$ and
$\sup\limits_{|\xi|\leq 1}\sup\limits_{n\geq 1}\|\frac{\partial^2}{\partial \xi^2}h_n(\cdot,\xi)\|_\infty=\|\frac{\partial^2}{\partial \xi^2}h(T)\|<\infty$.
\qed
\end{proof}

\subsection{Choosing the parameters}\label{Section-xi}
Given $\xi\in\R$ and $\{a_n\}\subset\R$ bounded,  let $\{\wt{X}^\xi_n\}_{n\geq 1}$ denote the Markov chain with the  initial distribution and state spaces of $\mathsf X$, but with transition probabilities
$$
\wt{\pi}_{n,n+1}^{\xi}(x,dy)=e^{\xi f_n(x,y)}\frac{h_{n+1}(y,\xi)}{e^{p_n(\xi)} h_n(x,\xi)}\cdot \pi_{n,n+1}(x,dy),
$$
where $p_n(\xi)$ and $h_k^\xi(\cdot)=h_k(\cdot,\xi)$ are as in  Lemma \ref{Lemma-h-exist}.
(This chain does not depend on the choice of $\{a_n\}$, see the remark after the statement of Lemma \ref{Lemma-h-exist}.)
Denote the expectation and variance operators of this chain by $\wt{\E}^{\xi}$, $\wt{V}^\xi$.

In this section we  show that if $V_N:=\Var(S_N)\to\infty$ and  $\frac{z_N-\E(S_N)}{V_N}$
is sufficiently small, then it is possible to choose $\xi_N$ and $a_n$ bounded s.t.
$$
\frac{z_N-\wt{\E}^{\xi_N}(S_N)}{\sqrt{\wt{V}^{\xi_N}(S_N)}}\xrightarrow[N\to\infty]{}0
\quad \text{and}\quad\sum_{n=1}^N p_n'(0)=\E(S_N).
$$
Indeed, we will find $\xi_N$ so that $\wt{\E}^{\xi_N}(S_N)=z_N+O(1)$.
The construction will show that  if $\frac{z_N-\E(S_N)}{V_N}\to 0$, then $\xi_N\to 0$.

Let $\ov{h}_n^\xi:=\ov{h}_n(\cdot,\xi):\mathfrak S_n\to (0,\infty)$ and $\ov{p}_n(\xi)\in\R$ be the fundamental solution: $L_n^\xi \ov{h}_{n+1}^\xi=e^{\ov{p}_n(\xi)} \ov{h}_n^\xi$ and $\int\ov{h}_n(x,\xi)\mu_n(dx)=1$. Then $\ov{h}_n^\xi=e^{-a_n\xi}h_n(\cdot,\xi)$ and $\ov{p}_n(\xi)=p_n(\xi)+a_n\xi-a_{n+1}\xi$ so
$$
\wt{\pi}_{n,n+1}^\xi(x,dy)= e^{\xi f_n(x,y)}\frac{\ov{h}_{n+1}(x,\xi)}{e^{\ov{p}_{n}(\xi)} \ov{h}_{n}(y,\xi)}\pi_{n,n+1}(x,dy).
$$
Let
$
\ov{P}_N(\xi):=\ov{p}_1(\xi)+\cdots+\ov{p}_N(\xi).
$

\begin{lemma}\label{Lemma-Changed-Expectation-Variance-bar}
$\xi\mapsto \ov{P}_N(\xi)$ is real analytic, and for every $R>0$ there is a constant $C(R)$ such that  for all $|\xi|\leq R$ and $N\in\mathbb N$,
\begin{enumerate}[(1)]
\item $|\ov{P}_N'(\xi)-\wt{\E}^\xi(S_N)|\leq C(R)$;
\item Suppose $V_N\to\infty$.
Then $C(R)^{-1}\leq \wt{V}^\xi(S_N)/V_N\leq C(R)$ for all $N$ and $|\xi|\leq R$,
and $$\ov{P}_N''(\xi)/\wt{V}^\xi(S_N)\xrightarrow[N\to\infty]{}1\text{ uniformly in $|\xi|\leq R$.}$$
\end{enumerate}
\end{lemma}
\begin{proof} We have the identity $e^{\ov{P}_N(\xi)}=\int(L_1^\xi\cdots L_N^\xi \ov{h}_{N+1}^\xi)(x)\mu_1(dx)$.  Since
$\xi\mapsto \ov{h}^\xi$ and $\xi\mapsto L_n^\xi$ are real-analytic, $\xi\mapsto \ov{P}_N(\xi)$ is real-analytic.

Given $x\in\mathfrak S_1$ (the state space of $X_1$), define  two measures on $\prod_{i=2}^{N+1}\mathfrak S_i$ so that for every $E_i\in\mathfs B(\mathfrak S_i)$ ($1\leq i\leq N+1$),
\begin{align*}
\pi_x(E_2\times\cdots\times E_{N+1})&:=\Prob(X_2\in E_2,\ldots,X_{N+1}\in E_{N+1}|X_1=x_1),\\
\wt{\pi}^\xi_x(E_2\times\cdots\times E_{N+1})&:=\wt{\Prob}^\xi(\wt{X}_2^\xi\in E_2,\ldots,\wt{X}^\xi_{N+1}\in E_{N+1}|\wt{X}^\xi_1=x_1).
\end{align*}
Let $S_N(x,\un{y}):=f(x,y_1)+\sum_{i=1}^{N} f_i(y_i,y_{i+1})$, then
$$
\frac{d\wt{\pi}_{x}^\xi}{d\pi_{x}}(y_2,\ldots,y_{N+1})=e^{\xi S_N(x,\un{y})} e^{-\ov{P}_N(\xi)} \left(\frac{\ov{h}_{N+1}(y_{N+1},\xi)}{\ov{h}_1(x,\xi)}\right).
$$
By Lemma \ref{Lemma-Analyticity}, $\xi\mapsto \frac{d\wt{\pi}_{x}^\xi}{d\pi_{x}}(y_2,\ldots,y_{N+1})$ is real-analytic. Differentiating, gives\\
$\frac{d}{d\xi}\bigl[\frac{d\wt{\pi}_{x}^\xi}{d\pi_{x}}\bigr]=\left[S_N(x,\un{y}) -\ov{P}_N'(\xi)
\frac{\ov{h}_1(x,\xi)}{\ov{h}_{N+1}(y_{N+1},\xi)}
\frac{d}{d\xi}\left(\frac{\ov{h}_{N+1}(y_{N+1},\xi)}{\ov{h}_1(x,\xi)}\right)\right]
\frac{d\wt{\pi}_{x}^\xi}{d\pi_{x}}.
$
We write this as
\begin{equation}
\frac{d}{d\xi}\left[\frac{d\wt{\pi}_{x}^\xi}{d\pi_{x}}\right]
=\left[S_N(x,\un{y})-\ov{P}_N'(\xi)+\epsilon_N(x,y_{N+1},\xi)\right]\frac{d\wt{\pi}_{x}^\xi}{d\pi_{x}},\label{suave}
\end{equation}
where $\epsilon_N(x,y_{N+1},\xi):=\frac{\ov{h}_1(x,\xi)}{\ov{h}_{N+1}(y_{N+1},\xi)}\frac{d}{d\xi}\left(\frac{\ov{h}_{N+1}(y_{N+1},\xi)}{\ov{h}_1(x,\xi)}\right)$.
By Lemmas \ref{Lemma-h-bounded} and \ref{Lemma-Analyticity},
$\epsilon_N(x,y_{N+1},\xi)$ is
uniformly bounded in $N$, $x,\un{y}$, and $|\xi|\leq R$.

By the intermediate value theorem and the uniform boundedness of $\xi\mapsto\frac{d}{d\xi}\left[\frac{d\wt{\pi}_{x}^\xi}{d\pi_{x}}\right]$ on compact subsets of $\xi\in\R$,
$\frac{1}{\delta}\left[\frac{d\wt{\pi}_{x}^{\xi+\delta}}{d\pi_{x}}-\frac{d\wt{\pi}_{x}^\xi}{d\pi_{x}}\right]$
is uniformly bounded for $0<|h|<1$.  So by the bounded convergence theorem
$$
\int \lim_{\delta\to 0}\frac{1}{h}\left[\frac{d\wt{\pi}_{x}^{\xi+\delta}}{d\pi_{x}}-\frac{d\wt{\pi}_{x}^\xi}{d\pi_{x}}\right]d\pi_x
=\lim_{\delta\to 0}\int \frac{1}{h}\left[\frac{d\wt{\pi}_{x}^{\xi+\delta}}{d\pi_{x}}-\frac{d\wt{\pi}_{x}^\xi}{d\pi_{x}}\right]d\pi_x=0.
$$
So $\int \frac{d}{d\xi}\left[\frac{d\wt{\pi}^\xi_x}{d\pi_x}\right]d\pi_x=0$, whence by \eqref{suave},
$
0=\wt{\E}^\xi_x(S_N)-\ov{P}_N'(\xi)+O(1),
$
where $\wt{\E}^\xi_x=\wt{\E}^\xi(\cdot|\wt{X}^\xi_1=x)$.
 Integrating with respect to $x$ we obtain that
 $$\ov{P}_N'(\xi)=\wt{\E}^\xi(S_N)+O(1)$$ uniformly in $|\xi|\leq R$, $N\to\infty$.

Differentiating \eqref{suave} again we obtain
\begin{align*}
&\frac{d^2}{d\xi^2}\left[\frac{d\wt{\pi}_{x}^\xi}{d\pi_x}\right]
=\frac{d}{d\xi}\left[\frac{d\wt{\pi}_{x}^\xi}{d\pi_x}\left(S_N(x,\un{y})-\ov{P}_N'(\xi)+\epsilon_N(x,y_{N+1},\xi)\right)\right]\\
&=\frac{d\wt{\pi}_{x}}{d\pi_x}\left[\left(S_N(x,\un{y})-\ov{P}_N'(\xi)+\epsilon_N(x,y_{N+1},\xi)\right)^2
-\ov{P}_N''(\xi)+\frac{d\epsilon_N}{d\xi}\right].
\end{align*}
By Lemmas \ref{Lemma-h-bounded} and \ref{Lemma-Analyticity}, $\frac{d\epsilon_N}{d\xi}$ is  uniformly bounded in $x,y_{N+1},N$ and $|\xi|\leq R$.
As before,  $\int \frac{d^2}{d\xi^2}\frac{d\wt{\pi}_{x}^\xi}{d\pi_x} d\pi_x= \frac{d^2}{d\xi^2}\int \frac{d\wt{\pi}_{x}^\xi}{d\pi_x} d\pi_x=0$, whence
\begin{align}
0&=\wt{\E}^{\xi}\left[\left(S_N-\ov{P}_N'(\xi)+O(1)\right)^2\right]-\ov{P}_N''(\xi)+O(1) \notag\\
&=\wt{\E}^{\xi}\left[\left(S_N-\wt{\E}^\xi(S_N)+O(1)\right)^2\right]-\ov{P}_N''(\xi)+O(1),\label{pergolesi}\\
&=\wt{V}^\xi(S_N)-\ov{P}_N''(\xi)+O\left(\sqrt{\wt{V}^\xi(S_N)}\right) \notag
\end{align}
where the $O(1)$ terms  are  uniformly bounded in $N$ when $|\xi|\leq R$.

If $|\xi|\leq R$, then $\wt{\pi}^\xi_{n,n+1}(x,dy)$ are uniformly elliptic  with $\eps_0$ replaced by $\eps_0/(C^2 e^{KR})$ for the $C$ in Lemma \ref{Lemma-h-bounded}. Therefore by Theorem \ref{Theorem-MC-Variance},  $\wt{V}^\xi(S_N)\asymp\sum_{n=3}^N u_n^2(\xi)$ where $u_n(\xi)$ are the structure constants of $\{\wt{X}^\xi_n\}$. Clearly, $u_n(\xi)\asymp u_n$ where $u_n=u_n(0)$ are the structure constants of $\{X_n\}$. So $\wt{V}^\xi(S_N)\asymp V_N\to\infty$ where the multiplicative error bounds is uniform in $N$ and  $|\xi|\leq R$.
By \eqref{pergolesi},
 $\ov{P}_N''(\xi)/\wt{V}^\xi(S_N)\xrightarrow[N\to\infty]{}1$.\qed
\end{proof}

\noindent
{\bf The choice of $a_N$:}  Lemma \ref{Lemma-Changed-Expectation-Variance-bar}(1) with $\xi=0$ says that $\ov{P}_N'(0)=\E(S_N)+O(1)$.
The error term is a nuisance, and we will choose $a_n$ to get rid of it. Given $N$, let
\begin{equation}\label{choice-of-a_n}
a_n:={\E}(S_{n-1})-\ov{P}'_{n-1}(0)\ , \ a_1:=0\\
\end{equation}
This is a bounded sequence, because of Lemma \ref{Lemma-Changed-Expectation-Variance-bar}(1).
The choice of $\{a_n\}$ leads to the following objects:
\begin{equation}\label{choice-of-a_n}
\begin{aligned}
&h_n^\xi(x)=h_n(x,\xi):=\exp(a_n\xi)\ov{h}_n(x,\xi),\\
&p_n(\xi):=\ov{p}_n(\xi)+(a_{n+1}-a_{n})\xi.\\
\end{aligned}
\end{equation}
The transition kernel $\wt{\pi}_{n,n+1}^{\xi}$ is left unchanged, because the differences between  $\ov{h}_n$ and $h_n$ and between  $\ov{p}_n$ and $p_n$ cancel out.  But now,
\begin{equation}\label{PNXI}
P_N(\xi):=p_1(\xi)+\cdots+p_N(\xi)\equiv \ov{P}_N(\xi)+\bigl({\E}(S_N)-\ov{P}'_N(0)\bigr)\xi,
\end{equation}
satisfies $P_N'(0)={\E}(S_N)$.

\medskip
\noindent
{\bf Properties of $P_N(\xi)$:}
These functions  turn out to be closely related to the distributional properties of $\mathsf X$ and its change of measure $\mathsf X^\xi$.

Recall that $\mathfs F_N(\xi):=\frac{1}{V_N}\log\E(e^{\xi S_N})$, and  that $\wt{V}^\xi$ is the variance of $S_N$ with respect to the change of measure $\wt{X}^\xi$. Then:
\begin{lemma}\label{Lemma-Changed-Expectation-Variance}
Suppose $V_N\to\infty$
then  $\xi\mapsto {P}_N(\xi)$ is real analytic, and
\begin{enumerate}[(1)]
\item $P_N'(0)=\E(S_N)$
\item For every $R>0$, there exists $C(R)>0$ s.t.   $$|{P}_N'(\xi)-\wt{\E}^\xi(S_N)|\leq C(R)
    \ \ \text{ for all\  $|\xi|\leq R, N\in\N$}.
    $$
\item For every $R>0$, there exists $C(R)>0$ s.t.
$$
      C(R)^{-1}\leq \wt{V}^\xi(S_N)/V_N\leq C(R)
         \ \ \text{ for all\  $|\xi|\leq R,\; N\in\N$}.
$$
\item
${P}_N''(\xi)/\wt{V}^\xi(S_N)\xrightarrow[N\to\infty]{}1$ uniformly on compact subsets of $\xi$.
\item $P_N(\xi)/V_N=\mathfs F_N(\xi)+o(V_N^{-1})$ uniformly on compact subsets of $\xi$, as $N\to\infty$. Specifically,
let
$\Delta_N(R):=\sup\limits_{|\xi|\leq R} V_N\left|\mathcal F_N(\xi)-\frac{P_N(\xi)}{V_N}\right|.$ Then
$
\displaystyle\sup\limits_{N}\Delta_N(R)<\infty$  for all $R>0$, and $\sup\limits_{N}\Delta_N(R)\xrightarrow[R\to 0^+]{}0.
$
\item $P_N'(\xi)/V_N=\mathfs F_N'(\xi)+O(V_N^{-1})$  uniformly on compact subsets of $\xi$, as $N\to\infty$. Specifically,
let
$\brDelta_N(R):=\sup\limits_{|\xi|\leq R} V_N\left|\mathcal F_N'(\xi)-\frac{P_N'(\xi)}{V_N}\right|.$ Then
$
\displaystyle\sup\limits_{N\geq N_0}\brDelta_N(R)<\infty
.$
\end{enumerate}
\end{lemma}
\begin{proof}
The real analyticity of $P_N(\xi)$ and parts (1)--(4)  follow directly  from Lemma \ref{Lemma-Changed-Expectation-Variance-bar}, the identity
$P_N(\xi)=\ov{P}_N(\xi)+(a_{N+1}-a_1)\xi$, and the boundedness of $a_n$.

The proof of part (5) uses the  operators $L_n^\xi: L^\infty(\mathfrak S_{n+1})\to L^\infty(\mathfrak S_n)$ from  \eqref{Operators-L_n-xi},
$(L_n^\xi h)(x):=\int_{\mathfrak S_{n+1}}e^{\xi f_n(x,y)}h(y) \pi_{n,n+1}(x,dy)\equiv \E_x[e^{\xi f_n(x,X_{n+1})}h(X_{n+1})].$

Let $h_n^\xi:=h_n(\cdot,\xi)\in L^\infty(\mathfrak S_n)
$ be the unique positive functions constructed so that $L_n^{\xi}h^\xi_{n+1}=e^{p_n(\xi)}h_n^\xi,\text{ where $p_1(\xi)+\cdots+p_N(\xi)=P_N(\xi)$.}$
(To construct $h_n^\xi$, apply  Lemma  \ref{Lemma-h-exist} with $a_n$ as in
\eqref{choice-of-a_n}.)
In particular, $h_n^0\equiv 1$ and
\begin{equation}
\label{EigenEq}
 \EXP_x\left(e^{\xi S_N} h_{N+1}(X_{N+1})\right)=e^{P_N(\xi)} h_1^\xi(x).
\end{equation}

By Lemma \ref{Lemma-h-bounded}, there exists $C_1=C_1(R)>1$ such that $C_1^{-1}\leq h_{N+1}^\xi\leq C_1$ for all $|\xi|\leq R$ and $N\geq 1$. Thus by
\eqref{EigenEq},
$$ C_1(R)^{-2} e^{P_N(\xi)} \leq \EXP\left(e^{\xi S_N} \right)\leq C_1(R)^2 e^{P_N(\xi)}.$$
Taking logarithms, we deduce that  $|\mathfs F_N(\xi)-P_N(\xi)/V_N|\leq 2C_1(R)/V_N$ for all $N\geq 1$ and $|\xi|\leq R$. Equivalently, $\sup_N\Delta_N(R)\leq 2C_1$.

Next, by Lemma \ref{Lemma-Analyticity} and the identity $h_n^0\equiv 1$, $\|h_N^\xi-1\|_\infty\xrightarrow[N\to\infty]{}0$ uniformly on compact subsets of $\xi$. Returning to the definition of $C_1(R)$ we find that we may choose  $C_1(R)\xrightarrow[R\to 0^+]{}1$. As before, this implies that $\sup_N\Delta_N(R)\xrightarrow[R\to 0]{}0$.

\medskip
Here is the proof of part (6). Fix $R>0$ and let $\wt\E^\xi$ denote the expectation operator with respect to the change of measure $\mathsf X^\xi$, then
\begin{equation}\label{VNFP}
 V_N \cF_N'(\xi)=\frac{\EXP(S_N e^{\xi S_N})}{\EXP(e^{\xi S_N})}=
\frac{\wt\EXP^\xi(S_N (h_1^{\xi}/h_{N+1}^\xi))}{\wt\EXP^\xi(h_1^\xi/h_{N+1}^\xi)} .
\end{equation}
We have already remarked that $\mathsf X^\xi$ are uniformly elliptic, and that their uniform ellipticity constants are bounded away from zero for $\xi$ ranging on a compact set. This gives us the mixing bounds in Proposition \ref{Proposition-Exponential-Mixing} with the same $C_{mix}>0$, $0<\theta<1$ for all $|\xi|\leq R$. So
$$ \wt\EXP^\xi\left(\frac{h_1^\xi S_N}{ h_{N+1}^\xi}\right)=
\wt\EXP^\xi(h_1^\xi) \wt\EXP^\xi\left(1/h_{N+1}^\xi\right) \wt\EXP^\xi(S_N)+O(1)\text{ as $N\to\infty$},
$$
$$
\wt\EXP^\xi\left(\frac{h_1^\xi}{ h_{N+1}^\xi}\right)=
\wt\EXP^\xi(h_1^\xi) \wt\EXP^\xi\left(1/h_{N+1}^{\xi}\right)+O(\theta^N), \text{ as $N\to\infty$}$$
where the big oh's are uniform for $|\xi|\leq R$. Plugging this into \eqref{VNFP} gives
$$
V_N\mathfs F_N'(\xi)=\wt{\E}^\xi(S_N)+O(1)\text{ as $N\to\infty$, uniformly for $|\xi|\leq R$}.
$$
Part (6) follows from this from part (2) of the lemma.
\qed
\end{proof}

\medskip
\noindent
{\bf The choice of $\xi_N$:} We choose $\xi_N$ so that
$
{P}_N'(\xi_N)=z_N\ , \ \wt{\E}^{\xi_N}(S_N)=z_N+O(1).
$
The following lemma gives sufficient conditions for the existence of such $\xi_N$.

\begin{lemma}\label{Lemma-xi-N-exists}
Suppose $V_N\to\infty$, $R>0$,  and
$$
[\ha_N^R, \hb_N^R]:=
\left[\mathfs F_N'(-R)-\frac{\E(S_N)}{V_N},
\mathfs F_N'(R)-\frac{\E(S_N)}{V_N}\right].
$$
\begin{enumerate}[(1)]
\item For each $R$ there is $C(R)$, $N(R)$  s.t. if $\frac{z_N-\E(S_N)}{V_N} \in [\ha_N^R, \hb_N^R]$, and $N>N(R)$ then
\begin{enumerate}[(a)]
\item   $\exists! \xi_N\in [-(R+1),(R+1)]$ s.t. ${P}_N'(\xi_N)=z_N$;
\item $C(R)^{-1}\left|\frac{z_N-\E(S_N)}{V_N}\right|\leq |\xi_N|\leq C(R)\left|\frac{z_N-\E(S_N)}{V_N}\right|$;
\item  $\mathrm{sgn}(\xi_N)=\mathrm{sgn}(\frac{z_N-\E(S_N)}{V_N})$;
\item $\left|\wt{\E}^{\xi_N}(S_N)-z_N\right|\leq C(R)$.
\end{enumerate}
\item For every $R>1$
there exists $c(R)>0$  such that for all $N$ large enough,
\begin{equation}
\label{HRange}
\text{if $\left|\frac{z_N-\E(S_N)}{V_N}\right|\leq c(R)$, then  $\frac{z_N-\E(S_N)}{V_N}\in [\ha_N^R, \hb_N^R]$}.
\end{equation}
\end{enumerate}
Consequently, if $|\frac{z_N-\E(S_N)}{V_N}|<c(R)$, then there exists a unique $\xi_N$ with (a)--(d) above.
\end{lemma}

\begin{proof} Let
 $\DS
[\ta_N^R, \tb_N^R]:=
\left[\frac{P_N'(-R)-\E(S_N)}{V_N},
\frac{P_N'(R)-\E(S_N)}{V_N}\right].
$

\noindent
{\sc Claim:\/} {\em For all $R>0$, for all $N$ large enough,
$$
[\ha_N^R,\hb_N^R]\subset [\ta_N^{R+1},\tb_N^{R+1}]\subset [\ha_N^{R+2},\hb_N^{R+2}].
$$}
\noindent
{\em Proof of the claim:\/} By parts (3) and (4) of Lemma \ref{Lemma-Changed-Expectation-Variance}, there exists $\delta>0$ such that $P_N''(\xi)/V_N\geq \delta$ on $[-(R+2),(R+2)]$. Thus by the mean value theorem,
$$
\tb_N^{R+2}\geq \tb_N^{R+1}+\delta,\ \  \tb_N^{R+1}\geq \tb_N^{R}+\delta,\ \  \ta_N^{R+2}\leq \ta_N^{R+1}-\delta,\ \ \ta_N^{R+1}\leq \ta_N^{R}-\delta.
$$
Next by part (6) of Lemma \ref{Lemma-Changed-Expectation-Variance},
$
|\hb_N^{R'}-\tb_N^{R'}|=O(V_N^{-1})$ and $|\ha_N^{R'}-\ta_N^{R'}|=O(V_N^{-1})$ for all $R'\leq R+2.
$
For all $N$ large enough  $|O(V_N^{-1})|<\delta$, and
$$\ha_N^{R+2}<\ta_N^{R+1}<\ha_N^{R}<\hb_N^{R}<\tb_N^{R+1}<\hb_N^{R+2},$$
which proves the claim.

\medskip
We can now prove  part (1) of the lemma. Let
$\DS\vf_N(\xi):=\frac{{P}_N(\xi)-\xi{P}_N'(0)}{V_N}.$
By Lemma \ref{Lemma-Changed-Expectation-Variance}, $\vf_N(\xi)$ is strictly convex, smooth, and
$$
{P}_N'(\xi_N)=z_N \quad \text{iff}\quad \vf_N'(\xi_N)=\frac{z_N-{P}_N'(0)}{V_N}.
$$

Fix $R>0$. By the claim, for all $N$ large enough, if $\frac{z_N-\E(S_N)}{V_N}\in [\ha_N^R,\hb_N^R]$, then
$
\frac{z_N-P_N'(0)}{V_N}\equiv \frac{z_N-\E(S_N)}{V_N}\in [\ta_N^{R+1},\tb_N^{R+1}]\equiv\vf_N'[-(R+1),(R+1)].
$
Since $\vf_N'$ is continuous and strictly increasing, there  $\exists!\xi_N\in [-(R+1),(R+1)]$ such that $\vf_N'(\xi_N)=\frac{z_N-P_N'(0)}{V_N}$. Equivalently, there exists a unique $|\xi_N|\leq R+1$  such that $P_N'(\xi_N)=z_N$.

This argument shows that for every  $N$ sufficiently large, for every $\eta\in [\ha_N^R,\hb_N^R]$
there exists a unique  $\xi=\xi(\eta)\in [-(R+1),(R+1)]$ such that $$\vf_N'(\xi(\eta))=\eta.$$

By Lemma \ref{Lemma-Changed-Expectation-Variance},  $\exists \delta(R)>0$ so that $\delta(R)\leq \vf_N''\leq \delta(R)^{-1}$ on $[-(R+1),(R+1)]$. So $\eta\mapsto\xi(\eta)$ is
 $\frac{1}{\delta(R)}$-bi-Lipschitz on $[\ha_N^R,\hb_N^R]$.
 By construction, $\vf_N'(0)=0$. So $\xi(0)=0$, whence by the bi-Lipschitz property
$$
\delta(R)|\eta|\leq |\xi(\eta)|\leq \delta(R)^{-1}|\eta| \text{ on }[\ha_N^R,\hb_N^R].
$$
Since $\vf_N$ is real-analytic and strictly convex, $\vf_N'$ is smooth and strictly increasing. By the inverse mapping theorem, $\eta\mapsto\xi(\eta)$ is smooth and strictly increasing. So
$$
\mathrm{sgn}(\xi(\eta))=\mathrm{sgn}(\eta)\text{ on }[\ha_N^R,\hb_N^R].
$$
Specializing to the case $\eta=\frac{z_N-\E(S_N)}{V_N}$, gives properties (a)--(c) of $\xi_N$.

Property (d) is because of by  Lemma \ref{Lemma-Changed-Expectation-Variance}, which  says that
$$z_N=P_N'(\xi_N)=\wt{E}^{\xi_N}(S_N)+O(1) . $$
Notice that the big oh is uniform because $|\xi_N|\leq R+1$. This completes the proof of part (1).

\medskip
Here is the proof of part (2): For every $R>1$, for all $N$ large enough
\begin{align*}
&[\ha_N^R,\hb_N^R]\supset [\ta_N^{R-1},\tb_N^{R-1}]\equiv \vf_N'[-(R-1),(R-1)]\text{ ($\because$ claim, $\vf_N'$ is increasing)}\\
&\supset [-\delta(R-1)(R-1),\delta(R-1)(R-1)] \ \ (\because\vf_N'(0)=0, \vf_N''\geq \delta(R+1)).
\end{align*}
So $[\ha_N^R,\hb_N^R]\supset [-c,c]$
 for $R\geq 2$  where $c:=\delta(1)$.
\qed
\end{proof}

\begin{corollary}\label{Cor-xi-tends-to-zero}
Suppose $V_N\to\infty$ and $\frac{z_N-\E(S_N)}{V_N}\to 0$, then for all $N$ large enough, there exists a unique $\xi_N$ such that $P_N'(\xi_N)=z_N$. Furthermore, $\xi_N\to 0$.
\end{corollary}

\subsection{The asymptotic behavior of $\wt{V}^{\xi_N}(S_N)$}
\label{SSAsymVarLD}
Let  $\wt{V}^\xi_N$ denote  the variance of $S_N$ with respect to the change of measure $\mathsf X^\xi$. We compare $\wt{V}^{\xi}_N$ to $V_N$.
\begin{lemma}\label{Lemma-Variance-Asymp}
Suppose $V_N\xrightarrow[N\to\infty]{}\infty$, and define $\xi_N$ as in Lemma \ref{Lemma-xi-N-exists}.
\begin{enumerate}[(1)]
\item  Suppose $R>0$ and $\frac{z_N-\E(S_N)}{V_N}\in [\ha_N^R, \hb_N^R]$ for all $N$,
then $\wt{V}^{\xi_N}_N\asymp V_N$ as $N\to\infty$.
\item If $\frac{z_N-\E(S_N)}{V_N}\to 0$, then $\wt{V}^{\xi_N}_N\sim V_N$ as $N\to\infty$.
    \item $\wt{V}^{\xi}_N\sim V_N$ as $N\to\infty$ uniformly on compact subsets of $\xi$:  For every $\epsilon>0$ there are $\xi^\ast>0$ and $N_0>1$, so that  $\wt{V}^\xi_N/V_N\in [e^{-\epsilon},e^{\epsilon}]$ for all  $|\xi|<\xi^\ast, N>N_0$.
\end{enumerate}
\end{lemma}
\begin{proof} Part (1) is because of  Lemma \ref{Lemma-Changed-Expectation-Variance}(3) and the bound $|\xi_N|\leq R+1$ from Lemma~\ref{Lemma-Changed-Expectation-Variance}. Part (2) follows from  part (3) and Corollary \ref{Cor-xi-tends-to-zero}. It remains to prove part (3).

To do this we decompose $S_N$ into weakly correlated large blocks of roughly the same $\mathsf X$-variance, and check that the $\mathsf X^\xi$-variance of the $i$-th block converges {\em uniformly} in $i$ to its $\mathsf X$-variance.

Let $\{X_n\}$ and $\{\wt{X}_n^\xi\}$ denote the Markov chains with transition kernels $\{\pi_{n,n+1}(x,dy)\}$, $\{\wt{\pi}^\xi_{n,n+1}(x,dy)\}$ and initial distribution $\mu_1(dx)$. Given natural numbers  $n>m$, let
\begin{align*}
S_{n,m}&:=X_n+\cdots+X_{m-1}\\
\wt{S}_{n,m}^\xi&:=\wt{X}^\xi_n+\cdots+\wt{X}^\xi_{m-1}\\
p_{n,m}(\xi)&:=p_n(\xi)+\cdots+p_{m-1}(\xi).
\end{align*}
Notice that for all $R>0$, $n<m$, and  $|\xi|\leq R$,
\begin{equation}\label{p-prime}
p_{n,m}(0)=0,\quad p_{n,m}'(0)=\E(S_{n,m}),\quad
|p_{n,m}'(\xi)-\wt{\E}^{\xi}(\wt{S}^\xi_{n,m})|\leq C(R).
\end{equation}
The first identity is because $\ov{h}_n(\cdot,0)\equiv 1$, $\ov{p}_n(0)=1$ by the uniqueness of the fundamental solution.
The second identity is because $$p_{n,m}'(0)=P_{m-1}'(0)-P_{n-1}'(0)=\E(S_{m-1})-\E(S_{n-1})$$
by choice of $\{a_n\}$. The inequality can be proved by  applying Lemma \ref{Lemma-Changed-Expectation-Variance} to the shifted Markov chain $\{X_k\}_{k\geq n}$.

Let $V(S_{n,m}):=\Var(S_{n,m})$.
The application of Lemma \ref{Lemma-Changed-Expectation-Variance} to the shifted Markov chain $\{X_k\}_{k\geq n}$ also gives a constant $M_0$ s.t. for all $|\xi|\leq R$,
\begin{equation}\label{p-double-prime}
V(S_{n,m})\geq M_0\Rightarrow\begin{cases}
C(R)^{-1}\leq \wt{V}^{\xi}(\wt{S}^\xi_{n,m})/V(S_{n,m})\leq C(R) &\\
2^{-1}\leq p_{n,m}''(\xi)/\wt{V}^{\xi}(\wt{S}^\xi_{n,m})\leq 2.
\end{cases}
\end{equation}
 $M_0$ is independent of $n$: It is a function of $R$, $K$, $\epsilon_0$, and the uniform bounds on $h_n(\cdot,\xi)$ and its derivatives.

\medskip
\noindent
{\sc Step 1 (uniform exponential mixing).} {\em There are
$C_{mix}^\ast=C_{mix}(R)>0$, $\eta=\eta(R)\in (0,1)$ such that for every $|\xi|\leq R$, for all $n<m$, }
$$
\bigl|
\mathrm{Cov}\bigl(f_m(\wt{X}_m^\xi,\wt{X}_{m+1}^\xi),f_n(\wt{X}_n^\xi,\wt{X}_{n+1}^\xi)\bigr)
\bigr|\leq C_{mix}^\ast\eta^{m-n}.
$$

\medskip
\noindent
Proof: If $|\xi|\leq R$, then the Markov chain $\wt{X}^\xi_n$ is uniformly elliptic with ellipticity constant $\epsilon_0(R)>0$.
The step follows from Proposition \ref{Proposition-Exponential-Mixing}.

\medskip
\noindent
{\sc Step 2 (Block decomposition).} {\em For every $\epsilon>0$ small enough, for every $R>1$, there exists  $M>1$ and integers $n_i\uparrow\infty$ such that:
\begin{enumerate}[(1)]
\item $M\leq V(S_{n_i,n_{i+1}})\leq 2M$;
\item $|\Cov(\wt{S}^\xi_{n_i, n_{i+1}},\wt{S}^\xi_{n_j, n_{j+1}})|\leq C_{mix}^\# \eta^{n_j-n_{i+1}}$ for all $|\xi|\leq R$ and $i<j$, where the constant $C_{mix}^\#$ is independent of $M,i,j$;
    \item For all $|\xi|\leq R$, for all  $i>3$, for all $n\in [n_i,n_{i+1}]$,
\begin{equation}\label{block-variance-estimate}
e^{-\epsilon}\leq \frac{\wt{V}^\xi(\wt{S}^\xi_{1,n})}{\DS \sum_{k=1}^{i-1}\wt{V}^\xi(\wt{S}^{\xi}_{n_k , n_{k+1}})+\wt{V}^\xi_{n_i,n}}\leq e^{\epsilon}.
\end{equation}
\item $\DS M^\ast:=\sup_i \sup_{n\in [n_i,n_{i+1}]} \sup_{|\xi|\leq R}|p_{n_i,n}''(\xi)|<\infty$
\end{enumerate}
}

\medskip
\noindent
{\em Proof.\/} We write $V_{n,m}:=V(S_{n,m})$ and
$\wt{V}^\xi_{n,m}:=\wt{V}^\xi(\wt{S}_{n,m}^\xi)$, and fix
$$
M>\max\left\{2\left(K^2+\frac{C^\ast_{mix}}{1-\eta}\right),\frac{4 C_{mix}^\ast C(R)}{\epsilon^{-1}(1-\eta)^3}\right\}.
$$

Construct $n_i=n_i(M)\in\N$ by induction as follows: $n_1:=1$, and
$$
n_{i+1}:=\min\{n>n_i: V_{n_i,n_{i+1}}>M\}.
$$
There does indeed exist $n>n_i$ with $V_{n_i,n_{i+1}}>M$, because $V_{n_i,n}\xrightarrow[n\to\infty]{}\infty$, as can be seen from the following calculation:
\begin{align*}
\infty\xleftarrow[\infty\leftarrow\ n]{} V_{1,n}&=V_{1,n_i}+V_{n_i,n}+2\Cov(S_{1,n_i},S_{n_i,n})\\
&=V_{n_i,n}+V_{1,n_i}+O\left(\sum_{m=1}^{n_i-1}\sum_{k=0}^\infty |\Cov(X_m,X_{n_i+k})|\right)\\
&=V_{n_i,n}+O(1),\text{ by step 1 with $\xi=0$.}
\end{align*}
By construction, $V_{n_i,n_{i+1}}>M$, and
\begin{align*}
&V_{n_i, n_{i+1}} \leq V_{n_i,n_{i+1}-1}+|V_{n_i,n_{i+1}}-V_{n_i,n_{i+1}-1}|\\
&\leq M+|V_{n_i,n_{i+1}}-V_{n_i,n_{i+1}-1}| \ \ \text{ by the minimality of $n_{i+1}$}\\
& \leq M+V(f_{n_{i+1}-2}(X_{n_{i+1}-2},X_{n_{i+1}-1}))\\
&\ \ \ \ \ +2|\Cov (f_{n_{i+1}-2}(X_{n_{i+1}-2},X_{n_{i+1}-1}),S_{n_{i+1}-1})| \\
&\leq M+2\left(K^2+\frac{C^\ast_{mix}}{1-\eta}\right)\leq 2M\ \ \text{ by the choice of $M$}.
\end{align*}
So $M<V_{n_i,n_{i+1}}\leq 2M$, and  $\{n_i\}$ satisfies part (1).

\medskip
If $i<j$, then $\displaystyle |\Cov(\wt{S}^\xi_{n_i, n_{i+1}},\wt{S}^\xi_{n_j, n_{j+1}})|\leq \sum_{k=n_i}^{n_{i+1}-1}\sum_{\ell=n_j}^{n_{j+1}-1} C_{mix}^\ast \eta^{\ell-k}$
\begin{align*}
&\leq C_{mix}^\ast
\sum_{k=n_i}^{n_{i+1}-1}\frac{\eta^{n_{j}-k}}{1-\eta}
=\frac{C^\ast_{mix}\eta^{n_j-n_{i+1}}}{1-\eta}\sum_{k=n_i}^{n_{i+1}-1}\eta^{n_{i+1}-k}
=\frac{C_{mix}^\ast \eta^{n_j-n_{i+1}}}{(1-\eta)^2}.
\end{align*}
Part (2) follows with $C_{mix}^\#:={C_{mix}^\ast}/{(1-\eta)^2}$.

\medskip
Part (3) follows from parts (1),(2). Namely,  fix $n\in [n_i, n_{i+1}]$, then
\begin{align*}
&\left|\wt{V}^\xi(\wt{S}^\xi_{1,n})-\sum_{k=1}^{i-1}\wt{V}^\xi_{n_k, n_{k+1}}-\wt{V}^\xi_{n_{i}, n}\right|\\
&\leq 2\!\!\!\!\sum_{1\leq k<\ell\leq i-1}|\Cov(\wt{S}^\xi_{n_k, n_{k+1}},\wt{S}^\xi_{n_\ell, n_{\ell+1}})|
+2\sum_{1\leq k\leq i-1}|\Cov(\wt{S}^\xi_{n_k,n_{k+1}},\wt{S}^\xi_{n_i,n}) |\\
&\leq 2\sum_{1\leq k<\ell\leq i-1} C_{mix}^\#\eta^{n_\ell-n_{k+1}}+2\sum_{1\leq k\leq i-1} C_{mix}^{\#}\eta^{n_i-n_{k+1}}
\leq 2\sum_{1\leq k<\ell\leq i} C_{mix}^\#\eta^{\ell-k-1}\\
&=2C_{mix}^\# \sum_{k=1}^{i-1}\sum_{\ell=k+1}^{i-1} \eta^{\ell-k-1}
\leq \frac{2C_{mix}^\#i}{1-\eta}=\frac{2C_{mix}^\ast i}{(1-\eta)^3}.
\end{align*}

By \eqref{p-double-prime},
$\displaystyle
\sum_{k=1}^{i-1}\wt{V}^\xi_{n_k n_{k+1}}\geq \frac{M(i-1)}{C(R)}.
$
So
$$\displaystyle
\left|\frac{\wt{V}^\xi(\wt{S}^\xi_{n_i-1})}{\sum_{k=1}^{i-1}\wt{V}^\xi_{n_k n_{k+1}}+\wt{V}^\xi_{n_i,n}}-1\right|\leq \frac{\left(\frac{2C_{mix}^\ast }{(1-\eta)^3}\right)i}{C(R)^{-1}M(i-1)}
\leq \frac{1}{M}\cdot \frac{2 C_{mix}^\ast C(R)}{\left(1-\eta\right)^3}\cdot \frac{i}{i-1}
\leq \frac{\epsilon}{2}\cdot \frac{i}{i-1}
,$$
where the last inequality is by the choice of $M$. If $i>3$, the last bound is less than $\frac{3}{4}\epsilon$, and \eqref{block-variance-estimate} follows for all $\epsilon$ sufficiently small.

\medskip
Part (4) is a uniform bound on $|p_{n_i,n}''(\xi)|$ for $i\in\N$, $n\in [n_i,n]$, $|\xi|\leq R$.
By construction, $V_{n_i,n}\leq 2M$. By Theorem \ref{Theorem-MC-Variance}, this implies a uniform upper bound on $\sum_{k=n_i}^{n-1} u_k^2$. The structure constants of $\{X_n\}$ and $\{\wt{X}_n^\xi\}$ are equal up to a bounded multiplicative error. So the same theorem, applied to the Markov chain $\{\wt{X}^\xi_k\}_{k\geq n_i}$, gives a uniform upper bound for $\wt{V}^{\xi}_{n_i,n}$, whence
$
\sup_i \sup_{n\in [n_i,n_{i+1}]} \sup_{|\xi|\leq R} \wt{V}^\xi_{n_i,n}<\infty.
$

A routine modification of the argument we used to show \eqref{pergolesi} shows that
$$
\left|p_{n,m}''(\xi)-\wt{\E}^\xi\left[\left(\wt{S}^\xi_{n_i,n}-\wt{\E}^\xi(\wt{S}^\xi_{n_i,n})+O(1)\right)^2\right]\right|\leq \text{const}
.$$
The expectation term is uniformly bounded because of the bound on $ \wt{V}^\xi_{n_i,n}$ and the Minkowski inequality, so part (4) follows.

\medskip
\noindent
{\sc Step 3 (Block expectation).} {\em For every $\epsilon>0$ there exists $\xi^\ast>0$ such that for all $|\xi|\leq \xi^\ast$,}
 $$
\left|\wt{\E}^\xi(\wt{S}^\xi_{n_i,n})-\E(S_{n_i,n})\right|\leq \epsilon\text{ for all $i\in\N$ and  $n_i\leq n\leq n_{i+1}$.}
$$

\medskip
\noindent
{\em Proof.}
By Lemma \ref{Lemma-h-bounded} $h_k(\cdot,\xi)$ is uniformly bounded away from zero and infinity when $|\xi|\leq R$. By Lemma \ref{Lemma-Analyticity}, $\xi\mapsto h_k(\cdot,\xi)$ is uniformly Lipschitz on $[-R,R]$. It follows that
$$
\frac{h_{n+1}(y,\xi)}{h_{n_i}(x,\xi)}\xrightarrow[\xi\to 0]{}1\text{ uniformly for $i\in\N$, $n\in [n_i,n_{i+1}]$, $(x,y)\in \mathfrak S_{n_i}\times\mathfrak S_{n+1}$.}
$$
In particular, there is a $\xi^\ast_1$ s.t. for all $|\xi|\leq \xi^\ast_1$
$$
2^{-1}\leq \frac{h_{n+1}(y,\xi)}{h_{n_i}(x,\xi)}\leq 2 \text{ for all $i\in\N$, $n\in
 [n_i,n_{i+1}]$, and $(x,y)\in\mathfrak S_{n_i}\times\mathfrak S_n$.}
$$
This has a useful consequence. Since
$$\E\left(e^{\xi S_{n_i,n}}\frac{h_{n+1}(X_{n+1},\xi)}{e^{p_{n_i,n}(\xi)}h_{n_i}(X_{n_i},\xi)}\right)=\E\left(\E_{X_{n_i}}\left(e^{\xi S_{n_i,n}}\frac{h_{n+1}(X_{n+1},\xi)}{e^{p_{n_i,n}(\xi)}h_{n_i}(X_{n_i},\xi)}\right)\right)=\E(1)=1,$$
\begin{equation}\label{exponential-moment}
2^{-1}\E\left(e^{\xi S_{n_i,n}}\right)\leq e^{p_{n_i,n}(\xi)}\leq 2\E\left(e^{\xi S_{n_i,n}}\right)\text{ whenever }|\xi|\leq \xi^\ast_1.
\end{equation}

Fix $L>0$ and let $A_L:=[|S_{n_i,n}-\E(S_{n_i,n})|\leq L]$, then:
\begin{align*}
&\wt{\E}^{\xi}_{ X_{n_i}}(\wt{S}^{\xi}_{n_i,n})-\E({S}_{n_i,n})={\E}\left((S_{n_i,n}-\E({S}_{n_i,n})) e^{\xi S_{n_i,n}}\cdot \frac{h_{n+1}(X_{n+1},\xi)}{e^{p_{n_i,n}(\xi)}h_{n_i}(X_{n_i},\xi)}\right)\\
&=\E_{ X_{n_i}}\left((S_{n_i,n}-\E({S}_{n_i,n}))e^{\xi S_{n_i,n}-p_{n_i,n}(\xi)}\cdot \frac{h_{n+1}(X_{n+1},\xi)}{h_{n_i}(X_{n_i},\xi)} \cdot 1_{A_L}\right)\\
&\ \ \ \ \ + \E_{ X_{n_i}}\left((S_{n_i,n}-\E({S}_{n_i,n}))e^{\xi S_{n_i,n}-p_{n_i,n}(\xi)}\cdot \frac{h_{n+1}(X_{n+1},\xi)}{h_{n_i}(X_{n_i},\xi)} \cdot 1_{A^c_L}\right).
\end{align*}
{\em Expectation of the first summand\/}:  $M^\ast:=\sup\limits_i \sup\limits_{n\in [n_i,n_{i+1}]} \sup\limits_{|\xi|\leq R} |p_{n_i,n}''(\xi)|<\infty$. Therefore by \eqref{p-prime}, for all $|\xi|\leq R$, $n\in [n_i,n_{i+1}]$,
\begin{align}
p_{n_i,n}(\xi)&=p_{n_i,n}(0)+\xi p_{n_i,n}'(0)+O(\xi^2)
=\xi{\E}(S_{n_i,n})+O(\xi^2),\label{p-n-n}
\end{align}
where $|O(\xi^2)|\leq M^\ast\xi^2$.

So on $A_L$, $|\xi S_{n_i,n}-p_{n_i,n}(\xi)|\leq |\xi|\cdot |S_{n_i,n}-{\E}(S_{n_i,n})|+M^\ast\xi^2\leq L|\xi|+M^\ast\xi^2$, uniformly in $i,n\in [n_i,n]$. In particular,
$$
e^{\xi S_{n_i,n}-p_{n_i,n}(\xi)}\xrightarrow[\xi\to 0]{}1\text{ on $A_L$, uniformly in  $i$, $n\in [n_i,n]$. }
$$
Together with the uniform convergence $\frac{h_{n+1}(X_{n+1},\xi)}{h_{n_i}(X_{n_i},\xi)}\xrightarrow [\xi\to 0]{}1$, this implies that the first summand converges to
$\E_{X_{n_i}}[(S_{n_i,n}-\E(S_{n_i,n}))1_{A_L}]$ uniformly in $i$, $X_{n_i}$, and $n\in [n_i,n_{i+1}]$.

The expectation of the  limit satisfies
$$
|\E[(S_{n_i,n}-\E(S_{n_i,n}))1_{A_L}]|
=|\E[(S_{n_i,n}-\E(S_{n_i,n}))1_{A_L^c}]|$$
$$\leq
\E[L^{-1}(S_{n_i,n}-\E(S_{n_i,n}))^2 1_{A_L^c}] \leq
\frac{V(S_{n_i,n})}{L}\leq \frac{2M}{L}.
$$
Thus, for every $\epsilon>0$, for every $L$ large enough,   for all $|\xi|$ sufficiently small, for all $i, n\in [n_i,n_{i+1}]$, the first summand has expectation   $<\epsilon/2$.

\medskip
\noindent
{\em Expectation of the second summand\/}: Fix $0<\delta\ll \xi_1^\ast$.  Assume $L$ is so large  s.t.
$|t|<\delta e^{\delta|t|}$ for all $|t|>L$.

Decompose $A^c_L:=A^c_+\uplus A^c_-$, where $A^c_+:=[S_{n_i,n}-\E({S}_{n_i,n})>L]$ and $A^c_+:=[S_{n_i,n}-\E({S}_{n_i,n})<-L]$. Then
\begin{align*}
& \E_{ X_{n_i}}\left(|S_{n_i,n}-\E({S}_{n_i,n})|e^{\xi S_{n_i,n}-p_{n_i,n}(\xi)}\cdot \frac{h_{n+1}(X_{n+1},\xi)}{h_{n_i}(X_{n_i},\xi)} \cdot 1_{A^c_+}\right)\\
&\leq 2\E_{ X_{n_i}}\left(|S_{n_i,n}-\E({S}_{n_i,n})|e^{\xi S_{n_i,n}-p_{n_i,n}(\xi)}\cdot 1_{A^c_+}\right),\text{  provided $|\xi|\leq \xi^\ast_1$}\\
&\leq 4\E_{X_{n_i}}\left((S_{n_i,n}-\E({S}_{n_i,n}))e^{\xi S_{n_i,n}}\cdot 1_{A^c_+}\right)\big/\E(e^{\xi S_{n_i,n}}),\text{  by \eqref{exponential-moment}}\\
&=4\E_{ X_{n_i}}\left((S_{n_i,n}-\E({S}_{n_i,n}))e^{\xi (S_{n_i,n}-\E(S_{n_i,n}))}\cdot 1_{A^c_+}\right)\bigg/\E(e^{\xi (S_{n_i,n}-\E(S_{n_i,n}))})\\
&\leq 4\delta \E_{ X_{n_i}}(e^{(\xi+\delta) (S_{n_i,n}-\E(S_{n_i,n}))})
\big/\E(e^{\xi (S_{n_i,n}-\E(S_{n_i,n}))})\\
&\leq 16\delta \exp\left({p_{n_i,n}(\xi+\delta)-p_{n_i,n}(\xi)-\delta\E(S_{n_i,n})}\right),\text{ provided $|\xi+\delta|<\xi^\ast_1$}.
\end{align*}
(see \eqref{exponential-moment}).
Expanding $p_{n_i,n}(\xi+\delta)$ into Taylor series around $\xi$, and recalling $|p_{n_i,n}''(\xi)|\leq M^\ast$ for $|\xi|\leq R$, we find that the term in the exponent is bounded above by
\begin{align*}
&\delta|p_{n_i,n}'(\xi)-\E(S_{n_i,n})|+M^\ast \delta^2=\delta|p_{n_i,n}'(\xi)-p_{n_i,n}'(0)|+M^\ast\delta^2\\
&\leq M^\ast(\delta|\xi|+\delta^2)\leq M^\ast (R\delta+\delta^2),
\end{align*}
which can be made as small as we wish by choosing $\delta$ properly.

The conclusion is that for all $L$ large enough, for all $|\xi|$ sufficiently small, for all $i, n\in [n_i,n_{i+1}]$,
$$
\E\left(|S_{n_i,n}-\E({S}_{n_i,n})|e^{\xi S_{n_i,n}-p_{n_i,n}(\xi)}\cdot \frac{h_{n+1}(X_{n+1},\xi)}{h_{n_i}(X_{n_i},\xi)} \cdot 1_{A^c_+}\right)\leq \frac{\epsilon}{4}.$$
Similarly, one can show that for all $L$ large enough, for all $|\xi|$ sufficiently small, for all $i, n\in [n_i,n_{i+1}]$,
$$
\E\left(|S_{n_i,n}-\E({S}_{n_i,n})|e^{\xi S_{n_i,n}-p_{n_i,n}(\xi)}\cdot \frac{h_{n+1}(X_{n+1},\xi)}{h_{n_i}(X_{n_i},\xi)} \cdot 1_{A^c_-}\right)<\frac{\epsilon}{4}.$$

Thus, for every $\epsilon>0$, for all $L$ sufficiently large, for all $|\xi|$ sufficiently small, for all $i, n\in [n_i,n_{i+1}]$, the  expectation of the second summand is less than $\epsilon/2$ in absolute value.

\medskip
\noindent
{\sc Step 4 (Block variance).} {\em For every $\epsilon>0$ there exists $\xi^\ast>0$ such that for all $|\xi|\leq \xi^\ast$,} $
\left|\wt{V}^\xi_{n_i,n}-V_{n_i,n}\right|\leq \epsilon\text{ for all $i\in\N$ and  $n_i\leq n\leq n_{i+1}$.}
$

\medskip
\noindent
{\em Proof.\/} The proof is similar to the proof of step 3. Fix $L$ to be determined later and let
$A_L:=[|S_{n_i,n}-\E(S_{n_i,n})|\leq L]$, then
\begin{align*}
&\wt{V}^\xi(\wt{S}_{n_i,n}^\xi)={\E}\left((S_{n_i,n}-\wt{\E}(\wt{S}_{n_i,n}^\xi))^2 e^{\xi S_{n_i,n}}\cdot \frac{h_{n+1}(X_{n+1},\xi)}{e^{p_{n_i,n}(\xi)}h_{n_i}(X_{n_i},\xi)}1_{A_L}\right)\\
&\hspace{2cm} +{\E}\left((S_{n_i,n}-\wt{\E}(\wt{S}_{n_i,n}^\xi))^2 e^{\xi S_{n_i,n}}\cdot \frac{h_{n+1}(X_{n+1},\xi)}{e^{p_{n_i,n}(\xi)}h_{n_i}(X_{n_i},\xi)}1_{A_L^c}\right).
\end{align*}

The second summand can be analyzed as in step 3, this time with the inequality $t^2<\delta e^{\delta|t|}$ for all $|t|$ large enough. The conclusion is that for every $\epsilon>0$, for all $L$ sufficiently large, for all $|\xi|$ sufficiently small, for all $i,n\in [n_i,n_{i+1}]$,
\begin{equation}
\label{SecSumWeight}
{\E}\left((S_{n_i,n}-\wt{\E}(\wt{S}_{n_i,n}^\xi))^2 e^{\xi S_{n_i,n}}\cdot \frac{h_{n+1}(X_{n+1},\xi)}{e^{p_{n_i,n}(\xi)}h_{n_i}(X_{n_i},\xi)}1_{A^c_L}\right)<\frac{\epsilon}{2}.
\end{equation}

The first summand converges to $\E((S_{n_i,n}-\E(S_{n_i,n}))^2 1_{A_L})$ as $\xi\to 0$ uniformly in $i, n\in [n_i,n_{i+1}]$ because
\begin{enumerate}[$\bullet$]
\item  $
e^{\xi S_{n_i,n}}\cdot \frac{h_{n+1}(X_{n+1},\xi)}{e^{p_{n_i,n}(\xi)}h_{n_i}(X_{n_i},\xi)}1_{A_L}\xrightarrow[\xi\to 0]{}1_{A_L}
$
uniformly in $i\in\N, n\in [n_i,n_{i+1}]$, see the proof of step 3; and
\item $(S_{n_i,n}-\wt{\E}({S}_{n_i,n}^\xi))^2 1_{A_L} \xrightarrow[\xi\to 0]{} (S_{n_i,n}-{\E}({S}_{n_i,n}))^2 1_{A_L}
$
uniformly in $i, n\in [n_i,n]$, because for some $t$ between $\wt{\E}(\wt{S}_{n_i,n}^\xi)$ and ${\E}({S}_{n_i,n})$, \begin{align*}
&\bigl|(S_{n_i,n}-\wt{\E}({S}_{n_i,n}^\xi))^2-(S_{n_i,n}-{\E}({S}_{n_i,n}))^2\bigr|\\ &=2|S_{n_i,n}-t||\wt{\E}({S}_{n_i,n}^\xi)-{\E}({S}_{n_i,n})|\\
&\leq 2(L+|\wt{\E}({S}_{n_i,n}^\xi)-{\E}({S}_{n_i,n})|)|\wt{\E}({S}_{n_i,n}^\xi)-{\E}({S}_{n_i,n})|\text{ on $A_L$}\\
&\xrightarrow[\xi\to 0]{}0\text{ uniformly on $A_L$ in $i\in\N, n\in [n_i,n_{i+1}]$, by step 3.}
\end{align*}
\end{enumerate}

The limit of the first summand $\E((S_{n_i,n}-\E(S_{n_i,n}))^2 1_{A_L})\xrightarrow[L\to\infty]{}V_{n_i,n}$ uniformly in  $i, n\in [n_i,n_{i+1}]$.
 Indeed, applying \eqref{SecSumWeight} with $\xi=0$
$$
|V_{n_i,n}-\E((S_{n_i,n}-\E(S_{n_i,n}))^2 1_{A_L})|=\E((S_{n_i,n}-\E(S_{n_i,n}))^2 1_{A_L^c})
<\frac{\epsilon}{2} $$
for all $L$ large enough, for all  $i\in\N, n\in [n_i,n_{i+1}]$.
 Step 4 follows.

\medskip
\noindent
{\sc Proof of part (3) of the Lemma.} Fix $\epsilon>0$, and construct the block decomposition as in step 2.

By step 4 there exists $\xi^\ast>0$ s.t. for all $|\xi|<\xi^\ast$, for all $k\in\N, n\in [n_k,n_{k+1}]$,
$
e^{-\epsilon} V_{n_k,n}\leq  \wt{V}^\xi_{n_k,n}\leq e^{\epsilon} V_{n_k,n}
$. Therefore
$$
e^{-\epsilon}\leq \frac{\sum_{k=1}^{i-1} \wt{V}^\xi_{n_k,n_{k+1}}+\wt{V}^\xi_{n_i,n}}
{\sum_{k=1}^{i-1}  {V}_{n_k,n_{k+1}}+{V}_{n_i,n}} \leq e^{\epsilon}.
$$
By part (3) of step 2, for all $n>n_{3}$, for all $|\xi|<\xi^\ast$,
$
e^{-3\epsilon}\leq {\wt{V}^\xi_{n}}\big/{V_n}\leq e^{3\epsilon}.
$\qed
\end{proof}

\subsection{Asymptotics of the log moment generating functions }
\label{Section-Proof-Log-Moment-Generating}
We need  an elementary observation from probability theory.
Let $X,Y$ be two random variables on the same probability space $(\Omega,\mathfs F,\Prob)$. Suppose  $X$ has finite non-zero variance, and $Y$ is positive and bounded. Let
 $\Var^Y(X)$ be the  variance of $X$ with respect to the change of measure $\frac{Y}{\E(Y)}d\Prob$, i.e.
$$
\Var^Y(X):=\frac{\E(X^2 Y)}{\E(Y)}-\left(\frac{\E(XY)}{\E(Y)}\right)^2.
$$

\begin{lemma}\label{Lemma-Var-Y}
Suppose $0<\Var(X)<\infty$ and $C^{-1}\leq Y\leq C$ with $C$ a positive constant, then
$
\displaystyle C^{-4}\leq \frac{\Var^Y(X)}{\Var(X)}\leq C^4.
$
\end{lemma}
\begin{proof}
For every random variable $W$, if $W_1,W_2$ are two independent copies of $W$ then $\Var(W)=\frac{1}{2}\E[(W_1-W_2)^2]$. In particular, if  $(X_1,Y_1)$, $(X_2,Y_2)$ are two independent copies of the random vector $(X,Y)$, then

$\DS
\Var^Y(X)=\frac{1}{2}\frac{\E[(X_1-X_2)^2 Y_1 Y_2]}{\E(Y_1 Y_2)}=C^{\pm 4}\frac{1}{2}\E[(X_1-X_2)^2]=C^{\pm 4}\Var(X).$
\qed
\end{proof}

\medskip
\noindent
{\bf Proof of Theorem \ref{Theorem-F_N} on the asymptotic behavior of $\mathfs F_N(\xi):=\frac{1}{V_N}\log\E(e^{\xi S_N})$:}\
Let $\mathsf f$ be an a.s. uniformly bounded additive functional on a uniformly elliptic Markov chain $X$, s.t. $V_N:=\Var(S_N)\neq 0$ for $N\geq N_0$.

Since $\|S_N\|_\infty<\infty$, we may differentiate under the expectation and obtain that for all $k$,
$\frac{d^k}{d\xi^k}\E(e^{\xi S_N})=\E(S_N^k e^{\xi S_N})$. A direct calculation now shows that
$$
\begin{aligned}
&\mathfs F_N'(\xi)=\frac{1}{V_N}\frac{\E(S_N e^{\xi S_N})}{\E(e^{\xi S_N})}=
\frac{1}{V_N}\E^{Y_N^\xi}(S_N),\\
&\mathfs F_N''(\xi)=\frac{1}{V_N}\left[\frac{\E(S_N^2 e^{\xi S_N})}{\E(e^{\xi S_N})}-\left(\frac{\E(S_N e^{\xi S_N})}{\E(e^{\xi S_N})}\right)^2\right]=\frac{\Var^{Y_N^\xi}(S_N)}{\Var(S_N)},\text{ where   $Y_N^\xi:=e^{\xi S_N}.$}
\end{aligned}
$$

\smallskip
\noindent
{\bf Part 1:} Substituting $\xi=0$ gives $\mathfs F_N(0)=0$,
$\mathfs F_N'(0)=\frac{\E(S_N)}{V_N}$, $\mathfs F_N''(0)=1$.

\medskip
\noindent
{\bf Part 2:} $\mathcal F_N''(\xi)=0$ $\Leftrightarrow$ $\Var^{Y_N^\xi}(S_N)=0$ $\Leftrightarrow$  $S_N=const$ $\frac{Y_N^\xi}{\E(Y_N^\xi)}d\Prob$--a.s. $\Leftrightarrow$
$S_N=const$ $\Prob$--a.s. $\Leftrightarrow$ $\Var(S_N)=0$. So $\mathcal F_N$ is strictly convex on $\R$ for all $N>N_0$.

\medskip
\noindent
{\bf Part 3:}
$
\displaystyle\frac{\wt{V}^\xi_N(S_N)}{\Var(S_N)}\equiv \frac{\Var^{Z_N}(S_N)}{\Var(S_N)},
$
where $Z_N^\xi:=e^{\xi S_N}\frac{h_{N+1}^\xi}{h_1^\xi}$ (the normalization constant  does not matter). Next,
$Z_N^\xi\equiv Y_N^\xi W_N^\xi$, where $W_N^\xi:=h^\xi_{N+1}/h^\xi_1$.
Lemma \ref{Lemma-h-bounded} says that for every $R>0$ there is a constant $C=C(R)$ s.t. $C^{-1}\leq W_N^\xi\leq C$ for all $N$ and $|\xi|\leq R$.  Lemma \ref{Lemma-Analyticity} and the obvious identity $h_n^0\equiv 1$ imply that $W_N^\xi\xrightarrow[\xi\to 0]{}1$ uniformly in $N$. So there is no loss of generality in assuming that  $C(R)\xrightarrow[R\to 0]{}1$.

By Lemma \ref{Lemma-Var-Y} with the probability measure $\frac{e^{\xi S_N}}{\E(e^{\xi S_N})}d\Prob$ and $Y=W_N^\xi$,
\begin{align}\label{F_N-V_N}
&\frac{\wt{V}^\xi_N(S_N)}{V_N\mathcal F_N''(\xi)}=
\frac{\Var^{Y_N^\xi W_N^\xi}(S_N)}{\Var^{Y_N^\xi}(S_N)}\in
\left[{C(R)^{-4}},C(R)^4\right],\text{ }\forall|\xi|\leq R, \;\;N\geq 1.
\end{align}

By Lemma \ref{Lemma-Changed-Expectation-Variance}(3),
$\wt{V}^\xi_N(S_N)\asymp V_N$ uniformly on compact sets of $\xi$, and by Lemma \ref{Lemma-Variance-Asymp}
for every $\epsilon$ there exists $\delta, N_\epsilon>0$ s.t. $e^{-\epsilon}<\wt{V}^\xi_N(S_N)/V_N<e^\epsilon$ for all $N>N_\epsilon$ and $|\xi|\leq \delta$.
It follows that for every $R$ there exists $C_2(R)>1$ such that $C_2(R)\xrightarrow[R\to 0]{}1$ and
$
C_2(R)^{-1}\leq \mathcal F_N''(\xi)\leq C_2(R)\text{ for all }|\xi|\leq R.
$

\medskip
\noindent
{\bf Part 4:}  Suppose $\epsilon>0$. We saw in part 3 that there exist $\delta, N_\epsilon$ s.t.
$
e^{-\epsilon}\leq \mathcal F_N''(\xi)\leq e^{\epsilon}$ for all $|\xi|\leq \delta, N\geq N_\epsilon.
$

Recall that   $\mathcal F_N(0)=0$ and   $\mathcal F_N'(0)=\E(S_N)/V_N$. So   for all $|\xi|\leq \delta$,
$$\mathcal F_N(\xi)=\mathcal F_N(0)+\int_0^\xi\left(\mathcal F_N'(0)+\int_{\mathcal F_N'(0)}^\eta \mathcal F_N''(\alpha)d\alpha\right) d\eta.$$
Since  $\mathfs F_N''=e^{\pm\epsilon}$ on $[-\delta,\delta]$ and $|\eta|\leq |\xi|\leq \delta$,  \\
$\DS \mathcal F_N(\xi)=\frac{\E(S_N)}{V_N}\xi+\frac{1}{2}e^{\pm \epsilon}
\left(\xi-\frac{\E(S_N)}{V_N}\right)^2$.
\hfill$\Box$

\subsection{Asymptotics of the rate functions.}\label{Section-Legendre}
The {\em rate functions} $\mathcal I_N(\eta)$ are the Legendre transforms of  $\mathcal F_N(\xi)=\frac{1}{V_N}\log\E(e^{\xi S_N})$.
Recall that  the Legendre transform \index{Legendre transform!definition} of a strictly convex function $\vf:\R\to\R$ is the function $\vf^\ast:(\inf\vf', \sup\vf')\to\R$,
$$
\vf^\ast(\eta)=\xi\eta-\vf(\xi)\text{ for the unique $\xi$ s.t. $\vf'(\xi)=\eta$}.
$$
On its domain, $\vf^\ast(\eta)=\max\{\xi\eta-\vf(\xi)\}$.

\begin{lemma}\label{LmDerLegendre}
Suppose $\vf(\xi)$ is strictly convex and twice differentiable on $\R$, and let $\vf'(\pm\infty):=\lim\limits_{\xi\to\pm\infty}\vf'(\xi)$. Then the Legendre transform $\vf^\ast$ is strictly convex and  twice differentiable on $(\vf'(-\infty),\vf'(+\infty))$, and for every $\xi\in\R$,
\begin{equation}\label{Legendre-Inverse-Fact}
\vf^\ast(\vf'(t))=t\vf'(t)-\vf(t),\
(\vf^\ast)'(\vf'(t))=t\text{ , }(\vf^\ast)''(\vf'(t))=\frac{1}{\vf''(t)}
\end{equation}
\end{lemma}
\begin{proof}
Under the assumptions of the lemma,  $\vf'$ is strictly increasing and differentiable. So $(\vf')^{-1}:(\vf'(-\infty),\vf'(\infty))\to\R$ is well-defined, strictly increasing and differentiable, and
$$
\vf^\ast(\eta)=\eta (\vf')^{-1}(\eta)-\vf[(\vf')^{-1}(\eta)]
$$
The lemma follows by differentiation of right-hand-side.
\qed
\end{proof}

\noindent
{\bf Proof of Theorem \ref{Theorem-I_N} on the asymptotics of the rate functions $\mathfs I_N:=\mathfs F_N^\ast$:}\

\medskip
\noindent
{\bf Part 1:} Since $\mathfs F_N$ is strictly convex and smooth, $\mathfs F_N'$ is strictly increasing and continuous. So
$\mathfs F_N'[-1,1]=[\mathfs F_N'(-1),\mathfs F_N'(1)]\equiv[a_N^1,b_N^1]$,  and
for every $\eta\in [a_N^1,b_N^1]$, there exists a unique $\xi\in [-1,1]$ such that $\mathcal F_N'(\xi)=\eta$.  So $\mathrm{dom}(\mathfs I_N)\supset [a_N^1,b_N^1]$.

By Theorem  \ref{Theorem-F_N} there is
 $C>0$ such that $C^{-1}\leq \mathcal F_N''\leq C$ on $[-1,1]$ for all $N\geq N_0$.
Since $\mathcal F_N'(0)=\frac{\E(S_N)}{V_N}$ and
$\mathcal F_N'(\rho)=\cF_N'(0)+\int_0^\rho \mathcal F_N''(\xi)d\xi$,  we have
$$
b_N^1\equiv \mathcal F_N'(1)\geq \frac{\E(S_N)}{V_N}+C^{-1}\ , \ a_N^1\equiv\mathcal F_N'(-1)\leq \frac{\E(S_N)}{V_N}-C^{-1}.$$
So
$
\mathrm{dom}(\mathcal I_N)\supseteq [a_N^1,b_N^1]\supseteq\left[\frac{\E(S_N)}{V_N}-C^{-1}, \frac{\E(S_N)}{V_N}+C^{-1}\right]\text{ for all }N\geq N_0.
$

\medskip
\noindent
{\bf Part 2} follows from Lemma \ref{LmDerLegendre}
and the strict convexity of $\cF_N$ on $[-R, R].$

\medskip
\noindent
{\bf  Part 3:} Let $J_N:=\left[\frac{\E(S_N)}{V_N}-C^{-1},
\frac{\E(S_N)}{V_N}+C^{-1}\right]$. In part 1 we constructed  functions
$\xi_N:J_N\to [-1,1]$ such that
$
\mathfs F_N'(\xi_N(\eta))=\eta.
$

Clearly $\xi_N\left(\frac{\E(S_N)}{V_N}\right)=0$.
Recalling that  $C^{-1}\leq \mathfs F_N''\leq C$ on $[-1,1]$, we see that $\xi_N'(\eta)=\frac{1}{\mathfs F_N''(\xi_N(\eta))}\in [C^{-1},C]$ on $J_N$. Hence
$$
|\xi_N(\eta)|\leq C|\eta-\tfrac{\E(S_N)}{V_N}|\text{ for all $\eta\in J_N$,  $N\geq N_0$}.
$$
Fix $0<\epsilon<1$. By
 Theorem \ref{Theorem-F_N}(4) there are $\delta, N_\epsilon>0$ s.t. $e^{-\epsilon}\leq \mathfs F_N''\leq e^\epsilon$  on $[-\delta,\delta]$ for all $N>N_\epsilon$. If
 $|\eta-\tfrac{\E(S_N)}{V_N}|<\delta/C$, then
 $|\xi_N(\eta)|<\delta$, and $\mathfs F_N''(\xi_N(\eta))\in [e^{-\epsilon},e^{\epsilon}]$.

Since $\mathcal F_N(0)=0$  and $\mathcal F_N'(0)=\frac{\E(S_N)}{V_N}$, we have by
 \eqref{Legendre-Inverse-Fact} that $\mathfs I_N(\frac{\E(S_N)}{V_N})=\mathfs I_N'(\frac{\E(S_N)}{V_N})=0$ and
 $\mathcal I_N''(\eta)=1/\mathcal F_N''(\xi_N(\eta))\in [e^{-\epsilon}, e^{\epsilon}].$
Writing
$$
\mathcal I_N(\eta)=\mathcal I_N(\tfrac{\E(S_N)}{V_N})+\int_{\tfrac{\E(S_N)}{V_N}}^\eta \left(\mathcal I_N'(\tfrac{\E(S_N)}{V_N})+\int_{\tfrac{\E(S_N)}{V_N}}^\alpha \mathcal I_N''(\beta)d\beta\right)d\alpha,
$$
we find that $\mathcal I_N(\eta)=e^{\pm \epsilon} \frac{1}{2}(\eta-\frac{\E(S_N)}{V_N})^2$ for all
$\eta$ s.t. $|\eta-\frac{\E(S_N)}{V_N}|\leq  \delta/C$.

\medskip
\noindent
{\bf  Part 4:} If $\frac{z_N-\E(S_N)}{V_N}\to 0$, then
$\DS \frac{z_N}{V_N}\in \left[\frac{\E(S_N)}{V_N}-\delta_N, \frac{\E(S_N)}{V_N}+\delta_N\right]$
with $\delta_N\to 0$. By part 3, $\mathcal I_N(\tfrac{z_N}{V_N})\sim\frac{1}{2}\left(\frac{z_N-\E(S_N)}{V_N}\right)^2$, whence
$V_N\mathcal I_N(\tfrac{z_N}{V_N})\sim\frac{1}{2}\left(\frac{z_n-\E(S_N)}{\sqrt{V_N}}\right)^2$.
\hfill$\Box$

\medskip
Let $H_N(\eta)$ denote the Legendre transform of $P_N(\xi)/V_N$.
We will compare $H_N(\eta)$  to $\mathcal I_N(\eta)$. This is needed to  link the change of measure we performed in section \S\ref{Section-xi} to the functions  $\mathcal I_N$ which appear in the statement of the local limit theorem for large deviations.

\begin{lemma}\label{Lemma-Legendre-Domain}
Suppose $R>0$ and $V_N\neq 0$ for all $N$ large enough. Then
\begin{enumerate}[(1)]
\item $H_N$ is well-defined and real-analytic on
$\left[\frac{P_N'(-R)}{V_N},\frac{P_N'(R)}{V_N}\right]$ for all $N$ large enough.

\item  There exists  $c>0$ such that
$H_N(\cdot)$ is well-defined and real-analytic on\\
$\left(\frac{\E(S_N)}{V_N}-c,\frac{\E(S_N)}{V_N}+c\right)$ for all $N$ large enough.
\end{enumerate}
\end{lemma}
\begin{proof}
 Lemma \ref{Lemma-xi-N-exists}
and its proof provide
real analytic maps
$$
\xi_N:
\biggl[\frac{P_N'(-R)}{V_N},\frac{P_N'(R)}{V_N}\biggr]\to [-R, R]\quad
\text{s.t.}\quad  \frac{P_N'(\xi_N(\eta))}{V_N}=\eta.
$$
Hence
$\DS H_N(\eta)=\frac{1}{V_N}\left[\xi_N(\eta) P_N'(\xi(\eta))-P_N(\xi(\eta))\right]$
is well-defined and real-analytic on the interval $[\frac{P_N'(-R)}{V_N},\frac{P_N'(R)}{V_N}]$.
This proves part (1). Part (2) follows from Lemma \ref{Lemma-xi-N-exists}(2).
\qed
\end{proof}

\begin{lemma}\label{Lemma-I_N-and-H_N}
Suppose $V_N\neq 0$ for all $N\geq N_0$,
then $\exists c>0$ such that
\begin{enumerate}[(1)]
\item $\mathrm{dom}(\mathcal I_N)\cap\mathrm{dom}(H_N)\supset\left[\frac{\E(S_N)}{V_N}-c, \frac{\E(S_N)}{V_N}+c\right]$ for all $N\geq N_0$.
\item  Recall that $[a_N^R, b_N^R]=[\mathfs F_N'(-R),\mathfs F_N'(R)]$.  For every $R>0$ there exists  $C(R)>0$ s.t. if $z/V_N\in[a_N^R, b_N^R]$ and $N\geq N_0$, then $$\bigl|V_N\mathcal I_N(\tfrac{z}{V_N})-V_N H_N(\tfrac{z}{V_N})\bigr|\leq C(R).$$
\item For every $\epsilon>0$,  $\exists \delta, N_\epsilon>0$ s.t. if $N\geq N_\epsilon$ and $\left|\frac{z-\E(S_N)}{V_N}\right|<\delta$, then
$$\bigl|V_N\mathcal I_N(\tfrac{z}{V_N})-V_N H_N(\tfrac{z}{V_N})\bigr|\leq \epsilon.$$
\end{enumerate}
\end{lemma}
\begin{proof}
Part (1) is a direct consequence of  Lemma \ref{Lemma-Legendre-Domain} and Theorem \ref{Theorem-I_N}(1).

 To prove the other parts of the lemma, we use the following consequence of
 Lemma \ref{Lemma-Changed-Expectation-Variance}(6):  For every $R>0$,  for all $N$ large enough, for every
$\eta\in  [a_N^R, b_N^R]$,
there exist
$\xi_N^{(1)}, \xi_N^{(2)}\in  [-(R+1), (R+1)]$ such that
$$
 \frac{P_N'(\xi_N^{(1)})}{V_N}=\eta\ ,\ \mathcal F_N'(\xi_N^{(2)})=\eta.$$
Arguing as in the proof of part 3 of Theorem \ref{Theorem-I_N}, we can also find a constant $C(R)$ such that $|\xi_N^{(i)}|\leq C(R)\bigl|\eta-\frac{\E(S_N)}{V_N}\bigr|$.

It is a general fact that the Legendre transform of a convex function $\vf$ is equal on its domain to  $\vf^\ast(\eta)=\sup\limits_{\xi}\{\xi\eta-\vf(\xi)\}$. Thus  for every
 $z\in [a_N^R V_N, b_N^R V_N]$,
\begin{align*}
&V_N\mathcal I_N\left(\frac{z}{V_N}\right)=V_N\sup_\xi\left\{\xi\frac{z}{V_N}-\mathcal F_N(\xi)\right\}
=V_N \left(\xi_N^{(2)}\frac{z}{V_N}-\mathcal F_N(\xi_N^{(2)})\right)\\
&\leq V_N\biggl(\xi_N^{(2)}\frac{z}{V_N}-\frac{P_N(\xi_N^{(2)})}{V_N}\biggr)
+\Delta_N\left(R+1\right), \text{ see Lemma \ref{Lemma-Changed-Expectation-Variance}(5)}\\
&\leq V_N\sup_\xi \left\{\xi\frac{z}{V_N}-\frac{P_N(\xi)}{V_N}\right\}+\Delta_N(R+1)
\equiv V_N H_N\left(\frac{z}{V_N}\right)+\Delta_N(R+1).
\end{align*}
So $V_N\mathcal I_N\bigl(\frac{z}{V_N}\bigr)-V_N H_N\bigl(\frac{z}{V_N}\bigr)
\leq \Delta_N(R+1)$.

Similarly, one can show that
$V_N H_N\bigl(\frac{z}{V_N}\bigr)- V_N\mathcal I_N\bigl(\frac{z}{V_N}\bigr)
\leq  \Delta_N(R+1)$, whence
$$
\sup_{N\geq N_0}\sup_{ z\in  \left[a_N^R V_N, b_N^R V_N\right]} \left|
V_N\mathcal I_N\left(\frac{z}{V_N}\right)-V_N H_N\left(\frac{z}{V_N}\right)
\right|\leq \sup_{N\geq N_0}\Delta_N(R+1).
$$
Part (2) now follows from Lemma \ref{Lemma-Changed-Expectation-Variance}(5).

If instead of taking  $z/V_N\in [a_N^R,b_N^R]$ we take
$\DS z/V_N\in \left(\frac{\E(S_N)}{V_N}-\delta,\frac{\E(S_N)}{V_N}+\delta\right)$,
then $|\xi_N^{(i)}|<C\delta$, and the same argument will show that
$$
\sup_{N\geq N_0}\sup_{\left|\frac{z-\E(S_N)}{V_N}\right|\leq \delta} \left|
V_N\mathcal I_N\left(\frac{z}{V_N}\right)-V_N H_N\left(\frac{z}{V_N}\right)
\right|\leq \sup_{N\geq N_0}\Delta_N(C\delta ).
$$
Part (3) follows from    Lemma \ref{Lemma-Changed-Expectation-Variance}(5).
\qed
\end{proof}

\subsection{The local limit theorem for large deviations.}
\label{SSLLT-LD-Proof}
\noindent
{\bf Proof of Theorem \ref{Thm-LLT-LDP}.} We give the proof in the non-lattice case; the modifications needed for the lattice case are routine.

Suppose $\mathsf f$ is an a.s. uniformly bounded additive functional of a uniformly elliptic Markov chain $\mathsf X$. We assume that $\mathsf f$ is irreducible, and that $\mathsf f$ has algebraic range $\R$.
In this case $\mathsf f$ is not center-tight, and $V_N:=\Var(S_N)\to\infty$ (see \S \ref{Section-Tight-Results}). There is no loss of generality in assuming that $V_N\neq 0$ for all $N$.

Recall that $[\ha_N,\hb_N]=[\mathfs F_N'(-R)-\frac{\E(S_N)}{V_N},\mathfs F_N'(R)-\frac{\E(S_N)}{V_N}]$, and suppose $$\frac{z_N-\E(S_N)}{V_N}\in [\ha_N,\hb_N].$$
Let  $h_n^\xi(\cdot):=h_n(\cdot,\xi)$, $p_n(\xi)$, and $P_N(\xi)$ be as in \S\S\ref{Section-h}, \ref{Section-xi}. The assumption on $z_N$ allows us to construct
$\xi_N\in \left[-(R+1), (R+1)\right]$ as in Lemma \ref{Lemma-xi-N-exists}:
$$
 P_N'(\xi_N)=z_N \text{ and }\xi_N=O\left(\tfrac{z_N-\E(S_N)}{V_N}\right).
$$

Define a Markov array
$\wt{\mathsf X}:=\{\wt{X}^{(N)}_n: 1\leq n\leq N+1\}$ with
state spaces $(\mathfrak S_n,\mathfs B(\mathfrak S_n),\mu_n)$ (the state spaces of $\mathsf X$), and transition probabilities
$$\wt{\pi}^{(N)}_{n,n+1}(x,dy):=e^{\xi_N f_n(x,y)}\frac{h_{n+1}(y,\xi_N)}{e^{p_n(\xi_N)}
h_n(x,\xi_N)}\cdot \pi_{n,n+1}(x,dy).$$
Let $\wt{\mathsf f}=\{f^{(N)}_n: 1\leq n\leq N+1, N\in\mathbb N\}$ where $f^{(N)}_n:=f_n$, and set
$$
\wt{S}_N:=f_1(\wt{X}_1^{(N)},\wt{X}_2^{(N)})+\cdots+f_N(\wt{X}^{(N)}_N,\wt{X}_{N+1}^{(N)}).
$$

Recall that $e^{\xi_N f_n}$, $h_n$,  and $e^{p_n(\xi_N)}$ are uniformly bounded away from zero and infinity, by the assumption on $\mathsf f$, and  Lemma \ref{Lemma-h-bounded}. So $\wt{\pi}^{(N)}_{n,n+1}(x,dy)$ differ from  ${\pi}_{n,n+1}(x,dy)$ by densities which are bounded away from zero and infinity uniformly in $N$.
It follows that $\wt{\mathsf X}$ is uniformly elliptic, $\wt{\mathsf f}$ is a.s. uniformly bounded, and  the structure constants of $(\wt{\mathsf X},\wt{\mathsf f})$ are equal to the structure constants of $(\mathsf X, \mathsf f)$ up to a uniformly bounded multiplicative error. Thus
\begin{enumerate}[(1)]
\item $(\wt{\mathsf X},\wt{\mathsf f})$ and $(\mathsf X, \mathsf f)$ have the same algebraic ranges, co-ranges, and essential ranges. In particular, $(\wt{\mathsf X},\wt{\mathsf f})$  is irreducible and non-lattice.
\item  $(\wt{\mathsf X},\wt{\mathsf f})$ is stably hereditary
(see Examples \ref{ExIIDHer} and \ref{Example-Non-Stably-Hereditary}  in \S\ref{Section-Hereditary}).
\item $\wt{V}_N:=\Var(\wt{S}_N)\xrightarrow[N\to\infty]{}\infty$ (because $\wt{V}_N\asymp\sum_{n=3}^N u_n^2\asymp V_N\to\infty$).
\end{enumerate}
Furthermore, by the choice of $\xi_N$,
$
\E(\wt{S}_N)\equiv \wt{\E}^{\xi_N}(S_N)=z_N+O(1)
$, so
$$
\frac{z_N-\E(\wt{S}_N)}{\sqrt{V_N}}=O\bigg(\frac{1}{\sqrt{V_N}}\bigg)\xrightarrow[N\to\infty]{}0.
$$
Therefore $\wt{S}_N$ satisfies the local limit theorem (Theorem \ref{ThLLT-classic}):
$$\Prob_x(\wt{S}_N-z_N\in (a,b))\sim {|a-b|}\bigg/{\sqrt{2\pi \wt{V}_N^{\xi_N}}}$$
 for every $x\in\fS_1$ and $(a,b)\neq \emptyset$.

\medskip
We will  translate this into an asymptotic for $\Prob(S_N-z_N\in (a,b))$.
For all $N$ large enough, for every $x\in\mathfrak S_1$,
\begin{align}
\label{LDToLLT}
&\Prob_{x}[S_N-z_N\in (a,b)]=
 e^{P_N(\xi_N)-\xi_N z_N}\times \notag \\
&\times\E_x\left(e^{\xi_N S_N}
\frac{h_{N+1}^{\xi_N}(X_{N+1}^{(N)})}{e^{P_N(\xi_N)}h_1^{\xi_N}(x)} \cdot
\frac{h_1^{\xi_N}(x)} {h_{N+1}^{\xi_N}(X_{N+1}^{(N)})}
\cdot e^{\xi_N(z_N-S_N)}
1_{(a,b)}(S_N-z_N)\right) \notag  \\
& =e^{P_N(\xi_N)-\xi_N z_N} h_1^{\xi_N}(x)
 \wt{\EXP}_x\left(h_{N+1}^{\xi_N}(\wt{X}_{N+1}^{(N)})^{-1}
\phi_{a, b} (\wt{S}_N-z_N) \right)
\end{align}
where $\phi_{a,b}(t):=1_{(a,b)}(t) e^{-\xi_N t}$.

The pre-factor simplifies as follows. By construction $\frac{P_N'(\xi_N)}{V_N}=\frac{z_N}{V_N}.$ Thus
$$\xi_N \,z_N-P_N(\xi_N)=V_N \left(\xi_N \frac{z_N}{V_N}-\frac{P_N(\xi_N)}{V_N}\right)=V_N \left(\xi_N \frac{P_N'(\xi_N)}{V_N}-\frac{P_N(\xi_N)}{V_N}\right).$$
So
\begin{equation}
\label{LegendreBrH}
 e^{P_N(\xi_N)-\xi_N z_N}=e^{-V_N H_N\left(\frac{z_N}{V_N}\right)},
\end{equation}
where $H_N(\eta)$ is the Legendre transform of $P_N(\xi)/V_N$.

Using the mixing LLT for Markov arrays Theorem \ref{Theorem-Mixing-LLT},  one can see that
\begin{align}\label{KeyExp}
& \EXP_x\left(h_{N+1}^{\xi_N}(\wt{X}_{N+1}^{(N)})^{-1}
\phi_{a, b} (\wt{S}_N-z_N) \right)
\sim \frac{\mu_{N+1}\left({1}/{h_{N+1}^{\xi_N}} \right)}{\sqrt{2\pi\wt{V}_N^{\xi_N}}} \int_a^b e^{-\xi_N t} dt,
\end{align}
as $N\to\infty$. To do this approximate $\phi_{a,b}$ in $L^1(\R)$ from below and above continuous functions with compact support, and approximate $h_{N+1}^{\xi_N}$ in $L^1(\fS^{(N)}_{N+1},\mathfs B(\fS^{(N)}_{N+1}),\mu_{N+1}^{(N)})$ from above and below by finite linear combinations of indicators of sets with uniformly bounded measure (here $\mu^{(N)}_{N+1}$ is the distribution of $X^{(N)}_{N+1}$).

Since $\xi_N$ is bounded,  Lemma \ref{Lemma-Changed-Expectation-Variance}(4) tells us that $\wt{V}_N^{\xi_N}\sim P_N''(\xi_N)$ as $N\to\infty$. Since $H_N(\eta)$ is the Legendre transform of $P_N(\xi)/V_N$, and $P_N'(\xi_N)/V_N=z_N/V_N$,
\begin{equation}\label{V-N-tilde-term}
\wt{V}^{\xi_N}_N\sim V_N\cdot \left(\frac{P_N''(\xi_N)}{V_N}\right)=\frac{V_N}{H_N''(\frac{z_N}{V_N})}\quad\text{as}\quad N\to\infty.
\end{equation}
Substituting \eqref{LegendreBrH},  \eqref{KeyExp}, and \eqref{V-N-tilde-term} in   \eqref{LDToLLT}, we  obtain the following:
\begin{align*}
&\Prob_x[S_N-z_N\in (a,b)]\sim \left[\frac{e^{-V_N\mathcal I_N(\frac{z_N}{V_N})}}{\sqrt{2\pi V_N}}\int_a^b e^{-\xi_N t}dt\right]\, \times \\
&\times\underset{\hat \rho_N\bigl(\frac{z_N-\E(S_N)}{V_N}\bigr)}{\underbrace{\left[e^{V_N\mathcal I_N(\frac{z_N}{V_N})-V_N H_N(\frac{z_N}{V_N})}\sqrt{ H_N''(\tfrac{z_N-\E(S_N)}{V_N})}\right]}}\times
\underset{\bar\rho_N\bigl(x,\frac{z_N-\E(S_N)}{V_N}\bigr)}{\underbrace{\left[h_1^{\xi_N}(x)\mu_{N+1}\left(\tfrac{1}{h_{N+1}^{\xi_N}} \right)\right]}}
\end{align*}

Let $\eta_N:=\frac{z_N-\E(S_N)}{V_N}$, then  $\xi_N=\xi_N(\eta_N)$ where
$\xi_N:[\ha_N,\hb_N]\to [-(R+1),(R+1)]$ is defined implicitly by $P_N'(\xi_{N}(\eta))=\eta V_N+\E(S_N)$. Lemma \ref{Lemma-xi-N-exists} shows that $\xi_N(\cdot)$ is well-defined.

Notice that there exists a constant
 $L=L(R)$ such that $|\eta_N|\leq L(R)$. Indeed, $\eta_N\in [\ha_N^R,\hb_N^R]$ and $|\ha_N^R|,|\hb_N^R|\leq |\mathfs F'(\pm R)-\mathfs F'(0)|\leq  R\sup\limits_{[-R,R]}\mathfs F_N''$, which is uniformly bounded by Theorem \ref{Theorem-F_N}(3).

The functions $\hat\rho_N:{[-L,L]}\to\R$ are defined by
$$
\hat\rho_N(\eta):=e^{V_N\mathcal I_N\left(\eta+\frac{\E(S_N)}{V_N}\right)-
V_N H_N\left(\eta+\frac{\E(S_N)}{V_N}\right)}\sqrt{{H_N''}(\eta)}.
$$
Lemma \ref{Lemma-I_N-and-H_N}  and Theorem \ref{Theorem-I_N} say that there exists $C$ such that
$$
C^{-1}\leq \hat\rho_N(\eta)\leq C\text{ for all $N$ and $|\eta|\leq L$}.
$$
They also say that for every $\epsilon>0$ there are $\delta, N_\epsilon>0$ s.t.
$$
e^{-\epsilon}\leq \hat\rho_N(\eta)\leq e^\epsilon\text{ for all $N>N_\epsilon$ and $|\eta|\leq \delta$}.
$$
In particular, if $\frac{z_N-\E(S_N)}{V_N}\to 0$, then $\hat\rho_N\bigl(\frac{z_N-\E(S_N)}{V_N}\bigr)\xrightarrow[N\to\infty]{}1$.

The functions $\bar\rho_N:\mathfrak S_1\times (-c,c)\to\R$ are defined by
$$
\bar\rho_N(x,\eta):=h_1(x,\xi(\eta))\mu_{N+1}\left(\frac{1}{h_{N+1}(x,\xi(\eta))}\right).
$$
 By Lemma \ref{Lemma-h-bounded}, there exists a constant $C$ such that
 $$
 C^{-1}\leq \bar\rho_N(x,\eta)\leq C\text{ for all $N$ and $|\eta|\leq L$}.
 $$
By Lemma   \ref{Lemma-Analyticity} and the obvious identity $h_n(\cdot,0)\equiv 1$,  $\|h_n^\xi-1\|_\infty\xrightarrow[\xi\to 0]{}0$ uniformly in $n$.
Since $|\xi(\eta)|\leq C|\eta|$, for every $\epsilon>0$ there are $\delta, N_{\epsilon}>0$ such that
$$
e^{-\epsilon}\leq \bar\rho_N(x,\eta)\leq e^{\epsilon}\text{ for all $x\in\mathfrak S_1$,  $N>N_\epsilon$, and $|\eta|\leq \delta$}.
$$
Setting $\rho_N:=\hat\rho_N \cdot \bar\rho_N$ we complete the proof of theorem in the non-lattice case.
The modifications needed for the lattice case are routine, and are left to the reader.\qed

\subsection{Rough bounds in the reducible case.}
\noindent
{\bf Proof of  Theorem \ref{ThLDOneSided}:}
We proceed as in the proof of Theorem \ref{Thm-LLT-LDP} in \S\ref{SSLLT-LD-Proof}, but using
 the rough bounds of
\S\ref{SSUniversalMA}
instead of
the precise LLT to estimate the probabilities for the change of measure.

Let $\hslash=100K+1$ where $K=\ess\sup(\mathsf{f})$.
Then
using Theorem \ref{Theorem-Other-Universal-Bounds} and the assumption that $z_N\in [\mathfs F_N'(\eps),b_N^R]$ we  get that  there exist a constant $\brc=\brc(R)$ and  $\xi_N:=\xi_N\left(\frac{z_N}{V_N}\right)\in [\eps,R+1]$ such that
for all $N$ large enough,
\begin{equation}
\label{LargeIntRough}
 \brc e^{{-\xi_N \hslash }} {\hslash} \leq
\frac{\sqrt{V_N} \Prob(S_N-z_N\in [0, \hslash ])} {e^
{-V_N \cI_N \left(\frac{z_N}{V_N}\right)}}.
\end{equation}
Note that Theorem \ref{Theorem-Other-Universal-Bounds} is applicable since $\hslash>2\delta(\mathsf f)$
due to Corollary \ref{Cor-delta-f}).

Since $\Prob(S_N\geq z_n)\geq \Prob(S_N-z_N\in [0, \hslash ])$ the lower bound follows.

 Likewise applying Lemma \ref{LmAntiCon} we conclude
that there is a constant
$C^*=C^*(R)$ s.t
for all $N$ large enough we have,  uniformly in $j\in \naturals\cup\{0\}$,
$$\frac{\sqrt{V_N} \Prob(S_N-z_N\in [\hslash j, \hslash (j+1)])} {e^
{-V_N \cI_N \left(\frac{z_N}{V_N}\right)}} \leq C^*  e^{{-\xi_N \hslash j}}.$$
 Summing over $j$ we obtain the lower bound.
\qed

\section{Large deviations threshold}\label{Section-Threshold}
\label{SS-LDTreshold}
The results of this chapter are all stated for  $z_N$ s.t. for some $R>0$ and  all sufficiently large $N,$
$\frac{z_N-\E(S_N)}{V_N}\in [\ha_N^R, \hb_N^R]$. In this section we will discuss how restrictive is this assumption.

We say that a sequence $\{z_N\}$ is {\bf $R$-admissible}\index{admissible}\index{$R$-admissible}
if there is a constant $N_0$ s.t. for $N\geq N_0$
$\exists \xi_N\in [-R, R]$ such that
$\DS  P_N'(\xi_N)=z_N . $
A sequence $\{z_N\}$ is {\bf admissible} if it is $R$-admissible for some $R$.

A number
$z$ is called {\bf reachable}\index{reachable} (respectively {\bf $R$-reachable}) if the
sequence $\{z V_N\}$ is admissible (respectively $R$-admissible).

We denote the set of $R$--reachable points by $\cC_R$ and the set of reachable points by $\cC.$
Since $P_N'$ is monotone increasing, $$\mathrm{int}(\cC)=(\fc_-, \fc_+)$$ for some
$\fc_\pm=\fc_\pm(\mathsf{X}). $

\begin{example}[Sums of iid's]
\label{ExLDInd2}
\end{example}
Let $\DS S_N=\sum_{n=1}^N X_n$ where $X_n$ are iid random variables having law $X$ with expectation zero and variance one.
Recall from
Example \ref{ExLDInd1}
that in this case $\cF_N$ does not depend on $N$
\footnote{Note that in this case we also have $P_N(\xi)/N=\cF_N(\xi)$,
since $\cL_{N,\xi}(e^{\xi x_{n+1}})=\EXP(e^{\xi X}) \cdot e^{\xi x_n}$, whence
$p_n(\xi)=\ln \EXP(e^{\xi X})$.}
 so by property (ii) of
Example \ref{ExLDInd1} we obtain
\begin{equation}
\label{TrInd}
\fc_-=\ess\inf (X), \quad \fc_+=\ess\sup (X).
\end{equation}
Then $S_N/N\in [\fc_-,\fc_+]$ almost surely for all $N$, and therefore $\Prob[S_N-zN\in (a,b)]$ is  zero when $z\not\in [\fc_-,\fc_+]$. Henceforth we refer to such $z$ as ``irrelevant."

Not all relevant $z$ are reachable: $z$ is reachable only when  $z\in (\fc_-,\fc_+)$. Our results do not apply for $z=\fc_{\pm}$. Indeed different asymptotic behavior may hold for $z_N$ s.t.
$\frac{z_N}{V_N}\to\fc_{\pm}$, see Example \ref{Example-Edge}.
Still, the  large deviation LLT for $\Prob[S_N-zN\in (a,b)]$ holds for most ``relevant" values of $z$.
Our next example shows that this is not always the case:

\begin{example}
\label{ExCoreDrop}
\end{example}
Let $X_n=(Y_n, Z_n)$ where $\{Y_n\}$, $\{Z_n\}$ are two independent sequences of iid
random variables having uniform distribution on $[0,1].$ Fix a sequence $\{p_n\}$ and let
$$ f_n(Y_n,Z_n)=\begin{cases} Z_n & \text{if } Y_n>p_n \\ 2 & \text{if } Y_n\leq p_n. \end{cases} $$
We now discuss two possible choices of $\{p_n\}.$

(a)
Let $\mathsf{f}'$ be defined as above with  $p_n\equiv \frac{1}{2}.$
Then $f_n'$ are iid so by discussion of the Example \ref{ExLDInd2} the results of the present
chapter apply to
$\Prob(S_N'\in zN+(a,b))$  provided that
$ z\in (0,2)  $ while the possible range of $\frac{S_N(\mathsf f')}{N}$ is $[0,2].$

(b)
Let $\mathsf{f}''$ be defined as above with  $p_n$ tending to $0$ as $n\to\infty.$
Since $\Var(Z_n)=\frac{1}{12}$ it follows that $\DS V_N=(1+o(1))\frac{N}{12}.$
We shall show below that in case (b)
\begin{equation}
\label{CoreTrim}
\fc_-=0, \quad \fc_+=12.
\end{equation}
In other words the results of the present chapter apply to
$\Prob(S_N(\mathsf f'')\in zN+(a,b))$  provided that
$ z\in (0,1)  $. On the other hand, the possible range of $\frac{S_N(\mathsf f'')}{N}$ is $[0,2]$
since for each fixed $N$ the distributions of $S_N(\mathsf f')$ and $S_N(\mathsf f'')$ are absolutely continuous with respect
to each other. We will see that the reason our results
do not apply for $z>1$ is that in that case $\Prob(S_N(\mathsf f'')\geq z N)$ decays super exponentially.

In this section we discuss methods for computing $\fc_\pm$ (in particular, proving \eqref{CoreTrim})
and provide sufficient conditions for good behavior, when $(\fc_-, \fc_+)$ covers
``most" relevant $z.$

\begin{lemma}
\label{LmAdmOpen}
$\forall R>0$ $\exists \eps=\eps(R)>0$ s.t. if $\{z_N\}$ is $R$-admissible, and
$|\brz_N-z_N|\leq \eps V_N,$ then $\{\brz_N\}$ is $(R+1)$-admissible.
\end{lemma}

\begin{proof}
By the uniform strict convexity of $\frac{P_N}{V_N}$ on $[-(R+1), (R+1)]$,
there exists $\eps>0$ such that
 $\DS P_N'(R+1)\geq z_N+\eps V_N$ and
 $\DS P_N'(-(R+1))\leq z_N-\eps V_N.$
\qed \end{proof}

\begin{corollary}\label{Corollary-frak-c}
(a) $\cC$ is open, and (b)
if $\EXP(S_n)\equiv 0$, then $\cC$ is a non-empty neighborhood of zero.
\end{corollary}

\begin{proof}
Part (a) follows from Lemma \ref{LmAdmOpen}. Part (b) follows from \eqref{HRange}.
\qed
\end{proof}

Without the assumption $\E(S_N)=0$, $\cC$ may be empty.
Even though Theorem \ref{Theorem-I_N} provides many admissible sequences,
the associated $\frac{z_N}{V_N}$ need not converge:
\begin{example}
An example with $\cC=\emptyset$ and with admissible sequences $\{z_N\}$ such that  $z_N/V_N$ does not converge.
\end{example}
Let $N_k=10^k.$ Consider $X_n=a_n+U_n$ where $U_n$ are iid having uniform distribution on $[0,1]$
and $$a_n=\begin{cases} 10 & \text{if } N_{2k}\leq n<N_{2k+1}, \\
-10 & \text{if } N_{2k+1}\leq n<N_{2k+2}. \end{cases}$$
With probability one
$\DS S_{N_{2k+1}}>N_{2k+1}$,
$\DS S_{N_{2k}}<-N_{2k}. $
The first inequality gives $\cC\cap (-\infty, 0]=\emptyset,$
the second one gives $\cC\cap [0, +\infty)=\emptyset.$
Hence $\cC=\emptyset.$

In this example, if $\{z_N\}$ is an admissible sequence then
$\frac{z_{N_{2k+1}}}{N_{2k+1}}\geq 1$ and $\frac{z_{N_{2k}}}{N_{2k}}\leq -1.$
Since $V_N=\frac{N}{12}$ the ratio $\frac{z_N}{V_N}$ does not converge.

\begin{theorem}
\label{ThLDP=LLTLD}
Let $\mathsf f$ be an a.s. uniformly bounded additive functional on a
uniformly elliptic Markov chain $\mathsf X$, with essential range $\Z$ or $\R$.
The following are equivalent:
\begin{enumerate}[(a)]
\item  $\{z_N\}$ is admissible.
\item  $\exists \eps>0, \eta>0$ s.t.
$\forall \{\brz_N\}$ with $|\brz_N-z_N|\leq \eps V_N$
and $\forall a_N, b_N$ s.t.
$|a_N|,|b_N|\leq 10$ and  $b_N-a_N > 1$ we have
$ \Prob(S_N\in \brz_N+(a_N, b_N))\geq \eta^{V_N}. $
\item $\exists \eps>0, \eta>0$ s.t.
$\DS \Prob(S_N\geq z_N+\eps V_N)\geq \eta^{V_N}$ and
$\Prob(S_N\leq z_N-\eps V_N)\geq \eta^{V_N}.$
\end{enumerate}
\end{theorem}

\begin{example}
The case $\eps=0$.
\end{example}
Let $\DS S_N=\sum_{n=1}^N X_n$ where $X_n$ are iid supported on $[\alpha, \beta]$ and such that
$X$ has an atom on the right edge: $\Prob(X=\beta)=\gamma>0.$ Then
$ \Prob[S_N\geq \beta N]=\Prob[S_N=\beta N]=\gamma^N $
while
$\DS \Prob[S_N\geq \beta N+1]=0.$ Thus $\{\beta N\}$ is not admissible. This example shows
that taking $\eps=0$ in part (c) of Theorem \ref{ThLDP=LLTLD}
gives a condition which is {\bf not} equivalent to
the conditions (a)--(c) of the theorem.

\begin{proof} $\mathbf{(a)\Rightarrow (b):}$
If $\{z_N\}$ is admissible then by Lemma \ref{LmAdmOpen}
$\exists \eps>0$ such that if $|\brz_N-z_N|\leq \eps V_N$ then $\{\brz_N\}$ is admissible. Now
(b) follows from formula \eqref{LargeIntRough}
 in the proof of Theorem \ref{ThLDOneSided}. \medskip

\noindent$\mathbf{(b)\Rightarrow (c):}$
The  bound
$\Prob[S_N\geq z_N+\eps V_N]\geq \eta^{V_N}$ follows from part (b) with $\brz_N=z_N+\eps V_N,$
$a_N=0,$ $b_n=1.1.$ The lower bound is similar. \medskip

\noindent$\mathbf{(c)\Rightarrow (a):}$ Our assumptions on the essential range imply that $(\mathsf X,\mathsf f)$ is not center-tight, and therefore $V_N\to\infty$. By Lemma \ref{Lemma-Changed-Expectation-Variance}(5) $P_N(R)-V_N\mathfs F_N(R)$ is eventually bounded, and therefore for some $c(R)>0$ and all $N>N(R)$,
$$ e^{P_N(R)}\geq c(R) \EXP\left(e^{R S_N}\right)\geq
c(R) \EXP\left(e^{R S_N}1_{[S_N\geq z_N+\epsilon V_N]}\right)\geq
c(R) \eta^{V_N} e^{R (z_N+\eps V_N)}.$$
This implies that for all $N$ large enough
$P_N(R)\geq R(z_N+(\eps/2)) V_N$.

Since $P_N(0)=0$ the Mean Value Theorem tells us that
$\exists \xi_N^+\in [0, R]$ such that
$\DS  P_N'(\xi_N^+)\geq z_N+\frac{\eps V_N}{2}. $
Likewise we can find $\xi_N^-\in [-R, 0]$ such that
$\DS  P_N'(\xi_N^-)\leq z_N-\frac{\eps V_N}{2}. $ By the Intermediate Value Theorem
$\exists \xi_N \in [\xi_N^-, \xi_N^+]$ s.t. $P_N'(\xi_N)=z_N.$
\qed\end{proof}

\begin{corollary}
\label{CrIntRate}
Under the assumptions of the previous theorem,
if $\EXP(S_N)\equiv 0$ then
$\displaystyle  \fc_+=\sup\{z: \fI(z)<\infty\}, $ where
$$\fI(z)=\limsup_{N\to\infty}  \frac{|\log \Prob(S_N\in z V_N+[-1,1])|}{\log V_N}. $$
\end{corollary}

\begin{proof}
 By Theorem \ref{ThLDP=LLTLD}(b),  if $z\in (\fc_-,\fc_+)$, then $\mathfrak I(z)<\infty$. So $\fc_+\leq \sup\{z:\mathfrak I(z)<\infty\}$.

 To see the other inequality, note  that $\fc_+>0$  (by Corollary \ref{Corollary-frak-c}), and  $\mathfrak I(0)<\infty$ (by \eqref{LargeIntRough}).  We will show that
\begin{equation}\label{cold}
\frac{1}{2}\sup\{z: \fI(z)<\infty\}<\brz<\sup\{z: \fI(z)<\infty\}\Rightarrow \text{$\brz$  is admissible,}
\end{equation}
and deduce that $\fc_+\geq \sup\{z: \fI(z)<\infty\}$.

Fix $\brz$ as in \eqref{cold}, then $\exists \eps>0$ s.t. $\mathfrak I(\brz+2\epsilon)<\infty$ and $\brz-\epsilon>0$. Necessarily  $\exists\eta>0$ s.t. for all $N$ large enough
\begin{align*}
&\Prob[S_N\geq (\brz+\epsilon) V_N]\geq \Prob[S_N\in (\brz +2\eps) V_N+[-1, 1]]\geq \eta^N \\
& \Prob[S_N\leq (\brz-\epsilon) V_N]\geq \Prob[S_N\leq 0]=\frac{1}{2}+o(1)\geq \eta^N
\end{align*}
By Theorem \ref{ThLDP=LLTLD}(c), $z$ is admissible.
\qed\end{proof}

We say that $(\mathsf{X},\mathsf f)$ and $(\widetilde{\mathsf X},\wt{\mathsf f})$ {\em are related by the change of measure}
if
$f_n\equiv \tf_n$ and $\pi_n(x, dy)$ is equivalent to $\tilde\pi_N(x, dy)$ with
$$ \breps\leq \frac{\tilde\pi_n(x,dy)}{\pi_n(x, dy)}\leq \breps^{-1}. $$

\begin{lemma}
Suppose $\mathsf f$ is an a.s. uniformly bounded additive functional on a uniformly elliptic Markov chain $\mathsf X$.
If $(\mathsf{X},\mathsf f)$ and $(\widetilde{\mathsf X},\wt{\mathsf f})$
are related by the change of measure
and $V_N\geq cN $ for some $c>0$,
then
$\{z_N\}$ is $(\mathsf{X},\mathsf f)$-admissible
iff $\{z_N\}$ is $(\widetilde{\mathsf X},\wt{\mathsf f})$-admissible.
\end{lemma}

\begin{proof}
Since $\mathsf X$ is uniformly elliptic, $\wt{\mathsf X}$ is uniformly elliptic. The exponential mixing bounds for uniformly elliptic chains imply that $\wt{V}_N:=\Var[S_N(\wt{\mathsf{X}},\wt{\mathsf f})]$ and $V_N:=\Var[S_N(\wt{\mathsf{X}},\wt{\mathsf f})]$ are both $O(N)$. Without loss of generality, $cN\leq V_N\leq c^{-1} C_N$.

 Under the assumptions of the Lemma, the structure constants of $(\mathsf{X},\mathsf f)$ are equal to the structure constants of $(\wt{\mathsf{X}},\wt{\mathsf f})$ up to bounded multiplicative error.
 By Theorem~\ref{Theorem-MC-Variance}, $\wt{V}_N:=\Var[S_N(\wt{\mathsf{X}},\wt{\mathsf f})]\asymp V_N$ as $N\to\infty$. So $\exists\wt{c}>0$ s.t. $\wt{c}N\leq \wt{V}_N\leq \wt{c}^{-1}N$.

Let $\{z_N\}$ be $(\mathsf{X},\mathsf f)$-admissible. Then there are $\eps>0, \eta>0$ such that
$$ \Prob[S_N\geq z_N+\eps V_N]\geq \eta^N, \quad
\Prob[S_N\leq z_N- \eps {V}_N]\geq \eta^N. $$
It follows that
$ \wt\Prob[S_N(\wt{\mathsf{X}},\wt{\mathsf f})\geq z_N+\wt{\eps} \wt{V}_N]\geq \wt\eta^N, \quad
\wt\Prob[S_N(\wt{\mathsf{X}},\wt{\mathsf f})\leq z_N-\wt{\eps} \wt{V}_N]\geq \wt\eta^N$
where $\wt\eta=\eta\breps$ and $\wt{\eps}:=\wt{c}c\eps$.
Hence $\{z_N\}$ is $\wt{\mathsf X}$-admissible.
\qed\end{proof}

\begin{lemma}
\label{LmSmallChange}
Let $\mathsf{f}$ and $\wt{\mathsf{f}}$ be two {\em a.s. uniformly bounded} additive functionals
on the same uniformly elliptic Markov chain. Suppose $ V_N:=\Var[S_N(\mathsf f)]\to\infty$ and
\begin{equation}
\label{SmallChange}
\lim_{N\to\infty} \frac{\|S_N(\wt{\mathsf f})-S_N(\mathsf f)\|_\infty}{V_N}=0.
\end{equation}
Then $\{z_N\}$ is $\mathsf{f}$-admissible
iff $\{z_N\}$ is $\widetilde{\mathsf f}$-admissible.
\end{lemma}

\begin{proof} We write $S_N=S_N(\mathsf f)$, and $\wt{S}_N=S_N(\mathsf f)$. By the assumptions of the lemma, $\wt{V}_N:=\Var(\wt{S}_N)\sim V_N$ as $N\to\infty$.

Let $\{z_N\}$ be $\mathsf{f}$-admissible. By Theorem \ref{ThLDP=LLTLD}(b),
there are $\eps>0, \eta>0$ such that
$$ \Prob[S_N\geq z_N+\eps V_N]\geq \eta^N, \quad
\Prob[S_N\leq z_N-\eps V_N]\geq \eta^N. $$
It now follows from \eqref{SmallChange} that for large $N$
$$ \wt\Prob\left[\wt{S}_N\geq z_N+\frac{\eps }{2} \wt{V}_N \right]\geq \eta^N, \quad
\wt\Prob\left[\wt{S}_N\leq z_N-\frac{\eps }{2} \wt{V}_N\right]\geq \eta^N.$$
 Hence $\{z_N\}$ is $\wt{\mathsf f}$-admissible.
\qed\end{proof}

We end this section by proving \eqref{CoreTrim}.

\medskip
\noindent{\em Proof of \eqref{CoreTrim}.} To show that $\fc_+\leq 12$ assume by contradiction that $\mathrm{int}(\mathfs C)$ contained  some  $z>12$.
Then {Theorem~\ref{ThLDOneSided} } would imply that
\begin{equation}
\label{LDPExample}
 \Prob[S_N\geq z V_N]\geq \eta^{V_N} \text{  for some }\eta>0.
\end{equation}

 Note that
$$ \log\E(e^{\xi f_n(Y_n,Z_n)})=\log\left(p_n e^{2\xi}+(1-p_n)\E(e^{\xi U[0,1]})\right)$$
$$ =\log\left( p_n e^{2\xi}+(1-p_n)\frac{e^{\xi}-1}{\xi}\right)
\xrightarrow[n\to\infty]{}\log\frac{e^\xi-1}{\xi} $$
because $p_n\to 0.$ So
$$ \mathfs F_N(\xi)= \frac{1}{V_N}\log\prod_{n=1}^N
\E\left(e^{\xi f_n(Y_n,Z_n)}\right)\sim \frac{12}{N}\sum_{n=1}^N
\log\E\left(e^{\xi f_n(Y_n,Z_n)}\right)\xrightarrow[N\to\infty]{}12\log\left(\frac{e^\xi-1}{\xi}\right). $$
The last expression is strictly smaller than $12\xi$ if $\xi> 0.$
Therefore for any $\xi> 0$ we have for sufficiently large $N,$\;
$\DS  \E\left(e^{\xi S_N}\right)\leq e^{12 V_N \xi}. $ By Markov's inequality,
$$ \Prob\left[S_N\geq z V_N\right]\leq e^{(12-z) V_N \xi}\text{ for all $\xi>0$ and $N$ sufficiently large}. $$
But this is incompatible with \eqref{LDPExample}, since $z>12.$
Therefore $\mathfrak c_+\leq 12$.

Next we show that $(0,12)\in$Int$(\cC).$ By Theorem \ref{ThLDP=LLTLD} it suffices to show that
for every $\mathfrak{z}:=\frac{z}{12}\in (0,1)$,
$\exists \eps, \eta>0$ such that
$\DS \Prob[A_\eps^\pm(N)]{ \geq} \eta^N$ where
$$ A_\eps^+(N)=\Prob[S_N\geq (\mathfrak{z}+\eps) {N} ], \quad
A_\eps^-(N)=\Prob[S_N\leq (\mathfrak{z}-\eps) N].$$
 Take $\eps>0$ so small that $\overline{\mathfrak{z}}:=z+\eps<1.$
Since $S_N\geq Z_N$ we have
$$ \Prob[S_N\geq \overline{\mathfrak{z}}N]\geq
\Prob\left[\sum_{n=1}^N Z_n\geq \overline{\mathfrak{z}} N\right]. $$
 The RHS is  greater than some $\breta^N$ in view of
Theorem \ref{ThLDOneSided} and equation \eqref{TrInd}
from Example \ref{ExLDInd2}.   It follows that $\fc^+=12.$

The proof of the fact that $\fc^-=0$ is similar but easier.
\hfill$\Box$

\section{Notes and references}
The reader should note the difference between the LLT for large deviations and the large deviations principle (LDP): LLT for large deviations give the asymptotics of $\Prob[S_N-z_N\in (a,b)]$ or $\Prob[S_N>z_N]$; The LDP gives the  asymptotics of the {\em logarithm } of $\Prob[S_N>z_N]$, see Dembo \& Zeitouni \cite{Dembo-Zeitouni} and Varadhan \cite{Varadhan-LD-Book}.

The interest in precise asymptotics for $\Prob[S_N>z_N]$ in the regime of large deviations  goes back to the first paper  on large deviations, by Cram\'er \cite{Cramer-LDP}. That paper gave an asymptotic series expansion for  $\Prob[S_N-\E(S_N)>x]$ for $S_N=$sums of iid's. The first sharp asymptotics for $\Prob[S_N-z_N\in (a,b)]$ appear to be the work of
Richter \cite{Richter},\cite[chapter 7]{Ibragimov-Linnik} and  Blackwell \& Hodges \cite{Blackwell-Hodges}.

These results were refined by many authors, with important contributions by Petrov \cite{Petrov-LD}, Linnik \cite{Linnik}, Moskvin \cite{Moskvin}, Bahadur \& Ranga Rao \cite{Bahadur-Ranga-Rao},  Statulavicius \cite{Statulevicius-1966} and  Saulis  \cite{Saulis}.
 We refer the reader to the books of Ibragimov \& Linnik \cite{Ibragimov-Linnik}, Petrov \cite{Petrov-Book}, and of Saulis \& Statulevicius \cite{Saulis-Statulevicius-Book} for accounts of these and other results, and also to the survey of Nagaev \cite{Nagaev-LD} for a discussion of the case of sums of independent random variables which are not necessarily identically distributed.

Plachky and Steinebach \cite{Plachky-Steinebach} and Chaganty \& Sethuraman \cite{Chaganty-Sethuraman-1985,Chaganty-Sethuraman-1993}  proved LLT for large deviations for arbitrary sequences of random variables $T_n$ (e.g. sums of {dependent} random variables), subject only to assumptions on the asymptotic behavior of the normalized log-moment generating functions of $T_n$ and their Legendre-Fenchel transforms (their rate functions). Our LLT for large deviations are  in the spirit of these results.

Corollary \ref{CrConditioning-R} is an example of a limit theorem conditioned on a large deviation. For other examples of such results, in the context of statistical physics, see \cite{Derrida-Sadhu}.

We comment on some of the technical devices in the proofs. The ``change of measure" trick discussed in section \ref{section-strategy-LD} goes back to Cram\'er \cite{Cramer-LDP} and is a standard idea in large deviations.
In the classical homogeneous setup, a single parameter $\xi_N=\xi$ works for all times $N$, but in  our inhomogeneous setup, we need to allow the parameter $\xi_N$ to depend $N$. For other instances of changes of measure which involve a time dependent parameter, see Dembo \& Zeitouni \cite{Dembo-Zeithouni-LD-parameter-dependent} and references therein.

Birkhoff's Theorem on the contraction of Hilbert's projective metric is proved in \cite{Bi}. Results similar to Lemma \ref{Lemma-h-exist} on the existence of the generalized eigenfunction $h^\xi_n$  were proved by many authors in many different contexts, see for example \cite{K}, \cite{FS},\cite{BG}, \cite{Rugh}, \cite{Du}, \cite{Hafouta-Kifer-Book}, \cite{Hafouta-Sequential}. The analytic dependence of the generalized eigenvalue and eigenvector on the parameter $\xi$ was considered in a different context (the top Lyapunov exponent) by Ruelle \cite{Ruelle-Analyticity} and Peres \cite{Peres}.  Our proof of Lemma \ref{Lemma-Analyticity} follows  closely a proof in \cite{Du}. For an account of the theory of real-analyticity for vector valued functions, see \cite{Die} and \cite{Wh}.


\chapter{Miscellaneous examples and special cases}
\noindent{\em In this chapter we consider several special cases where our general results take stronger form. These include homogeneous Markov chains, asymptotically homogeneous additive functionals.
We also explain how continuity assumptions can be used to strengthen the results of the previous
chapters.}

\section{Homogenous Markov chains}
A Markov chain $\mathsf X=\{X_n\}$ is called {\bf  homogeneous}\index{homogeneous Markov chain}\index{Markov chain!homogeneous} if its state spaces and transition probabilities do not depend on $n$
$$\fS_n=\fS,\quad \mu_n=\mu, \quad \pi_n(x,dy)=\pi(x,dy)\quad\text{ for all }n,$$
and $X_n$ is stationary.

An additive functional on a homogeneous Markov chain is called  {\bf homogeneous}\index{homogeneous additive functional}\index{additive functional!homogeneous} if $\mathsf f=\{f_n\}$  and
$$
f_n(x,y)=f(x,y)\text{ for all }n.
$$

The LLT for homogeneous countable state Markov chains is due to Nagaev. The following version, which allows continuous spaces, follows from results in \cite{HenHer}.

\begin{theorem}
\label{ThLLTHom}
Let $\mathsf f$ denote an a.s. uniformly bounded homogeneous additive functional on a uniformly elliptic homogeneous Markov chain $\mathsf X$.
\begin{enumerate}[(1)]
\item  {\bf Asymptotic Variance:} The limit $\DS \sigma^2=\lim_{N\to\infty} \frac{1}{N}\Var(S_N)$
 exists, and
$\sigma^2=0$ iff we can represent $f(X_1,X_2)=a(X_2)-a(X_1)+\kappa\text{ a.s.}$
where $a:\fS\to\R$ is a bounded measurable
function and $\kappa$ is a constant, equal to $\E(f(X_1,X_2))$.

\medskip
\item {\bf CLT:} If $\sigma^2>0$, then
 $\frac{S_N-\E(S_N)}{\sqrt{N}}$ converges in probability as $N\to\infty$
to the Gaussian distribution with mean zero and variance $\sigma^2$.

\medskip
\item  {\bf LLT:} If $\sigma^2>0$ then exactly  one of the following options  holds:
\begin{enumerate}[(a)]
\item {\em Non-Lattice LLT:}  If
$\frac{z_N-\E(S_N)}{\sqrt N}\to z$, then  for every interval $[a,b]$,
$$ \Prob[S_N-z_N\in [a, b]]=[1+o(1)]\frac{e^{-z^2/(2\sigma^2)}}{\sqrt{2\pi \sigma^2 N} }
(b-a),\text{ as }N\to\infty; $$
\item {\em Periodicity:} There exist $\kappa\in\R, t>0$ and  a bounded measurable function $a:\fS\to\R$  such that
$ f(X_1,X_2)+a(X_1)-a(X_2)+\kappa\in t\Z$ a.s.
\end{enumerate}
\end{enumerate}
\end{theorem}

\medskip
\noindent
{\em Proof.\/}
 Let $V_N:=\Var(S_N)$ and $f_k:=f(X_k,X_{k+1})$, and assume without loss of generality that $\E[f(X_1,X_2)]=0$.

\medskip
\noindent
{\bf Proof of part (1):}
By stationarity, $\E(f_n)=0$ for all $n$, and so
$$ V_N=\EXP\left(f_n^2\right)=
\sum_{n=1}^N \EXP(f_n^2)+2\sum_{1\leq m<n=N} \EXP(f_n f_m). $$
By stationarity, $\DS \EXP(f_n f_m)=\EXP(f_0 f_{n-m})$  and
$$\frac{1}{N}V_N=\EXP(f_0^2)+2\sum_{k=1}^{N-1} \EXP(f_0 f_k) \left(1-\frac{k}{N}\right). $$
$|\EXP(f_0 f_m)|$ decays exponentially (Prop. \ref{Proposition-Exponential-Mixing}), so $\sum |\E(f_0 f_k)|<\infty$, whence
\begin{equation}
\label{GreenKubo}
\sigma^2:=\lim_{N\to\infty}\frac{1}{N}\Var(S_N)=\EXP(f_0^2)+2 \sum_{k=1}^\infty \EXP(f_0 f_k).
\end{equation}
(This identity for $\sigma^2$ is called the {\bf Green-Kubo formula}\index{Green-Kubo formula}.)

 Let $u_n$ denote the structure constants of $(\mathsf X,\mathsf f)$. The homogeneity assumptions implies that $u_n$ is independent of $n$, say $u_n=u$ for all $n$. It follows that $U_N\equiv u_3^2+\cdots+u_N^2=(N-2)u^2$.  Now we have two cases:
\begin{enumerate}[(I)]
\item  $u>0$: In this case by Theorem \ref{Theorem-MC-Variance}, $V_N\asymp U_N\asymp N$, whence $\sigma^2>0$.
\item $u=0$: In this case, $\Var(S_N)=O(1)$ by Theorem \ref{Theorem-MC-Variance}, whence $\sigma^2=0$ and $\mathsf f$ is center-tight. By the Gradient Lemma, (Lemma \ref{LmVarAbove}),
$ f(X_1,X_2)=a_2(X_2)-a_1(X_1)+\kappa$ for some $a_1,a_2:\fS\to\R$ bounded and measurable and $\kappa\in\R$. In the homogeneous case, we may take $a_1\equiv a_2$, see \eqref{DefGrad} in the proof of the Gradient Lemma. So $f(X_1,X_2)=a(X_2)-a(X_1)+\kappa$ a.s.

    \end{enumerate}

\medskip
\noindent
{\bf Proof of part (2):} This follows from part (1) and Dobrushin's CLT.

\medskip
\noindent
{\bf Proof of part (3):}  By homogeneity, the structure constants $d_n(\xi)$ are independent of $n$, and they are all equal to $d(\xi):=\E(|e^{i\xi\Gamma}-1|^2)^{1/2}$, where $\Gamma$ is the balance of a random hexagon at position $3$. So $D_N(\xi)=\sum_{k=3}^N d^2_k(\xi)=(N-3) d^2(\xi)$.

If $d(\xi)\neq 0$ for all $\xi\neq 0$, then $D_N(\xi)\to\infty$ for all $\xi\neq 0$,   $f$ is irreducible by Theorem \ref{Theorem-essential-range}.
and the LLT follows
from Theorem \ref{ThLLT-classic}.

If $d(\xi)=0$ for some $\xi\neq 0$, then $D_N(\xi)=0$ for all $N$, $\xi$ is in the co-range of $(\mathsf X,\mathsf f)$, and our reduction lemma says that there  exist $c_n\in\R$ and  uniformly bounded measurable $a_n:\fS\to\R$ and $h_n(X_n,X_{n+1})$  such that $\sum h_n(X_n,X_{n+1})$ converges a.s., and
$
f(X_n,X_{n+1})+a_{n}(X_n)-a_{n+1}(X_{n+1})+h_n(X_n,X_{n+1})+\kappa_n\in \frac{2\pi}{\xi}\Z\text{ a.s.}
$

Let $A_n(X_n,X_{n+1},\ldots):=a_n(X_n)+\sum_{k\geq n}h_k(X_k,X_{k+1})$, then  for all $n$
\begin{equation}\label{cohomology}
f_n(X_n,X_{n+1})+A_n(X_n,X_{n+1},\ldots)-A_{n+1}(X_{n+1},X_{n+2},\ldots)+\kappa_n\in \frac{2\pi}{\xi}\Z\text{ a.s. .}
\end{equation}
We need to replace $A_i(X_i,X_{i+1},\ldots)$ by $a(X_i)$. This is the purpose of the following proposition, whose proof will complete the proof of the theorem:

\begin{proposition}
Let $\mathsf X$ be a uniformly elliptic homogeneous Markov chain with state space $(\fS,\mathcal B,\mu)$, and let $f:\fS\times\fS\to\R$ be a  measurable function such that $\ess\sup|f(X_1,X_2)|<\infty$. If there exist measurable functions $A_n:\fS^\N\to\R$ and $\kappa_n\in \R$ satisfying \eqref{cohomology},
then there exist $\kappa\in\R$ and a measurable $a:\fS\to\R$ such that
$$
f(X_n,X_{n+1})+a(X_n)-a(X_{n+1})+\kappa\in\Z\text{ a.s. for all $n$.}
$$
\end{proposition}

\begin{proof} Throughout this proof, let $\Omega:=\fS^\N$, equipped with the $\sigma$-algebra $\mathfs F$ generated by the cylinder sets
$$
[A_1,\ldots,A_n]:=\{x\in\fS^\N: x_i\in A_i\ (i=1,\ldots,n)\}\ \ (A_i\in\mathcal B)
$$
and the unique probability measure $m$ on $(\Omega,\mathfs F)$ s.t.
$$
m[A_1,\ldots,A_n]=\Prob[X_1\in A_1,\ldots,X_n\in A_n]
$$
Let $\sigma:\Omega\to\Omega$ denote the left-shift map, $\sigma[(x_n)_{n\geq 1}]=(x_{n+1})_{n\geq 1}$. The  stationarity of $\mathsf X$ translates to the shift invariance of $m$: $m\circ \sigma^{-1}=m$.

\smallskip
\noindent
{\sc Step 1 ({\bf Zero-One Law}\index{Zero-one law}):} Let $\sigma^{-n}\mathfs F:=\{\sigma^{-n}(A):A\in\mathfs F\}$, then for every $A\in \bigcap_{n\geq 1}\sigma^{-n}\mathfs F$, either $m(A)=0$ or $m(A)=1$.

\medskip
\noindent
{\em Proof.\/} Fix a cylinder $A:=[A_1,\ldots,A_{\ell}]$.

By uniform ellipticity, for every cylinder $B=[B_1,\ldots,B_n]$,
$$
m(A\cap \sigma^{-({\ell+1})} B)=m([A_1,\ldots,A_{\ell},\ast,B_1,\ldots,B_n])\geq \epsilon_0 m(A)m(B).
$$
 Applying this to cylinders
$[\underbrace{\Omega, \ldots, \Omega}_{k \text{ times}}, B_1,\ldots,B_n]$
we find that
$$
m(A\cap \sigma^{-(\ell+k)} [B_1,\ldots,B_n])\geq \epsilon_0 m(A)m(B)\text{ for all }k\geq 1.
$$
By the monotone class theorem,
\begin{equation}
\label{GrMixing}
m(A\cap\sigma^{-(\ell+k)}E)\geq \epsilon_0 m(A)m(E)
\text{ for {\em every}  }\mathfs F\text{--measurable }E\text{ and }k\geq 1.
\end{equation}

Suppose $E\in \bigcap_{k\geq 1}\sigma^{-n}\mathfs F$, and let $A$ be an arbitrary cylinder of length
$ \ell$.
By the assumption on $E$,  $E=\sigma^{-n}E_n$ with $E_n\in\mathfs F$ and $n>\ell$. So
$$m(A\cap E)=m(E\cap \sigma^{-n}E_n)\geq \epsilon_0 m(A)m(E_n)=\epsilon_0 m(A)m(E).$$ We see that  $\frac{m(A\cap E)}{m(A)}\geq \epsilon_0 m(E)$ for all cylinders $A$, whence
$$
\E(1_E|X_1,\ldots,X_\ell)\geq \epsilon_0 m(E)\text{ for all $\ell$}.
$$
By the martingale convergence theorem, $1_E\geq \epsilon_0 m(E)$ a.e., whence $m(E)=0$ or $1$.

\smallskip
\noindent
{\sc Step 2:} {\em Identify $f$ with a function $f:\Omega\to\R$ s.t. $f[(x_i)_{i\geq 1}]=f(x_1,x_2)$. Then there exist $A:\Omega\to\R$ measurable and $\kappa\in\R$ s.t.}
$
f+A-A\circ\sigma+\kappa\in \Z\text{ almost surely. }
$

\smallskip
\noindent
{\em Proof.\/} The assumptions of the proposition say that there exist $A_n:\Omega\to\R$ measurable and $\kappa_n\in\R$ s.t.
$$
f\circ \sigma^n+A_n\circ\sigma^n-A_{n+1}\circ\sigma^{n+1}+\kappa_n\in\Z\text{ $m$-a.e. for every $n$.}
$$

Let $w_n:=e^{2\pi i A_n}$, then
$
e^{2\pi i f\circ \sigma^n}\frac{w_n\circ \sigma^n}{w_{n+1}\circ\sigma^{n+1}}=1
$ $m$-a.s. Since $m\circ\sigma^{-1}=m$ we have
$
e^{2\pi i f}\frac{w_n}{w_{n+1}\circ\sigma}=1\text{ a.s. for all $n$}
$. This gives the chain of identities
\begin{align*}
w_n=e^{-2\pi i f}w_{n+1}\circ \sigma=e^{-2\pi i (f+f\circ\sigma)}w_{n+2}\circ \sigma^2=\cdots=e^{-2\pi i \sum_{j=0}^{k-1}f\circ\sigma^k} w_{n+k}\circ\sigma^k.
\end{align*}
It follows that $w_n/w_{n+1}=(w_{n+k}/w_{n+k+1})\circ\sigma^k$ for all $k$. Hence
$w_n/w_{n+1}$ is $\sigma^{-k}\mathfs F$--measurable for all $k$. By the zero-one law, $w_n/w_{n+1}$ is constant almost surely.
In particular, there exists a constant $c$ such that  $A_2-A_1\in c+\Z$ $m$--a.e., and the step follows with $A:=A_1$ and $\kappa:=\kappa_1+c$.

\smallskip
\noindent
{\sc Step 3:} {\em There exists $a:\Omega\to\R$ constant on cylinders of length one such that  $f+a-a\circ\sigma+\kappa\in\Z$ $m$-a.e.}

\smallskip
\noindent
{\em Proof.\/}  Let $L:L^1(\Omega)\to L^1(\Omega)$ denote the {\bf transfer operator}\index{transfer operator} of $\sigma:\Omega\to\Omega$, which describes the action of $\sigma$ on mass densities on $\Omega$:
$
\sigma_\ast[\vf d\mu]=L\vf d\mu.
$
Formally, $
L\vf:=\frac{dm_\vf\circ\sigma^{-1}}{dm},\text{ where }m_\vf:=\vf dm
$. We will need the following (standard) facts:
\begin{enumerate}[(a)]
\item If $\vf$ depends only on the first $m$-coordinates, then $L\vf$ depends only on the first $(m-1)\vee 1$--coordinates. Specifically, $(L\vf)[(y_i)_{i\geq 1}]=\Phi(y_1,\ldots,y_{m-1})$ where
$$
\Phi(y_1,\ldots,y_{m-1}):=\E[\vf(X_1,\ldots,X_m)|X_i=y_i\ (1\leq i\leq m-1)];
$$

\item $L\vf$ is characterized by the condition $\int \psi L\vf dm=\int \psi\circ\sigma \vf dm$ $\forall\psi\in L^\infty(\fS)$;
\item $L(\vf\psi\circ\sigma)=\psi L\vf$ $\forall\vf\in L^1,\psi\in L^\infty$;
\item $L1=1$;
\item $\forall \vf\in L^\infty$, $L^n\vf\xrightarrow[n\to\infty]{}\int\vf dm$ in $L^1$.
\end{enumerate}

\smallskip
Part (b) is standard. Parts (c) and (d) follow from (b) and the $\sigma$-invariance of $m$.
Part (a) follows from  (b), and the identity
\begin{align*}
&\int\psi L\vf dm=\int \psi dm_\vf\circ\sigma^{-1}=\int \psi\circ\sigma \vf dm=\E[\psi(X_2,X_3,\ldots)\vf(X_1,\ldots,X_m)]\\
&=\E\bigl(\psi(X_2,X_3,\ldots)\E[\vf(X_1,\ldots,X_m)|X_2,X_3,\ldots]\bigr)
\overset{!}{=}\E\bigl(\psi(X_2,X_3,\ldots)\E[\vf|X_2,\ldots,X_m]\bigr)\\
&=\int\psi\Phi dm
\end{align*}
where $\overset{!}{=}$ is because of the Markov property.
To see part (e) note that it is enough to consider $\vf\in L^\infty $ such that $\int \vf dm=0$ (otherwise work with $\vf-\int\vf dm$). For such functions,
\begin{align*}
&\|L^n\vf\|_1=\int \mathrm{sgn}(L^n\vf) L^n\vf\, dm
=\int \mathrm{sgn}(L^n\vf)\circ\sigma^n \cdot\vf\, dm\\
&=\int \mathrm{sgn}(L^n\vf)\circ\sigma^n \E(\vf|\sigma^{-n}\mathfs F)\, dm\leq \int |\E(\vf|\sigma^{-n}\mathfs F)| dm
\end{align*}
The integrand is uniformly bounded (by $\|\vf\|_\infty$), and it converges pointwise to $\E(\vf|\bigcap \sigma^{-n}\mathfs F)=\E(\vf|\{\emptyset,\Omega\})=\E(\vf)=0$.

\smallskip
Let $w:=e^{2\pi i A}$ where $A:\Omega\to\R$ is as in step 2, and assume w.l.o.g. that $\kappa=0$ (else absorb it into $f$). Set $S_n=f+f\circ\sigma+\cdots+f\circ\sigma^{n-1}$, then $e^{-2\pi i f}=w/w\circ\sigma$, whence $e^{-2\pi i S_n}=w/w\circ\sigma^{n}$. By (c), for all $\vf\in L^1(\Omega)$,
\begin{align*}
 w L^n(e^{-2\pi i S_n}\vf)&= L^n(e^{-2\pi i S_n}w\circ\sigma^n\vf)=L^n(w\vf)\xrightarrow[n\to\infty]{L^1}\int w\vf dm.
\end{align*}
Since $|w|=1$ a.e., $\exists m\geq 2$ and  $\exists \vf=\vf(x_1,\ldots,x_m)$ bounded measurable so that $\int w\vf dm\neq 0$. For this $\vf$, we have
$$
w^{-1}=\text{$L^1$-}\lim_{n\to\infty} \frac{L^n(e^{-2\pi i S_n}\vf)}{\int w\vf dm}.
$$

We claim that the right-hand-side depends only on the first coordinate. This is because $e^{-2\pi i f}\vf$ is function of  the first $m$ coordinates, whence by (a),
$L(e^{-2\pi i f}\vf)$ is a function of the first $(m-1)\vee 1$ coordinates. Applying this argument again we find that $L^2(e^{-2\pi i S_2}\vf)=L[e^{-2\pi i f} L(e^{-2\pi i f}\vf)]$ is a function of the first $(m-2)\vee 1$ coordinates. Continuing by induction, we find that $L^n(e^{-2\pi i S_n}\vf)$ is a function of $(m-n)\vee 1$-coordinates, and eventually of the first coordinate only.

So $w^{-1}$ is an $L^1$-limit of a functions of  the first coordinate. Therefore we can write $w[(x_i)_{i\geq 1}]=\exp[2\pi i a(x_1)]$ a.e.,  where $a:\fS\to\R$ is measurable. By construction $e^{2\pi i f}w/w\circ\sigma=1$, so $f(X_1,X_2)+a(X_1)-a(X_2)\in \Z$ almost surely. By stationarity,
$f(X_n,X_{n+1})+a(X_n)-a(X_{n+1})\in \Z$ almost surely for all $n$.
 \qed
\end{proof}

We now determine the domain of the rate functions for large deviations. We note that the results of
Chapter \ref{Chapter-LDP} concern
$\Prob[S_N\geq z V_N]=\Prob[S_N\geq z \sigma^2 (1+o(1)) N ]$, while in large deviation literature
it is common to use the normalization $\Prob[S_N\geq z N].$ To simplify the comparison with
other results we will assume till the end of this section that $\sigma^2=1$ which can always be achieved
by scaling $f.$

 Let $\cS_N=\ess\sup S_N$. Using the stationarity of $\{X_n\}$ and the homogeneity of $\mathsf f$ it is not difficult to see that  $\cS_{n+m}\leq\cS_n+\cS_m$, and therefore the limit
$$ \fs_+=\lim_{N\to\infty} \frac{\ess\sup S_N}{N} =\lim_{N\to\infty} \frac{\cS_N}{N} $$
exists. Repeating the same argument for $(-\mathsf f)$ gives that
$$\fs_-=\lim_{N\to\infty} \frac{\ess\inf S_N}{N} $$ exists as well.

Recall the notation for large deviation thresholds $\fc_-, \fc_+$ introduced in
\S \ref{SS-LDTreshold}.

\begin{theorem}
\label{ThRateDomH}
Let
$\mathsf {f}$ be an a.s. uniformly bounded homogenous additive functional on a uniformly elliptic homogeneous Markov chain, and assume $\mathsf f$ has zero mean and asymptotic variance $\sigma^2=1.$
Then
$\fc_+=\fs_+$ and $\fc_-=\fs_-$.
\end{theorem}

\begin{proof}
We prove the first identity, the second one is similar.

First, for any $\eps>0$,
$\Prob[S_N\geq (\fs_++\eps) N]=0$ for sufficiently large $N$, whence by Theorem \ref{ThLDP=LLTLD}, $\fc_+\leq\fs_+.$

Let $K:=\ess\sup|\mathsf f|$.
 For every $\eps>0$, for all sufficiently  large $M$,
$$\delta_M:= \Prob[S_M\geq (\fs_+-\eps) M]>0. $$
Let $\sigma(X_i,\ldots,X_j)$ denote the $\sigma$-field generated by $X_i,\ldots,X_j$. By uniform ellipticity, if $E\in \sigma(X_1,\ldots,X_{M+1})$ and $F\in \sigma(X_{M+3},\ldots,X_{2M+3})$, then
$
\Prob[E\cap F]\geq \epsilon_0 \Prob(E)\Prob(F)
$ {(see \eqref{GrMixing}).}
Consequently,
\begin{align*}
&\Prob[S_{2(M+2)}\geq 2(\fs_+-\eps) M-2K]\\
&\geq  \Prob\left[\sum_{k=1}^M f_k(X_k,X_{k+1})\geq M(\fs_+-\epsilon)\ ,\sum_{k=M+3}^{2M+2} f_k(X_k,X_{k+1})\geq M(\fs_+-\epsilon) \right]\\
&\geq \epsilon_0 \Prob\left[\sum_{k=1}^M f_k(X_k,X_{k+1})\geq M(\fs_+-\epsilon)\right]\Prob\left[\sum_{k=M+3}^{2M+2} f_k(X_k,X_{k+1})\geq M(\fs_+-\epsilon) \right].
\end{align*}
Thus by stationarity,
$
\Prob[S_{2(M+2)}\geq 2(\fs_+-\eps) M-2K]\geq \epsilon_0\delta_M^2.
$
Applying this argument repeatedly,  we find that for each $\ell$,
$$ \Prob[S_{(M+2)\ell} \geq ((\fs_+-\eps) M-2K)\ell ]\geq \left(\eps_0 \delta_M^2\right)^{\ell} . $$
Now Corollary \ref{CrIntRate} tells us that for all sufficiently large $M$,
$\DS \fc_+\geq \frac{(\fs_+-\eps)M-2K}{M+2}. $
Letting $M\to\infty$ we obtain $\fc_+\geq \fs_+-\eps.$ Since $\eps$ is arbitrary,
$\fc_+\geq\fs_+$.
\qed \end{proof}

\section{Perturbations of homogeneous chains}
\label{ScAlmHom}
Let $(\mathsf{X}, \mathsf{f})$
be a bounded homogenous additive functional on a uniformly elliptic Markov chain  with stationary measure $\mu$ and transition probability $\pi(x,dy)=p(x,y)\mu(dy)$.
We consider {\em non-homogeneous} perturbations $(\wt{\mathsf X}, \wt{\mathsf f})$  of the form
$$
\tf_n(x,y)=f(x,y)+g_n(x,y)\ , \tpi_n(x,dy)=\tp_n(x,y)\mu(dy).
$$
We assume that the strength of the perturbation decays at infinity.
Namely for each $\eps>0$ there is $n_0$ such that for $n\geq n_0$
$$ \|g_n\|_\infty\leq \eps\quad\text{and}\quad
1-\eps\leq \frac{\tp_n(x,y)}{p(x,y)}\leq 1+\eps. $$

\begin{theorem}
\label{ThEssRangePert}
If the additive functional $\mathsf{g}$ is center tight on $\mathsf X$, then $G_{ess}(\wt{\mathsf X},\wt{\mathsf f})=G_{ess}({\mathsf X},\mathsf{f}).$ If $\mathsf{g}$ is not center tight then
$G_{ess}(\wt{\mathsf X},\wt{\mathsf f})=\R.$
\end{theorem}

\begin{proof}
We note that it suffices to prove the result in the case $\wt{p}_n\equiv p.$ Indeed by our assumptions,
\begin{equation}
\label{UpToTwo}
\frac{1}{2} \leq \frac{\tp_n(x,y)}{p(x,y)}\leq 2
\end{equation}
if $n$ is sufficiently large. Since discarding
a finite number of terms does not change the essential range (since any functional vanishing for
large $n$ is center tight) we may assume that \eqref{UpToTwo} holds for all $n.$
Now Example \ref{Example-COM-is-Hereditary} shows that the essential range of the
functionals defined via $p$ and via $\wt{p}_n$ are the same. Thus we assume henceforth
that $\wt{p}_n\equiv p$ for all $n.$

If $\mathsf g$ is center tight then the essential ranges of $\mathsf f$ and $\wt{\mathsf f}$ are the same,
so we shall assume that $\mathsf g$ is not center tight, and prove that $D_N(\xi,\wt{\mathsf f})\to\infty$ for every $\xi\neq 0$.
Let $\fd:=d_n(\xi,\mathsf f)$ (the RHS does not depend on $n$ by stationarity).

 Suppose first that $\fd\neq 0.$ By Lemma \ref{Lemma-Sum}(2) we have
$$ \fd^2=d_n^2(\xi,\mathsf f)\leq 8\left[d_n(\xi,\wt{\mathsf f})^2+d_n(\xi,{\mathsf g})^2\right].$$
Next, the assumption $\|g_n\|_\infty\xrightarrow[n\to\infty]{}0$ implies that $d_n^2(\xi,{\mathsf g})\xrightarrow[n\to\infty]{}0$.
Accordingly $\DS d_n(\wt{\mathsf f}, \xi)^2\geq \frac{\fd^2}{10}$ for all $n$ large enough, so that
$D_N(\xi,\wt{\mathsf f})\to \infty$ as needed.

Next assume  $\fd=0.$  In this case for any hexagon $P_n$ we have
$e^{i\xi\Gamma(\mathsf f, P_n)}=1$, where $\Gamma(\mathsf f, \cdot)$ denotes the balance for the additive functional $\mathsf f$.
Hence
$e^{i\xi \Gamma(\wt{\mathsf f}, \cdot)}=e^{i\xi\Gamma({\mathsf g}, \cdot)}$,
and so
$$d_n(\xi,\wt{\mathsf f})= d_n(\xi,{\mathsf g}).$$

Let $\gamma_N:=\max_{n\geq N}\ess\sup|g_n|$, and
fix $\tau_0>0$ such that
$
|e^{it}-1|^2\geq \frac{1}{2}t^2\text{ for all }|t|<\tau_0.$
If $n\geq N$ and
 $0<|\eta|\leq \tau_0(6\gamma_N)^{-1}$, then
 \eqref{DN-UN-SmallXi} tells us that
$$ d_n^2 ( \eta,\mathsf{g})\geq \frac{ \eta^2}{2} u_n^2(\mathsf g) \text{ for all $n>N+3$}.$$
By assumption, $\mathsf g$ is not center-tight, so $\sum u_n^2(\mathsf g)=\infty$. It follows that
$D_N(\eta,g)\to \infty$ for all $0<|\eta|\leq \tau_0(6\gamma_N)^{-1}$.

By assumption, $\gamma_N\to 0$, so $D_N(\eta,g)\to \infty$ for all $\eta\neq 0$.
It follows that the  co-range of $\mathsf g$ equals $\{0\}$, and the essential range of $\mathsf g$ equals $\R$. \qed
\end{proof}

Next we discuss the large deviation thresholds for $\wt{\mathsf f}.$

\begin{theorem}
(a) If $\mathsf f$ is not a coboundary then
$\fc_\pm(\wt{\mathsf f})=\fc_\pm(\mathsf f)=
\fs_\pm(\mathsf f). $

(b) If $\mathsf f$ is a homogeneous
gradient, $\EXP(g_n)=0$ for all $n$,  and $\mathsf{g}$ is not center tight,
then $\fc_+(\wt{\mathsf f})=+\infty,\quad \fc_-(\wt{\mathsf f})=-\infty.$
\end{theorem}

\begin{proof}
The proof of part (a) is very similar to the proof of Lemma \ref{LmSmallChange} so we omit it.

In the proof of part (b) we may assume that $\mathsf f=0$ since adding a homogeneous gradient does not change the
large deviation threshold. In particular in the rest of the proof we will abbreviate
$S_N=S_N(\mathsf{g}),$ $S_{n_1, n_2}=S_{n_1, n_2}(\mathsf{g})=\sum_{k=n_1}^{n_2-1} g_k(X_k,X_{k+1}),$
$V_N=\Var(S_N(\mathsf{g})).$
Since $\mathsf g$ is not center tight, $V_N\to\infty.$

Assume without loss of generality that $\ess\sup|g|\leq \frac{1}{2}$,
then $\Var [g_k(X_k,X_{k+1})]\leq 1$ for all $k$.
Divide the interval $[0, N]$ into blocks
$$
 \biggl[n_1,n_2\biggr]\cup\{n_2+1\}\cup \biggl[n_3,n_4\biggr]\cup\cdots\cup \biggl[n_k,n_{k+1}\biggr]\cup \{n_{k+1}+1\}\cup \biggl[n_{k+2},N\biggr]
$$
where $n_i$ is increasing,
$ 1\leq \mathrm{Var}(S_{n_j, n_{j+1}}) \leq 2$
for $j\leq k+1$, and $\Var(S_{n_{k+2},N})\leq 1$.

Since $\|g_n\|_\infty\to 0$, $\min\{n_{j+1}-n_j: \ell\leq j\leq k\}\xrightarrow[\ell\to\infty]{}\infty$.
Also,
the analysis of \S \ref{SSAsymVarLD} shows that
$$ \frac{1}{V_N} \sum_{j} \mathrm{Var}(S_{n_j, n_{j+1}})\to 1. $$
In particular,  the number of blocks $\beta_N$, is between $V_N/2$ and $3V_N/2.$

Let $\DS M_j=\max_{n_j\leq l\leq n_{j+1}} \|g_l\|_\infty.$ Note that $M_j\to 0.$
Therefore applying Dobrushin CLT to the array $\displaystyle \{g_l/M_j\}_{n_j\leq l\leq n_{j+1}}$
we conclude that $S_{n_j, n_{j+1}}/\sqrt{\mathrm{Var}(S_{n_j, n_{j+1}})}$ is asymptotically normal.
In particular, for each $z>0$ there exists $\eta=\eta(z)>0$ such that for $j$ large enough and all $x_j\in \fS_{n_j}$
\begin{equation}
\label{CLTCons}
\Prob_{x_j}(S_{n_j, n_{j+1}}\geq 3z)\geq \eta.
\end{equation}
A uniform ellipticity argument similar to the one we used in the proof of Theorem~\ref{ThRateDomH} gives
$$ \Prob_{x_j}(S_N\geq \beta_N z)\geq c \eta^{\beta_N}, $$
where $c$ incorporates the contribution of blocks (with small $j$) where \eqref{CLTCons} fails.

Now Corollary \ref{CrIntRate} implies that $\fc^+\geq z.$ Since $z$ is arbitrary, $\fc^+=+\infty.$
A similar argument
shows that $\fc_-(\mathsf{g})=-\infty.$
\qed \end{proof}

\section{Small additive functionals.}
\label{SSSmall}
The perturbations of $\mathsf f\equiv 0$ were analyzed in the previous section, however, since this case is
of independent interest it makes sense to summarize the results obtained for this particular case.

\begin{theorem}
\label{ThSmallAF}
Let $\mathsf{g}$ be a uniformly bounded additive functional of uniformly elliptic Markov chain.
Suppose that $\EXP(g_n)=0$ and that $\DS \lim_{n\to\infty} \|g_n\|_\infty=0.$ Then

either $\mathsf g$ is center tight in which case $\DS \sum_{n=1}^\infty g_n$ converges almost surely

or $\mathsf g$ is not center tight in which case $S_N(\mathsf{g})$
satisfies non lattice LLT \eqref{LLT-non-arithmetic-limit} and $\fc_\pm(\mathsf{g})=\pm \infty.$
\end{theorem}

\begin{proof}
The non-center tight case was analyzed in \S \ref{ScAlmHom}. In the center tight case
the results of Chapter \ref{Chapter-Variance} tell us that $g$ can be decomposed as
$$ g_n(x,y)=c_n+a_{n+1}(y)-a_n(x)+h_n(x,y)\quad\text{where}\quad \sum_n \Var(h_n)<\infty. $$
Changing $a_n$ if necessary we may assume that $\EXP(a_n)=0$ in which case
$$\EXP(g_n)=0=\EXP(h_n+c_n). $$ Therefore the additive functional
$\wt{\mathsf{h}}=\mathsf{h}+\mathsf{c}$ has zero mean and finite variance. Hence
by Theorem \ref{Proposition-Kolmogorov-Three-Series}
$\DS \sum_{n=1}^\infty (h_n+c_n)$ converges almost surely.
In summary $S_N(\mathsf{g})-a_N+a_1$ converges almost surely, and hence
$S_N(\mathsf{g})-a_N$ converges almost surely. On the other hand equation
\eqref{DefGrad} shows that $\DS \lim_{N\to\infty} a_N=0$
completing the proof.
\qed\end{proof}

The following result which a direct consequence of Theorem \ref{ThSmallAF} shows that
for small additive functionals a vague limit of the local distribution of $S_N$ always exists.

\begin{corollary}
\label{CrLocDistSmall}
Let $\mathsf{g}$ satisfy the assumptions of Theorem \ref{ThSmallAF}. Then either
 and $S_N$ converges a.s. to some random variable
$\cS$ in which case for each continuous compactly supported function $\phi$
$$ \lim_{N\to\infty} \EXP(\phi(S_N))=\EXP(\phi(\cS)) $$
or $S_N$ satisfies a non-lattice LLT. That is, for each continuous compactly supported function $\phi$
for each sequence $z_N$ such that the limit $\DS z=\lim_{N\to\infty} \frac{z_N}{\sqrt{V_N}}$ exists
we have
$$ \lim_{N\to\infty} V_N \EXP(\phi(S_N))=\EXP(\phi(\cS))=\frac{e^{-z^2/2}}{\sqrt{2\pi}}
\int_{-\infty}^\infty \phi(s) ds.  $$
\end{corollary}

\section{Equicontinuous additive functionals}

 In this section we examine the consequences of {\em topological} assumptions on $\mathsf f$ and $\mathsf X$.
Specifically we will say that $(\mathsf X,\mathsf f)$ is {\bf equicontinuous}\index{equicontinuous additive functional}\index{additive functional!equicontinuous} if
\begin{enumerate}[(1)]
\item[(T)] $(\fS_n,\mathcal B_n,\mu_n)$ are complete separable metric spaces, $\mathcal B_n$ are the Borel $\sigma$-algebras, and $\mu_n$ are Borel probability measures;
\item[(S)]
for every $\epsilon>0$ there exists $\delta>0$ such that for all $x_n\in\fS_n$ and $n\geq 1$,  $\mu_n[B(x_n,\epsilon)]>\delta$. Here $B(x,{ \eps}):=\{y\in \fS_n:\dist(x,y)<{ \eps}\}$.

\item[(U)] for every $\epsilon>0$ there exists $\delta>0$ such that for all $n\geq 1$ and $x_n,y_n\in\fS_n$,
$\dist(x_n,y_n)<\delta\Rightarrow |f_n(x_n)-f_n(y_n)|<\epsilon$.
\end{enumerate}

\subsection{Range.}

\begin{theorem}
\label{ThContChains}
Suppose $(\mathsf X,\mathsf f)$ is equicontinuous and a.s. uniformly bounded. Assume in addition the following:
\begin{enumerate}[(a)]
\item {\bf One-step ellipticity condition:}\index{one-step ellipticity condition}\index{uniform ellipticity!one-step} $\exists \epsilon_0$ s.t. for every $n$,  $\pi_n(x,dy)=p_n(x,y)\mu_{n+1}(dy)$ where $\epsilon_0\leq p_n(x,y)\leq \epsilon_0^{-1}$.
\item $\fS_n$ are all connected.
\end{enumerate}
Then $\mathsf f$ is  either irreducible with algebraic range $\R$, or it is center tight.
\end{theorem}

\begin{proof}
Choose $c_1>0$ such that
$\DS |e^{i\theta}-1|^2=4\sin^2\left(\frac{\theta}{2}\right)\geq c_1 \theta^2$ for all $|\theta|\leq 0.1$.
We fix $\xi\neq 0$, and consider the following two cases:
\begin{enumerate}[(I)]
\item $\exists N_0$ such that $|\xi\Gamma(P)|<0.1$ for every position $n$ hexagon $P$, for each $n\geq N_0$.
\item $\exists n_k\uparrow\infty$  and $\exists$ position $n_k$ hexagons $P_{n_k}$ such that $|\xi\Gamma(P_{n_k})|\geq 0.1$.
\end{enumerate}

In case (I), for all $n\geq N_0$,  $d_n^2(\xi)=\E(|e^{i\xi\Gamma}-1|^2)\geq c_1\E(\Gamma^2)\equiv c_1 u_n^2$.
So either $\sum u_n^2=\infty$ and then $\sum d_n^2(\xi)=\infty$ for all $\xi\neq 0$, and $\mathsf f$ is irreducible with essential range $\R$; or $\sum u_n^2<\infty$ and then $\mathsf f$ is center-tight by Corollary \ref{Corollary-Tight}.

In case (II), for every $k$ there is a position $n_k$ hexagon
 $P_{n_k}$ with $|\xi\Gamma(P_{n_k})|\geq 0.1.$ There is also a position $n_k$ hexagon $P_{n_k}'$ with balance zero (such hexagons always exist { because we can  take
 $y_{n_k-1}=x_{n_k-1},$ $y_{n_k}=x_{n_k}$}). We would like to apply the intermediate value to deduce the existence of a position $n_k$ hexagon $\ov{P}_{n_k}$ such that $0.05<\xi\Gamma(\ov{P}_{n_k})<0.1$. To do this we note that:
\begin{enumerate}[$\circ$]
\item  Because of the one-step ellipticity condition, the space of position $n_k$ hexagons is homeomorphic to $\fS_{n_k-2}\times\fS_{n_k-1}^2\times\fS_{n_k}^2\times \fS_{n_k}$.
\item The product of connected topological spaces is connected.
\item Real-valued continuous functions on connected topological spaces satisfy the intermediate value theorem.
\item The balance of hexagon depends continuously on the hexagon.
\end{enumerate}
So $\ov{P}_{n_k}$ exists. Necessarily, $|e^{i\xi\Gamma(\ov{P}_{n_k})}-1|\geq c_1\xi^2\Gamma^2(\ov{P}_{n_k})=:c_2$.

Write $\DS
\brP_{n_k}$  in coordinates:  $\DS
\brP_{n_k}:=\left(\brx_{n_k-2}; \begin{array}{c}\brx_{n_k-1}\\ \bry_{n_k-1}\end{array};
\begin{array}{c}\brx_{n_k}\\ \bry_{n_k}\end{array}; \bry_{n_k+1}\right)
$.
By the equicontinuity of $\mathsf f$, $\exists \epsilon>0$ such that $|e^{i\xi\Gamma(P)}-1|>\frac{1}{2}c_2$ for every hexagon $P$ whose coordinates are in the $\epsilon$-neighborhood of the coordinates of $\brP_{n_k}$.
By the equicontinuity of $\mu_n$ and the one-step ellipticity condition, this collection of hexagons $P$ have hexagon measure $\geq \delta$ for some $\delta>0$ independent of $k$. So
$
d_{n_k}^2(\xi)\geq \frac{1}{2} c_2 \delta.
$

Summing over all $k$, we find that $\sum d_{n_k}^2(\xi)=\infty$. Since $\xi\neq 0$ was arbitrary, $(\mathsf X,\mathsf f)$ has essential range $\R$.
\qed
\end{proof}
\subsection{Large deviation threshold.}

\begin{lemma}
\label{LmReach-Variational}
Suppose that $\fS_n$ are metric spaces, $f_n$ are equicontinuous,  and
for each $\eps>0$ there exists $\delta>0$ such that
if $p_n(x, y)>0$ then
\begin{equation}
\label{BallMeasure}
\pi_n(x, B(y, \eps))>\delta.
\end{equation}
Suppose that $V_N>c N$ and that
$$ \limsup_{N\to\infty}   \frac{\displaystyle \inf_{x_1, \dots, x_{N+1}}
\sum_{j=1}^N f_j(x_j , x_{j+1})}{V_N}
<z <
\liminf_{N\to\infty}
\frac{\displaystyle \sup_{x_1, \dots, x_{N+1}}
\sum_{j=1}^N f_j(x_j , x_{j+1})}{V_N} . $$
Then $z\in \cC.$
\end{lemma}

Note that assumption \eqref{BallMeasure} is satisfied whenever $\mathsf X$ satisfies (S) and the one step ellipticity condition.
We also remark that Example \ref{ExCoreDrop} shows that
equicontinuity assumption on  $\mathsf f$ is essential.

\begin{proof}
Fix $N.$ Consider first the case where $z> \frac{\EXP(S_N)}{V_N}.$
By assumption there is an $\eps$ such that for all sufficiently large $n$ there is a sequence $\brx_1, \dots, \brx_{N+1}$
such that
$$ \sum_{j=1}^N f_j (\brx_j, \brx_{j+1})\geq (z+\eps) V_N. $$
By ellipticity, for each $x\in \fS_1$ there a sequence $\tx_1, \tx_2\dots \tx_{N+1}$ such that
$\tx_1=x$ and
$$ \sum_{j=1}^N f_j (\tx_j, \tx_{j+1})\geq (z+\eps) V_N-4K,\text{ where }K:=\ess\sup|\mathsf f|. $$
(In fact one can take $\tx_j=\brx_j$ for $j\geq 3$.) By uniform continuity of $f_j$ and the fact that
$V_N$ grows linearly, there is
$r$ such that if $X_j\in B(\tilde x_j, r)$ for $j\leq N+1$ then
$$ \sum_{j=1}^N f_j (X_j, X_{j+1})\geq (z+\eps/2) V_N-4K .$$
By \eqref{BallMeasure} there is $\delta>0$ such that
$\Prob_x(X_j\in B(\tx_j, r))\geq \delta^N .$
Hence
$$ \Prob(S_N\geq (z+\eps/3) V_N)\geq  \delta^N .$$
 Next, by the CLT and the assumption that  $z> \frac{\E(S_N)}{V_N}$, if $\eps$ is small enough and $N$ is large enough, then\footnote
{Alternatively, combining  Theorem~\ref{Theorem-I_N}(1)) and Theorem \ref{ThLDOneSided}
(applied to $-S_N$)
we get that $\DS \Prob\left(S_N\leq -\frac{\EXP(S_N)}{V_N}-\eps\right)\geq \delta^N$ provided that
$\eps$ is small enough and $\delta$ is close to $1.$
}
  $\Prob(S_N\leq (z-\eps) V_N)\geq  \delta^N$.
 Now Theorem \ref{ThLDP=LLTLD} shows that $z\in \cC.$

The case $z<\frac{\EXP(S_N)}{V_N}$ is analyzed similarly now using the estimate

$\displaystyle \limsup_{N\to\infty}   \frac{\displaystyle \inf_{x_1, \dots, x_{N+1}}
\sum_{j=1}^N f_j(x_j , x_{j+1})}{V_n} \leq
z-\eps. $ \qed
\end{proof}

\begin{corollary}
Under the assumptions of Lemma \ref{LmReach-Variational} if
$$ \mathfrak{z}^-=\lim_{N\to\infty}
\frac{\displaystyle \inf_{x_1, \dots, x_{N+1}}
\sum_{j=1}^N f_j(x_j , x_{j+1})}{V_N}  \quad \text{exists then} \quad \mathfrak{z}^-=\fc^-,$$
$$ \mathfrak{z}^+=\lim_{N\to\infty}
\frac{\displaystyle \sup_{x_1, \dots, x_{N+1}}
\sum_{j=1}^N f_j(x_j , x_{j+1})}{V_N}  \quad \text{exists then} \quad \mathfrak{z}^+=\fc^+. $$
\end{corollary}

\begin{proof}
We will prove the second statement, the first one is similar.
$\mathfrak{z}^+\leq \fc^+$ by Lemma \ref{LmReach-Variational}. On the other hand if
$z>\mathfrak{z}^+$ then for large $N$,
$\Prob\left(S_N>V_N z\right)=0.$ Hence $\fc^+\leq \mathfrak{z}^+.$
\qed
\end{proof}

We now restate the result of the last
corollary in a slightly different way under an extra assumption.
Namely, we suppose that

\begin{equation}
\label{JumpComp}
\text{$\fS_n$ are compact \& $\forall x_n, x_{n+1}:  p_n(x_n, x_{n+1})>0$}
\end{equation}

\begin{definition}
Let $\cM_N$ denote the space of sequences $\DS \bx\in \prod_n \fS_n$ such  that if
$\by_n=\bx_n$ for $n\geq N+1$ then
$$ \sum_{n=1}^N f(\by_n, \by_{n+1})\geq \sum_{n=1}^N f(\bx_n, \bx_{n+1}). $$
Denote $\DS \cM=\bigcap_{N=1}^\infty \cM_N. $
The elements of $\cM$ will be called {\bf minimizers}.
\end{definition}

The properties of $\cM_N$ are summarized below.

\begin{lemma}
\label{LmPropMin}
 Suppose that \eqref{JumpComp} holds and that $f_n:\fS_n\to [-K, K]$ are continuous. Then
\begin{enumerate}[(a)]
\item $\cM_N$ are closed sets.

\item If $N\geq M$ then $\cM_{N}\subset \cM_M$.

\item $\cM$ is non empty.

\item  If $\DS \sum_{n=1}^N f(\bx_n, \bx_{n+1})=\inf_{\by} \sum_{n=1}^N f(\by_n, \by_{n+1}) $
then $\bx\in \cM_N.$

\item If $\bx\in \cM_N$ then
$\DS \sum_{n=1}^N f(\bx_n, \bx_{n+1})\leq \inf_{\by} \sum_{n=1}^N f(\by_n, \by_{n+1})+2K . $
\end{enumerate}
\end{lemma}

\begin{proof}
(a) If $\cM_N\ni \bx^j\xrightarrow[j\to\infty]{}\bx\not\in \cM_N$, then  there would exist $\bar{\bx}$ such that
$\bar\bx_n=\bx_n$ for $n\geq N+1$, and
$$  \sum_{n=1}^N f(\bar\bx_n, \bar\bx_{n+1})< \sum_{n=1}^N f(\bx_n, \bx_{n+1}). $$
Let $\by^j$ be the sequence such that $\by^j_n=\bx^j_n$ for $n\geq  N+1$,
$\by^j_n=\bar \bx_n$ for $n\leq N$. By the continuity of $f$,
$$  \sum_{n=1}^N f(\by_n^j, \by_{n+1}^j)< \sum_{n=1}^N f(\bx_n^j, \bx_{n+1}^j) $$
for large $j$ contradicting, $\bx^j\in \cM_N.$

Next let $\bx\in \cM_N$ and $\bx_n=\by_n$ for $n\geq M$ with $N>M. $
Then
$$ \sum_{n=1}^M \left[f_n(\by_n, \by_{n+1}) -f_n(\bx_n, \bx_{n+1}) \right]=
\sum_{n=1}^N \left[f_n(\by_n, \by_{n+1}) -f_n(\bx_n, \bx_{n+1}) \right]\geq 0. $$
This proves (b).

Combining (a) and (b) we see that $\cM_n$ are nested compact sets, hence their intersection is
non-empty.

(d) is clear.

Next, let $\bx$ be the argmin of $\sum_{n=1}^N (\bz_n, \bz_{n+1})$
and $\by\in \cM_N.$ Let $\bz$ be such that
$\bz_n=\bx_n,$ for $1\leq n\leq N-1$ and $\bz_n=\by_n$ for $n\geq N.$
Then
$$ \sum_{n=1}^N f_n(\by_n, \by_{n+1}) \leq \sum_{n=1}^N f_n(\bz_n, \bz_{n+1}) \leq
\sum_{n=1}^N f_n(\bx_n, \bx_{n+1})+2K
$$
proving (e).
\qed \end{proof}

 If, in addition to the assumptions of Lemma \ref{LmPropMin} we also suppose that
$\mathsf f$ satisfies (U), then part (e) of the lemma implies that for each $\bx\in \cM$ (which is non-empty by part (c))
$$ \fc_-=\lim_{N\to\infty} \frac{1}{N} \sum_{n=1}^N f(\bx_n, \bx_{n+1}). $$

\section{Notes and references}
Theorem \ref{ThLLTHom} is well-known, see \cite{N,HenHer,RE,GH,PP90}.
We note that in the homogeneous
setting the assumptions on $f$ can be significantly weakened. In particular, the assumption that
$f$ is bounded can be replaced by the assumption that the distribution of $f$ is in the domain
of attraction of the Gaussian distribution \cite{N}, one can allow $f$ to depend on infinitely
many $X_n$ assuming that the dependence of $f(x_1,x_2,\ldots)$ on $(x_n,x_{n+1},\ldots)$ decays exponentially in $n$ \cite{GH}, and the ellipticity
assumption can be replaced by the assumption that the generator has a spectral gap
\cite{N, HenHer}. In particular, the LLT holds under the Doeblin condition saying that $\exists \eps_0>0$
and a measure $\zeta$ on $\fS$ such that
$$ \pi(x, dy)=\eps_0 \zeta+(1-\eps_0) \wt{\pi}(x, dy) $$
where $\wt{\pi}$ is an arbitrary transition probability (cf. equation \eqref{DoeblinTD} in the proof of
Lemma \ref{Lemma-Contraction}).
There are also versions of this theorem for $f$ in the domain of attraction of a stable law, see \cite{Aaronson-Denker-LLT}.

The aforementioned weaker conditions however are not sufficient
to get LLT in the large deviation regime, in fact large deviation probabilities could be polynomially
small for unbounded functions, see \cite{Wentzell}.

There is a vast literature on the sufficient conditions for the Central Limit Theorem for homogenous chains,
see \cite{DL, Gordin, GL, HenHer, KV, KLO, MW} and references wherein, however, the local limit
theorem is much less understood, see notes to Chapter \ref{Section-LLT-irreducible}.

The characterization of coboundaries in terms of vanishing of the asymptotic variance $\sigma^2$
is due to Leonov \cite{Leonov}. A large number of papers discuss the regularity of the gradients
in case an additive functional is a gradient, see
\cite{BHN, dLMM, KK96, Liv71, Liv72, PP06, PP90, Wil}
and the references wherein. Our approach is closest to \cite{D-Prev, KK96, PP90}.
We note that the condition $u(f)=0$ which is sufficient for $f$ being a coboundary, is
simpler than the equivalent condition $\sigma^2=0.$ For example for finite chains,
to compute  $\sigma^2$ one needs to compute infinitely many correlations
$\EXP(f_0 f_n)$ while checking that $u=0$ involves checking balance of finitely many
hexagons.

Inhomogeneous Markov processes arising from perturbations of homogeneous Markov chains as in section \ref{ScAlmHom} arise naturally in some stochastic optimization algorithms such as the Metropolis algorithm. For large deviations and other limit theorems for such examples, see \cite{Dietz-Sethuraman-Electronic,Dietz-Sethuraman-LDP} and references therein.

Minimizers play important role is statistical mechanics where they are called {\em ground states}.
See e.g. \cite{Sinai-SM, RAS}. In the case the phase spaces $\fS_n$ are non-compact and/or
the observable $f(x,y)$ is unbounded, the minimizers have an interesting geometry, see e.g.
\cite{ConIt}. For finite states we have the following remarkable result \cite{Br03}:
 for each $d$ there is a constant
$p(d)$ such that for any   homogeneous Markov chain with $d$ states for any additive functional
we have
$$ \fs_+=\max_{q\leq p} \frac{1}{q} \max_{x_1, \dots x_q}
\left[f(x_1, x_2)+\dots+f(x_{q-1}, x_q)+f(x_q, x_1)\right]. $$
This result is false for more general homogenous chains, consider for example the case
$\fS=\N$ and $f(x,y)=1$ if $y=x+1$ and $f(x,y)=0$ otherwise.

Corollary \ref{CrLocDistSmall} was proven in \cite{D-Ind} for inhomogeneous sums of independent
random variables (in the independent case one does not need the assumption that
$\DS \lim_{n\to\infty} \|g_n\|_\infty=0$ since the gradient obstruction does not appear in the
independent case).

\chapter{LLT for  Markov chains in random environment}
\label{ChMCRE}

\noindent
{\em We prove quenched local limits theorems for Markov chains in random environment with stationary ergodic noise processes.}

\section{Markov chains in random environment}\index{Markov chain!random environment}\label{Section-MCRE}

Informally, Markov chains in random environment (MCRE) are  Markov chains whose transition probabilities depend on a noisy parameter $\omega$ which varies in time.\footnote{MCRE should not be confused with ``random walks in random environment," see \S\ref{Section-Notes-Chapter-8}.}  It is customary to model the time evolution of $\omega$ by orbits of a dynamical system called the ``noise process." Here are the formal definitions:

\medskip
\noindent
{\bf Noise process:}\index{Noise process} This is an ergodic
 measure preserving invertible Borel transformation $T$ on a standard
   measure space $(\Omega,\mathfs F,{m})$.  ``{\bf Invertible}" means that there exists $\Omega_1\subset\Omega$ of full measure such that $T:\Omega_1\to\Omega_1$ is injective and surjective, and $T^{-1}, T:\Omega_1\to\Omega_1$ are measurable.\footnote{Invertibility is convenient, but not necessary. Non-invertible ergodic noise processes can always be replaced by their ergodic and invertible natural extensions. See \cite[Ch. 10]{Cornfeld-Fomin-Sinai-Book}}    ``{\bf Measure preserving}" means that\index{measure preserving transformations}  for every $E\in\mathfs F$, $m(T^{-1}E)=m(E)$. ``{\bf Ergodic}"\index{ergodic} means that for every $E\in\mathfs F$,
 $T^{-1}E=E\Rightarrow m(E)=0$ or $m(E^c)=0$.

If $m(\Omega)<\infty$ then we will speak of a {\bf finite noise process}, and we will always normalize $m$ so that $m(\Omega)=1$. If $m(\Omega)=\infty$, then we will speak of an {\bf infinite noise process}. The infinite noise processes we consider here will all be defined on  $\sigma$-finite non-atomic measure spaces. Such processes arise naturally in the study of noise driven by a null recurrent Markov chain, see Example \ref{Example-NR-Markov-Noise} below.

\medskip
\noindent
{\bf Markov chains in Random Environment (MCRE):} A MCRE with noise process $(\Omega,\mathfs F,m,T)$ is given by the following data:
\begin{enumerate}[$\circ$]
\item {\bf State space:} A separable complete metric space $\mathfrak S$, with its Borel $\sigma$-algebra $\mathfs B$.

\medskip
\item {\bf Random transition kernel:} A measurable family of Borel probability measures  $\pi(\omega,x,dy)$ on $(\fS,\mathfs B)$, indexed by $(\omega,x)\in\Omega\times \fS$. Measurability means that  $(\omega,x)\mapsto \int \vf(y)\pi(x,\omega,dy)$ is measurable for every bounded Borel $\vf:\fS\to\R$.

\medskip
\item {\bf Initial probability distribution:} A measurable family of Borel  probability measures $\mu_\omega$ on $(\fS,\mathfs B)$ indexed by $\omega\in\Omega$,
Measurability means that for all bounded Borel $\vf:\fS\to\R$, $\omega\mapsto \int \vf(x)\mu_\omega(dx)$ is measurable.
\end{enumerate}
This data gives for each $\omega$  an inhomogeneous Markov chain $\mathsf X^\omega=\{X^\omega_n\}$  with state space $\fS$, initial distribution $\mu_\omega$, and
transition kernels
$
\pi_n^\omega(x,dy)=\pi(T^{n}\omega, x, dy).
$

Here a some examples. Suppose $(\mathfrak S,\mathfs B,\mu_0)$ is a standard probability space, $S$ is a finite or countable set, and $\{\pi_i(x,dy)\}_{i\in S}$ are transition probabilities on $\mathfrak S$.

\begin{example}[Bernoulli noise]\index{Bernoulli noise process}\index{Noise process!Bernoulli} \end{example}
Consider the noise process $(\Omega,\mathfs F,m,T)$ where
\begin{enumerate}[$\circ$]
\item $\Omega=S^\Z=\{(\cdots,\omega_{-1},\omega_0,\omega_1,\cdots):\omega_i\in S\}$;
\item $\mathfs F$ is generated by the {\bf cylinders}\index{cylinders} ${_k}[a_k,\ldots,a_n]:=\{\omega\in\Omega: \omega_i=a_i, k\leq i\leq n\}$
\item $\{p_i\}_{i\in S}$ are non-negative numbers s.t. $\sum p_i=1$, and  $m$ is the unique measure s.t. $m({_k}[a_k,\ldots,a_n])=p_{a_k}\cdots p_{a_n}$ for all cylinders.
\item $T:\Omega\to\Omega$ is the {\bf left shift map}, $T[(\omega_i)_{i\in\Z}]=(\omega_{i+1})_{i\in\Z}$
\end{enumerate}
It's well-known that $(\Omega,\mathfs F,\mu,T)$ is an ergodic probability preserving map.

Define
$
\pi(\omega, x,dy):=\pi_{\omega_0}(x,dy)
$.
Notice that $\pi(T^n\omega,x,dy)=\pi_{\omega_n}(x,dy)$, and $\omega_n$ are iid's taking the values $i\in S$  with probabilities $p_i$. Since $\omega_n$ are iid,  $\{\mathsf X^\omega: \omega\in\Omega\}$ represent a random Markov chain whose transition probabilities vary randomly and independently in time.

\begin{example}[Positive recurrent Markov noise]\label{Example-PR-noise}\index{Markovian noise process}\index{Noise process!Markovian} \end{example}
Suppose $(Y_n)_{n\in\Z}$ is a stationary ergodic Markov chain with state space $S$ and a stationary probability vector $(p_s)_{s\in S}$. In particular, $(Y_n)_{n\in\Z}$ is positive recurrent.
Let:
\begin{enumerate}[$\circ$]
\item $\Omega:=\{(\omega_i)\in S^\Z: \Prob[Y_1=\omega_i, Y_2=\omega_{i+1}]\neq 0\text{ for all }i\in\Z\}$;
\item $\mathfs F$ is the $\sigma$-algebra generated by the cylinders (see above);
\item $m$ is the unique (probability) measure such that $m({_k[}a_k,\ldots,a_n])=\Prob[Y_k=a_k,\ldots,Y_n=a_n]$ for all cylinders;
\item $T$ is the left shift map (see above).
\end{enumerate}
Define as before, $
\pi(\omega, x,dy):=\pi_{\omega_0}(x,dy)
$. The resulting MCRE represents a Markov chain whose transition probabilities at time $n=1,2,3,\ldots$ are $\pi_{Y_{n-1}}(x, dy)$.

\begin{example}[General stationary ergodic noise processes]
\end{example}
The previous construction works verbatim with any stationary ergodic stochastic process $\{Y_n\}$ taking values in $S$.
The assumption that $S$ is countable can be replaced by requiring only that $S$ be   complete, separable, metric space, see e.g. \cite{Doob}.

\begin{example}[Quasi-periodic noise]\index{Quasi-periodic noise process}\index{Noise process!Quasi-periodic}\end{example}
Let  $(\Omega,\mathfs F,m,T)$ be the circle rotation: $\Omega=\bbT^1:=
\{\omega\in\mathbb C:|\omega|=1\}$; $\mathfs F$ is the Borel $\sigma$-algebra; $m$ is the normalized Lebesgue measure; and $T:\Omega\to\Omega$ is the rotation by an angle $\alpha$: $T(\omega)=e^{i\alpha}\omega$.
$T$ is probability preserving, and  it is well-known that $T$ is ergodic iff  $\alpha/2\pi$ is irrational.

Choose a partition of the unit circle $\Omega$ into disjoint arcs $\{I_i\}_{i\in S}$ and define $\vf:\Omega\to S$ by $\vf(\omega)=i$ for $\omega\in I_i$. For example, if $S=\{1,2\}$ we can take $I_1,I_2$ to be two equal halves of the circle. Next define
$$
\pi(\omega,x,dy)=\pi_{\vf(\omega)}(x,dy)
$$
Now $\mathsf X^\omega$ are inhomogeneous Markov chains whose transition probabilities  vary quasi-periodically: They are given by  $\pi_{\vf(e^{in\alpha}\omega)}(x,dy)$.

More generally, one can take a $d$ parameter measurable family of transition probabilities $\pi_{\bomega}(x,y)$,
where $\omega=(\omega_1, \omega_2, \dots \omega_d)\in \R^d/\Z^d$, fix some ``initial phase" $(\bromega_1, \dots, \bromega_d)$,  and consider the chain with
transition probabilities
$$\pi_n(x,y)=\pi_{(\bromega_1+n\alpha_1, \dots, \bromega_d+ n\alpha_d)\mod\Z^d}(x,y). $$

\begin{example}[Null recurrent Markov noise]\label{Example-NR-Markov-Noise}\index{Noise process!null recurrent Markov}
\end{example}

This is an example with infinite noise process.
 Suppose $(Y_n)_{n\in\Z}$ is an ergodic null recurrent Markov chain with countable state space $S$, and stationary positive vector $(p_i)_{i\in S}$. Here $p_i>0$ and (by null recurrence) $\sum p_i=\infty$. For example, $(Y_n)_{n\in\Z}$ could be the simple random walk on $\Z^d$ for $d=1,2$, with the stationary measure which assigns the same mass to each site of $\Z^d$. Let
\begin{enumerate}[$\circ$]
\item $\Omega=\{(\omega_i)_{i\in\Z}\in S^\Z: \Prob[Y_1=\omega_i,Y_2=\omega_{i+1}]\neq 0\text{ for all }i\in\Z\}$;
\item $\mathfs F$ is the $\sigma$-algebra generated by the cylinders;
\item $m$ is the unique (infinite)  Borel measure which satisfies for each cylinder
$$
m({_k[}a_k,\ldots,a_n])=p_{a_k}\Prob[Y_i=a_i\ (k\leq i\leq n)|Y_k=a_k]
$$
\item $T:\Omega\to\Omega$ is the left shift map $T[(\omega_i)_{i\in\Z}]=\omega_{i+1}$.
\end{enumerate}
Then it is well-known that $(\Omega,\mathfs F,m,T)$ is an {\em infinite} ergodic measure preserving invertible map, see \cite{Aaronson-Book}.

Just as in Example \ref{Example-PR-noise},
one can easily construct many MCRE with transition probabilities $\pi_{Y_n}(x,dy)$ which vary randomly in time according to $(Y_n)_{n\in\Z}$. For each particular realization of $\omega=(Y_i)_{i\in\Z}$, $\mathsf X^\omega$ is an ordinary inhomogeneous Markov chain (on a probability space). But as we shall see below, some features of $\mathsf X^\omega$ such as the growth of variance, are different than in the finite noise process case.

\begin{example}[Transient Markov noise: a non-example]
\end{example}

The previous construction fails for {transient} Markov chains such as the random walk on $\Z^d$ for $d\geq 3$, because in the transient case, $(\Omega,\mathfs F,m,T)$ is not ergodic, \cite{Aaronson-Book}.

We could try to work with the ergodic components of $m$, but this does not yield a new mathematical object, because of the following general fact \cite{Aaronson-Book}:  Every ergodic component of an {\em invertible} totally dissipative infinite measure preserving  map is concentrated on a single orbit $\{T^n(\omega)\}_{n\in\Z}$. MCRE with such noise processes  have just one possible realization of noise up to time shift. Their theory is the same as the theory of general inhomogeneous Markov chains, and does not merit separate treatment.

\medskip
Suppose $\mathsf X^\Omega$ is a MCRE with noise space $(\Omega,\mathfs F,m,T)$. A {\bf Random additive functional} is a measurable function $f:\Omega\times\fS\times\fS\to\R$. This induces  the additive functional  $\mathsf f^\omega$ on $\mathsf X^\omega$
$$f_n^{\omega}(x,y)=f(T^n\omega,x,y).$$
For each $\omega\in\Omega$ we define
\begin{align*}
S_N^\omega&:=\sum_{n=1}^N f^\omega_n(X^\omega_n,X^{\omega}_{n+1})\equiv \sum_{n=1}^N f(T^n\omega,X^\omega_n,X^{\omega}_{n+1}),\\
V_N^\omega&:=\Var(S_N^\omega)\ \ \text{w.r.t. the distribution of $\mathsf X^\omega$}.
\end{align*}

\medskip
 Throughout this chapter, we make the following {\bf standing assumptions}:
\begin{enumerate}[(A)]
\item[(B)] {\bf Uniform boundedness:} $|f|\leq K$ where $K<\infty$ is a constant;
\item[(E)] {\bf Uniform ellipticity:} There is a constant $0<\epsilon_0<1$ and a Borel function $p:\Omega\times\fS\times\fS\to [0,\infty)$ such that
\begin{enumerate}[(a)]
\item $\pi(\omega,x,dy)=p(\omega,x,y)\mu_\omega(dy)$;
\item $0\leq p\leq 1/\epsilon_0$;
\item $\int_{\fS} p(\omega,x,y)p(T\omega,y,z)\mu_{T\omega}(dy)>\epsilon_0$ for all $\omega,x,  z$.
\end{enumerate}
\item[(S)] {\bf Stationarity:} For every $\vf:\fS\to\R$ bounded and Borel, for every $\omega\in\Omega$,
$$
\int \vf(y)\mu_{T\omega}(dy)=\int_{\fS}\left( \int_{\fS}\vf(y)\pi(\omega,x,dy)\right)\mu_\omega(dx).
$$
\end{enumerate}
(B) and (E) imply that $\mathsf f^\omega$ is a uniformly bounded additive functional and that $\mathsf X^\omega$ is uniformly elliptic for every $\omega$.  (S) is equivalent to saying that
 if $X_0$ is distributed according to $\mu_\omega$ then
 $X_n$ is distributed according to $\mu_{T^n \omega}$ for all $n>0$.
Subject to  (E),  (S) can always be assumed without loss of generality, because of
Proposition~\ref{Proposition-nu} and the discussion which follows it.

\medskip
Some of our results will require the following {\bf continuity hypothesis}:
\begin{enumerate}[(C)]
\item The Borel structure of $\Omega$ and $\fS$ is generated by a topologies so that $\Omega$ and $\fS$ are complete and separable metric spaces, and
\begin{enumerate}[(C1)]
\item   $T:\Omega\to\Omega$ is a homeomorphism and $\supp(m)=\Omega$.
\item $(\omega,x,y)\mapsto p(\omega,x,y)$ is continuous, and  $\omega\mapsto \int_\fS \vf d\mu_\omega$ is continuous for every bounded continuous $\vf:\fS\to\R$.
\item  $(\omega,x,y)\mapsto f(\omega,x,y)$ are continuous.
\end{enumerate}
\end{enumerate}
(C) is not part of  our standing assumptions, and we will state it explicitly whenever it is used.

\section{Main results}  Let $\Prob$ denote the  measure on $\Omega\times\fS\times\fS$ which represents the joint distribution of $(\omega,X_1^\omega,X_2^\omega)$:
\begin{equation}\label{P-measure}
\Prob(d\omega,dx,dy):=\int_{\fS}\int_{\fS}\int_{\Omega} m(d\omega)\mu_\omega(dx) \pi(\omega,x,dy).
\end{equation}
\begin{enumerate}[(1)]
\item $f(\omega,x,y)$ is called {\bf relatively cohomologous to a constant}\index{cohomologous}\index{relatively cohomologous}
  if there are bounded measurable  functions $a:\Omega\times\fS\to\R$ and $c:\Omega\to\R$ such that
$$
f(\omega,x,y)=a(\omega,x)-a(T\omega,y)+c(\omega)\text{  $\Prob$-a.e.}
$$
\item Fix $t\neq 0$, then $f(\omega,x,y)$ is {\bf relatively cohomologous to a coset of $t\Z$}  if there are measurable  functions $a:\Omega\times\fS\to S^1$ and $\lambda:\Omega\to S^1$ s.t.
$$
e^{(2\pi i/t) f(\omega,x,y)}=\lambda(\omega)\frac{a(\omega,x)}
{a(T\omega, y)} \text{ $\Prob$-a.e.}
$$
\end{enumerate}

\begin{theorem}\label{Thm-Variance-RMC}
Assume $\mathsf f$ is an additive functional on a MCRE with finite noise process.
Under the standing assumptions (B), (E), (S):
\begin{enumerate}[(1)]
\item If $f$ is relatively cohomologous to a constant,
  then  $|V^\omega_N|\leq C$  for all $N$, for a.e. $\omega$, where  $C=C(\epsilon_0,K)$ is a constant.
\item If $f$ is not relatively cohomologous to a constant, then there is a constant
$\sigma^2>0$  such that for a.e. $\omega$, $V^\omega_N\sim N\sigma^2$ as $N\to\infty$.
\end{enumerate}
\end{theorem}

\begin{theorem}\label{Thm-Non-Lattice-LLT-RMC}
Let $\mathsf f$ be an additive functional on a MCRE with finite noise process.
Assume the standing assumptions (B), (E),(S) and that
\begin{enumerate}[(a)]
\item Either $|\fS|\leq \aleph_0$,  or  $|\fS|>\aleph_0$ and the continuity hypothesis (C) holds.
\item $f$ is not relatively cohomologous to a coset of $t\Z$ for any $t\neq 0$.
\end{enumerate}
Then there exists  $\sigma^2>0$ such that for a.e. $\omega$,  for every open interval $(a,b)$, and for every $z_N,z\in\R$ such that $\frac{z_N-\E^\omega(S^\omega_N)}{\sqrt{N}}\to z$,
$$
\Prob\bigl[S^\omega_N-z_N\in (a,b)\bigr]\sim \frac{1}{\sqrt{N}}\left(\frac{e^{-z^2/2\sigma^2}}{\sqrt{2\pi \sigma^2}}\right)|a-b|\text{ as }N\to\infty.
$$
\end{theorem}

\begin{theorem}\label{Thm-Lattice-LLT-RMC}
Let $\mathsf f$ be an additive functional on a MCRE with finite noise process.
Assume the standing assumptions (B),(E),(S), and that  all the values of  $f$ are integers. If  $f$ is not relatively cohomologous to a coset of $t\Z$ with $t\neq 1$,
then there exists  $\sigma^2>0$ such that for a.e. $\omega$,  and for every $z_N,z\in\R$ such that $\frac{z_N-\E^\omega(S^\omega_N)}{\sqrt{N}}\to z$,
$$
\Prob\bigl[S^\omega_N=z_N\bigr]\sim \frac{1}{\sqrt{N}}\left(\frac{e^{-z^2/2\sigma^2}}{\sqrt{2\pi \sigma^2}}\right)\text{ as }N\to\infty.
$$
\end{theorem}

\begin{theorem}\label{Theorem-MCRE-free-energy}
Let $\mathsf f$ be an additive functional on a MCRE with finite noise process $(\Omega,\mathfs F,m,T)$.
Assume (B),(E),(S). If $f$ is not relatively cohomologous to a constant, then
\begin{enumerate}[(1)]
\item  There exists a continuously differentiable and strictly convex function $\mathfs F:\R\to\R$ such that for a.e. $\omega\in\Omega$,
$
\mathfs F(\xi)=\lim\limits_{N\to\infty}\frac{1}{N}\log \E(e^{\xi S_N^\omega})\text{ for all }\xi\in\R.
$
\item $
\frac{1}{N}\E(S^\omega_N)\xrightarrow[N\to\infty]{}\mathfs F'(0)
$ for a.e. $\omega$.
\item Let $\mathfs F'(\pm\infty):=\lim\limits_{\xi\to\pm\infty}\mathfs F'(\xi)$, and let $\mathfs I_N(\eta,\omega)$, $\mathfs I(\eta)$ denote the Legendre transforms of $\mathfs F_N(\xi):=\frac{1}{N}\log\E(e^{\xi S_N^\omega})$, $\mathfs F(\xi)$. Then for a.e. $\omega$, for every $\eta\in (\mathfs F'(-\infty),\mathfs F'(\infty))$,
    $
    \mathfs I_N(\eta,\omega)\xrightarrow[N\to\infty]{}\mathfs I(\eta).
    $
\item $\mathfs I(\eta)$ is strictly convex, has compact level sets, is equal to zero at $\eta={\mathfs F'}(0)$, and is strictly positive elsewhere.

\item  With probability one
$$ \fc_-=\cF(-\infty)=\lim_{N\to\infty} \frac{\ess\inf S_N^\omega}{N}, \quad
\fc_+=\cF(+\infty)=\lim_{N\to\infty} \frac{\ess\sup S_N^\omega}{N}. $$

\end{enumerate}
\end{theorem}
\begin{corollary}
Under the conditions of the previous theorem, for a.e. $\omega$, $S_N^\omega/N$ satisfies the large deviations principle with the rate function $\mathfs I(\eta)$:
\begin{enumerate}[(1)]
\item $\limsup\limits_{N\to\infty}\frac{1}{N}\log\Prob[S_N^\omega/N\in K]\leq -\inf_{z\in K}\mathfs I(z)$ for all closed sets $K\subset\R$.
\item $\limsup\limits_{N\to\infty}\frac{1}{N}\log\Prob[S_N^\omega/N\in G]\geq -\inf_{z\in K}\mathfs I(z)$ for all open sets $G\subset\R$.
\end{enumerate}
\end{corollary}
\begin{proof}
This is a consequence of the  G\"artner-Ellis Theorem.
\qed
\end{proof}

\medskip
So far we have only considered MCRE with finite noise spaces. We will now discuss the case of infinite noise spaces $(\Omega,\mathfs F,m,T)$. The main new phenomena in this case are:

\begin{example}\label{Example-Sublinear-Variance} For MCRE with an infinite noise process:
\begin{enumerate}[(a)]
\item It is possible that $V_N^\omega\to \infty$ $m$-a.e., but that $V_N^\omega=o(N)$ a.e.
\item It is possible that $\not\exists a_N$ s.t. $V_N^\omega\sim a_N$ for $m$-a.e. $\omega$.
\end{enumerate}
\end{example}
\begin{proof}
Let $X_n$ be iid bounded real random variables with variance one and distribution
$\mu$. Let $f_n(x)=x$. Let $(\Omega,\mathfs F,m,T)$ be an infinite noise process, and fix $E\in \mathfs F$ of finite positive measure.  Let
$$
\pi(\omega,x,dy):=\mu(dy)\ , \ f(\omega,x,y):=1_E(\omega)x
$$
Then $\DS S_N^\omega=\sum_{n=1}^N 1_E(T^n\omega) X_n$, and $V_N^\omega=\sum_{n=1}^N 1_E(T^n\omega)$.

We now appeal to the following general results from infinite ergodic theory. Let $(\Omega,\mathfs F,m,T)$ be an ergodic, invertible, measure preserving map on a non-atomic $\sigma$-finite measure space, and let
$L^1_+:=\{A\in L^1(\Omega,\mathfs F,m):A\geq 0, \int A dm>0\}$. If $m(\Omega)=\infty$, then
\begin{enumerate}[(1)]
\item  $\sum_{n=1}^N A\circ T^n=\infty$ almost everywhere for all $A\in L^1_+$;
\item $\frac{1}{N}\sum_{n=1}^N A\circ T^n\xrightarrow[N\to\infty]{}0$ almost everywhere for all $A\in L^1$;
\item Let $a_N$ be a sequence of positive real numbers, then at least one of the following possibilities happens:
\begin{enumerate}[(a)]
\item $\liminf_{N\to\infty}
\frac{1}{a_N}\sum_{n=1}^N A\circ T^n=0$ a.e. for all $A\in L^1_+$;
\item  $\limsup_{N\to\infty}\frac{1}{a_{N}}\sum_{n=1}^{N} A\circ T^n=\infty$ a.e. for all $A\in L^1_+$.
\end{enumerate}
So $\not\exists a_N\uparrow\infty$ s.t. $\sum_{n=1}^N A(T^n\omega)\sim a_N$ for a.e. $\omega$, even for a single $A\in L^1_+$.
\end{enumerate}
These results can all be found in \cite{Aaronson-Book}:
(1) is a consequence of the Halmos Recurrence Theorem; (2) follows from the Ratio Ergodic Theorem; and (3) is a theorem of J.~Aaronson.
Specializing to the case $A=1_E$ we find that $V_N^\omega\to\infty$ a.e.; $V_N^\omega=o(N)$ a.e. as $N\to\infty$; and $\not\exists a_N$ so that $V_N^\omega\sim a_N$ for a.e. $\omega\in \Omega$. \end{proof}

Here are our general results on MCRE with infinite noise spaces.

\begin{theorem}\label{Thm-Variance-RMC-infinite}
Suppose $\mathsf f^\omega$ is a random additive functional on a MCRE with infinite noise space on a non-atomic $\sigma$-finite measure space.
Under the standing assumptions (B), (E), (S):
\begin{enumerate}[(1)]
\item If $f$ is relatively cohomologous to a constant,
  then  $|V^\omega_N|\leq C$  for all $N$, for a.e. $\omega$, where  $C=C(\epsilon_0,K)$ is a constant.
\item If $f$ is not relatively cohomologous to a constant then $V^\omega_N\to \infty $ a.s.
\end{enumerate}
\end{theorem}

\begin{theorem}\label{Thm-Non-Lattice-LLT-RMC-infinite}
Suppose $\mathsf f^\omega$ is a random additive functional on a MCRE with infinite noise space on a non-atomic $\sigma$-finite measure space.
Assume the standing assumptions (B), (E),(S) and that
\begin{enumerate}[(a)]
\item Either $|\fS|\leq \aleph_0$,  or  $|\fS|>\aleph_0$ and the continuity hypothesis (C) holds.
\item $f$ is not relatively cohomologous to a coset of $t\Z$ for any $t\neq 0$.
\end{enumerate}
Then for a.e. $\omega$,  for every open interval $(a,b)$, and for every $z_N,z\in\R$ such that $\frac{z_N-\E^\omega(S^\omega_N)}{V_N^\omega}\to z$,
$\displaystyle{
\Prob\bigl[S^\omega_N-z_N\in (a,b)\bigr]\sim \frac{e^{-z^2/2}}{\sqrt{2\pi V_N^\omega}}|a-b|\text{ as }N\to\infty.}
$
\end{theorem}

\begin{theorem}\label{Thm-Lattice-LLT-RMC-infinite}
Suppose $\mathsf f^\omega$ is a random additive functional on a MCRE
with infinite noise space on a non-atomic $\sigma$-finite measure space.
Assume the standing assumptions (B),(E),(S), and that  all the values of  $f$ are integers. If  $f$ is not relatively cohomologous to a coset of $t\Z$ with $t\neq 1$,
then  for every $z_N,z\in\R$ such that $\frac{z_N-\E^\omega(S^\omega_N)}{\sqrt{V_N^\omega}}\to z$,  \;
for a.e. $\omega$,
$
\Prob\bigl[S^\omega_N=z_N\bigr]\sim \frac{e^{-z^2/2}}{\sqrt{2\pi V_N^\omega}}$  as $N\to\infty
$.
\end{theorem}

\section{Proofs}
Throughout this section  $\mathsf X^{\omega}$ is a Markov chain in random environment with stationary ergodic, possibly infinite, noise process $(\Omega,\mathfs F,m,T)$, and $\mathsf f^{\omega}$ is a random additive functional on $\mathsf X^{\omega}$. We assume throughout (B),(E),(S).

\subsection{The essential range is a.s. constant}
The purpose of this section is to prove the following result:

\begin{proposition}\label{Prop-Co-Range-Random-MC}
There exist closed subgroups $H,G_{ess}\leq \R$ s.t. for $m$--a.e. $\omega$,
 the co-range of $(\mathsf X^\omega,\mathsf f^\omega)$ equals $H$ ,  the essential range of $(\mathsf X^\omega,\mathsf f^\omega)$ equals $G_{ess}$, and
 $$
 G_{ess}=\begin{cases}
 \R & H=\{0\},\\
 \frac{2\pi}{t}\Z & H=t\Z, \;\;t\neq 0,\\
 \{0\} & H=\R.
 \end{cases}
 $$
\end{proposition}
We call  $H$ and $G_{ess}$ the {\bf a.s. co-range}\index{co-range!for Markov chains in random environment} and {\bf a.s. essential range}\index{essential range!for Markov chains in random environment}.

We begin with a calculation of the structure constants of $(\mathsf X^\omega, \mathsf f^\omega)$.
Fix an element $\omega$ in the noise space, and let  $\mathrm{Hex}(\omega)$ denote the probability space of position 3 hexagons for $\mathsf X^\omega$. Let $m_\omega$ denote the hexagon measure,  as defined in
\S \ref{Section-Hexagons}. Recall the definition of the balance $\Gamma(P)$ of a hexagon $P$, and define
$$
\begin{array}{l}
u(\omega):=\E(|\Gamma(P)|^2)^{1/2}\\
d(\xi,\omega):=\E(|e^{i\xi\Gamma(P)}-1|^2)^{1/2}
\end{array} \text{ (expectation on $P\in \mathrm{Hex}(\omega)$ w.r.t. $m_\omega$)}.
$$
Since the space of position $n+3$ hexagons for $\mathrm X^\omega$ is $\mathrm{Hex}(T^n\omega)$, together with the hexagon measure  $m_{T^{n}\omega}$,
it follows that the structure constants of $(\mathsf X^\omega, \mathsf f^\omega)$ are
\begin{equation}\label{eq-structure-cont-RMC}
d_{n+3}(\xi,f^\omega)=d(T^n\omega,\xi)\quad \text{and}\quad u_{n+3}(f^\omega)=u(T^n\omega) \ \ (n\geq 0).
\end{equation}

\begin{lemma}\label{Lemma-u-d-Borel}
 $u(\cdot), d(\cdot, \cdot)$ are Borel measurable, and for every $\omega$,  $d(\cdot,\omega)$ is continuous. Under the continuity hypothesis (C), $u(\cdot), d(\cdot,\cdot)$ are continuous.
\end{lemma}
\begin{proof}
To check this, express the hexagon measure  explicitly as a measure on $\fS^6$ in terms of the  transition kernel $\pi(\omega,x,y)$, using the formulas for the bridge distributions of \S \ref{Section-Bridge},  and write $\Gamma(P)$ explicitly a function on $\fS^6$ in terms of $f(\omega,x,y)$. We omit the details, which are routine.\qed
\end{proof}

\medskip
\noindent
{\bf Proof of Proposition \ref{Prop-Co-Range-Random-MC}.}
Let  $H_\omega:=H(\mathsf X^\omega,\mathsf f^\omega)$
{be} the essential range of $(\mathsf X^\omega,\mathsf f^\omega)$.
By Theorem \ref{Theorem-co-range},   $H_\omega$ is either $\R$ or $t\Z$ for some $t\geq 0$.
By \eqref{eq-structure-cont-RMC}
$$
D_N(\xi,\omega):=\sum\limits_{n=3}^N d_n(\xi,\mathsf f^\omega)^2\equiv \sum\limits_{n=0}^{N-3} d(T^n\omega,\xi)^2.
$$

\medskip
\noindent
{\sc Step 1:} {\em
$
U(a,b):=\{\omega\in\Omega: D_N(\cdot,\omega)\xrightarrow[N\to\infty]{}\infty\text{ uniformly on }(a,b)\}
$
is measurable and $T$-invariant for all $a<b$.
}

\medskip
\noindent
{\em Proof.\/} $T$-invariance is because $d^2\leq 4$ whence  $|D_N(\xi,T\omega)-D_N(\xi,\omega)|\leq 8$. Measurability is because   of the  identity
$$
U(a,b)=\left\{\omega\in \Omega: \begin{array}{l}
\forall M\in\Q\; \exists N\in\N\text{ s.t.  }\\
\text{ for all }\xi\in (a,b)\cap\Q,\
D_N(\omega,\xi)>M
\end{array}
\right\}.
$$
The inclusion $\subseteq$ is obvious. The inclusion $\supseteq$ is because if $\omega\not\in U(a,b)$ then for some $M\in\Q$, for all $N\in\N$ there exists some $\eta\in (a,b)$ s.t. $D_N(\omega,\eta)<M$, whence by the continuity of $\eta\mapsto D_N(\omega,\eta)$ there is some $\xi\in (a,b)\cap\Q$ such that $D_N(\omega,\xi)<M$. So $\omega\not\in U(a,b)\Rightarrow \omega\not\in\text{RHS}$.

\medskip
\noindent
{\sc Step 2:} {\em The sets  $\Omega_1:=\{\omega\in\Omega: H_\omega=\{0\}\}$,\; $\Omega_2:=\{\omega\in \Omega: H_\omega=\R\}$, and $\Omega_3:=\{\omega\in \Omega: H_\omega=t\Z\text{ for some }t\neq 0\}$ are measurable and $T$-invariant.  Therefore by ergodicity, for each $i$, either $m(\Omega_i)=0$ or $m(\Omega_i^c)=0$.}

\medskip
\noindent
{\em Proof.} Recall that for Markov chains, $D_N\to\infty$ uniformly on compact subsets of the complement of the co-range (Theorem \ref{Theorem-MC-array-difference}). So
$$
\Omega_1=\bigcap_{n=1}^\infty U(\tfrac{1}{n},n)\ ,\
\Omega_2=\bigcap_{0<a<b\text{ rational}} U(a,b)^c\ ,\ \Omega_3=\Omega_1^c\cap\Omega_2^c.
$$
By step 1, $\Omega_i$ are $T$-invariant and measurable. Since $T$ is ergodic, these sets are either of measure zero or of full measure.

\medskip
By Theorem \ref{Theorem-essential-range}, if $\Omega_1$ has full measure, then  the essential range is a.s. $\R$. Similarly, if  $\Omega_2$ has full measure, then  the essential range is $\{0\}$ almost surely. It remains to consider the case when $\Omega_3$ has full measure.

\medskip
\noindent
{\sc Step 3:} {\em If $\Omega_3$ has full measure, then there exist $t\neq 0$ such that $\Omega_3(t):=\{\omega\in\Omega: H_\omega=t\Z\}$ has full measure, and then the essential range is $(2\pi/t)\Z$ almost surely.}

\medskip
\noindent
{\em Proof.} For every $\omega\in\Omega_3$ there exists $t(\omega)>0$ such that $H_\omega=t(\omega)\Z$. We can characterize $t(\omega)$  as follows:
$$
t(\omega)=\sup\left\{t\in\Q\cap(0,\infty): \begin{array}{l}
D_N(\omega,\cdot)\to\infty\text{ uniformly }\\
\text{on compact subsets of }(0,t)
\end{array}\right\}.
$$
It is clear from this expression that $t(T\omega)=t(\omega)$, and that for every $A>0$,
$$
[t(\omega)\geq A]=\bigcap_{0<a<b<A\text{ rational}}U(a,b).
$$
So $t(\cdot)$ is a measurable $T$-invariant function, whence by ergodicity constant. Let $t$ denote this constant, then $H_\omega=t\Z$ for a.e. $\omega$. By Theorem \ref{Theorem-essential-range}, $G_{ess}(\mathsf X^\omega,\mathsf f^\omega)=(2\pi/t)\Z$ almost surely.\qed

\subsection{Variance growth}
In this section we prove Theorems \ref{Thm-Variance-RMC} and \ref{Thm-Variance-RMC-infinite} on the behavior of $V_N^\omega$ as $N\to\infty$.

\begin{lemma}
\label{LmInfSum}
Suppose $(\Omega,\mathfs F,m,T)$ is an invertible, ergodic, measure preserving map of a probability space or of a non-atomic infinite measure space. Let $A:\Omega\to \R$ be a non-negative measurable function. Either $A=0$ a.e.,
or $\DS \sum_{n\geq 0} A\circ T^n=\infty$ a.e.
\end{lemma}

\begin{proof}
 If $m(\Omega)<\infty$, then the Lemma follows from the Birkhoff ergodic theorem. In the more general case $m(\Omega)\leq \infty$, the lemma follows from the well-known fact that invertible ergodic measure preserving maps on non-atomic measure spaces are conservative. We supply the details,  for completeness.

If  $A$ is not equal to 0 a.e., then there is $\eps>0$ s.t.
$E:=\{\omega\in\Omega: A(\omega)\geq \eps\} $ has positive measure.
We claim that
\begin{equation}
\label{InfRet}
 \sum_{n\geq 0}1_E(T^n\omega)=\infty
\end{equation}
almost everywhere on $E$.
Since $A\geq \eps 1_E$ \eqref{InfRet} implies that $\DS \sum_{n\geq 0} A(T^n \omega)=\infty$
almost everywhere on $E,$ and, by ergodicity, almost everywhere on $\Omega,$
proving the lemma.

It remains to prove \eqref{InfRet}.
Suppose by way of contradiction that it is not true that $\DS \sum_{n\geq 0}1_E(T^n\omega)=\infty$
almost everywhere on $E$. Then  there exists $N$ s.t.
$$
W:=\{\omega\in E: \sum_{n=0}^\infty 1_E(T^n\omega)=N\}
$$
has positive measure. The invertibility and measurability of $T$ implies that $T^n(W)$ is measurable for all $n\in\Z$, and that $\{T^n(W)\}_{n\in\Z}$ are pairwise disjoint.

Since $(\Omega,\mathfs F,m)$ is non-atomic, we can break $
W=W_1\cup W_2
$
where $W_i$ are measurable, disjoint, and with positive measure. By invertibility, $\hat{W}_i:=\bigcup_{n\in\Z} T^n W_i$ are disjoint $T$-invariant sets with positive measure. But this contradicts ergodicity.  \qed
\end{proof}

\medskip
\noindent
{\bf Part 1: $V_N^\omega$ is bounded, or tends to infinite almost surely.}
Recall that $K$ is a bound for $\ess\sup|f|$, and $\epsilon_0$ is a uniform ellipticity constant for $\mathsf X^\omega$. By Theorem \ref{Theorem-MC-Variance}  and \eqref{eq-structure-cont-RMC} there are positive constants $C_i=C_i(\epsilon_0,K)$ $(i=1,2)$ such that for all $N$,
$$
C_1^{-1}\sum_{n=3}^N u(T^n\omega)^2-C_2\leq V^\omega_N\leq
C_1\sum_{n=3}^N u(T^n\omega)^2+C_2.
$$
If $u(\omega)=0$ $m$-a.e., then for a.e. $\omega$,  $V_N^\omega\leq C_2$ for all $N$. Otherwise, by Lemma \ref{LmInfSum},\\ $\sum_{n=3}^N u(T^n \omega)^2\xrightarrow[N\to\infty]{}\infty$, whence $V^\omega_N\to\infty$ almost everywhere.

\medskip
\noindent
{\bf Part 2: Linear growth of variance when  $V_N^\omega\to\infty$ a.e. and $m(\Omega)=1$}.
Suppose $m(\Omega)=1$ and $V_N^\omega\to\infty$ almost surely. We claim that
\begin{equation}\label{sigma-square-for-RMC}
\exists \sigma^2>0\text{ s.t. }V^\omega_N\sim N\sigma^2\text{  a.s.}
\end{equation}

Let $\sigma_0^2:=\int_\Omega u^2 dm$. This is a finite number, because  $\|u\|_\infty\leq 6K$ and $m(\Omega)=1$. This is a  positive number,  because as we saw in part 1, if $u=0$ a.e., then $V_N^\omega=O(1)$ a.e. contrary to our assumptions. By the pointwise ergodic theorem,  $\DS \sum_{n=3}^N u(T^n\omega)^2=[1+o(1)]\sigma_0^2 N$. Hence
$V^\omega_N\geq [1+o(1)] C_1(\epsilon_0,K)^{-1} N\sigma^2_0\to\infty.$

Let
$F_n:=f(T^{n-1}\omega,X^\omega_n,X^\omega_{n+1})$ and let $\E^\omega$, $\Var^\omega$, $\Cov^\omega$ denote the expectation, variance and covariance with respect to $\mathsf{X}^\omega$, then
\begin{align*}
&V^\omega_N=\sum_{n=1}^N \Var^\omega(F_n)+2\sum_{n=1}^N \sum_{m=n+1}^N \Cov^\omega(F_n, F_m)\\
&=\sum_{n=1}^N \Var^\omega(F_n)+2\sum_{n=1}^N \sum_{k=1}^{N-n} \Cov^\omega(F_n, F_{n+k}).
\end{align*}

By  assumption (S) $\{\mu_\omega\}$ are stationary, so
$\{X^{T^n\omega}_i\}_{i\geq 1}$ has the same distribution as $\{X^{\omega}_1\}_{i\geq n}$. Therefore  $\Var^\omega(F_n)=\Var^{T^{n-1}\omega}(F_1)$ and $\Cov^\omega(F_n,F_{n+k})=\Cov^{T^{n-1}\omega}(F_1,F_{1+k})$. Thus
$$
V_N^\omega= \sum_{n=0}^{N-1} \psi_0(T^n\omega)+2\sum_{n=0}^{N-1} \sum_{k=1}^{N-n}\psi_k(T^n\omega),
$$
where $\psi_0(\omega):=\Var^\omega[f(\omega,X^\omega_1,X^\omega_2)]$ and
$$\psi_k(\omega)=\Cov^\omega[f(\omega,X^\omega_1,X^\omega_2),f(T^k\omega,X^\omega_{k+1},X^\omega_{k+2})].
$$

By the ergodic theorem
$
\lim\limits_{N\to\infty}\frac{1}{N}\sum\limits_{n=1}^N\psi_0(T^n\omega)=\int \psi_0 dm.
$ To find the limit of the normalized double sum we first recall that by the uniform mixing of $\{X^\omega_n\}$ (a consequence of the ellipticity assumption), $\|\psi_k\|_\infty\leq C_{mix}\|f\|_\infty^2
\theta^k$ with $C_{mix}, 0<\theta<1$ which only depend on $\epsilon_0$ (Proposition \ref{Proposition-Exponential-Mixing}).  Therefore for every $M$,
$$
\lim_{N\to\infty}\frac{1}{N}\sum_{n=0}^{N-1} \sum_{k=1}^{N-n}\psi_k(T^n\omega)
={\Big[}\lim_{N\to\infty}\frac{1}{N}\sum_{n=0}^{N-1}  \sum_{k=1}^{M-1} \psi_k(T^n\omega)
{\Big]}
+O(\theta^M),
$$
whence by the ergodic theorem
$
\DS \lim\limits_{N\to\infty}\frac{1}{N}\sum\limits_{n=0}^{N-1} \sum\limits_{k=1}^{N-n}\psi_k(T^n\omega)
=\sum\limits_{k=1}^\infty \int\psi_k dm
$, with the last sum converging exponentially fast.  In summary,
$$
\frac{1}{N}V^\omega_N\xrightarrow[N\to\infty]{}
\sigma^2:=\int\left[\psi_0+2\sum_{k=1}^\infty\psi_k\right]dm.
$$
Since as we saw above $\liminf\frac{1}{N}V^\omega_N\geq C_1\sigma_0^2$ and $\sigma_0^2>0$, it must be the case that $\sigma^2>0$, and \eqref{sigma-square-for-RMC} is proved.

\medskip
\noindent
We now relate the following two properties:
\begin{enumerate}[(a)]
\item $f$ is relatively cohomologous to a constant;
\item
 $V_N^\omega$ is bounded $m$-a.e.
\end{enumerate}

\medskip
\noindent
{\bf Part 3: (a)$\Rightarrow$(b):}
Suppose $f$ is relatively cohomologous to a constant. By Fubini's theorem, for $m$-a.e. $\omega$, for every $n$,
$$
f^\omega_n(X^\omega_n,X^\omega_{n+1})=a(T^n\omega,X^\omega_n)-a(T^{n+1}\omega,X^\omega_{n+1})+c(T^n\omega) \text{ a.s.}
$$
with  respect to the distribution of $\{X^\omega_n\}$.\footnote{Here we use the assumption that $(\Omega,\mathfs F,m)$ is $\sigma$-finite.  Fubini's theorem may be false otherwise.}

Summing over $n$, we obtain that for a.e. $\omega$, for every $N$,
$$
|S_N^\omega-\sum_{n=1}^N c(T^n\omega)|=|a(\omega,X^\omega_1)-a(T^N\omega, X^\omega_{N+1})|\leq 2\mathrm{ess\, sup\, } a(\cdot,\cdot).
$$
In particular, for every $\omega$,  $V^N_\omega$ is bounded. By the first part of the proof, for a.e. $\omega$, for all $N$, $|V^\omega_N|\leq C_2(\epsilon_0, K).$

\medskip
\noindent
{\bf Part 4: (b)$\Rightarrow$(a):}
Next suppose that $f$ is not relatively cohomologous to a constant.
 Recall that $\sigma_0^2=\int u^2 dm$.

We claim that  $\sigma^2_0>0$, and deduce from the first part of the proof that
 $V_N^\omega\to+\infty$ a.e.

Assume by way of contradiction that $\sigma^2_0=0$, then  $u(\omega)=0$ a.e., whence for a.e. $\omega$, for every $n$,   almost every position $n$ hexagon of $\mathsf X^\omega$  has balance zero.
 Applying the gradient lemma to $\mathsf X^\omega$,  we  find bounded functions $g^\omega_n$ and constants $c^\omega_n$ such that
$$
f^\omega_n(X^\omega_n,X^\omega_{n+1})=g^\omega_n(X^\omega_n)-g^\omega_{n+1}(X^\omega_{n+1})+c^\omega_n \text{ a.s. }
$$

The issue is to show that $g^\omega_n, c^\omega_n$ can be given the form $g_n^\omega(x)=a(T^n\omega,x)$ and $c^\omega_n=c(T^n\omega)$ where $a(\cdot,\cdot), c(\cdot)$ are measurable.

This is indeed the case, because  the proof of the gradient lemma shows that  we can take
\begin{align*}
c^\omega_n&=\E^\omega[f^\omega_{n-2}(X^\omega_{n-2},X^\omega_{n-1})] \\
g_n^\omega(z)&=\E\biggl(f^\omega_{n-2}(X^\omega_{n-2},X^\omega_{n-1})+f^\omega_{n-1}(X^\omega_{n-1},X^\omega_{n})\bigg|X^\omega_n=z\biggr).
\end{align*}
So $c^\omega_n=c(T^n\omega)$ and $g^\omega_n(z)=a(T^n\omega,z)$ for
\begin{align*}
&c(\omega):=\int_{\fS^2}\mu_{T^{-2}\omega}(x)\mu_{T^{-1}\omega}(dy) p(T^{n-2}\omega,x,y)
f(T^{-2}\omega,x,y).\\
&a(\omega,z):=\int_{\fS^3}\mu_{T^{-2}\omega}(dx) \mu_{T^{-1}\omega}(dy)  p(T^{n-2}\omega,x,y)p(T^{-1}\omega,y,z)\times\\
&\hspace{3cm}{{f(T^{-2}\omega,x,y)+f(T^{-1}\omega,y,z)}\over{\int_{\fS^2}\mu_{T^{-2}\omega}(dx) \mu_{T^{-1}\omega}(dy)  [p(T^{n-2}\omega,x,y)p(T^{-1}\omega,y,z)}]}.
\end{align*}
These are measurable functions, and our standing assumptions imply that they are bounded.

We see that $f$ is relatively cohomologous to a constant in contradiction to our assumption.
So $\sigma^2_0>0$, whence by the first part of the proof  $V_N^\omega$ tends to infinity.\hfill$\Box$

\subsection{The local limit theorem}
In this section, we prove Theorems \ref{Thm-Non-Lattice-LLT-RMC},
\ref{Thm-Lattice-LLT-RMC}, \ref{Thm-Non-Lattice-LLT-RMC-infinite} and \ref{Thm-Lattice-LLT-RMC-infinite} on the local limit theorem for Markov chains in random environment.
We need the following lemmas:
\begin{lemma}\label{Lemma-Measurable-Selection}
Suppose $\Omega$ is a Borel space,  $\fS$ is a  separable metric space, and $\psi:\Omega\times\fS\to\R$ is a Borel function such that for every $\omega\in\Omega$, $\psi(\omega,\cdot)$ is continuous on $\fS$ and positive somewhere. Then  there exists a Borel measurable $x:\Omega\to\fS$ such that $\psi(\omega,x(\omega))>0$.
\end{lemma}
\begin{proof}
Fix a countable dense set $\{x_i\}\subset\fS$.  Our assumptions on $\psi$ imply that for every $\omega$ there exists an $i$ such that $\psi(\omega,x_i)>0$. So
$$
i(\omega):=\min\{i\in\N: \psi(\omega,x_i)>0\}
$$
is well-defined and Borel measurable. Take $x(\omega):=x_{i(\omega)}$.\qed
\end{proof}

\begin{lemma}\label{Lemma-e-to-the-i-W}
If $W_1, W_2$ are two independent random variables such that for some $a,t\in\R$, $W_1+W_2\in a+t\Z$ with full probability, then $a=a_1+a_2$ where  $W_1\in a_1+t\Z$, $W_2\in a_2+t\Z$ with full probability.
\end{lemma}
\begin{proof}
Without loss of generality $a=0$, $t=2\pi$. Then
$$|\E(e^{i W_1})|\cdot |\E(e^{iW_2})|=|\E(e^{i(W_1+W_2)})|=1,$$ whence $|\E(e^{iW_k})|=1 \ \ (k=1,2).$
Choose $a_k$ such that $\E(e^{i(W_k-a_k)})=1$, then $\E(\cos(W_k-a_k))=1$, whence $W_k-a_k\in 2\pi\Z$ almost surely.
Necessarily $a_1+a_2\in 2\pi\Z$, and there is no problem in adjusting $a_1$ to get that the sum zero.\qed
\end{proof}

\medskip
\noindent
{\bf Proof of  Theorems \ref{Thm-Non-Lattice-LLT-RMC} and \ref{Thm-Non-Lattice-LLT-RMC-infinite} on the non-lattice case.}
Theorems \ref{Thm-Non-Lattice-LLT-RMC} and \ref{Thm-Non-Lattice-LLT-RMC-infinite} provide the  LLT for Markov chains in random environment with finite and infinite noise process, under the assumption that
$f$ is not relatively  cohomologous to a coset of $t\Z$ with $t\neq 0$.

In this case,  $f$ is also not relatively cohomologous to a constant, and by Theorems \ref{Thm-Variance-RMC} and \ref{Thm-Variance-RMC-infinite},
$V^\omega_N\to\infty$ as $N\to\infty$. Moreover, if the noise process $(\Omega,\mathfs F,m,T)$ satisfies $m(\Omega)=1$, then $\exists \sigma^2>0$ s.t. $V_N^\omega\sim N\sigma^2$.

To prove the theorems it is sufficient to show that for a.e. $\omega$, $G_{ess}(\mathsf X^\omega, \mathsf f^\omega)=\R$, as this will imply the LLT by the general results of Chapter \ref{Section-LLT-irreducible}.

Assume by way of contradiction that $G_{ess}(\mathsf X^\omega, \mathsf f^\omega)\neq\R$
 on a
set of $\omega$'s of positive measure.   By Proposition \ref{Prop-Co-Range-Random-MC}, $G_{ess}(\mathsf X^\omega, \mathsf f^\omega)=G_{ess}$ a.e. where $G_{ess}=\{0\}$ or $\frac{2\pi}{t}\Z$ with $t\neq 0$. The first possibility cannot happen, because it implies that $\mathsf f^\omega$ is center-tight, whence $V^\omega_N=O(1)$, whereas $V^\omega_N\to\infty$.  So there exists $t\neq 0$ such that  $G_{ess}(\mathsf X^\omega,\mathsf f^\omega)=(2\pi/t)\Z$ a.s., and
$H_\omega:=H(\mathsf X^\omega,\mathsf f^\omega)=t\Z$ a.e.

By the reduction lemma, for every $\omega$ s.t. $H_\omega=t\Z$ there are measurable functions $g^\omega_n(x)$, $h_n^\omega(x,y)$ with $\sum \Var[h_n^\omega]<\infty$,
and constants $c^\omega_n$ such that
$$
\exp\bigl[it(f^\omega_n(x,y)-g^\omega_n(x)+g^\omega_{n+1}(y)+ h_n^\omega(x,y)
-c^\omega_n)\bigr]=1
$$
a.s. with respect to the distribution of $(X^\omega_n,X^\omega_{n+1})$.
So $e^{it(f(\omega,x,y)+h_n^\omega(x,y))}=\lambda^\omega_n \frac{a^\omega_n(x)}{a^\omega_{n+1}(y)}$, where
$\lambda^\omega_n=e^{itc^\omega_n}$, $a^\omega_n(x)= e^{it g^\omega_n(x)}$.

\medskip
But now we run into a problem: Our proof of the reduction lemma does not provide $g_n^\omega$ and $c^\omega_n$ of the form
$
c^\omega_n=c(T^n\omega)\ ,\ a^\omega_n=a(T^n\omega,x)
$
with $c(\cdot), a(\cdot,\cdot)$  measurable, and we need to
show that $h_n^\omega=0$.

\smallskip
To this end we use the following additional structure: For a.e. $\omega$,  $H_\omega=t\Z$  so
$\sum d(T^n\omega,t)^2<\infty$ $\mu$-almost everywhere. By the ergodic theorem, this  can only happen if
$
d(\omega,t)=0$ almost everywhere. Hence
\begin{equation}\label{eq-integer-balance}
\Gamma\left(
Z^\omega_{1},\begin{array}{l}
Z^\omega_{2}\\
Y^{\omega}_{2}
\end{array},\begin{array}{l}
Y^\omega_{3}\\
X^{\omega}_{3}
\end{array}, X^\omega_{4}
\right)\in \frac{2\pi}{t}\Z\text{ a.e. in $\mathrm{Hex}(\omega)$ for $m$-a.e. $\omega$.}
\end{equation}

Recall the ladder process $\un{L}_n^\omega=(Z^\omega_{n-2},Y_{n-2}^\omega, X^\omega_n)$ associated with  $\{X^\omega_n\}$. Let $\Prob^\omega$ denote its distribution, and define as in the proof of the reduction lemma,
\begin{align*}
&H^\omega(\un{L}_n^\omega,\un{L}_{n+1}^\omega):=\Gamma\left(
Z^\omega_{n-2},\begin{array}{l}
Z^\omega_{n-1}\\
Y^{\omega}_{n-1}
\end{array},\begin{array}{l}
Y^\omega_{n}\\
X^{\omega}_{n}
\end{array}, X^\omega_{n+1}
\right)\\
&\Gamma\left(
Z^\omega_{1},
\begin{array}{l}
Z^\omega_{2}\\
Y^{\omega}_{2}
\end{array},
\begin{array}{l}
Z^\omega_{3}\\
X^{\omega}_{3}
\end{array},
\begin{array}{l}
Y^\omega_{4}\\
X^{\omega}_{4}
\end{array}, X^\omega_{5}
\right)\overset{!}{:=}H^\omega(\un{L}_3^\omega,\un{L}_{4}^\omega)+
H^\omega(\un{L}_{4}^\omega,\un{L}_{5}^\omega)
\end{align*}
The last definition requires justification because the RHS seems at first sight to  depend on $Y^\omega_3$. In fact it does not. To see this  observe that the last expression is the balance of the octagon obtained by stacking
$\left(
Z^\omega_{2},\begin{array}{l}
Z^\omega_{3}\\
Y^{\omega}_{3}
\end{array},\begin{array}{l}
Y^\omega_{4}\\
X^{\omega}_{4}
\end{array}, X^\omega_{5}
\right)$ on top of $\left(
Z^\omega_{1},\begin{array}{l}
Z^\omega_{2}\\
Y^{\omega}_{2}
\end{array},\begin{array}{l}
Y^\omega_{3}\\
X^{\omega}_{3}
\end{array}, X^\omega_{4}
\right)$ and removing the common edge $(Z^\omega_2, Y^\omega_3, X^\omega_4)$ which ``cancels out."

\medskip
\noindent
{\sc Claim 1.\/} {\em Let $\Prob^\omega$ denote the distribution of $\{\un{L}^\omega_n\}$, then   there exist  measurable functions $\zeta_1(\omega),\zeta_2(\omega)\in\fS$ such that for a.e. $\omega$
$$
\Gamma\left(
\zeta_{1}(\omega),
\begin{array}{c}
\zeta_{2}(\omega)\\
Y^{\omega}_{2}
\end{array},
\begin{array}{c}
\zeta_{1}(\omega)\\
X^{\omega}_{3}
\end{array},
\begin{array}{c}
Y^\omega_{4}\\
X^{\omega}_{4}
\end{array}, X^\omega_{5}
\right)\in \frac{2\pi}{t}\Z\hspace{0.5cm} ; \hspace{0.5cm}
\Prob^\omega\left(\ \ \cdot
\bigg|\ \ {
\begin{array}{l}
Z^\omega_3=\zeta_1(\omega)\\
Z^\omega_2=\zeta_2(\omega)\\
Z^\omega_1=\zeta_1(\omega)\\
\end{array}}\right)\text{--a.e.}
$$
}

\noindent
{\em Proof.\/} By \eqref{eq-integer-balance}, $\Gamma\in \frac{2\pi}{t}\Z$ with full  $\Prob^\omega$--probability, for a.e. $\omega$. The point it to obtain this a.s. with respect to the conditional measures.

Suppose first that $\fS$ is countable, then for  fixed $\omega$,  the $\Prob^\omega$-distribution of $(\un{L}^\omega_3,\un{L}^
\omega_2,\un{L}^\omega_3)$ is purely atomic, and
 $\Gamma\in \frac{2\pi}{t}\Z$ for every octagon with positive $\Prob^\omega$--probability.  So the claim holds  for any pair $(\zeta_1(\omega),\zeta_2(\omega))\in\fS$ such that
$$
\Prob^\omega\bigl[(Z^{\omega}_1, Z^{\omega}_2, Z^{\omega}_3)=(\zeta_1(\omega),\zeta_2(\omega),\zeta_1(\omega))\bigr]>0.
$$
Such pairs exist by the ellipticity assumption. Since $\fS$ is countable there is no problem to choose such $(\zeta_1,\zeta_2)$  measurably.

Now suppose $\fS$ is uncountable but with the continuity property (C). By Fubini's theorem and \eqref{eq-integer-balance}, for a.e. $\omega\in\Omega$, for a.e. $(\zeta_1,\zeta_2,\zeta_3)$ with respect to the distribution $(\zeta_1,\zeta_2,\zeta_3)\sim (Z_1^\omega,Z_2^\omega,Z_3^\omega)$,
\begin{equation}\label{eq-conditional-exp}
\E_{\Prob^\omega}\left(\bigl|e^{(2\pi i/t)\Gamma\left({\tiny
Z_1^\omega,
\begin{array}{c}
Z^\omega_2\\
Y^{\omega}_{2}
\end{array},
\begin{array}{c}
Z^\omega_3\\
X^{\omega}_{3}
\end{array},
\begin{array}{c}
Y^\omega_{4}\\
X^{\omega}_{4}
\end{array}, X^\omega_{5}}
\right)}-1\bigr|^2
\bigg|{\tiny
\begin{array}{l}
Z^\omega_1=\zeta_1\\
Z^\omega_2=\zeta_2\\
Z^\omega_3=\zeta_3\\
\end{array}}\right)=0.
\end{equation}

By the Markov property, this conditional expectation has canonical interpretation for {\em every} $(\omega,\zeta_1,\zeta_2,\zeta_3)$ in the set
$$
A=\{(\omega,a,b,c): p(\omega,a,b)p(T\omega,b,c)>0 \}.
$$
By assumption (C2),  $A$ is open. By assumption (C1), every open subset of $A$ has positive measure with respect to the measure $\int \Prob^\omega m(d\omega)$. By assumption (C2), the left-hand-side of \eqref{eq-conditional-exp} depends continuously on
$(\omega,\zeta_1,\zeta_2,\zeta_3)$.
Therefore \eqref{eq-conditional-exp}  is true for {\em all} $(\zeta_1,\zeta_2,\zeta_3)\in A$.

Thus to prove the claim it remains to construct measurable functions $\zeta_1(\omega),\zeta_2(\omega)$ such that $(\omega,\zeta_1(\omega),\zeta_2(\omega),\zeta_1(\omega))\in A$ for all $\omega$.

By the ellipticity condition $\int_{\fS}p(\omega,a,\zeta)p(T\omega,\zeta,a)\mu_{T\omega}(d\zeta)>\epsilon_0$, so for every $\omega$ there are $(\zeta_1,\zeta_2)$ s.t.
$$
\psi(\omega,(\zeta_1,\zeta_2)):=p(\omega,\zeta_1,\zeta_2)p(T\omega,\zeta_2,\zeta_1)>0.
$$
By Lemma \ref{Lemma-Measurable-Selection} it is possible to choose measurable $\zeta_1(\omega), \zeta_2(\omega)$ like this. Claim 1 is proved.

\medskip
Given $\omega\in\Omega$ and $a,b\in\fS$, construct the bridge distribution $\Prob^\omega_{ab}(E)=\Prob^\omega(Y^\omega_2\in E|Z_1^\omega=a,X^\omega_3=b)$ as in \S\ref{Section-Bridge}.

\medskip
\noindent
{\sc Claim 2.\/} {\em
For a.e. $\omega$, for a.e.  $(\xi_3,\xi_4,\xi_5)$ sampled from the joint distribution of $(X^\omega_3,X^{\omega}_4,X^\omega_5)$, the  random variables
\begin{align*}
&W^\omega_3:=f(\omega,\zeta_1(\omega),Y_2)+f(T\omega,Y_2,\xi_3),\ \ \ Y_2\sim \Prob^\omega_{\zeta_1,\xi_3}\\
&W^{T^2\omega}_5:=f(T^2\omega,\zeta_1(\omega),Y_4)+f(T^3\omega,Y_4,\xi_5),\ \ \ Y_4\sim \Prob^{T^2\omega}_{\zeta_1,\xi_5}
\end{align*}
are purely atomic, and belong to some  coset of $\frac{2\pi}{t}{\Z}$ with full probability. (These cosets could be different.)
}

\medskip
\noindent
{\em Proof.\/} By choice of $\zeta_i(\omega)$ and Fubini's theorem,  for a.e. $(\xi_3,\xi_4,\xi_5)\sim (X^\omega_3,X^\omega_4,X^\omega_5)$,
$$
\Gamma\left(
\zeta_{1}(\omega),
\begin{array}{c}
\zeta_{2}(\omega)\\
Y^{\omega}_{2}
\end{array},
\begin{array}{c}
\zeta_{1}(\omega)\\
\xi_{3}
\end{array},
\begin{array}{c}
Y^\omega_{4}\\
\xi_{4}
\end{array}, \xi_{5}
\right)\in \frac{2\pi}{t}\Z\hspace{0.5cm} \Prob^\omega\left(\cdot
\bigg|{\tiny
\begin{array}{ll}
Z^\omega_3=\zeta_1 & X^\omega_3=\xi_3\\
Z^\omega_2=\zeta_2 & X^\omega_4=\xi_4\\
Z^\omega_1=\zeta_1 & X^\omega_5=\xi_5\\
\end{array}}\right)\text{--a.e.}
$$
Notice that $\Gamma\left(
\zeta_{1}(\omega),
\begin{array}{c}
\zeta_{2}(\omega)\\
Y^{\omega}_{2}
\end{array},
\begin{array}{c}
\zeta_{1}(\omega)\\
\xi_{3}
\end{array},
\begin{array}{c}
Y^\omega_{4}\\
\xi_{4}
\end{array}, \xi_{5}
\right)$ is equal to the {\em independent} difference of $W^\omega_3$ and $W^{T^2\omega}_5$, plus a constant which only depends on $\omega$.
The claim now follows from Lemma \ref{Lemma-e-to-the-i-W}.

\medskip
\noindent
{\sc Claim 3.\/} {\em Given $\omega$ and $(\xi_3,\xi_4,\xi_5)$ as in claim 2, let
\begin{align*}
g(\omega,\xi_3)&:=\left(\begin{array}{l}
\text{the smallest positive atom of $W^\omega_3$ if $\exists$ positive atoms, }\\
\text{otherwise, the largest non-positive atom of $W^\omega_3$}
\end{array}\right)\\
c(\omega)&:=-f(\omega,\zeta_1(\omega),\zeta_2(\omega))-f(T\omega,\zeta_2(\omega),\zeta_1(\omega)).
\end{align*}
These functions are well-defined, measurable, and
\begin{equation}\label{eq-comolog}
[f(T^2\omega,\xi_3,\xi_4)+f(T^3\omega,\xi_4,\xi_5)]+g(\omega,\xi_3)-g(T^2\omega,\xi_5)+c(\omega)\in \frac{2\pi}{t}\Z
\end{equation}
 for $\mu$-a.e. $\omega$, for a.e. $(\xi_3,\xi_4,\xi_5)\sim (X^\omega_3,X^\omega_4,X^\omega_5)$.
}

\noindent
\medskip
{\em Proof.\/} The function $g(\omega,\xi_3)$ is well-defined for a.e. $\omega$ because of claim 2. To see that it is measurable, we note that $(\omega,\xi_3)\mapsto \Prob(W^\omega_3\in (a,b))$ are measurable, and
\begin{align*}
&[g(\omega,\xi_3)>a]=
\{(\omega,\xi_3): \Prob(0<W^\omega_3\leq a)=0\ ,
\Prob(W^\omega_3>a)\neq 0\} & (a>0)\\
&[g(\omega,\xi_3)>a]=
\{(\omega,\xi_3): \Prob(W^\omega_3>a)\neq 0\} &(a\leq 0)
\end{align*}
are measurable.
The measurability of $c(\omega)$ is clear.

Equation \eqref{eq-comolog} holds because the left-hand-side of \eqref{eq-comolog}  is, up to a sign, an atom of the random variable
$$
\Gamma\left(
\zeta_{1}(\omega),
\begin{array}{c}
\zeta_{2}(\omega)\\
Y^{\omega}_{2}
\end{array},
\begin{array}{c}
\zeta_{1}(\omega)\\
\xi_{3}
\end{array},
\begin{array}{c}
Y^\omega_{4}\\
\xi_{4}
\end{array}, \xi_{5}
\right),\ \ \ (\un{L}^\omega_3,\un{L}^\omega_4)\sim \Prob^\omega\left(\cdot
\bigg|{
\begin{array}{ll}
Z^\omega_3=\zeta_1(\omega) & X^\omega_3=\xi_3\\
Z^\omega_2=\zeta_2(\omega) & X^\omega_4=\xi_4\\
Z^\omega_1=\zeta_1(\omega) & X^\omega_5=\xi_5\\
\end{array}}\right)
$$
and we chose $(\zeta_1(\omega),\zeta_2(\omega))$ so that this random variable takes values in $\frac{2\pi}{t}\Z$ a.s.

\medskip
Claim 3 gives us measurable functions $a(\omega,x):=\exp(-itg(T^{-2}\omega,x))$ and
$\lambda(\omega):=\exp(-itc(T^{-2}\omega))$ such that
$$
e^{it [f(\omega,X^\omega_1,X^\omega_2)+f(T\omega,X^\omega_2,X^\omega_3)]}=\lambda(\omega) \frac{a(\omega,X^\omega_1)}{a(T^2\omega, X^\omega_3)}.
$$
Multiplying both sides of the equation by $e^{it[f^\omega-f^\omega\circ T]}$ gives
$$
e^{2it f(\omega,X^\omega_1,X^\omega_2)}=\lambda(\omega)\frac{b(\omega,X_1^\omega,X_2^\omega)}{b(T\omega,X^\omega_2,X^\omega_3)},
$$
where $b(\omega,x,y):=a(\omega,x)a(T\omega,y)e^{it f(\omega,x,y)}$.

\smallskip
This not quite a relative cohomology to a coset of $(\pi/t)\Z$, because  $b(\omega,x,y)$ seems to depend not just on $x$ but also on $y$. In fact there is a bounded measurable function $\beta(\omega,x)$ such that
$$
b(\omega,X^\omega_1,X^\omega_2)=\beta(\omega,X_1^\omega)\ \  \mathbb P-\text{ almost everywhere},
$$
where $\mathbb P$ is given by \eqref{P-measure}.
This can be seen as follows. Rearrange terms to see that
$$
b(T\omega,X_2^\omega,X^\omega_3)=\lambda(\omega)e^{-2i t f(\omega,X^\omega_1,X^\omega_2)}b(\omega,X_1^\omega,X_2^\omega).
$$
By the Markov property of $\mathsf X^\omega$, for fixed $\omega$, the left-hand-side and the right-hand-side of this equation are conditionally independent given $X^\omega_2$. Two independent random variables which are equal, must be constant. So for $m$-a.e. $\omega$,
$$b(T\omega,X_2^\omega,X_3^\omega)=\E^\omega(b(T\omega,X_2^\omega,X_3^\omega)|X_2^\omega).$$
 Setting
$$
\beta(\omega,X_1^\omega):=\E^{\omega}(b(\omega,X_1^\omega,X_2^\omega)|X_1^\omega)
$$
and using stationarity to shift indices where needed, we find that
$$b(\omega,X_1^\omega,X_2^\omega)=\beta(\omega,X_1^\omega)\;\; \mathbb P-a.e.,
\quad b(T\omega,X_2^\omega,X_3^\omega)=\beta(T\omega,X_2^\omega)\;\;
\mathbb P-a.e.$$
Hence
$$
e^{2it f(\omega,X^\omega_1,X^\omega_2)}=\lambda(\omega)\frac{\beta(\omega,X_1^\omega)}{\beta(T\omega,X^\omega_2)}\ \ \mathbb P{\text{-a.e.}}.
$$
So $f$ is relatively cohomologous to a coset of $\frac{\pi}{t}\Z$.

\medskip
We obtained a contradiction to our assumptions. This contradiction shows that $G_{ess}(\mathsf X^\omega,\mathsf f^\omega)=\R$ for a.e. $\omega$. The local limit theorem now follows from
Theorem~\ref{ThLLT-classic}, applied to $(\mathsf X^\omega, \mathsf f^\omega)$, since Theorem \ref{Thm-Variance-RMC} gives the a.s. asymptotic $V_N^\omega\sim N\sigma^2$ for some $\sigma^2>0$ independent of $\omega$.
\hfill$\Box$

\medskip
\noindent
{\bf Proofs of Theorem \ref{Thm-Lattice-LLT-RMC} and \ref{Thm-Lattice-LLT-RMC-infinite} on the lattice case.}
Theorems \ref{Thm-Lattice-LLT-RMC} and \ref{Thm-Lattice-LLT-RMC-infinite} provide the LLT for Markov chains in random environment with finite and infinite noise processes for integer valued additive functionals, under the assumption that $f$ is not relatively cohomologous to a coset of $t\Z$
with $t\neq 1$.

The proof is similar to the proof in the non-lattice case, except that now to check irreducibility we need to show that $H_\omega=\Z$ almost surely. Since $f$ is integer valued, $1\in H_\omega$, so if this is not the case then necessarily $H_\omega=t\Z$ for $t=\frac{1}{n}$ and $n\in\N$. Now repeat the proof of Theorems \ref{Thm-Non-Lattice-LLT-RMC} and \ref{Thm-Non-Lattice-LLT-RMC-infinite} verbatim.\hfill$\Box$

\subsection{Log-moment generating functions and rate functions}
We prove  Theorem \ref{Theorem-MCRE-free-energy} on the a.s. convergence of the log-moment generating functions of $(\mathsf X^\omega,\mathsf f^\omega)$ and their Legendre transforms.
 Suppose $\mathsf f$ is an essentially bounded additive functional on a MCRE with a finite noise space $(\Omega,\mathfs B,m,T)$. Without loss of generality, $m(\Omega)=1$.

\medskip
\noindent
{\bf Part (1): Convergence of log-moment generating functions:}
We are asked  to show that  for a.e. $\omega$, $\mathfs G_N^\omega(\xi):=\frac{1}{N}\log\E(e^{\xi S_N^\omega})$ converge pointwise on $\R$. To do this we recall three facts from chapter \ref{Chapter-LDP}:

\medskip
\noindent
{\sc Fact 1:} {\em Given $\xi\in\R$, for every $\omega\in\Omega$ there are unique numbers $\ov{p}_n(\xi,\omega)\in\R$ and unique non-negative functions $\ov{h}_n(\cdot,\xi,\omega)\in L^\infty(\fS,\mathfs B(\fS),\mu_{T^{n-1}\omega})$ such that $\int_{\fS}\ov{h}_n(x,\xi,\omega)\mu_{T^{n-1}\omega}(dx)=1$ for all $n\geq 1$, and }
\begin{equation}\label{four-pm}
\int_{\fS} e^{\xi f(T^n\omega,x,y)}\frac{\ov{h}_{n+1}(y,\xi,\omega)}{e^{\ov{p}_n(\xi,\omega)}\ov{h}_{n}(x,\xi,\omega)}
\pi(T^n\omega,x,dy)=1.
\end{equation}
{\em
Furthermore, $\ov{p}_n(\xi,\omega)=\ov{p}(\xi,T^n\omega)$ for all $n$, where and $\ov{p}(\xi,\omega)$ is measurable.  }

\medskip
\noindent
{\em Proof.\/} The existence and uniqueness of $\ov{h}_n, \ov{p}_n$ follows from Lemma \ref{Lemma-h-exist}, applied to $(\mathsf X^\omega,\mathsf f^\omega)$ with $a_n=0$. Writing \eqref{four-pm} first for $(n,\omega)$ and then for $(n-1, T\omega)$, and then invoking uniqueness, we find that $\ov{p}_n(\xi,\omega)=\ov{p}_{n-1}(\xi,T\omega)$. So $$\ov{p}_n(\xi,\omega)=\ov{p}_{n-1}(\xi,T\omega)=\cdots=\ov{p}_1(\xi,T^{n-1}\omega)=\ov{p}(\xi,T^n\omega),$$ where $\ov{p}(\xi,\omega):=\ov{p}_1(\xi,T^{-1}\omega)$. The proof of Lemma  \ref{Lemma-h-exist} represents  $h_n(x,\xi,\omega)$ as a limit of expressions which are measurable in $(x,\xi,\omega)$, so $(x,\xi,\omega)\mapsto \ov{h}_n(x,\xi,\omega)$ is measurable. By \eqref{four-pm}, $(\omega,\xi)\mapsto\ov{p}(\xi,\omega)$ is measurable.

\medskip
\noindent
{\sc Fact 2:} {\em Let $K:=\ess\sup|\mathsf f|$ and let $\epsilon_0$ denote a uniform ellipticity bound for $\mathsf X^\omega$. For every $R>0$ there exists a constant $C(\epsilon_0,K,R)$ such that $|\ov{p}(\xi,\omega)|\leq C(\epsilon_0,K,R)$ for all $\omega\in\Omega$ and $|\xi|\leq R$.}

\medskip
\noindent
{\em Proof.\/} See the proof of Lemma \ref{Lemma-h-bounded}.

\medskip
\noindent
{\sc Fact 3:} {\em Let $\DS \ov{P}_N(\xi,\omega):=\sum_{k=1}^N \ov{p}(\xi,T^k\omega)$, then  for a.e. $\omega\in\Omega$,
$$\mathfs G_N^\omega(\xi)=\left(\frac{V_n^\omega}{N}\right)
\left[\frac{\ov{P}_N(\xi,\omega)}{V_N^\omega}+O\left(\frac{1}{V^\omega_N}\right)\right]\text{ uniformly on compact subsets of $\xi\in\R$.}
$$}

\medskip
\noindent
{\em Proof.\/} It is convenient to work with $\mathfs F_N^\omega(\xi):=\frac{1}{V_N^\omega}\log\E(e^{\xi S_N^\omega})\equiv (N/V_N^\omega)\mathfs G_N^\omega(\xi)$.  Let $P_N(\xi,\omega):=\ov{P}_N(\xi,\omega)+
\bigl(\E(S_N^\omega)-\tfrac{d}{d\xi}\big|_{\xi=0}\ov{P}_N'(0,\omega)\bigr)\xi$.
For each $\omega\in\Omega$ such that $V_N^\omega\to\infty$,
\begin{enumerate}[(1)]
\item $\tfrac{d}{d\xi}\big|_{\xi=0}\ov{P}_N'(0,\omega)$ exists, by Lemma \ref{Lemma-Analyticity}.
\item $|P_N(\xi,\omega)-\ov{P}_N(\xi,\omega)|=O(1)$ uniformly on compact subsets of $\xi\in\R$, by Lemmas \ref{Lemma-Changed-Expectation-Variance-bar} and  \ref{Lemma-Changed-Expectation-Variance}.
\item $|\mathfs F_N^\omega(\xi)-P_N(\xi)/V_N^\omega|=O(1/V_N^\omega)$ uniformly on compact subsets of $\xi\in\R$, by  Lemma \ref{Lemma-Changed-Expectation-Variance}. Fact 3 follows.
\end{enumerate}

\smallskip
We can now prove the a.s. convergence of $\mathfs F_N^\omega(\xi)$.
By the assumptions of the theorem, $\mathsf f$ is not relatively cohomologous to a constant.
Therefore,  by Theorem~\ref{Thm-Variance-RMC},  there exists $\sigma^2>0$ such that   $V_N^\omega\sim \sigma^2 N$ as $N\to\infty$ for a.e. $\omega$.

Fix a countable dense set $\{\xi_1,\xi_2,\ldots\}\subset\R$. For each $i$, $\omega\mapsto \ov{p}(\xi_i,\omega)$ is bounded and measurable. So for a.e. $\omega$,
\begin{align*}
&\lim\limits_{N\to\infty}\mathfs G_N^\omega(\xi_i)=\sigma^2\lim\limits_{N\to\infty}\frac{1}{V_N^\omega}\sum_{k=1}^N \ov{p}(\xi_i,T^k\omega)=\lim\limits_{N\to\infty}\frac{1}{N}\sum_{k=1}^N \ov{p}(\xi_i,T^k\omega)\\
&=\int_\Omega \ov{p}(\xi_i,\omega)m(d\omega), \text{ by the Birkhoff ergodic theorem}.
\end{align*}
This shows that for all $i$ there exists $\mathfs G(\xi_i)\in\R$ such that
$
\lim\limits_{N\to\infty}\mathfs G_N^\omega(\xi_i)=\mathfs G(\xi_i)$ for a.e. $\omega$. Let $\Omega'$ denote the set of full measure of $\omega$ where this holds for all $i\in\N$.

Fix $\omega\in\Omega'$, then  the functions $\xi\mapsto\mathfs F_N^\omega(\xi)$ are equicontinuous on compacts, because if $K:=\ess\sup|\mathsf f|$, then
$$
|(\mathfs F_N^\omega)'(\xi)|\leq \left|\frac{|\xi|\E(|S_N^\omega|e^{\xi S_N^\omega})}{V_N^\omega\E(e^{\xi S_N^\omega})}\right|=\frac{|\xi|KN}{V_N^\omega}=O(|\xi|).
$$
Therefore for a.e. $\omega$,  the functions $\xi\mapsto\mathfs G_N^\omega(\xi)$ are equicontinuous on compacts.

Recall that if a sequence of functions $\vf_n(\xi)$ which is equicontinuous on compacts converges on a dense subset of $\R$,  then $\vf_n(\xi)$ converges for all $\xi\in\R$. Moreover, the limit is continuous. So there is a continuous function $\mathfs F^\omega(\xi)$ such that
$$
\lim_{N\to\infty}\mathfs G_N^\omega(\xi)=\mathfs F^\omega(\xi)\text{ for all $\xi\in\R$, $\omega\in\Omega'$}.
$$
In fact $\mathfs F^\omega(\xi)$ does not depend on $\omega$, because by virtue of  continuity,
$$
\mathfs F^\omega(\xi)=\lim_{k\to\infty}\mathfs F^\omega(\xi_{i_k})=\lim_{k\to\infty}\mathfs G(\xi_{i_k}),\text{ whenever }\xi_{i_k}\xrightarrow[k\to\infty]{}\xi_i.
$$
We are therefore free to write $\mathfs F^\omega(\xi)=\mathfs F(\xi)$.

It remains to show that $\mathfs F(\xi)$ is differentiable and strictly convex on $\R$. Fix $\omega\in\Omega'$. Applying Theorem \ref{Theorem-F_N} to $(\mathsf X^\omega,\mathsf f^\omega)$ we find that for every $R>0$ there is a $C=C(R)$ such that $C^{-1}\leq (\mathfs F^\omega_N)''\leq C$ on $[-R,R]$. This implies that $\mathfs F$ is differentiable and strictly convex on $(-R,R)$ because of the following general lemma:

\begin{lemma}\label{Lemma-More-Convex-Magic}
Suppose $\vf_n:\R\to\R$ are twice differentiable convex functions such that $C^{-1}\leq \vf''_n\leq C$ with $C>0$, on $(-R,R)$. If $\vf_n\xrightarrow[N\to\infty]{}\vf$ pointwise on $(-R,R)$, then $\vf$ is continuously differentiable and strictly convex on  $(-R,R)$.
\end{lemma}

\begin{proof}
A pointwise limit of convex functions is convex, and convex functions have one sided derivatives. Let $\vf'_{\pm}(\xi)$ denote the one-sided derivatives of at $\xi$.

\smallskip
\noindent
{\sc Differentiability:}
For all $|\xi|<R$,
\begin{align*}
&|\vf'_+(\xi)-\vf'_-(\xi)|=\lim_{h\to 0^+}\left|\frac{\vf(\xi+h)-\vf(\xi)}{h}-\frac{\vf(\xi-h)-\vf(\xi)}{h}\right|\\
&=\lim_{h\to 0^+}\lim_{n\to\infty}\left|\frac{\vf_n(\xi+h)-\vf_n(\xi)}{h}-\frac{\vf_n(\xi-h)-\vf_n(\xi)}{h}\right|\\
&=\lim_{h\to 0^+}\lim_{n\to\infty}|\vf_n'(\xi_n)-\vf_n'(\eta_n)|\text{ for some }\xi_n,\eta_n\in (\xi-h,\xi+h)\\
&\leq \lim_{h\to 0^+}\lim_{n\to\infty} 2Ch=0,\text{ because }|\vf_n''|\leq C\text{ on a neighborhood of $\xi$.}
\end{align*}
We find that  $\vf'_+(\xi)=\vf'_-(\xi)$, whence $\vf$ is differentiable at $\xi$.

\smallskip
\noindent
{\sc Strict convexity:} Suppose $-R<\xi<\eta<R$, then
\begin{align*}
&\vf'(\eta)-\vf'(\xi)=\vf'_+(\eta)-\vf'_-(\xi)=\lim_{h\to 0^+} \frac{\vf(\eta+h)-\vf(\eta)}{h}-\frac{\vf(\xi-h)-\vf(\xi)}{h}\\
&=\lim_{h\to 0^+}\lim_{n\to\infty} \frac{\vf_n(\eta+h)-\vf_n(\eta)}{h}-\frac{\vf_n(\xi-h)-\vf_n(\xi)}{h}\\
&=\lim_{h\to 0^+} \vf_n'(\eta_n)-\vf_n'(\xi_n)\text{ for some }\xi_n\in [\xi-h,\xi],\ \eta_n\in [\eta,\eta+h]\\
&\geq \liminf_{n\to\infty}C^{-1}|\eta_n-\xi_n|\geq C^{-1}(\eta-\xi), \text{ because }\vf_n''>C^{-1}\text{ on }(-R,R).
\end{align*}
It follows that $\vf'$ is strictly increasing on $(-R,R)$.

\smallskip
\noindent
{\sc The derivative is continuous:} The same calculation as before shows that if $-R<\xi<\eta<R$, then
$|\vf'(\eta)-\vf'(\xi)|\leq C|\xi-\eta|$, whence $\vf'$ is (Lipschitz) continuous on $(-R,R)$.
\qed
\end{proof}

\smallskip
\noindent
{\bf Part (2): Convergence of $\E(S_N^\omega)/N$:} We need the following standard fact.

\begin{lemma}\label{Lemma-Convex-Magic}
Suppose $\vf_n(\xi),\vf(x)$ are finite, convex, and differentiable on $(-R,R)$. If $\vf_n(\xi)\xrightarrow[n\to\infty]{}\vf(\xi)$ on $(-R,R)$, then $\vf_n'(\xi)\xrightarrow[n\to\infty]{}\vf'(\xi)$ on $(-R,R)$.
\end{lemma}

\begin{proof}
Fix $\xi\in (-R,R)$. By convexity, for every $h>0$ sufficiently small,
\begin{equation}\label{trump}
\frac{\vf_n(\xi)-\vf_n(\xi-h)}{h}\leq\vf_n'(\xi)\leq \frac{\vf_n(\xi+h)-\vf_n(\xi)}{h}.
\end{equation}
To see this note that the LHS is at most $(\vf_n)'_-(\xi)$, the RHS is at least $(\vf_n)'_+(\xi)$, and both one-sided derivatives equal $\vf_n'(\xi)$.

Passing to the limit $n\to\infty$ in \eqref{trump}, we find that
$$
\limsup \vf_n'(\xi), \liminf \vf_n'(\xi)\in \left[\frac{\vf(\xi)-\vf(\xi-h)}{h},\frac{\vf(\xi+h)-\vf(\xi)}{h}\right].
$$
We now invoke the differentiability of $\vf$, pass to the limit $h\to 0^+$,  and discover that
$\limsup \vf_n'(\xi)$ and  $\liminf \vf_n'(\xi)$ are both equal to $\vf'(\xi)$.\qed
\end{proof}

For a.e. $\omega$, $(V_N^\omega/N)\mathfs F_N^\omega(\xi)\equiv \mathfs G_N(\xi,\omega)\xrightarrow[N\to\infty]{}\mathfs F(\xi)$. So by the lemma
$$
(V_N^\omega/N)\left.\frac{d}{d\xi}\right|_{\xi=0}\mathfs F_N^\omega(\xi,\omega)\xrightarrow[N\to\infty]{}\mathfs F'(0).
$$
A calculation shows that the derivative equals $\E(S_N^\omega)/V_N^\omega$. So $\E(S_N^\omega)/N\to \mathfs F'(0)$.

\smallskip
\noindent
{\bf Part (3): Convergence of Legendre transforms.} Again, the proof is based on a general property of convex functions.

\begin{lemma}
\label{LmConvLT}
Suppose $\vf_n(\xi),\vf(\xi)$ are finite, strictly convex, continuously differentiable functions on $\R$, s.t. $\vf_n(\xi)\to\vf(\xi)$ for all $\xi\in\R$. Let $\vf'(\pm\infty):=\lim\limits_{\xi\to\pm\infty}\vf'(\xi)$. Let $\vf_n^\ast, \vf^\ast$ denote the Legendre transforms of $\vf_n,\vf$. For all $\eta\in (\vf'(-\infty), \vf'(+\infty))$, $\vf_n^\ast(\eta)$ is well-defined for all $n$ sufficiently large, and $\vf_n^\ast(\eta)\to\vf^\ast(\eta)$.
\end{lemma}
\begin{proof}
Fix $\eta\in (\vf'(-\infty), \vf'(+\infty))$. By assumption, $\vf'$ is continuous and strictly increasing. Therefore, there exists $\xi$ such that $\vf'(\eta)=\xi$.

Fix $\epsilon>0$ and $\xi_1<\xi<\xi_2$ such that $|\xi_1-\xi_2|<\epsilon$. Then
$
\vf'(\xi_1)<\eta<\vf'(\xi_2).
$
By Lemma \ref{Lemma-Convex-Magic}, $\vf_n'(\xi_i)\to\vf'(\xi_i)$, and therefore there exists $N$ such that  for all $n>N$,
$$
\vf'(\xi_1)-1<\vf'_n(\xi_1)<\eta<\vf_n'(\xi_2)<\vf'(\xi_2)+1.
$$

Since
$
\eta\in (\vf_n'(\xi_1),\vf_n'(\xi_2))
$
and $\vf_n'$ is continuous and strictly increasing, there exists a unique $\xi_n\in (\xi_1,\xi_2)$ so that
$
\vf_n'(\xi_n)=\eta.
$
So $\vf_n^\ast(\eta)$ is well-defined, and $$\vf_n^\ast(\eta)=\xi_n\eta-\vf_n(\xi_n).$$
Similarly, $\vf^\ast(\eta)=\xi\eta-\vf(\xi)$.

We now  estimate the distance between $\vf_n^\ast(\eta)$ and $\vf^\ast(\eta)$. Recall first that for all $n>N$, $\vf'(\xi_1)-1<\vf_n'(\xi_1)<\vf_n'(\xi_2)<\vf'(\xi_2)+1$. Let $$M:=\max\{|\vf'(\xi_1)-1|, |\vf'(\xi_2)+1|\},$$ then $|\vf_n'|\leq M$ on $(\xi_1,\xi_2)$ for all $n>N$. Consequently,
\begin{align*}
&|\vf_n^\ast(\eta)-\vf^\ast(\eta)|\leq |\xi_n-\xi|\cdot|\eta|+|\vf_n(\xi_n)-\vf(\xi)|\\
&\leq |\xi_1-\xi_2|\cdot|\eta|+|\vf_n(\xi_n)-\vf_n(\xi)|+|\vf_n(\xi)-\vf(\xi)|\\
&\leq
\eps |\eta|+M|\xi_n-\xi|+|\vf_n(\xi)-\vf(\xi)|\leq \epsilon(M+|\eta|)+o(1), \text{ as $n\to\infty$},
\end{align*}
because $\vf_n(\xi)\to\vf(\xi)$, $\xi,\xi_n\in (\xi_1,\xi_2)$, and $|\xi_1-\xi_2|\leq\epsilon$. Since $\epsilon$ is arbitrary, we have that $\vf_n^\ast(\eta)\to\vf^\ast(\eta)$.\qed
\end{proof}

\smallskip
\noindent
{\bf Part (4): Properties of $\mathfs I(\eta)$.} Fix $\omega$ such that $\vf_N(\xi):=\frac{1}{N}\log\E(e^{\xi S_N^\omega})$ converges pointwise to $\mathfs F$. By
{ Lemma \ref{LmConvLT}}, $\vf_N^\ast$ converges pointwise to $\mathfs I$. Since $\vf_N''$ is uniformly bounded away from zero and infinity on compacts (see the first part of the proof), $(\vf_N^\ast)''$ is uniformly bounded away from zero and infinity on compacts. Hence
by Lemma \ref{Lemma-More-Convex-Magic}
$$
\mathfs I=\lim\vf_N^\ast\text{ is strictly convex and continuously differentiable}.
$$
By Lemma \ref{Lemma-Convex-Magic},
$(\vf_N^\ast)'(\eta)\xrightarrow[N\to\infty]{}\mathfs I'(\eta)$ for all $\eta$ in the interior of the range of  $\vf'$, and $\vf_N(\xi)\xrightarrow[N\to\infty]{}\mathfs F'(\xi)$ for all $\xi\in\R$.  The convergence is uniform on compacts, because $(\vf_N^\ast)'', \vf_N''$ are bounded on compacts.

It is easy to verify that $\vf_N$ is twice differentiable. Therefore by Lemma \ref{LmDerLegendre}, $\vf_N^\ast$ is twice differentiable and
$
(\vf_N^\ast)'(\vf_N'(\xi))=\xi\text{ for all }\xi.
$
Passing to the limit as $N\to\infty$ we obtain the important identity
$
\mathfs I'(\mathfs F'(\xi))=\xi\text{ for all }\xi\in\R.
$

One consequence of the identity $\mathfs I'(\mathfs F'(\xi))=\xi$ is that   $\mathfs I'(\mathfs F'(0))=0$, so $\eta=\mathfs F'(0)$ is a critical point of $\mathfs I(\cdot)$. By strict convexity, $\mathfs I$ attains its global minimum at $\mathfs F'(0)$. The value there is zero:
$$\mathfs I(\mathfs F'(0))=0\cdot\mathfs F'(0)-\mathfs F(0)=0.$$ We conclude that $\mathfs I(\eta)=0$ when $\eta=\mathfs F'(0)$, and $\mathfs I(\eta)>0$ for $\eta\neq \mathfs F'(0)$.

Another consequence of the identity $\mathfs I'(\mathfs F'(\xi))=\xi$ (and the fact that  $\mathfs F'$ is increasing) is that $\mathfs I'(\xi)\xrightarrow[\xi\to\pm\infty]{}\pm\infty$, and therefore
$\mathfs I$
 has compact level sets.\hfill$\Box$

\smallskip
\noindent
{\bf Part (5): Large deviation threshold.} We prove the identity for $\mathfs F'(\infty)$, the identity for $\mathfs F'(-\infty)$ follows by replacing
$\mathsf{f}\to-\mathsf{f}.$

\smallskip
\noindent
{\em Step I.} $\fc_+\geq \cF'(+\infty).$

\smallskip
\noindent
{\em Proof.} Given $\eta\in (\cF'(-\infty), \cF(+\infty))$ choose
$$ \cF'(-\infty)<\eta^-< \eta<\eta^+<\cF(+\infty) .$$ Take $\xi^\pm$ be s.t.
$\cF'(\xi^\pm)=\eta^\pm.$ By Lemma \ref{Lemma-Convex-Magic},  $\DS \lim_{N\to\infty}\cF_N'(\xi^\pm)=\eta^\pm.$
Hence for large $N$ $\cF_N'(\xi^-)\leq \eta \leq \cF_N(\xi^+)$ and so $\eta$ is reachable.

\smallskip
\noindent
{\em Step II.} $\fc_+\leq \cF'(+\infty).$

\smallskip
\noindent
{\em Proof.} Take $\eta>\fc_+.$ If $\eta\in \cC_R$ for some $R$ we would have that for some $R$, for all $N$ large enough
$\cF_N'(R)\geq \eta$ (see Lemma \ref{Lemma-Changed-Expectation-Variance}(5)). However by
Lemma \ref{Lemma-Convex-Magic},  $\DS \lim_{N\to\infty} \cF_N(R)=\cF'(R)<\cF'(+\infty)<\eta$
contradicting our assumption that $\eta$ is reachable.

\smallskip
\noindent
{\em Step III.} Denote $\cS_N(\omega)=\ess\sup S_N^\omega.$ Then the limit
$\DS \fs^+:=\lim_{N\to\infty} \frac{\cS_N(\omega)}{N}$ exists and is independent
of $\omega$ with probability one.

\smallskip
\noindent
{\em Proof.} By our ellipticity assumption
$$\cS_{N+M}(\omega)\leq \cS_N(\omega)+S_M(T^N \omega)-4K. $$
Thus the sequence $\cT_N(\omega)=\cS_N(\omega)-4K$ is subadditive.
Since $\cS_N(\omega)\geq -KN$ the Subadditive Ergodic Theorem implies that
the limit
$\DS \lim_{N\to\infty} \frac{\cS_N(\omega)}{N}=\lim_{N\to\infty} \frac{\cT_N(\omega)}{N}$ exists and is independent
of $\omega$ with probability one.

\smallskip
\noindent
{\em Step IV.} $\fc_+\leq \fs^+$ because for each $\eps>0$ we have that with probability one
for large $N$, $\Prob^\omega(S_N\geq (\fs^++\eps) N)=0.$

\smallskip
\noindent
{\em Step V.} $\fc_+\geq \fs^+$

\smallskip
\noindent
{\em Proof.} Fix $\eps>0.$ By Step III for each sufficiently large $N_0$
there exists $\gamma_{\eps, N_0}>0$ and a set $\Omega_{\eps, N_0}$ s.t.
$m(\Omega_{\eps, N_0})\geq 1-\eps$ and for all $\omega\in\Omega_{\epsilon,N_0}$,
for $\mu_\omega$-a.e. $x\in\fS$,
\begin{equation}
\label{EqLocLD}
\Prob^\omega(S_{N_0}\geq (\fs^+-\eps) N_0|X_1=x)\geq \gamma_{\eps, N_0}.
\end{equation}
Given $M$ let $j_1(\omega)<j_2(\omega)<\cdots<j_{n_M(\omega)}(\omega)$ be all the times $1\leq j<M$ when $T^{j N_0}(\omega)\in \Omega_{\epsilon,N_0}$, then
$$ \Prob^\omega\left(\cS_{N_0 M}\geq n_M (\fs^+-\eps) N_0-(M-n_M) N_0 K\right)
\geq \gamma_{\eps, N_0}^{N_0 M}. $$
(To see this,  estimate conditional probabilities of this event given $X_{j_1}^\omega,\ldots,X_{j_{n_M}}^\omega$ using \eqref{EqLocLD},
and take expectation over $X_{j_1}^\omega,\ldots,X_{j_{n_M}}^\omega$.)

By the Ergodic Theorem, for a.e. $\omega$ there is a limit
$$\beta(\omega)=\lim_{M\to\infty}\frac{n_M(\omega)}{M}\quad\text{and}\quad
\int \beta(\omega) dm=m(\Omega_{\epsilon,N_0}). $$
So for large $M$, and on a set $\brOmega_\eps$ of positive measure, $n_M/M>1-2\epsilon$ whence
$$ n_M (\fs^+-\eps) N_0-(M-n_M) N_0 K\geq \left[(1-2\eps)(\fs^+-\eps)-2\eps K \right] N_0 M.$$
Now Theorem \ref{ThLDP=LLTLD}(c) shows that on $\brOmega_\eps$,
$ \fc^+(\omega) \geq (1-2\eps) (\fs^+-\eps)-2\eps K . $

{ By steps I and II above, $\fc^+$ actually does not depend on $\omega$
(in fact, using Theorem \ref{ThLDP=LLTLD}, it is easy to verify directly
that $\fc^+$ is $T$-invariant, and therefore by ergodicity, constant).
 we get
that
$$ \fc^+(\omega) \geq (1-2\eps) (\fs^+-\eps)-2\eps K $$
almost surely.} Since $\eps$ is arbitrary the result follows.
\hfill $\Box$

\section{Notes and references}\label{Section-Notes-Chapter-8}
Markov chains in random environment (MCRE) should not be confused with ``random walks in random environment" (RWRE). In the RWRE model, the transition kernel at time $n$ depends on the  {\em position} of random walk at time $n$, i.e. $\pi_n(x,dy)=\pi(S_n,x,dy)$. In a MCRE, the transition kernel at time $n$ depends on the {\em noise} at time $n$, i.e. $\pi_n(x,dy)=\pi(T^n\omega,x,dy)$. For a recent treatment of the LLT for RWRE, see \cite{Castell} and references therein.

Markov chains in random environment were introduced by Cogburn \cite{Cogburn-1980}. The setup is a particular case of a ``random dynamical system."
For a fixed realization of noise, a Markov chain in random environment reduces to an inhomogeneous Markov chain, and a  random dynamical system reduces to  a  ``sequential" (aka ``time-dependent" or ``non-autonomous") dynamical system.
Various authors considered probabilistic limit theorems in these contexts.
Limit theorems for Markov chains in random environment are given in  Cogburn \cite{Cogburn}, Sepp\"al\"ainen \cite{Seppalainen}, Kifer \cite{Kifer-RPF}, \cite{Kifer-CLT}  and Hafouta \& Kifer \cite[chapters 6,7]{Hafouta-Kifer-Book},\cite{Hafouta-Kifer-Nonconventional}. Results for  random dynamical systems can be found in  Kifer \cite{Kifer-CLT},
Conze, Le Borgne \& Roger \cite{Conze-Le-Borgne}, Denker \& Gordin \cite{Denker-Gordin}, Aimino, Nicol \& Vaienti \cite{Aimino-Nicol-Vaienti}, Nicol, T\"or\"ok \& Vaienti \cite{Nicol-Torok-Vaienti}, and Dragi\v{c}evi\'c, Froyland \& Gonz\'alez-Tokman \cite{Dragicevic-Froyland-Gonzalez-Tokman}. For limit theorems for sequential dynamical systems, see  Bakhtin \cite{Ba},
Conze \& Raugi \cite{Conze-Raugi-Sequential}, Haydn, Nicol \& T\"or\"ok \cite{Haydn-Nicol-Torok}, Korepanov, Kosloff \& Melbourne \cite{Korepanov-Kosloff-Melbourne}, and Hafouta \cite{Hafouta-Skew-Products,Hafouta-Sequential}.

If we set the noise process to be the identity on the one point space, then the LLT in this chapter reduce to  LLT for homogeneous stationary Markov chains, as in  Theorem \ref{ThLLTHom}.  For more general LLT for homogeneous Markov chains, see  Nagaev \cite{N}, Guivarc'h \& Hardy \cite{GH}.

The results  of this chapter are all essentially known  in the case $T$ preserves
a finite measure.
Theorem \ref{Thm-Variance-RMC} was proved in the more general setup of random dynamical systems by Kifer \cite{Kifer-CLT},\cite{Kifer-1986-book}. Theorems \ref{Thm-Non-Lattice-LLT-RMC} and \ref{Thm-Lattice-LLT-RMC} are close to the (earlier) results of  Dragi\v{c}evi\'c, Froyland \& Gonz\'alez-Tokman \cite{Dragicevic-Froyland-Gonzalez-Tokman}, and  Hafouta \& Kifer \cite[chapter 7, Theorem 7.1.5]{Hafouta-Kifer-Book}. The main difference is  in the irreducibility assumptions. Our condition of non-relative cohomology to a coset is replaced in \cite{Hafouta-Kifer-Book} by what these authors call the ``lattice" and ``non-lattice" cases (this is not the same as our terminology). In the paper \cite{Dragicevic-Froyland-Gonzalez-Tokman}, the non-cohomology condition is replaced by  a condition on the decay of the norms of certain perturbed characteristic function operators, and a connection to a non-cohomology condition is made under additional assumptions.

The results for infinite noise processes seem to be new.
The reason we can also treat this case, is that the LLT we provide in this work do not require any assumptions on the rate of growth of $V_N$, and they also work when it grows sub-linearly.

\newpage
\small
\bibliographystyle{plain}
\bibliography{MC-LLT-bibliography}{}

\printindex

\end{document}

%% file: abbrev.tex
\def \qed{\hfill$\square$}

\def\tilde{\widetilde}
\def\bar{\overline}

\def\diam{{\rm diam}}
\def\dist{{\rm dist}}
\def\eps{{\varepsilon}}
\def\teps{{\tilde\varepsilon}}

\def\supp{{\rm supp}}

\def\Cov{{\rm Cov}}

\def\Osc{\mathop{\rm Osc}}
\def\Prob{{\mathbb{P}}}

\def\Var{{\rm Var}}

\def\EXP{{\mathbb{E}}}

\def\bbf{\mathbb{F}}

\def\bbT{\mathbb{T}}

\def\naturals{\mathbb{N}}

\def\bx{\mathbf{x}}
\def\by{\mathbf{y}}
\def\bz{\mathbf{z}}

\def\bomega{{\boldsymbol{\omega}}}

\def\brC{{\bar C}}

\def\brK{{\bar K}}

\def\brN{{\bar N}}

\def\brP{{\bar P}}

\def\brV{{\bar V}}

\def\brc{{\bar c}}

\def\bru{{\bar{u}}}

\def\brz{{\bar z}}

\def\brx{{\bar x}}

\def\bry{{\bar y}}

\def\brDelta{{\bar \Delta}}

\def\brdelta{{\bar \delta}}

\def\breps{{\bar\eps}}

\def\breta{{\bar \eta}}

\def\brPhi{{\bar \Phi}}

\def\bromega{{\bar\omega}}

\def\brOmega{{\bar\Omega}}

\def\cA{\mathcal{A}}

\def\cC{\mathcal{C}}

\def\cI{\mathcal{I}}

\def\cF{\mathcal{F}}

\def\cL{\mathcal{L}}
\def\brcL{\bar{\mathcal{L}}}
\def\hcL{\hat{\mathcal{L}}}
\def\tcL{\tilde{\mathcal{L}}}

\def\cM{\mathcal{M}}

\def\cS{\mathcal{S}}

\def\cT{\mathcal{T}}

\def\fA{\mathfrak{A}}

\def\fd{\mathfrak{d}}

\def\fI{\mathfrak{I}}

\def\fS{\mathfrak{S}}

\def\fu{\mathfrak{u}}

\def\fX{\mathfrak{X}}

\def\fZ{\mathfrak{Z}}

\def\fa{\mathfrak{a}}

\def\fc{\mathfrak{c}}

\def\fs{\mathfrak{s}}

\def\hC{{\hat C}}

\def\hK{{\hat K}}

\def\ha{{\hat a}}

\def\hb{{\hat b}}

\def\hc{{\hat c}}

\def\heps{{\hat{\eps}}}

\def\hmu{{\hat\mu}}

\def\hphi{{\hat\phi}}

\def\hPhi{{\hat\Phi}}

\def\htheta{{\hat\theta}}

\def\tC{{\tilde C}}

\def\ta{{\tilde a}}

\def\tb{{\tilde b}}

\def\tf{{\tilde f}}

\def\tp{{\tilde p}}

\def\tq{{\tilde q}}

\def\tu{{\tilde u}}

\def\tx{{\tilde x}}

\def\tdelta{{\tilde\delta}}

\def\teps{{\tilde\varepsilon}}

\def\tpi{{\tilde\pi}}

\def\txi{{\tilde\xi}}

\def\tPhi{{\tilde{\Phi}}}

\makeatother

\def\hat{\widehat}

\def\beq{\begin{equation}}
\def\eeq{\end{equation}}

\def\R{\mathbb R}
\def\E{\mathbb E}
\def\N{\mathbb N}
\def\Z{\mathbb Z}